\documentclass[11pt]{article}
\usepackage{amsmath}[1996/11/01]
\usepackage{amssymb,amsthm,amsxtra}
\usepackage{amscd}
\usepackage[margin=1in]{geometry}
\usepackage[title,titletoc]{appendix}

\def\[#1\]{\begin{equation}#1\end{equation}}
\makeatletter
\def\beq{%
   \relax\ifmmode
      \@badmath
   \else
      \ifvmode
         \nointerlineskip
         \makebox[.6\linewidth]%
      \fi
      $$
   \fi
}
\def\eeq{%
   \relax\ifmmode
      \ifinner
         \@badmath
      \else
         $$
      \fi
   \else
      \@badmath
   \fi
   \ignorespaces
}

\def\enddisplaymath{\eeq\global\@ignoretrue}
\makeatother

\newtheorem{thm}{Theorem}
\newtheorem{cor}[thm]{Corollary}
\newtheorem{lem}[thm]{Lemma}
\newtheorem{prop}[thm]{Proposition}

\theoremstyle{remark}
\newtheorem*{rem}{Remark}
\newtheorem{rems}{Remark}[thm]

\theoremstyle{definition}
\newtheorem{defn}{Definition}

\numberwithin{equation}{section}
\numberwithin{thm}{section}
\numberwithin{eg}{section}
\numberwithin{defn}{section}

\renewcommand{\P}{\mathbb P}
\newcommand{\Q}{\mathbb Q}
\newcommand{\Z}{\mathbb Z}
\newcommand{\N}{\mathbb N}
\newcommand{\C}{\mathbb C}
\newcommand{\R}{\mathbb R}
\newcommand{\A}{\mathbb A}

\DeclareMathOperator{\Aut}{Aut}
\DeclareMathOperator{\GL}{GL}
\DeclareMathOperator{\PGL}{PGL}
\DeclareMathOperator{\SL}{SL}
\DeclareMathOperator{\Pic}{Pic}

\DeclareMathOperator{\Ext}{Ext}
\DeclareMathOperator{\Sym}{Sym}
\DeclareMathOperator{\End}{End}
\DeclareMathOperator{\Hom}{Hom}

\DeclareMathOperator{\rank}{rank}
\DeclareMathOperator{\coker}{coker}

\DeclareMathOperator{\tr}{tr}
\DeclareMathOperator{\ad}{ad}

\DeclareMathOperator{\Spec}{Spec}
\DeclareMathOperator{\Proj}{Proj}
\DeclareMathOperator{\Quot}{Quot}
\DeclareMathOperator{\Tor}{Tor}
\DeclareMathOperator{\im}{im}

\DeclareMathOperator{\coh}{coh}
\DeclareMathOperator{\qcoh}{qcoh}
\DeclareMathOperator{\Br}{Br}
\DeclareMathOperator{\reg}{reg}
\newcommand{\sO}{\mathcal O}
\newcommand{\sK}{\mathcal K}
\newcommand{\cS}{\mathcal S}
\newcommand{\cN}{\mathcal N}
\newcommand{\oD}{\mathcal D}
\newcommand{\sExt}{\mathcal Ext}
\newcommand{\sHom}{\mathcal Hom}

\newcommand{\Spl}{\mathcal Spl}

\newcommand{\Mod}{\mathcal Mod}
\newcommand{\Vect}{\mathcal Vect}
\newcommand{\MerDiff}{\text{\it${\mathcal M}\!$er${\mathcal D}$iff}}
\newcommand{\EllDiff}{\text{\it${\mathcal E}\!$ll${\mathcal D}$iff}}
\newcommand{\Mer}{\text{\it ${\mathcal M}\!$er}}
\newcommand{\Skl}{\mathcal Skl}
\DeclareMathOperator{\mer}{Mer}

\DeclareMathOperator{\Coh}{Coh}
\DeclareMathOperator{\Hilb}{Hilb}

\newcommand{\dL}{{\bf L}}
\newcommand{\ratto}{\dashrightarrow}
\newcommand{\Gampq}{\Gamma_{\!\!p,q}}
\def\Gamm#1{\Gamma_{\!\! #1}}
\newcommand{\fD}{{\mathfrak D}}
\DeclareMathOperator{\Res}{Res}
\DeclareMathOperator{\ord}{ord}
\DeclareMathOperator{\Sol}{Sol}

\DeclareMathOperator{\ASL}{ASL}
\DeclareMathOperator{\AGL}{AGL}

\let\div\relax
\DeclareMathOperator{\div}{div}
\DeclareMathOperator{\id}{id}
\renewcommand{\_}{{\underline{\ \ }}}

\begin{document}

\title{The noncommutative geometry of elliptic difference equations}
  \author{
Eric M. Rains\\Department of Mathematics, California
  Institute of Technology}

\date{July 25, 2019}
\maketitle

\begin{abstract}
We give a new construction of noncommutative rational surfaces via elliptic
difference operators, which enables us to attach a 1-parameter
noncommutative deformation to any projective rational surface with smooth
anticanonical curve.  The result turns out to agree with a construction
implicit in work of Van den Bergh (as blowups of noncommutative Hirzebruch
surfaces), but the construction, by giving an alternate interpretation of
morphisms between line bundles, enables one to prove a number of new facts
about these surfaces.  In particular, we show that they are noncommutative
smooth proper surfaces in the sense of Chan and Nyman, with projective
$\Quot$ schemes, that moduli spaces of simple sheaves are Poisson and that
moduli spaces classifying semistable sheaves of rank $0$ or $1$ are
projective.  We further show that the action of $\SL_2(\Z)$ as derived
autoequivalences of rational elliptic surfaces extends to an action as
derived equivalences of surfaces in our family with $K^2=0$.

This construction also enables a number of new results in the theory of
special functions.  Since morphisms between line bundles on these
noncommutative surfaces are difference operators, this leads to an
interpretation of certain moduli spaces of sheaves on the surfaces as
moduli spaces of difference equations.  When the moduli space is a single
point, the equation is rigid, and we show that one can construct solutions
as generalized hypergeometric integrals.  More generally, twisting by line
bundles turns out to give rise to isomonodromy deformations, and thus such
moduli spaces naturally produce Lax pairs.  In particular, we show that
when the moduli space is $2$-dimensional, it gives rise to a Lax pair for
the elliptic Painlev\'e equation; using the derived equivalences, this
gives such a Lax pair for every rational number, of order twice the
denominator.  We also construct an elliptic analogue of the Riemann-Hilbert
correspondence, an analytic equivalence between categories of sheaves
corresponding to elliptic difference equations, swapping the role of the
shift of the equation and the nome of the curve.

\end{abstract}

\tableofcontents

\section{Introduction}

A particularly striking characterization of Gauss' hypergeometric function
${}_2F_1$ is as the solution of a second-order Fuchsian differential
equation with three singular points.  In fact, such equations are {\em
  rigid}, in the sense that if one specifies the local behavior of the
equation near the singular points, this is enough to determine the equation
up to isomorphism.  If one adds an additional singular point, one no longer
has rigidity, but now acquires a new structure in the form of a flow (the
Painlev\'e VI equation) that moves the fourth singular point without
changing the monodromy of the equation.  Both of these structures
generalize considerably; e.g., the ${}_nF_{n-1}$ functions satisfy rigid
$n$-th order equations, and there are flows analogous to Painlev\'e on
higher-dimensional moduli spaces of equations.  Moreover, there are
discrete analogues: the contiguity relations for ${}_nF_{n-1}$ can be
viewed as rigid {\em difference} equations, and similar things hold for the
$q$-hypergeometric ${}_n\phi_{n-1}$; there are also discrete and
$q$-analogues of Painlev\'e constructed as isomorphisms between moduli
spaces of difference or $q$-difference equations.

We thus see that much of the modern theory of special functions is closely
related to moduli spaces of differential or difference equations.  In
particular, we might hope that a better understanding of these moduli
spaces would lead to new insights into existing special functions, as well
as to new special functions and integrable systems.  Indeed, both the
hypergeometric equations and the Painlev\'e equations actually extend
beyond the $q$-difference case to the {\em elliptic} difference case, in
which multiplication by $q$ is replaced by translation on an elliptic
curve.  However, Sakai's construction of the elliptic Painlev\'e equation
in \cite{SakaiH:2001} was based on the geometry of rational surfaces rather
than on difference equations; though later work produced such
interpretations \cite{YamadaY:2009,isomonodromy}, it has been unclear how
to extend this to analogues of the Garnier equations (flows on second-order
Fuchsian equations with $n>4$ singular points) or other generalizations.

In \cite{rat_Hitchin}, the author gave a first pass at understanding such
moduli spaces, by giving a correspondence between various sorts of
difference and differential equations and sheaves on Hirzebruch surfaces.
However, though this behaves well on an open subset of the moduli space of
equations, it leads to an incorrect compactification.  For instance, one
can view a differential equation as a connection on a vector bundle, and
the correspondence in \cite{rat_Hitchin} behaves well only when the vector
bundle is trivial.  Now, differential equations with rational coefficients
can be interpreted as $D$-modules on $\P^1$, which we can equivalently view
as sheaves on a noncommutative deformation of $\A^1\times \P^1$.  This
suggests that to fix this issue, we should replace the Hirzebruch surfaces
with noncommutative deformations.  (This is further supported by the
construction in \cite{P2Painleve} of the elliptic Painlev\'e equation via
sheaves on a noncommutative projective plane.)

The natural hope would be that the homogeneous coordinate ring of such a
noncommutative Hirzebruch surface would have a faithful representation as
an algebra of difference operators.  This, however, encounters a
significant difficulty.  Although many early constructions of
noncommutative schemes were obtained from flat deformations of graded
algebras, there is in general {\em no} graded algebra associated to a given
noncommutative scheme.  The difficulty here is that although line bundles
can often be extended along a noncommutative deformation, they deform only
as left modules rather than as bimodules.  In practice, only multiples of
the canonical bundle extend as bimodules\footnote{In the cases we study
  here, this follows from Proposition \ref{prop:abelian_isomorphisms}
  below.}, and thus we only obtain a homogeneous coordinate ring when some
such multiple is very ample.  Thus in particular, we would only expect to
have a homogeneous coordinate ring associated to a deformation of the
Hirzebruch surfaces $F_0=\P^1\times \P^1$ and $F_1$, as these are the only
Fano Hirzebruch surfaces (i.e., such that the anticanonical bundle is
ample).  This issue becomes even worse when we realize that
\cite{rat_Hitchin} tells us that the most natural way to encode the
singularity data is to blow up our surface and consider ($1$-dimensional)
sheaves which are {\em disjoint} from the anticanonical curve, since the
existence of such sheaves ensures that the anticanonical divisor cannot
possibly be ample!

The standard way \cite{BondalAI/PolishchukAE:1993} to resolve this issue is
to suppose that there is some ample bundle $\sO(1)$ such that all powers
$\sO(k)$ extend to the noncommutative deformation.  In that case, we can
construct not an algebra but a category, in which the objects are integers,
and the $\Hom$ spaces are defined by
\[
\Hom(k,l) = \Hom(\sO(-l),\sO(-k)).
\]
Such a category with objects indexed by $\Z$ and all morphisms having
nonnegative degree is called a $\Z$-algebra, and one can indeed recover the
(category of sheaves on the) original noncommutative scheme as a quotient
of the category of modules over this $\Z$-algebra.  Of course, in addition
to the choice of original ample bundle and the requirement that the bundles
extend, there is also in principle a choice of such extension for each $l$.
These are both serious issues in general, though in our case it is
reasonable to hope that neither issue arises.  Indeed, it follows from the
deformation theory of abelian categories \cite{LowenW/VandenBerghM:2006}
that the infinitesimal obstructions to extending a line bundle lie in
$H^2(\sO_X)$, while the choice of extension is controlled by $H^1(\sO_X)$.
In our case, since we are deforming a rational surface, both cohomology
groups vanish, so neither problem arises for a formal deformation.  (There
is still a mild nonuniqueness, in that we can multiply each line bundle by
a unit, but this freedom is easy enough to deal with.)

Although this approach indeed works in our case, the requirement to choose
an ample bundle can make it difficult to detect when two $\Z$-algebras
induce isomorphic schemes (indeed, it can be difficult to determine when
two {\em commutative} graded algebras correspond to isomorphic projective
schemes).  As a result, it will be convenient to put off making such a
choice as long as possible, suggesting that we should deform {\em all}
ample line bundles.  In fact, since we cannot hope for the ample cone to
remain constant under deformation, the simplest thing to do is simply
deform all line bundles, ample or otherwise.  Given a smooth projective
rational surface $X$, the Picard group $\Pic(X)\cong \Z^m$ for some $m$ (to
be precise, $m=10-K_X^2$).  If for each $\lambda\in \Pic(X)$, we choose a
representative ${\cal L}_\lambda$ of that isomorphism class of line bundles,
then we obtain a category associated to $X$ as follows.  The objects of the
category are the vectors $\lambda\in \Pic(X)$, while for each
$\lambda,\mu$,
\[
\Hom(\lambda,\mu):=\Hom({\cal L}_{-\mu},{\cal L}_{-\lambda}).
\]
Given a deformation of $X$ such that all line bundles extend uniquely (up
to isomorphism), we obtain a corresponding deformation of this category.
This introduces some other issues (there are difficulties in defining the
corresponding category of sheaves), but has distinct advantages for present
purposes.  In particular, for a large class of projective rational surfaces
$X$ (those having a smooth anticanonical curve), we will be able to
construct a flat noncommutative deformation of this category.  Moreover,
the entire category will (by construction) have a faithful representation
in suitable difference operators.  (We will extend this in \cite{elldaha}
to an analogous family of {\em multivariate} difference operators, giving a
$2$-parameter deformation of $\Sym^n(X)$ for any projective rational
surface $X$ with a smooth anticanonical curve.)

Note that one cost of working with the full $\Pic(X)$ is that our
$\Z^m$-algebras are not (locally) finitely generated in general; indeed,
even in the commutative setting the full multihomogeneous coordinate ring
needs at least one generator for every $-1$-curve on the surface, and there
can easily be infinitely many such curves once $K^2\le 0$.  Thus we mostly
must eschew the traditional approach the to construction of deformations
via generators and relations.  Here we are helped greatly by the fact that
we are specifying a sub-$\Z^m$-algebra of a $\Z^m$-algebra of difference
operators, so that even in the finitely generated case we need not bother
with relations, and in general it suffices to specify subspaces respected
by composition.

This construction should be contrasted with one implicit in
\cite{VandenBerghM:1998,VandenBerghM:2012}: to construct a noncommutative
rational surface, one can simply construct a noncommutative Hirzebruch
surface (or noncommutative $\P^2$) and then repeatedly blow it up.  This
has some advantages over our approach, in that it is clear the resulting
categories of sheaves are noetherian and otherwise well-behaved, but
compensatory disadvantages, most significantly that it is difficult in this
approach to understand when different such constructions result in
isomorphic schemes.  For instance, the commutative surface $\P^1\times
\P^1$ can be viewed as a Hirzebruch surface in two different ways, but this
is difficult to see if we are given the surface {\em as} a Hirzebruch
surface, and only becomes more difficult in the noncommutative setting.  In
contrast, we will find that our deformations have an explicit symmetry
corresponding to this phenomenon, and in general will be isomorphic
whenever the commutative surfaces being deformed are isomorphic.  We will
in fact see below that, at least in the smooth anticanonical curve case,
both approaches are actually equivalent: the schemes associated to our
noncommutative rational surfaces can indeed be obtained via Van den Bergh's
constructions.  However, the corresponding representation in
difference operators is not only new, but useful in proving symmetries of
the constructions.

Though the construction \`a la Van den Bergh works for any rational surface
with a choice of Poisson structure (i.e., an infinitesimal noncommutative
deformation), we will be restricting our attention in the present work to
the case when the kernel of the Poisson structure is a smooth
(anticanonical) curve $C$.  Since we many of our arguments use this curve
to reduce to commutative questions, the smooth case is the simplest to deal
with.  Of course, this excludes cases corresponding to the vast majority of
the work in special functions, most of which corresponds to differential
equations, for which \cite{rat_Hitchin} tells us the anticanonical curve is
not even reduced.  We will thus return to that case in \cite{noncomm2},
replacing the present $\Z^m$-algebra approach by one via derived categories
based on insights gained below.

Given a smooth genus 1 curve $C$, the general rational surface $X$ with an
embedding $C\to X$ as an anticanonical curve can be constructed in the
following way from data on $C$, specifically line bundles $\eta,\eta'\in
\Pic^2(C)$ and points $x_1,\dots,x_m\in C$.  The line bundle $\eta$
determines a degree 2 map $\rho:C\to \P^1$, and thus $\rho_*\eta'$ is a
rank 2 vector bundle on this $\P^1$.  The corresponding Hirzebruch surface
$X_0$ contains $C$ in a natural way, and we in fact find that $C$ is an
anticanonical curve on $X_0$.  Blowing up the point $x_1\in C\subset X_0$
gives a surface $X_1$ which still contains $C$ as an anticanonical curve,
so that we may iterate to obtain a surface $X=X_m$.  In this way, we obtain
every rational surface containing $C$ anticanonically, except for $F_1$ and
$\P^2$, which we may obtain by blowing down suitable $-1$-curves on $X_1$.
Note that this construction comes with a natural choice of basis for
$\Pic(X)$: we have classes $f$, $s$, $e_1$,\dots,$e_m$, such that $f$ is
the class of a fiber of $\rho$, $s$ is the
relative $\sO(1)$, and $e_i$ is the exceptional curve of the $i$-th blowup.
The intersection form in this basis is
\[
s\cdot f=1,\quad e_i^2 = -1,
\]
with all other intersections 0.  There are two other bases of interest.  If
we start with $F_1$, we have a similar basis, but now the minimal section
is a $-1$-curve so $s^2=-1$; starting with $\P^2$ gives a basis
$h,e_0,e_1,\dot,s,e_m$ with
\[
h^2 = 1,\quad e_i^2 = -1.
\]
Note that since a Poisson structure on $X$ vanishing on $C$ depends on a
choice of differential on $C$, the infinitesimal noncommutative
deformations of $X$ are parametrized by the Lie algebra of $\Pic^0(C)$;
this suggests that the actual deformations should be parametrized by a
point $q\in \Pic^0(C)$.  (In this, we may also be guided by the fact that
the standard ``Sklyanin'' noncommutative $\P^2$ is parametrized in
precisely that way.)

Thus for each $C$ and integer $m\ge 0$, we wish to construct a category
with objects in $\Z^{m+2}$ depending on the above data together with a
point $q\in \Pic^0(C)$.  As we vary the parameters (including $C$ itself),
this deformation should be flat in a suitable sense; we will in fact be
able to arrange for {\em every} $\Hom$ space to be a flat sheaf on
parameter space.  When $q$ becomes trivial, we want to recover the category
associated to the original surface $X_m$.  The cost of flatness is that
this will only literally be true when $X_m$ has no $-2$-curves, but any
$\Hom$ space associated to an {\em ample} divisor will have the correct
limit as $q$ becomes trivial, and this is all we really need.  (To be
precise, the $\Hom$ spaces in the commutative limit are the direct images
of corresponding line bundles on the universal surface, shown to be flat in
\cite{rat_Hitchin}; thus any $\Hom$ space corresponding to an acyclic line
bundle will have the correct limit, while in general we will obtain only a
subspace of the true $\Hom$ space.  This is the analogue of working with a
model for a commutative variety which is not projectively normal.)

As noted in \cite{rat_Hitchin}, the surface does not uniquely determine the
data $C$, $\eta$, $\eta'$, $x_1$, \dots, $x_m$.  In addition to the obvious
freedom to replace $C$ by an isomorphic curve (which will only require an
easy functoriality argument on our part), there is also an action of an in
general infinite Coxeter group $W(E_{m+1})$ on the data that preserves the
isomorphism class of $X_m$.  This group is generated by the following
reflections: (1) For each $1\le i<m$, we may swap $x_i$ and $x_{i+1}$, (2)
We may replace $\eta'$, $x_1$, $x_2$ by $\eta'+\eta-x_1-x_2$, $\eta-x_2$,
$\eta-x_1$, and (3) we may swap $\eta$ and $\eta'$.  Indeed, (1) if $x_i\ne
x_{i+1}$, then the locality of blowing up implies that the order of blowing
the two points up is irrelevant, (2) if we blow up $x_1$ and $x_2$ and blow
down the corresponding fibers, we obtain a new, generically isomorphic,
Hirzebruch surface (with the given parameters), and (3) if $\eta\ne \eta'$,
then $X_0\cong \P^1\times \P^1$, and swapping $\eta$ and $\eta'$ simply
switches to the other interpretation of $X_0$ as a Hirzebruch surface.
(For the somewhat harder fact that these are the only isomorphisms, see
\cite{rat_Hitchin} or Proposition \ref{prop:abelian_isomorphisms} below.)
Luckily, our flat deformation turns out to be invariant under the action of
this Coxeter group.  More precisely, the Coxeter group acts on the basis of
$\Pic(X)$ as well as on the parameters, and the combined action leaves the
$\Hom$ spaces invariant.  The results of \cite{rat_Hitchin} suggest that
the reflections of type (1) will have no effect beyond bookkeeping, while
the reflection of type (2) will essentially be just conjugation by the
solution of a certain first-order difference equation (more precisely,
gauging by a scalar $1$-cocycle).  The main difficulty will thus be to
understand how the reflections of type (3) act.  Luckily, just as in the
commutative setting, we will only need to understand the analogue of
$\P^1\times \P^1$ to make this work.

To construct the deformations of $\P^1\times \P^1$, we can be guided by the
fact mentioned above that certain kinds of elliptic difference equations
should correspond to sheaves on such deformations; this will indeed give
enough guidance to allow us to construct the deformed category in this
setting.  Moreover, since we have a fair understanding from
\cite{rat_Hitchin} of how elementary transformations should appear in the
difference operator interpretation, we can use this to determine what
blowing up a point should do, and thus (with some additional work) how the
general $\Z^{m+2}$-algebra in our family should be defined.

At this point, we must now contend with what we gave up when working with
$\Z^{m+2}$-algebras.  The problem here is that a $\Z^{m+2}$-algebra is the
analogue of a multigraded algebra, and even in the commutative setting,
there is no natural way to turn a multigraded algebra into a scheme!  If
one starts with a scheme, then one certainly has a functor from the
category of modules over the multihomogeneous coordinate ring into the
category of quasicoherent sheaves: restrict to the graded subalgebra
corresponding to an ample divisor class, and then translate to a sheaf in
the usual way.  The issue is thus to determine the kernel of this functor:
i.e., which modules over the $\Z^{m+2}$-algebra correspond to the trivial
quasicoherent sheaf?  The point now is that different choices of ample
bundle in principle give rise to different ``torsion'' subcategories.
Luckily, we can give a uniform definition of torsion elements in modules
over our $\Z^{m+2}$, and show that this agrees with the notion coming from
any choice of ample divisor class.  Moreover, one can recover the
construction \`a la Van den Bergh via a particular inductive definition of
``torsion'', which again agrees with our uniform description.

We thus obtain a flat family of noncommutative schemes (more precisely, the
corresponding categories of (quasi)coherent sheaves) associated to our
construction.  Moreover, by using both what we know from the
$\Z^{m+2}$-algebra construction and what Van den Bergh tells us about
blowups, we are able to show that these schemes are ``noncommutative smooth
proper surfaces'' in the sense of \cite{ChanD/NymanA:2013}.  In particular,
there is an analogue of Serre duality, in which twisting by the canonical
bundle is replaced by the application of a certain invertible endofunctor.
(Moreover, our surfaces all have an ``anticanonical'' natural
transformation, from this canonical functor to the identity.)  We are also
able to characterize the analogues of the effective and nef cones, and
show that a divisor is ample iff it is in the interior of the nef cone.
In addition, we prove a stronger version of the ``halal Hilbert scheme''
condition: $\Quot$ schemes on our surfaces are not merely countable unions
of projective schemes, but are in fact projective.

Now that we have well-behaved noncommutative schemes, the next step is to
investigate the corresponding moduli spaces of sheaves.  In the commutative
setting, the presence of an anticanonical curve induces Poisson structures
\cite{TyurinAN:1988,BottacinF:1995,HurtubiseJC/MarkmanE:2002b} on the
moduli spaces, and \cite{poisson} moreover showed that a number of natural
morphisms between such moduli spaces preserved the Poisson structure.
Since our surfaces have an anticanonical natural transformation, the
construction in \cite{TyurinAN:1988} of a pairing on the cotangent sheaf
carries over directly, and it remains only for us to show that it is
alternating and satisfies the Jacobi identity.  We in fact find that the
moduli space of ``simple'' coherent sheaves (i.e., with no nontrivial
endomorphisms) is indeed Poisson in this way, with symplectic leaves
determined by how the sheaves meet the anticanonical curve.  We already
know from \cite{rat_Hitchin} that there are open subsets of the moduli
space that are independent of $q$, and we can show that this not only
extends to appropriate torsion-free sheaves, but that the correspondence
preserves the Poisson structure.  The key insight there is that the
correspondence actually works at the level of derived categories: the
kernel of the derived global sections functor is actually independent of
$q$.

One disadvantage of the moduli space of simple sheaves is that it is only
an algebraic space.  We would thus like to replace it in general by the
moduli space of stable sheaves (which would inherit the Poisson structure)
or semistable sheaves.  Although we are able to show that these are
moderately well-behaved (separated, resp. proper) schemes, there are
difficulties in showing projectivity in general.  There are two special
cases in which we can control things, however.  The first is the analogue
of the Hilbert scheme of points on $X$.  Following
\cite{NevinsTA/StaffordJT:2007}, this is not simply the $\Quot$ scheme of
$0$-dimensional quotients of $\sO_X$, but rather the moduli space of
sheaves with the same numerical invariants as an ideal sheaf of $n$ points.
We show that in general, this is a smooth, irreducible, rational,
projective variety of dimension $2n$, with a natural Poisson structure.
Two low-dimensional cases are of particular interest: the $n=0$ case leads
to a numerical characterization of line bundles, while the $n=1$ case gives
us a smooth Poisson rational surface associated to any of our
noncommutative rational surfaces.  In particular, we find that not only
does our family of noncommutative rational surfaces have one more parameter
than the family of commutative rational surfaces, but it is in fact in a
natural way a $\Pic^0(C)$-bundle over that family.  For general $n$, this
construction gives a new deformation of the Hilbert scheme of a rational
surface, extending the construction of \cite{NevinsTA/StaffordJT:2007} for
$\P^2$.

The other case in which we have sufficient boundedness to control the
moduli space is the case of $1$-dimensional sheaves; this is particularly
important to us since the moduli spaces of difference equations are open
subschemes.  Here we are again able to prove sufficient boundedness and
inequalities on global sections to show that the moduli space of semistable
sheaves is projective.  If $\chi(M)=0$, we find as expected from
\cite{rat_Hitchin} that one component of the moduli space is birational to
a commutative moduli space, while if $\chi(M)=\pm 1$, we obtain a
birational map between the moduli space of $1$-dimensional sheaves and an
appropriate Hilbert scheme.  One interesting consequence of this is a
description of certain special cases of the isomonodromy transformations in
terms of arithmetic on moving Jacobians as in
\cite{KajiwaraK/MasudaT/NoumiM/OhtaY/YamadaY:2006,P2Painleve}.

When the moduli space is $2$-dimensional and there is a universal family,
we can say much more: the moduli space is a rational surface with $K^2=0$
and a smooth anticanonical curve.  In particular, any such case gives rise
to a Lax pair for an elliptic Painlev\'e equation.  We moreover have enough
control over the resulting surface to be able to say precisely which
surface arises, and thus can describe the full family of Lax pairs arising
this way for a given instance of the elliptic Painlev\'e equation.  We find
as a result that there is such a Lax pair associated to any rational
number (which in a sense includes the case of $\infty$, corresponding to
Sakai's original construction).

The key idea to control the moduli space in the $2$-dimensional case is
that when the moduli space is $2$-dimensional, the universal family gives
rise to a derived functor \`a la Fourier-Mukai, and in fact that derived
functor is an equivalence.  It turns out that any such derived equivalence
respects the canonical functor and the anticanonical natural
transformation, and thus respects restriction to the anticanonical curve;
this makes it relatively straightforward to read off the parameters.  Of
course, by composing these derived equivalences with commutative surfaces,
we also obtain a number of derived equivalences between noncommutative
surfaces.  We find in this way that there is an action of $\SL_2(\Z)$ on
the parameters in such a way that any two surfaces in the same orbit are
derived equivalent whenever one of the surfaces in the orbit is
commutative.

In fact, it turns out that that last condition is unnecessary: we can
construct the desired equivalences using known semiorthogonal
decompositions and thus obtain an action of an extension of $\SL_2(\Z)$ as
derived equivalences in complete generality.  In other words, the action of
$\SL_2(\Z)$ on the derived category of a (Jacobian) elliptic surface
extends to general noncommutative rational surfaces with $K^2=0$.  We are
also able to understand derived equivalences more generally; if $K^2>0$, or
if a certain root system is finite, then we can completely characterize the
derived equivalences between the surfaces we have constructed.  In
addition, the same ideas that let us control derived equivalences are even
more powerful when applied to abelian equivalences, allowing us to
completely control those.

Although we used the notional association with difference equations to
motivate the construction via difference operators and various of the
results above, most of the results discussed above have little to do with
that interpretation.  There are several resulting loose ends in connecting
things back to equations, the most important of which being the relation
between the various Poisson maps between moduli spaces and ``isomonodromy''
transformations.  From one perspective, this is fairly straightforward: we
develop enough machinery below to let the computations of such
transformations in \cite{rat_Hitchin} to carry over {\em mutatis mutandis};
in particular, we find that the guesses made there for how to incorporate
$q$ in the actions of those transformations are indeed correct descriptions
of the corresponding actions from noncommutative geometry.  This is still
essentially algebraic, apart from one issue, namely what it means for such
a transformation to preserve ``monodromy'' in the elliptic setting.

One answer comes from \cite{isomonodromy}, which shows that given a
$q$-difference equation with coefficients meromorphic on $\C^*/\langle
p\rangle$, there is a natural way to associate an equivalence class of
$p$-difference equations on $\C^*/\langle q\rangle$, which one can view as
the monodromy of the equation.  The noncommutative geometry perspective
suggests an improvement: we are able to show that there is a version of
this construction that gives an equivalence between the corresponding
categories of coherent sheaves disjoint from $C$.  (Curiously, unlike the
differential case, there is no need to introduce perverse sheaves to make
this an abelian equivalence.)  In addition, this ``elliptic Riemann-Hilbert
functor'' behaves well in holomorphic families, and thus induces
biholomorphic maps between the relevant moduli spaces.  The construction of
this functor is based on a general construction of a sheaf of solutions
associated to any $q$-difference equation with meromorphic coefficients,
which may be of independent interest.

Moreover, the elliptic Riemann-Hilbert functor commutes with the various
Cremona transformations, while twisting by a line bundle has no effect on
the image of the functor (so preserves ``monodromy'' as required).  The
commutation with Cremona transformations is mostly straightforward, but
there is one tricky case.  The fact that $\P^1\times \P^1$ is a Hirzebruch
surface in two different ways means that a sheaf on $\P^1\times \P^1$ has
{\em two} interpretations as a difference equation.  In terms of difference
operators, this corresponds to a sort of formal Fourier transform, swapping
certain multiplication operators and difference operators.  To show that
this operation respects the elliptic Riemann-Hilbert functor, we need to
remove the word ``formal'', and instantiate the operation as an actual
integral transform.  This leads to a couple of interesting features: one is
that we obtain an action of a certain large Coxeter group via integral
transformations, and can recover several known elliptic hypergeometric
identities as relations between different reduced words.  The other is that
the geometry tells us that any rigid sheaf corresponds to a $-2$-curve,
which is thus related by a sequence of Cremona transformations to the
$-2$-curve on the Hirzebruch surface $F_2$.  Reversing these operations and
using the analytic interpretation, we find that we can construct a solution
to any rigid equation via a sequence of integral transformations.  The case
with only one integration step corresponds to the higher-order elliptic
beta integral \cite{SpiridonovVP:2001,dets}, showing that these indeed
satisfy rigid equations.

The plan of the paper is as follows.  In Section 2, using results of
\cite{rat_Hitchin} as motivation, we construct the $\Z^2$-algebras
corresponding to our noncommutative deformations of $\P^1\times \P^1$;
Section 3 then shows that this is a flat deformation (and constructs
important symmetries), and Section 4 shows that there is still a
well-behaved category of coherent sheaves.  Section 4 also considers the
``missing'' operators: in order to have a flat deformation, we had to omit
some of the morphisms between line bundles, but it turns out that the
omitted morphisms still correspond to difference operators.

Section 5 gives the analogous construction for the Hirzebruch surface
$F_1$.  The presence of a $-1$-curve on this surface leads to two important
features: one is that the orthogonal $\Z$-subalgebra is a noncommutative
deformation of $\P^2$, while the other is that there is a {\em natural}
first-order difference equation associated to every object in the category.
If we gauge by solutions of those equations, we obtain an equivalent
category in which now every operator that appears has elliptic function (as
opposed to meromorphic theta function) coefficients.  Moreover, this allows
us to describe the generators in a purely algebraic way, which is key to
making our later constructions work for general genus 1 curves.

Section 6 constructs surfaces with $K^2=7$.  The idea here is that such a
surface is a blowup of both kinds of Hirzebruch surface, and thus to
understand such surfaces one must merely understand elementary
transformations.  This leads to a description of the $\Hom$ spaces in
the $\Z^3$-algebra as intersections of $\Hom$ spaces in the $F_0$ and $F_1$
algebras, mediated by a suitable gauge transformation; we can also give a
description via generators that lets us extend the algebraic construction.

Section 7 finally extends the construction to general $K^2<7$.  Here the
main difficulty is that we have been unable to give an explicit description
of the relevant $\Hom$ spaces: there are difficulties when one blows up the
same point multiple times.  What saves the construction is that we are
hoping for a {\em flat} deformation, which it turns out allows us to
reconstruct the deformation from the generic fiber, where the $\Hom$ spaces
are easy to describe.  Moreover, the generic fiber has all the symmetries
we are hoping for, making it easy to extend the symmetries to the full
family, and letting us reduce the proof of flatness to the fundamental
chamber for the appropriate Coxeter group, where it is reasonably
straightforward.

Section 8 is a mainly technical section regarding when the multiplication
maps in the $\Z^{m+2}$-algebra are surjective, in preparation for showing
that the various notions of ``torsion'' module agree.  This is shown in
Section 9, which in particular shows that our construction agrees with the
construction \`a la Van den Bergh, as well as satisfying symmetries under
the appropriate Coxeter group of Cremona transformations.  Section 9 also
gives an algorithm for computing $\Hom$ spaces in the saturated
$\Z^{m+2}$-algebra (i.e., maps between line bundles on the noncommutative
scheme).  Section 10 then shows that the resulting noncommutative schemes
satisfy the axioms of \cite{ChanD/NymanA:2013}.  Finally, Sections 11-13
discuss moduli spaces, derived equivalences, and analytic issues, as
described above.

There are also two appendices: Although we mostly avoid working with
presentations in favor of simply specifying generators inside a large
ambient algebra of difference operators, there are some low-order cases in
which there are particularly nice generators and relations.  Indeed, for
both $\P^1\times \P^1$ and two-point blowups thereof, not only can we give
an explicit presentation for generic parameters, but that presentation
actually takes the form of a Gr\"obner basis.  (This is in sharp contrast
to the elliptic noncommutative $\P^2$, for instance.)  Appendix A considers
the general question of when relations of the given form for $\P^1\times
\P^1$ satisfy the Gr\"obner basis property, and shows that this leads to a
nice integrable system on the $\P^{15}$ parametrizing the relations.
Appendix B considers the $K^2=6$ case, and shows that it extends to a
similar deformation of the blowup of $\P^n$ in the $n+1$ coordinate
vectors.  This leads to a noncommutative deformation of $(\P^1)^n$ for
general $n$, and a number of results essentially stating that
noncommutative surfaces of the kind we consider can be embedded as
subschemes of such blowups.

{\bf Further directions}

We intend to extend most of the above results in future work.  If one
allows the anticanonical curve to degenerate, then the strong flatness
result of Section 7 already fails to hold for commutative surfaces.  (More
precisely, this happens whenever the anticanonical curve has a smooth
rational component of self-intersection $<-2$, see \cite{rat_Hitchin}.)  As
a result, the present approach would become considerably more complicated,
but the resulting insights are still applicable, and indeed we will be able
to extend nearly all of the results of Sections 9--12 to the general case,
as well as (to a lesser extent) to more general noncommutative surfaces
\cite{noncomm2}.

This is not to say that there is no generalization of the flatness result.
Many of the various explicit difference operators appearing below have
natural multivariate analogues, and it turns out that this extends to give
a multivariate analogue of the family of $\Z^{m+2}$-algebras constructed
below, satisfying the same Coxeter group symmetry and a slightly weaker
version of flatness.  These considerations will also lead to an alternate
construction of the present $\Z^{m+2}$-algebras as ``spherical algebras''
of an ``elliptic double affine Hecke algebra'' constructed in
\cite{elldaha}.

Another class of open problems involves morphisms between our
noncommutative surfaces.  Although we have fully determined the {\em
  isomorphisms} between such surfaces, we have little understanding of
higher-degree morphisms.  For instance, if one of our surfaces is equipped
with a faithful action of a finite group, then presumably the quotient by
that action will not only exist, but be another noncommutative rational
surface.  Moreover, any morphism of noncommutative rational surfaces would
give rise to a corresponding pair of functors acting on difference
equations, some of which will presumably have natural interpretations.  We
can also consider sheaves which are fixed by the action of an automorphism;
examples of this include difference equations in which the shift matrix is
in the orthogonal or symplectic group, leading to the question of what the
noncommutative framework might tell us about cocycles in general semisimple
groups over the field of elliptic functions.

There are also a number of open questions suggested by the results of
Appendix B.  Indeed, the blowups of $\P^n$ constructed there presumably
give rise to smooth proper noncommutative $n$-folds in the sense of
\cite{ChanD/NymanA:2013}, and similarly for the $(\P^1)^n$ constructed
there.  Moreover, the latter can presumably be viewed as a $\P^1$-bundle
over $(\P^1)^{n-1}$, suggesting that the construction of
\cite{VandenBerghM:2012} can be extended to $\P^1$-bundles over
noncommutative schemes.  The embeddings of noncommutative Hirzebruch
surfaces in noncommutative $\P^n$s also raises the question of whether
other very ample divisor classes on noncommutative rational surfaces give
rise to such embeddings (for different deformations of $\P^n$).  A possibly
related question is whether there is an analogue of Castelnuovo-Mumford
regularity for general very ample divisor classes.

\bigskip

{\bf Notation}

Although for most of the paper, we will be working with general (algebraic)
genus 1 curves, it will be convenient towards the beginning and necessary
in Section 13 to work with analytic curves.  It is not only traditional in
the theory of elliptic special functions but convenient when developing the
theory of sheaves of solutions to work with such curves in {\em
  multiplicative} notation.  That is, rather than view a typical analytic
elliptic curve as $\C/\Lambda$ for some lattice $\Lambda$, we instead
exponentiate away one of the periods to obtain an expression as
$\C^*/\langle p\rangle$ for $|p|<1$.\footnote{Again, since $q$ is
  traditionally used to denote the shift in the difference equation, it is
  traditional in work on elliptic special functions to use $p$ to denote
  the nome of the elliptic curve.}  This makes modular transformations
relatively inaccessible, but those are of only minor interest below.  One
perhaps somewhat odd effect this has on the notation below is that we will
continue to use multiplicative notation for the group law of algebraic
elliptic curves, and will similarly express divisors in multiplicative
notation.  Thus if $x\in C$, then $[x]^l$ denotes the divisor with order
$l$ at $x$ and $0$ elsewhere, and similarly for products and ratios.  (This
is not as unnatural as it may appear to those coming from an algebraic
perspective, as it means that if $f$ and $g$ are two elliptic functions,
then $\div(fg)=\div(f)\div(g)$.)

As usual in the analytic setting, to specify an elliptic function, it is
simplest to express it as a ratio of products of quasiperiodic functions.
The primary such function in the multiplicative notation is the function
\[
\theta_p(z):= \prod_{0\le i} (1-p^i z)(1-p^{i+1}/z),
\]
defined for $|p|<1$, which satisfies the symmetries
$\theta_p(1/z)=-z^{-1}\theta_p(z)$ and $\theta_p(p/z)=\theta_p(z)$ as well
as the quasiperiodicity
\[
\theta_p(pz) = -z^{-1}\theta_p(z).
\]
More generally, a ``$p$-theta function'' is a holomorphic function $f$
satisfying a quasiperiodicity relation of the form $f(pz) = A z^{-k} f(z)$.
A particularly important special case is that of symmetric theta functions:
a ``$BC_1(\eta)$-symmetric $p$-theta function of degree $d$'' is a
holomorphic function $f$ such that $f(\eta/z)=f(z)$ and $f(pz) =
(\eta/pz^2)f(z)$.  (By standard convention a theta function is always
holomorphic unless specifically stated otherwise; thus a meromorphic
function satisfying the above conditions would be called a ``meromorphic
($BC_1(\eta)$-symmetric) theta function''.)

Another important function is the ``elliptic Gamma function''
\cite{RuijsenaarsSNM:1997}:
\[
\Gampq(z):=\prod_{0\le j,k} \frac{1-p^{j+1}q^{k+1}/z}{1-p^j q^k z},
\]
defined for $|p|$,$|q|<1$, which satisfies the difference equations
\begin{align}
\Gampq(qz) &= \theta_p(z)\Gampq(z),\notag\\
\Gampq(pz) &= \theta_q(z)\Gampq(z),\notag
\end{align}
as well as the reflection principle $\Gampq(pq/z)=\Gampq(z)$.

We also define
\[
\theta_p(z;q)_k
:=
\prod_{0\le j<k} \theta_p(q^j z),
\]
defined for $|p|<1$, $q\in \C^*$, $k\in \Z$, where by convention we
interpret such products for negative $k$ by taking
\[
\prod_{0\le j<k} f_j  = \prod_{k\le j<0} f_j^{-1}.
\]
Note that if $|q|<1$, then
\[
\theta_p(z;q)_k = \frac{\Gampq(q^k z)}{\Gampq(z)}.
\]
We will mainly be using it in this form; i.e., the extension to general
$q\in \C^*$ of the cocycle corresponding to $\Gampq$.

Finally, by convention, we take the natural numbers $\N$ to include $0$.

\bigskip

{\bf Acknowledgements}.  The author would like to thank A. Borodin,
P. Etingof, T. Graber, A. Okounkov, M. Van den Bergh, and X. Zhu for
helpful conversations.  This work was partially supported by grants from
the National Science Foundation, DMS-1001645 and DMS-1500806.

\section{Difference equations as modules}

An (analytic) {\em elliptic difference equation} is a formal equation of
the form
\[
v(qz) = A(z)v(z)
\]
where $A(z)$ is a matrix of meromorphic theta functions, i.e., $A(pz) =
\alpha z^k A(z)$ for some $\alpha\in \C^*$, $k\in \Z$, and $|q|<1$.  While
general equations of this type can be related to sheaves on noncommutative
surfaces (specifically noncommutative $\P^1$-bundles over elliptic curves
\cite{noncomm2}), for present purposes, we are only interested in a special
case.  Indeed, the elliptic difference equations arising in the theory of
elliptic special functions are invariably {\em symmetric}, in the sense
that $A(1/qz)=A(z)^{-1}$ (note that this forces $k=0$).  The reason for
this is quite simple: for the most part, the solutions we have in mind for
those equations satisfy an additional condition, namely that $v(1/z)=v(z)$.
The symmetry condition is merely the obvious compatibility condition,
extending the difference equation to a $1$-cocycle for the infinite
dihedral group.  (Indeed, it follows from results of \cite{PraagmanC:1986}
that the symmetry condition is precisely that needed to obtain a
fundamental matrix of symmetric meromorphic solutions to the difference
equation, see also the discussion in Section \ref{sec:diffeq2_sheaves} below.)

As noted in \cite{rat_Hitchin}, a symmetric elliptic difference equation
admits a factorization: by the matrix version of Hilbert's Theorem 90, $A$
can always be written in the form
\[
A(z) = B(1/qz)^{-t} B(z)^t,
\]
where $B(pz)\propto B(z)$.  This factorization is nonunique for two
reasons.  The first is that if $B(pz)=\beta z^l B(z)$, then $A(pz) =
\beta^2 (pq)^{-l} A(z)$, so that there are multiple choices for which line
bundle the coefficients of $B$ are sections of.  Of course, when the
coefficients of $A$ are actually elliptic, this is no issue, and in any
case, it turns out that there is usually an obvious choice of multiplier
for $B$.  The more serious issue is that we can right multiply $B$ by any
meromorphic theta function matrix $C$ such that $C(1/qz)=C(z)$; this, as
noted in \cite{rat_Hitchin} can be fixed by replacing $B$ by a holomorphic map
of vector bundles, and insisting that its domain be maximal.  To be
precise, $B$ should be a morphism
\[
B:\pi_1^*V\to \pi_0^*W,
\]
where $V$ and $W$ are vector bundles on $\P^1$, and $\pi_i:\C/\langle
p\rangle\to \P^1$ is the quotient by $z\mapsto 1/q^iz$.  (Note that the
matrices $C$ we may multiply by are themselves pulled back through
$\pi_1$.)  Since vector bundles on $\P^1$ are sums of line bundles, we can
replace this by the requirement that $B$ have meromorphic coefficients with
\[
B_{ij}(pz) = \beta z^l (1/pz^2)^{d_{1i}} (1/pqz^2)^{-d_{2j}} B_{ij}(z).
\]
In the canonical factorization of $A$, $d_{1i}\equiv 0$, and $\sum_j
d_{2j}$ is as large as possible.  The result is then unique up to
automorphisms of the domain bundle $V$.

We thus arrive at equations of the form
\[
B(z)^t v(z) - B(1/qz)^t v(qz) = 0
\]
and want to interpret these as modules over a suitable algebra.  Of course,
as we have already noted, we cannot expect this to work, so instead should
obtain modules over a suitable $\Z^2$-algebra.  The coefficients of the
above matrix equation are sums of operators of the form
\[
v_i(z)\mapsto B_{ij}(z) v_i(z) - B_{ij}(1/qz)v_i(qz)
\]
Of course, these operators are only really determined up to scalar
multiplication.  We can be guided here by the fact that since we want to
consider symmetric meromorphic solutions, our operators need to act nicely
on the space of symmetric meromorphic functions.  If $v_i(z)=v_i(1/z)$,
then
\[
B_{ij}(z) v_i(z) - B_{ij}(1/qz)v_i(qz)
=
B_{ij}(z) v_i(z) - B_{ij}(1/qz)v_i(1/qz)
\]
is antisymmetric under $z\mapsto 1/qz$.  We should thus at the very least
multiply by an antisymmetric function so that the operator preserves
symmetry.  When $v_i$ is holomorphic near a point, the output of the above
operator will be holomorphic, and will vanish at fixed points of $z\mapsto
1/qz$.  This suggests instead considering the operator
\[
v_i(z)\mapsto
\frac{B_{ij}(z)v_i(z)-B_{ij}(1/qz)v_i(qz)}{z^{-1}\theta_p(qz^2)}.
\]
This operator takes meromorphic functions invariant under $z\mapsto 1/z$ to
meromorphic functions invariant under $z\mapsto 1/qz$.  At this point, we
could if desired choose a square root of $q$, and use it to shift the
output to make it invariant under $z\mapsto 1/z$.  This is somewhat clumsy,
and in particular introduces the question of whether the result would
depend on the choice of square root.  Since we are looking for a category
in any event, we may as well simply allow the symmetry involution to vary
with the object in the category.  Once we do this, there is nothing
particularly special about the involution $z\mapsto 1/z$, so we may as well
replace it by $z\mapsto q\eta/z$.  (The appearance of $q$ here is to
simplify later formulas.)

We thus end up with operators of the form
\[
\oD_{\eta;q;p}(b):=\frac{z}{\theta_p(z^2/\eta)} (b(z)-b(\eta/z)T)
\]
with $b(z)$ a (holomorphic) theta function, and $T$ the operator
$Tf(z)=f(qz)$.  Note that left-multiplication by symmetric theta functions
changes the multiplier of $b(z)$ by a power of $\eta/pz^2$.  These
operators take functions invariant under $z\mapsto q\eta/z$ to functions
invariant under $z\mapsto \eta/z$.  For the composition of such operators
to make sense, we need two things.  First, the symmetry conditions must
match: we must only consider compositions
$\oD_{\eta/q,q}(b_1)\oD_{\eta,q}(b_2)$.  Second, in such a composition, the
coefficient of $T$ is a sum of two terms, and we will naturally want both
terms to be sections of the same line bundle.  If $b_1(pz)=\beta_1 z^{l_1}
b_1(z)$, $b_2(pz)=\beta_2 z^{l_2} b_2(z)$, then this condition forces
\[
\frac{\beta_1^2 (\eta/pq)^{l_1}}
     {q^4 \beta_2^2 (q\eta/p)^{l_2}}=1.
\]
This expresses $\beta_1/\beta_2$ as a square root; the choice of square
root is suggested by the fact that the coefficient of $T$ is invariant
under $z\mapsto \eta/qz$, and both terms have a denominator factor
$z^{-1}\theta_p(qz^2/\eta)$ which we would prefer to cancel.  This factor
introduces four poles in the fundamental annulus; two of those are
automatically cancelled, but the poles at $\pm \sqrt{\eta/pq}$ give
conditions
\[
\frac{\beta_1 (\pm \sqrt{\eta/pq})^{l_1-l_2}}
     {q^{2+l_2}\beta_2}
=
1.
\]
In particular, we are also forced here to have $l_1-l_2$ even, lest the two
conditions force opposite signs.  This will of course split our
construction into two cases, corresponding to even and odd Hirzebruch
surfaces.  

In general, we will want a $\Z^2$-algebra in which the objects are linear
combinations of symbols $s$ and $f$, in such a way that $f$ records the
extent to which we have multiplied by symmetric theta functions, and $s$
records the degree of the difference operator.  We will in particular want
a category in which $\Hom(ds+d'_1f,(d+1)s+d'_2f)$ consists of operators
$\oD_{\eta;q;p}(b)$ with $b$ a theta function with suitable multiplier; if
we insist that $\Hom(ds+d'_1f,ds+d'_2f)$ consist of multiplication by
suitably symmetric theta functions, then this together with the above
consistency condition will determine all of the multipliers once we have
chosen the multiplier corresponding to $\Hom(0,s)$.  Moreover, there will
always be some $k$ such that the line bundle corresponding to
$\Hom(0,s+kf)$ has degree 1 or 2, and we can perform a change of basis in
the group of objects so that $k=0$.  The degree 1 case will correspond to
odd Hirzebruch surfaces, while the degree 2 case will correspond to even
Hirzebruch surfaces.  We will focus on the latter case for the moment, as
it has a very important additional symmetry.

In the above considerations, we needed to assume $|q|<1$, or more precisely
that $q$ is not a root of unity, so that the operators act faithfully on
meromorphic functions.  However, we may ignore that representation, and
simply view $T$ as a formal symbol satisfying the {\em operator} equation
$T f(z) = f(qz) T$, where we view $f(z)$ as the operator of multiplication
by a meromorphic function.  More generally, we obtain an algebra of
meromorphic $q$-difference operators, which are polynomials
\[
\sum_{0\le k\le d} c_k(z) T^k
\]
with each $c_k(z)$ a meromorphic function, and with multiplication
given by
\[
(\sum_{0\le k\le d} c_k(z)T^k)
(\sum_{0\le k\le d'} c'_k(z)T^k)
=
\sum_{0\le l\le d+d'}
\sum_{\max(0,l-d')\le k\le \min(l,d)}
c_k(z) c'_{l-k}(q^k z) T^l.
\]

We may now define a $\C$-linear category $\cS_{\eta,\eta';q;p}$ as
follows.  First, define a $\Z s+\Z f$-algebra $\MerDiff_q$ in which each
$\Hom$ space is the algebra of meromorphic $q$-difference operators, with the
obvious composition.  Then $\cS_{\eta,\eta';q;p}$ is the smallest
subcategory of $\MerDiff_q$ having the same objects, and containing the
following morphisms:
\begin{itemize}
\item
If $g(z)$ is a $BC_1(q^{1-d}\eta)$-symmetric theta function of degree
$1$, then 
\[
g(z)\in \cS_{\eta,\eta';q;p}(ds+d'_f,ds+(d'+1)f).
\]
\item
If $h(z)$ is a $BC_1(q^{1-d'}\eta')$-symmetric theta function of degree 1,
then
\[
\oD_{q^{-d}\eta;q;p}(h)\in \cS_{\eta,\eta';q;p}(ds+d'f,(d+1)s+d'f).
\]
\item If $b(z)$ is holomorphic such that
  $b(pz)=(q^{-d-d'+1}\eta\eta'/p^2z^4)b(z)$, then
\[
\oD_{q^{-d}\eta;q;p}(b)\in \cS_{\eta,\eta';q;p}(ds+d'f,(d+1)s+(d'+1)f).
\]
\end{itemize}

We will show below that this family of categories is flat: the dimension of
any given $\Hom$ space is independent of the parameters.  Moreover, when
$q=1$ and $\eta/\eta'\notin p^\Z$, the category is isomorphic to the
subcategory of $\Coh_{\P^1\times \P^1}$ with objects $\sO(ds+d'f)$ (where
$\sO(s)$ is the class of a section and $\sO(f)$ the class of a fiber
relative to one of the two rulings of $\P^1\times \P^1$), and the natural
``swap rulings'' automorphism of $\P^1\times \P^1$ extends to the full
family of categories.

This construction is essentially invariant under translation of the group
of objects, a fact we may use to largely restrict our attention to
morphisms from the object $0$.

\begin{lem}
As spaces of operators,
\[
\cS_{\eta,\eta';q;p}(d_1s+d'_1f,d_2s+d'_2f)
=
\cS_{q^{-d}\eta,q^{-d'}\eta';q;p}((d_1-d)s+(d'_1-d')f,(d_2-d)s+(d'_2-d')f),
\]
and this gives an isomorphism
\[
\cS_{\eta,\eta';q;p}
\cong
\cS_{q^{-d}\eta,q^{-d'}\eta';q;p}
\]
which on objects takes $d_1s+d'_1f$ to $(d_1-d)s+(d'_1-d')f$.
\end{lem}

\begin{proof}
Indeed, it suffices to prove this for the spaces of generators, where it is
immediate.
\end{proof}

One might worry that we are missing some first-order operators in the above
construction, since we only included those in which $b(z)$ comes from a
line bundle of degree 2 or 4.  In fact, we are essentially missing no such
operators.  For convenience, we only consider morphisms starting from $0$,
as the remaining cases will follow by translation invariance.

\begin{lem}
If $b(z)$ is a theta function with
\[
b(pz) = (q\eta'/pz^2) (\eta/pz^2)^{d'} b(z),
\]
$d'\ge 0$, then
\[
\oD_{\eta;q;p}(b)\in \cS_{\eta,\eta';q;p}(0,s+d'f).
\]
\end{lem}

\begin{proof}
If $d'\in \{0,1\}$, there is nothing to prove, as the given operators are
already generators.  We proceed by induction, and assume we have already
proved the result for $d'-1$.  Now, if $g(z)$ is a $BC_1(\eta)$-symmetric
theta function of degree 1, then
\[
g\oD_{\eta;q;p}(b')=\oD_{\eta;q;p}(gb');
\]
it will thus suffice to show that the space of theta functions with
multiplier
\[
b(pz)/b(z)=(q\eta'/pz^2)(\eta/pz^2)^{d'}
\]
is spanned by
those functions of the form $g(z)b'(z)$ where $g$, $b'$ are theta functions
with multipliers
\[
g(pz)/g(z)=\eta/pz^2,\qquad 
b'(pz)/b'(z)=(q\eta'/pz^2)(\eta/pz^2)^{d'-1}.
\]
This is a special case of the following lemma, which we state in an
algebraic way (and will use extensively).
\end{proof}

\begin{lem}\label{lem:bundle_products}
  Let $C$ be a smooth genus 1 curve over a field $k$, and let ${\cal L}$,
  ${\cal L}'$ be line bundles on $C$ with $\deg({\cal L}),\deg({\cal
    L}')\ge 2$ and either ${\cal L}\not\cong {\cal L}'$ or $\deg({\cal
    L})>2$.  Then the multiplication map
\[
\Gamma(C;{\cal L})\otimes \Gamma(C;{\cal L}')
\to
\Gamma(C;{\cal L}\otimes {\cal L}')
\]
is surjective.
\end{lem}

\begin{proof}
  Since $\deg({\cal L})\ge 2$, we may choose two global sections $f_1$,
  $f_2\in \Gamma({\cal L})$ that have no zeros in common.  As a result, they
  define a short exact sequence
\[
\begin{CD}
0\to {\cal L}'\otimes {\cal L}^{-1}
@>(f_1,f_2)>>
{\cal L}'\oplus {\cal L}'
@>(f_2,-f_1)>>
{\cal L}\otimes {\cal L}'
\to 0
\end{CD}
\]
If ${\cal L}'\otimes {\cal L}^{-1}$ is acyclic, then the second morphism
remains surjective on taking global sections, and the claim follows.  Since
the symmetric argument applies whenever ${\cal L}\otimes {\cal L}^{\prime
  {-}1}$ is acyclic, we have proved the desired result whenever ${\cal
  L}\not\cong {\cal L}'$.

If ${\cal L}\cong {\cal L}'$, our hypotheses ensure that $\deg({\cal
  L})>2$.  But then for any point $x\in C$ (the claim is geometric, so we
may extend scalars as necessary to ensure a point exists), the Lemma gives
a surjection
\[
\Gamma({\cal L}\otimes \sO(-x))\otimes \Gamma({\cal L})
\to
\Gamma({\cal L}^{\otimes 2}\otimes \sO(-x)).
\]
It follows that for any point $x\in C$, the image of
\[
\Gamma({\cal L})\otimes \Gamma({\cal L})
\to
\Gamma({\cal L}^{\otimes 2})
\]
contains every section vanishing at $x$.  Since any section vanishes at
{\em some} point, the claim follows.
\end{proof}

\begin{rem} For an analytic version of essentially the same argument, see
  \cite{sklyanin_anal}.
\end{rem}

The only missing operators are those with $d'<0$; since the multiplier
of a theta function is determined (up to a power of $p$) by its zeros, we
find that such operators can only exist if $d'=-1$ and
$q\eta' \in p^\Z \eta$.  While we could include such
operators, the result would necessarily fail to be a flat deformation;
moreover, even including such operators would fail to include all the
operators that ``ought'' to be there, see Section \ref{sec:saturate1} below.

To complete the connection between difference equations and the category
$\cS_{\eta,\eta';q;p}$, it remains only to construct a module for each
difference equation, in such a way that we can recover the difference
equations (and, ideally, the {\em solutions} of the difference equation)
from the module.  To this end, we introduce a natural representation of
$\cS_{\eta,\eta';q;p}$: each morphism of $\cS_{\eta,\eta';q;p}$ is a
difference operator with meromorphic coefficients, and thus acts on the
space of all meromorphic functions.  More precisely, we should impose the
appropriate symmetry conditions, and thus obtain a module $\Mer$ over
$\cS_{\eta,\eta';q;p}$ (i.e., a functor from $\cS_{\eta,\eta';q;p}$ to
$\C-\Vect$) as follows: $\Mer(ds+d's)$ is the space of meromorphic
functions $f(z)$ satisfying $f(\eta/q^{d-1} z)=f(z)$, while if
\[
\oD=\sum_{0\le k\le d} c_k(z) T^k \in \cS_{\eta,\eta';q;p}(d_0s+d'_0f,(d_0+d)s+(d'_0+d')f),
\]
then
\[
\oD\cdot f(z) = \sum_{0\le k\le d} c_k(z) f(q^k z).
\]
For any $d_0,d'_0\in \Z$, we also define a module $P_{d_0s+d'_0f}$ by
\[
P_{d_0s+d'_0f}(ds+d'f):=\cS_{\eta,\eta';q;p}(-d_0s-d'_0f,ds+d'f),
\]
with the obvious multiplication.  Note that the usual Yoneda construction gives
\[
\Hom(P_{d_0s+d'_0f},P_{ds+d'f})
\supset
\cS(-ds-d'f,-d_0s-d'_0f).
\]

\begin{thm}
Given any symmetric elliptic difference equation
\[
v(qz) = A(z) v(z)
\]
with $A(pz) = (q\eta'/\eta)^2 A(z)$, there is a corresponding module $M_A$
over $\cS_{\eta,\eta';q;p}$ such that the space of homomorphisms
$\Hom(M_A,\Mer)$ can be naturally identified with the space of meromorphic
vectors such that
\[
v(qz) = A(z) v(z),\qquad v(1/z)=v(z).
\]
\end{thm}

\begin{proof}
Factor the equation as
\[
B(\eta/z)^t v(qz) = B(z)^t v(z)
\]
as above, and interpret each operator
\[
\oD_{\eta;q;p}(B_{ij}(z))
\]
as an element of
\[
\cS_{\eta,\eta';q}(0,s+d'f)
\subset
\Hom(P_{-s-d'f},P_0).
\]
(This fails if $q\eta'=\eta$ and some coefficients of $B$ are constant; to
fix this, we may multiply such operators by a basis of the space of
$BC_1(\eta)$-symmetric theta functions.  We omit the details, since this
problem goes away once we have introduced the saturated category.)

This allows us to interpret $B$ as a morphism of the form
\[
\bigoplus_{1\le k\le n}
P_{-s+d'_kf}
\to
P_0^n.
\]
Let $M_A$ be the cokernel of this morphism.  The morphisms from $P_0$ to
$\Mer$ can be naturally identified with meromorphic functions invariant
under $z\mapsto q\eta/z$, and thus morphisms from $M_A$ to $\Mer$ can be
identified with vectors of such functions.  The vectors that arise are
those such that the corresponding morphism from $P_0^n$ becomes trivial on
the domain of $B$; in other words, such that
\[
B(\eta/z)^t v(qz) = B(z)^t v(z).
\]
\end{proof}

\begin{rem}
Since $B$ is essentially unique once we maximize $\sum_k d'_k$
(automorphisms of the vector bundle clearly lift to automorphisms of the
above domain), $M_A$ as constructed is itself unique up to isomorphism.
A similar construction applies to a symmetric equation in which
\[
A_{ij}(pz) = (q\eta'/\eta)^2 (q\eta/pz^2)^{e_i-e_j} A_{ij}(z),
\]
the only modification being to replace $P_0^n$ by a sum
$\oplus_{1\le i\le n} P_{e_i f}$.
\end{rem}

In general, translation invariance tells us that we are missing operators
whenever $\eta/\eta'\in p^\Z q^\Z$.  (We will see later that this is the
only instance in which we can enlarge the $\Hom$ spaces in a natural way.)
On the other hand, most of the time, our spaces of generators are
redundant.  Indeed, the same argument tells us that the natural composition
map
\[
\cS(s,s+f)\otimes \cS(0,s)\to \cS(0,s+f)
\]
is surjective unless $\eta/q\eta'\in p^\Z$, when we only obtain operators
$\oD_{\eta;q;p}(b)$ where $z^l b$ is $BC_1(\eta)$-symmetric of degree 2 for
some integer $l$ (depending on the power of $p$ in the ratio of
multipliers).  Similarly,
\[
\cS(f,s+f)\otimes \cS(0,f)\to \cS(0,s+f)
\]
is surjective unless $q\eta/\eta'\in p^\Z$, when the leading coefficient is
essentially a $BC_1(q\eta)$-symmetric theta function.  In particular, at
least one of the two compositions is surjective unless $q^2\in p^\Z$ and
$q\eta/\eta'\in p^\Z$.  If neither map is surjective but $q\notin p^\Z$, we
obtain the span of the spaces of $BC_1(\eta)$-symmetric and
$BC_1(q\eta)$-symmetric theta functions of degree 2; these are distinct
codimension 1 subspaces of the relevant space of all theta functions, so
span.  In other words, we find that $\cS_{\eta,\eta';q;p}$ is
generated in degrees $f$, $s$ unless $q,\eta/\eta'\in p^\Z$.  (This
corresponds to the commutative Hirzebruch surface $F_2$).

\section{Flatness (even Hirzebruch case)}

As mentioned above, one reason for the above definition of
$\cS_{\eta,\eta';q;p}$ is that the resulting family of $\Z^2$-algebras is
flat.  To be precise, we have the following.

\begin{thm}\label{thm:flat_F0}
For all $\eta, \eta',q,p\in \C^*$, $|p|<1$, and any integers $d$, $d'$,
\[
\dim(\cS_{\eta,\eta';q;p}(0,ds+d'f))=\max(d+1,0)\max(d'+1,0).
\]
\end{thm}

\begin{rem}
This is, of course, trivially true for $d<0$ or $d'<0$, since
the only generators are in the positive quadrant.
\end{rem}

To show this, we will prove both a lower and an upper bound on the
dimensions of the $\Hom$ spaces.  The lower bound, of course, simply
involves constructing sufficiently many linearly independent difference
operators in the space.  The key idea will be to understand
the image and kernel of the leading coefficient map $\oD\mapsto [T^0]\oD$.
Note that this is not just linear: it clearly takes composition in $\cS$ to multiplication of (meromorphic) theta functions.

\begin{lem}
If $d'\ge 0$, then
\[
\dim([T^0]\cS_{\eta,\eta';q;p}(0,d'f)) = d'+1;
\]
if $d\ge 0$, then
\[
\dim([T^0]\cS_{\eta,\eta';q;p}(0,ds)) = d+1.
\]
Finally, if $d,d'>0$, then
\[
\dim([T^0]\cS_{\eta,\eta';q;p}(0,ds+d'f)) = 2d+2d'.
\]
\end{lem}

\begin{proof}
  A section of $\cS_{\eta,\eta';q;p}(0,d'f)$ is a linear combination
  of products of generators of degree $f$, and as such is a
  $BC_1(q\eta)$-symmetric theta function of degree $d'$.  Any such function
  factors into degree 1 functions, and thus we obtain the full space, of
  dimension $d'+1$.  A similar calculation applies to the $d'=0$ case.

  For $d,d'>0$, the leading
  coefficient is a section of a bundle of degree $2d+2d'$, so we need merely
  show that all sections arise.  This follows from Lemma
  \ref{lem:bundle_products} by induction: if it holds for $ds+d'f$, then it
  holds for $(d+1)s+d'f$ and $ds+(d'+1)f$ upon multiplication by generators
  of degrees $s$, $f$ respectively.  We thus reduce to the case $d=d'=1$,
  where it follows by inspection of the generators of that degree.
\end{proof}

This already gives us a tight lower bound if $d$ or $d'$ is $0$ or $1$ (so
that we reduce to showing that in that case, the operator is determined by
its leading coefficient).  The first case where we obtain an additional
operator is $d=d'=2$, where it turns out that the operator $T$ can be
obtained.  This was already shown in \cite{sklyanin_anal}, but we will give
a different argument for a slightly stronger fact.

\begin{lem}\label{lem:central_elt}
For any $\eta$, $\eta'$, $q$, $p$, there exists an expansion
\[
T = \sum_i \oD_{\eta/q;q;p}(b_{i1})\oD_{\eta;q;p}(b_{i2}),
\]
where for each $i$, $b_{i1}$, $b_{i2}$ are theta functions with
\begin{align}
b_{i1}(pz) &= (\eta\eta'/p^2qz^4)b_{i1}(z),\notag\\
b_{i2}(pz) &= (q\eta\eta'/p^2z^4)b_{i2}(z).\notag
\end{align}
Moreover, for any $v\in \C^*$, there exists such an expansion with
$b_{i1}(v)=0$ for all $i$.
\end{lem}

\begin{proof}
Let $b_i(x,y):=b_{i1}(x)b_{i2}(y)$, and note that we can write
\begin{align}
\oD_{\eta/q;q;p}(b_{i1})\oD_{\eta;q;p}(b_{i2})
={}&
\frac{z^2 b_i(z,z)}{\theta_p(z^2/\eta,qz^2/\eta)}
+
\frac{qz^2 b_i(\eta/qz,\eta/qz)}{\theta_p(z^2/\eta,q^2z^2/\eta)}
T^2\notag\\
&-
\frac{z^2 b_i(z,\eta/z)}{\theta_p(z^2/\eta,qz^2/\eta)}
T
-
\frac{qz^2 b_i(\eta/qz,qz)}{\theta_p(qz^2/\eta,q^2z^2/\eta)}
T
.
\end{align}
This is, of course, linear in $b_i$, so a similar formula holds for the sum
of these product operators, in terms of the function $b(x,y)=\sum_i
b_i(x,y)$.  This function can be any element of the tensor product of the
two spaces of theta functions; i.e., the only constraint on $b$ is that it
be holomorphic and have multipliers
\begin{align}
b(px,y) &= (\eta\eta'/p^2qx^4)b(x,y),\notag\\
b(x,py) &= (q\eta\eta'/p^2y^4)b(x,y).\notag
\end{align}
Note that the constraint that $b_{i1}(v)=0$ for all $i$ translates to the
condition that $b(v,y)=0$.

Now, consider the function
\[
b(x,y)
=
\frac{\theta_p(y/x,\eta\eta'/q xyvv',x/v,x/v',y/w,q^2 vv'/wy)}
     {\theta_p(\eta/q vv',\eta'/q vv',qv/w,qv'/w)}
\]
for $v'$, $w\in \C^*$ such that the denominator is nonzero.  This certainly
satisfies $b(v,y)=0$, so the corresponding operator has an expansion as
required.  The coefficients of $T^0$ and $T^2$ are 0, since $b(z,z)=0$,
while the coefficient of $T^1$ simplifies via the addition law.  We find
that the resulting operator is $q^{-1}\eta T$, so the result
follows.
\end{proof}

\begin{rem}
  Note that as long as $v'$ and $w$ are such that the denominator does not
  vanish, we obtain an expansion satisfying
  $b_{i1}(v)=b_{i1}(v')=b_{i2}(w)=0$ for all $i$.
\end{rem}

In particular, we find that
\[
\dim(\cS_{\eta,\eta';q;p}(0,ds+d'f))
\ge
\dim([T^0]\cS_{\eta,\eta';q;p}(0,ds+d'f))
+
\dim(T\cS_{\eta,\eta';q;p}(0,(d-2)s+(d'-2)f)),
\]
which gives the desired lower bound by induction.

For the upper bound, we note that since our category is given as a
generated subcategory of a flat family, we have semicontinuity:
specializing the parameters can only make the $\Hom$ spaces smaller.  We
thus need only prove the upper bound for generic parameters.  Thus, suppose
$\eta/\eta'\notin p^\Z q^\Z$.  As we have already noted, this means that
the generators of degree $s+f$ are redundant, so we have precisely four
generators from each object.  To be precise, for each $d\in \Z$, let
$g_{1d}$, $g_{2d}$ be a basis of the space of $BC_1(q^{1-d}\eta)$-symmetric
theta functions of degree 1, and for each $d'\in \Z$, let $h_{1d'}$,
$h_{2d'}$ be a basis of the space of $BC_1(q^{1-d'}\eta)$-symmetric theta
functions of degree 1.  Then the category $\cS_{\eta,\eta';q;p}$
is generated by the elements
\[
x_i(d,d'):=g_{id}\in \cS_{\eta,\eta';q;p}(ds+d'f,ds+(d'+1)f)
\]
and
\[
y_i(d,d'):=\oD_{q^{-d}\eta;q;p}(h_{id'})\in
\cS_{\eta,\eta';q;p}(ds+d'f,(d+1)s+d'f).
\]
To prove the desired upper bound, we will need to consider the relations
that these generators satisfy.  Luckily, it will turn out that quadratic
relations suffice.

There are three cases to consider for quadratic relations: relations of
degree $2f$, $s+f$, and $2s$.  The first set is quite simple: since 
elements of degree $f$ are multiplication operators, they simply commute:
\[
x_1(d,d'+1)x_2(d,d') = x_2(d,d'+1)x_1(d,d').
\]
The degree $s+f$ is only slightly more complicated: since
\begin{align}
x_i(d+1,d')y_j(d,d') &= \oD_{q^{-d}\eta;q;p}(g_{i(d+1)}h_{jd'})\notag\\
y_i(d,d'+1)x_j(d,d') &= \oD_{q^{-d}\eta;q;p}(h_{i(d'+1)}g_{jd}),\notag
\end{align}
we find that
\[
\sum_{ij} \alpha_{ij} y_i(d,d'+1)x_j(d,d')
=
\sum_{ij} \beta_{ij} x_i(d+1,d')y_j(d,d')
\]
iff
\[
\sum_{ij} \alpha_{ij} h_{i(d'+1)}g_{jd}
=
\sum_{ij} \beta_{ij} g_{i(d+1)}h_{jd'}.
\]
The constraint $\eta/\eta'\notin p^\Z q^\Z$ implies that the four functions
on either side are linearly independent, and thus we obtain four relations
of the form
\[
y_i(d,d'+1)x_i(d,d') = \sum_{kl} M_{ijkl}(d,d') x_i(d+1,d')y_j(d,d').
\]
The relations of degree $2s$ require somewhat more thought, but it turns
out that they are in fact quite nice.

\begin{lem}
If $h_1$, $h_2$ are $BC_1(\eta')$-symmetric theta functions of degree 1,
then
\[
\oD_{\eta/q;q;p}(h_1)\oD_{\eta;q;p}(h_2)
=
\oD_{\eta/q;q;p}(h_2)\oD_{\eta;q;p}(h_1).
\]
\end{lem}

\begin{proof}
As above, we can express the difference of the two operators in terms
of the function
\[
h(x,y) = h_1(x)h_2(y)-h_2(x)h_1(y)
\]
By symmetry considerations, we find
\[
h(x,y)\propto x^{-1}\theta_p(x/y,xy/\eta'),
\]
from which it is easy to see that the operator vanishes.
\end{proof}

In other words, there is one degree $2s$ relation, to the effect that
\[
y_1(d+1,d')y_2(d,d') = y_2(d+1,d')y_1(d,d').
\]

Now, given any homogeneous linear combination of products of generators, we
may use the degree $s+f$ relations to move the ``$x$'' generators to the
right, then use the degree $2f$ and $2s$ relations to sort the ``$x$'' and
``$y$'' generators amongst themselves.  It follows that
$\cS_{\eta,\eta';q;p}(0,ds+d'f)$ is spanned by products of the form
\[
y_1^k y_2^{d'-k} x_1^l x_2^{d'-l},
\]
where by $y_i^m$ we mean the appropriate product of the form
\[
y_i^m:=y_i(d_0+m-1,d'_0)y_i(d_0+m-2,d'_0)\cdots y_i(d_0,d'_0)
\]

The upper bound follows, and concludes the proof of Theorem
\ref{thm:flat_F0}.
\qed

Since the lower and upper bounds agree, the proof of the lower bound gives
us the following useful fact.

\begin{cor}
If $\oD\in \cS_{\eta,\eta';q;p}(0,ds+d'f)$ is such that $[T^0]\oD=0$,
then
\begin{align}
T^{-1}\oD&\in \cS_{\eta,\eta';q;p}(0,(d-2)s+(d'-2)f),\\
\oD T^{-1}&\in \cS_{\eta,\eta';q;p}(2s+2f,ds+d'f).
\end{align}
\end{cor}

Similarly, the fact that the bounds agree implies that the relations
considered in the proof of the upper bound actually give a presentation.
More than that, the relations actually give a quadratic Gr\"obner basis
(with respect to a suitable monomial ordering), since we can put elements
into canonical form with moves consisting only of replacing subwords with
linear combinations of smaller subwords.  This presumably
implies that when $\eta/\eta'\notin p^\Z q^\Z$, the category
$\cS_{\eta,\eta';q;p}$ is Koszul in some suitable sense.  Similarly, apart
from a presumably mild nondegeneracy condition, this implies that
$\cS_{\eta,\eta';q;p}$ satisfies a form of Artin-Schelter regularity.
We will show in Appendix \ref{sec:integrable} below that this essentially
characterizes the category: any $\Z^2$-algebra with a similar presentation
is either isomorphic to $\cS_{\eta,\eta';q;p}$ or to a degeneration thereof.

An even more striking consequence of this presentation stems from the fact
that the degree $f$ generators and the degree $s$ generators satisfy very
similar relations.  Indeed, each generator is identified with a suitable
theta function, and the relations are just the quadratic relations
satisfied by those functions.  (Compare the abstract description of the
three-generator Sklyanin algebra \cite{ArtinM/TateJ/VandenBerghM:1990}.)
In particular, we find that when $\eta/\eta'\notin p^\Z q^\Z$, we obtain an
isomorphism
\[
\cS_{\eta,\eta';q;p}\cong \cS_{\eta',\eta;q;p},
\]
which on objects takes $ds+d'f$ to $d's+df$.  It turns out that the
constraint on $\eta/\eta'$ can be removed here.

\begin{thm}
There is an isomorphism
\[
\cS_{\eta,\eta';q;p}\cong \cS_{\eta',\eta;q;p},
\]
which on objects takes $ds+d'f$ to $d's+df$, and on generators takes
\begin{align}
g(z)
\in{}&\cS_{\eta,\eta';q;p}(ds+d'f,ds+(d'+1)f)\notag\\
\oD_{q^{-d}\eta;q;p}(h)
\in{}&\cS_{\eta,\eta';q;p}(ds+d'f,(d+1)s+d'f) \notag\\
\oD_{q^{-d}\eta;q;p}(b)
\in{}&\cS_{\eta,\eta';q;p}(ds+d'f,(d+1)s+(d'+1)f)
\notag
\end{align}
to
\begin{align}
\oD_{q^{-d'}\eta';q;p}(g)\in{}&\cS_{\eta',\eta;q;p}(d's+df,(d'+1)s+df),
\notag\\
h\in {}&\cS_{\eta',\eta;q;p}(d's+df,d's+(d+1)f),
\notag\\
\oD_{q^{-d'}\eta';q;p}(b)\in {}&\cS_{\eta',\eta;q;p}(d's+df,(d'+1)s+(d+1)f)
\notag
\end{align}
respectively.
\end{thm}

\begin{proof}
  This transformation is clearly an involution, so we need only show that
  it extends to a homomorphism.  In other words, given any relation
  satisfied by the generators of $\cS_{\eta,\eta';q;p}$, we need to
  show that their images in $\cS_{\eta',\eta;q;p}$ satisfy the same
  relation.  The key point is that for any given relation, this is a {\em
    closed} condition, and thus we need only prove it for $\eta/\eta'\notin
  p^\Z q^\Z$.  For relations involving only the degree $f$ and $s$
  generators, this follows by checking the above quadratic relations, which
  as we have noted simply reduces to ordinary multiplication.

  It thus remains only to show that the homomorphism defined by the given
  action on degree $f$ and $s$ generators acts in the stated way on degree
  $s+f$ generators.  By our genericity assumption, there is a basis of such
  generators consisting of elements of the form
\[
g\oD_{q^{-d}\eta;q;p}(h)
=
\oD_{q^{-d}\eta;q;p}(gh)
\in
\cS_{\eta,\eta';q;p}(ds+d'f,(d+1)s+(d'+1)f)
\]
But this transforms to
\[
\oD_{q^{-d'}\eta'}(g) h
=
\oD_{q^{-d'}\eta'}(gh)
\]
as required.
\end{proof}

Since this involution exchanges multiplication and difference operators, we
call it the ``Fourier transform''.  (In fact, under a suitable limit, it
degenerates to the usual Fourier-Laplace transform, see \cite{noncomm2}.)
We should also note that the Fourier transform essentially preserves the
elements $T$; to be precise, if we express $T$ in terms of degree $s+f$
generators as above, it is then straightforward to show that the transform
takes
\[
q^{-d}\eta T\in \cS_{\eta,\eta';q;p}(ds+d'f,(d+2)s+(d'+2)f)
\]
to $q^{-d'}\eta' T$.

A similar argument shows that any family of isomorphisms between leading
coefficient categories $[T^0]\cS_{\eta,\eta';q;p}$ extends to isomorphisms
between the full categories.  Indeed, for $\eta/\eta'\notin p^\Z q^\Z$, we
find that both $\cS_{\eta,\eta';q;p}$ and $[T^0]\cS_{\eta,\eta';q;p}$ have
the same generators, and the same low-degree relations; moreover, the
low-degree relations suffice to give a presentation of
$\cS_{\eta,\eta';q;p}$.  Any isomorphism of leading coefficient categories
takes generators to generators preserving low-degree relations, so extends.
We may then use flatness to relax the constraint on the parameter: any
family of isomorphisms extends to the closure of the family in parameter
space.  The morphisms in $[T^0]\cS_{\eta,\eta';q;p}$ can be identified with
(holomorphic) theta functions of multiplier
\[
\frac{z^{-2d_2-2d_2'} q^{d_2+d_2'-d_2d_2'}(\eta/p)^{d_2'}(\eta'/p)^{d_2}}
     {z^{-2d_1-2d_1'} q^{d_1+d_1'-d_1d_1'}(\eta/p)^{d_1'}(\eta'/p)^{d_1}},
\]
so any transformation of theta functions that gives the correct multiplier
and respects multiplication will produce an isomorphism of $\cS$.

For instance, if $g(z)$ is a theta function with the above multiplier, then
$g(\alpha z)$ has essentially the same multiplier, except with $\eta,\eta'$
replaced by $\alpha^{-2}\eta,\alpha^{-2}\eta'$.  We thus conclude that
there is an isomorphism
\[
\cS_{\eta,\eta';q;p}\cong \cS_{\alpha^{-2}\eta,\alpha^{-2}\eta';q;p}
\]
for any $\alpha\in \C^*$, acting as the identity on objects.
(Note that when $\alpha=-1$, this isomorphism becomes an automorphism.)
The action of this isomorphism on difference operators is simple: just
conjugate by the operator $T_\alpha$ that translates by $\alpha$.
We similarly obtain an isomorphism
\[
\cS_{\eta,\eta';q;p}\cong \cS_{\beta^2/\eta,\beta^2/\eta';1/q;p},
\]
conjugating by the reflection $g(z)\mapsto g(\beta/z)$.  There is also the
trivial gauge transformation that acts on leading coefficients by
\[
g(z)\mapsto (C(d_2s+d'_2f)/C(d_1s+d'_1f))g(z),
\]
where $C:\Z^2\to \C^*$ is any map.

A slightly more subtle transformation multiplies $g(z)$ by
$z^{d_2-d_1}$;
this produces an isomorphism
\[
\cS_{\eta,\eta';q;p}\cong \cS_{\eta,p\eta';q;p}.
\]
This acts on difference operators by taking
\[
\sum_{0\le k\le d_2-d_1} c_k(z) T^k
\in
\cS_{\eta,\eta';q;p}(d_1s+d'_1f,d_2s+d'_2f)
\]
to
\[
\sum_{0\le k\le d_2-d_1}
z^{d_2-d_1-2k} \eta^k q^{-k(k+d_1-1)} c_k(z) T^k;
\]
it is easy to see that this acts as an automorphism on $\MerDiff_q$, and
takes generators to generators.  We can also view this as a gauge
transformation, but with nonconstant functions: for any
$a\in \C^*$, if we define functions
\[
F_{ds+d'f}(z)
=
q^{d(d-1)/2} (-\eta a)^{-d}
\theta_q(a z,a q^{1-d}\eta/z),
\]
then
\[
\sum_{0\le k\le d_2-d_1} 
z^{d_2-d_1-2k} \eta^k q^{-k(k+d_1-1)} c_k(z) T^k
=
F_{d_2s+d'_2f}(z)^{-1}
\sum_{0\le k\le d_2-d_1} c_k(z) T^k
F_{d_1s+d'_1f}(z).
\]

If we conjugate the above isomorphism by the Fourier transform, we obtain
an isomorphism
\[
\cS_{\eta,\eta';q;p}\cong \cS_{p\eta,\eta';q;p};
\]
note that together with the previous isomorphism, this lets us produce
automorphisms from translations by $\pm\sqrt{p}$, in addition to the
translation by $-1$ automorphism considered above.  (It follows from the
elliptic construction below that this can also be described as a
nonconstant gauge transformation, we omit the details.)  There is also an
isomorphism which acts on leading coefficients by
\[
g(z)\mapsto z^{-d_2 d'_2+d_1 d'_1} g(z),
\]
and gives an isomorphism
\[
\cS_{\eta,\eta';q;p}\cong \cS_{\eta,\eta';pq;p}.
\]
This invariance under $p$-shifting the parameters is, of course, not
unexpected, and simply reflects (as does modular invariance, the details of
which we omit) the underlying algebraic nature of the construction.  This
will be clearer in the case of the noncommutative $F_1$ considered below.

\medskip

There is one more isomorphism of interest coming from the above
construction.  The morphisms in $[T^0]\cS$ are morphisms between line
bundles on $\C/\langle p\rangle$, and as such we can obtain a {\em
  contravariant} isomorphism by simply dualizing those bundles.  In this
way, we obtain an isomorphism
\[
\cS_{\eta,\eta';q;p}\cong \cS_{\eta,\eta';1/q;p}^{\text{op}},
\]
which acts on objects by $ds+d'f\mapsto (2-d)s+(2-d')f$, and on leading
coefficients as the identity.  (This action on objects is motivated by the
fact that $2s+2f$ is the anticanonical divisor.)  This involution has a
particularly nice interpretation at the level of difference operators.
There is a natural contravariant involution between $\MerDiff_q$ and
$\MerDiff_{1/q}$, acting via
\[
\sum_{0\le k\le d} c_k(z) T_q^k
\mapsto
\sum_{0\le k\le d} T_{1/q}^k c_k(z)
=
\sum_{0\le k\le d} c_k(q^{-k} z) T_{1/q}^k.
\]
We can think of this as a sort of formal adjoint: if $\oD$ is a meromorphic
$q$-difference operator, and $f$, $g$ are any meromorphic functions, then
we have the formal integral identity
\[
\int (\oD f)(z) g(z) \mu(dz) = \int f(z) (\oD^{\ad}g)(z) \mu(dz)
\]
for any measure $\mu$ invariant under $q$-shifts, with $\oD^{\ad}$ given by the
above involution.  (Of course, this only applies as an actual identity of
integrals when the various terms on both sides are integrable, a relatively
unlikely event, and only determines $\oD^{\ad}$ when $q$ is non-torsion; for
our purposes, though, there is no problem with doing purely formal
manipulations of integrals.)  While this na\"{\i}ve formal adjoint does not
preserve the form of morphisms in $\cS$, we can fix things by a
suitable gauge transformation.  We thus replace the na\"{\i}ve formal
adjoint on $\MerDiff$ by the adjoint formally satisfying
\[
\int z^{-1}\theta_p(q^{d_2-1}z^2/\eta) g(z) (\oD f)(z) \mu(dz)
=
\int z^{-1}\theta_p(q^{d_1-1}z^2/\eta) (\oD^{\ad}g)(z) f(z) \mu(dz),
\]
for $\oD\in \MerDiff_q(d_1s+d'_1f,d_2s+d'_2f)$, in which case
\[
\oD^{\ad}\in
\MerDiff_{1/q}((2-d_2)s+(2-d'_2)f,(2-d_1)s+(2-d'_1)f).
\]
Explicitly we have
\[
(\sum_{0\le k} c_k(z) T_q^k)^{\ad}
=
\sum_{0\le k}
\frac{q^k \theta_p(q^{d_2-2k-1}z^2/\eta)}
     {\theta_p(q^{d_1-1}z^2/\eta)}
c_k(q^{-k}z) T_{1/q}^k.
\]
This is still contravariant, and
acts correctly on generators, so restricts to an involution
\[
\cS_{\eta,\eta';q;p}\cong \cS_{\eta,\eta';1/q;p}^{\text{op}}
\]
as required.

The notion of formal adjoint also gives an alternate interpretation of the
Fourier transform.  We would like the Fourier transform to arise as
conjugation by a suitable family of operators $F_{\eta,\eta';q;p}$, so that
\[
  \oD F_{q^{-d_1}\eta,q^{-d'_1}\eta';q;p}
  =
  F_{q^{-d_2}\eta,q^{-d'_2}\eta';q;p} \hat\oD.
\label{eq:fourier_on_D_as_operator}
\]
If we take the ansatz that $F_{\eta,\eta';q;p}$ should be an integral
operator
\[
(F_{\eta,\eta'} g)(z) = \int K_{\eta,\eta'}(z,w) g(w) \mu(dw)
\]
(which, again, will have definitional issues as anything other than a
formal object), then \eqref{eq:fourier_on_D_as_operator} becomes an
equation for $K_{\eta,\eta'}(z,w)$, comparing the action of $\oD$ in
the $z$ variables to the action of the na\"{\i}ve formal adjoint of
  $\hat\oD$ in the $w$ variables.  Since we know the Fourier transform
  is an isomorphism, we need only consider the equations coming from
  generators.  The equation for $\oD\in \cS'(0,s)$ is given by
\[
\oD_{\eta;q;p}(\theta(z/a,q\eta'/az))_z
K_{\eta,\eta';q;p}(z,w)
=
\theta(w/a,q\eta'/aw)
K_{\eta/q,\eta';q;p}(z,w).
\]
Setting $a=w$ gives a first-order equation in $K_{\eta,\eta';q;p}(z,w)$ as
a function of $z$, implying
\[
K_{\eta,\eta';q;p}(z,w)
=
\theta_q(z,q\eta/z)
\Gampq(z/w,q\eta/zw,zw/q\eta',w\eta/\eta'z)
G_{\eta,\eta'}(z,w),
\]
where $G_{\eta,\eta'}(z,w)$ is $q$-elliptic in $z$.  Taking $a=z$ then
gives the first-order equation
\[
G_{\eta/q,\eta'}(z,w)
=
\theta_p(\eta/q\eta')
G_{\eta,\eta'}(z,w),
\]
allowing us to eliminate the dependence in $\eta$ (again up to an elliptic
factor).  A similar calculation using the other generators lets us solve
for the $w$ and $\eta'$-dependence, giving us
\[
G_{\eta,\eta'}(z,qw)
=
-\frac{\theta_p(w^2/q\eta')}{\theta_p(\eta'/qw^2)}
G_{\eta,\eta'}(z,w),
\]
leading us to the following expression for the kernel (up to an
indeterminable factor which is $q$-elliptic in each of $\eta$, $\eta'$,
$z$, $w$):
\[
K_{\eta,\eta';q;p}(z,w)
=
\frac{\theta_q(z,q\eta/z)}
     {\theta_q(w,q\eta'/w)}
\frac{\Gampq(z/w,q\eta/zw,zw/q\eta',w\eta/\eta'z)}
     {\Gampq(w^2/q\eta',q\eta'/w^2,\eta/\eta')}
\]
Apart from the $q$-theta function factors, this is essentially the same as
the kernel for an integral operator considered in
\cite{SpiridonovVP/WarnaarSO:2006}.

The fact that we need a formal interpretation above is analogous to the
fact that the corresponding identity for the Laplace transform (swapping
differentiation and multiplication) only holds up to boundary terms from
the requisite integration by parts.  Here, if we replace the formal
invariant measure with an actual measure, we will have error terms
corresponding to the failure to be invariant under $q$-shifts.  If one can
control those terms (say if the integral is a contour integral, and the
region between the contour and its $q$ shift contains no poles), then the
formal calculation gives rise to actual difference equations.

As an example, consider the operator
\[
\oD=\oD_{\eta;q;p}(z^{-1}\theta_p(a_1 z,a_2 z,a_3 z,q a_1a_2a_3\eta\eta'/z))
\in
\cS_{\eta,\eta';q;p}(0,s+f).
\]
The equation $\hat\oD g(w)=0$ is straightforward to solve:
\[
g(w)
=
\theta_q(w,q\eta'/w)
\frac{\prod_{1\le i\le 3} \Gampq(a_i w,qa_i\eta'/w)}
 {\Gampq(q^2 a_1a_2a_3\eta\eta'/w,q a_1a_2a_3 w\eta/b_1)}.
\]
Of course, this is only defined up to an overall $q$-elliptic factor, but
with this particular choice, we find (generically) that for a suitable
choice of contour $C$, the formal equation
\[
  \oD\int_C K_{\eta,\eta'}(z,w) g(w) \frac{dw}{w}
  =
  \int_C K_{\eta,\eta'}(z,w) (\hat\oD g)(w) \frac{dw}{w}
  =
  0
\]
is actually valid analytically.  The corresponding evaluation of the
integral arising from this first-order equation (up to an overall elliptic
factor) is essentially just the elliptic beta integral of
\cite{SpiridonovVP:2001}, and the proof is essentially the same.

A similar calculation shows that the above integral operator takes
solutions of equations of degree $s+df$ to solutions of equations of degree
$ds+f$, recovering the fact that order $d-1$ elliptic beta integrals
satisfy symmetric elliptic difference equations of degree $d$ \cite{dets}.
One can also make the case $2s+2f$ precise, see
\cite{RuijsenaarsSNM:2015}.  More generally, one expects that if
$\oD$ and $\hat\oD$ are Fourier transforms, then it should be possible to
span the space of functions annihilated by $\oD$ by integrals of the form
\[
\int_C K_{\eta,\eta'}(z,w) g(w) \frac{dw}{w}
\]
with $\hat\oD g=0$.  We will prove a precise version of this (subject to
some technical assumptions) below.


\section{Torsion modules and sheaves I}
\label{sec:saturate1}

One disadvantage of dealing with a $\Z^d$-algebra is that, as with
multigraded commutative algebras, the usual dictionary between modules and
sheaves breaks down.  An obvious solution is to choose a coset of some
subgroup $\Z\subset \Z^d$ and restrict the category to the objects in that
coset to obtain a $\Z$-algebra.  This certainly gives a notion of sheaf,
but the resulting category of sheaves can depend significantly on the
choice of subgroup (e.g., in the commutative setting, the Proj of the
graded algebra associated to a divisor depends in particular on whether the
divisor is ample), and it can be tricky to determine when two such
choices give rise to equivalent categories of sheaves.

In the case of the $\Z^2$-algebra $\cS_{\eta,\eta';q;p}$, we may take
an alternate approach.  First, note that in the commutative degeneration,
the divisor class $ds+d'f$ is {\em always} ample if $d'>d>0$; if $d'\le d$,
it fails to be ample on $F_2$ (it contains or is orthogonal to the
$-2$-curve, which has class $s-f$), while if $d=0$, the divisor class maps
the surface to $\P^1$.  We will show that any of these always ample classes
gives rise to the same category of sheaves on $\cS_{\eta,\eta';q;p}$.

The key ingredient in the construction of the category of sheaves on a
$\Z$-algebra is the notion of a torsion module, one in which each element
is annihilated by all morphisms of sufficiently high degree.  For
$\cS_{\eta,\eta';q;p}$, we replace that notion by the following.

\begin{defn}
  Let ${\cal M}$ be a module over $\cS_{\eta,\eta';q;p}$.  A
  homogeneous element $v\in {\cal M}(d_0s+d'_0f)$ is torsion (of order
  $d_1s+d'_1f$) if there exists some $d'_1\ge d_1\ge 0$ such that for all
  $d'\ge d\ge 0$, we have
\[
\cS_{\eta,\eta';q;p}(d_0s+d'_0f,(d_0+d_1+d)s+(d'_0+d'_1+d')f)v = 0.
\]
The module ${\cal M}$ is torsion iff all homogeneous elements of ${\cal M}$
are torsion.
\end{defn}

Note that a finite linear combination of torsion elements is torsion:
simply maximize $d'_1-d_1$, $d_1$ over the orders of the elements in the
linear combination.

We will then define a sheaf on $\cS_{\eta,\eta';q;p}$ to be an object
of the quotient of $\cS_{\eta,\eta';q;p}$ by the subcategory of
torsion modules.  For this to make sense, we need to know that this
is a Serre subcategory.

\begin{lem}\label{lem:torsion_is_easy}
Let ${\cal M}$ be a module over $\cS_{\eta,\eta';q;p}$. An element
$v\in {\cal M}(0)$ is torsion of order $ds+d'f$ with $d'\ge
d\ge 0$ iff
\[
\cS_{\eta,\eta';q;p}(0,ds+d'f)v = 0.
\]
\end{lem}

\begin{proof}
A simple inductive argument reduces to showing that if
\[
\cS_{\eta,\eta';q;p}(0,ds+d'f)v = 0,
\]
then
\[
\cS_{\eta,\eta';q;p}(0,ds+(d'+1)f)v
=
\cS_{\eta,\eta';q;p}(0,(d+1)s+(d'+1)f)v
=
0,
\]
and thus it will suffice to show that
\[
\cS_{\eta,\eta';q;p}(0,ds+(d'+1)f)
=
\cS_{\eta,\eta';q;p}(ds+d'f,ds+(d'+1)f)
\cS_{\eta,\eta';q;p}(0,ds+d'f)
\]
and
\[
\cS_{\eta,\eta';q;p}(0,(d+1)s+(d'+1)f)
=
\cS_{\eta,\eta';q;p}(ds+d'f,(d+1)s+(d'+1)f)
\cS_{\eta,\eta';q;p}(0,ds+d'f).
\]
The construction of the lower bound on the $\Hom$ space dimensions shows
in particular that whenever $d'\ge d\ge 0$,
\[
\cS_{\eta,\eta';q;p}(0,ds+d'f)
=
\cS_{\eta,\eta';q;p}(ds+df,ds+d'f)
\cS_{\eta,\eta';q;p}(0,ds+df),
\]
and thus
\begin{align}
\cS_{\eta,\eta';q;p}&(ds+d'f,ds+(d'+1)f)
\cS_{\eta,\eta';q;p}(0,ds+d'f)\notag\\
&{}=
\cS_{\eta,\eta';q;p}(ds+d'f,ds+(d'+1)f)
\cS_{\eta,\eta';q;p}(ds+df,ds+d'f)
\cS_{\eta,\eta';q;p}(0,ds+df)\notag\\
&{}=
\cS_{\eta,\eta';q;p}(ds+df,ds+(d'+1)f)
\cS_{\eta,\eta';q;p}(0,ds+df)\notag\\
&{}=
\cS_{\eta,\eta';q;p}(0,ds+(d'+1)f)
\end{align}
as required.  Similarly, the adjoint gives
\[
\cS_{\eta,\eta';q;p}(0,ds+d'f)
=
\cS_{\eta,\eta';q;p}((d'-d)f,ds+d'f)
\cS_{\eta,\eta';q;p}(0,(d'-d)f),
\]
from which the second claim follows.
\end{proof}

\begin{cor}
If $d'_1\ge d_1\ge 0$ and $d'_2\ge d_2\ge 0$, then
\[
\cS_{\eta,\eta';q;p}(0,(d_1+d_2)s+(d'_1+d'_2)f)
=
\cS_{\eta,\eta';q;p}(d_1s+d'_1f,(d_1+d_2)s+(d'_1+d'_2)f)
\cS_{\eta,\eta';q;p}(0,d_1s+d'_1f).
\]
\end{cor}

\begin{cor}
Let $d$, $d'$ be integers with $d'>d>0$.  Then an element $v\in M(0)$ is
torsion iff for some $k>0$,
\[
\cS_{\eta,\eta';q;p}(0,k(ds+d'f))v=0.
\]
\end{cor}

\begin{proof}
The ``if'' direction has already been shown; for the only if direction,
note that if $v$ is torsion of order $d_1s+d'_1f$, then there exists
$k>0$ such that $k(d'-d)\ge d'_1-d_1$ and $d\ge kd_1$, and thus
$kd'-d'_1\ge kd-d_1\ge 0$, so that $v$ is annihilated by morphisms of
degree $k(ds+d'f)$ as required.
\end{proof}

\begin{prop}
The subcategory of torsion modules is a Serre subcategory of the category
of all $\cS_{\eta,\eta';q;p}$-modules.
\end{prop}

\begin{proof}
Since ``torsion'' is defined in terms of elements, it is preserved as a
property of modules by taking submodules.  In addition, the image of a
torsion element in a quotient module is clearly torsion, so any quotient of
a torsion module is torsion.  Finally, suppose we have a short exact
sequence
\[
0\to T_1\to M\to T_2\to 0
\]
with $T_1$, $T_2$ torsion.  For any element $v\in M(0)$, say, its image in
$T_2$ is torsion, and thus there exists $d'_2\ge d_2\ge 0$ such that
\[
\cS_{\eta,\eta';q;p}(0,d_2s+d'_2f)v\subset T_1.
\]
Since this Hom space is finite, we can uniformly choose $d'_1\ge d_1\ge 0$
such that
\[
\cS_{\eta,\eta';q;p}(d_2s+d'_2f,(d_1+d_2)s+(d'_1+d'_2)f)w = 0
\]
for all
\[
w\in \cS_{\eta,\eta';q;p}(0,d_2s+d'_2f)v.
\]
But then
\begin{align}
\cS_{\eta,\eta';q;p}&(0,(d_1+d_2)s+(d'_1+d'_2)f)
v\notag\\
&{}=
\cS_{\eta,\eta';q;p}(d_2s+d'_2f,(d_1+d_2)s+(d'_1+d'_2)f)
\cS_{\eta,\eta';q;p}(0,d_2s+d'_2f)
v\notag\\
&{}=
0,
\end{align}
so that $v$ is torsion as required.
\end{proof}

\begin{rem}
This also follows from the second corollary.  Indeed, the second corollary
tells us that we obtain the same notion of ``torsion'' upon restricting to
any subgroup $\Z(ds+d'f)$ with $d'>d>0$, and thus we may reduce to known
facts about $\Z$-algebras (which by the first corollary are generated in
degree 1).
\end{rem}

Of course for $\P^1\times \P^1$, there are far more ample bundles than just
those with $d'>d>0$, and this remains true for the noncommutative deformation.

\begin{prop}
Fix integers $d,d'>0$.  If $\eta'\notin p^\Z q^\Z\eta$, then an element $v\in
M(0)$ is torsion iff for some $k>0$,
\[
\cS_{\eta,\eta';q;p}(0,k(ds+d'f))v=0.
\]
\end{prop}

\begin{proof}
We first note that if $d,d'\ge 0$, then
\[
\cS_{\eta,\eta';q;p}(0,(d+1)s+d'f)
=
\cS_{\eta,\eta';q;p}(ds+d'f,(d+1)s+d'f)
\cS_{\eta,\eta';q;p}(0,ds+d'f),
\]
and
\[
\cS_{\eta,\eta';q;p}(0,ds+(d'+1)f)
=
\cS_{\eta,\eta';q;p}(ds+d'f,ds+(d'+1)f)
\cS_{\eta,\eta';q;p}(0,ds+d'f).
\]
By the Fourier transform, it suffices to prove the latter.  If $d'\ge d$,
we have already shown this; if $d'<d$, then again the Fourier transform
tells us that
\[
\cS_{\eta,\eta';q;p}(0,ds+d'f)
=
\cS_{\eta,\eta';q;p}((d-1)s+d'f,ds+d'f)
\cS_{\eta,\eta';q;p}(0,(d-1)s+d'f).
\]
But then
\begin{align}
&\cS_{\eta,\eta';q;p}(ds+d'f,ds+(d'+1)f)
\cS_{\eta,\eta';q;p}(0,ds+d'f)\notag\\
&{}=
\cS_{\eta,\eta';q;p}(ds+d'f,ds+(d'+1)f)
\cS_{\eta,\eta';q;p}((d-1)s+d'f,ds+d'f)
\cS_{\eta,\eta';q;p}(0,(d-1)s+d'f)\notag\\
&{}=
\cS_{\eta,\eta';q;p}((d-1)s+d'f,ds+(d'+1)f)
\cS_{\eta,\eta';q;p}(0,(d-1)s+d'f)\notag\\
&{}=
\cS_{\eta,\eta';q;p}(0,ds+(d'+1)f),
\end{align}
where the second step uses the fact that $\eta'\notin p^\Z q^\Z\eta$.

In particular, we find that in the above considerations, we may replace the
cone $d'\ge d\ge 0$ with the quadrant $d',d\ge 0$ and obtain the same
notions of torsion.  The result follows as before.
\end{proof}

In particular, given any symmetric elliptic difference equation, we may
turn the corresponding module into a sheaf using the above notions.  It
turns out that we can still recover the difference equation from this
sheaf.  This reduces to showing that the module of symmetric meromorphic
functions is already saturated with respect to torsion modules: i.e., it
contains no torsion elements, and any extension by a torsion module is
split.  This reduces to the following fact.  Let $P_0$ denote the
projective module generated by the object $0$; that is, $P_0$ is the functor
\[
P_0(ds+d'f) = \cS_{\eta,\eta';q;p}(-ds-d'f,0),
\]
with the obvious action.

\begin{prop}
Let $I\subset P_0$ be a submodule, and let $I':=T^{-1}(I\cap TP_0)$.
Then any morphism $\phi:I\to \Mer$ extends uniquely to $I'$.
\end{prop}

\begin{proof}
We certainly have an action of $T^{-1}$ on meromorphic functions,
so may define for $\oD\in I'$, $\psi(\oD)=T^{-1}\phi(T\oD)$.  This is certainly
a homomorphism, and if $\oD\in I$, then
\[
\psi(\oD) = T^{-1}\phi(T\oD) = T^{-1}T\phi(\oD)=\phi(\oD),
\]
so $\psi$ is the desired extension.  For uniqueness, note that if
$\psi:I'\to \Mer$ vanishes on $I$, then
\[
T\psi(\oD)=\psi(T\oD)=0
\]
for all $\oD\in I'$, so that $\psi(\oD)=0$.
\end{proof}

\begin{thm}
For any torsion module $M$, $\Hom(M,\Mer)=\Ext^1(M,\Mer)=0$.
\end{thm}

\begin{proof}
  If $v\in \Mer(0)$ is torsion of order $ds+d'f$, then the corresponding
  morphism $P_0\to \Mer$ restricts to the 0 morphism on the ideal generated
  by degree $ds+d'f$ elements in $P_0$.  If $d\ge 2$, then
\[
T
\cS_{\eta,\eta';q;p}(0,(d-2)s+(d'-2)f)
\subset
\cS_{\eta,\eta';q;p}(0,ds+d'f),
\]
and thus the morphism vanishes on the ideal generated by elements of degree
$(d-2)s+(d'-2)f$.  Thus we may as well assume that $d<2$; since the ideal
contains the ideal generated by elements of degree $(d+1)s+(d'+1)f$, we may
further assume that $d=0$.  But since elements of degree $d'f$ act on $\Mer$
as multiplication operators, there can certainly be no torsion elements of
order $d'f$.

For the extensions, we note that it suffices (using translation invariance)
to prove that
\[
\Ext^1(P_0/I,\Mer)=0
\]
for any submodule $I$ such that $P_0/I$ is torsion.  The long exact
sequence associated to $0\to I\to P_0\to P_0/I\to 0$ gives an
exact sequence
\[
\Hom(P_0,\Mer)\to \Hom(I,\Mer)\to \Ext^1(P_0/I,\Mer)\to \Ext^1(P_0,\Mer)=0.
\]
Now, since $P_0/I$ is torsion, $I$ contains the ideal generated by elements
of degree $ds+d'f$ for some $d'\ge d\ge 0$, and the same argument allows us
to reduce to the case $d=0$.  Again, since then some submodule of $I$ is
generated by multiplication operators, the result is immediate.
\end{proof}

In particular, if the module $M$ is a subquotient of the module $M'$ such
that the other pieces of the filtration are torsion, then
$\Hom(M,\Mer)=\Hom(M',\Mer)$.  Thus passing to sheaves preserves the
correspondence with solutions of a difference equation.  This is not quite
enough to let us reconstruct the equation from the sheaf, but should be
analogous to the commutative case, and simply rely on mild properties of
the direct and inverse image functors corresponding to the map to $\P^1$.
Without this, we need to remember the specific presentation as a quotient
of direct sums of the sheaves $\hat{P}_v$ associated to the projective
modules $P_v$ (analogues of line bundles).

In particular, we may have $\Hom(M,\Mer)=\Hom(M',\Mer)$ even if the other
pieces of the filtration are non-torsion, as long as they are sufficiently
small.  The point is that just as we can ``divide'' meromorphic functions
by $T$, we may also divide them by theta functions.  As a
result, the extension property for homomorphisms holds for any ideal
containing any element of degree a multiple of $f$.  This in particular
holds for the ideals $\cS_{\eta,\eta',x;q;p}(0,ds+d'f-re_1)$ we will
construct below.  This implies that extension by an extension of point
sheaves has no effect on maps to the sheaf of meromorphic functions; this
reflects the fact that such extensions correspond to isomonodromy
transformations.  (This will be discussed in more detail in Section
\ref{sec:diffeq2_fourier} below.)

\medskip

Now that we have a satisfactory notion of ``sheaf'', we can address the
question of ``missing'' operators in a more systematic way.  To be precise,
we wish to extend our $\Z^2$-algebra to a $\Z^2$-algebra
$\hat{\cS}_{\eta,\eta';q;p}$, in which the morphisms are now given by {\em
  sheaf} morphisms.  That is
\[
\hat\cS_{\eta,\eta';q;p}(d_1s+d'_1f,d_2s+d'_2f)
=
\Hom(\hat{P}_{-d_2s-d'_2f},\hat{P}_{-d_1s-d'_1f}).
\]
It turns out that the morphisms of $\hat{\cS}$ can still be identified
with difference operators.

Given two modules $M_1$, $M_2$ over $\cS_{\eta,\eta';q;p}$, the
sheaf morphisms from $M_1$ to $M_2$ are limits of module morphisms
\[
I_{ds+d'f}M_1\to M_2
\]
where $I_{ds+d'f}$ is the ideal generated by elements of degree $ds+d'f$;
the limit is taken with respect to the product partial order on the pair
$(d'-d,d)$.  Thus when $M_1=P_v$, we again reduce to a question of when
morphisms from such ideals extend.

\begin{lem}
  Fix integers $d,d'\in \Z$, and a submodule $I\subset P_0$.  If $T^k\in I$
  for some $k\ge 0$ and $I\not\subset TP_0$, then any morphism $\phi:I\to
  P_{ds+d'f}$ extends uniquely to $P_0$.
\end{lem}

\begin{proof}
Since $I\not\subset TP_0$, there exists an element $\oD\in I$ such that
$[T^0]\oD\ne 0$.  We may then write
\[
T^k\phi(\oD) = \phi(T^k \oD) = (T^k \oD T^{-k})\phi(T^k).
\]
Since the first $k$ coefficients of $T^k\phi(\oD)$ vanish, and the leading
coefficient of $T^k \oD T^{-k}$ is nonzero, we conclude that the first $k$
coefficients of $\phi(T^k)$ vanish, and thus that there exists a
well-defined element
\[
T^{-k}\phi(T^k)\in \cS(-ds-d'f,0)
\]
But this gives the image of $1$ under the desired extension, which is
unique since $\cS$ is a domain.
\end{proof}

In particular, for a general ideal $I\not\subset TP_0$, to extend a
map $\phi:I\to P_v$ to $P_0$, it will suffice to extend it to $I+TP_0$,
since the lemma tells us that the extension from $I+TP_0$ to $P_0$
exists uniquely.  We have a short exact sequence
\[
0\to I\to I+TP_0\to TP_0/(I\cap TP_0)\to 0,
\]
and thus we reduce to considering the extension problem from the ideal
$I\cap TP_0\subset TP_0=P_{2s+2f}$.  If we start with the ideal
generated by elements of degree $ds+d'f$ with $d'\ge d\ge 2$, this
reduces the problem to that for the ideal generated by elements of
degree $(d-2)s+(d'-2)f$.  As before, this lets us reduce to the case
$d=0$.  Of course, we cannot expect this extension problem to be solvable
in general, but this tells us that if we saturate with respect to the $d=0$
extension problem, then the resulting module is saturated.

Given a map $\phi:I\to P_v$ where $I\subset P_0$ is the ideal generated by
elements of degree $d'f$, then for any nonzero element $g$ of that degree,
we may define a meromorphic difference operator $g^{-1}\phi(g)$.  Since
\[
g\phi(h) = gh g^{-1}\phi(g) = h\phi(g) = \phi(gh),
\]
this operator is independent of $g$.  This extends to assign a meromorphic
operator to any element of $P_0$, which on $I$ agrees with $\phi$.  By the
above discussion, this allows us to assign a meromorphic difference
operator to any morphism in the saturated category $\hat{\cS}$.

\begin{lem}
The meromorphic operator $\oD$ is in $\hat\cS_{\eta,\eta';q;p}(-ds-d'f,0)$
iff for some $k\ge 0$ and every $BC_1(q\eta)$-symmetric theta function
$g$ of degree $k$, $g\oD\in \cS_{\eta,\eta';q;p}(-ds-d'f,kf)$.
\end{lem}

\begin{lem}
If $d\le 0$ or $d'\ge d-1$, then
\[
\hat\cS_{\eta,\eta';q;p}(0,ds+d'f)
=
\cS_{\eta,\eta';q;p}(0,ds+d'f).
\]
\end{lem}

\begin{proof}
If $d<0$, this is immediate, as no difference operator can have negative
order.  If $d=0$, then this follows by observing that the regular
representation of the $\Z$-algebra associated to $\P^1$ is already
saturated.

Thus suppose $d'\ge d-1\ge 0$.  Let $F$ be the sheaf on $\P^1$ obtained by
restricting $P_0$ to the $\Z$-algebra with objects $ds+\Z f$.  The module is
finitely generated, since it is generated by the elements of degree $ds+df$
together with finitely many objects of lower degree in $f$; as a result, we
obtain a coherent sheaf.  Moreover, the filtration by powers of $T$
corresponds to a filtration of this coherent sheaf in which the successive
quotients are contained in pushforwards of line bundles on $\C^*/\langle
p\rangle$.  Since we know the Hilbert polynomial of the coherent sheaf, and
it agrees with the sum of the Hilbert polynomials of the line bundles,
we find that $F$ is an iterated extension of those bundles.  

In particular, we have an extension
\[
0\to TF\to F\to \rho_*{\cal L}\to 0
\]
where ${\cal L}$ is a line bundle of degree $2d+2d'>0$ on our elliptic
curve, and is thus acyclic, as is its direct image.  Thus $F$ is acyclic
iff $TF$ is acyclic.  In other words, the claim holds for $ds+d'f$ iff it
holds for $(d-2)s+(d'-2)f$.  The base case is $d\in \{-1,0\}$, where the
claim is immediate.
\end{proof}

\begin{lem}
  Let $d'\ge 0$ be a nonnegative integer, and let $I\subset P_0$ be the
  ideal generated by elements of degree $d'f$.  Then $T^{\lceil
    d'/2\rceil}\in I$ unless $q^l\eta/\eta'\in p^\Z$ for some $1\le l\le
  d'$.
\end{lem}

\begin{proof}
It suffices to show that every element of degree $d's+d'f$ is contained in
$I$.  We can write any such element as a composition of elements of degree
$s+f$, and unless $q\eta/\eta'\in p^\Z$, each of those elements of
degree $s+f$ can be written as a composition $xy$ with $x$ of degree $s$
and $y$ of degree $f$.  We thus find that
\[
\cS_{\eta,\eta';q;p}(0,d's+d'f)
=
\cS_{\eta,\eta';q;p}((d'-1)s+d'f,d's+d'f)
\cS_{\eta,\eta';q;p}(f,(d'-1)s+d'f)
\cS_{\eta,\eta';q;p}(0,f)
\]
unless $q\eta/\eta'\in p^\Z$.  Since
\[
\cS_{\eta,\eta';q;p}(f,(d'-1)s+d'f)
=
\cS_{\eta,\eta'/q;q;p}(0,(d'-1)s+(d'-1)f),
\]
the claim follows by induction.
\end{proof}

\begin{cor}
Let $d$, $d'$ be integers with $d>\max(d',0)$.
Then
\[
\hat\cS_{\eta,\eta';q;p}(0,ds+d'f)
=
\cS_{\eta,\eta';q;p}(0,ds+d'f)
\]
unless $q^{-l}\eta/\eta'\in p^\Z$ for some $l$ with $0<l<\min(d-d',d+1)$.
\end{cor}

\begin{proof}
Any operator
\[
\oD\in \hat\cS_{\eta,\eta';q;p}(0,ds+d'f)
\]
becomes a morphism of $\cS$ when left-multiplied by any
$BC_1(\eta/q^{d-1})$-symmetric theta function of sufficiently large degree.
In fact, that degree need be no more than $d-d'-1$, since past that point
the module is already saturated.  In other words,
\[
\cS_{\eta,\eta';q;p}(ds+d'f,ds+(d-1)f)
\oD
\in
\cS_{\eta,\eta';q;p}(0,ds+(d-1)f).
\]
If
\begin{align}
\cS_{\eta,\eta';q;p}&(ds+(d-1)f,(2d-1-d')s+(d-1)f)
\cS_{\eta,\eta';q;p}(ds+d'f,ds+(d-1)f)\notag\\
&=
\cS_{\eta,\eta';q;p}(ds+d'f,(2d-1-d')s+(d-1)f),
\end{align}
then $T^{d-d'-1}\oD\in \cS$, and thus $\oD\in \cS$
as required.

If $d'<-1$, then $\cS_{\eta,\eta';q;p}(0,ds+d'f)=0$, and thus if
\[
\hat\cS_{\eta,\eta';q;p}(0,ds+d'f)
\ne
\cS_{\eta,\eta';q;p}(0,ds+d'f)
=0,
\]
then
\[
\hat\cS_{\eta,\eta';q;p}(0,ds-f)
\ne
0
=
\cS_{\eta,\eta';q;p}(0,ds-f),
\]
and thus $q^{-l}\eta/\eta'\in p^\Z$ for some $1\le l\le d$.
\end{proof}

\begin{cor}
If $\eta\notin p^\Z q^\Z\eta'$, then
$\hat\cS_{\eta,\eta';q;p}=\cS_{\eta,\eta';q;p}$.
\end{cor}

As we will see, the above conditions on $\eta$, $\eta'$ are necessary
conditions to have no additional operators in a given saturated Hom
space.  To see this, we will first need to construct an important case of
such a missing operator.

\begin{lem}
Fix an integer $d\ge 0$, and let $g_1$,\dots,$g_d$, $h_1,\dots,h_d$
be nonzero $BC_1(q^{1-d}\eta)$-symmetric theta functions of degree 1.  Then
\[
h_d\cdots h_1
\oD_{q^{1-d}\eta}(g_1)
\oD_{q^{2-d}\eta}(g_2)
\cdots
\oD_{\eta}(g_d)
=
g_d\cdots g_1
\oD_{q^{1-d}\eta}(h_1)
\oD_{q^{2-d}\eta}(h_2)
\cdots
\oD_{\eta}(h_d),
\]
and thus
\[
\oD_{d;\eta;q;p}
:=
(g_1\cdots g_d)^{-1}
\oD_{q^{1-d}\eta}(g_1)
\oD_{q^{2-d}\eta}(g_2)
\cdots
\oD_{\eta}(g_d)
\in
\hat{\cS}_{\eta,q^{-d}\eta;q;p}(0,ds-df).
\]
\end{lem}

\begin{proof}
Since
\[
\oD_{q^{k-d}\eta}(g_k)\oD_{q^{k+1-d}\eta}(g_{k+1})
=
\oD_{q^{k-d}\eta}(g_{k+1})\oD_{q^{k+1-d}\eta}(g_k),
\]
we can freely commute the $g$ and $h$ functions amongst themselves.
It thus remains to show that we can swap $g_1$ and $h_1$.  In other words,
we need to show
\[
h_1\oD_{q^{1-d}\eta}(g_1)=g_1\oD_{q^{1-d}\eta}(h_1).
\]
Since both sides equal $\oD_{q^{1-d}\eta}(g_1h_1)$, the claim is immediate.
\end{proof}

\begin{rem}
Of course, since the category is invariant up to isomorphism under
translation of $\eta$ and $\eta'$ by powers of $p$, we obtain such
additional elements whenever $q^d\eta'\in p^\Z\eta$.
\end{rem}

Note in particular that since we have a nontrivial section of
$\hat{\cS}_{\eta,q^{-l}\eta;q;p}(0,l(s-f))$, we find
\begin{align}
\dim(\hat{\cS}_{\eta,q^{-l}\eta;q;p}(0,ds+d'f))
&{}\ge
\dim(\hat{\cS}_{\eta,q^{-l}\eta;q;p}(ls-lf,ds+d'f))\notag\\
&{}\ge
(d-l+1)(d'+l+1)\notag\\
&{}>
(d+1)(d'+1)\notag\\
&{}=
\dim(\cS_{\eta,q^{-l}\eta;q;p}(0,ds+d'f))
\end{align}
as long as $0<l<\min(d-d',d+1)$.  This is already enough to tell us that
our condition above was tight.

To do better and determine precisely the structure of $\hat{\cS}$ where
it differs from $\cS$, we need to understand the case
$\eta/\eta'=q^l$ better.  It turns out that this case has an additional
significance.

\begin{lem}
Any operator in $\hat{\cS}_{\eta,q^{-l}\eta;q;p}(0,ds+d'f)$ 
takes $BC_1(q\eta)$-symmetric theta functions of degree $l-1$ to
$BC_1(q^{1-d}\eta)$-symmetric theta functions of degree $l+d'-d-1$.
\end{lem}

\begin{proof}
Given a $BC_1(q\eta)$-symmetric theta function $f$, and an operator
$\oD\in \hat{\cS}_{\eta,q^{-l}\eta;q;p}(0,ds+d'f)$, we may compute
$\oD\cdot f=(\oD f)\cdot 1$.  In this way, we immediately reduce to the case
$l=1$, and wish to understand the image of $1$ under such an operator.
The claim is certainly preserved under saturation: if $g\oD\cdot 1$
is a $BC_1(q^{1-d}\eta)$-symmetric theta function of degree $d'-d+1$ for
all $g$ symmetric of degree 1, then $\oD\cdot 1$ is symmetric of degree
$d'-d$.  As a result, we may restrict to the case $d'\ge d$.

Now, we have the factorization
\[
\hat{\cS}_{\eta,\eta/q;q;p}(0,ds+d'f)
=
\cS_{\eta,\eta/q;q;p}(0,ds+d'f)
=
\cS_{\eta,\eta/q;q;p}(ds+df,ds+d'f)
\cS_{\eta,\eta/q;q;p}(0,ds+df).
\]
Any operator in $\cS_{\eta,\eta/q;q;p}(ds+df,ds+d'f)$ clearly takes
constants to $BC_1(q^{1-d}\eta)$-symmetric theta functions of degree $d'-d$,
so we may reduce to the case $d'=d$.  Since any such operator is a linear
combination of products of operators of degree $s+f$, we reduce to that
case.  That is, we need to know that
\[
\frac{(b(z)-b(\eta/z))z}{\theta_p(z^2/\eta)}
\]
is constant for any holomorphic function $b$ with
$b(pz)=(\eta/pz^2)^2 b(z)$.  Since this function is elliptic, it suffices
to prove it holomorphic, which is easily checked pole-by-pole.
\end{proof}

Since 
\[
\hat{\cS}_{\eta,q^{-l}\eta;q;p}(0,ds+d'f)
=
\hat{\cS}_{\eta,\eta;q;p}(lf,ds+(d'+l)f),
\]
and the property $\eta=\eta'$ is invariant under twisting by multiples of
$s+f$, we can restate this fact in the following form.

\begin{cor}
The $\hat{\cS}_{\eta,\eta;q;p}$-module $\Mer$ has a natural submodule
$\Skl$ such that $\Skl(ds+d'f)$ is the space of
$BC_1(q^{1-d}\eta)$-symmetric theta functions of degree $d'-d-1$.
\end{cor}

\begin{rem}
  If we restrict this module to the $\Z$-algebra $lf+\Z(s+f)$, we may view
  it as an $l$-dimensional module over Sklyanin's noncommutative $\P^3$
  \cite{SklyaninEK:1982}.  This family of representations is precisely the
  ``principal analytic series'' constructed in \cite{SklyaninEK:1983}.
  From a geometric perspective, $\Skl$ is the noncommutative analogue of
  the sheaf $O_{s-f}(-1)$ on $F_2$ (that is, the sheaf $O(-1)$ on the
  section of self-intersection $-2$).
\end{rem}

In particular, if $d'\le d-l$, then any operator in
$\hat{\cS}_{\eta,q^{-l}\eta;q;p}(0,ds+d'f)$ {\em annihilates}
all $BC_1(q\eta)$-symmetric theta functions of degree $l-1$, since
it is supposed to take such functions to $BC_1(q^{1-d}\eta)$-symmetric theta
functions of negative degree.  This in particular applies to the
operator $\oD_{l;\eta;q;p}$.  In fact, this is an $l$-th order operator, and
is annihilating an $l$-dimensional space, so this essentially characterizes
$\oD_{l;\eta;q;p}$.  This is not quite true, as there are some technicalities
when $q$ is torsion.  We do, however, have the following fact.

\begin{lem}
Suppose $l>0$ is a positive integer and that the subgroup of $\C^*/\langle
p\rangle$ generated by $q$ has order at least $l$.  If
\[
\oD\in \hat{\cS}_{\eta,q^{-l}\eta;q;p}(0,ds+d'f)
\]
annihilates every $BC_1(q\eta)$-symmetric theta function of degree $l-1$,
then
\[
\oD\in \hat{\cS}_{\eta,q^{-l}\eta;q;p}(l(s-f),ds+d'f) \oD_{l;\eta;q;p}.
\]
\end{lem}

\begin{proof}
This is clearly preserved by saturation, so we may as well assume $d'\ge
d$.  If $d>l$, then we find
\[
[T^0]\hat{\cS}_{\eta,q^{-l}\eta;q;p}(0,ds+d'f)
=
[T^0](\hat{\cS}_{\eta,q^{-l}\eta;q;p}(l(s-f),ds+d'f) \oD_{l;\eta;q;p}),
\]
since both sides surject on to the space of global sections of the
same line bundle.  We may thus subtract a suitable multiple of
$\oD_{l;\eta;q;p}$ from $\oD$ to arrange that $[T^0]\oD=0$ without affecting
the fact that $\oD$ annihilates the space of theta functions.  Then
$T^{-1}\oD\in \hat{\cS}_{\eta,q^{-l}\eta;q;p}(0,(d-2)s+(d'-2)f)$
again annihilates the space of $BC_1(q\eta)$-symmetric theta functions
of degree $l-1$.  Thus, by induction, we reduce to the case $d\le l$.
By the following lemma, there are no operators of order $<l$ that
annihilate all $BC_1(q\eta)$-symmetric theta functions of degree $l-1$,
and such an operator of order $l$ is unique up to left multiplication by a
meromorphic function.  In particular, when $d=l$, we obtain an operator of
the form $g \oD_{l;\eta;q;p}$ for some meromorphic function $g$.  Comparing
the coefficients of $T^0$ and $T^l$ shows that $g(\eta/q^{l-1}z)=g(z)$, and
since $g \oD_{l;\eta;q;p}\in \hat{\cS}_{\eta,q^{-l}\eta;q;p}(0,ls+d'f)$,
we find that $f$ must be holomorphic.
\end{proof}

\begin{lem}
Let $f_1$, \dots, $f_l$ be a basis of the space of $BC_1(\eta)$-symmetric
theta functions of degree $l-1$.  Then
\[
\det_{1\le i,j\le l} f_i(z_j)\ne 0
\]
unless $z_i\in p^\Z z_j\cup p^\Z \eta/z_j$ for some $i<j$.  In particular,
as a function of $z$,
\[
\det_{1\le i,j\le l} f_i(q^{j-1} z)\ne 0
\]
unless $q^j\in p^\Z$ for some $1\le j<l$.
\end{lem}

\begin{proof}
The second claim is an immediate consequence of the first, which follows
from the usual degree considerations.  To be precise, the given conditions
certainly force the determinant to vanish, and thus the determinant is a
holomorphic multiple of
\[
\prod_{i<j} z_i^{-1}\theta_p(z_i/z_j,z_iz_j/\eta).
\]
Since this is already a $BC_1(\eta)$-symmetric theta function of degree
$l-1$ in each variable, the remaining factor of the determinant is a
constant, nonzero since $f_1$,\dots, $f_l$ are a basis.
\end{proof}

\begin{cor}
If $q^\Z\cap p^\Z=1$, then
\[
\dim(\hat\cS_{\eta,\eta';q;p}(0,ds+d'f))=(d+1)(d'+1)
\]
unless $\eta/\eta'\in q^l p^\Z$ for some $l$ with $1\le l\le
\min(d-d',d+1)$, when
\[
\dim(\hat\cS_{\eta,\eta';q;p}(0,ds+d'f))=(d-l+1)(d'+l+1).
\]
\end{cor}

\begin{proof}
Indeed, we either have
\[
\hat\cS_{\eta,\eta';q;p}(0,ds+d'f)
=
\cS_{\eta,\eta';q;p}(0,ds+d'f)
\]
or
\[
\hat\cS_{\eta,\eta';q;p}(0,ds+d'f)
=
\cS_{\eta,\eta';q;p}(ls-lf,ds+d'f)\oD_{l;\eta;q;p},
\]
as described.  (When $l=d-d'$, both are true.)
\end{proof}

\begin{cor}
  Let $q$ have order $r$ in $\C^*/\langle p\rangle$, and let $1\le l\le r$.
  Then
\[
\hat{\cS}_{\eta,q^{-l}\eta;q;p}(0,k(s-f))
=
0
\]
unless $k\in r\N\cup l+r\N$, when it has dimension $1$.
Moreover,
\[
\hat{\cS}_{\eta,q^{-l}\eta;q;p}(0,ds+d'f)
=
\hat{\cS}_{\eta,q^{-l}\eta;q;p}(k(s-f),ds+d'f)
\hat{\cS}_{\eta,q^{-l}\eta;q;p}(0,k(s-f))
\]
where
\[
k=\begin{cases}
0 & 0\le d-d'<l\\
mr+l   & \text{$2mr+l\le d-d'<(2m+1)r+l,$ some $m\ge 0$,}\\
(m+1)r & \text{$(2m+1)r+l\le d-d'<(2m+2)r+l$, some $m\ge 0$}.
\end{cases}
\]
\end{cor}

\begin{proof}
We have
\[
\hat{\cS}_{\eta,q^{-l}\eta;q;p}(0,ds+d'f)
=
\hat{\cS}_{\eta,q^{-l}\eta;q;p}(l(s-f),ds+d'f)
\oD_{l;\eta;q;p}
\]
as long as $l\le \min(d+1,d-d')$.  Now,
\[
\hat{\cS}_{\eta,q^{-l}\eta;q;p}(l(s-f),ds+d'f)
=
\hat{\cS}_{q^{-l}\eta,\eta;q;p}(0,(d-l)s+(d'+l)f).
\]
If $l=r$, then
\[
\hat{\cS}_{q^{-r}\eta,\eta;q;p}(0,(d-l)s+(d'+l)f).
=
\hat{\cS}_{\eta,\eta;q;p}(0,(d-l)s+(d'+l)f).
=
\hat{\cS}_{\eta,q^{-r}\eta;q;p}(0,(d-l)s+(d'+l)f),
\]
while if $1\le l<r$,
\[
\hat{\cS}_{q^{-l}\eta,\eta;q;p}(0,(d-l)s+(d'+l)f)
=
\hat{\cS}_{q^{-l}\eta,q^{l-r}(q^{-l}\eta);q;p}(0,(d-l)s+(d'+l)f).
\]
In particular, we may factor out $\oD_{r-l;\eta;q;p}$ or $\oD_{r;\eta;q;p}$
as appropriate, unless the appropriate inequality holds.  The result
follows by induction.
\end{proof}

\begin{prop}
  If $q$ is a primitive $r$-th root of unity, then for $1\le l\le r$, then
  the $\Hom$ space $\hat{\cS}_{\eta,q^{-l}\eta;q;p}(0,r(s-f))$ is
  spanned by the operator
\[
\frac{z^r}{\theta_{p^r}(z^{2r}/\eta^r)}(1-T^r).
\]
\end{prop}

\begin{proof}
If $l=r$, this operator and $\oD_{r;\eta;q;p}$ have the same
leading coefficient and annihilate the space of $BC_1(q\eta)$-symmetric
theta functions of degree $r-1$, so are the same operator.

For $1\le l<r$, we find that for any $BC_1(q\eta)$-symmetric theta function
$g$ of degree $r-l$,
\[
\frac{z^r}{\theta_{p^r}(z^{2r}/\eta^r)}(1-T^r)g
\in
\hat{\cS}_{\eta,\eta;q;p}(0-(r-l)f,r(s-f))
=
\hat{\cS}_{\eta,q^{-l}\eta;q;p}(0,rs-lf)
\]
But $T^r$ commutes with $g$, so that
\[
g\frac{z^r}{\theta_{p^r}(z^{2r}/\eta^r)}(1-T^r)
\in
\hat{\cS}_{\eta,q^{-l}\eta;q;p}(0,rs-lf)
\]
for all $g$, implying that
\[
\frac{z^r}{\theta_{p^r}(z^{2r}/\eta^r)}(1-T^r)
\in
\hat{\cS}_{\eta,q^{-l}\eta;q;p}(0,rs-rf)
\]
as required.
\end{proof}

Now, there exists a sheaf $V_{d;\eta,\eta';q;p}$ on $\P^1$ for each $d\ge
0$ such that
\[
\hat\cS_{\eta,\eta';q;p}(0,ds+d'f)
=
\Gamma(V_{d;\eta,\eta';q;p}(d')),
\]
and it will be useful below to understand the structure of this sheaf.
It would, in fact, be simple enough to work out the structure explicitly,
but for our purposes, the following inductive fact will suffice.

\begin{lem}\label{lem:calS_bundles}
  Let $l$ be the smallest positive integer such that $\eta/\eta'\in p^\Z
  q^l$, or $\infty$ if no such integer exists.  If $l>d$, then
\[
V_{d;\eta,\eta';q;p}\cong \sO_{\P^1}^{d+1},
\]
and otherwise
\[
V_{d;\eta,\eta';q;p}
\cong
V_{d-l;q^{-l}\eta,q^l \eta';q;p}(l)
\oplus
\sO_{\P^1}(l-d-1)^l.
\]
\end{lem}

\begin{proof}
  Since $\hat\cS_{\eta,\eta';q;p}$ is a domain (inherited from the algebra
  of difference operators), the sheaf $V_{d;\eta,\eta';q;p}$ is
  torsion-free, so a vector bundle on $\P^1$.  Moreover, since we know how
  many global sections sufficiently large twists of this sheaf have, we may
  compute its Hilbert polynomial: $h(t;V_{d;\eta,\eta';q;p})=(d+1)(t+1)$.
  If $\eta/\eta'\notin p^\Z q^l$ for $1\le l\le d$, then
  $V_{d;\eta,\eta';q;p}(-1)$ has no global sections, so is acyclic, and
  thus $V_{d;\eta,\eta';q;p}\cong \sO_{\P^1}^{d+1}$.  Otherwise, let $l$ be
  the smallest such integer.  Then left-multiplication by
  $\oD_{l;\eta;q;p}$ induces the injective map in a short exact sequence
\[
0\to V_{d-l;q^{-l}\eta,q^l\eta';q;p}(l)\to V_{d;\eta,\eta';q;p}
 \to {\cal Q}\to 0.
\]
We know the Hilbert polynomials of the first two sheaves, and thus find
$h(t;{\cal Q})=l(t-d+l)$.  Now,
$h^0(V_{d;\eta,\eta';q;p}(d-l))=(d+1)(d-l+1)=\chi(V_{d;\eta,\eta';q;p}(d-l))$,
and thus $V_{d;\eta,\eta';q;p}(d-l)$ is acyclic.  Moreover, as a quotient
of an acyclic sheaf, the same holds for ${\cal Q}(d-l)$, and since
$\chi({\cal Q}(d-l))=0$, we find that ${\cal Q}(d-l)$ also has no global
sections; it follows that ${\cal Q}\cong \sO_{P^1}(l-d-1)^l$.  Now,
\[
V_{d-l;q^{-l}\eta,\eta;q;p}(d-1)
\]
is acyclic (since $d-1\ge d-l-1$), and thus the above short exact sequence
must split, giving
\[
V_{d;\eta,q^{-l}\eta;q;p}
\cong
V_{d-l;q^{-l}\eta,\eta;q;p}(l)
\oplus
\sO_{P^1}(l-d-1)^l
\]
as required.
\end{proof}

Of course, each $\Hom$ space in $\hat{\cS}$ is not only a $\P^1$-module on
the left, but also on the right.  We may thus view the morphisms in
$\hat{\cS}$ as global sections of suitable sheaves on $\P^1\times \P^1$.
Each such sheaf has locally free direct image under either projection, and
the family of sheaves is closed under twisting by line bundles.  It follows
that the support of the sheaf must be $1$-dimensional, with no component
contained in a fiber of either ruling.  In other words, $\hat{\cS}$ is the
space of global sections of a ``coherent sheaf bimodule'' over $\P^1$ in
the terminology of \cite{VandenBerghM:2012}, and these fit together to form
a sheaf $\Z$-algebra.  That is, for each $d_1$, $d_2\in \Z$, we have a
sheaf $S_{d_1,d_2}$ on $\P^1\times \P^1$ such that
\[
\Gamma(S_{d_1,d_2}(-d'_1,d'_2))
=
\hat{\cS}(d_1s+d'_1f,d_2s+d'_2f),
\]
and the composition law induces a morphism
\[
\pi_{13*}(\pi_{12}^*S_{d_1,d_2}\otimes \pi_{23}^*S_{d_2,d_3})
\to
S_{d_1,d_3},
\]
where $\pi_{ij}$ are the appropriate projection maps from
$\P^1\times\P^1\times \P^1$.  Now, Van den Bergh defined the category of
sheaves on such a sheaf $\Z$-algebra.  When translated back to $\hat{\cS}$,
we obtain essentially the same definition as above, with one important
difference: our definition of torsion above had quantifiers
\[
\exists D. \exists D'. \forall d\ge D. \forall d'\ge D'-D+d,
\]
while Van den Bergh's definition involves the quantifiers
\[
\exists D. \forall d\ge D. \exists D'. \forall d'\ge D'.
\]
(More precisely, Van den Bergh's definition states that for $d\ge D$, every
image of $v$ in degree $ds+*f$ is torsion as an element of the
corresponding horizontal $\Z$-module.)  Of course, it follows immediately
from Lemma \ref{lem:torsion_is_easy} that the two notions are equivalent.
In other words, the category of sheaves on $\cS_{\eta,\eta';q;p}$ is
precisely the same as the category of sheaves on the corresponding sheaf
$\Z$-algebra.

Since $\cS$ is generated in degrees $s$, $f$, $s+f$, it follows that
the sheaf $\Z$-algebra is generated in degree 1.  The generating sheaves
$S_{d,d+1}$ are easily seen to be supported on the curve $\C/\langle
p\rangle$, embedded in $\P^1\times\P^1$ by the pair of involutions
$z\mapsto q^{1-d}\eta/z$, $z\mapsto q^{-d}\eta/z$.  Moreover, $S_{d,d+1}$
is an invertible sheaf on its support, namely the line bundle corresponding
to theta functions with multiplier $q\eta'/pz^2$.  The quadratic relations
similarly form a sheaf $R_{d,d+2}$ on $\P^1\times \P^1$ with direct images
of rank 1.  In this case, the two spaces of theta functions are related by
translation by $q$, so there is a natural way to identify the coordinates.
If we do so, then we find that $R_{d,d+2}$ is the structure sheaf of the
diagonal.  Moreover, this subsheaf of $S_{d,d+1}\otimes S_{d+1,d+2}$ is
nondegenerate in the sense of \cite{VandenBerghM:2012}; indeed,
otherwise $\hat{\cS}$ would have a zero-divisor.  Since $S_{d_1,d_2}$
has rank $d_2-d_1+1$, we conclude that the category of sheaves on
$\hat{\cS}$ is a noncommutative Hirzebruch surface.

We should note in this context that the representation in difference
operators is new (and extends to all noncommutative ruled surfaces
satisfying a certain separability condition \cite{noncomm2}), as are the
presence of a canonical family of subcategories $\cS$ in which all $\Hom$
spaces are flat, and the above formulae for the dimensions of the $\Hom$
spaces in $\hat{\cS}$.

\medskip

We can also give a formula for the operators $\oD_{l;\eta;q;p}$ in general.
We work by first assuming that $q$ is nontorsion, then take a limit to get the
torsion case.  We can calculate $\oD_{l;\eta;q;p}$ in several different ways:
e.g., we know the leading coefficient, so can solve for the coefficients using
Cramer's rule and the fact that it annihilates all symmetric theta
functions of degree $l-1$.  Another approach uses the fact that
\[
g_1\cdots g_l \oD_{l;\eta;q;p}
=
\oD_{q^{1-l}\eta;q;p}(g_1)
\oD_{q^{2-l}\eta;q;p}(g_2)
\cdots
\oD_{\eta;q;p}(g_l)
=
g_1\cdots g_{l-1}\oD_{l-1;\eta/q;q;p}
\oD_{\eta;q;p}(g_l),
\]
so that for any $BC_1(q^{1-l}\eta)$-symmetric theta function $g$ of degree 1,
\[
g
\oD_{l;\eta;q;p}
=
\oD_{l-1;\eta/q;q;p}
\oD_{\eta;q;p}(g).
\]
If $g$ vanishes at $w$, we can use this to give a simple first-order
recurrence for the values of the coefficients of $\oD_{l-1;\eta/q;q;p}$ at
$w$.  Either approach gives the following answer.

\begin{prop}
For all integers $l\ge 0$,
\[
\oD_{l;\eta;q;p}
=
\sum_{0\le k\le l}
\frac{q^{lk} \theta_p(q^{-l};q)_k}
     {\theta_p(q;q)_k}
\frac{z^l \theta_p(q^{2k-1} z^2/\eta)}
     {\theta_p(q^{k-1} z^2/\eta;q)_{l+1}}
T^k.
\]
\end{prop}

\begin{rem}
Note that the coefficient
\[
\frac{q^{lk} \theta_p(q^{-l};q)_k}
     {\theta_p(q;q)_k}
\]
is essentially the elliptic analogue of a $q$-binomial coefficient.  In
particular, the zeros of the denominator at torsion $q$ are all cancelled
by zeros of the numerator, so the limit to torsion $q$ always exists.
Also, in the limit $p\to 0$, we obtain the $l$-th power of the Askey-Wilson
lowering operator, according to a formula of \cite{CooperS:2002}.  One
should also compare the formulas of \cite{IRS} and
\cite{SchlosserMJ/YooM:2016}.  Note that the latter (up to shifting by a
half-integer power of $q$) can be more directly interpreted as giving the
generator of the $1$-dimensional $\Hom$ space ${\cal
  S}'_{\eta,x_0;q;\C^*/p}(0,ds)$ in the algebraic category ${\cal S}'$
constructed below.
\end{rem}

It is worth noting one last thing about these operators.  Although these
operators cause the Fourier transform to fail for
$\hat{\cS}_{\eta,\eta';q;p}$, they also in a sense {\em induce} the Fourier
transform.  Suppose $l>d'-d$ (and $q$ is not torsion, for simplicity), and
consider an operator
\[
\oD\in \hat{\cS}_{\eta,q^{d'-d-l}\eta;q;p}(0,ds+d'f)
   = \cS_{\eta,q^{d'-d-l}\eta;q;p}(0,ds+d'f)
\]
Such an operator takes $BC_1(q\eta)$-symmetric theta functions
of degree $d-d'+l-1$ to functions of degree $l-1$.  It follows, therefore,
that $\oD_{l;q^{-d}\eta;q;p}\oD$ annihilates all $BC_1(q\eta)$-symmetric theta
functions of degree $d-d'+l-1$.  As a result, there exists an operator
$\hat{\oD}$ such that
\[
\oD_{l;q^{-d}\eta;q;p}\oD
=
\hat{\oD} \oD_{d-d'+l;\eta;q;p}
\]
In this way, we obtain an operator in
\[
\hat{\cS}_{q^{d'-d-l}\eta,\eta;q;p}(0,d's+df)
=
\cS_{q^{d'-d-l}\eta,\eta;q;p}(0,d's+df).
\]

Since this acts on parameters and objects just as the Fourier transform, we
might expect that it agrees on $\cS$, and this is indeed the case.  It
suffices to check this for elements of degree $f$ and $s$, where we have
already seen
\[
g \oD_{l;\eta;q;p}
=
\oD_{l-1;\eta/q;q;p}
\oD_{\eta;q;p}(g)
\]
and taking adjoints easily gives
\[
\oD_{q^{1-l}\eta;q;p}(g)
\oD_{l-1;\eta;q;p}
=
\oD_{l;\eta;q;p} g.
\]

Though this only works for special values of the parameters, we can use the
above formula to analytically continue the operators $\oD_{l;\eta;q;p}$ to
formal operators, and use this to define the Fourier transform whenever $q$
is not torsion.  We avoid this approach here, since the existing approach
works for all parameters, but this idea will be the key to defining the
multivariate analogue of the transform in \cite{elldaha}.  It may be worth
keeping this approach in mind even in the univariate case, since it gives a
way to compute transforms of operators without having to express them in
terms of the generators.  For more discussion of these formal operators
(and in particular the multivariate analogue), see \cite[\S 4]{kernel}.

\section{$F_1$ and $\P^2$}

Before we discuss blowing up $\cS_{\eta,\eta';q;p}$, it will be useful
to consider the analogous deformation of the Hirzebruch surface $F_1$.  Of
course, we would want to consider that case in any event, as it is one of
the only two kinds of elliptic rational surface that is {\em not} a blow up
of $\P^1\times \P^1$ or $F_2$.  (The other surface, $\P^2$, will appear as
a $\Z$-algebra contained in the $F_1$ deformation, though as we will see,
there are some subtle issues that arise in that case.)

The Picard lattice of $F_1$ is spanned by elements $s$ and $f$ with
$s^2=-1$ (the unique section of minimal self-intersection) and $f^2=0$
(the class of a fiber), with $s\cdot f=1$.  An operator of degree $ds+d'f$
in the corresponding $\Z^2$-algebra should again be a $d$-th order difference
operator, and the formula given for the multiplier of the leading
coefficient in the $\P^1\times \P^1$ case suggests that for $F_1$ that
multiplier should be
\[
(-1)^d z^{-d-2d'} q^{d(d+1)/2+d'-dd'} (\eta/p)^{d'} x_0^d
\]
for a morphism from 0 to $ds+d'f$, and the appropriate ratio more
generally.  (Of course, the above guess is not particularly unique, but
turns out to have the best behavior with respect to elementary
transformations.)

This suggests that we should take the following generators:
\begin{itemize}
\item
If $g(z)$ is a $BC_1(q^{1-d}\eta)$-symmetric theta function of
  degree 1, then
\[
g(z)\in \cS'_{\eta,x_0;q;p}(ds+d'f,ds+(d'+1)f).
\]
\item
If $h(z)$ is a holomorphic function with $h(pz)/h(z)=-q^{d-d'+1}x_0/z$,
then
\[
\oD_{q^{-d}\eta;q;p}(h)\in \cS'_{\eta,x_0;q;p}(ds+d'f,(d+1)s+d'f).
\]
\item
 If $b(z)$ is a holomorphic function with $b(pz)/b(z)=-q^{1-d'}\eta
  x_0/pz^3$, then
\[
\oD_{q^{-d}\eta;q;p}(b)\in \cS'_{\eta,x_0;q;p}(ds+d'f,(d+1)s+(d'+1)f).
\]
\end{itemize}
We define $\cS'_{\eta,x_0;q;p}$ to be the $\Z s+\Z f$-algebra generated by
these morphisms (with composition via multiplication of difference
operators).  Note that this again satisfies an invariance under twisting,
in the form
\[
\cS'_{\eta,x_0;q;p}(d_1s+d'_1f,d_2s+d'_2f)
=
\cS'_{q^{-d_1}\eta,q^{d_1-d'_1}x_0;q;p}(0,(d_2-d_1)s+(d'_2-d'_1)f)
\]

We should note here that since the generators of degree $s$ depend on a
theta function with a single zero, they are unique up to scalar
multiplication, and have the form
\[
\oD_{q^{-d}\eta;q;p}(\theta_p(z/q^{d-d'+1}x_0)).
\]
The multiples of this operator account for a $2$-dimensional subspace of
the $3$-dimensional space of operators of degree $s+f$, so we again have a
set of generators of degree 4 (i.e., each object is the domain of four
generators).  These difference operators again satisfy commutation
relations.

\begin{lem}
For any holomorphic function $b(z)$ with $b(pz)/b(z)=-q\eta x_0/pz^3$,
\[
\oD_{\eta/q;q;p}(\theta_p(z/qx_0))
\oD_{\eta;q;p}(b)
=
\oD_{\eta/q;q;p}(b)
\oD_{\eta;q;p}(\theta_p(z/qx_0)).
\]
\end{lem}

\begin{proof}
  Again, the coefficients of $T^0$ and $T^2$ in the difference vanish,
  while the coefficient of $T$ is a $BC_1(\eta/q)$-symmetric theta
  function of degree $-2$, so also vanishes.
\end{proof}

Also, as in the even case, we can identify the general first-order operator
that appears as a morphism in our $\Z^2$-algebra.

\begin{lem}
If $b(z)$ is holomorphic with
\[
b(pz)/b(z) = -q x_0 (\eta/p)^{d'} z^{-2d'-1},
\]
then
\[
\oD_{\eta;q;p}(b)\in \cS'_{\eta,x_0;q;p}(0,s+d'f).
\]
\end{lem}

\begin{proof}
For $d'<0$, there is no such function, while for $d'=0$ and $d'=1$, the
given operator is a generator.  For $d'>1$, the result follows from Lemma 
\ref{lem:bundle_products} by induction.
\end{proof}

The canonical class in $F_1$ is $-2s-3f$, so we should expect to find
the operator $T$ in the negative of that degree.

\begin{lem}
We have
\[
T\in \cS'_{\eta,x_0;q;p}(ds+d'f,(d+2)s+(d'+3)f).
\]
\end{lem}

\begin{proof}
  From Lemma \ref{lem:central_elt}, we have for any $\eta'\in \C^*$ an
  expansion of the form
\[
T = \sum_i \oD_{\eta/q;q;p}(b_{i1})\oD_{\eta;q;p}(b_{i2}),
\]
where
\begin{align}
b_{i1}(pz) &= (\eta\eta'/p^2qz^4)b_{i1}(z),\notag\\
b_{i2}(pz) &= (q\eta\eta'/p^2z^4)b_{i2}(z),\notag
\end{align}
and $b_{i1}(q\eta'/x_0)=0$ for all $i$.  Suppose $|q|<1$, and write
\begin{align}
T
&=
\Gampq(q^2z\eta'/\eta x_0,q\eta'/x_0z)
T
\Gampq(qz\eta'/\eta x_0,q^2\eta'/x_0z)^{-1}\notag\\
&{}=
\sum_i
\Gampq(q^2z\eta'/\eta x_0,q\eta'/x_0z)
\oD_{\eta/q;q;p}(b_{i1})
\Gampq(q^2z\eta'/\eta x_0,q^2\eta'/x_0z)^{-1}\notag\\
&\hphantom{{}=\sum_i{}}
\Gampq(q^2z\eta'/\eta x_0,q^2\eta'/x_0z)
\oD_{\eta;q;p}(b_{i2})
\Gampq(qz\eta'/\eta x_0,q^2\eta'/x_0z)^{-1}\notag\\
&{}=
\sum_i
\oD_{\eta/q;q;p}(b_{i1}\theta_p(q\eta'/x_0z)^{-1})
\oD_{\eta;q;p}(b_{i2}\theta_p(qz\eta'/\eta x_0)).
\end{align}
Since this expansion is an algebraic identity, it continues to hold for
general $q$.  The result follows by observing that the first factor of each
summand is an element of $\cS'_{\eta,x_0;q;p}(s+2f,2s+3f)$, while the
second factor is an element of $\cS'_{\eta,x_0;q;p}(0,s+2f)$.
\end{proof}

We can then prove flatness as in the even case; the main difference is that
now the presentation is {\em always}, rather than merely generically valid.

\begin{thm}
For any integers $d,d'$,
\[
\dim(\cS'_{\eta,x_0;q;p}(0,ds+d'f))=
\begin{cases}
0 & d<0\\
(d+1)(2d'+2-d)/2 & 0\le d\le d'\\
(d'+1)(d'+2)/2 & d\ge d'.
\end{cases}
\]
In particular, $\cS'_{\eta,x_0;q;p}$ is a flat family of $\Z^2$-algebras.
\end{thm}

\begin{proof}
  The dimension is certainly 0 if $d<0$.  If $d>d'$, then we have
\[
\cS'_{\eta,x_0;q;p}(0,ds+d'f)
\supset
\cS'_{\eta,x_0;q;p}((d-1)s+d'f,ds+d'f)
\cS'_{\eta,x_0;q;p}(0,(d-1)s+d'f),
\]
so that
\[
\dim(\cS'_{\eta,x_0;q;p}(0,ds+d'f))
\ge
\dim(\cS'_{\eta,x_0;q;p}(0,(d-1)s+d'f)).
\]
If $d<d'$, then we find that all functions with the appropriate
multiplier and divisor occur as leading coefficients, and thus
\[
\dim(\cS'_{\eta,x_0;q;p}(0,ds+d'f))
\ge
d+2d'
+
\dim(\cS'_{\eta,x_0;q;p}(0,(d-2)s+(d'-3)f)).
\]
If $d=d'$, we may combine the two inequalities to find
\[
\dim(\cS'_{\eta,x_0;q;p}(0,d(s+f)))
\ge
3d
+
\dim(\cS'_{\eta,x_0;q;p}(0,(d-3)(s+f)).
\]
We thus find by induction that the claimed formula gives a lower bound on
the dimension.

To prove the corresponding upper bound, we choose four generators as
described above, and note that using quadratic relations, we may express
any element of $\cS'_{\eta,x_0;q;p}(0,ds+d'f)$ in the form
\[
\sum_k
\sum_{0\le j\le d'-k}
\alpha_{kj}
x_0^{d'-k-j} x_1^j
y_0^{d-k} y_1^k,
\]
where $x_0$, $x_1$ are the two generators of degree $f$, and $y_0$, $y_1$
are the generators of degree $s$, $s+f$ respectively.
This gives the upper bound
\[
\dim(\cS'_{\eta,x_0;q;p}(0,ds+d'f))
\ge
\sum_{0\le k\le \min(d,d')} (d'-k+1),
\]
agreeing with the above formula.
\end{proof}

\begin{cor}
If $d>d'\ge 0$, then
\begin{align}
\cS'_{\eta,x_0;q;p}(0,ds+d'f)
&{}=
\cS'_{\eta,x_0;q;p}(d's+d'f,ds+d'f)
\cS'_{\eta,x_0;q;p}(0,d's+d'f)\notag\\
&{}=
\cS'_{\eta,x_0;q;p}((d-d')s,ds+d'f)
\cS'_{\eta,x_0;q;p}(0,(d-d')s).
\end{align}
\end{cor}

\begin{rem}
  This should be compared to the fact that on a commutative surface, if the
  divisor class $D$ has negative intersection with a $-1$-curve, then it
  has that $-1$-curve as a fixed component with multiplicity the negative
  of the intersection.  (Indeed, the intersection pairing on $F_1$ has $s\cdot
  (ds+d'f)=d'-d$.)  Note that the behavior of $-1$-curves in the
  noncommutative setting is thus much simpler than that of $-2$-curves.
\end{rem}

Analogously to the even case, we may define sheaves on $\cS'$ by
calling an element torsion if for some $d'\ge d\ge 0$ it is annihilated by
all elements of degree $(d+d_1)s+(d'+d'_1)f$ for any $d'_1\ge d_1\ge 0$.
(Again, this is the cone of divisors which are nef on all commutative
surfaces in the family.)  The same arguments carry over to the odd case, so
in particular we find that both the natural representation in symmetric
meromorphic functions and the regular representation are already saturated.

Just as in the $F_0/F_2$ case, we could have obtained the corresponding
category of sheaves from the construction of \cite{VandenBerghM:2012}.
Indeed, we find that a homogeneous element $v$ of a module is torsion iff
there exists $d_0$ such that for $d\ge d_0$, any image of $v$ under a
morphism of degree $ds+*f$ is torsion w.r.to morphisms of degree a multiple
of $f$ (i.e., is 0 as an element of the appropriate sheaf on $\P^1$).  This
exhibits our category of sheaves as the category of sheaves associated to a
sheaf $\Z$-algebra on $\P^1$.  As in the $F_0/F_2$ case, we find that the
resulting noncommutative scheme is a noncommutative Hirzebruch surface.

\medskip

The odd case has another useful special feature.  Since the morphisms of
degree $s$ are unique up to scalar multiplication, they give rise to a
well-defined family of difference equations.  It turns out that we can use
the solutions to those equations to produce a canonical version of the
$\Z^2$-algebra in which the morphisms are operators with {\em elliptic} (as
opposed to meromorphic theta function) coefficients.  To be precise, if
$|q|<1$, the function
\[
\Gampq(z/q^{d+1-d'}x_0,\eta/q^{2d-d'}x_0z)
\]
is annihilated by all sections of $\cS'_{\eta,x_0;q;p}(ds+d'f,(d+1)s+d'f)$,
and has the following property.

\begin{prop}\label{prop:elliptic_gauge_F1}
Suppose $|q|<1$. For any operator
\[
\oD\in \cS'_{\eta,x_0;q;p}(d_1s+d'_1f,d_2s+d'_2f),
\]
the coefficients of the operator
\[
\Gampq(z/q^{d_2+1-d_2'}x_0,\eta/q^{2d_2-d_2'}x_0z)^{-1}
\oD
\Gampq(z/q^{d_1+1-d_1'}x_0,\eta/q^{2d_1-d_1'}x_0z)
\]
are invariant under $z\mapsto pz$.
\end{prop}

\begin{proof}
This transformation of the operators respects composition, as does the
property of having elliptic coefficients.  It thus remains only to show
that the generators transform to elliptic difference operators.
For degree $f$, we have
\begin{align}
&\Gampq(z/q^{d-d'}x_0,\eta/q^{2d-d'-1}x_0z)^{-1}
g(z)
\Gampq(z/q^{d+1-d'}x_0,\eta/q^{2d-d'}x_0z)\notag\\
&{}=
\frac{g(z)}
     {\theta_p(z/q^{d+1-d'}x_0,\eta/q^{2d-d'}x_0z)};
\end{align}
for degree $s$, we have
\begin{align}
&\Gampq(z/q^{d+2-d'}x_0,\eta/q^{2d+2-d'}x_0z)^{-1}
\oD_{\eta;q;p}(\theta_p(z/q^{d-d'+1}x_0))
\Gampq(z/q^{d+1-d'}x_0,\eta/q^{2d-d'}x_0z)\notag\\
&{}=
\frac{z\theta_p(q^{d'-d-1}z/x_0,q^{d'-d-2}z/x_0,q^{d'-2d-1}\eta/x_0z,q^{d'-2d-2}\eta/x_0z)}{\theta_p(q^{d}z^2/\eta)}
(1-T);
\end{align}
and for degree $s+f$ we have
\begin{align}
&\Gampq(z/q^{d+1-d'}x_0,\eta/q^{2d+1-d'}x_0z)^{-1}
\oD_{\eta;q;p}(b)
\Gampq(z/q^{d+1-d'}x_0,\eta/q^{2d-d'}x_0z)\notag\\
&{}=
\oD_{\eta;q;p}(b(z)\theta_p(q^{d'-2d-1}\eta/x_0z)).
\end{align}
In each case, we can directly check that the coefficients are elliptic.
\end{proof}

This allows us to define an analogue of $\cS'$ on an {\em algebraic}
elliptic curve.  That is, if $C$ is a smooth genus 1 curve (now over an
arbitrary field, possibly of finite characteristic) and we are given
parameters $\eta\in \Pic^2(C)$, $x_0\in \Pic^1(C)$, $q\in \Pic^0(C)$, then
we can define a $\Z^2$-algebra $\cS'_{\eta,x_0;q;C}$ generated by the
following elliptic $q$-difference operators.  Given $\eta\in \Pic^2(C)$, we
define $R_\eta$ to be the corresponding ramification divisor; i.e.,
$R_\eta$ is the divisor associated to the subscheme of fixed points of the
involution $z\mapsto \eta/z$.
\begin{itemize}
\item If $g\in \Gamma(C;\sO([q^{d+1-d'}x_0][\eta/q^{2d-d'}x_0]))$, then
$g\in \cS'_{\eta,x_0;q;C}(ds+d'f,ds+(d'+1)f)$.
\item If $h\in \Gamma(C;\sO(R_{q^{-d}\eta}/[q^{d+1-d'}x_0][q^{d+2-d'}x_0][\eta/q^{2d+1-d'}x_0][\eta/q^{2d+2-d'}x_0]))$,
  then
\[
h(z)+h(q^{-d}\eta/z)T=h(z)(1-T)\in \cS'_{\eta,x_0;q;C}(ds+d'f,(d+1)s+d'f).
\]
\item If $b\in \Gamma(C;\sO(R_{q^{-d}\eta}/[\eta/q^{2d+1-d'}x_0]))$, then
\[
b(z)+b(q^{-d}\eta/z)T\in \cS'_{\eta,x_0;q;C}(ds+d'f,(d+1)s+(d'+1)f).
\]
\end{itemize}

\begin{rem}
Note that in general, we have
\[
[T^0]\cS_{\eta,x_0;q;C}(0,ds+d'f)
\subset
\Gamma(C;
\sO(
\prod_{0\le k<d} R_{q^{-k}\eta}
\prod_{0\le k<d'-d} [q^{1-k}x_0]
\prod_{0\le k<d'-2d} [\eta/q^{-k}x_0]
)),
\]
extended in the natural way if $d'<d$ or $d'<2d$; more generally, the
leading coefficient of an element of
$\cS_{\eta,x_0;q;C}(d_1s+d'_1f,d_2s,d'_2f)$ is naturally a section of the
line bundle associated to the corresponding ratio of divisors.
\end{rem}

The above argument carries over directly; the only nontrivial issue is the
commutation relations, which follow by taking a limit.

\begin{cor}
This family of categories is flat.
\end{cor}

The above description is clearly functorial in the quadruple $(C,\eta,x_0,q)$,
immediately implying that $\cS'_{\eta,x_0;q;p}$ is modular.
Two important special cases of functoriality are the isomorphisms
\begin{align}
\cS'_{\eta,x_0;q;C}&\cong \cS'_{\tau^2\eta,\tau x_0;q;C}\notag\\
                        &\cong \cS'_{\rho^2/\eta,\rho/x_0;1/q;C}\notag
\end{align}
for any $\tau\in \Pic^0(C)$, $\rho\in \Pic^2(C)$; both isomorphisms
act as the identity on objects.  We also have an analogue of
twist-invariance:
\[
\cS'_{\eta,x_0;q;C}
\cong
\cS'_{q^{-d}\eta,q^{d-d'}x_0;q;C},
\qquad
d_0s+d'_0f\mapsto (d+d_0)s+(d'+d'_0)f,
\]
as well as an analogue of the formal adjoint:
\[
\cS'_{\eta,x_0;q;C}
\cong
\cS'_{\eta,x_0,1/q;C}
\qquad
d_0s+d'_0f \mapsto 2s+3f-(d_0s+d'_0f).
\]
Analytically, the latter is formally an adjoint with respect to the
family of (elliptic) densities
\[
\frac{z\theta_p(z/q^{d-d'}x_0,z/q^{d+1-d'}x_0,\eta/q^{2d-d'-1}x_0z,\eta/q^{2d-d'}x_0z)}
     {\theta_p(q^{d-1}z^2/\eta)}.
\]
Algebraically, the same holds, except that we now need to choose for each
$ds+d'f$ an elliptic function with the appropriate divisor (or,
equivalently, a generator of $\cS'((d-1)s+d'f,ds+d'f)$).  We should
note in particular that the adjoint involution is thus only defined up to a
scalar gauge automorphism.  This is mostly unavoidable, but we note that if
we choose the scalars so that
$\Delta_{(d+2)s+(d'+3)f}(z)=\Delta_{ds+d'f}(qz)$, then the adjoint of $T$
will be $T$ in every degree.

We will have use for the following cancellation lemma in the sequel.

\begin{lem}\label{lem:cancel_F1}
Let $b_l(z)$ be a holomorphic function such that
\[
b_l(pz)/b_l(z) = -(q^2\eta/pz^2)^l x_0/z,
\]
and such that if $b_l(z)/\theta_p(z/u_1,z/u_2)$ is holomorphic, then
$q\eta/u_1u_2\notin p^\Z q^\Z$.  Then for any meromorphic difference
operator $\oD$ of order $d$, we have
\[
\oD\in \cS'_{\eta,x_0;q;p}(0,ds+d'f)
\]
iff
\[
\oD \oD_{q\eta;q;p}(b_l)
\in
\cS'_{\eta,x_0;q;p}(-s-lf,ds+d'f).
\]
\end{lem}

\begin{proof}
  We certainly have $\oD_{q\eta;q;p}(b)\in \cS'_{\eta,x_0;q;p}(-s-lf,0)$, so
  one direction is immediate.  For the other direction, we note first that
  since $\cS'_{\eta,x_0;q;p}$ is saturated, it suffices to show that
  all left-multiples of $\oD$ by symmetric theta functions of sufficiently
  large degree are in $\cS'_{\eta,x_0;q;p}$.  In particular, we may feel
  free to assume $d'\ge d$.

Now, let $F_0$ be a formal symbol on which meromorphic difference operators
act via
\[
T F_0 = \frac{b_l(z)}{b_l(q\eta/z)} F_0.
\]
(For $|p|,|q|<1$, we can take $F_0$ to be a suitable product of elliptic
Gamma functions.)  This symbol is annihilated by $\oD_{q\eta;q;p}(b_l)$, and
thus by the composition $\oD \oD_{q\eta;q;p}(b_l)$.  It will thus suffice to
show that any element of $\cS'_{\eta,x_0;q;p}(-s-lf,ds+d'f)$ that
annihilates $F_0$ is in the image of $\cS'_{\eta,x_0;q;p}(0,ds+d'f)$
under right-multiplication by $\oD_{q\eta;q;p}(b_l)$.  Equivalently, we need
to show that the space
\[
\cS'_{\eta,x_0;q;p}(-s-lf,ds+d'f)F_0
\]
has dimension
\[
\dim\cS'_{\eta,x_0;q;p}(-s-lf,ds+d'f)
-
\dim\cS'_{\eta,x_0;q;p}(0,ds+d'f)
=
(l-1)d + 2l+d'
\]
Now, if we define
\[
F_k = \prod_{0\le i<k} b_l(\eta/q^{i-1} z)^{-1} F_0,
\]
then we find
\[
\oD_{q^{1-k}\eta;q;p}(b(z))
F_k
=
\frac{
b(z)
b_l(\eta/q^{k-1} z)
-
b(\eta/q^{k-1} z)
b_l(z)
}
{z^{-1}\theta_p(q^{k-1} z^2/\eta)}
F_{k+1}.
\]
If $b(z)$ is holomorphic with
\[
b(pz)/b(z)
= 
(\eta/p q^{k-1} z^2)^j
(-q^{(k+1)l}x_0/z),
\]
then the coefficient here is a $BC_1(q^{1-k}\eta)$-symmetric theta function
of degree $l+j-1$.  It follows by an easy induction that
\[
\cS'_{\eta,x_0;q;p}(-s-lf,ds+d'f)
F_0
\]
is contained in
\[
\cS'_{\eta,x_0;q;p}(-s-lf+(d+1)(s-(l-1)f),ds+d'f)
F_{d+1}.
\]
The latter has exactly the required dimension, so we need to prove that
the map is surjective; again an inductive argument reduces to the case
$d=0$.  This then reduces to understanding when
\[
b(z)b_l(q\eta/z)-b(q\eta/z)b_l(z)=0.
\]
The condition on the multiplier of $b$ then implies that
\[
b(z) b_l(q\eta/z)
\]
is a $BC_1(q\eta)$-symmetric theta function of degree $j+l+1$.
Since $b_l$ is not a multiple of any $BC_1(q\eta)$-symmetric theta function,
it follows that
\[
\frac{b(z)}{b_l(z)}
\]
must itself be a $BC_1(q\eta)$-symmetric theta function.
In particular, it follows that the map
\[
\cS'_{\eta,x_0;q;p}(-s-lf,-f)
\to
\cS'_{\eta,x_0;q;p}(-(2l-1)f,-f)
\]
(the case $d=0$, $d'=-1$ above) is injective.  For $l>0$, both spaces have
dimension $2l-1$, and the map is thus an isomorphism, implying surjectivity
for all $d'\ge -1$.  For $l=0$, both spaces are trivial, but in that case
it is easy to show that
\[
\cS'_{\eta,x_0;q;p}(-s,f)F_0 = \C F_1,
\]
again giving surjectivity.
\end{proof}

\medskip

In the commutative case, $F_1$ is the blow up of $\P^2$ in one point.  We
should thus be able to recover a noncommutative deformation of $\P^2$ by
``blowing down'' $s$.  In the commutative setting, if $X$ is a blowup of
$Y$ in a single point, then $\Pic(Y)\subset \Pic(X)$ can be identified as
the orthogonal complement of the exceptional curve, and this identifies the
natural category with objects $\Pic(Y)$ as a subcategory of the
corresponding category for $X$.  In our case, we should thus obtain a
noncommutative $\P^2$ by restricting $\cS'_{\eta,x_0;q;C}$ to the
objects $\Z(s+f)$.  And, indeed, the dimension formula above agrees with
the Hilbert function of $\P^2$ as expected.

Here, though, we encounter a slightly odd effect.  The spaces of operators
we obtain clearly depend on $\eta$, $x_0$ and $q$, and we thus obtain a
$2+1$-parameter family.  On the other hand, the moduli space of elliptic
curves in $\P^2$ is only $1$-dimensional, so we should only have a
$2$-parameter family of deformations.  This is indeed the case, but relates
to a somewhat unexpected isomorphism.

\begin{lem}
The $\Z$-algebra $\cS'_{\eta,x_0;q;C}|_{\Z(s+f)}$ is
generated in degree $s+f$.
\end{lem}

\begin{proof}
We need to show
\[
\cS'_{\eta,x_0;q;C}(0,(d+1)(s+f))
=
\cS'_{\eta,x_0;q;C}(d(s+f),(d+1)(s+f))
\cS'_{\eta,x_0;q;C}(0,d(s+f)).
\]
Since the generators of $\cS'$ have degrees $f$, $s$, and $s+f$,
we thus reduce to showing
\begin{align}
\cS'_{\eta,x_0;q;C}&((d+1)s+df,(d+1)(s+f))
\cS'_{\eta,x_0;q;C}(0,(d+1)s+df)\notag\\
&{}\subset
\cS'_{\eta,x_0;q;C}(d(s+f),(d+1)(s+f))
\cS'_{\eta,x_0;q;C}(0,d(s+f))
\end{align}
and
\begin{align}
\cS'_{\eta,x_0;q;C}&(ds+(d+1)f,(d+1)(s+f))
\cS'_{\eta,x_0;q;C}(0,ds+(d+1)f)\notag\\
&{}\subset
\cS'_{\eta,x_0;q;C}(d(s+f),(d+1)(s+f))
\cS'_{\eta,x_0;q;C}(0,d(s+f)).
\end{align}
For the first claim, we note that $s\cdot ((d+1)s+df)<0$, and thus
\begin{align}
&\cS'_{\eta,x_0;q;C}((d+1)s+df,(d+1)(s+f))
\cS'_{\eta,x_0;q;C}(0,(d+1)s+df)\notag\\
&{}=
\cS'_{\eta,x_0;q;C}((d+1)s+df,(d+1)(s+f))
\cS'_{\eta,x_0;q;C}(ds+df,(d+1)s+df)
\cS'_{\eta,x_0;q;C}(0,ds+df)\notag\\
&{}\subset
\cS'_{\eta,x_0;q;C}(ds+df,(d+1)(s+f))
\cS'_{\eta,x_0;q;C}(0,ds+df).
\end{align}
For the second claim, we note that the proof of flatness shows that
\[
\cS'_{\eta,x_0;q;C}(0,ds+(d+1)f)
=
\cS'_{\eta,x_0;q;C}(ds+df,ds+(d+1)f)
\cS'_{\eta,x_0;q;C}(0,ds+df),
\]
and the claim again follows immediately.
\end{proof}

In particular, we find that we obtain a $\Z$-algebra in which the
generators can be identified with global sections of appropriate line
bundles of degree $3$, and those generators satisfy the same quadratic
relations as those global sections.  This presentation appeared in
\cite{ArtinM/TateJ/VandenBerghM:1990}, where it was shown that the result
was a flat deformation of $\P^2$.  It follows in particular that our
$\Z$-algebra is quadratic, and isomorphic to that constructed in
\cite{BondalAI/PolishchukAE:1993}, the $\Z$-algebra analogue of the usual
three-generator Sklyanin algebra.  (Note that in sharp contrast to the
quadratic relations for $\cS$ and $\cS'$, these relations do not in general
form a Gr\"obner basis.)

In particular, we find that $\cS'_{\eta,x_0;q;C}|_{\Z(s+f)}$ depends up to
isomorphism only on $q$ and the product bundle $\eta x_0$ (a fact we will
give a direct proof of below), but the resulting $\Z$-algebra has a
$1$-parameter family of difference operator representations.  We should
note here that this does {\em not} give rise in any obvious way to a
difference operator representation of the three-generator Sklyanin {\em
  algebra}.  The difficulty here is that the ``central'' element (the
unique operator of degree $3d+3f$ with vanishing leading coefficient)
starts as a {\em third}-order operator, and thus remains a nontrivial
difference operator even when its extreme coefficients vanish.  In
particular, we cannot in general divide by this operator and still hope to
obtain a difference operator.  With care, we can divide by the operator
{\em formally}, and even, given a choice of cube root of $q$, construct a
suitable cube root of the operator; the result is a representation of the
three-generator Sklyanin algebra in {\em formal} difference operators.  We
postpone discussion of this to the multivariate case \cite{elldaha}.

One should also note that, as one might expect, the category of sheaves on
our noncommutative $F_1$ is indeed the same category as would be obtained
from a noncommutative blowup of the above noncommutative $\P^2$.  Indeed,
it is straightforward to show that a homogeneous element $v$ is torsion iff
there exists $r_0$ such that for $r>r_0$, any image of $v$ under a morphism
of degree $d(s+f)-rs$ is torsion w.r.to the restriction to the
corresponding coset of $\Z(s+f)$.  This exhibits our category as the
category of sheaves on a sheaf $\Z$-algebra on the relevant noncommutative
$\P^2$ (which is essentially independent of $r$).  We furthermore see that
the sheaf bimodule corresponding to $\Z(s+f)-s$ is precisely the two-sided
ideal of a point, and since the sheaf algebra is generated in degree 1, our
category is the category of sheaves on a noncommutative blowup (in the
sense of \cite{VandenBerghM:1998}) of the appropriate noncommutative
$\P^2$.  Note that Van den Bergh's blowup construction requires that the
point be contained in a commutative curve which is a ``divisor'', but this
follows from the classification of point modules over $3$-generator
Sklyanin algebras \cite{ArtinM/TateJ/VandenBerghM:1991}.  (Specifically, in
our case, the commutative curve $(\cong C)$ is cut out by the two-sided
ideal $\ker [T^0]$, which is (locally) principal: if
\[
  \oD\in \ker[T^0]\cap \cS'_{\eta,x_0;q;C}(d_1(s+f),d_2(s+f)),
\]
then there is a factorization $\oD=T\oD_1\oD_2$ with $\oD_2\in
\cS'_{\eta,x_0;q;C}(d_1(s+f),(d_2-3)(s+f))$ and $\oD_1$ the unique
(mod scalars) section of
$\cS'_{\eta,x_0;q;C}((d_2-3)(s+f),(d_2-3)(s+f)+s)$, as well as a
corresponding factorization on the right.)

\section{Surfaces with $K_X^2=7$}

While we could extend the above definition via generators to one-point
blowups of the above Hirzebruch surfaces, it will be more convenient to
construct the blowups by imposing appropriate conditions on the morphisms.
We will essentially define the blowup as the largest $\Z^3$-algebra that
blows down to both $\P^1\times \P^1$ and $F_1$.

It turns out that the elementary transformation is relatively
straightforward at the level of difference operators, with one mild
technicality.  Roughly speaking, the transformation involves gauging by a
suitable product of elliptic Gamma functions:
\[
\Gampq(q^r z/x,q^{r-d+1}\eta/x z).
\]
Of course, this only makes sense per se when $|q|<1$; though we could
always arrange this by multiplying $q$ by a sufficiently large power of
$p$, with a little care, we can make sense of the gauge transformation
regardless.  To do this, define a system of {\em formal} symbols
\[
\Gamm{r;q;p}(z;x)
\]
with an action of $T$, such that
\[
T \Gamm{r;q;p}(z;x) = \Gamm{r+1;q;p}(z;x) = \theta_p(q^r z/x) \Gamm{r;q;p}(z;x).
\]
If $|q|<1$, this can be instantiated by the meromorphic functions
$\Gampq(q^r z/x)$, or more precisely by the product of any such function by
a $q$-elliptic function.  In general, though, we should think of
$\Gamm{r;q;p}(z;x)$ as a family of explicitly cohomologous $1$-cocycles.
We thus have well-defined operations
\[
\oD \mapsto \Gamm{r_2;q;p}(z;x)^{-1} \oD \Gamm{r_1;q;p}(z;x)
\]
with the property that
\[
(\Gamm{r_3;q;p}(z;x)^{-1} \oD_2 \Gamm{r_2;q;p}(z;x))
(\Gamm{r_2;q;p}(z;x)^{-1} \oD_1 \Gamm{r_1;q;p}(z;x))
=
\Gamm{r_3;q;p}(z;x)^{-1} \oD_2 \oD_1 \Gamm{r_1;q;p}(z;x).
\]
If
\[
\oD = \sum_{0\le k\le d} c_k(z) T^k,
\]
then
\[
\Gamm{r_2;q;p}(z;x)^{-1} \oD \Gamm{r_1;q;p}(z;x)
=
\sum_{0\le k\le d}
c_k(z)
\theta_p(q^{r_2}z/x;q)_{r_1+k-r_2}
T^k,
\]
where we recall that
\[
\theta_p(z;q)_r
=
\begin{cases}
\prod_{0\le j<r} \theta_p(q^j z) & r\ge 0\\
\prod_{-r\le j<0} \theta_p(q^j z)^{-1} & r\le 0
\end{cases}
\]
Note in particular that
\[
\Gamm{r_2;q;p}(z;q^{-r}x)^{-1} \oD \Gamm{r_1;q;p}(z;q^{-r}x)
=
\Gamm{r_2+r;q;p}(z;x)^{-1} \oD \Gamm{r_1+r;q;p}(z;x).
\]
Of course, these formal gauge transformations commute with each other, so
we need not specify the order in which we perform such operations.

In particular, we may represent the function
\[
\Gampq(q^r z/x,q^{r-d+1}\eta/x z)
=
\frac{\Gampq(q^r z/x)}{\Gampq(q^{d-r} p xz/\eta)}
\sim
\frac{\Gamm{r;q;p}(z;x)}
     {\Gamm{d-r;q;p}(z;\eta/px)}.
\]
We should note, of course, that though it is somewhat obscured by the
notation, this is still a cocycle for the infinite dihedral group,
corresponding to the fact that the original product of Gamma functions is
invariant under $z\mapsto q^{1-d}\eta/z$.  We also note that Proposition
\ref{prop:elliptic_gauge_F1} above extends naturally to all $q$ if we
replace the Gamma functions as above; we find that
\[
\frac{\Gamm{2d_2+1-d'_2;q;p}(z;\eta/px_0)}
     {\Gamm{d'_2-d_2-1;q;p}(z;x_0)}
\oD
\frac{\Gamm{d'_1-d_1-1;q;p}(z;x_0)}
     {\Gamm{2d_1+1-d'_1;q;p}(z;\eta/px_0)}
\]
is an elliptic difference operator whenever
\[
\oD\in \cS'_{\eta,x_0;q;p}(d_1 s+d'_1f,d_2 s+d'_2f).
\]

\begin{defn}
$\cS_{\eta,\eta',x_1;q;p}$ is the category with objects
$\Z s+\Z f+\Z e_1$ and $\Hom$ spaces
\[
\cS_{\eta,\eta',x_1;q;p}(d_1 s+d'_1 f-r_{11}e_1,d_2 s+d'_2 f-r_{21}e_1)
\subseteq
\cS_{\eta,\eta';q;p}(d_1 s+d'_1 f,d_2 s+d'_2 f)
\]
consisting of those operators $\oD$ such that
\begin{align}
\frac{\Gamm{d_2-r_{21};q;p}(z;\eta/px_1)}
     {\Gamm{r_{21};q;p}(z;x_1)}
&
\oD
\frac{\Gamm{r_{11};q;p}(z;x_1)}
     {\Gamm{d_1-r_{11};q;p}(z;\eta/px_1)}
\notag\\
&{}\in
\cS'_{\eta,\eta'/px_1;q;p}(d_1s+(d_1+d'_1-r_{11})f,d_2s+(d_2+d'_2-r_{21})f).
\end{align}
Similarly, $\cS'_{\eta,x_0,x_1;q;p}$ is the category with objects $\Z s+\Z
f+\Z e_1$ and $\Hom$ spaces
\[
\cS'_{\eta,x_0,x_1;q;p}(d_1 s+d'_1 f-r_{11}e_1,d_2 s+d'_2 f-r_{21}e_1)
\subseteq
\cS'_{\eta,x_0;q;p}(d_1 s+d'_1 f,d_2 s+d'_2 f)
\]
consisting of those operators $\oD$ such that
\begin{align}
\frac{\Gamm{d_2-r_{21};q;p}(z;\eta/px_1)}
     {\Gamm{r_{21};q;p}(z;x_1)}
&\oD
\frac{\Gamm{r_{11};q;p}(z;x_1)}
     {\Gamm{d_1-r_{11};q;p}(z;\eta/px_1)}
\notag\\
&{}\in
\cS_{\eta,x_0\eta/x_1}(d_1 s+(d'_1-r_{11}) f,d_2 s+(d'_2-r_{21}) f)
\end{align}
\end{defn}

Note that each $\Hom$ space in either category is an intersection of a
$\Hom$ space and a gauge transformation of a $\Hom$ space, and thus the
given conditions are indeed closed under composition of difference
operators.

The definition clearly inherits a twisting symmetry from
$\cS_{\eta,\eta';q;p}$ and $\cS'_{\eta,x_0;q;p}$, giving the following.

\begin{lem}
There are isomorphisms
\begin{align}
\cS_{\eta,\eta',x_1;q;p}
&\cong
\cS_{q^{-d}\eta,q^{-d'}\eta',q^{-r_1}x_1;q;p}\notag\\
\cS'_{\eta,x_0,x_1;q;p}
&\cong
\cS'_{q^{-d}\eta,q^{d-d'}x_0,q^{-r_1}x_1;q;p}\notag
\end{align}
which act on objects as $v\mapsto v-ds-d'f+r_1e_1$, and on
morphisms as the identity.
\end{lem}

This, as usual, allows us to mostly restrict our attention to morphisms
from the origin.  Note that if
\[
\oD=\sum_{0\le k\le d} c_k(z) T^k,
\]
then
\[
\frac{\Gamm{d-r_1;q;p}(z;\eta/px_1)}
     {\Gamm{r_1;q;p}(z;x_1)}
\oD
\frac{\Gamm{0;q;p}(z;x_1)}
     {\Gamm{0;q;p}(z;\eta/px_1)}
=
\sum_{0\le k\le d}
\frac{c_k(z)}
     {\theta_p(q^k z/x_1;q)_{r_1-k}
      \theta_p(\eta/x_1 z q^{k-1};q)_{r_1-d+k}}
T^k.
\]

\begin{lem}\label{lem:elem_xform1}
We have
\begin{align}
\frac{\Gamm{d-r_1;q;p}(z;\eta/px_1)}
     {\Gamm{r_1;q;p}(z;x_1)}
&\cS_{\eta,\eta',x_1;q;p}(0,d s+d' f-r_1e_1)
\frac{\Gamm{0;q;p}(z;x_1)}
     {\Gamm{0;q;p}(z;\eta/px_1)}
\notag\\
&{}=
\cS'_{\eta,\eta'/px_1,\eta/px_1;q;p}(0,ds+(d+d'-r_1)f-(d-r_1)e_1)
\end{align}
and
\begin{align}
\frac{\Gamm{d-r_1;q;p}(z;\eta/px_1)}
     {\Gamm{r_1;q;p}(z;x_1)}
&\cS'_{\eta,x_0,x_1;q;p}(0,ds+d'f-r_1e_1)
\frac{\Gamm{0;q;p}(z;x_1)}
     {\Gamm{0;q;p}(z;\eta/px_1)}
\notag\\
&{}=
\cS_{\eta,x_0\eta/x_1,\eta/px_1;q;p}(0,d s+(d'-r_1) f-(d-r_1)e_1)
\end{align}
\end{lem}

\begin{proof}
The relevant constraints on the difference operators are manifestly
equivalent.
\end{proof}

In particular, we find that the two categories differ merely by a
combination of a relabelling of the objects and a gauge transformation, and
are in particular isomorphic.  This is, of course, hardly surprising; it is
simply a noncommutative analogue of the fact that $\P^1\times \P^1$ and
$F_1$ become isomorphic upon blowing up a point.

\begin{lem}
We have
\begin{align}
T&\in \cS_{\eta,\eta',x_1;q;p}(0,2s+2f-e_1),\notag\\
T&\in \cS'_{\eta,x_0,x_1;q;p}(0,2s+3f-e_1).\notag
\end{align}
Moreover, for any $v\in \Z s+\Z f+\Z e_1$, we have short exact sequences
\[
0\to T\cS_{\eta,\eta',x_1;q;p}(0,v-2s-2f+e_1)
 \to \cS_{\eta,\eta',x_1;q;p}(0,v)
 \to [T^0]\cS_{\eta,\eta',x_1;q;p}(0,v)
 \to 0
\]
and
\[
0\to T\cS'_{\eta,x_0,x_1;q;p}(0,v-2s-3f+e_1)
 \to \cS'_{\eta,x_0,x_1;q;p}(0,v)
 \to [T^0]\cS'_{\eta,x_0,x_1;q;p}(0,v)
 \to 0.
\]
\end{lem}

\begin{proof}
Indeed, $T$ is invariant under the gauge transformation, so is indeed in
the given $\Hom$ spaces.  Similarly, if $[T^0]\oD=0$, then $T^{-1}\oD$ lies in
the appropriate $\Hom$ space as required.
\end{proof}

The reason for the negative signs above on the coefficients of $e_1$ is
that when those coefficients are positive, there is no constraint.

\begin{lem}
For any $d,d'\in \Z$, $r_1\ge 0$,
\begin{align}
\cS_{\eta,\eta',x_1;q;p}(0,ds+d'f+r_1e_1) & = \cS_{\eta,\eta';q;p}(0,ds+d'f)\notag\\
\cS'_{\eta,x_0,x_1;q;p}(0,ds+d'f+r_1e_1) & = \cS'_{\eta,x_0;q;p}(0,ds+d'f)\notag
\end{align}
\end{lem}

\begin{proof}
Since we have one inclusion by definition, it will suffice to prove the
claim when $d=0$, $d'=1$ and when $d=1$, when it follows from a
straightforward calculation.
\end{proof}

Conjugating by the standard elementary transformation gives the following.

\begin{lem}
For any $d,d'\in \Z$, $r_1\ge d$,
\begin{align}
\cS_{\eta,\eta',x_1;q;p}(0,ds+d'f-r_1e_1)
&=
\theta_p(q^d z/x_1,q\eta/x_1z;q)_{r_1-d}
\cS_{\eta,\eta',x_1;q;p}(0,ds+d'f-de_1)
\notag\\
\cS'_{\eta,x_0,x_1;q;p}(0,ds+d'f-r_1e_1)
&=
\theta_p(q^d z/x_1,q\eta/x_1z;q)_{r_1-d}
\cS'_{\eta,x_0,x_1;q;p}(0,ds+d'f-de_1).\notag
\end{align}
\end{lem}

Now, for $d\ge 0$, $r\in \Z$, there is a natural sheaf $V_{d,r}$ on $\P^1$
such that
\[
\cS_{\eta,\eta',x_1;q;p}(0,ds+d'f-re_1)
\subset
\Gamma(V_{d,r}(d')),
\]
with equality for sufficiently large $d'$.  Indeed, this is simply the
sheaf associated to the graded module
$\cS_{\eta,\eta',x_1;q;p}(0,ds+*f-re_1)$ over a $\Z$-algebra isomorphic to
$\P^1$; this has no torsion elements, since a torsion element would be a
zero divisor.  Similarly, the sheaves $V_{d,r}$ are all torsion-free, and
coherent (they are either equal to one or bounded between two sheaves we
already understand).  We conclude in particular that $V_{d,r}$ is a vector
bundle for all $d\ge 0$, $r\in \Z$, of rank $d+1$.

\begin{lem}
The vector bundle $V_{d,r}$ has Hilbert polynomial
\[
h(t;V_{d,r})
=
(d+1)(t+1)-\begin{cases}
0 & r\le 0\\
r(r+1)/2 & 0\le r\le d\\
(2r-d)(d+1)/2 & d\le r
\end{cases}
\]
Moreover, for $0<r<d$, there is a natural short exact sequence
\[
0\to V_{d-2,r-1}(-2)\to V_{d,r}\to \rho_*{\cal L}\to 0,
\]
where ${\cal L}$ is a line bundle on $C$ of degree $2d-r$.
\end{lem}

\begin{proof}
If $r\le 0$ or $r\ge d$, this follows from the identification of $V_{d,r}$
with a vector bundle coming from $\cS_{\eta,\eta';q;p}$ or $\cS'_{\eta,\eta/px_1}$.

Now, multiplication by $T$ gives an injective morphism of graded modules,
so that we may identify the image of multiplication by $T$ in $V_{d,r}$
with $V_{d-2,r-1}(-2)$.  We may identify the quotient with the image of
$[T^0]$, and an easy induction shows that for $d'\ge d$, the image of
$[T^0]$ contains all sections of some line bundle of degree $2d+2d'-r$.
(Here we are assuming, as we may, that $d>r>0$.)  We thus find that
\[
h(t;V_{d,r})\ge h(t-2;V_{d-2,r-1})+2d+2d'-r,
\]
implying by induction the lower bound
\[
h(t;V_{d,r})\ge (d+1)(t+1)-r(r+1)/2.
\]
(Note that each step in the induction decreases both $r$ and $d-r$ by 1,
and thus we will reach the base cases $d=r$ or $r=0$ before forcing $d$
negative.)

It remains to show that the lower bound is tight.  We note that while the
the sequence $V_{d,r}$ of vector bundles is decreasing, it cannot decrease
{\em too} quickly, since
\begin{align}
\cS_{\eta,\eta',x_1;q;p}(ds+(d'-1)f-(r-1)e_1,ds+d'f-re_1)&
\cS_{\eta,\eta',x_1;q;p}(0,ds+(d'-1)f-(r-1)e_1)\notag\\
&\subset
\cS_{\eta,\eta',x_1;q;p}(0,ds+d'f-re_1)\notag\\
&\subset
\cS_{\eta,\eta',x_1;q;p}(0,ds+d'f-(r-1)e_1)
\end{align}
and thus the vector bundles $V_{d,r}$ satisfy
\[
V_{d,r-1}(-1)\subset V_{d,r}\subset V_{d,r-1}.
\]
In particular if we write
\[
V_{d,r} = \bigoplus_{0\le i\le d} \sO_{\P^1}(e_{d,r,i})
\]
with $e_{d,r,0}\ge e_{d,r,1}\ge\cdots\ge e_{d,r,d}$, then
\[
e_{d,r-1,i}-1\le e_{d,r,i}\le e_{d,r-1,i}
\]
for all $d\ge i\ge 0$, $r\in \Z$.

Now, for $r\ge d$, the identification of $V_{d,r}$ with a family of $\Hom$
spaces of $\cS'$, together with the known basis of these spaces lets
us compute the corresponding degree sequence.  We thus find that
$e_{d,r,i}=d-r-i$ for $0\le i\le d\le r$.  For $r\le 0$, the situation is
somewhat more subtle, since the degree sequence depends on the parameters
in general.  Luckily, we will only need to understand the differences
$e_{d,0,i}-e_{d,d,i}$.  Although these still depend on the parameters
(since the sequence $e_{d,d,i}$ doesn't), it turns out that the dependence
on the parameters is relatively mild.  In particular, although the sequence
$e_{d,0,i}-e_{d,d,i}$ depends on the parameters, the {\em multiset}
doesn't!  Indeed, an easy induction using Lemma \ref{lem:calS_bundles}
tells us that the sequence $e_{d,0,i}-e_{d,d,i}=e_{d,0,i}+i$ is a
permutation of the sequence $0,1,\dots,d$ corresponding to the generic case.

The significance of this fact is that as we increase $r$ from $0$ to $d$,
each degree can only decrease by $1$.  In particular, one of the degrees is
decreased at every step, one at all but one steps, etc.  In particular, the
first $r$ steps must decrease the largest difference every time, the second
largest difference at least $r-1$ times, etc.  But this implies
\[
\sum_{0\le i\le d} e_{d,r,i}\le \sum_{0\le i\le d} e_{d,0,i}-r(r-1)/2,
\]
giving the other required inequality on $\deg(V_{d,r})$.
\end{proof}

\begin{rem}
  This corresponds to the following interpretation of
  $\cS_{\eta,\eta',x_1;q;p}(0,ds+d'f-r_1e_1)$ for generic parameters: it is
  the subspace of $\cS_{\eta,\eta';q;p}(0,ds+d'f)$ consisting of operators
  for which $[T^r]\oD$ vanishes at $q^{-k}x_1$ for $r\le k<r_1$.  (These are
  precisely the points where the relevant gauge transformation introduces
  poles, and the total number of conditions is correct.)  In the same way,
  $\cS'_{\eta,x_0,x_1;q;p}(0,ds+d'f-r_1e_1)$ is generically the subspace of
  $\cS'_{\eta,x_0;q;p}(0,ds+d'f)$ consisting of operators for which
  $[T^r]\oD$ vanishes at $q^{-k}x_1$ for $r\le k<r_1$.
\end{rem}

\begin{cor}
If $d'\ge d\ge r_1\ge 0$, then
\[
\dim(\cS_{\eta,\eta',x_1;q;p}(0,ds+d'f-r_1e_1))
=
(d+1)(d'+1)-r_1(r_1+1)/2.
\]
\end{cor}

\begin{proof}
Since $\cS_{\eta,\eta';q;p}(0,ds+d'f)=\hat{\cS}_{\eta,\eta';q;p}(0,ds+d'f)$
for $d'\ge d-1$, the same applies when $r\ge 0$.  But then we may identify
the given space with the space of global sections of a bundle $V_{d,r_1}$,
and we need simply show that this bundle is acyclic.  This follows by an
easy induction from the leading coefficient short exact sequence, with base
case $r_1=0$ or $r_1=d$.
\end{proof}

\begin{lem}\label{lem:cancel_K7}
Suppose $\oD\in \cS_{\eta,\eta';q;p}(0,ds+d'f)$ is such that
\[
\cS_{\eta,\eta',x_1;q;p}(ds+d'f,(d+1)s+(d'+1)f)
\oD
\subset
\cS_{\eta,\eta',x_1;q;p}(0,(d+1)s+(d'+1)f-r_1e_1).
\]
Then $\oD\in \cS_{\eta,\eta',x_1;q;p}(0,ds+d'f-r_1e_1)$.
\end{lem}

\begin{proof}
By assumption, for any operator
\[
\oD_{q^{-d}\eta;q;p}(b(z))
\in
\cS_{\eta,\eta',x_1;q;p}(ds+d'f,(d+1)s+(d'+1)f),
\]
we have
\[
\oD_{q^{-d}\eta;q;p}(b(z))\oD
\in
\cS_{\eta,\eta',x_1;q;p}(0,(d+1)s+(d'+1)f-r_1e_1).
\]
Thus
\begin{align}
\frac{\Gamm{d+1-r_1;q;p}(z;\eta/px_1)}
     {\Gamm{r_1}(x;z_1)}
\oD_{q^{-d}\eta;q;p}(b(z))
&\oD
\frac{\Gamm{0}(x;z_1)}
     {\Gamm{0;q;p}(z;\eta/px_1)}
\notag\\
&\in
\cS'_{\eta,\eta'/px_1}(0,(d+1)s+(d+d'-r_1+2)f),
\end{align}
and it remains to show
\[
\oD'
:=
\frac{\Gamm{d-r_1;q;p}(z;\eta/px_1)}
     {\Gamm{r_1}(x;z_1)}
\oD
\frac{\Gamm{0}(x;z_1)}
     {\Gamm{0;q;p}(z;\eta/px_1)}
\in
\cS'_{\eta,\eta'/px_1}(0,ds+(d+d'-r_1)f).
\]
Since
\[
\frac{\Gamm{d+1-r_1;q;p}(z;\eta/px_1)}
     {\Gamm{r_1}(x;z_1)}
\oD_{q^{-d}\eta;q;p}(b(z))
\frac{\Gamm{r_1}(x;z_1)}
     {\Gamm{d-r_1;q;p}(z;\eta/px_1)}
=
\oD_{q^{-d}\eta;q;p}(\theta_p(q^{r_1-d}\eta/x_1z)b(z)),
\]
we reduce to showing that
\[
\oD_{q^{-d}\eta;q;p}(\theta_p(q^{r_1-d}\eta/x_1z)b(z))
\oD'
\in
\cS'_{\eta,\eta'/px_1}(0,(d+1)s+(d+d'-r_1+2)f)
\]
for all $b$ implies $\oD'\in \cS'_{\eta,\eta'/px_1}(0,ds+(d+d'-r_1)f)$.
This follows immediately from the adjoint of Lemma \ref{lem:cancel_F1}
above: simply choose $b$ so that
\[
b_2(z)=\theta_p(q^{r_1-d}\eta/x_1z)b(z)
\]
satisfies the hypotheses of that Lemma.
\end{proof}

\begin{thm}
The Fourier transform
\[
\cS_{\eta,\eta';q;p}(0,ds+d'f)\cong \cS_{\eta',\eta;q;p}(0,d's+df)
\]
restricts to isomorphisms
\[
\cS_{\eta,\eta',x_1;q;p}(0,ds+d'f-r_1e_1)\cong
\cS_{\eta',\eta,x_1;q;p}(0,d's+df-r_1e_1)
\]
for all integers $r_1$.
\end{thm}

\begin{proof}
  Note first that Lemma \ref{lem:cancel_K7} implies that if the result
  holds for $ds+d'f-r_1e_1$, then it holds for $(d-1)s+(d'-1)f-r_1e_1$.  We
  may thus reduce to the case $\min(d,d')\ge r_1$.  If $r_1\le 0$, there is
  nothing to prove.  Otherwise, we claim that
\[
\cS_{\eta,\eta',x_1;q;p}(0,ds+d'f-r_1e_1)
=
\cS_{\eta,\eta',x_1;q;p}(s+f-e_1,ds+d'f-r_1e_1)
\cS_{\eta,\eta',x_1;q;p}(0,s+f-e_1).
\]
Indeed, the right-hand side accounts for all
possible leading coefficients, while
\[
T\in
\cS_{\eta,\eta',x_1;q;p}(s+f-e_1,2s+2f-e_1)
\cS_{\eta,\eta',x_1;q;p}(0,s+f-e_1),
\]
by the adjoint of Lemma \ref{lem:central_elt}.  We thus reduce to showing
the claim in the case $d=d'=1$.

Now,
\[
\oD_{\eta;q;p}(b)\in \cS_{\eta,\eta',x_1;q;p}(0,s+f-e_1)
\]
iff
\[
\Gampq(q z/x_1,q\eta/x_1 z)^{-1}
\oD_{\eta;q;p}(b)
\Gampq(z/x_1,q\eta/x_1 z)
=
\oD_{\eta;q;p}(b/\theta_p(z/x_1))
\in
\cS'_{\eta,\eta'/px_1}(0,s+f).
\]
In other words, $\oD_{\eta;q;p}(b)\in \cS_{\eta,\eta',x_1;q;p}(0,s+f-e_1)$ if
$b(x_1)=0$.  This condition is independent of $\eta$, $\eta'$, and is
thus preserved by the Fourier transform.
\end{proof}

\begin{cor}
The family $\cS_{\eta,\eta',x_1;q;p}$ of $\Z^3$-algebras is flat.
\end{cor}

\begin{proof}
We need to show that the dimension of
\[
\cS_{\eta,\eta',x_1;q;p}(0,ds+d'f-r_1e_1)
\]
depends only on $d$, $d'$, $r_1$.  The Theorem tells us that
\[
\dim\cS_{\eta,\eta',x_1;q;p}(0,ds+d'f-r_1e_1)
=
\dim\cS_{\eta',\eta,x_1;q;p}(0,d's+df-r_1e_1),
\]
and we may thus reduce to the case $d'\ge d$.  If $r_1\ge d$, then the
space is essentially a $\Hom$ space in $\cS'_{\eta,\eta'/px_1;q;p}$, and is
thus flat.  Similarly, if $r_1\le 0$, then the space is just
$\cS_{\eta,\eta';q;p}(0,ds+d'f)$, so again is flat.  Finally, if $0\le
r_1\le d$, then $\cS_{\eta,\eta';q;p}(0,ds+d'f-r_1e_1)$ is the space
of global sections of an acyclic vector bundle of Euler characteristic
$(d+1)(d'+1)-r_1(r_1+1)/2$, and the claim follows immediately.
\end{proof}

Applying the standard elementary transformation gives the following.

\begin{cor}\label{cor:F1blowup_s0}
The family $\cS'_{\eta,x_0,x_1;q;p}$ of $\Z^3$-algebras is flat, and
there is an isomorphism
\[
\cS'_{\eta,x_0,x_1;q;p}\cong \cS'_{x_0\eta/x_1,x_1,x_0;q;p},
\]
which on objects acts as $ds+d'f-r_1e_1\mapsto (d'-r_1)s+d'f-(d'-d)e_1$.
\end{cor}

This of course is just a noncommutative analogue of the fact that if we
blow up $\P^2$ in two distinct points, the order of the two points is
irrelevant.  (Of course, it is somewhat stronger than that even in the
commutative case, as we have made no assumption on $x_0$ and $x_1$.)

The adjoint symmetry also extends to this case.  The key observation is
that we should take
\[
\Gamm{r;1/q;p}(z;x) = \Gamm{1-r;q;p}(z;x)^{-1},
\]
insofar as both sides have the same behavior under $T=T_{1/q}^{-1}$ and
$r\mapsto r+1$.  The following is then an immediate consequence of the
adjoint symmetries for $\cS_{\eta,\eta';q;p}$ and $\cS'_{\eta,x_0;q;p}$.

\begin{prop}
The adjoint isomorphism
\[
\cS_{\eta,\eta';q;p}\cong \cS_{\eta,\eta';1/q;p}^{\text{\rm op}}
\]
extends to an isomorphism
\[
\cS_{\eta,\eta',x_1;q;p}\cong \cS_{\eta,\eta',x_1;1/q;p}^{\text{\rm op}}
\]
which on objects acts as $ds+d'f-r_1e_1\mapsto
(2s+2f-e_1)-(ds+d'f-r_1e_1)$.
\end{prop}

\begin{rem} Similarly, for $\cS'_{\eta,x_0,x_1;q;p}$, the adjoint acts as
  $ds+d'f-r_1e_1\mapsto (2s+3f-e_1)-(ds+d'f-r_1e_1)$.
\end{rem}

\medskip

In order to extend this construction to algebraic curves, it will be
helpful to pin down a relatively small set of generators.

\begin{prop}
The category $\cS_{\eta,\eta',x_1;q;p}$ is generated by elements of degrees
$e_1$, $f-e_1$, $s-e_1$, $f$, $s$, $s+f-e_1$.
\end{prop}

\begin{proof}
Consider a $\Hom$ space of degree $ds+d'f-r_1e_1$.  If $r_1<0$, then
we may factor out morphisms of degree $e_1$ to reduce to $r_1=0$,
for which the claim is immediate.  Similarly, if $r_1>d$, we may factor out
morphisms of degree $f-e_1$, while if $d'>d$, we may apply the Fourier
transform.  We thus see that we need only consider $\Hom$ spaces of degree
$ds+d'f-r_1e_1$ with $d'\ge d\ge r_1\ge 0$.

Now, the degrees we have reduced to considering form a semigroup, and it is
easy to see the semigroup is generated by $f$, $s+f-e_1$, and $s+f$.
The first two degrees are already accounted for by our generators, while
for the third, we may write
\begin{align}
\cS_{\eta,\eta',x_1;q;p}(0,s+f)
\supset{}
&\cS_{\eta,\eta',x_1;q;p}(f,s+f)
\cS_{\eta,\eta',x_1;q;p}(0,f)
+
\cS_{\eta,\eta',x_1;q;p}(s,s+f)
\cS_{\eta,\eta',x_1;q;p}(0,s)\notag\\
{}+{}&
\cS_{\eta,\eta',x_1;q;p}(s+f-e_1,s+f)
\cS_{\eta,\eta',x_1;q;p}(0,s+f-e_1)
\end{align}
Now, we have already seen that the first two terms span
$\cS_{\eta,\eta',x_1;q;p}(0,s+f)=\cS_{\eta,\eta';q;p}(0,s+f)$ unless $q=1$,
$\eta=\eta'$, when their span can be identified with the space of
$BC_1(\eta)$-symmetric theta functions.  Since the third term then consists
of those theta functions with the same multiplier that vanish at
$x_1$, it follows that $\cS_{\eta,\eta',x_1;q;p}(0,s+f)$ is indeed
contained in the generated subcategory.

We next observe that $\cS_{\eta,\eta',x_1;q;p}(0,2s+2f-e_1)$ is contained
in the generated subcategory.  We first note that by Lemma
\ref{lem:bundle_products}, for every operator in
$\cS_{\eta,\eta',x_1;q;p}(0,2s+2f-e_1)$, there is an operator in
$\cS_{\eta,\eta',x_1;q;p}(0,s+f-e_1)\cS_{\eta,\eta',x_1;q;p}(s+f)$ with the
same leading coefficient.  Since it follows immediately from Lemma
\ref{lem:central_elt} that
\[
T\in \cS_{\eta,\eta',x_1;q;p}(0,s+f-e_1)\cS'_{\eta,\eta',x_1;q;p}(s+f),
\]
the claim follows.

Now, consider an element $ds+d'f-r_1e_1$ of the semigroup.  If $r_1=0$,
then it follows from the definition of $\cS_{\eta,\eta';q;p}$ that every
operator of this degree is in the generated subcategory.  Similarly, if
$r_1=d$, then the claim follows by considering
$\cS'_{\eta,\eta'/px_1;q;p}$.  Finally, if $d>r_1>0$, then we may subtract
$2s+2f-e_1$ without leaving the semigroup.  We thus find that
every element of $\cS_{\eta,\eta',x_1;q;p}(0,(d-2)s+(d'-2)f-(r_1-1)e_1)$ is
contained in the subcategory, so that it suffices to show that
\begin{align}
&\cS_{\eta,\eta',x_1;q;p}((d-2)s+(d'-2)f-(r_1-1)e_1,ds+d'f-r_1e_1)\notag\\
&\circ\cS_{\eta,\eta',x_1;q;p}(0,(d-2)s+(d'-2)f-(r_1-1)e_1)
=
\cS_{\eta,\eta',x_1;q;p}(0,ds+d'f-r_1e_1).
\end{align}
This certainly accounts for every element with leading coefficient $0$, and
by Lemma \ref{lem:bundle_products} accounts for every possible leading
coefficient.
\end{proof}

\begin{rem} If $q\ne 1$ or $\eta\ne \eta'$, then the generators of degree
  $s+f-e_1$ are redundant.  If $q\ne 1$, then the generators of degree $s$
  and $f$ are also redundant.
\end{rem}

It immediately follows that for $\cS'_{\eta,x_0,x_1;q;p}$, we can generate
the category by elements of degree $s$, $e_1$, $f-e_1$, $f$, $s+f-e_1$, and
$s+f$.  Thus to define $\cS'_{\eta,x_0,x_1;q;C}$, it will suffice to
determine the operators corresponding to morphisms of degree $e_1$,
$f-e_1$ and $s+f-e_1$.  Now, the operators of degree $e_1$ are just scalar
multiples of $1$ (since they must be contained in the space of operators of
degree 0), while for $f-e_1$ and $s+f-e_1$, the only constraint is that the
leading coefficient must vanish at $x_1$ (in addition to the already forced
divisor).  Of course, we can then also define $\cS_{\eta,\eta',x_1;q;C}
:=\cS'_{\eta,\eta'/x_1,\eta/x_1;q;C}$, and can then obtain an elliptic
version of $\cS_{\eta,\eta';q;p}$ by taking the appropriate
$\Z^2$-subalgebra.  The result here depends on the choice of $x_1$,
but changing $x_1$ gives an isomorphic category.

When $q=1$, $\eta\ne \eta'$, the generators of degree $s+f-e_1$ are
redundant, and we moreover see that (assuming we choose the other
generators in a way independent of the source object) the only relations
are that the generators commute.  This gives us the following, not
unexpected, result.

\begin{prop}
Let $C$ be a genus 1 curve (over an algebraically closed field) embedded in
$\P^1\times \P^1$ by a pair of (nonisomorphic) line bundles $\eta$, $\eta'$
of degree 2, and let $X$ be the surface obtained from $\P^1\times \P^1$ by
blowing up the point $x_1\in C$.  The subcategory of $\Coh(X)$ generated by
line bundles is equivalent to $\cS_{\eta,\eta',x_1;1;C}$, in such a way
that the restriction to $C\subset X$ corresponds to taking leading
coefficients.
\end{prop}

\begin{proof}
For each $d,d',r_1\in \Z$, let $\oD_{ds+d'f-r_1e_1}$ be the divisor on 
$C$ given by
\[
\oD_{ds+d'f-r_1e_1}
=
[\eta'/x_1]^{d'-r_1}[\eta/x_1]^{r_1-d}[x_1\eta/\eta']^{d'-d-r_1}R_\eta^d.
\]
This divisor plays two roles.  First, the leading coefficient of an element
of $\cS_{\eta,\eta',x_1;1;C}(0,ds+d'f-r_1e_1)$ is naturally a global
section of the line bundle $\sO(\oD_{ds+d'f-r_1e_1})$.  (This is really
a statement about $\cS'_{\eta,\eta'/x_1,\eta/x_1;1;C}$, where it follows
readily from the corresponding statement on $F_1$.)  Second, if ${\cal
  L}_{ds+d'f-r_1e_1}$ is a line bundle of the specified divisor class on
$X$, then
\[
{\cal L}_{ds+d'f-r_1e_1}|_C
\cong
\sO(\oD_{ds+d'f-r_1e_1}).
\]
(Indeed, $\oD_{ds+d'f-r_1e_1}$ is a representative of the divisor class
$\eta^{d'} \eta^{\prime d} x_1^{-r_1}$.)

Now, choose a representative of each divisor class on $X$, and for
each such representative choose an identification
\[
{\cal L}_{ds+d'f-r_1e_1}|_C
\cong
\sO(\oD_{ds+d'f-r_1e_1}).
\]
Then there is a functor from $\cS_{\eta,\eta',x_1;1;C}$ to this subcategory
(acting in the obvious way on objects) making it naturally an equivalence.
Indeed, this identification removes any freedom in choosing how generators
map, and since the only relations in $\cS_{\eta,\eta',x_1;1;C}$ are that
the generators commute, this gives a functor.  It is moreover
straightforward to verify that the generators are algebraically independent
as functions on $X$, and that the $\Hom$ spaces on both sides have the same
dimension.
\end{proof}

\begin{rem}
  When $\eta=\eta'$, we find in a similar way that
  $\cS_{\eta,\eta,x_1;1;C}$ can be identified with a non-full subcategory
  of $\Coh(X)$, where now $X$ is obtained by blowing up the point $x_1$ of
  $C$, embedded in $F_2$ via $\eta$.
\end{rem}

Another use for this small set of generators is that it allows us to define
morphisms from $\cS_{\eta,\eta',x_1;q;p}$.  Of course, we have already
dealt with most of the morphisms we will need, with the one caveat that
some of them are only truly defined up to a scalar gauge (as the Fourier
transform and the adjoint both require us to a choose a certain system of
elliptic functions).  There is one more that is worth considering, this
time no longer an equivalence.

\begin{prop}\label{prop:center_1}
Suppose $q$ is a primitive $r$-th root of unity.  There is a functor
\[
\cS_{\eta^r,\eta^{\prime r},x_1^r;1;p^r}
\to
\cS_{\eta,\eta',x_1;q;p}
\]
which on objects acts as multiplication by $r$, and on morphisms acts as
\[
\sum_{0\le k\le d} c_k(z) T^k
\mapsto
\sum_{0\le k\le d} c_k(z^r) T^{rk}.
\]
\end{prop}

\begin{proof}
This operation clearly respects composition, so it remains only to show
that it takes generators to morphisms.  For degree $e_1$, this is
immediate, as the functor takes the operator $1$ to the operator $1$.  For
degree $f-e_1$, we note that any element of
\[
\cS_{\eta,\eta',x_1;q;p}(0,rf-re_1)
\]
is a composition of elements of degree $f-e_1$, and is thus proportional to
\[
\prod_{0\le i<r} \theta_p(z/q^r x_1,\eta/q^r x_1 z)
=
\theta_{p^r}(z^r/x_1^r,\eta^r/x_1^r z^r),
\]
from which the claim follows. Now, this implies in particular that
\[
\theta_{p^r}(z^r/x_1^r,\eta^r/x_1^r z^r)\in \cS_{\eta,\eta';q;p}(0,rf);
\]
it immediately follows that
\[
\theta_{p^r}(z^r/x^r,\eta^r/x^r z^r)\in \cS_{\eta,\eta';q;p}(0,rf)
\]
for any $x$, so that the putative functor acts correctly on elements of
degree $f$.

Now, for the generator of degree $s-e_1$, we first apply the elementary
transform and consider 
\[
\cS'_{\eta,\eta'/px_1,\eta/px_1;q;p}(0,rs)
=
\cS'_{\eta,\eta'/px_1;q;p}(0,rs)
=
\cS'_{\eta,\eta'/px_1,\eta'/px_1;q;p}(0,rs).
\]
Thus up to a gauge transformation, we are considering elements of
\[
\cS_{\eta,\eta,x_1\eta/\eta';q;p}(0,rs-re_1)
\subset
\cS_{\eta,\eta;q;p}(0,rs)
\]
By our known results on saturation in $\cS_{\eta,\eta;q;p}$, any
such operator factors as
\[
\cS_{\eta,\eta;q;p}(rs-rf,rs)
\cS_{\eta,\eta;q;p}(0,rs-rf)
\]
and thus is a left multiple of the operator
\[
\frac{z^r}{\theta_{p^r}(z^{2r}/\eta^r)}(1-T^r).
\]
It follows (without needing to consider details of the gauge
transformations) that any element of
\[
\cS_{\eta,\eta',x_1;q;p}(0,rs-re_1)
\]
has the form $c_0(z)+c_r(z)T^r$ for suitable functions $c_0(z)$, $c_r(z)$.
Any such operator is also a composition of operators of degree $s-e_1$,
and it is straightforward to compute the resulting extreme coefficients.
(Indeed, via the Fourier transform, we see that this is the same as the
calculation for degree $rf-re_1$.)

We thus find that the (unique up to scalars) global section of
\[
\cS_{\eta,\eta',x_1;q;p}(0,rs-re_1)
\]
has the form
\[
\frac{z^r}{\theta_{p^r}(z^{2r}/\eta^r)} 
\theta_p(z^r/x_1^r,\eta^{\prime r}/z^r x_1^r)
-
\frac{z^r}{\theta_{p^r}(z^{2r}/\eta^r)} 
\theta_p(\eta^r/z^r x_1^r,z^r \eta^{\prime r}/\eta^r x_1^r)T^r
\]
as required.  Again, we find that the same operator gives a global section
of $\cS_{\eta,\eta';q;p}(0,rs)$ even if we change $x_1$, and thus obtain
a full set of generators of degree $s$.

Since the generator of degree $s+f-e_1$ is redundant unless
$\eta^r=\eta^{\prime r}$, flatness implies that such elements map to
morphisms in general.
\end{proof}

\begin{rem}
For $\cS_{\eta,\eta',x_1;q;C}$, the corresponding fact is that if $q\in
\Pic^0(C)$ has order $r$ and $\phi:C\to C'$ is the corresponding quotient
morphism of degree $r$, then there is a functor from
$\cS_{\phi_*\eta,\phi_*\eta',\phi_*x_1;1;C'}$ to
$\cS_{\eta,\eta',x_1;q;C}$.  We will see below that this essentially makes
$\cS_{\eta,\eta',x_1;q;C}$ a (degenerate) Azumaya algebra over the
commutative surface corresponding to
$\cS_{\phi_*\eta,\phi_*\eta',\phi_*x_1;1;C'}$.
\end{rem}

\section{General blowups}
\label{sec:blowups}

The main difficulty in extending the above construction to more general
blowups is that the above condition for the $K^2=7$ case is very difficult
to extend in a uniform (and flat) manner.  For the commutative case, the
difficulty arises when blowing up the same point multiple times, as then we
need to consider functions vanishing with multiplicity along jets.  In the
noncommutative case, this becomes even worse, as any two points in the same
$q^\Z$ orbit will encounter similar issues, and would seem to require us to
define the conditions inductively.

If the points being blown up are in sufficiently general position, then
there is no difficulty: we expect the conditions at the different points to
be entirely independent in such a case.  The idea, then, is to deal with
the general case by taking a suitably defined limit from the generic case.
Of course, in order to make sense of such a limit, we will need to ensure
that our $\Z^{m+2}$-algebras have a suitable sheaf structure.

In general, let $F$ be a flat sheaf on an integral scheme $S$, and suppose
we are given a finite-dimensional subspace $V$ of the generic fiber of $F$.
(For instance, $F$ might be a $\Hom$ space in $\cS$ or $\cS'$, and $V$ the
subspace corresponding to a $\Hom$ space in the generic iterated blowup.)
Then there is a natural extension of $V$ to a subsheaf $F_V$ of $F$ having
that generic fiber.  Indeed, any such sheaf will have $\Gamma(U,F_V)\subset
\Gamma(U,F)$ for any open subset $U$, and the restriction of any element of
$\Gamma(U,F_V)$ to the generic fiber will be an element of $V$.  Since the
presheaf of sections satisfying those conditions is a sheaf, this gives the
desired natural extension (which is clearly maximal among all subsheaves
with generic fiber $V$).  We note that $F/F_V$ is torsion-free, since
otherwise we could replace $F_V$ by the preimage of the torsion subsheaf of
$F/F_V$ without affecting the generic fiber.  (Conversely, any subsheaf with
torsion-free quotient is the natural extension of its generic fiber.)
Also, if $F$ is an algebra, then the product structure induces a
multiplication on canonical extensions; if $V$ and $W$ are two subspaces of
the generic fiber, there is a natural multiplication $F_V\otimes F_W\to
F_{VW}$.

Naturally, we will want to prove that the canonical extension in our case
is flat (as the whole point of the exercise is to obtain a flat family).
In fact, we want something stronger.  Our objective, after all, is not just
to construct a flat family of $\Z^{m+2}$-algebras, but also faithful
representations of the fibers of that family in terms of difference
operators.  The difficulty is that, in terms of the above setup, even if
$F_V$ is flat, the fibers of $F_V$ may fail to inject in the fibers of $F$.
The kernel of the map on fibers comes from $\Tor_1$ of the quotient
$F/F_V$, and thus we see that we also want to insist that $F/F_V$ be flat.
(In fact, it is then redundant to require $F_V$ to be flat.)  We will say
that $V\subset F_{k(S)}$ ``has coflat extension in $F$'' if $F/F_V$ is
flat, and similarly say that the extension $F_V$ is ``coflat''.

We note here that if $V\subset W\subset F_{k(S)}$, and $W$ has coflat
extension in $F$, then $V$ has coflat extension in $F$ iff it has coflat
extension in $F_W$.  We will thus mainly omit mention of the ambient sheaf,
though we caution the reader that varying the ambient sheaf in any other
way can have a significant effect on the extension.

For any pair $(C,q)$ with $C$ a smooth genus 1 curve and $q\in \Pic^0(C)$,
let $\EllDiff_{q;C}$ denote the algebra of elliptic difference operators;
that is $\EllDiff_{q;C}$ is the algebra of finite linear combinations
$\sum_{0\le k} c_k(z) T^k$,
where $c_k(z)$ are elements of the function field of $C$, and conjugation
by $T$ acts via translation by $q$.  These algebras fit in a natural way
into a flat family of algebras over the moduli stack of pairs $(C,q)$.
This lets us define for any $m\ge 0$ a flat family of $\Z^{m+2}$-algebras
over the moduli stack of tuples $(C,q,\eta,x_0,x_1,\dots,x_m)$ with $C$
smooth of genus 1, $q\in \Pic^0(C)$, $\eta\in \Pic^2(C)$, and $x_i\in C$,
in which every $\Hom$ space is $\EllDiff_{q;C}$.  (This moduli stack is not
a scheme, but is an integral algebraic space, as can be seen by adjoining
suitable level structures; as a result, the canonical extension
construction still applies.)  It will be convenient on occasion to encode
the parameters $\eta,x_0,\dots,x_m$ in a single object, namely the
homomorphism
\[
\rho:\langle s,f,e_1,\dots,e_m\rangle\to \Pic(C)
\]
given by
\[
\rho(s) = x_0,\quad \rho(f)=\eta,\quad \rho(e_i)=x_i.
\]
Note that in the commutative situation, $\rho$ is just the natural
restriction map $\Pic(X)\to \Pic(C)$.

We then define a family of $\Z^{m+2}$-algebras
$\cS'_{\rho;q;C}=\cS'_{\eta,x_0,\dots,x_m;q;C}$ as the canonical extension
in the above $\Z^{m+2}$-algebra version of $\EllDiff_{q;C}$ of the
following generic category:
\begin{align}
&\cS'_{\eta,x_0,x_1,\dots,x_m;q;C}(d_1s+d'_1f-r_{11}e_1-\cdots-r_{1m}e_m,
d_2s+d'_2f-r_{21}e_1-\cdots-r_{2m}e_m)\notag\\
&{}=
\bigcap_{1\le i\le m}
\cS'_{\eta,x_0,x_i;q;C}(d_1s+d'_1f-r_{1i}e_i,
d_2s+d'_2f-r_{2i}e_i).
\end{align}
In other words, given an open subset $U$ of the moduli stack, the elements of
\[
\Gamma(U,\cS'_{\eta,x_0,x_1,\dots,x_m;q;C}(d_1s+d'_1f-r_{11}e_1-\cdots-r_{1m}e_m,
d_2s+d'_2f-r_{21}e_1-\cdots-r_{2m}e_m))
\]
are precisely the holomorphic families of difference operators parametrized
by $U$ which are generically in the above intersection.  In light of Lemma
\ref{lem:elem_xform1}, we also define a family
$\cS_{\eta,\eta',x_1,\dots,x_m;q;C}$ by
\begin{align}
\cS_{\eta,\eta',x_1,\dots,x_m;q;C}
(d_1s+d'_1f-r_{11}e_1-\cdots-r_{1m}e_m,
 d_2s+d'_2f-r_{21}e_1-\cdots-r_{2m}e_m&)\notag\\
{}:=
\cS'_{\eta,\eta'/x_1,\eta/x_1,x_2,\dots,x_m;q;C}
(d_1s+(d_1+d'_1-r_{11})f-(d_1-r_{11})e_1-r_{12}e_2\cdots-r_{1m}e_m&,\notag\\
d_2s+(d_2+d'_2-r_{21})f-(d_2-r_{11})e_1-r_{22}e_2\cdots-r_{2m}e_m&).
\end{align}
Again, it will occasionally be useful to encode the parameters other than
$q$ and $C$ in a single homomorphism
\[
\rho:\langle s,f,e_1,\dots,e_m\rangle\to \Pic(C),
\]
with
\[
\rho(s)=\eta',\quad \rho(f)=\eta,\quad \rho(e_i)=x_i.
\]

Since the canonical extension only depends on the ambient sheaf and the
generic fiber, the following is immediate.

\begin{prop}
For any permutation $\pi\in S_m$, there is an isomorphism (of sheaves of
$\Z^{m+2}$-algebras)
\[
\cS'_{\eta,x_0,x_1,\dots,x_m;q;C}
\cong
\cS'_{\eta,x_0,x_{\pi(1)},\dots,x_{\pi(m)};q;C}
\]
acting linearly on objects by $s,f\mapsto s,f$ and $e_i\mapsto e_{\pi(i)}$.
\end{prop}

\begin{proof}
Indeed, the corresponding generic fibers are manifestly isomorphic.
\end{proof}

\begin{rem}
  Note that in terms of the homomorphism $\rho$, the action on parameters
  is just $\rho\mapsto \rho\circ \pi$, where $\pi$ acts on $\langle
  s,f,e_1,\dots,e_m\rangle$ in the same way as it does on objects.
\end{rem}

The twisting symmetries also carry over immediately, as the corresponding
$\Hom$ sheaves are equal as subsheaves of $\EllDiff_{q;C}$.

\begin{prop}
There are natural isomorphisms
\begin{align}
\cS'_{\eta,x_0,x_1,\dots,x_m;q;C}
&\cong
\cS'_{q^{-d}\eta,q^{d-d'}x_0,q^{-r_1}x_1,\dots,q^{-r_m}x_m;q;C}\notag\\
\cS_{\eta,\eta',x_1,\dots,x_m;q;C}
&\cong
\cS_{q^{-d}\eta,q^{-d'}\eta',q^{-r_1}x_1,\dots,q^{-r_m}x_m;q;C}\notag
\end{align}
which act on objects as $v\mapsto v-ds-d'f+r_1e_1+\cdots+r_m e_m$.
\end{prop}

\begin{rem}
  The action on parameters has a uniform description in terms of $\rho$ and
  the intersection pairing: translating by
  $D_0=ds+d'f-r_1e_1-\cdots-r_me_m$ has the effect of replacing $\rho$ by
\[
\rho':D\mapsto q^{-D\cdot D_0} \rho(D).
\]
\end{rem}

More generally, we can construct isomorphisms between canonical extensions
by giving isomorphisms between ambient sheaves that restrict to
isomorphisms between the respective generic fibers.  The simplest is the
adjoint.

\begin{prop}
  There is a natural isomorphism
  \[
  \cS'_{\rho;q;C}\cong (\cS'_{\rho;q^{-1};C})^{\text{op}}
  \]
  which acts on objects as
  \[
  D\mapsto 2s+3f-e_1-\cdots-e_m-D
  \]
  and on operators as the formal adjoint in $\cS'_{\eta,x_0;q;p}$.
\end{prop}

\begin{proof}
  This follows immediately from the case $m=1$.
\end{proof}

\begin{rem}
  There is a technical issue we have glossed over here and below, namely in
  this case that the adjoint is only determined once we have chosen a
  countable infinity of elliptic functions with specified divisors, and it
  is not clear that one can make such choices on an open covering of the
  parameter space.  This is not a serious issue, though, as until we return
  to considering individual fibers, we only ever consider finitely many
  objects at a time, and thus there is no difficulty in making the
  requisite choices Zariski locally.  In addition, if we base change to
  give ourselves a generic point $v$, then we could simply insist that all
  the chosen functions evaluate to 1 at $v$.
\end{rem}

\begin{prop}
There is an isomorphism of sheaves of $\Z^{m+2}$-algebras
\[
\cS'_{\eta,x_0,x_1,\dots,x_m;q;C}
\cong
\cS'_{\eta,x_0\eta/x_1x_2,\eta/x_2,\eta/x_1,x_3,\dots,x_m;q;C}
\]
acting on objects as
\begin{align}
&ds+d'f-r_1e_1-r_2e_2-r_3e_3-\cdots-r_m e_m\notag\\
&{}\mapsto
ds+(d+d'-r_1-r_2)f-(d-r_2)e_1-(e-r_1e_2)-r_3e_3-\cdots-r_m e_m.
\end{align}
\end{prop}

\begin{proof}
Since the generic point of the moduli space lives over a characteristic 0
field, we can represent it as a configuration on a complex elliptic curve,
and thus to understand the generic fiber of
$\cS'_{\eta,x_0,x_1,\dots,x_m;q;C}(0,ds+d'f-r_1e_1-\cdots-r_m e_m)$,
it suffices to understand
\[
\bigcap_{1\le i\le m}
\cS'_{\eta,x_0,x_i;q;p}(0,ds+d'f-r_ie_i),
\]
as this differs only by the appropriate gauge transformation.  (We can even
choose a representative such that $|q|<1$, so that the gauge transformation
is by actual $\Gamma$ functions.)  Now, since the configuration is generic,
\[
\cS'_{\eta,x_0,x_i;q;p}(0,ds+d'f-r_ie_i)
\subset
\cS'_{\eta,x_0;q;p}(0,ds+d'f)
\]
is the subspace of operators such that $[T^r]\oD$ vanishes at $q^{-k}x_i$, for
$r\le k<r_i$.  The conditions for $i>1$ are preserved under the isomorphism
\[
\cS'_{\eta,x_0,x_1;q;p}(0,ds+d'f-r_1e_1)
\cong
\cS_{\eta,x_0\eta/x_1,\eta/px_1;q;p}(0,ds+(d'-r_1)f-(d-r_1)e_1),
\]
as this involves gauging by a family of cocycles which are holomorphic and
nonzero at all of the points $x_2,\dots,x_m$.  It follows that this gauge
transformation maps into
\[
\cS_{\eta,x_0\eta/x_1,\eta/px_1;q;p}(0,ds+(d'-r_1)f-(d-r_1)e_1)
\cap
\bigcap_{2\le i\le m}
\cS_{\eta,x_0\eta/x_1,x_i;q;p}(0,ds+(d'-r_1)f-r_ie_i).
\]
The same argument applies if we then apply the gauge transformation
\[
\cS_{\eta,x_0\eta/x_1,x_2;q;p}(0,ds+(d'-r_1)f-r_2e_2)
\cong
\cS'_{\eta,x_0\eta/px_1x_2,\eta/px_2;q;p}(0,ds+(d+d'-r_2)f-(d-r_2)e_2),
\]
and a final gauge transformation relates this to
\[
\cS'_{\eta,x_0\eta/x_1x_2,\eta/x_2,\eta/x_1,x_3,\dots,x_m;q;C}.
\]

In other words, the statement holds for the generic fiber, and (tracking
the various gauge transformations above) the isomorphism is given by the gauge
transformation
\begin{align}
\oD
\mapsto
{}&\frac{\Gampq(q^{d'-d-1}z/x_0,q^{d'-2d}\eta/x_0z, q^{r_1+r_2+d+1-d'}x_0z/x_1x_2, q^{r_1+r_2+2-d'}x_0\eta/zx_1x_2)}
     {\Gampq(q^{r_1}z/x_1,q^{r_1+1-d}\eta/zx_1,q^{r_2}z/x_2,q^{r_2+1-d}\eta/zx_2)}
\notag\\
&\times \oD\frac{\Gampq(z/x_1,q\eta/zx_1,z/x_2,q\eta/zx_2)}
     {\Gampq(z/qx_0,\eta/x_0z,qx_0z/x_1x_2,q^2x_0\eta/zx_1x_2)}.
\end{align}
It will thus suffice to give an extension of this gauge transformation to
the full algebraic moduli stack.  In other words, we need to construct a
family of cohomologous cocycles which for analytic curves induces the above
gauge transformation (possibly up to a scalar gauge transformation).

Now, at the origin, the above product of $\Gamma$ functions satisfies the
$q$-difference equation
\[
\frac{F(qz)}{F(z)} = 
\frac{\theta_p(\eta/x_1z,\eta/x_2z,z/qx_0,qx_0z/x_1x_2)}
     {\theta_p(z/x_1,z/x_2,\eta/qzx_0,qx_0\eta/zx_1x_2)}.
\]
But this cocycle extends in a canonical way to the moduli stack.  Indeed,
we can characterize the right-hand side as the unique elliptic function of
$x_0$ with divisor $[z/q][x_1x_2/qz]/[\eta/qz][zx_1x_2/q\eta]$ that takes
the value $1$ when $x_0=x_1/q$.  This gives us the required family of
cocycles, and it remains to show that the various cocycles are
cohomologous on the full stack.  But this follows immediately from the fact
that the functions representing the cocycles are elliptic in all
parameters, a trivial consequence of the extension to the full stack.
\end{proof}

The morphism of Proposition \ref{prop:center_1} extends similarly.

\begin{prop}\label{prop:center_m}
  Suppose $q\in \Pic^0(C)$ is an $r$-torsion point, with corresponding
  isogeny $\phi:C\to C'$.  Then there is a functor
  \[
  \Phi:\cS'_{\phi_*\circ \rho;1;C'} \to \cS'_{\rho;q;C}
  \]
  acting on objects as $D\mapsto rD$ and on (families of) morphisms as
  \[
  \sum_{0\le k\le d} c_k T^k
  \mapsto
  \sum_{0\le k\le d} \phi^* c_k T^{rk}.
  \]
\end{prop}

The remaining symmetry involves the Fourier transform, and is thus somewhat
difficult to define on the full space of difference operators.  Here the
fact that we have some freedom to choose the ambient sheaf is quite helpful.

\begin{prop}
There is an isomorphism of sheaves of $\Z^{m+2}$-algebras
\[
\cS'_{\eta,x_0,x_1,x_2,\dots,x_m;q;C}
\cong
\cS'_{x_0\eta/x_1,x_1,x_0,x_2,\dots,x_m;q;C}
\]
which on objects acts as
\[
ds+d'f-r_1e_1-r_2e_2-\cdots-r_m e_m
\mapsto
(d'-r_1)s+d'f-(d'-d)e_1-r_2e_2-\cdots-r_m e_m.
\]
\end{prop}

\begin{proof}
Since $\cS'_{\eta,x_0,x_1;q;C}$ is a (co)flat family and its generic fiber
contains the generic fiber of
\[
\cS'_{\eta,x_0,x_1,x_2,\dots,x_m;q;C},
\]
we can use it as the ambient sheaf in constructing the canonical extension.
It thus remains to show that the isomorphism
\[
\cS'_{\eta,x_0,x_1;q;C} \cong \cS'_{x_0\eta/x_1,x_1,x_0;q;C}
\]
induces the isomorphism we desire in the generic case.

Again, we may reduce understanding the generic case to understanding the
analytic case, where an elementary transformation maps the problem to
\begin{align}
&\cS_{\eta,x_0\eta/x_1,\eta/px_1;q;p}(0,ds+(d+d'-r_1)f-(d-r_1)e_1)\notag\\
&{}\cap
\bigcap_{2\le i\le m}
\cS_{\eta,x_0\eta/x_1,x_i;q;p}(0,ds+(d+d'-r_1)f-r_ie_1),
\end{align}
where now the isomorphism is just the Fourier transform.  Since the Fourier
transform preserves each subspace being intersected, the claim follows
immediately.
\end{proof}

As in the commutative case, the above symmetries generate a Coxeter group
of type $W(E_{m+1})$, with simple roots (relative to the intersection form)
\[
s-e_1,f-e_1-e_2,e_1-e_2,\dots,e_{m-1}-e_m
\]
for $\cS'$, and
\[
s-f,f-e_1-e_2,e_1-e_2,\dots,e_{m-1}-e_m
\]
for $\cS$, acting in the appropriate way ($\rho\mapsto \rho\circ s_i$) on
parameters.  (The Coxeter relations for the action on the category are
mostly straightforward; the only tricky ones can be dealt with for $m=2$,
and follow from the generic presentation given in Appendix \ref{sec:genK6}
below.)  One key observation here is that since each simple reflection
corresponds to an isomorphism of appropriately chosen ambient sheaves, the
action of $W(E_{m+1})$ preserves coflatness.  As a result, when proving
coflatness, we may restrict our attention to $\Hom$ spaces of degree $D$
where $D\cdot s\ge 0$ for every simple root $s$ (essentially a fundamental
chamber for $W(E_{m+1})$).  Indeed, per the proof of
\cite[Thm.~3.4]{rat_Hitchin}, every orbit either meets this fundamental
chamber or contains an element with $D\cdot f<0$; in the latter case, the
sheaf is identically 0, so automatically coflat.  We may also assume
$D\cdot e_m\ge 0$, since otherwise we have an isomorphism (of sheaves)
\[
\cS'_{\rho;q;C}(0,D)
\cong
\cS'_{\rho;q;C}(0,D-e_m)
\]
as this is true generically.  Moreover, we may either assume $m\le 1$ (where
coflatness is immediate) or that $D\cdot e_m>0$, since when $D\cdot e_m=0$
we have an isomorphism
\[
\cS'_{\eta,x_0,\dots,x_m;q;C}(0,D)
\cong
\cS'_{\eta,x_0,\dots,x_{m-1};q;C}(0,D).
\]

To proceed further, we need information about the leading coefficients.
The analogue of the canonical divisor will play a prominent role,
naturally; more generally, we will also need to consider divisors
corresponding to pullbacks of the anticanonical curve.  With this in mind,
define divisors (relative to $F_1$)
\[
C_l = 2s+3f-e_1-\cdots-e_l,
\]
for $0\le l\le m$.  Note that for $m\ge 2$, a divisor with $D\cdot e_m\ge
0$ is in the fundamental chamber iff it is a nonnegative linear combination
of the divisors
\[
f,s+f,s+2f-e_1,C_2,\dots,C_m.
\]
(We have imposed $m+2$ inequalities on the element $D\in \Z^{m+2}$, thus
cutting out a simplicial cone; since the cone is unimodular, the integral
points are generated by the $m+2$ divisor classes corresponding to the dual
basis.)  We also define, for any divisor class $D\in \langle
s,f,e_1,\dots,e_m\rangle$ a divisor $\fD_{\rho;q;C}(D)$ (depending on the
parameters) by
\begin{align}
\fD_{\rho;q;C}&(ds+d'f-r_1e_1-\cdots-r_me_m)\notag\\
&:=
\prod_{0\le k<d} R_{q^{-k}\eta}
\prod_{0\le k<d'-d} [q^{1-k}x_0]
\prod_{0\le k<d'-2d} [\eta/q^{-k}x_0]
\prod_{\substack{1\le i\le m\\ 0\le k<r_i}} [q^{-k}x_i]^{-1},
\end{align}
extended in the natural way to products of negative length.  Note that
\[
\sO(\fD_{\rho;q;C}(D)) \sim q^{-D\cdot(D+C_m)/2} \rho(D)
\]
in $\Pic(C)$, so in particular has degree $C_m\cdot D$.

\begin{lem}
For any divisor class $D\in \langle s,f,e_1,\dots,e_m\rangle$,
\[
T\in \cS'_{\rho;q;C}(D-C_m,D),
\]
and this induces an exact sequence
\[
\begin{CD}
0@>>> \cS'_{\rho;q;C}(0,D-C_m)
@>T>> \cS'_{\rho;q;C}(0,D)
@>[T^0]>> \Gamma(C;\sO(\fD_{\rho;q;C}(D))),
\end{CD}
\label{eq:lc_exact}
\]
where $\Gamma(C;\cdot)$ denotes the direct image to the moduli stack.
\end{lem}

\begin{proof}
  By the twisting symmetry, the first claim reduces to the claim that
  \[
  T\in \cS'_{\rho;q;C}(0,C_m),
  \]
  and thus in turn to the claim that this holds generically.  But this
  follows immediately from the fact that it holds for $m=1$.  Injectivity
  is then immediate from the fact that $T$ is not a zero-divisor.

  Since $\Gamma(C;\sO(\fD_{\rho;q;C}(D)))$ is a coflat subsheaf of the
  flat sheaf $k(C)$, we may reduce the claim that $[T^0]\oD\in
  \Gamma(C;\sO(\fD_{\rho;q;C}(D)))$ for any family $\oD\in
  \cS'_{\rho;q;C}(0,D)$ to the generic case, where it follows immediately
  from the case $m=1$.

  It remains to show exactness in the middle.  For this, we need to show
  that if $\oD$ is a section of $\cS'_{\rho;q;C}(0,D)$ on some open
  subset of parameter space, and $[T^0]\oD$ is identically 0, then
  $T^{-1}\oD$ is a section of $\cS'_{\rho;q;C}(0,D-C_m)$.  But this
  again reduces to the generic case.
\end{proof}

\begin{rem}
  It will follow from the proof of coflatness that this sequence is
  actually exact on fibers (though the reader should note that if
  $\fD_{\rho;q;C}(D)$ is a nonconstant family of degree 0 divisors, then
  the fibers of $\Gamma(C;\sO(\fD_{\rho;q;C}(D)))$ are 0, even for
  parameters where $\fD_{\rho;q;C}(D)$ is principal.)
\end{rem}

To finish proving coflatness, it will suffice to show that this becomes a
short exact sequence whenever $D$ is in the fundamental chamber and $D\cdot
e_m>0$; we can then deduce coflatness from coflatness in degree $D-C_m$.
If $D\cdot C_m<0$ or $D\cdot C_m=0$ with $D\ne 0$, then the leading
coefficient line bundle generically has no global sections, and thus short
exactness is immediate.

\begin{lem}
  The above sequence is short exact for $D=C_m+af$ for $m\ge 2$, $a\ge 0$.
\end{lem}

\begin{proof}
  The (generic) condition for an element of
  \[
  \cS'_{\eta,x_0;q;C}(0,2s+(a+3)f)
  \]
  to lie in
  \[
  \cS'_{\eta,x_0,\dots,x_m;q;C}(0,C_m+af)
  \]
  is precisely that its leading coefficient lies in
  \[
  \Gamma(C;\sO(\fD_{\rho;q;C}(C_m+af))).
  \]
  Since the leading coefficient sequence is short exact for $m=0$, the
  claim follows.
\end{proof}

Note that the leading coefficient bundle has degree $D\cdot C_m$.

\begin{lem}
  If $D_0\in C_m+\N\langle f,s+f,s+2f-e_1,C_2,\dots,C_m\rangle$, then
  the sequence \eqref{eq:lc_exact} is short exact for $D=D_0+af$, $a\gg 0$.
\end{lem}

\begin{proof}
Suppose by induction that the claim holds for $D_0-C_m$ (this will
eventually either lead to a case with smaller $m$, and thus eventually to
the case $m\le 1$, or to the case $D_0=C_m+af$ that we have already
considered).  Then there exists $a_1\ge 0$ such that the sequence is short
exact for $D_0+a_1f-C_m$.  We may further assume that $(D_0+a_1f-C_m)\cdot
C_m\ge 3$, since this condition is monotonic in $a_1$.  Similarly, take any
$a_2$ such that $a_2\ge 0$ and $(C_m+a_2f)\cdot C_m\ge 3$.  Then both
  \[
  \cS'_{\rho;q;C}(0,D_0+a_1f-C_m)\qquad\text{and}\qquad
  \cS'_{\rho;q;C}(D_0+a_1f-C_m,D_0+(a_1+a_2)f)
  \]
  have precisely the expected leading coefficients, and in each case they
  are global sections of line bundles of degree $\ge 3$.  It then follows
  from Lemma \ref{lem:bundle_products} that the composition has the
  expected leading coefficients.
\end{proof}

It follows from this that $\cS'(0,D)$ consists of global sections of a
suitable vector bundle on $\P^1$.  As a result, in order to prove short
exactness for leading coefficients, it suffices to prove that the vector
bundle is acyclic whenever
\[
D\in \N\langle f,s+f,s+2f-e_1,C_2,\dots,C_m\rangle
\]
with $D\cdot C_m>0$, $m\ge 2$.  This is certainly enough to make the
leading coefficient line bundle acyclic, so we need only check that $D-C_m$
satisfies the same conditions (or has $m\le 1$).  If $m\ge 8$, this is
immediate:
\[
(D-C_m)\cdot C_m = D\cdot C_m+m-8\ge D\cdot C_m>0.
\]
while if $m<8$, then every nonzero element of the cone satisfies $D\cdot
C_m>0$.  Since the bundle is acyclic for $m\le 1$ (so in particular for
$D=0$), we find that the bundle is always acyclic in the given cone.
Combining this with the above arguments finishes our proof of the following
fact.

\begin{thm}
  Let $D,D'\in \langle s,f,e_1,\dots,e_m\rangle$.  Then the $\Hom$ sheaf
  \[
    \cS'_{\rho;q;C}(D,D')
  \]
  is a coflat subsheaf of $\EllDiff(D,D')$.  In particular, the resulting
  family of $\Z^{m+2}$-algebras is flat, and every fiber has a faithful
  representation in which the morphisms are difference operators.
\end{thm}

Having shown this, we may now feel free to let $\cS'_{\rho;q;C}$
denote an individual fiber of this family.  Note that since this is a fiber
of a family over the appropriate moduli stack, this construction is
automatically functorial; in addition, all of the symmetries shown above
carry over to the individual fibers.  In particular, if $w\in W(E_{m+1})$,
then there is an isomorphism $\cS'_{\rho\circ w;q;C}\cong \cS'_{\rho;q;C}$
acting as $w$ on objects, and similarly for twisting and the adjoint.

\begin{cor}
  For any parameters, the $\Z^{m+2}$-algebra $\cS'_{\rho;q;C}$ is a domain.
\end{cor}

\begin{proof}
  This follows immediately from the fact that the algebra of elliptic
  difference operators is a domain.
\end{proof}

\medskip

We close with an application of the above ideas that does not depend on the
later sheaf structure.  Call a divisor class $D$ a ``rational pencil'' if
it is the class of a rational ruling on the generic rational surface; note
that this holds iff $D$ is in the $W(E_{m+1})$-orbit of $f$; similarly call
a divisor class $e$ a ``$-1$-class'' if it is in the $W(E_{m+1})$-orbit of
$s$, or equivalently is the class of a $-1$-curve on the generic rational
surface.

Now, suppose we are working over $\C$ (with $C=\C^*/\langle p\rangle$), and
let $D$ be a rational pencil other than $f$.  The leading coefficient map
is injective for rational pencils (as this property is preserved under
$W(E_{m+1})$), so in particular every element of $\cS'_{\rho;q;p}(0,D)$ is
a difference operator of order precisely $d=D\cdot f$, so has a
$d$-dimensional kernel (as a vector space over the field of $q$-elliptic
functions).  Call such a space a ``generalized eigenspace'' for $D$, and
the nonzero elements of that space ``generalized eigenfunctions''.  One
prominent example of such generalized eigenfunctions are the elliptic
biorthogonal functions of \cite{SpiridonovVP:2009}, which are generalized
eigenfunctions of a pencil of class $2s+2f-e_1-e_2-e_3-e_4$.

One natural question in this setting is how the generalized eigenfunctions
transform under various operators: if we multiply by a section of
$\cS'_{\rho;q;p}(0,f)$ (or apply some more complicated difference
operator), can the result be expanded in a finite sum of generalized
eigenfunctions for $\cS'_{\rho;q;p}(f,D+f)$, and if so, how many?

To control the number of terms in such a decomposition, we may use the
following result.

\begin{lem}
  Let $D$ be a rational pencil.  For any $r>0$, the solution space of a
  generic element of $\cS'_{\rho;q;p}(0,rD)$ is a direct sum of $r$
  generalized eigenspaces for $D$.
\end{lem}

\begin{proof}
  By dimension considerations, it will suffice to show that the solution
  space generically contains $r$ distinct generalized eigenspaces.  In
  particular, it will suffice to show that the generic element of
  $\cS'_{\rho;q;p}(0,rD)$ admits $r$ distinct right factors in
  $\cS'_{\rho;q;p}(0,D)$.  But this property is $W(E_{m+1})$-invariant
  (since it no longer refers to solutions), and thus we may reduce to the
  case $D=f$, where it is just factorization of polynomials on $\P^1$.
\end{proof}

\begin{rem}
  Note that the condition on the section of $\cS'_{\rho;q;p}(0,rD)$ is
  simply that the divisor of its leading coefficient (relative to the
  leading coefficient bundle) be multiplicity free.  (More precisely, the
  divisor is necessarily a sum of effective divisors linearly equivalent to
  $\fD_{\rho;q;p}(D)$, and thus is invariant under a suitable hyperelliptic
  involution; we simply need the orbits to be distinct.)
\end{rem}

Now, let $\oD\in\cS'_{\rho;q;p}(0,D)$, and let $\oD'\in
\cS'_{\rho;q;p}(0,D')$ be another operator, with the problem being to
understand how $\oD'g$ expands in generalized eigenfunctions of
$\cS'_{\rho;q;p}(D',D'+D)$ given that $\oD g=0$.  To apply the previous
lemma, we need to find a section of $\cS'_{\rho;q;p}(D',D'+rD)$ which when
multiplied by $\oD'$ admits $\oD$ as a right factor.  Note that
if such a section is to exist for generic $\oD'$, there should be a
section that works for {\em all} $\oD'$.  (Indeed, the span over all
$\oD'$ should still lie in a finite union of eigenspaces, so we just
need to take a suitable $\gcd$.)

To do this, it will be convenient to blow up our surface further.  Let
$x_{m+1}$ be a point where the leading coefficient of $\oD$ vanishes
as a section of the leading coefficient bundle (one of two such points).
Then we have $\oD\in \cS'_{\rho,x_{m+1};q;p}(0,D-e_{m+1})$, where now
$D-e_{m+1}$ is a $-1$-class.  For any $r$, there is a unique (modulo
scalars) section of
\[
\cS'_{\rho,x_{m+1};q;p}(D',D'+r(D-e_{m+1}))
\subset
\cS'_{\rho,x_{m+1};q;p}(D',D'+rD),
\]
giving us a natural candidate for the operator to apply.  For this to work,
we need any product of that section with an operator of degree $D'$ to have
$\oD$ as a right factor.  To accomplish this, we may use the following
result.

\begin{lem}
  Let $e$ be a $-1$-class.  If $D\cdot e<0$, then the composition morphism
  \[
  \cS'_{\rho;q;C}(e,D)\circ\cS'_{\rho;q;C}(0,e)\to \cS'_{\rho;q;C}(0,D)
  \]
  is surjective.
\end{lem}

\begin{proof}
  This is a $W(E_{m+1})$-invariant statement, so we may reduce to the case
  $e=e_m$, when it corresponds to the fact that
  \[
  \cS'_{\rho;q;C}(0,D-e_m)=\cS'_{\rho;q;C}(0,D)
  \]
  whenever the coefficient of $e_m$ in $D$ is positive.
\end{proof}

In particular, we need simply choose $r$ above so that
$(D'+r(D-e_{m+1}))\cdot (D-e_{m+1})<0$, or equivalently, $r\ge D\cdot
D'+1$.  This gives us the following result.

\begin{prop}
  Let $D$ be a rational pencil other than $f$, and let $D'$ be any other
  divisor class.  Let $g$ be a generic generalized eigenfunction for $D$.
  Then for any $\oD\in \cS'_{\rho;q;p}(0,D')$, $\oD g$ can be
  expressed as a sum of at most $D\cdot D'+1$ generalized eigenfunctions
  for $\cS'_{\rho;q;p}(D',D'+D)$.
\end{prop}

We can be more precise, of course: the generalized eigenvalues can be
labelled by pairs of points on $\C^*/\langle p\rangle$ (i.e., the two
points where the leading coefficient of the operator annihilating $g$
vanishes), and the new generalized eigenvalues correspond to points of the
form $q^{-k}x$, $0\le k\le D\cdot D'$, where $x$ is one of the original two
points.  The genericity condition is that these form $D\cdot D'+1$ distinct
pairs under the appropriate involution.

As an example, consider the case $D=2s+2f-e_1-e_2-e_3-e_4$, $D'=f$.  In
this case, the claim is that multiplying a generalized eigenfunction by any
section of $\cS'_{\rho;q;C}(0,f)$ gives a linear combination of three
generalized eigenfunctions (which must be elliptic functions if the
original was an elliptic function); the resulting (known) expansion is
essentially the analogue for the elliptic biorthogonal functions of the
three-term recurrence for orthogonal polynomials.  Similarly, the fact that
the unique operator of degree $D'=s+f-e_1-e_2$ takes biorthogonal functions
to biorthogonal functions with shifted parameters corresponds to the fact
that $D\cdot D'=0$ in that case, and similarly for the lowering operator
$D'=s$ and the raising operator $D'=s+2f-e_1-e_2-e_3-e_4$.  The latter case
gives rise to Rodrigues-type formulas: for any $r\ge 0$, the image of $1$
under the unique section of $\cS'_{\rho;q;p}(-r(s+2f-e_1-e_2-e_3-e_4),0)$
is a generalized eigenfunction, and the operator of that degree can be
obtained from the operator of degree $rs$ (roughly speaking, the $r$-th
power of the lowering operator) via gauging by a product of $\Gamma$
functions.  Another example is the case $D'=e_4$, with $D\cdot D'=1$; here
the operator itself is trivial, but the equations change, giving an
instance of a sparse change of basis matrix between two families of
biorthogonal functions.

\section{Surjectivity of composition}

Of course, the above flat family of $\Z^{m+2}$-algebras is not quite a
family of noncommutative surfaces, for the simple reason that they do not
come with a natural notion of sheaf.  As we will see below, there are
several natural notions of sheaf (or, rather, of ``torsion''), and it is
nontrivial to show that they all agree.  In order to prove this, we will
need a number of results about surjectivity of composition morphisms.  The
results we obtain along these lines will all hold subject to conditions on
$\rho$ that are invariant under the twist isomorphism, and thus it will
always suffice to show them in the case that the domain of the composed
morphism is the $0$ object.  With this in mind, we introduce the following
short-hand notation:
\[
  [D_1][D_2]\cdots [D_n]
\]
denotes the image of the composition morphism
\begin{align}
\cS'_{\rho;q;C}(D_2+\cdots+D_n,D_1+\cdots+D_n)
\otimes
\cdots
\otimes
\cS'_{\rho;q;C}(D_n,D_{n-1}+D_n)
\otimes
\cS'_{\rho;q;C}(0,D_n)\quad&\\
{}\subset
\cS'_{\rho;q;C}(0,D_1+\cdots+D_n)&.
\notag\end{align}
Thus for instance
\[
  [D_1][D_2]\cdots[D_n]=[D_1+\cdots+D_n]
\]
iff the above composition morphism is surjective, and to show that
$[D_1][D_2][D_3]=[D_1+D_2+D_3]$ for a twist-invariant subset of parameters,
it would suffice to show that $[D_1][D_2]=[D_1+D_2]$ and
$[D_1+D_2][D_3]=[D_1+D_2+D_3]$.

Most of the results we will consider will involve divisor classes in the
following cone.  Say that $D$ is ``universally nef'' if $D\in \N\langle
f,s+f,s+2f-e_1,C_2,\dots,C_m\rangle$ and $D\cdot C_m\ge 0$; we will call
the corresponding convex cone the ``universal nef cone''.  (These are the
divisor classes on $m$-fold blowups of commutative $F_1$ which are nef on
all such blowups admitting smooth anticanonical curves.)  Note that if
$m\le 8$, then the additional constraint is redundant, since all of the
generators of the simplicial cone have $D\cdot C_m\ge 0$.  Also, given any
divisor class
\[
D = af+b(s+f)+c(s+2f-e_1)+d_2C_2+\cdots d_mC_m,
\]
define
\[
X_l(D):=af+b(s+f)+c(s+2f-e_1)+\sum_{2\le i\le l} d_i C_i
\]
for $1\le l\le m$, as well as
\[
X'_0(D):=af+b(s+f)\qquad
X_0(D):=af+c(s+2f-e_1);
\]
note that $X_l(D)$ is supported on the $l$-fold blowup of $F_1$, while
$X'_0(D)$ is supported on $F_1$ and $X_0(D)$ is supported on $F_0/F_2$.

\begin{lem}
  Let $D$ be a universally nef divisor class with $D\cdot C_m>0$, $D$ not a
  multiple of $f$.  Then the leading coefficient morphism
  \[
    [D]\to \Gamma(C;\sO(\fD_{\rho;q;C}(D)))
  \]
  is surjective.
\end{lem}

\begin{proof}
  This follows from known results for $m\le 1$, so we may as well assume
  $m\ge 2$.  We may also assume that $D\cdot e_m\ne 0$, since otherwise
  $D=X_{m-1}(D)$ and we may as well work on in the corresponding
  subcategory.  It follows that $D-C_m$ is universally nef (it is certainly
  in the relevant simplicial cone, so this is immediate if $m\le 8$, and
  for $m>8$ we have $(D-C_m)\cdot C_m>D\cdot C_m$).  We may thus use
  flatness to compute
  \[
  \dim[D]-\dim[D-C_m] = D\cdot C_m = \dim\Gamma(C;\sO(\fD_{\rho;q;C}(D))).
  \]
\end{proof}

\begin{rem}
  Since the surjectivity of the leading coefficient morphism is invariant
  under $W(E_{m+1})$, this actually applies to any divisor class which is
  $W(E_{m+1})$-equivalent to a class of the above form; in other words
  $W(E_{m+1})D$ meets the universal nef cone, $D\cdot C_m>0$, and $D$ is
  not a multiple of a rational pencil.
\end{rem}

It turns out that this is enough to mostly let us reduce to the case $m\le 7$.

\begin{lem}
  Suppose $D_1$, $D_2$ are universally nef divisor classes with $D_i\cdot
  C_m\ge 3$ such that $[X_7(D_1)][X_7(D_2)]=[X_7(D_1+D_2)]$ for all $\rho$,
  $q$, $C$.  Then $[D_1][D_2]=[D_1+D_2]$ for all $\rho$, $q$, $C$.
\end{lem}

\begin{proof}
  If $m\le 7$, the claim is immediate; if $D_1\cdot e_m=D_2\cdot e_m=0$,
  then we may work entirely in the blown down subcategory, and thus the
  claim will hold by induction.  We may thus assume (using the adjoint
  symmetry as necessary) that $m>7$ and $D_1\cdot e_m>0$.  It follows that
  $D_1-C_m$ is again universally nef, and since $m>7$,
  \[
  (D_1-C_m)\cdot C_m\ge D_1\cdot C_m\ge 3.
  \]
  It follows, therefore (by induction on $D_1\cdot e_m$) that
  $[D_1-C_m][D_2]=[D_1+D_2-C_m]$.  In particular, we find that
  \[
  T[D_1+D_2-C_m]=T[D_1-C_m][D_2]\subset [D_1][D_2],
  \]
  and thus to prove surjectivity, it remains only to prove it for the
  leading coefficients.  But this follows immediately from Lemma
  \ref{lem:bundle_products}.
\end{proof}

\begin{rem}
  Note that since the constraints $D_i\cdot C_m\ge 3$ only arise in
  applying Lemma \ref{lem:bundle_products}, the argument works equally well
  if $D_1\cdot C_m=2$, $D_2\cdot C_m\ge 3$, or vice versa.
\end{rem}

Here and in the remainder of the section, the phrase ``for all $\rho$, $q$,
$C$'' should be understood, in the absence of a specific condition
otherwise.

\begin{lem}\label{lem:surject_X5}
  Suppose $D_1,D_2$ are universally nef divisor classes with
  $X_5(D_1)=D_1$, $X_5(D_2)=D_2$.  Then $[D_1][D_2]=[D_1+D_2]$.
\end{lem}

\begin{proof}
  We first note that this holds if $D_1\in \{C_1,\dots,C_5\}$.  Indeed,
  in that case, we find that $[T^0][D_1]$ are the global sections of a line
  bundle of degree $\ge 3$, and thus for any operator in $[D_1+D_2]$, there
  is a corresponding sum of compositions with the same leading
  coefficient.  Since $T[D_2]\subset [C_i][D_2]$, every operator in
  $[D_1+D_2]$ with vanishing leading coefficient is in $[C_i][D_2]$, and
  the claim follows.  Taking adjoints, we also find that the claim holds if
  $D_2\in \{C_1,\dots,C_5\}$.

  Now, suppose more generally that $D_1-C_i$ is universally nef for some
  $i$.  Then we have
  \[
    [D_1][D_2] = [D_1-C_i][C_i][D_2]=[D_1-C_i][D_2][C_i].
  \]
  In particular, if $[D_1-C_i][D_2]=[D_1+D_2-C_i]$, then
  $[D_1][D_2]=[D_1+D_2]$; as a result (and using the adjoint symmetry), we
  may reduce to the case that none of $D_i-C_j$ are universally nef.

  It follows in particular that each of $D_1$, $D_2$ is in one of the
  simplicial cones $\N\langle f,s+f\rangle$, $\N\langle f,s+2f-e_1\rangle$.
  If they are in the same cone, then we reduce to known surjectivity
  results for $F_1$ or $F_0$ as appropriate.  Thus suppose that $D_1$ is in
  the first cone (but not the second) and $D_2$ is in the second cone (but
  not the first); the other case will follow by taking adjoints.  We then
  have
  \begin{align}
    [D_1][D_2] &= [D_1-s-f][s+f][s+2f-e_1][D_2-s-2f+e_1]\notag\\
    &= [D_1-s-f][C_1][D_2-s-2f+e_1]\notag\\
    &= [D_1-s-f][D_2-s-2f+e_1][C_1],
  \end{align}
  and the claim again follows by induction.  Here the fact that
  $[s+f][s+2f-e_1]=[2s+3f-e_1]$ is just Lemma \ref{lem:central_elt}
  combined with surjectivity for leading coefficients.
\end{proof}

In this strong form, this fails for $m>5$; in particular, $C_6$ is
universally nef, but $[C_6][C_6]\subsetneq [2C_6]$ when $q^2=1$.  Indeed,
the leading coefficients of the two copies of $[C_6]$ then lie in a degree
2 line bundle, and thus we only obtain three of the four requisite leading
coefficients of $[2C_6]$.  A related issue is that there are parameters for
which $[f][C_6]\subsetneq [f+C_6]$, since it is possible for the two
leading coefficient bundles to be isomorphic.  These are essentially the
only issues for $m=6$, and indeed we have the following.

\begin{lem}
  Let $D=X_7(D)$ be a divisor class other than $C_6$ such that $X_6(D)$ is
  not a multiple of $f$.  Then $[C_6][D]=[D][C_6]=[D+C_6]$.
\end{lem}

\begin{proof}
  If $X_6(D)\cdot C_6\ge 3$, then we can argue as before: we have
  surjectivity on leading coefficients, and either $D\ne X_6(D)$, so we
  reduce to the $D-C_7$ case, or $D=X_6(D)$, in which case usual argument
  applies, using $T\in [C_6]$.  (Note that this includes the case
  $D=2C_6$.)

  Since $X_6(D)$ is not a multiple of $f$, the only surviving possibility
  for $X_6(D)$ is $C_6$ itself, and the same reduction applies until we
  reach the case $D=C_6+C_7$.  Here, we again have surjectivity on the
  leading coefficients, but the usual argument would require
  $[C_6][C_6]=[2C_6]$.  Luckily, it suffices to show that
  \[
    [C_6-C_7][C_6+C_7] + [C_6][C_6] = [2C_6].
  \]
  This actually {\em is} surjective on leading coefficients; the first term
  gives a codimension 1 subspace of $[T^0][2C_6]$, consisting of sections
  vanishing at the appropriate twist of $x_7$, and it is straightforward to
  check that $[C_6][C_6]$ contains a leading coefficient not in that
  subspace.  Since $[C_6-C_7][C_6]=[e_7][C_6]=[C_6+e_7]$, the claim follows
  in general.
\end{proof}

Since $C_7^2=1$, multiplication by $[C_7]$ cannot be surjective on leading
coefficients, so we must consider multiples thereof.

\begin{lem}\label{lem:C7_surj}
  Let the universally nef divisor class $D=X_7(D)$ be such that $D\cdot
  C_7\ge 3$ and $D\cdot e_7>0$.  Then for all $r>1$,
  $[rC_7][D]=[D][rC_7]=[D+rC_7]$.
\end{lem}

\begin{proof}
  The usual argument reduces this to the case $r=2$, and since we have
  surjectivity on leading coefficients, to showing that
  \[
    [C_7][D]+[2C_7][D-C_7]=[D+C_7].
  \]
  As in the previous Lemma, it suffices to show surjectivity of leading
  coefficients, the first term gives a codimension 1 subspace vanishing at
  a particular point, and the second term includes leading coefficients not
  vanishing there.
\end{proof}

\begin{rem}
  The condition $D\cdot e_7>0$ is necessary for the argument, as otherwise
  the second term {\em also} has a nontrivial $\gcd$, which will agree with
  the $\gcd$ of the first term on a nonempty closed substack of parameter
  space.
\end{rem}

\begin{lem}
  Let $D$ be a universally nef divisor class such that $D\cdot C_m\ge 3$
  and $X_5(D)\ne 0$.  Then $[f][D]=[D][f]=[D+f]$.
\end{lem}

\begin{proof}
  We may as usual assume that $D\cdot e_m>0$, and that $m\in \{6,7\}$
  (since the case $m\le 5$ has already been shown).  Surjectivity on
  leading coefficients is immediate from Lemma \ref{lem:bundle_products},
  and $[f][D-C_m]=[f+D-C_m]$ by induction on $D\cdot (e_6+e_7)$.
\end{proof}

\begin{prop}\label{prop:surject_have_X5}
  Let $D_1$, $D_2$ be universally nef divisor classes such that $D_i\cdot
  C_m\ge 3$ and both $X_5(D_1)$ and $X_5(D_2)$ are nonzero.  Then
  $[D_1][D_2]=[D_1+D_2]$.
\end{prop}

\begin{proof}
  We may assume $m\le 7$, and by the usual argument that $D_1\cdot e_m>0$;
  moreover, if $m\le 5$, we have already shown this to be true.  Since we
  automatically have surjectivity on leading coefficients, we may proceed
  by induction, with the one exception being when $(D_1-C_m)\cdot C_m<3$.
  By inspection of the generators of the (simplicial) universal nef cone,
  this implies that $D_1-C_m\in \{0,f,C_6,C_7,2C_7\}$; together with the
  assumption $X_5(D_1)\ne 0$, this yields the possibilities $D_1=C_6+f$ for
  $m=6$ and $D_1=C_7+f$ for $m=7$.  But in that case, we can use the
  previous lemma as an additional base case for the induction.
\end{proof}

Our most general surjectivity result is the following.

\begin{prop}\label{prop:main_surject}
  Let $D_1$, $D_2$ be universally nef divisor classes with $D_1\cdot C_m\ge
  2$, $D_2\cdot C_m\ge 3$, and suppose that $D_1$ is contained in the
  closure of the open face of the universal nef cone containing $D_2$.
  Then $[D_1][D_2]=[D_2][D_1]=[D_1+D_2]$.
\end{prop}

\begin{proof}
  The usual reduction applies to let us assume $m\in \{6,7\}$.  If
  $D_2\cdot e_m=0$, then $D_1\cdot e_m=0$ (every bounding inequality of the
  universal nef cone which is tight for $D_2$ must by assumption be tight
  for $D_1$ as well), so that we may as well assume $D_2\cdot e_m>0$.
  Similarly, if $X_5(D_1)\ne 0$, then $X_5(D_2)\ne 0$, and again we have
  already shown surjectivity of multiplication.

  If $m=6$, then $D_1\propto C_6$, while the conditions on $D_2$ imply
  $D_2\ne C_6$ and $D_2\not\propto f$, and thus the claim reduces to the
  corresponding claim for $D_1=C_6$, shown above.

  The remaining case case $m=7$, $D_1=aC_6+bC_7$.  If $a>0$ and
  $X_6(D_2)\cdot C_6\ge 3$, then we may reduce to the $m=6$ case in the
  usual way.  Since $X_5(D_2)\ne X_6(D_2)$, the only other case for $a>0$
  is that $X_6(D_2)=C_6$.  Since $D_2\ne C_6$, we can again reduce to the
  case $D_1=C_6$.

  Finally, if $a=0$, so $D_1\propto C_7$, then the claim is just Lemma
  \ref{lem:C7_surj}.
\end{proof}

Applying this to the case $D_2\propto D_1$ gives the following.

\begin{cor}
  If $D$ is a universally nef divisor class with $D\cdot C_m\ge 3$, then
  the sub-$\Z$-algebra with objects $\Z D$ is generated in degree 1.
\end{cor}

Also of interest are the $\Z$-algebras corresponding to the anticanonical
class for $m\in \{6,7\}$.  Note that we can use the elements $T\in [C_6]$
to turn these $\Z$-algebras into graded algebras, which are then generated
in the specified degrees.  (For $m\le 5$, they are of course generated in
degree 1.)

\begin{cor}
  The sub-$\Z$-algebra with objects $\Z C_6$ is locally generated by
  morphisms of degrees $1$ and $2$; the sub-$\Z$-algebra with objects $\Z
  C_7$ is locally generated by morphisms of degrees $1$, $2$, $3$.
\end{cor}

\begin{proof}
  The only claim that does not immediately follow from the surjectivity
  results above is that $[4C_7]$ can be expressed in terms of lower
  degrees, or equivalently that
  $[4C_7]=[C_7][3C_7]+[2C_7][2C_7]+[3C_7][C_7]$.  But this follows from the
  usual argument: the first term is codimension 1, with leading
  coefficients having nontrivial $\gcd$, letting us verify that the second
  term is not contained in the first.
\end{proof}

\medskip

The above results will be enough to let us identify the various notions of
torsion that hold universally, but in general we expect that there should
be far more ample divisors than just those coming from the interior of the
universal nef cone.  Indeed, we expect that every ample divisor on the
generic commutative rational surface of our form should deform to an ample
divisor on the generic noncommutative rational surface in our family.
Relatedly, we expect the $W(E_{m+1})$ symmetry to extend (at least
generically) to the categories of sheaves, which will necessitate having
notions of torsion which are at least invariant under subgroups of the Weyl
group.

To control this, we will need some surjectivity results that only hold away
from countable unions of hypersurfaces in parameter space.  We will always
state the conditions in translation-invariant ways; in general, of course,
this will be far more stringent than necessary (the failure to surject is a
closed condition, so only finitely many of the excluded hypersurfaces
actually need to be excluded), but we will have better methods below for
dealing with such cases, and translation-invariance makes the arguments
simpler.

\begin{lem}\label{lem:surject_Sn}
  Let $S_1, S_2\subset \{1,\dots,m\}$ be subsets for which $\{x_i:i\in
  S_1\setminus S_2\}$ and $\{x_i:i\in S_2\setminus S_1\}$ are disjoint in
  $C/q^\Z$, and define $D_i=2s+(a_i+3)f-\sum_{j\in S_i} e_j$, where $a_i\ge
  \max(|S_i|/2-5/2,0)$.  Then $[D_1+a_1f][D_2+a_2f]=[D_1+D_2+(a_1+a_2)f]$.
\end{lem}

\begin{proof}
  The constraints on $a_1$, $a_2$ are enough to ensure that the leading
  coefficient bundles have degree $\ge 3$, and thus we have surjectivity on
  leading coefficients.  The claim then reduces to showing
  \[
    [D_1-C_m][D_2] + [D_1][D_2-C_m] = [D_1+D_2-C_m]
  \]
  Since $D_1-C_m=a_1f + \sum_{j\notin S_1} e_j$, we see that
  $[D_1-C_m][D_2]=[\sum_{j\notin S_1}e_j][a_1f][D_2]$ and similarly for the
  second term.  We thus see that the first term gives the subspace of the
  leading coefficient bundle obtained by imposing vanishing at suitable
  translates of $x_j$ for $j\in S_2$, and similarly for the second term.
  It then follows immediately that we again have surjectivity on leading
  coefficients, and thus may reduce to
  \[
    [D_1-C_m][D_2-C_m] = [D_1+D_2-2C_m]
  \]
  But this reduces immediately to $[a_1f][a_2f]=[(a_1+a_2)f]$, and the
  claim follows.
\end{proof}

\begin{lem}\label{lem:surject_elem_xform}
  If $\eta/x_1x_2$ is not a power of $q$, then for $b_1\ge 4$, $b_2\ge 3$,
  \[
    [2s+b_1f-2e_1-e_2][2s+b_2f-e_1]
    =
    [4s+(b_1+b_2)f-3e_1-e_2]
  \]
\end{lem}

\begin{proof}
  This is essentially just an elementary transformation of the same
  argument (only on $\cS$ rather than $\cS'$).
\end{proof}

This of course generalizes as much as Lemma \ref{lem:surject_Sn}, but we
will only need the following somewhat technical special case, again based
on $\cS$ rather than $\cS'$ but otherwise the same argument.

\begin{lem}\label{lem:surject_technical}
Suppose that the points $\eta/x_1,x_2,\dots,x_m\in C/q^{\Z}$ are distinct, and
let $D_i = 2s+4f-2e_1-\sum_{m-i<k\le m} e_k$.  Then for $1\le i<m$,
\[
  [C_i+if][D_{m-i}+(m-i)f]=[D_{m+1-i}+(m-i)f][C_i+if]
  =
  [C_i+D_{m+1-i}+mf].
\]
\end{lem}

\begin{rem}
  Note that
  \[
  C_i+D_{m+1-i}=C_m+2s+(m+4)f-2e_1,
  \]
  which is independent of $i$, and is universally nef.
\end{rem}

The above cases are related to the two subsystems $A_{m-1}$; there is also
a surjectivity result similarly related to $D_m$.  Let $\iota_\eta$ denote the
(hyperelliptic) involution on $C$ corresponding to $\eta$.

\begin{lem}\label{lem:surject_Dm}
  For $m\ge 2$, suppose that the images of $x_1,\dots,x_m$ under the
  quotient of $C$ by $\langle q^\Z,\iota_\eta\rangle$ are distinct.  Then for
  any $S\subset \{1,\dots,m\}$,
  \[
    \bigl[s+(|S|+1)f-\sum_{i\in S} e_i\bigr]
    \bigl[s+(m-|S|+1)f-\sum_{i\notin S} e_i\bigr]
    =
    [C_m+(m-1)f].
  \]
\end{lem}

\begin{proof}
  We have surjectivity on leading coefficients since the two leading term
  bundles have degree $|S|+3$ and $m-|S|+3$ respectively, both $\ge 3$.  It
  thus remains only to show that we obtain all requisite multiples of
  $T$. For $j\in S$, we may write
  \begin{align}
    \bigl[s+(|S|+1)f-\sum_{i\in S}&e_i\bigr]
    \bigl[s+(m-|S|+1)f-\sum_{i\notin S} e_i\bigr]\notag\\
    &\supset
    \bigl[(|S|-1)f-\sum_{i\in S\setminus\{j\}} e_i\bigr]
    [s+2f-e_j]
    [s+f]
    \bigl[(m-|S|)f-\sum_{i\notin S}e_i\bigr]
    \notag\\
    &=
    \bigl[(|S|-1)f-\sum_{i\in S\setminus\{j\}} e_i\bigr]
    [2s+3f-e_j]
    \bigl[(m-|S|)f-\sum_{i\notin S}e_i\bigr]
    \notag\\
    &\supset
    \bigl[(|S|-1)f-\sum_{i\in S\setminus\{j\}} e_i\bigr]
    \bigl[\sum_{1\le i\le m;i\ne j} e_i\bigr]
    \bigl[(m-|S|)f-\sum_{i\notin S}e_i\bigr]
    T
    .\notag
  \end{align}
  Similarly, for $j\notin S$, we find
  \begin{align}
    \bigl[(|S|)f-\sum_{i\in S} e_i\bigr]&
    \bigl[\sum_{1\le i\le m} e_i\bigr]
    \bigl[(m-|S|-1)f-\sum_{i\notin S\cup\{j\}}e_i\bigr]
  T\notag\\
  &\subset
  \bigl[s+(|S|+1)f-\sum_{i\in S} e_i\bigr]
  \bigl[s+(m-|S|+1)f-\sum_{i\notin S} e_i\bigr].\notag
  \end{align}
  We have thus obtained a total of $m$ 1-dimensional subspaces of
  $[(m-1)f]$, each consisting of sections vanishing on all but one of a
  fixed set of $m$ distinct points of $\P^1$.  In particular, those
  subspaces span $[(m-1)f]$, and the claim follows.
\end{proof}

We will only need this in the following form.

\begin{cor}\label{cor:surject_Dm}
  For $m\ge 2$, suppose that the images of $x_1,\dots,x_m$ under the
  quotient of $C$ by $\langle q^\Z,\iota_\eta\rangle$ are distinct.  Then
  \[
    [C_i+(i-1)f][2C_m-C_i+(2m-i)f] = [2C_m+(2m-2)f].
  \]
\end{cor}

\begin{proof}
  We have
  \begin{align}
    &[C_i+(i-1)f][2C_m-C_i+(2m-i)f]\notag\\
    &\supset
    [s+(i+1)f-\sum_{1\le j\le i} e_j]
    [s+f]
    [s+(m+1)f-\sum_{1\le j\le m} e_j]
    [s+(m-i+1)f-\sum_{i<j\le m} e_j]\notag\\
    &=
    [s+(i+1)f-\sum_{1\le j\le i} e_j]
    [C_m+(m-1)f]
    [s+(m-i+1)f-\sum_{i<j\le m} e_j]\notag\\
    &=    
    [s+(i+1)f-\sum_{1\le j\le i} e_j]
    [s+(m-i+1)f-\sum_{i<j\le m} e_j]
    [C_m+(m-1)f]\notag\\
    &=
    [C_m+(m-1)f][C_m+(m-1)f]\notag\\
    &=
    [2C_m+(2m-2)f]
  \end{align}
  as required.
\end{proof}

We also recall a surjectivity result from the $F_0$ case (which over $F_0$
simply states that $[s][f]=[f][s]=[s+f]$ under the appropriate constraint).

\begin{lem}\label{lem:surject_fourier}
  If $x_0/x_1$ is not a power of $q$, then
  $[s+f-e_1][f]=[f][s+f-e_1]=[s+2f-e_1]$.
\end{lem}

We will need one more surjectivity result for the case that $q$ is torsion,
in order to show that the resulting surfaces are finite over commutative
surfaces.  Let $q$ be an $r$-torsion point with $C'$ the corresponding
quotient of $C$, and recall the functor $\Phi$ of Proposition
\ref{prop:center_m}.  Here we only need to prove surjectivity for a single
sufficiently ample $D$.

\begin{lem}\label{lem:surject_torsion}
  For all $m\ge 1$ there exists $a\ge 0$ such that for $D=af+\sum_{1\le
    l\le m} C_l$, the multiplication map
  \[
  \cS'_{\rho;q;C}(rD,2rD)\otimes \Phi(\cS'_{\phi_*\circ \rho;1;C'}(0,D))
  \to
  \cS'_{\rho;q;C}(0,2rD)
  \]
  is surjective.
\end{lem}

\begin{proof}
  Note that for $a\gg 0$, $D$ is in the interior of the universal nef cone.
  We also observe that there is an analogue of Lemma
  \ref{lem:bundle_products} in this situation, which states that if ${\cal
    L}$ is a line bundle on $C$ of degree $\ge 3$, then for any line bundle
  ${\cal L}'$ on $C'$ of degree $\ge 2$, the multiplication map
  \[
  \Gamma(C;{\cal L})\otimes \Gamma(C';{\cal L}') \to \Gamma(C;{\cal L}\otimes \phi^*{\cal L}')
  \]
  is surjective.  (If ${\cal L}\otimes (\phi^*{\cal L}')^{-1}$ is acyclic,
  then we can use the same argument as in the first part of Lemma
  \ref{lem:bundle_products}, and simply take the two relatively prime
  global sections of $\phi^*{\cal L}'$ to be pullbacks of global sections
  of ${\cal L}'$; otherwise, we may reduce to the case $\deg{\cal L}'=2$.)

  We may thus use this (increasing $a$ as necessary to make the degrees
  sufficiently large) to give surjectivity on leading coefficients for all
  cases of the form $[rD-b C_m]\Phi([D])$ or $[rD-bC_m]\Phi([D-C_m])$,
  $0\le b<r$, allowing us to induct in $m$.  The base case is $D=af$, for
  which surjectivity for $a\gg 0$ follows from the fact that $\phi$
  descends to a degree $r$ map on $\P^1$.
\end{proof}

\begin{rem}
  Note that the same argument applies to $D=2s+af$ for $a\gg 0$ in either
  case with $m=0$.  For the $\P^2$ case, we may refer to
  \cite[\S 7]{ArtinM/TateJ/VandenBerghM:1991}.
\end{rem}

\section{Torsion modules and sheaves II}

As we mentioned, there are several ways we could try to define
``torsion'' for $\cS'_{\rho;q;C}$-modules.  Since we intend to prove these
are all equivalent, we may as well use the one which is simplest to work
with.  If $v$ is a homogeneous element of a $\cS'_{\rho;q;C}$-module, we
denote by $[D]v$ the space of all images of $v$ under the action of
$\cS'_{\rho;q;C}(\deg(v),\deg(v)+D)$.

\begin{defn}
  Let $M$ be a $\cS'_{\rho;q;C}$-module.  A homogeneous element $v\in M$ is
  {\em torsion} if there is some divisor class $D_0$ such that for any
  universally nef divisor class $D$, $[D+D_0]v=0$.  If every homogeneous
  element of $M$ is torsion, we say that $M$ is torsion.
\end{defn}

Call a divisor class ``universally very ample'' if it is in the interior of
the universal nef cone and satisfies $D\cdot C_m\ge 3$.  This notion gives
us an alternate description of when an element is torsion.

\begin{prop}
   A homogeneous element $v$ is torsion iff there is a universally very ample
   divisor class $D$ such that $[D]v=0$.
\end{prop}

\begin{proof}
  One direction is easy: any translate of the universal nef cone meets the
  interior of the universal nef cone; if $D$ is in the intersection,
  then so is $3D$, and $3D$ will be universally very ample.

  Thus suppose $[D]v=0$ with $D$ universally very ample.  Then for any
  universally nef divisor class $D'$ with $D'\cdot C_m\ge 2$, Proposition
  \ref{prop:main_surject} tells us that $[D'][D]=[D'+D]$, and thus
  $[D'+D]v=[D'][D]v=0$.  Since the space of such $D'+D$ contains a
  translate of the universal nef cone, $v$ is torsion.
\end{proof}

Call a class ``universally ample'' if it is in the interior of the
universal nef cone.

\begin{prop}
  Let $D$ be any universally ample class.  Then a homogeneous element $v$
  is torsion iff there exists $a\ge 0$ such that $[bD]v=0$ for all $b\ge a$.
\end{prop}

\begin{proof}
  If $D$ is universally ample, then $3D$ is universally very ample, so one
  direction is trivial.  For the other direction, we need simply observe
  that for any $D_0$, $aD-D_0$ is universally nef for $a\gg 0$, and thus
  any translate of the universal nef cone contains all sufficiently large
  multiples of $D$.
\end{proof}

The converse can be strengthened.

\begin{prop}
  Let $D$ be a universally nef class with $D\cdot C_m\ge 2$.  If $[D]v=0$
  for some homogeneous element $v$, then $v$ is torsion.
\end{prop}

\begin{proof}
  Let $D_2$ be a universally very ample class, and note that Proposition
  \ref{prop:main_surject} tells us that $[D_2][D]=[D+D_2]$.  But then $v$
  is annihilated by $[D+D_2]$; since $D+D_2$ is universally very ample, the
  claim follows.
\end{proof}

\begin{prop}
  The torsion modules form a Serre subcategory of $\cS'_{\rho;q;C}-\Mod$.
\end{prop}

\begin{proof}
  That ``torsion'' is inherited by sub- and quotient modules is trivial
  (indeed, we may use the same translate of the universal nef cone).  Thus
  it remains only to consider extensions.  Suppose we have a short exact
  sequence $0\to M_1\to M\to M_2\to 0$ with $M_1$, $M_2$ torsion, and
  consider a homogeneous element $v\in M$.  The image of $v$ in $M_2$ is
  torsion by assumption, and thus there exists universally very ample $D_1$
  such that $[D_1]v\in M_1$.  Choose a basis of $w_1,\dots,w_n\in [D_1]v$,
  and note that since $M_1$ is torsion, each $w_i$ is torsion.  It follows,
  therefore, that each $w_i$ is annihilated by some translate of the
  universal nef cone.  The intersection of those translates will itself
  contain a translate of the universal nef cone; it follows that there is a
  divisor $D_0$ such that for all universally nef $D$, $[D+D_0]w_i=0$ for
  $1\le i\le n$.  It follows that $[D+D_0][D_1]v=0$, and thus as long as
  $(D+D_0)\cdot C_m\ge 2$, $[D+D_0+D_1]v=0$.  The resulting set of divisors
  contains a translate of the universal nef cone, and thus $v$ is indeed
  torsion.
\end{proof}

\begin{defn}
  A ``quasicoherent sheaf on $X_{\rho;q;C}$'' is an object in the
  quotient of $\cS'_{\rho;q;C}-\Mod$ by the Serre subcategory of torsion
  modules.
\end{defn}

Note that for any divisor class $D$, the object $-D$ induces a projective
$\cS'_{\rho;q;C}$-module, namely the functor $\cS'_{\rho;q;C}(-D,\_)$.  The
corresponding sheaf will be denoted by $\sO_X(D)$, with $(D)$ omitted when
$D=0$; we will call these sheaves ``line bundles'' by obvious analogy with
the commutative case.  We let $X_{\rho;q;C}$ denote the quasi-scheme
corresponding to the category $\qcoh X_{\rho;q;C}$ of quasicoherent sheaves
and with marked object $\sO_X$.

\begin{defn}
  The {\em saturated $\Z^{m+2}$-algebra} $\hat\cS'_{\rho;q;C}$ is defined
  by
  \[
  \hat\cS'_{\rho;q;C}(D_1,D_2)
  :=
  \Hom(\sO_X(-D_2),\sO_X(-D_1))
  \]
\end{defn}

We will discuss this $\Z^{m+2}$-algebra in more detail below, but will need
a few facts now.  Note that since $\cS'_{\rho;q;C}$ is a domain, the
natural functor $\cS'_{\rho;q;C}\to \hat\cS'_{\rho;q;C}$ is faithful.  We
also have the following key point: the morphisms in the saturated
$\Z^{m+2}$-algebra are still difference operators.  More precisely, we have
the following.

\begin{lem}
  The space $\hat\cS'_{\rho;q;C}(D_1,D_2)$ can be identified
  (compatibly with composition) with the space of elliptic difference
  operators $\oD$ such that
  \[
  \cS'_{\rho;q;C}(D_2,D_3)\oD\subset \cS'_{\rho;q;C}(D_1,D_3)
  \]
  for all $D_3$ in some translate of the universal nef cone.
\end{lem}

\begin{proof}
  It follows from the definition of torsion that any morphism from
  $\sO_X(-D_2)$ to $\sO_X(-D_1)$ is represented by a compatible family of
  maps
  \[
  \Phi:\cS'_{\rho;q;C}(D_2,D_3)\to \cS'_{\rho;q;C}(D_1,D_3)
  \]
  for $D_3$ ranging over a translate of the universal nef cone.  We thus
  need to show that this family can be expressed as right multiplication by
  a difference operator.  Let
  \[
  D = \sum_{1\le i\le m} (C_i+if)
  \]
  and note that any translate of the universal nef cone contains some point
  of the form $D_2+rD$.  Now, $\cS'_{\rho;q;C}(D_2,D_2+rD)$ contains an
  operator of the form $g(z) T^{rm}$, and for any difference operator
  $\oD\in \cS'_{\rho;q;C}(D_2,D_3)$, there is an integer $a\ge 0$ and a
  morphism
  \[
  G(z)\in \cS'_{\rho;q;C}(D_3+rC_m,D_3+rD+af)
  \]
  such that
  \[
  G(z) T^{rm} \oD\in \cS'_{\rho;q;C}(D_2+rD,D_3+rD+af) g(z)T^{rm}.
  \]
  Then
  \begin{align}
  \Phi(\oD)
  &= T^{-rm} G(z)^{-1} \Phi(G(z)T^{rm}\oD)\notag\\
  &= T^{-rm} G(z)^{-1} \Phi(G(z)T^{rm}\oD T^{-rm}g(z)^{-1} g(z)T^{rm})\notag\\
  &= \oD T^{-rm} g(z)^{-1} \Phi(g(z) T^{rm}),
  \end{align}
  making it right multiplication by an elliptic difference operator as required.
\end{proof}

As a result, there is a natural notion of leading coefficient for elements
of $\hat\cS'_{\rho;q;C}(D_1,D_2)$.

\begin{lem}\label{lem:T0_saturated}
  There is an exact sequence
  \[
  \begin{CD}
  0@>>> \hat\cS'_{\rho;q;C}(0,D-C_m)@>T>> \hat\cS'_{\rho;q;C}(0,D)
  @>[T^0]>> \Gamma(C;\sO(\fD_{\rho;q;C}(D))).
  \end{CD}
  \]
\end{lem}

\begin{proof}
  That the leading coefficient lies in the appropriate bundle follows by
  comparing leading coefficients when multiplying by operators of
  sufficiently (universally) ample degree.  It thus remains only to show
  that if the leading coefficient vanishes, then we can factor out $T$.
  But this follows immediately from the corresponding statement for the
  original $\Z^{m+2}$-algebra.
\end{proof}

As one might expect, we do not in general need to test the entire translate
of the universal nef cone.

\begin{lem}\label{lem:saturate_nef_test1}
  Let $D$, $D'$ be divisor classes such that $D'$ is universally nef, with
  $D'\cdot C_m\ge 2$.  If $\oD$ is an elliptic difference operator
  such that
  \[
  \cS'_{\rho;q;C}(D,D+D')\oD\subset \cS'_{\rho;q;C}(0,D+D'),
  \]
  then $\oD\in \hat\cS'_{\rho;q;C}(0,D)$.
\end{lem}

\begin{proof}
  Let $D''$ be any universally very ample divisor class.  Then
  \[
    [D''+D']\oD = [D''][D']\oD\subset [D''][D+D']\subset [D+D'+D'']
  \]
  as required.
\end{proof}

\begin{lem}
  For $m\ge 1$, if $D\cdot e_m=0$, then
  \[
  \hat\cS'_{\eta,x_0,\dots,x_m;q;C}(0,D)
  =
  \hat\cS'_{\eta,x_0,\dots,x_{m-1};q;C}(0,D).
  \]
\end{lem}

\begin{proof}
   If $\oD\in \hat\cS'_{\eta,x_0,\dots,x_{m-1};q;C}(0,D)$, then we
   certainly have $[D']\oD\subset [D+D']$ for some divisor $D'$ which
   is universally very ample for $X_{m-1}$.  But this divisor remains
   universally nef after blowing up, and thus $\oD\in
   \hat\cS'_{\eta,x_0,\dots,x_m;q;C}(0,D)$.

   Conversely, if $\oD\in \hat\cS'_{\eta,x_0,\dots,x_m;q;C}(0,D)$, then
   there exists $D'$ such that (a) $X_{m-1}(D')\cdot C_m\ge 2$,
   and (b) $D'$ is universally nef with $[D']\oD\subset [D+D']$, since (a)
   and (b) both contain translates of the universal nef cone.  But then the
   claim follows by restricting ones attention to those operators of degree
   $D'$ for which the first $D'\cdot e_m$ coefficients vanish.
\end{proof}

\begin{rem}
  This is essentially the same argument that showed in the $m=0$ cases that
  it was enough to saturate with respect to $f$.
\end{rem}

%

\begin{lem}
  Let $D$ be a universally nef divisor class with $D\cdot C_m>0$.  Then
  \[
  \hat\cS'_{\rho;q;C}(0,D)=\cS'_{\rho;q;C}(0,D).
  \]
\end{lem}

\begin{proof}
  We first note that we may assume that either $m=0$ (where the claim is
  true for {\em all} divisors) or $D\cdot e_m>0$, since otherwise we may
  instead saturate in the blown down $\Z^{m+1}$-algebra.

  If $D_2-D_1$ is a multiple of $f$, then $\hat\cS'_{\rho;q;C}(D_1,D_2)$
  consists of operators of degree 0, and the corresponding functions must
  not only lie in the correct leading coefficient bundle, but must also
  have the correct symmetry; the claim follows immediately.

  Otherwise, the leading coefficient map is already surjective before
  saturation, and thus we reduce to the claim for $D-C_m$, so eventually to
  smaller $m$.
\end{proof}

This lets us show that we get the same result if we saturate on the other
side.

\begin{cor}
  Let $D$ be any divisor class.  Then $\oD\in
  \hat\cS'_{\rho;q;C}(0,D)$ iff for all $D'$ in some translate of the
  universal nef cone,
  \[
    \oD\cS'_{\rho;q;C}(-D',0)\subset \cS'_{\rho;q;C}(-D',D).
  \]
\end{cor}

\begin{proof}
  Suppose first that $\oD\in \hat\cS'_{\rho;q;C}(0,D)$.  Then it is
  certainly true that
  \[
    \oD\hat\cS'_{\rho;q;C}(-D',0)\subset \hat\cS'_{\rho;q;C}(-D',D)
  \]
  for all $D'$.  If $D+D'$ is universally nef with
  $(D+D')\cdot C_m>0$, then $\hat\cS'_{\rho;q;C}(-D',D)=
  \cS'_{\rho;q;C}(-D',D)$ and thus
  \[
    \oD\cS'_{\rho;q;C}(-D',0)\subset 
    \oD\hat\cS'_{\rho;q;C}(-D',0)\subset \hat\cS'_{\rho;q;C}(-D',D)
    =\cS'_{\rho;q;C}(-D',D)
  \]
  as required.

  Taking into account the translation symmetry, we have shown that if an
  operator $\oD$ satisfies
  \[
  \cS'_{\rho;q;C}(D_2,D_2+D')\oD\subset \cS'_{\rho;q;C}(D_1,D_2+D')
  \]
  for all $D'$ in some translate of the universal nef cone, then
  \[
  \oD\cS'_{\rho;q;C}(D_1-D',D_1)\subset \cS'_{\rho;q;C}(D_1-D',D_2)
  \]
  for all $D'$ in some translate of the universal nef cone.  In particular,
  the converse is just the adjoint of what we have already shown!
\end{proof}

\begin{rem}
  In particular, we see that the adjoint symmetry extends to the saturated
  $\Z^{m+2}$-algebra.
\end{rem}
  
We also obtain the following slight improvement of Lemma
\ref{lem:saturate_nef_test1}.  (The difference is that the product is only
required to be in the saturated $\Z^{m+2}$-algebra, rather than the
unsaturated $\Z^{m+2}$-algebra.)

\begin{lem}\label{lem:saturate_nef_test2}
  Let $D$, $D'$ be divisor classes such that $D'$ is universally nef, with
  $D'\cdot C_m\ge 2$.  If $\oD$ is an elliptic difference operator
  such that
  \[
  \cS'_{\rho;q;C}(D,D+D')\oD\subset \hat\cS'_{\rho;q;C}(0,D+D'),
  \]
  then $\oD\in \hat\cS'_{\rho;q;C}(0,D)$.
\end{lem}

\begin{proof}
  Let $D''$ be a universally nef divisor class such that $D+D'+D''$ is
  universally nef with $(D+D'+D'')\cdot C_m>0$.  Then
  \[
    [D''+D']\oD\subset 
    [D''][D']\oD\subset [D'']\hat\cS'_{\rho;q;C}(0,D+D')
    \subset \hat\cS'_{\rho;q;C}(0,D+D'+D'')
    =\cS'_{\rho;q;C}(0,D+D'+D'').
  \]
  Since $D'+D''$ is universally nef with $(D'+D'')\cdot C_m\ge 2$, the
  result follows from Lemma \ref{lem:saturate_nef_test1}.
\end{proof}

\medskip

Although the above is enough to give a reasonably sensible notion of sheaf,
we still need to compare it to other notions of torsion.  The most
important of these is the following.

\begin{defn}
  A homogeneous element $v$ of a $\cS'_{\eta,x_0,\dots,x_m;q;C}$-module is
  {\em inductively torsion} if there exists an integer $r_m$ such that for
  any $D$ with $D\cdot e_m\ge r_m$, every element of $[D]v$ is torsion as
  an element of the corresponding
  $\cS'_{\eta,x_0,\dots,x_{m-1};q;C}$-module.
\end{defn}

Here, we should really say ``inductively torsion'' in the condition on
$[D]v$, but luckily this is an unnecessary distinction.

\begin{prop}
  A homogeneous element $v$ is inductively torsion iff it is torsion.
\end{prop}

\begin{proof}
  Let $v$ be inductively torsion, and let $D_1$ be any universally
  very ample divisor class such that $D_1\cdot e_m\ge r_m$, so that
  $[D_1]v$ is torsion over $\cS'_{\eta,x_0,\dots,x_{m-1};q;C}$.  It follows
  that there is a divisor $D_2$ with $D_2\cdot e_m=0$ such that for any
  universally nef divisor $D_3$ with $D_3\cdot e_m=0$, $[D_2+D_3][D_1]v=0$.
  Again, we may choose $D_2$ so that Proposition \ref{prop:main_surject}
  gives $[D_2+D_3][D_1]=[D_1+D_2+D_3]$ for all $D_3$, and thus
  $[D_1+D_2+D_3]v=0$, so that $v$ is torsion.

  Now, suppose $[D]v=0$ for some universally very ample divisor class, and
  let $D'$ be any divisor class with $D'\cdot e_m>D\cdot e_m$.  It will
  suffice to show that every element of $[D']v$ is torsion over
  $\cS'_{\eta,x_0,\dots,x_{m-1};q;C}$.  Let $D''$ be a universally very
  ample divisor for $X_{m-1}$, and note that we may choose $D''$ so that
  $D'+D''-D$ is universally very ample for $X_m$.  But then
  \[
    [D''][D']v\subset [D'+D'']v = [D'+D''-D][D]v=0,
  \]
  so that every element of $[D']v$ is torsion as required.
\end{proof}

The significance of this is that ``inductively torsion'' is essentially
just the usual notion of torsion for a sheaf $\Z$-algebra, in this case a
sheaf $\Z$-algebra over $\cS'_{\eta,x_0,\dots,x_{m-1};q;C}$.

\begin{thm}\label{thm:Xm_is_blowup}
  The noncommutative scheme $X_{\eta,x_0,\dots,x_m;q;C}$ is the blowup in
  the sense of \cite{VandenBerghM:1998} of $X_{\eta,x_0,\dots,x_{m-1};q;C}$
  at the point $x_m$.  In particular, $X_{\eta,x_0,\dots,x_m;q;C}$ is a
  noetherian quasi-scheme.
\end{thm}

\begin{proof}
  As we have already mentioned, $\qcoh X_{\eta,x_0,\dots,x_m;q;C}$ is the
  category of sheaves on a sheaf $\Z$-algebra over
  $X_{\eta,x_0,\dots,x_{m-1};q;C}$.  Moreover, since every $\Hom$ space in
  $\cS'_{\eta,x_0,\dots,x_m;q;C}$ is a subspace of a corresponding $\Hom$
  space in $\cS'_{\eta,x_0,\dots,x_{m-1};q;C}$, the $\Hom$ sheaves in this
  sheaf $\Z$-algebra are all $2$-sided ideals.  Moreover, this sheaf
  $\Z$-algebra is generated in degree $1$ (i.e., $-e_m$): to see this, we
  need simply note that there is a divisor $D_0=X_{m-1}(D_0)$ such that for any
  universally nef divisor $D=X_{m-1}(D)$,
  \[
    [rD_0+D-re_m] = [D_0-e_m][(r-1)D_0+D-(r-1)e_m],
  \]
  and thus the ideal of degree $r$ is the product of the ideal of degree
  $r-1$ and the ideal of degree $1$.  Each degree $1$ ideal is cut out by
  the condition that the leading coefficient vanish at the appropriate
  point, and is thus the ideal sheaf of that point.

  In other words, this sheaf $\Z$-algebra is precisely the sheaf
  $\Z$-algebra appearing in Van den Bergh's noncommutative blowup.  (There
  is a technical point to be addressed below: we need to show that the
  point being blown up is contained in a commutative ``divisor''.  The
  argument for this will use Corollary \ref{cor:ample_is_ample}, which does
  not depend on $X_{\rho;q;C}$ being a blowup or even noetherian.)  Since
  the blowup construction preserves the noetherian quasi-scheme property
  (and this is known for noncommutative Hirzebruch surfaces
  \cite{ChanD/NymanA:2013}), the remaining claim follows by induction.
\end{proof}

\begin{rems}
  Strictly speaking, Van den Bergh's construction uses a graded
  sheaf-algebra rather than a sheaf $\Z$-algebra, but it is easy to see
  that the divisor condition lets one identify the $\Z$-algebra with one
  coming from a graded algebra.
\end{rems}
 
\begin{rems}
  The same argument shows that $X_{\eta,x_0,\dots,x_m;q;C}$ is {\em
    strongly} noetherian in the sense of
  \cite{ArtinM/ZhangJJ:2001,ChanD/NymanA:2013}: it remains noetherian under
  base change to any (commutative) noetherian $k$-algebra.
\end{rems}

\begin{cor}
  When $q=1$, the quasi-scheme $X_{\eta,x_0,\dots,x_m;1;C}$ is a
  commutative rational surface (the iterated blowup of $F_1\cong
  X_{\eta,x_0;1;C}$ in the points $x_1,\dots,x_m$).
\end{cor}

It thus makes sense to define ``coherent'' sheaves: these are precisely the
noetherian objects of the category of quasicoherent sheaves.

\begin{cor}
  For any divisor class $D$, the sheaf $\sO_X(D)$ is coherent.
\end{cor}

\begin{proof}
  That this holds for $\sO_X$ follows from the fact that $X_{\rho;q;C}$ is
  a noetherian quasi-scheme.  The translation symmetry induces an
  equivalence of sheaf categories taking the sheaf $\sO_X$ to $\sO_X(D)$,
  which must therefore also be noetherian.
\end{proof}

The Theorem tells us that the quasi-schemes we have constructed are not
new; however, our construction gives a great deal more control over
the category, and thus allows us to prove new symmetries.

\begin{prop}
  If $x_0/x_1$ is not a power of $q$, then the quasi-scheme
  $X_{\eta,x_0,x_1;q;C}$ is isomorphic to the quasi-scheme
  $X_{x_0\eta/x_1,x_1,x_0;q;C}$, in such a way that the isomorphism
  preserves the anticanonical curve.
\end{prop}

\begin{proof}
  Since the corresponding $\Z^3$-algebras are related by an isomorphism
  (essentially the Fourier transform), we simply need to show that the two
  notions of torsion are equivalent.  In particular, the divisor
  $D_1=3s+5f-2e_1$ is invariant under that symmetry; thus if we can show that an
  element is torsion iff $[rD_1]v=0$ for some $r>0$, the claim will
  follow.  Note that $D_1$ is universally nef with $D_1\cdot C_1=11\ge 2$,
  and thus if $[rD_1]v=0$ for some $r>0$, then $v$ is torsion.  It thus
  remains only to show the converse.  Let $D_0=2s+4f-e_1$.  Since $D_0$ is
  universally very ample, any homogeneous torsion element will be
  annihilated by $rD_0$ for $r\gg 0$.  If we can show that
  $[r(D_1-D_0)][rD_0]=[rD_1]$, then the claim will follow.  Since
  $[(r-1)D_1][D_1]=[rD_1]$ and
  $[D_1][(r-1)D_0]=[(r-1)D_0+D_1]=[(r-1)D_0][D_1]$, we may immediately
  reduce to the case $r=1$, which follows from the calculation
  \[
    [s+f-e_1][2s+4f-e_1]
    =
    [s+f-e_1][f][2s+3f-e_1]
    =
    [s+2f-e_1][2s+3f-e_1]
    =
    [3s+5f-e_1];
  \]
  here the first and last steps are just Lemma \ref{lem:surject_X5}, while
  the middle step is Lemma \ref{lem:surject_fourier}.
\end{proof}

\begin{prop}
  If $\eta/x_1x_2$ is not a power of $q$, then the quasi-schemes
  $X_{\eta,x_0,x_1,x_2;q;C}$ and
  $X_{\eta,x_0\eta/x_1x_2,\eta/x_2,\eta/x_1;q;C}$ are
  isomorphic, in such a way that the isomorphism preserves the
  anticanonical curve.
\end{prop}

\begin{proof}
  The same argument applies, with $D_1=6s+12f-4e_1-2e_2$,
  $D_0=4s+7f-2e_1-e_2$, and using Lemma \ref{lem:surject_elem_xform} in
  place of Lemma \ref{lem:surject_fourier}.  
\end{proof}

\begin{prop}
  If $x_{m-1}/x_m$ is not a power of $q$, then the quasi-schemes
  $X_{\eta,x_0,\dots,x_{m-2},x_{m-1},x_m;q;C}$ and
  $X_{\eta,x_0,\dots,x_{m-2},x_m,x_{m-1};q;C}$ are isomorphic, in such a
  way that the isomorphism preserves the anticanonical curve.
\end{prop}

\begin{proof}
  Let $D$ be a divisor which is in the interior of the facet of the
  universal nef cone cut out by $D\cdot (e_{m-1}-e_m)$; replacing $D$ by
  $3D$ as necessary, we may assume $D\cdot C_m\ge 3$.  Let
  $D'=2s+af-e_1-\cdots-e_m$ for $a\gg 0$.  Then
  $D_1=D+2D'+e_{m-1}+e_m$ is still in the interior of the relevant facet, 
  and $D_0=D+D'+e_m$ is universally very ample.  It thus suffices as before to
  show that $[D_1-D_0][D_0]=[D_1]$, or in other words that
  \[
    [D'+e_{m-1}][D+D'+e_m]=[D+2D'+e_{m-1}+e_m].
  \]
  This follows from the calculation
  \[
    [D'+e_{m-1}][D+D'+e_m]
    =
    [D'+e_{m-1}][D'+e_m][D]
    =
    [2D'+e_{m-1}+e_m][D]
    =
    [D+2D'+e_{m-1}+e_m].
  \]
  Here the first and third steps are Proposition \ref{prop:main_surject},
  while the middle step is Lemma \ref{lem:surject_Sn}.
\end{proof}

Combining these gives us the following result.

\begin{thm}
  Let $r_i$ be any simple root of $E_{m+1}$, with corresponding reflection
  $s_i$.  If $\rho(r_i)$ is not a power of $q$, then $X_{\rho;q;C}\cong
  X_{\rho\circ s_i;q;C}$.
\end{thm}

\begin{proof}
  The only thing left to observe is that the symmetries we have shown all
  preserve the anticanonical curve, and thus will continue to hold after
  blowing up.
\end{proof}

Call an element $w\in W(E_{m+1})$ {\em admissible} for $\rho$ if there is
some reduced word for $w$ such that each simple reflection in that word
induces an equivalence under the Theorem, starting with $X_{\rho;q;C}$,
thus giving an isomorphism $X_{\rho;q;C}\cong X_{\rho\circ w^{-1};q;C}$.

There is an alternate description that does not depend on a choice of
reduced word.

\begin{prop} An element $w$ is admissible iff no positive root that changes
  sign under $w$ is taken to a power of $q$ by $\rho$.
\end{prop}

\begin{proof}
  The claim is certainly true if $w$ is the identity, and by definition
  when $w$ is a simple reflection.  More generally, let $w=s_1s_2\cdots
  s_l$ be any reduced word for $w$. By induction on the length of $w$, we
  may assume that the claim is true for $s_1w$; that is, that $s_1w$ is
  admissible for precisely those $\rho$ satisfying the given condition.

  Suppose $s_1w$ is admissible for $\rho$, or equivalently that no positive
  root that changes sign under $s_1w$ maps to a power of $q$ under $\rho$.
  By standard Coxeter theory, the set of positive roots made negative by
  $w$ consists of those roots made negative by $s_1w$ together with the
  root that $w$ maps to the simple root corresponding to $s_1$.  Thus $s_1$
  is admissible for $\rho\circ (s_1w)^{-1}$ iff $\rho$ does not map that
  root to a power of $q$.

  In other words, $w$ is admissible for $\rho$ relative to the given
  reduced expression iff it satisfies the desired invariant description;
  thus the latter is equivalent to the admissibility of $w$ with respect to
  any reduced expression.
\end{proof}

Since the above symmetries do not preserve the universal nef cone (and thus
do not preserve the set of universally very ample divisor classes), we
obtain a large set of other notions of ``torsion'' by applying admissible
elements.

\begin{defn} A divisor class $D$ is {\em nef} for $X_{\rho;q;C}$ if for
  some admissible element $w\in W(E_{m+1})$, $wD$ is universally nef.
\end{defn}

Note that if $q=1$, this agrees with the usual notion of ``nef'' (having
nonnegative intersection with every curve); we will show below that there
is an analogous description in our setting.  As this will also show that
the nef divisors form a semigroup (the integer points of a suitable cone),
we will refer to that set as the ``nef cone''.  (This can also be shown
directly from Coxeter theory.)

\begin{lem}
  Let $D$ be a nef divisor class with $D\cdot C_m\ge 2$.  Then any
  homogeneous element $v$ with $[D]v=0$ is torsion.
\end{lem}

\begin{proof}
  The notion of ``torsion'' is invariant under admissible simple
  reflections, and thus we may reduce to the case that $D$ is in the
  universal nef cone.
\end{proof}

Ideally, we would have a result saying that we can replace the universal
nef cone with the nef cone when defining ``torsion''.  This has
difficulties with the fact that the nef cone is in general a union of
infinitely many images of the universal nef cone.  To make things more
manageable, it is useful to introduce an intermediate notion.  Suppose $S$
is a parabolic subsystem of $E_{m+1}$ such that $W(S)$ is finite and every
element of $W(S)$ is admissible.  Then we say that $D$ is ``$S$-nef'' if
some point of its $W(S)$-orbit is universally nef; note that these elements
form a cone.  We say that $D$ is ``$S$-very ample'' if it is contained in
the interior of the $S$-nef cone and $D\cdot C_m\ge 3$.

\begin{prop}
  Let $S$ be as above, and let $v$ be a homogeneous element of a
  $\cS'_{\rho;q;C}$-module.  Then the following are equivalent:
  \begin{itemize}
  \item[(1)] $v$ is torsion.
  \item[(2)] There is some translate of the $S$-nef cone annihilating $v$.
  \item[(3)] There is some $W(S)$-invariant translate of the $S$-nef cone
    annihilating $v$.
  \item[(4)] There is some $S$-very ample divisor $D$ such that $[D]v=0$.
  \item[(5)] There is some $W(S)$-invariant $S$-very ample divisor $D$ such
    that $[D]v=0$.
  \end{itemize}
\end{prop}

\begin{proof}
  That $3\implies 2$ and $5\implies 4$ are tautologies, that $2\implies 4$
  and $3\implies 5$ follows from the fact that any two translates of the
  $S$-nef cone intersect (and thus any translate contains an $S$-very ample
  divisor), while $5\implies 3$ by Proposition \ref{prop:main_surject}.
  Moreover, $4\implies 1$ since any $S$-very ample divisor is nef.  It thus
  remains only to show that $1\implies 5$.

  Suppose first that $s-e_1\in S$ (so, in particular, $s-e_1$ is
  admissible).  Let $D$ be an $W(S)$-invariant $S$-very ample divisor $D$
  such that $D-2f$ is $S$-very ample.  Note that $D-f$ is then also
  $S$-very ample, and invariant under a smaller subgroup; there is thus an
  admissible element taking $D-f$ to the universal nef cone and giving an
  $S'$-very ample element for some $S'\subsetneq S$.  It follows (by
  induction on $|S|$) that there exists $r>0$ such that $[r(D-f)]v=0$.  To
  finish, it will suffice to show that $[f][D-f]=[D]$.  Since $D+af$ is nef
  for $a\in \{-2,-1,0\}$, we have
  $\cS'_{\rho;q;C}(0,D+af)=\hat{\cS}'_{\rho;q;C}(0,D+af)$.  It follows in
  particular that as $a$ varies, $\cS'_{\rho;q;C}(0,D+af)$ is the saturated
  module representing a sheaf on $\P^1$.  Since
  \[
  \dim \cS'_{\rho;q;C}(0,D)-
  2\dim \cS'_{\rho;q;C}(0,D-f)
  +\dim \cS'_{\rho;q;C}(0,D-2f)
  =0,
  \]
  and thus the same holds for the saturated spaces, we immediately conclude
  that $[f][D-f]=[D]$.

  We may thus reduce to the case that $s-e_1\notin S$, so that $S$
  is contained in the root system $D_m$.  Let $w_0$ be the longest element
  of $W(S)$, and for $1\le i\le m$, let $F_i$ denote the divisor $C_i+(m-1)f$.
  The divisor $D=\sum_{1\le i\le m} (F_i+w_0F_i)$ is $W(S)$-invariant, and
  we claim that $[rD]v=0$ for $r\gg 0$.  This is certainly true for the
  universally ample divisor $\sum_i F_i$, and thus for its image under
  $w_0$.  It will thus suffice to show
  \[
    [\sum_i w_0 F_i][\sum_i F_i] = [D].
  \]
  Since
  \[
    [\sum_i F_i] \supset \prod_i [F_i]
  \]
  and similarly for $[\sum_i w_0 F_i]$, it will suffice to show that we can
  pair up the factors so that the product for each pair is surjective and
  results in a universally nef degree (which by inspection has $X_5(D)\ne 0$).
  We may then apply Proposition \ref{prop:surject_have_X5} to move those
  factors to the end and induct.  That $[F_i][w_0F_i]=[F_i+w_0(F_i)]$ for
    each $i$ follows from Corollary \ref{cor:surject_Dm},
    Lemma \ref{lem:surject_technical}, or Lemma \ref{lem:surject_Sn},
    depending on how the component(s) (if any) of $S$ not fixing $F_i$ meet
    $\{f-e_1-e_2,e_1-e_2\}$.
\end{proof}

Call a divisor class ``ample'' (for $X_{\rho;q;C}$) if it is in the
interior of the nef cone.

\begin{lem}
  Any ample divisor class satisfies $D\cdot C_m>0$ and $D\cdot e_m>0$.
\end{lem}

\begin{proof}
  It is equivalent to show that any nef divisor class satisfies $D\cdot
  C_m,D\cdot e_m\ge 0$.  The first claim is trivial, since $C_m$ is
  invariant under $W(E_{m+1})$.  For the second claim, we need simply
  observe as in the proof of \cite[Thm.~3.4]{rat_Hitchin} that for any
  element $w\in W(E_{m+1})$, $w e_m-e_m$ is a nonnegative sum of roots.  It
  follows that $D\cdot w e_m\ge 0$ for any universally nef divisor class,
  and thus $w^{-1}D\cdot e_m\ge 0$; since $w$ was arbitrary, this in
  particular applies to nef divisors.
\end{proof}

The name ``ample'' is justified by the following results.

\begin{thm}\label{thm:ample_is_ample_for_torsion}
  Let $D$ be an ample divisor class, and let $v$ be a homogeneous element
  of a $\cS'_{\rho;q;C}$-module.  Then $v$ is torsion iff $[rD]v$ for all
  sufficiently large $r$ iff $[rD-\deg(v)]v=0$ for all sufficiently large
  $r$.
\end{thm}

\begin{proof}
  As usual, we may assume that $D$ is universally nef.  Let $S$ be the
  parabolic subsystem of roots orthogonal to $D$.  If we can show that
  $W(S)$ is finite and admissible, then any torsion element will be
  annihilated by some $W(S)$-invariant translate of the $S$-nef cone, and
  that cone will contain $D'+rD$ for all $D'$ and all sufficiently large
  $r$ (depending on $D'$).  Conversely, $rD$ and $rD-\deg(v)$ will both be
  $S$-very ample for sufficiently large $r$, giving the desired
  equivalence.

  Since $D\cdot C_m,D\cdot e_m>0$, the only facets of the universal nef
  cone that might contain $D$ are those corresponding to simple
  reflections; in particular, the open face containing $D$ is precisely the
  interior of the subcone of $W(S)$-invariant elements.  Now, let $A_D$
  denote the set of directions in $\R\langle S\rangle$ for which
  infinitesimal deformations of $D$ remain inside the nef cone.  This
  contains the fundamental chamber of $S$, and all conjugates of the
  fundamental chamber by admissible elements.  In particular, it is
  contained in the Tits cone of $W(S)$.  Since $D$ is an interior point,
  $A=\R\langle S\rangle$, and must in particular equal the Tits cone.
  It follows that $W(S)$ is finite (so that the Tits cone is the entire
  space), and that every element of $W(S)$ is admissible (since otherwise
  $A$ would be missing some chambers).  But this was precisely what we
  needed to show.
\end{proof}

\begin{thm}
  Let $D$ be an ample divisor class.  Then for any sheaf $M$ on
  $X_{\rho;q;C}$, the natural map
  \[
  \bigoplus_{r\ge R} \sO_X(-rD)\otimes \Hom(\sO_X(-rD),M)\to M
  \]
  is surjective for all $R$.  If $M$ is coherent, then the term $r=R$
  suffices for $R\gg 0$.
\end{thm}

\begin{proof}
  The category $\qcoh X_{\rho;q;C}$ is generated by the set of {\em all}
  line bundles (since $\cS'_{\rho;q;C}-\Mod$ is generated by the
  corresponding projective modules).  As a result, we need only prove the
  Theorem in the case that $M=\sO_X(D')$.  In other words, we need to show
  that for each $D'$, there exists $r$ such that for all $D''$ in some
  translate of the universal nef cone, the map
  \[
  \cS'_{\rho;q;C}(rD,D'')\otimes \hat\cS'_{\rho;q;C}(-D',rD)
  \to
  \cS'_{\rho;q;C}(-D',D'')
  \]
  is surjective.  Moreover, by Theorem
  \ref{thm:ample_is_ample_for_torsion}, it suffices to prove this for
  $D''=sD$ for $s\gg 0$, where as usual we assume that $D$ is universally
  nef.  We need simply take $r$ large enough that $rD+D'$ is nef (and can
  moreover arrange for it to be universally nef using admissible symmetries
  fixing $D$), with $(rD+D')\cdot C_m\ge 3$, at which point surjectivity
  holds whenever $s\ge r+3$.
\end{proof}

The following is immediate. (See, e.g., \cite{ArtinM/ZhangJJ:1994}.)

\begin{cor}\label{cor:ample_is_ample}
  For any ample divisor class $D$, there is an isomorphism
  \[
  X_{\rho;q;C}\cong \Proj(\cS'_{\rho;q;C}|_{\Z D})
  \]
\end{cor}

\begin{rem}
  One special case of particular interest is the case $D=C_m$ for $m<8$.
  This is only ample when all of $W(E_{m+1})$ is admissible, or in other
  words when $\rho(r)\notin q^{\Z}$ for all roots $r$.  In this case, since
  each $\Hom$ space in the $\Z$-algebra has a canonical section ($T^m$), we
  may identify it with an actual graded algebra, which will be a flat
  deformation of the anticanonical embedding of a del Pezzo surface.
\end{rem}

\begin{cor}
  In the case $C=\C^*/\langle p\rangle$, the $\cS'_{\rho;q;p}$-module $\Mer$ is
  the module of global sections of a quasi-coherent sheaf.
\end{cor}

\begin{proof}
  We need to prove that $\Mer$ does not contain any torsion elements, and
  is saturated.  By Corollary \ref{cor:ample_is_ample}, it suffices to
  check this on the cosets of $\Z D$ for any ample divisor class $D$.
  Taking
  $D = (m(m+1)/2)f + \sum_{1\le k\le m} C_k$,
  we find that for any $r>0$, $[r D]$ contains operators of the form
  $g(z) T^{rm}$.  The claim follows immediately from the fact that such
  operators act invertibly on the space of meromorphic functions.
\end{proof}


When $q$ is torsion, we have already seen that the $\Z^{m+2}$-algebra
$\cS'_{\rho;q;C}$ contains a large ``commutative'' sub-$\Z^{m+2}$-algebra.
We have, in fact, the following.

\begin{prop}
  Let $q\in \Pic^0(C)$ be $r$-torsion, with corresponding isogeny
  $\phi:C\to C/\langle q\rangle=:C'$, and let $Y$ denote the commutative
  rational surface corresponding to $X_{\phi_*\circ \rho;1;C'}$.  Then
  there is a coherent, locally free $\sO_Y$-algebra ${\cal A}$ such that
  $\coh X_{\rho;q;C}$ is naturally equivalent to the category of ${\cal
    A}$-modules, in such a way that $\sO_X\mapsto {\cal A}$.  Moreover,
  ${\cal A}$ is an Azumaya algebra (of degree $r$) on the complement of
  $C'$.
\end{prop}

\begin{proof}
  Fix a divisor $D=af+\sum_{1\le i\le m} C_l$ with $a$ sufficiently large
  that Lemma \ref{lem:surject_torsion} applies.  Such a divisor is
  certainly ample, and thus by Corollary \ref{cor:ample_is_ample}, we have
  \[
  X_{\rho;q;C}\cong \Proj(\cS'_{\rho;q;C}|_{\Z rD})
  \]
  and similarly
  \[
  Y\cong \Proj(\cS'_{\phi_*\circ\rho;1;C'}|_{\Z D}).
  \]
  Note that here, although we are working with $\Z$-algebras, we could
  equally well consider them as graded algebras; since $q^r=1$, the
  translation by $rD$ (resp. $D$) symmetry is an automorphism.  Lemma
  \ref{lem:surject_torsion} implies that the graded algebra corresponding
  to $\cS'_{\rho;q;C}|_{\Z rD}$ is finitely generated over the graded
  algebra corresponding to $\cS'_{\phi_*\circ\rho;1;C'}|_{\Z D}$, and thus
  gives rise to a coherent $\sO_Y$-algebra as required.

  Now, if we blow up a further point of $X_{\rho;q;C}$, thus also blowing
  up a point of $Y$, the resulting $\sO_{\tilde{Y}}$-algebra is isomorphic
  to ${\cal A}$ on the complement of the exceptional locus.  In particular,
  if we remove $C'$ and the divisors $s,e_1,\dots,e_m$, the result will be
  the same as if we instead removed the anticanonical curve from the center
  of the appropriate noncommutative $\P^2$.  It follows from
  \cite[Thm.~7.3]{ArtinM/TateJ/VandenBerghM:1991} that ${\cal A}$ is an Azumaya
  algebra on the complement of $C'\cup s\cup e_1\cup\cdots\cup e_m$.  This
  can be rephrased more canonically: For any birational morphism $\psi:Y\to
  \P^2$, ${\cal A}$ is an Azumaya algebra on $\psi^{-1}(\P^2\setminus C')$.
  In fact, this extends to rational maps: if $\psi:Y\to \P^2$ is a rational
  map which is regular outside $C'$, then ${\cal A}$ is an Azumaya algebra
  on $\psi^{-1}(\P^2\setminus C')$.  Indeed, such a rational map factors
  through an iterated blowup through points of $C'$.  It is straightforward
  to verify that the open sets of the form $\psi^{-1}(\P^2\setminus C')$
  cover $Y\setminus C'$, and thus ${\cal A}$ is indeed an Azumaya algebra
  on $Y\setminus C'$.

  It remains to show that ${\cal A}$ is locally free.  This is immediate on
  $Y\setminus C'$, so it will suffice to show that the natural map ${\cal
    A}(-C')\to {\cal A}$ is injective and its cokernel is locally free as a
  sheaf on $C'$.  (It then follows that $\Tor_d({\cal A},O_p)=0$ for
  any $d>0$ and any point $p\in C'$).  The sheaf ${\cal A}(-C')$
  corresponds to the graded module which in degree $d$ is given by
  $\cS'_{\rho;q;C}(0,drD-rC_m)$.  It follows from the proof of Lemma
  \ref{lem:surject_torsion} that this module is generated in degree $1$,
  and that both it and $\cS'_{\rho;q;C}(0,drD)$ are saturated over the
  homogeneous coordinate ring of $Y$.  Injectivity thus follows immediately from
  injectivity for the graded modules, while the cokernel corresponds to the
  graded quotient module.  This quotient module has a natural filtration
  with subquotients $[T^0]\cS'_{\rho;q;C}(0,drD-lC_m)$ for $0\le l<r$;
  since those subquotients are invertible sheaves on $C$ and $\phi:C\to C'$
  is finite and flat, we conclude that the subquotients are locally free
  sheaves on $C'$, and the claim follows.
\end{proof}

\begin{rem}
  We will generalize this in \cite{noncomm2}, and in particular prove that
  ${\cal A}$ is a maximal order in its generic fiber.
\end{rem}

\medskip

In the proof of Theorem \ref{thm:Xm_is_blowup}, we postponed a technical
point.  Van den Bergh's construction only defines a blowup in a point of a
commutative curve that appears as a divisor in a suitable sense; in
particular, there must be an endofunctor corresponding to twisting by the
curve.

In our case, the curve should clearly be $C$, and the corresponding 2-sided
ideal in $\cS'_{\rho;q;C}$ is the kernel of $[T^0]$, or equivalently the
image of $T$.  Since $T$ is invertible as a difference operator, we may use
it to define an endofunctor of $\cS'_{\rho;q;C}$ as follows.

\begin{defn}
  The ``canonical functor'' of $\cS'_{\rho;q;C}$ is the endofunctor
  $\bar\theta$ which acts on objects as $D\mapsto D+C_m$ and on morphisms as
  $\oD\mapsto T\oD T^{-1}$.
\end{defn}

\begin{rem}
  We will introduce a slightly modified version of the same functor below.
\end{rem}

This induces an endofunctor of $\cS'_{\rho;q;C}-\Mod$ (the categorical
analogue of pulling back a module through a ring homomorphism), and since
that endofunctor preserves the subcategory of torsion modules, it also
induces an endofunctor of $\qcoh X_{\rho;q;C}$, which we also denote by
$\bar\theta$.  Note that $\bar\theta \sO_X(D)\cong \sO_X(D-C_m)$, and there is a
natural transformation $T:\bar\theta\to \id$ which on line bundles is the
operator
\[
T\in \hat\cS'_{\rho;q;C}(-D,-D+C_m) = \Hom(\sO_X(D-C_m),\sO_X(D)).
\]
Note that in the commutative setting, $\bar\theta M\cong M\otimes \omega_X$,
and $T$ corresponds to a Poisson structure $T:\omega_X\to \sO_X$.
It turns out that $T$ is nearly always unique.

\begin{prop}
  If $q\ne 1$, then any natural transformation $\bar\theta\to \id$ is
  proportional to $T$.
\end{prop}

\begin{proof}
  Such a natural transformation $T'$ is determined by its action on line
  bundles, and thus on the corresponding operators
  \[
  T'_D\in \hat\cS'_{\rho;q;C}(D,D+C_m).
  \]
  For these to give a natural transformation, they must satisfy
  \[
  T^{-1} T'_{D_2} \oD 
  =
  \oD T^{-1} T'_{D_1}
  \]
  for any operator $\oD\in \hat\cS'_{\rho;q;C}(D_1,D_2)$.  When
  $D_2-D_1\in f$, so $\oD$ is a multiplication operator, the fact that
  this holds for any $\oD$ of that degree easily implies (since $q\ne
  1$) that the operator $T^{-1} T'_{D_1}$ must itself be a multiplication
  operator.  But then degree considerations imply that $T^{-1}T'_D$ is a
  scalar for all $D$.  Since then $T^{-1}T'_{D_1}=T^{-1}T'_{D_2}$ whenever
  $\hat\cS'_{\rho;q;C}(D_1,D_2)\ne 0$, we see that $T'=\alpha T$ for some
  fixed scalar $\alpha$ as required.
\end{proof}

Now, the category $[T^0]\cS'_{\rho;q;C}$ can also be viewed as a right
$\cS'_{\rho;q;C}$-module, and this gives rise to a pair of adjoint functors
$[T^0]^*$ and $[T^0]_*$ (tensoring and taking $\Hom$s respectively).
Equivalently, we may fix an ample divisor $D$ and restrict both categories
to $\Z D$.  Since $[T^0]\cS'_{\rho;q;C}|_{\Z D}$ is a twisted homogeneous
coordinate ring of $C$, taking sheaves gives a pair of adjoint functors
\[
i_*:\qcoh C\to \qcoh X_{\rho;q;C}
\qquad\text{and}\qquad
i^*:\qcoh X_{\rho;q;C}\to \qcoh C.
\]
(It suffices to verify that $[T^0]_*$ and $[T^0]^*$ take torsion
modules to torsion modules, and this is straightforward.)  In addition, the
leading coefficient filtration establishes $\cS'_{\rho;q;C}|_{\Z D}$ as the
global section module of a locally free sheaf on $[T^0]\cS'_{\rho;q;C}|_{\Z
  D}$, and thus $i_*$ is exact.

\begin{prop}
  The functor $i_*$ embeds $C$ as a divisor in $X_{\rho;q;C}$ in the sense
  of \cite[\S 3.7]{VandenBerghM:1998}.
\end{prop}

\begin{proof}
  First note that if a sheaf $M$ is in the image of $i_*$, then the
  morphism $T_M:\bar\theta M\to M$ is 0; this simply follows from the fact that
  the operator $T$ is in the kernel of $[T^0]$.

  We next observe that $\sO_X(-rD)/T\bar\theta \sO_X(-rD)$ is in the image of
  $i_*$ (as the corresponding quotient of projective modules is certainly
  in the image of $[T^0]_*$).  It follows that if $M$ is a sheaf such that
  $T_M:\bar\theta M\to M$ is 0, then $M$ is a quotient of a sheaf $i_*N$.
  Indeed, $M$ is a quotient of a sum of sheaves $\sO_X(-rD)$, and the
  naturality of $T$ implies that any map $\sO_X(-rD)\to M$ kills $T\bar\theta
  \sO_X(-rD)$, so factors through a map $\sO_X(-rD)/T\bar\theta \sO_X(-rD)$.

  Now, the property $T_M=0$ is preserved under taking sub- and quotient
  sheaves.  Thus the kernel of a surjection $i_*N\to M$ is itself a
  quotient of a sheaf in the image of $i_*$.  Any morphism
  $i_*N'\to i_*N$ is itself in the image of $i_*$, and thus $M$ is the
  image of the appropriate cokernel.

  We thus see that $M$ is in the image of $i_*$ iff $T_M=0$, and this
  establishes that $\bar\theta=o_X(-C)$ in the notation of \cite{VandenBerghM:1998},
  making $C$ a divisor in $X_{\rho;q;C}$.
\end{proof}

\begin{rem}
In particular, we may indeed use \cite{VandenBerghM:1998} to blow up
$X_{\rho;q;C}$ in any point of $C$.  The identification of the relevant
sheaf $\Z$-algebra with a graded algebra uses the functor $\bar\theta$, which
establishes an isomorphism between the $\Z$-algebra and its shift.
\end{rem}

In the sequel, we will simply silently apply $i_*$, identifying $\qcoh C$
with the corresponding subcategory of $\qcoh X_{\rho;q;C}$.  The functor
$i^*$ will be denoted by $M|_C:=i^*M$, where we note that
\[
i_* i^* M \cong M/T\bar\theta M.
\]
With this notation, we have the following as an immediate consequence of
the adjunction.

\begin{prop}
  For any object $N\in D^b\coh C$, $R\Hom(M,N)\cong R\Hom_C(M|^{\dL}_C,N)$.
\end{prop}  

\begin{defn}
  A sheaf $M$ on $X_{\rho;q;C}$ is {\em transverse to} $C$ if $M|^{\dL}_C$
  is a sheaf, {\em disjoint from} $C$ if $M|_C=0$.
\end{defn}

\begin{prop}
  If $M$ is supported on $C$ and $N$ is disjoint from $C$, then
  $R\Hom(M,N)=0$.
\end{prop}

\begin{proof}
  The commutative diagram
\[
  \begin{CD}
    R\Hom(M,\bar\theta N)@>0>> R\Hom(\bar\theta M,\bar\theta N)\\
    @V{T_N}VV               @V{T_N}VV\\
    R\Hom(M,N) @>0>> R\Hom(\bar\theta M,N)
  \end{CD}
\]
  remains commutative if we include the isomorphism $\bar\theta:R\Hom(M,N)\cong
  R\Hom(\bar\theta M,\bar\theta N)$.  Since $T_N$ is an isomorphism, this implies
  $R\Hom(M,N)=0$.
\end{proof}

\medskip

We now return to the saturated category, and in particular the question of
computing the $\Hom$ spaces in that category (or, at the very least, their
dimensions); we will see below that this algorithm easily refines to one
for computing dimensions of $\Ext$ spaces between line bundles on
$X_{\rho;q;C}$.

By analogy with the commutative case \cite{rat_Hitchin}, our strategy will
be to reduce to the case that the degree is in the fundamental chamber.  We
have already mostly dealt with this case: we saw above that when $D$ is
universally nef with $D\cdot C_m>0$, then
\[
\hat\cS'_{\rho;q;C}(0,D)
=
\cS'_{\rho;q;C}(0,D),
\]
and thus flatness of the latter gives
\[
\dim\hat\cS'_{\rho;q;C}(0,D)
=
1+\frac{D\cdot (D+C_m)}{2}
\]
by reduction to the commutative case.

The other possibilities in the fundamental chamber are that we could have
$D\cdot e_m<0$ or $D\cdot C_m\le 0$.

\begin{lem}
  If $D\cdot e_m=-d$ for $d>0$, then
  $\hat\cS'_{\rho;q;C}(0,D)=\hat\cS'_{\rho;q;C}(0,D-de_m)$.
\end{lem}

\begin{proof}
  Since $\hat\cS'_{\rho;q;C}(0,D-de_m)$ can be computed on the blown down
  surface, we have
  \[
  T^d\hat\cS'_{\rho;q;C}(0,D-de_m)
  \subset \hat\cS'_{\rho;q;C}(0,D-de_m+dC_{m-1})
  = \hat\cS'_{\rho;q;C}(0,D+dC_m)
  \]
  and is thus in particular equal to the subspace of
  $\hat\cS'_{\rho;q;C}(0,D+dC_m)$ where the first $d$ coefficients vanish.
  But this is the same as $T^d\hat\cS'_{\rho;q;C}(0,D)$, so the claim
  follows after dividing by $T^d$.
\end{proof}

\begin{lem}
  If $D\cdot C_m<0$, then
  $\hat\cS'_{\rho;q;C}(0,D)=\hat\cS'_{\rho;q;C}(0,D-C_m)$.
\end{lem}

\begin{proof}
  This follows immediately from Lemma \ref{lem:T0_saturated}; the natural
  inclusion
  \[
  \hat\cS'_{\rho;q;C}(0,D-C_m)\to \hat\cS'_{\rho;q;C}(0,D)
  \]
  has quotient contained in
  \[
  \Gamma(C;\sO(\fD_{\rho;q;C}(D)))=0.
  \]
\end{proof}

We are thus left with the case of a universally nef divisor class with
$D\cdot C_m=0$.  For this, we have the following (surprisingly tricky)
result.

\begin{lem}
  Let $D$ be a universally nef divisor class with $D\cdot C_m=0$.  Then
  there is a short exact sequence
  \[
  \begin{CD}
    0@>>> \hat\cS'_{\rho;q;C}(0,D-C_m)@>T>> \hat\cS'_{\rho;q;C}(0,D)
    @>[T^0]>> \Gamma(C;\sO(\fD_{\rho;q;C}(D)))
    @>>> 0.
  \end{CD}
  \]
\end{lem}

\begin{proof}
  We may as well assume $D\cdot e_m>0$, as if not we can reduce to the
  blown down surface.  Once $m<8$, the only universally nef
  divisor with $D\cdot C_m=0$ is $D=0$, where the result is trivial.

  Suppose now that $m>8$.  If we multiply an element of
  $\hat\cS'_{\rho;q;C}(0,D)$ by any element of $\cS'_{\rho;q;C}(-f,0)$, the
  result is an element of
  $\hat\cS'_{\rho;q;C}(-f,D)=\cS'_{\rho;q;C}(-f,D)$.  It follows that
  $\hat\cS'_{\rho;q;C}(0,D)$ is the space of global sections of a suitable
  vector bundle on $\P^1$.  Moreover, there is an induced leading
  coefficient exact sequence for that vector bundle, and the subbundle,
  corresponding to $\hat\cS'_{\rho;q;C}(0,D-C_m)=\cS'_{\rho;q;C}(0,D-C_m)$,
  is acyclic.  (This holds because $(D-C_m)\cdot C_m>D\cdot C_m=0$.)  Since
  the quotient is just the direct image of
  $\Gamma(C;\sO(\fD_{\rho;q;C}(D))$, the claim follows.

  We are thus left with the case $m=8$, when $D=dC_8$ for some $d$.  Now,
  we have an isomorphism
  \[
    \sO(\fD_{\rho;q;C}(dC_8))
    \cong
    \sO(\fD_{\rho;q;C}(C_8))^d,
  \]
  so there are two natural cases to consider.  If
  $\sO(\fD_{\rho;q;C}(C_8))$ is non-torsion, then
  \[
  \Gamma(C;\sO(\fD_{\rho;q;C}(dC_8)))=0
  \]
  for all $d>0$, and surjectivity is automatic.  Otherwise, suppose it is
  $r$-torsion.  In this case, we still trivially have surjectivity whenever
  $d$ is not a multiple of $r$.  When $d$ is a multiple of $r$, we have
  $h^0(\sO(\fD_{\rho;q;C}(dC_8)))=1$, and thus the objective is simply to
  prove that there is an element of $\hat\cS'_{\rho;q;C}(0,dC_8)$ with
  nonzero leading coefficient.

  Clearly, it suffices to prove this in the case $d=r$.  In that case, we
  find in the same way as for $m>8$ that $\hat\cS'_{\rho;q;C}(0,rC_8)$ is
  the space of global sections of a suitable vector bundle on $\P^2$.  The
  corresponding filtration by leading degree has the sequence of subquotients
  \[
  \pi_*(\sO(\fD_{\rho;q;C}(rC_8))),
  \dots
  \pi_*(\sO(\fD_{\rho;q;C}(C_8))),
  \sO_{\P^1}
  \]
  where $\pi:C\to \P^1$ is the appropriate degree 2 morphism.  Now,
  $\sO_{\P^1}$ is certainly acyclic, and $\pi_*(\sO(\fD_{\rho;q;C}(lC_8)))$
  is acyclic whenever $l$ is not a multiple of $r$.  It follows, therefore,
  that the subbundle in the leading coefficient short exact sequence is
  acyclic as required.
\end{proof}

\begin{cor}
  If the point $\eta^3x_0^2/x_1\cdots x_8\in \Pic^0(C)$ is $r$-torsion,
  then there is a difference operator $\oD$ such that for any $d\ge 0$,
  \[
  T^{-rd}\hat\cS'_{\eta,x_0,\dots,x_8;q;C}(0,drC_8)
  =
  \langle 1, \oD,\dots,\oD^d \rangle.
  \]
\end{cor}

\begin{proof}
  This is certainly true for $d=1$ (bearing in mind that $\oD$ will
  include negative shifts as well as positive shifts).  It remains only to
  note that
  \[
  T^{rd}\oD T^{-r(d-1)}
  \in
  T^{r(d-1)} \cS'_{\eta,x_0,\dots,x_8;q;C}(0,rC_8) T^{-r(d-1)}
  =
  \cS'_{\eta,x_0,\dots,x_8;q;C}((d-1)rC_8,drC_8)
  \]
\end{proof}

\begin{rem}
  For $r=1$, the operator $\oD$ is the univariate case of an integrable
  Hamiltonian introduced by van Diejen \cite{vanDiejenJF:1994} (see
  \cite{sklyanin_anal} for a discussion of this fact in terms of the
  Sklyanin algebra).  It will be shown in \cite{elldaha} that this result
  generalizes to a multivariate statement, in which all of the relevant
  commuting operators appear.  Another consequence is that the $W(E_9)$
  symmetry induces a $W(E_8)$ symmetry on the parameters, which appears in
  the eigenvalue analysis of \cite{RuijsenaarsSNM:2015}.  For $r>1$, we
  obtain a new family of eigenvalue equations (now order $2r$), now with a
  symmetry of the form $W(E_8)\ltimes \Lambda_{E_8}/r\Lambda_{E_8}$, though
  it is unclear if this symmetry extends to the analytic setting of
  \cite{RuijsenaarsSNM:2015}.
\end{rem}

At this point, we can deal with any saturated $\Hom$ space with degree in
the fundamental chamber for $W(E_{m+1})$.  Thus, let $\alpha$ be a simple root
of $W(E_{m+1})$; we wish to show that if $D\cdot \alpha<0$, then
$\hat\cS'_{\rho;q;C}(0,D)$ is isomorphic to a space of the form
$\hat\cS'_{\rho';q;C}(0,D')$ with $D'\cdot \alpha>D\cdot \alpha$.  This will clearly
eventually result in a divisor with $D\cdot \alpha>0$; more generally, given any
set of simple roots that generate a finite root system, repeated reductions
of this form will eventually result in a divisor in the fundamental chamber
for that root system.  In particular, if we do this for the subsystem $D_m$
of roots orthogonal to $f$, none of the resulting reductions will change
the degree of the corresponding operators, while the reduction
corresponding to $s-e_1$ will always decrease that degree.  Thus if we
alternate reducing w.r.to $D_m$ and reducing w.r.to $s-e_1$, every other
step will strictly decrease $D\cdot f$.  Thus such a process will either
produce a divisor with $D\cdot f<0$ (for which the saturated $\Hom$ space
is necessarily empty) or one in the fundamental chamber.

The first case we need to deal with is when $D$ is itself a multiple of
$\alpha$.

\begin{lem}
  If $\rho(\alpha)$ is not a power of $q$, then
  $\hat\cS'_{\rho;q;C}(0,r\alpha)=0$ for all $r\in \Z$.  Otherwise, if $l$
  is the smallest positive integer such that $\rho(\alpha)=q^l$, then
  $\dim\hat\cS'_{\rho;q;C}(0,l\alpha)=1$.
\end{lem}

\begin{proof}
  There are three cases to consider: $\alpha=s-e_1$ for $m=1$,
  $\alpha=f-e_1-e_2$ for $m=2$, or $\alpha=e_{m-1}-e_m$ for $m\ge 2$.
  In fact, the first two cases can be reduced to the third: for $s-e_1$, we
  need simply work on $\hat\cS'_{\eta,x_0,x_1,x_2}$ for any $x_2$ for which
  $x_0/x_2$ is not a power of $q$ (and thus $x_1/x_2$ is also not a power
  of $q$), and use the isomorphisms
  \[
  \hat\cS'_{\eta,x_0,x_1,x_2;q;C}(0,r(s-e_1))
  \cong
  \hat\cS'_{\eta,x_0,x_2,x_1;q;C}(0,r(s-e_2))
  \cong
  \hat\cS'_{x_0\eta/x_2,x_2,x_0,x_1;q;C}(0,r(e_1-e_2)).
  \]
  Similarly, for $x_3$ such that $x_3/x_2$ is not a power of $q$, we have
  \begin{align}
  \hat\cS'_{\eta,x_0,x_1,x_2,x_3;q;C}(0,r(f-e_1-e_2))
  &\cong
  \hat\cS'_{\eta,x_0,x_1,x_3,x_2;q;C}(0,r(f-e_1-e_3))\notag\\
  &\cong
  \hat\cS'_{\eta,x_0\eta/x_1x_3,\eta/x_3,\eta/x_1,x_2;q;C}(0,r(e_2-e_3)).
  \end{align}

  Now, we have
  \[
  \hat\cS'_{\rho;q;C}(0,r(e_{m-1}-e_m))
  \subset
  \hat\cS'_{\rho;q;C}(0,r e_{m-1})
  =
  \langle 1\rangle,
  \]
  and thus the only question is when $1\in
  \hat\cS'_{\rho;q;C}(0,r(e_{m-1}-e_m))$.  It follows from the leading
  coefficient exact sequence that
  \[
  \hat\cS'_{\rho;q;C}(0,r(e_{m-1}-e_m))
  \subset
  \Gamma(C;\sO(\fD_{\rho;q;C}(r(e_{m-1}-e_m))))
  \]
  and thus a necessary condition is that
  \[
  \fD_{\rho;q;C}(r(e_{m-1}-e_m))=0,
  \]
  or in other words
  \[
  \prod_{0\le k<r} [q^{-k}x_{m-1}] = \prod_{0\le k<r} [q^{r-k}x_m].
  \]
  It follows in particular that $x_{m-1}=q^{r-k}x_m$ for some $0\le k<r$.
  (In fact, one can show that either $x_{m-1}=q^r x_m$ or $q^r=1$.)

  It remains only to show that when $x_{m-1}=q^l x_m$, with $l>0$ the
  minimum solution, then $\cS'_{\rho;q;C}(l(e_{m-1}-e_m),D)\subset
  \cS'_{\rho;q;C}(0,D)$ for all $D$ in some translate of the universal nef
  cone.  It will be convenient to use the translation symmetry to reduce to
  showing that if $x_{m-1}=x_m$, then for all $l\ge 0$,
  $\cS'_{\rho;q;C}(le_{m-1},D)\subset \cS'_{\rho;q;C}(le_m,D)$.

  If $D\cdot (e_{m-1}-e_m)=0$, then this follows immediately from the Weyl
  group symmetry.  Now, choose such a divisor subject to the additional
  condition that $D-le_{m-1}$ is universally nef with $(D-le_{m-1})\cdot
  C_m\ge 3$.  Then for any universally nef $D'$ with $D'\cdot C_m\ge 3$, we
  have the following:
  \[
  \cS'_{\rho;q;C}(le_{m-1},D+D')
  =
  [D']\cS'_{\rho;q;C}(le_{m-1},D)
  =
  [D']\cS'_{\rho;q;C}(le_m,D)
  \subset
  \cS'_{\rho;q;C}(le_m,D+D').
  \]
  The claim follows.
\end{proof}

\begin{rem}
  Of course, the case $s-e_1$ is just an elementary transformation away
  from the case $s-f$ we needed to consider in the $F_0$ case.  In
  addition, if we unpack the symmetries we used, we find that
  $\hat\cS'_{\rho;q;C}(0,l(s-e_1))\subset \hat\cS'_{\rho;q;C}(0,ls)$ and
  $\hat\cS'_{\rho;q;C}(0,l(f-e_1-e_2))\subset
  \hat\cS'_{\rho;q;C}(0,l(f-e_1))$, allowing us to identify the nonzero
  sections when they exist.
\end{rem}


\begin{lem}
  If $q^r=1$ and $\alpha$ is a simple root with $\rho(\alpha)$ a
  power of $q$, then $\dim\cS'_{\rho;q;C}(0,r\alpha)=1$.
\end{lem}

\begin{proof}
  We can either reduce to the commutative sub-$\Z^{m+2}$-algebra, or
  observe that both factors in
  \[
  \hat\cS'_{\rho;q;C}(0,l\alpha)\hat\cS'_{\rho;q;C}(l\alpha,r\alpha)
  \]
  are nonzero, where $l$ is the minimum positive integer such that
  $\rho(\alpha)=q^l$.
\end{proof}

By analogy with the $F_0$ case, we expect that any space
$\hat\cS'_{\rho;q;C}(0,D)$ will either equal its image under the
reflection, or will factor through one of the morphisms we have just
constructed.  It will be convenient as in the above proof to use the
translation symmetry to arrange that $x_{m-1}=x_m$, so that $\rho$ is
invariant under the corresponding reflection.

\begin{lem}
  Let $x_{m-1}=x_m$, and let $D$ be any divisor with $D=X_{m-2}(D)$.
  Then for $0\le l\le |\langle q\rangle|$, we have
  \[
  \hat\cS'_{\rho;q;C}(0,D-le_m)
  =
  \hat\cS'_{\rho;q;C}(0,D-le_{m-1}).
  \]
\end{lem}

\begin{proof}
  Using Lemma \ref{lem:saturate_nef_test2}, we may as well assume that $D$
  is very far inside the universal nef cone of $X_{m-2}$, so that both
  spaces are subspaces of $\hat\cS'_{\rho;q;C}(0,D)=\cS'_{\rho;q;C}(0,D)$.
  Also, since $\hat\cS'_{\rho;q;C}(D-le_{m-1},D-le_m)=\langle 1\rangle$,
  what we need to establish is that every $\oD\in
  \hat\cS'_{\rho;q;C}(0,D-le_m)$ factors through this morphism.

  Now, if the first $l$ coefficients of $\oD$ all vanish, then we find
  that
  \[
  T^{-l}\oD\in \hat\cS'_{\rho;q;C}(0,(D-lC_{m-2})+le_{m-1})
  = \hat\cS'_{\rho;q;C}(0,D-lC_{m-2})
  = \hat\cS'_{\rho;q;C}(0,D-lC_{m-2}+le_m),
  \]
  and thus the result follows.  Thus, in order to prove the desired
  result, all we need to show is that for each $0\le k<l$, the inclusion map
  \[
  \hat\cS'_{\rho;q;C}(0,D-le_{m-1}-kC_{m-1})
  =
  \hat\cS'_{\rho;q;C}(0,D-le_{m-1}-kC_m)
  \to
  \hat\cS'_{\rho;q;C}(0,D-le_m-kC_m)
  \]
  is surjective on leading coefficients.  Since both spaces are contained
  in $\cS'_{\rho;q;C}(0,D-kC_{m-2})$, it will suffice to show that
  \[
  \fD_{\rho;q;C}(D-le_{m-1}-kC_{m-1})^{-1}
  =
  \gcd(\fD_{\rho;q;C}(D-le_m-kC_m)^{-1},\fD_{\rho;q;C}(D-kC_{m-2})^{-1}).
  \]
  Equivalently, we may multiply all three divisors by
  $\fD_{\rho;q;C}(D-kC_{m-2})$, so that we are comparing
  \[
  \fD_{\rho;q;C}((k-l)e_{m-1})^{-1}
  \quad\text{and}\quad
  \gcd(\fD_{\rho;q;C}(ke_{m-1}+(k-l)e_m)^{-1},1).
  \]
  Further multiplying by $\fD_{\rho;q;C}((k-l)e_{m-1})$ reduces to showing
  \[
  \gcd(\prod_{-k\le j<0} [q^{-j}x_m],
  \prod_{0\le j<l-k} [q^{-j}x_m])
  =
  1.
  \]
  But this is clearly true iff $l\le |\langle q\rangle|$.
\end{proof}

\begin{lem}
  Suppose $x_{m-1}=x_m$, and let $D_1$, $D_2$ be divisors such that
  $|\langle q\rangle|\ge D_1\cdot (e_{m-1}-e_m)\ge 0\ge D_2\cdot
  (e_{m-1}-e_m)$.  Then
  \[
  \hat\cS'_{\rho;q;C}(D_1,D_2)
  =
  \hat\cS'_{\rho;q;C}(s_{m-1}D_1,D_2)
  \]
\end{lem}

\begin{proof}
  Note that
  $\hat\cS'_{\rho;q;C}(D_1,s_{m-1}D_1)=\langle 1\rangle$, we need only show
  that
  \[
    \hat\cS'_{\rho;q;C}(D_1,D_2) \subset \hat\cS'_{\rho;q;C}(s_{m-1}D_1,D_2);
  \]
  i.e., that every morphism from $D_1$ to $D_2$ factors through $s_{m-1}D_1$.

  If $D'$ is universally nef with $D'\cdot C_m\ge 2$ and
  \[
  \hat\cS'_{\rho;q;C}(D_1,D_2+D')
  =
  \hat\cS'_{\rho;q;C}(s_{m-1}D_1,D_2+D'),
  \]
  then we have
  \[
    [D']\hat\cS'_{\rho;q;C}(D_1,D_2)
    \subset \hat\cS'_{\rho;q;C}(D_1,D_2+D')
    =
    \hat\cS'_{\rho;q;C}(s_{m-1}D_1,D_2+D'),
  \]
  and thus Lemma \ref{lem:saturate_nef_test2} gives the desired result.  In
  particular, we may freely add an arbitrary universally nef divisor with
  $D'\cdot (e_{m-1}-e_m)=-D_2\cdot (e_{m-1}-e_m)$ and $D'\cdot C_m\ge 2$ to
  $D_2$ before proving our claim.  In particular, it suffices to prove the
  claim in the case $D_2\cdot (e_{m-1}-e_m)=0$, $(D_2-D_1)\cdot e_{m-1}\ge
  0$.  In addition, if $(D_2-D_1)\cdot e_{m-1}>0$, then we see that (adding
  a symmetric universally nef divisor class as necessary) both saturated
  $\Hom$ spaces surject onto their leading coefficient bundles, and thus we
  may reduce to the case that the leading coefficient vanishes.  In other
  words, it suffices to prove the result in the case $D_2\cdot
  (e_{m-1}-e_m)=0$, $(D_2-D_1)\cdot e_{m-1}=0$.

  Using the translation freedom, we thus reduce to showing
  \[
  \hat\cS_{\rho;q;C}(le_m,D)
  \subset
  \hat\cS_{\rho;q;C}(le_{m-1},D)
  =
  \cS_{\rho;q;C}(le_{m-1},D)
  =
  \cS_{\rho;q;C}(le_m,D),
  \]
  where $D=X_{m-2}(D)$ is sufficiently far in the interior of the universal
  nef cone for $m-2$.  But this is just the adjoint of the previous lemma.
\end{proof}

The adjoint is also useful.

\begin{cor}\label{cor:factor_Aroot}
  Suppose $x_{m-1}=x_m$, and let $D_1$, $D_2$ be divisors such that
  $D_1\cdot (e_{m-1}-e_m)\ge 0\ge D_2\cdot (e_{m-1}-e_m)\ge -|\langle
  q\rangle|$.  Then
  \[
  \hat\cS'_{\rho;q;C}(D_1,D_2)
  =
  \hat\cS'_{\rho;q;C}(D_1,s_{m-1}D_2)
  \]
\end{cor}

\begin{lem}\label{lem:edgy_si_xform}
  Suppose $x_{m-1}=x_m$, and let $D_1$, $D_2$ be divisors such that
  $-|\langle q\rangle|\le D_2\cdot (e_{m-1}-e_m),D_1\cdot(e_{m-1}-e_m)\le
  0$. Then
  \[
  \hat\cS'_{\rho;q;C}(D_1,D_2)
  =
  \hat\cS'_{\rho;q;C}(s_{m-1}D_1,s_{m-1}D_2).
  \]
\end{lem}

\begin{proof}
  The previous two results give
  \[
  \hat\cS'_{\rho;q;C}(s_{m-1}D_1,D_2)
  =
  \hat\cS'_{\rho;q;C}(D_1,D_2)
  \]
  and
  \[
  \hat\cS'_{\rho;q;C}(s_{m-1}D_1,D_2)
  =
  \hat\cS'_{\rho;q;C}(s_{m-1}D_1,s_{m-1}D_2)
  \]
  respectively.
\end{proof}

Similar results (using the above reduction) apply to the other simple
roots, and we find the following (again assuming for simplicity that
$\rho(\alpha)=1$) for $\hat\cS'_{\rho;q;C}(D_1,D_2)$ when $(D_2-D_1)\cdot
\alpha<0$:
\begin{itemize}
  \item[(1)] If $D_2\cdot \alpha$ and $D_1\cdot \alpha$ have opposite
    signs, or if there is a multiple of $|\langle q\rangle|$ between them,
    then
    \[
    \hat\cS'_{\rho;q;C}(D_1,D_2)
    =
    \hat\cS'_{\rho;q;C}(D_1,D_2-l\alpha),
    \]
    where $l=(-D_2\cdot\alpha)\bmod r$.
  \item[(2)] Otherwise,
    \[
    \hat\cS'_{\rho;q;C}(D_1,D_2)
    =
    \hat\cS'_{\rho;q;C}(s_\alpha D_1,s_\alpha D_2).
    \]
\end{itemize}
Here $(2)$ is Lemma \ref{lem:edgy_si_xform}, while $(1)$ is Corollary
$\ref{cor:factor_Aroot}$, translating by $|\langle q\rangle|e_m$ as necessary
to ensure that $-|\langle q\rangle|\le D_2\cdot \alpha<0$.
Since $(1)$ always increases $(D_2-D_1)\cdot\alpha$ and $(2)$ automatically
terminates, iterating this process will always end with a pair $D_1$, $D_2$
such that $(D_2-D_1)\cdot\alpha\ge 0$.

This completes the desired algorithm for computing
$\dim\hat\cS'_{\rho;q;C}(D_1,D_2)$.

\section{Cohomology and duality}

Clearly, our next objective, now that we have well-defined quasi-schemes
associated to the difference operator construction, should be to show that
the quasi-scheme behaves like a smooth surface.  More precisely, we would
like to show that $X_{\rho;q;C}$ is a ``noncommutative smooth proper
surface'' in the sense of \cite[Defn.~3.2]{ChanD/NymanA:2013}.  The bulk of
these axioms relate to the behavior of the $\Ext$ functor in $\coh
X_{\rho;q;C}$, which will thus occupy our attention for the moment.  As a
special case, we define $\Gamma(M):=\Hom(\sO_X,M)$, and
$H^i(M):=\Ext^i(\sO_X,M)$.  (We should really denote this by $\Ext^i_X$,
but will suppress the surface from the notation, and only note when the
$\Ext$ functor is being computed in some other (derived) category,
typically $D^b \coh C$)

We first need to consider a few functors between our categories.  We have
already discussed the functor $\bar\theta$ with its associated natural
transformation $T:\bar\theta\to\id$, as well as the functors $\_(D)$
corresponding to twisting by a line bundle.  We will also need a family of
functors relating to the blowup structure.  For convenience, let $X_l$
denote the $l$-th noncommutative blowup in our sequence; that is,
\[
X_l:=X_{\eta,x_0,\dots,x_l;q;C}.
\]
We include $X_0=X_{\eta,x_0;q;C}$ as well as $X_{-1}$, the noncommutative
$\P^2$.  (We could also consider the analogue of $F_0/F_2$.)

Then Van den Bergh \cite{VandenBerghM:1998} gives us a pair of adjoint
functors for each $0\le l\le m$:
\begin{align}
\alpha_{l*}&:\qcoh X_l\to \qcoh X_{l-1}\notag\\
\alpha^*_l&:\qcoh X_{l-1}\to \qcoh X_l
\end{align}
which preserve noetherian objects and structure sheaves.  (And the expected
exactness holds: $\alpha_{l*}$ is left-exact, and $\alpha^*_l$ is
right-exact.)  The first is quite straightforward in our terms, as it is
simply the restriction (after saturation) to the appropriate
$\Z^{l+1}$-algebra; the second is then the natural induction functor.  We
should note, however, that each functor comes in a natural 1-parameter
family, for the simple reason that we could have restricted to a different
coset of $\Z^{l+1}$.  It is notationally convenient that conjugation by
$\bar\theta$ moves between these alternate functors, and thus we need not
introduce a special notation.  The one exception is that we define (for
reasons which will become apparent)
\[
\alpha^!_l = \bar\theta \alpha^*_l \bar\theta^{-1},
\]
bearing in mind that the two copies of $\bar\theta$ are acting on different
categories, so are not really the same functor.  Note that $\alpha^!_l$ is
right-exact (inherited from $\alpha^*_l$); in addition, since $\alpha^*_l$
has a left-derived functor, so does $\alpha^!_l$.  (The left-derived
functor exists since line bundles are flat for $\alpha^*_l$ and
$\alpha^!_l$.)  In fact, Van den Bergh shows that $\alpha^*_l$ (and thus
$\alpha^!_l$) has homological dimension $1$, and similarly for
$\alpha_{l*}$.

Another key point is that Van den Bergh shows that $\bar\theta^{-1}$ is
relatively ample for $\alpha_{l*}$; that is, for any coherent sheaf $M$,
there is a bound $K_l$ such that for $k_l\ge K_l$, $\bar\theta^{-k_l} M$ is
acyclic for $\alpha_{l*}$, and $\alpha^*_l \alpha_{l*}
\bar\theta^{-k_l}M\to M$ is surjective.  This inductively gives a sequence
$k_l,\dots,k_1$ such that
\[
R\alpha_{1*} \bar\theta^{-k_1}
R\alpha_{2*} \bar\theta^{-k_2}
\cdots
R\alpha_{l*} \bar\theta^{-k_l}
M
\]
is a sheaf on $X_0$, and thus (since ruled surfaces satisfy the Chan-Nyman
axioms, \cite{ChanD/NymanA:2013}) there is a line bundle $D$ on $X_0$ such that
\[
R\Hom(\sO_X(-D),
R\alpha_{1*} \bar\theta^{-k_1}
R\alpha_{2*} \bar\theta^{-k_2}
\cdots
R\alpha_{l*} \bar\theta^{-k_l}
M
)
\]
is supported in cohomological degree 0.  Since the functor being applied to
$M$ is a composition of functors with left adjoints, we may instead apply
the adjoints to the line bundle, and thus find that there exists $D$ on
$X_l$ such that $R\Hom(\sO_X(-D),M)$ is supported in degree $0$.  (Moreover,
such $D$ exists in any translate of the universal nef cone.)

Define a sheaf $\sO_{e_l}(-1)\cong \sO_X(e_l)/\sO_X$.

\begin{lem}
  For any divisor $D$ such that $X_{l-1}(D)=D$, there is a short exact sequence
  \[
  0\to \sO_X(D)\to \sO_X(D+e_l)\to \sO_{e_l}(-1)\to 0.
  \]
\end{lem}

\begin{proof}
  Up to scalars, there is a unique nonzero morphism $\sO_X(D)\to
  \sO_X(D+e_l)$, and that morphism induces an isomorphism
  \[
  \Hom(\sO_X(-D'),\sO_X(D))\to \Hom(\sO_X(-D'),\sO_X(D+e_l))
  \]
  for all $D'=X_{l-1}(D')$.  It follows that if $M$ is the cokernel, then
  $R\Hom(\sO_X(-D'),M)=0$ for all $D'$, and thus $R\alpha_{l*}M=0$.  By
  \cite[Thm.~8.4.1]{VandenBerghM:1998}, this implies that $M$ is
  isomorphic to a power of $\sO_{e_l}(-1)$.  Since
  $\dim\Hom(\sO_X(-D'+e_l),M)=1$ for sufficiently ample $D'=X_{l-1}(D')$,
  that power must be $1$ as required.
\end{proof}

\begin{prop}
  On coherent sheaves, the derived functor $L\alpha^!_l$ is right
  adjoint to the derived functor $R\alpha_{l*}$.
\end{prop}

\begin{proof}
  We first claim that there is a natural isomorphism
  \[
  R\alpha_{l*} L\alpha^!_l N\cong N
  \]
  for any coherent sheaf $N$.  This immediately reduces to the case that
  $N\cong \sO_X(D)$, and thus to showing that
  \[
  R\alpha_{l*} \sO_X(D+e_l)\cong \sO_X(D).
  \]
  For any $D'$, we have
  \begin{align}
  R\Hom(\sO_X(-D'),R\alpha_{l*} \sO_X(D+e_l))
  &\cong
  R\Hom(L\alpha^*_l\sO_X(-D'),\sO_X(D+e_l))\notag\\
  &\cong
  R\Hom(\sO_X(-D'),\sO_X(D+e_l)),
  \end{align}
  and thus the claim follows from the fact that for sufficiently ample
  $D'$, $\cS'_{\rho;q;C}(-D-e_l,D')$ is saturated and equal to
  $\cS'_{\rho;q;C}(-D,D')$.

  Now, to prove the Proposition, we need to show that for any coherent
  sheaves $M$, $N$, there is a natural isomorphism
  \[
  R\Hom(M,L\alpha^!_l N)\cong R\Hom(R\alpha_{l*}M,N).
  \]
  By \cite[Thm.~8.4.1]{VandenBerghM:1998}, there is a natural distinguished
  triangle
  \[
  E\to L\alpha^*_lR\alpha_{l*}M\to M\to
  \]
  where $E$ is quasi-isomorphic to a (bounded) complex with entries of the form
  $\sO_{e_l}(-1)^n$ and all morphisms $0$.  Now,
  \[
  R\Hom(L\alpha^*_lR\alpha_{l*}M,L\alpha^!_l N)
  \cong
  R\Hom(R\alpha_{l*}M,R\alpha_{l*} L\alpha^!_l N)
  \cong
  R\Hom(R\alpha_{l*}M,N),
  \]
  and thus it remains only to show that
  \[
  R\Hom(\sO_{e_l}(-1),L\alpha^!_l N) = 0.
  \]
  The presentation $0\to \sO_X\to \sO_X(e_l)\to O_{e_l}(-1)\to 0$
  allows us to rephrase this as an isomorphism
  \[
  R\Hom(\sO_X,L\alpha^!_l N)
  \cong
  R\Hom(\sO_X(e_l),L\alpha^!_l N)
  \]
  or equivalently
  \[
  R\Hom(L\alpha^*\sO_X,L\alpha^!_l N)
  \cong
  R\Hom(\bar\theta L\alpha_l^*\bar\theta^{-1}\sO_X,L\alpha^!_l N)
  \]
  But this becomes
  \[
  R\Hom(\sO_X,R\alpha_*L\alpha^!_l N)
  \cong
  R\Hom(\bar\theta^{-1}\sO_X,R\alpha_*L\alpha^*_l \bar\theta^{-1} N)
  \]
  and thus reduces to the isomorphism
  $R\alpha_*L\alpha^!_lN\cong N$ already established.
\end{proof}

\begin{rem}
  In contrast to the adjunction between $\alpha^*$ and $\alpha_*$, this
  adjunction only holds at the level of derived categories; $\alpha_*$ is
  left exact, so cannot have a right adjoint as a functor of abelian
  categories.
\end{rem}

\begin{cor}\label{cor:disting!}
  There is a natural distinguished triangle
  \[
  M\to L\alpha^!_lR\alpha_{l*}M\to R\Hom(\sO_{e_l}(-1),M)\otimes^{\dL}_k \sO_{e_l}(-1)[1]\to.
  \]
\end{cor}

\begin{proof}
  Applying $R\alpha_*$ to the natural morphism $M\to
  L\alpha^!R\alpha_{l*}M$ gives the identity on $R\alpha_{l*}M$, and thus
  the cone is in the image of $\otimes^{\dL}_k \sO_{e_l}(-1)$.  Applying
  $R\Hom(\sO_{e_l}(-1),\_)$ then tells us which complex of vector spaces to
  take.
\end{proof}

This gives us the following weak form of Serre duality.

\begin{prop}
  For any coherent sheaf $M$, and any divisor class $D$, there is a natural
  isomorphism
  \[
  R\Hom(M,\bar\theta \sO_X(D)[2])
  \cong
  R\Hom(\sO_X(D),M)^*,
  \]
  functorially in $M$.
\end{prop}

\begin{proof}
  By twisting $X$ as necessary, we may assume $D=0$.  We also observe that
  this is known \cite{NymanA:2005,ChanD/NymanA:2013} for $X_0$.  We may
  thus assume by induction that it is known on the blown down surface, and
  proceed as follows:
  \begin{align}
  R\Hom(M,\bar\theta \sO_X[2])
  &\cong
  R\Hom(M,\alpha^!_m \bar\theta \sO_X[2])\notag\\
  &\cong
  R\Hom(R\alpha_{m*} M,\bar\theta \sO_X[2])\notag\\
  &\cong
  R\Hom(\sO_X,R\alpha_{m*} M)^*\notag\\
  &\cong
  R\Hom(\sO_X,M)^*.
  \end{align}
\end{proof}

\begin{cor}
  For any divisor class $D$, if $D\cdot f>-2$, then
  $\Ext^2(\sO_X,\sO_X(D))=0$.
\end{cor}

\begin{proof}
  This follows immediately from duality and the fact that
  $\cS'_{\rho;q;C}(0,D)=0$ if $D\cdot f<0$.
\end{proof}

\begin{cor}
  For any coherent sheaves $M$, $N$, $\Ext^p(M,N)$ is finite-dimensional
  for $p\ge 0$, and $0$ for $p>2$.
\end{cor}

\begin{proof}
  We first show that there is a uniform bound on the $\Ext$ degree for any
  pair of coherent sheaves.  We use the blowup structure and proceed by
  induction (the base case being known by \cite{ChanD/NymanA:2013}).  Applying
  $R\Hom(M,\_)$ to the distinguished triangle of Corollary
  \ref{cor:disting!} gives a distinguished triangle
  \[
  R\Hom(M,\sO_{e_m}(-1))\otimes^\dL_k R\Hom(\sO_{e_m}(-1),N)
  \to R\Hom(M,N)
  \to R\Hom(R\alpha_*M,R\alpha_*N)
  \to
  \]
  Now, $R\alpha_*M$ and $R\alpha_*N$ are bounded coherent complexes, and
  thus that term of the distinguished triangle is a uniformly bounded
  coherent complex by induction.  The distinguished triangle of Corollary
  $\ref{cor:disting!}$, as well as the analogue for $L\alpha^*$, tell us
  that $R\Hom(\sO_{e_m}(-1),N)$ and $R\Hom(M,\sO_{e_m}(-1))$ are uniformly
  bounded and coherent, respectively, and thus so is their tensor product.
  It thus follows that $R\Hom(M,N)$ is a uniformly bounded coherent
  complex.

  It remains only to show that $\Ext^p(M,N)=0$ for $p>2$.  Let $q$ be the
  minimum bound such that $\Ext^p(M,N)=0$ for $p>q$ and all coherent $M$,
  $N$, and consider a short exact sequence
  \[
  0\to N'\to \sO_X(-D)^n\to N\to 0
  \]
  Weak Serre duality tells us that $\Ext^p(M,\sO_X(-D))\cong
  \Ext^{2-p}(\bar\theta^{-1}\sO_X(-D),M)^*$, and thus for $p>2$,
  $\Ext^p(M,\sO_X(-D))=0$.  It thus follows that $\Ext^p(M,N)=0$ for
  $p>\max(q-1,2)$, so that $q\le \max(q-1,2)$, i.e., $q\le 2$.
\end{proof}

\begin{prop}\label{prop:perfect_complexes_exist}
  Any complex $M$ of quasicoherent sheaves with bounded coherent cohomology
  is quasi-isomorphic to a bounded complex in which each term is a sum of
  line bundles; moreover, those line bundles may be restricted to a finite
  set that works for all $M$.
\end{prop}

\begin{proof}
  This is known for $X_0$; indeed, we may take the line bundles to be
  $\sO_X$, $\sO_X(-f)$, $\sO_X(-s-f)$, $\sO_X(-s-2f)$.  Thus as usual we
  reduce to showing that this is preserved under blowing up.  We use the
  distinguished triangle
  \[
  L\alpha^*_mR\alpha_{m*} M\to M\to R\Hom_k(R\Hom(M,\sO_{e_m}(-1)),\sO_{e_m}(-1))\to
  \]
  Since $R\alpha_{m*} M$ has bounded coherent cohomology, it is
  quasi-isomorphic to a bounded complex of sums of line bundles, and this is
  preserved under $L\alpha^*_m$.  Since the set of line bundles used can be
  fixed independently of $M$, we may in particular choose a divisor class
  $D_m=X_{m-1}(D_m)$ such that for any of those line bundles,
  \[
  R\Hom(\sO_X(-D_m),\sO_X(D))\quad\text{and}\quad
  R\Hom(\sO_X(-D_m+e_m),\sO_X(D))
  \]
  are supported in degree 0.  We may express
  $\sO_{e_m}(-1)=\sO_X(-D_m+e_m)/\sO_X(-D_m)$, and thus find that
  \[
  R\Hom_k(R\Hom(M,\sO_{e_m}(-1)),\sO_{e_m}(-1))
  \]
  is quasi-isomorphic to a complex in which each term is a sum of
  $\sO_X(-D_m)$ and $\sO_X(-D_m+e_m)$.  Since $D_m$ was sufficiently ample,
  the map
  \[
  R\Hom_k(R\Hom(M,\sO_{e_m}(-1)),\sO_{e_m}(-1))
  \to
  L\alpha^*_mR\alpha_{m*}M[1]
  \]
  corresponds to an actual map of complexes, and $M$ is represented by the
  cone.
\end{proof}

\begin{rem}
  Note that the argument works equally well when $M$ and $N$ are flat
  families over some base (including if the parameters $\rho;q;C$
  themselves vary over the base).
\end{rem}

This allows us to define an Euler characteristic
\[
\chi(M,N):=\dim\Hom(M,N)-\dim\Ext^1(M,N)+\dim\Ext^2(M,N),
\]
which is constant as $M$ and $N$ vary over flat families.

\begin{prop}
  For any coherent sheaf $M$, there are integers $\rank(M)$, $\chi(M)$ and
  a divisor class $c_1(M)$ such that for any divisor class $D$,
  \[
  \chi(\sO_X(-D),M) = \rank(M)\frac{D\cdot (D+C_m)}{2}
  +c_1(M)\cdot D
  +\chi(M).
  \]
\end{prop}

\begin{proof}
  Since the function $D\mapsto \chi(O_X(-D),M)$ is additive in short exact
  sequences, we see that for any exact sequence
  \[
  0\to M_1\to M_2\to M_3\to 0,
  \]
  if the Proposition holds for two of the terms, then it holds for the
  third, with
  \begin{align}
    \rank(M_2)&=\rank(M_1)+\rank(M_3),\notag\\
    c_1(M_2)&=c_1(M_1)+c_1(M_3),\notag\\
    \chi(M_2)&=\chi(M_1)+\chi(M_3).
  \end{align}
  Since any coherent sheaf $M$ is represented by a bounded complex of line
  bundles, it thus suffices to prove the Proposition when $M$ is itself a
  line bundle.  Since every line bundle is a fiber of a flat family over
  the entirety of parameter space, we find that $\chi(\sO_X(-D),\sO_X(D'))$
  is constant over parameter space.  When $q=1$, we find that
  \[
  \chi(\sO_X(-D),\sO_X(D'))
  =
  \chi(\sO_X(D+D'))
  =
  1+\frac{(D+D')\cdot (D+D'+C_m)}{2},
  \]
  and thus $\sO_X(D')$ satisfies the given claim with
  \begin{align}
  \rank(\sO_X(D')) &= 1\notag\\
  c_1(\sO_X(D')) &= D'\\
  \chi(\sO_X(D')) &= 1+\frac{D' \cdot (D'+C_m)}{2}.\notag
  \end{align}
\end{proof}

Of course, the notation is consistent with the usual notation in the
commutative case; we will thus refer to $\rank(M)$, $c_1(M)$ and $\chi(M)$
as the ``rank'', ``first Chern class'', and ``Euler characteristic'' of $M$
even in the noncommutative case.  We will also refer to the tuple
$(\rank(M),c_1(M),\chi(M))$ as the ``numerical invariants'' of $M$.  Also,
since they are defined using $\chi(M,N)$, these quantities themselves must
be constant in flat families.

\begin{cor}
  The product map $(\rank,c_1,\chi):K_0(\coh X_m)\to \Z^{m+4}$ is an
  isomorphism.
\end{cor}

\begin{proof}
  The usual induction tells us that $K_0(\coh X_m)$ is generated by the
  classes
  \[
  [\sO_X],[\sO_X(-f)],[\sO_X(-s-f)],[\sO_X(-s-2f)],[\sO_X(e_1)],\dots,[\sO_X(e_m)].
  \]
  It is straightforward to see that their images in $\Z^{m+4}$ form a
  basis, and thus they are linearly independent classes in $K_0(\coh X_m)$,
  giving the desired isomorphism.
\end{proof}

\begin{cor}
  For any coherent sheaves $M$, $N$,
\[
 \chi(M,N)=-\rank(M)\rank(N)+\rank(M)\chi(N)+\chi(M)\rank(N)
 -c_1(M)\cdot (c_1(N) + \rank(N) C_m)\notag
 \]
\end{cor}

\begin{proof}
  The function $\chi(M,N)$ is a bilinear form on the Grothendieck group;
  since the latter has a basis of line bundles, we need simply verify this
  when both $M$ and $N$ are line bundles.
\end{proof}

Comparing the two sides of weak Serre duality gives the following.

\begin{prop}
  For any coherent sheaf $M$,
  \begin{align}
    \rank(\bar\theta M)&=\rank(M)\\
    c_1(\bar\theta M)&=c_1(M)-\rank(M)C_m\\
    \chi(\bar\theta M)&=\chi(M)-c_1(M)\cdot C_m.
  \end{align}
\end{prop}

\begin{rem} And, of course, $\chi(M,N)=\chi(N,\bar\theta M)$ in general, since
  it holds for line bundles.
\end{rem}

It is also straightforward to see how these numerical invariants are
affected by the twisting and adjoint operations.  Let $M(D)$ be the image
of the twist functor taking $\sO_X(D')\to \sO_X(D+D')$ (and taking
$X_{\rho;q;C}$ to $X_{q^{_\cdot D}\rho;q;C}$.)

\begin{prop}
  For any coherent sheaf $M$,
  \begin{align}
    \rank(M(D)) &= \rank(M)\notag\\
    c_1(M(D)) &= c_1(M)+\rank(M)D\notag\\
    \chi(M(D)) &= \chi(M) + c_1(M)\cdot D+\rank(M) \frac{D\cdot
      (D+C_m)}{2},\notag
  \end{align}
\end{prop}

Since $\sO_X(D)|_C\cong \sO(\fD_{\rho;q;C}(D))$, we may also easily compute
how $|^{\dL}_C$ acts on Grothendieck groups.

\begin{prop}
  Let $M$ be a coherent sheaf on $X_{\rho;q;C}$, with $M|^{\dL}_C$ the
  corresponding complex in $\coh C$.  Then
  \begin{align}
    \rank(M|^{\dL}_C) &= \rank(M)\\
    \deg(M|^{\dL}_C) &= c_1(M)\cdot C_m\\
    c_1(M|^{\dL}_C) &= q^{\chi(M)-\rank(M)} \rho(c_1(M)).
  \end{align}
\end{prop}

Note, in particular, that we can recover $\rho$ and $q$ from the
restriction map $K_0(\coh X)\to K_0(\coh C)$.

\begin{cor}
  If $M$ is disjoint from $C$, then $M$ is a rank 0 sheaf transverse to
  $C$, and $T_M$ is an isomorphism.
\end{cor}

\begin{proof}
  If $M|_C=0$, then $\rank(M)=\rank(M|^{\dL}_C)\le 0$, and thus $M$ must be
  have rank 0.  In particular, $\ker(T_M)$ is
  $0$-dimensional, and if it is nonzero, then $c_1(M)\cdot C_m<0$.
  This implies that $\chi(M,{\cal L})>0$ for all line bundles ${\cal L}$ on
  $C$ and thus that there exists a morphism from $M$ to some such line
  bundle, contradicting disjointness.  Thus $M$ is transverse to $C$, and
  since both the kernel and cokernel of $T_M$ vanish, $T_M$ is an isomorphism.
\end{proof}

\begin{prop}
  If $D$ is nef with $D\cdot C_m>0$, then $\Ext^i(\sO_X,\sO_X(D))=0$ for
  $i>0$.
\end{prop}

\begin{proof}
  For $i\ge 2$, this follows from weak Serre duality, at which point it
  follows for $i=1$ from flatness:
  \[
  \dim\Hom(\sO_X,\sO_X(D)) = \chi(\sO_X,\sO_X(D)),
  \]
  and thus $\dim\Ext^1(\sO_X,\sO_X(D))=0$.
\end{proof}

With this, it is relatively straightforward to obtain semicontinuity.

\begin{lem}
  Let $X/S$ be a family of quasi-schemes of the form $X_{\rho;q;C}$ (with
  fixed $m$) over the locally noetherian scheme $S$, and let $M\in D^b\coh
  X$.  Then for all $i\in \Z$, $s\mapsto \dim\Ext^i(\sO_X,M|^{\dL}_s)$ is
  an upper semicontinuous function of $s$.
\end{lem}

\begin{proof}
  There is a fixed finite collection of divisor classes $D_i$ such that for any
  divisor class $D$, we may represent $M$ as a bounded complex in which
  each term is a sum
  \[
  \bigoplus_i \sO_X(D+D_i)\boxtimes_S F_i
  \]
  with $F_i$ a sheaf on $S$.  Indeed, this is just the relative version of
  Proposition \ref{prop:perfect_complexes_exist}, applied to the twisted
  complex $M(-D)$.  If $D+D_i$ is ample for every $D$, this gives a bounded
  complex of sheaves on $S$ representing $R\Hom(\sO_X,M|^{\dL}_s)$, and we
  can then apply the usual argument to a locally free resolution of this
  complex.
\end{proof}

\begin{rem}
  We can also apply this to $\Ext^i(M,N)$ for $M$, $N\in D^b\coh X$, as
  long as $M$ has a representation as a bounded coherent complex in which
  each term is a sum of tensor products of line bundles on $X$ with {\em
    locally free} sheaves on $S$.  Over an integral (or locally integral)
  base, this can always be done when $M\in \coh X$; the coefficient sheaves
  arising in the construction can themselves be computed as $\Ext$ sheaves
  to which semicontinuity applies, and if we twist by a sufficiently ample
  bundle, we can arrange for there to be only one nonvanishing $\Ext$ group
  in each case.  At that point, Euler characteristic considerations tell us
  that the dimension is constant, and the analogue of Grauert's theorem
  tells us that the sheaves are locally free.
\end{rem}

\medskip

One use of the above linear functionals on the Grothendieck group is to
define a ``size'' for coherent sheaves.

\begin{defn}
  A nonzero coherent sheaf $M$ is {\em $2$-dimensional} if $\rank(M)\ne 0$,
  {\em $1$-dimensional} if $\rank(M)=0$ and $c_1(M)\ne 0$, and {\em
    $0$-dimensional} if $\rank(M)=0$ and $c_1(M)=0$.  If no subsheaf of $M$
  has smaller dimension, we say that $M$ is a {\em pure} $d$-dimensional
  sheaf.
\end{defn}

\begin{prop}
  A $2$-dimensional coherent sheaf has positive rank; a $0$-dimensional
  coherent sheaf has positive Euler characteristic.  Finally, if $M$ is
  $1$-dimensional, then for any ample divisor class $D$, $D\cdot c_1(M)>
  0$.
\end{prop}

\begin{proof}
  For any ample divisor $D$, we must have $\Hom(\sO_X(-rD),M)>0$ for $r\gg
  0$.  This implies $\rank(M)>0$ or $\chi(M)>0$ in the $2$- or
  $0$-dimensional cases respectively, while in the $1$-dimensional case
  implies $D\cdot c_1(M)>0$.
\end{proof}

\begin{cor}
  For any coherent $d$-dimensional sheaf $M$, there is a bound on the
  length of a descending chain $M=M_0\supset M_1\supset\cdots$ of
  subsheaves such that all quotients are $d$-dimensional.
\end{cor}

\begin{proof}
  Fix an ample divisor $D$, and consider the following identity of
  polynomials:
  \[
  \chi(\sO_X(-rD),M)
  =
  \chi(\sO_X(-rD),M_0/M_1)
  +
  \chi(\sO_X(-rD),M_1/M_2)
  +
  \cdots
  +
  \chi(\sO_X(-rD),M_l)
  \]
  for all $l$.  By assumption, all of the terms on either side of this
  expression are polynomials of the same degree with positive leading
  coefficient.  Since the leading coefficients are at worst half-integers,
  the number of terms on the right can be bounded in terms of the leading
  coefficient on the left; at most $\rank(M)$, $c_1(M)\cdot D$ or
  $\chi(M)$, depending on $d$.
\end{proof}

This leads to the following analogue of the usual definition of nef in the
commutative setting.

\begin{cor}
  A divisor class $D$ is nef iff for all $1$-dimensional sheaves $M$,
  $D\cdot c_1(M)\ge 0$.
\end{cor}

\begin{proof}
  If $M$ is $1$-dimensional, then $D\cdot c_1(M)>0$ for all ample divisors;
  since any nef divisor is a limit of rational multiples of ample divisors,
  we must at least have a weak inequality when $D$ is nef.

  Now, let $D$ be any divisor class having nonnegative intersection with
  the Chern class of every $1$-dimensional sheaf.  The sheaves $\sO_C$ and
  $\sO_{e_m}(-1)$ have Chern classes $C_m$ and $e_m$ respectively, and thus
  $D\cdot C_m,D\cdot e_m\ge 0$; we similarly have $D\cdot f\ge 0$ using a
  sheaf of the form $\sO_X/\sO_X(-f)$.  If $D\cdot\alpha\ge 0$ for all
  simple roots $\alpha$, then $D$ is universally nef, and we are done.
  Otherwise, choose a simple root $\alpha$ with $D\cdot \alpha<0$.  If
  $\rho(\alpha)\in q^{\Z}$, then there is a divisor class $D'$ such that
  \[
  \hat\cS'_{\rho;q;C}(D',D'+\alpha)\ne 0,
  \]
  and the cokernel of such a morphism will be a $1$-dimensional sheaf with
  Chern class $\alpha$, giving a contradiction.  Thus $\rho(\alpha)\notin
  q^{\Z}$, so that we may perform the corresponding simple reflection.
  Repeating this will either force $D\cdot f<0$ (a contradiction) or put
  $D$ into the universal nef cone, meaning that the original divisor class
  was nef.
\end{proof}

\begin{rem} Note, in particular, that this implies that the nef divisors
  are the integer points of a cone, as our terminology had suggested.
\end{rem}

With this in mind, we call a divisor class on $X_m$ {\em effective} if it
is the Chern class of a $1$-dimensional sheaf, or equivalently if it is in
the dual to the nef cone.  The set of effective divisor classes is clearly
the monoid generated by the Chern classes of {\em irreducible} sheaves
($1$-dimensional sheaves such that any subsheaf has the same Chern class).
Moreover, we can identify some such irreducible Chern classes: $-1$-curves,
i.e., those with $e^2=-1$, $e\cdot C_m=1$, and $-2$-curves, those with
$e^2=-2$, $e\cdot C_m=0$.  The prototypical example of a $-1$-curve is of
course $e_m$, and it follows as in the commutative case that any $-1$-curve
is related by an admissible element of $W(E_{m+1})$ to this case; in
particular, any $-1$-curve can be blown down.  Similarly, any $-2$-curve is
admissibly equivalent to a simple root.

We can further argue as in \cite[Cor.~4.5]{rat_Hitchin} to obtain the
following.

\begin{cor}
  If $m>0$, then the effective cone is generated by $-1$-curves,
  $-2$-curves, and $C_m$.
\end{cor}

We also have the following.

\begin{lem}\label{lem:subsheaf_Cherns_finite}
  Given any divisor class $D\in \Pic(X_m)$, there are only finitely many
  divisor classes $D'$ such that $D-D'$ and $D'$ are both effective.
\end{lem}

\begin{proof}
  Although the effective cone need not be finitely generated, it is
  contained in the cone generated by simple roots, $e_m$, and $C_m$ (since
  the nef cone contains the universal nef cone).  If $D_a$ is a universally
  ample divisor, then $D_a$ has positive intersection with the generators
  of this larger cone, and thus there are only finitely many $D'$ in the
  larger cone such that $D-D'$ is also in the larger cone.
\end{proof}

\begin{cor}
  A divisor class $D$ is ample iff for all $1$-dimensional sheaves $M$,
  $D\cdot c_1(M)>0$.
\end{cor}

\begin{proof}
  Indeed, this is precisely the same as saying that $D$ is in the interior
  of the nef cone.
\end{proof}

\begin{cor}
  A subquotient of a $d$-dimensional sheaf has dimension at most $d$.
\end{cor}

\begin{proof}
  If $\rank(M)=0$, then for any subsheaf $M'$, we have
  $\rank(M')+\rank(M/M')=0$, and thus both must vanish.  Similarly, if
  $\rank(M)=0$ and $c_1(M)=0$, then $c_1(M')=c_1(M/M')=0$.
\end{proof}

\begin{prop}
  Any nonzero subsheaf of $\sO_X$ has rank $1$.
\end{prop}

\begin{proof}
  Let $M$ be a nonzero subsheaf of $\sO_X$.  Then there exists a divisor
  class $D$ such that $\Hom(\sO_X(-D),M)\ne 0$, and composing gives a
  morphism $\sO_X(-D)\to \sO_X$.  Since $\hat\cS'_{\rho;q;C}$ is a domain,
  any such morphism is injective, and thus $M\supset \sO_X(-D)$.  It
  follows that
  \[
  1=\rank(\sO_X(-D))\le \rank(M)\le \rank(\sO_X)\le 1.
  \]
\end{proof}

\begin{rem}
  It follows easily that any sum of line bundles is a pure $2$-dimensional
  sheaf.
\end{rem}

\begin{cor}
  If $M$ is $1$-dimensional, then for any divisor class $D$,
  $\Hom(M,\sO_X(D))=0$ and $\Ext^2(\sO_X(-D),M)=0$.
\end{cor}

\begin{proof}
  The two claims are equivalent by weak Serre duality, so it suffices to
  show that for all $D$, $\Hom(M,\sO_X(D))=0$.  By twisting, we may reduce to the
  case $D=0$.  The image of any nonzero morphism $M\to \sO_X$ must be
  $1$-dimensional, as a quotient of $M$, and $2$-dimensional, as a subsheaf
  of $\sO_X$, and this gives a contradiction.
\end{proof}

\begin{prop}
  If $q$ is nontorsion, then for any $0$-dimensional sheaf $M$,
  the morphism $T_M^k:\bar\theta^k M\to M$ is 0 for some $k>0$.
\end{prop}

\begin{proof}
  Since $\rank(M)=0$ and $c_1(M)=0$, we compute that $\rank(M|^{\dL}_C)=0$
  and $c_1(M|^{\dL}_C)=q^{\chi(M)}$.  If $T:\bar\theta M\to M$ were
  surjective, then the kernel would need to be a nontrivial (since
  $q^{\chi(M)}\ne 1$) sheaf of rank 0 and degree 0 on $C$, a contradiction.
  We thus have an exact sequence
  \[
  0\to M'\to M\to M|_C\to 0,
  \]
  where, by induction in $\chi(M)$, $\bar\theta^k M'\to M'$ is 0 for some
  $k\ge 0$, forcing $\bar\theta^{k+1}M\to M$ to be 0.
\end{proof}

\begin{prop}
  If $M$ is $0$-dimensional, then $\Ext^p(\sO_X,M)=0$ for $p>0$.
\end{prop}

\begin{proof}
  Since the claim is preserved under extensions, we may restrict our
  attention to the case that $M$ has no proper subsheaves.  In particular,
  either $T_M=0$ or $\bar\theta M\cong M$, with the latter only occurring in
  the case that $q$ is torsion.  In the former case, $M$ is the structure
  sheaf of a point of $C$, and may thus be expressed as a quotient of some
  acyclic line bundle, which implies the desired result by dimension
  considerations.  On the other hand, if $\bar\theta M\cong M$, then we
  immediately find that $R\alpha_{m*} M$ is a sheaf, which must itself be
  $0$-dimensional, letting us reduce to the Hirzebruch surface case, where
  it is known \cite{ChanD/NymanA:2013}.
\end{proof}  

\medskip

Although we have seen that weak Serre duality is already quite powerful, we
would of course like to have the full form of duality.  The main difficulty
we encounter is that the above construction of the pairing corresponding to
weak Serre duality is not functorial in the line bundle, which prevents us
from using the fact that we can represent every sheaf as a bounded complex
of sums of line bundles.

A related difficulty is that, although weak Serre duality gives us maps
\[
\Ext^2(\sO_X(D),\bar\theta \sO_X(D))\cong k,
\]
these maps are not truly canonical.  Now, the map $T:\bar\theta \sO_X(D)\to
\sO_X(D)$ is injective, and thus we have a morphism
\[
\Ext^1(\sO_X(D),\sO_X(D)|_C)\to \Ext^2(\sO_X(D),\bar\theta \sO_X(D)).
\]
Since $\Ext^i(\sO_X(D),\sO_X(D))=0$ for $i>0$, this map is in fact an
isomorphism.  We also have a {\em canonical} isomorphism
\[
\Ext^1(\sO_X(D),\sO_X(D)|_C)
\cong
\Ext^1_C(\sO_X(D)|_C,\sO_X(D)|_C)
\cong
H^1_C(\sO_C),
\]
and thus we see that we can reduce the choice of isomorphism
\[
\Ext^2(\sO_X(D),\bar\theta \sO_X(D))\cong k
\]
to a choice of nonzero holomorphic differential on $C$.  With this in mind,
we define a new functor $\theta$ by
\[
\theta M := \bar\theta M\otimes_k H^0_C(\omega_C),
\]
and observe that the above calculation gives {\em natural} isomorphisms
\[
\Ext^2(\sO_X(D),\theta\sO_X(D))\to k.
\]

\begin{lem}
  The derived category $D^b \coh X_{\rho;q;C}$ is generated by an
  exceptional collection.
\end{lem}

\begin{proof}
  Indeed, the sequence
  \[
  \sO_X(e_m)/\sO_X,\sO_X(e_{m-1})/\sO_X,\dots
  \sO_X(e_1)/\sO_X,\sO_X(-s-2f),\sO_X(-s-f),\sO_X(-f),\sO_X
  \]
  is an exceptional collection, and we have already seen that it generates.
\end{proof}

\begin{rem}
  Another possibility is
  \[
  \sO_X(e_m)/\sO_X,\sO_X(e_{m-1})/\sO_X,\dots
  \sO_X(e_1)/\sO_X,\sO_X(s)/\sO_X,\sO_X(-2s-2f),\sO_X(-s-f),\sO_X,
  \]
  corresponding to blowing up $\P^2$ rather than $F_1$.
\end{rem}

\begin{thm}
  The functor $M\mapsto \theta M[2]$ is a Serre functor on $D^b\coh
  X_{\rho;q;C}$.  That is, for any coherent sheaves $M$, $N$, there are
  functorial isomorphisms
  \[
  \Ext^i(M,N)\cong \Ext^{2-i}(N,\theta M)^*.
  \]
\end{thm}

\begin{proof}
  Since $D^b\coh X_{\rho;q;C}$ has bounded and finite $\Ext$ groups, and is
  generated by an exceptional collection, it follows from
  \cite{BondalAI/KapranovMM:1990} that there {\em is} a Serre functor on
  $D^b\coh X_{\rho;q;C}$.  Write this Serre functor as $\hat\theta[2]$
  (where $\hat\theta$ is in principle only a {\em derived} equivalence).
  Weak Serre duality gives
  \[
  \Ext^i(M,\theta \sO_X(D))
  \cong
  \Ext^{2-i}(\sO_X(D),M)^*
  \cong
  \Ext^i(M,\hat\theta \sO_X(D))
  \]
  for any sheaf $M$, functorially in $M$; taking $M=\sO_X(-D')$ for
  sufficiently ample $D'$ implies that $\hat\theta \sO_X(D)$ is not only a
  sheaf, but isomorphic to $\theta \sO_X(D)$.  Since both morphisms are
  canonical in the case $M=\theta \sO_X(D)$, we conclude that we have a
  natural isomorphism $\theta\cong \hat\theta$ as required.
\end{proof}

We also want to consider another useful form of duality, analogous to the
functor $R\sHom(_,\omega_X)$ in the commutative case.  Since the formal
adjoint induces an isomorphism $(\hat{\cal S}'_{\rho;q;C})^{\text{op}}\cong
\hat{\cal S}'_{\rho;q;C}$, it induces a contravariant equivalence $\ad$
between the subcategories of $\coh X_{\rho;q;C}$ and $\coh X_{\rho;1/q;C}$
induced by the line bundles.  We fix a scaling for this functor by
insisting that $\ad \sO_X(D)\cong \theta \sO_X(-D)$.

\begin{prop}
  The functor $\ad$ extends to a contravariant equivalence $R\ad:D^b\coh
  X_{\rho;q;C}^{\text{\rm op}}\to D^b\coh X_{\rho;1/q;C}$, with $R\ad R\ad\cong
  \id$ and $R\ad \sO_X(D)\cong \theta \sO_X(-D)$.  Moreover,
  $R\ad \alpha_m^*\cong \alpha_m^!R\ad$, $R\ad \theta\cong \theta^{-1}
  R\ad$, and $R\ad$ acts on Chern classes as
  \begin{align}
    \rank(R\ad M) &= \rank(M)\\
    c_1(R\ad M) &= -c_1(M)-\rank(M) C_m\\
    \chi(R\ad M) &= \chi(M).
  \end{align}
\end{prop}

\begin{proof}
  If an extension $R\ad$ exists with the above properties, it must in
  particular satisfy
  \[
  R\Hom_{X_{\rho;1/q;C}}(\sO_X(-D),R\ad M)
  \cong
  R\Hom_{X_{\rho;q;C}}(M,\theta\sO_X(D))
  \]
  for all divisor classes $D$.  Now, the right-hand side certainly gives
  a well-defined functor from $D^b\coh X_{\rho;q;C}$ to $D^b ({\cal
    S}'_{\rho;q;C})^{\text{op}}$, and thus after applying the formal adjoint
  and quotienting by torsion modules, we obtain a functor $R\ad:D^b\coh
  X_{\rho;q;C}^{\text{op}}\to D\qcoh X_{\rho;1/q;C}$, with $R^0\ad$ agreeing
  with $\ad$ on line bundles since $\ad$ is an equivalence.

  Since
  \[
  R\Hom_{X_{\rho;q;C}}(\sO_X(D),\theta\sO_X(D'))
  \]
  is supported in degree 0 for $D'-C_m-D$ nef, so is
  \[
  R\Hom_{X_{\rho;1/q;C}}(\sO_X(-D'),R\ad \sO_X(D)).
  \]
  But this implies that $R\ad \sO_X(D)$ is a sheaf.  In particular,
  applying $R\ad$ to a bounded complex of finite sums of line bundles gives
  another such complex, and thus $R\ad$ preserves bounded coherent
  complexes in general.  The remaining claims follow similarly by reduction
  to the line bundle case.
\end{proof}

\begin{rem} It will be convenient to refer to $X_{\rho;1/q;C}$ as $\ad
  X_{\rho;q;C}$, especially in cases when we wish to suppress parameters.
\end{rem}

This satisfies the usual compatibility condition relative to the inclusion
of $C$ in $X$.

\begin{prop}
  For any sheaf $M$ supported on $C$, there is a natural isomorphism
  \[
  R\ad M\cong R\sHom_C(M,\omega_C)[-1].
  \]
\end{prop}

\begin{proof}
  We have
  \begin{align}
  R\Hom_C(\sO_X(-D)|_C,R\ad M)
  &\cong
  R\Hom_{\ad X}(\sO_X(-D),R\ad M)\notag\\
  &\cong
  R\Hom_X(M,\theta \sO_X(D))\notag\\
  &\cong
  R\Hom_X(\sO_X(D),M[2])\notag\\
  &\cong
  R\Hom_C(\sO_X(D)|_C,M[2])\notag\\
  &\cong
  R\Hom_C(M,\omega_C\otimes \sO_X(D))[-1]\notag\\
  &\cong
  R\Hom_C(\sO_X(-D)|_C,\sHom_C(M,\omega_C))[-1].\notag
  \end{align}
\end{proof}  

\medskip

Comparing to the axioms of \cite{ChanD/NymanA:2013}, we see that there is
only one thing left to prove, namely that $\Quot$ schemes are countable
unions of projective schemes.  Of course, in the commutative setting, they
are projective as long as we specify the Hilbert polynomial of the
quotient.

The usual treatment of this in the literature relies on the theory of
Castelnuovo-Mumford regularity.  It is unclear whether there is an analogue
of this theory for general very ample divisor classes on $X_m$.  The main
difficulty is that a very ample divisor class in principle need not
correspond to an embedding in a noncommutative $\P^n$.  The proof in
\cite[\S 14]{MumfordD:1966} of the main properties of Castelnuovo-Mumford
regularity involves an induction via hyperplane sections, which would seem
to be another issue, but in fact allows us to abstract away the embedding.
The key point is that the cohomology of a restriction of a sheaf to a
hyperplane can be expressed in terms of the hyper-$\Ext$ from an
appropriate linear complex.

For an ample divisor class $D_a$, call a coherent sheaf {\em
  $(l,D_a)$-regular} if
$\Ext^2(\sO_X((-l+2)D_a),M)=\Ext^1(\sO_X((-l+1)D_a),M)=0$, and {\em
  $D_a$-regular} if it is $(0,D_a)$-regular.

\begin{lem}\label{lem:CM_regularity}
  Suppose $D_a$ is an ample divisor class such that for all $l$ there is a
  sheaf $F_l$ with a resolution of the form
  \[
  0\to \sO_X((l-3)D_a)\to \sO_X((l-2)D_a)^{c_2}\to \sO_X((l-1)D_a)^{c_1}\to
  \sO_X(lD_a)^{c_0}\to F_l\to 0.
  \]
  Then any $(0,D_a)$-regular sheaf is globally generated and
  $(l,D_a)$-regular for all $l\ge 0$.
\end{lem}

\begin{proof}
  Suppose $M$ is a $(0,D_a)$-regular sheaf.  Since $\sO_X(D_a)\subset
  \sO_X(2D_a)$, we have a surjection $\Ext^2(\sO_X(2D_a),M)\to
  \Ext^2(\sO_X(D_a),M)$, and thus the latter also vanishes.  Now, the
  truncated complex
  \[
  0\to \sO_X\to \sO_X(D_a)^{c_2}\to \sO_X(2D_a)^{c_1}
  \]
  represents a sheaf, and thus the hyper-$\Ext$ from the complex to $M$ is
  supported in degrees $0$, $1$, $2$.  The $E_1$ page of the corresponding
  spectral sequence has terms
  \[
  \begin{CD}
    0@>>> 0 @>>> \Ext^2(\sO_X,M)\\
    \Ext^1(\sO_X(2D_a)^{c_1},M) @>>> 0 @>>> \Ext^1(\sO_X,M)\\
    \Hom(\sO_X(2D_a)^{c_1},M) @>>> \Hom(\sO_X(D_a)^{c_2},M) @>>>
    \Hom(\sO_X,M)
  \end{CD}
  \]
  and thus $M$ is acyclic.  Since then $M(D_a)$ is $D_a$-regular, it
  remains only to show that $M$ is globally generated.
  Here we use the spectral sequence for $R\Hom(F_2,M)$, where we find that
  the sheaves that would map to $E^{03}$ on all later pages already vanish
  on the $E_1$ page, and thus we have an exact sequence
  \[
  \Hom(\sO_X,M)^{c_2}\to \Hom(\sO_X(-D_a),M)\to E^{03}_\infty\to 0.
  \]
  But $\Ext^3(F_2,M)=0$ and thus $E^{03}_\infty=0$, so we conclude that
  $\Hom(\sO_X,M)^{c_2}\to \Hom(\sO_X(-D_a),M)$ is surjective.
  That map factors through the natural map
  \[
  \Hom(\sO_X(-D_a),\sO_X)\otimes_k \Hom(\sO_X,M)\to \Hom(\sO_X(-D_a),M),
  \]
  which thus must also be surjective.  It follows by induction that the
  natural map
  \[
  \Hom(\sO_X(-lD_a),\sO_X)\otimes_k \Hom(\sO_X,M)\to \Hom(\sO_X(-lD_a),M)
  \]
  is surjective for all $l\ge 0$.  Indeed, the composition
  \begin{align}
  &\Hom(\sO_X(-(l+1)D_a),\sO_X(-lD_a))\otimes_k \Hom(\sO_X(-lD_a),\sO_X)
  \otimes_k \Hom(\sO_X,M)\notag\\
  &\qquad\to
  \Hom(\sO_X(-(l+1)D_a),\sO_X(-lD_a))\otimes_k \Hom(\sO_X(-lD_a),M)\notag\\
  &\qquad\to
  \Hom(\sO_X(-(l+1)D_a),M)
  \end{align}
  is surjective, and factors through
  \[
  \Hom(\sO_X(-(l+1)D_a),\sO_X)\otimes_k \Hom(\sO_X,M)
  \to
  \Hom(\sO_X(-(l+1)D_a),M).
  \]
  Since $D_a$ is ample, this implies that $M$ is globally generated as required.
\end{proof}

\begin{cor}\label{cor:regularity_on_X0}
  Any ample divisor of the form $s+d'f$ on $X_0$ or $X'_0$ satisfies the
  hypothesis of Lemma \ref{lem:CM_regularity}.
\end{cor}

\begin{proof}
  Let $m=D_a^2-2>0$, and let $y_1,\dots,y_m$ be a sufficiently general
  $m$-tuple of points on $C$.  If we blow up those points, then
  $D_a-e_1-\cdots-e_m$ is $W(E_{m+1})$-equivalent to $f$.  We thus find
  that there is an acyclic complex of the form
  \[
  \sO_X(e_1+\cdots+e_m-D_a+D)\to \sO_X(D)^2\to \sO_X(D+D_a-e_1-\cdots-e_m)
  \]
  for any divisor $D$ on $X_m$.  (This follows from the corresponding fact
  for $f$.)  Taking the Yoneda product with the extension
  \[
  0\to \sO_X(D+D_a-e_1-\cdots-e_m)\to \sO_X(D+D_a)
  \to \bigoplus_i \sO_{e_i}(D\cdot e_i)\to 0
  \]
  gives a complex
  \[
  \sO_X(e_1+\cdots+e_m-D_a+D)\to \sO_X(D)^2\to \sO_X(D+D_a)
  \]
  representing $\bigoplus_i \sO_{e_i}(D\cdot e_i)$.  Taking the direct
  image in the case $D=(l-1)D_a$ gives a complex
  \[
  \sO_X((l-2)D_a)\to \sO_X((l-1)D_a)^2\to \sO_X(lD_a)
  \]
  representing $\oplus_i \sO_{qy_i}$ for each $l$.  Since $R\Hom(\sO_X(l_1
  D_a),\sO_X(l_2 D_a))$ is concentrated in degree 0 for $l_1\le l_2+1$, we
  find that the identity on this sum of point sheaves induces a chain map
  \[
  \begin{CD}
  \sO_X((l-3)D_a)@>>> \sO_X((l-2)D_a)^2 @>>> \sO_X((l-1)D_a)\\
  @VVV @VVV @VVV\\
  \sO_X((l-2)D_a)@>>> \sO_X((l-1)D_a)^2 @>>> \sO_X(lD_a)
  \end{CD}
  \]
  the cone of which is a linear complex which is acyclic (and so certainly
  represents a sheaf).
\end{proof}
  
The analogous argument on $X_{-1}$ gives the following.

\begin{cor}\label{cor:regularity_on_P2}
  The divisors $h$ and $2h$ on $X_{-1}$ satisfy the hypothesis of Lemma
  \ref{lem:CM_regularity}. 
\end{cor}

\begin{rem}
  We will see in Appendix \ref{sec:genK6} below that these cases actually
  {\em do} correspond to embeddings (in at least a weak sense) in
  noncommutative $\P^n$s.
\end{rem}

\begin{cor}
  If $m\le 5$ and $C_m$ is ample, then $D_a=C_m$ satisfies the hypotheses
  of Lemma \ref{lem:CM_regularity}.
\end{cor}

\begin{proof}
  If we restrict the four-term acyclic complex from Corollary
  \ref{cor:regularity_on_X0} or Corollary \ref{cor:regularity_on_P2} to
  $C$, we obtain (up to twisting by a line bundle) an acyclic complex of
  the form
  \[
    {\cal L}_3\to {\cal L}_2^3\to {\cal L}_1^3\to \sO_C
  \]
  where ${\cal L}_i$ are line bundles such that ${\cal L}_3\otimes {\cal
    L}_1^3\cong {\cal L}_2$.  Moreover, by varying the surface we use, we
  may arrange for ${\cal L}_1$ and ${\cal L}_2$ to be {\em arbitrary}
  bundles of the appropriate degrees.

  In particular, we deduce that there is an acyclic complex of the form
  \[
  \sO_X(-3C_m)|_C\to \sO_X(-2C_m)|_C^3\to \sO_X(-C_m)|_C^3\to \sO_C
  \]
  on our original choice of $X_m$.  If we represent each line bundle on $C$
  via a complex $\sO_X(-(l+1)C_m)\to \sO_X(-lC_m)$, then all of the
  relevant $\Ext$ groups vanish, and thus the above complex on $C$ induces
  an acyclic complex of the form
  \[
  \sO_X(-4C_m)\to \sO_X(-3C_m)^4\to \sO_X(-2C_m)^6\to \sO_X(-C_m)^4\to
  \sO_X.
  \]
  Applying powers of $\theta$ gives a family of sheaves with linear
  resolutions as required.
\end{proof}


We do not know how to apply this to general very ample divisor classes, but
the above is enough to let us prove the following.

\begin{prop}\label{prop:noetherian_bounded}
  Suppose $M$ is a flat family of coherent sheaves on $X_m$ over a
  noetherian base.  Then the set of divisors $D$ such that every fiber of
  $M(D)$ is acyclic and globally generated contains a translate of the nef
  cone.
\end{prop}

\begin{proof}
  We first show that there is {\em some} divisor class such that $M(D)$ is
  acyclic and globally generated.  To show this, first note that for any
  fiber $M_v$, if $\Hom(M_v,\sO_{e_m}(-2))=0$ then $R^1\alpha_{m*}M_v=0$;
  if $M_v$ is pointless this follows from Lemma \ref{lem:poisson5.10}, and
  in general it follows by writing $M_v$ as an extension of a pointless sheaf
  by a $0$-dimensional sheaf.  Now, $\Hom(M_v,\sO_{e_m}(-2))=0$ iff
  $\Ext^2(\sO_{e_m}(-1),M_v)=0$, and this is an open condition by
  semicontinuity.  Moreover, since $\theta \sO_{e_m}(-1)\subset
  \sO_{e_m}(-1)$, we find that $\Ext^2(\sO_{e_m}(-1),M_v)=0$ implies
  $\Ext^2(\sO_{e_m}(-1),\theta^{-1}M_v)$.  We thus obtain a nested sequence
  of closed subschemes of the base by taking those points where the fiber
  satisfies $\Ext^2(\sO_{e_m}(-1),\theta^{-r}M_v)\ne 0$.  Since
  $R^1\alpha_{m*}\theta^{-r}M_v$ is eventually acyclic, this nested sequence
  has empty intersection, so by the noetherian hypothesis is eventually
  empty.  In particular, there exists $r$ such that for {\em every} fiber,
  $R^1\alpha_{m*}\theta^{-r}M_v=0$.  If $(\alpha_{m*}\theta^{-r}M_v)(D)$ is
  globally generated, then so is $M_v(D+rC_m)$, so that the claim reduces by
  induction to the case $m\le 0$.

  For $m\le 0$, we choose $D_a$ as in Corollary \ref{cor:regularity_on_X0}
  or Corollary \ref{cor:regularity_on_P2} and take the nested sequence of
  closed subschemes of the base to be those fibers for which $M_v$ is not
  $(l,D_a)$-regular.

  We have thus shown that there is a divisor class $D_1$ such that for
  every fiber of the family, $\Ext^p(\sO_X(-D_1),M_v)=0$ for $p>0$ and
  \[
  \sO_X(-D_1)\otimes_k \Hom(\sO_X(-D_1),M_v)\to M_v
  \]
  is surjective.  Acyclicity implies that $\sHom(\sO_X(-D_1),M)$ is a flat
  coherent sheaf on the base, and thus the kernel of this family of
  surjective maps is itself a flat family over a noetherian base.  We thus
  conclude that there is an exact sequence of the form
  \[
  0\to M'\to \sO_X(-D_2)\otimes_k V\to \sO_X(-D_1)\otimes_k
  \sHom(\sO_X(-D_1),M)\to M\to 0,
  \]
  where again $V$ is a flat coherent sheaf on the base and $M'$ is a flat
  family of coherent sheaves on $X_m$.  Now, if $D-D_1$ and $D-D_2$ are
  both very ample, then $\sO_X(D-D_1)$ and $\sO_X(D-D_2)$ are acyclic and
  globally generated.  Since $M(D)$ is then a quotient of a globally
  generated sheaf, it is globally generated as required.

  To see that $M(D)$ is acyclic, we use the hyper-$\Ext$ spectral
  sequence, and observe that the terms coming from $R\Gamma(M'(D))$
  cannot cancel the terms corresponding to $\Ext^1(M(D))$ and
  $\Ext^2(M(D))$, which must therefore vanish as required.
\end{proof}

\begin{rem}
  Note that we can allow the surface to vary in the family, as long as we
  work relative to the intersection of the nef cones of the fibers.
\end{rem}

We will need a further boundedness result, but this depends on the theory
of Hilbert schemes of points, so will be proved as Lemma
\ref{lem:glob_generated_strongly_bounded} below.

\begin{thm}\label{thm:quot_is_projective}
Let $M$ be a coherent sheaf on $X_{\rho;q;C}$, and let
$\Quot(M;(r,D,l))$ denote the moduli functor classifying quotients of $M$
with numerical invariants $(r,D,l)$.  Then $\Quot(M;(r,D,l))$ is
represented by a projective scheme.
\end{thm}

\begin{proof}
  Since the claim is clearly invariant under twisting, we may as well
  assume that $M$ is acyclic and globally generated.  But then
  $\Quot(M;(r,D,l))$ is a closed subfunctor of $\Quot(\sO_X\otimes_k
  \Hom(\sO_X,M);(r,D,l))$, so that it suffices to consider $M=\sO_X^n$.
  Let $D_a$ be a universally very ample divisor class.  By Lemma
  \ref{lem:glob_generated_strongly_bounded} below, there is a bound $B$
  such that for all short exact sequences
  \[
  0\to I\to \sO_X^n\to N\to 0
  \]
  corresponding to points of the $\Quot$ functor, $I(bD_a)$ and $N(bD_a)$
  are acyclic and globally generated for $b\ge B$.

  Now, let $M_0$ be the module associated to $\sO_X$ over the $\Z$-algebra
  corresponding to $D_a$, and let $M_B$ be the truncation obtained by
  removing the subspaces of degree $<B$.  Then by the straightforward
  extension of \cite[Thm.~E4.3]{ArtinM/ZhangJJ:1994} to $\Z$-algebras, the
  quotients of $M_B^n$ that have dimension $\chi(N(bD_a))$ for all $b\ge B$
  are classified by a projective scheme.  Let
  \[
  Q^+=0\to I^+\to \sO_X\to N^+\to 0
  \]
  be the induced family of short exact sequences of sheaves on $X$.  By our
  choice of $B$, this contains a universal family for the $\Quot$ scheme,
  and thus it remains only to show that the condition that a given fiber of
  $N^+$ has the desired numerical invariants is both closed and open.  But
  this follows from the fact that the function $\chi(N^+(D))$ is locally
  constant for any $D$.
\end{proof}

\begin{rem}
  Since we took $D_a$ to be universally very ample, the argument actually
  works to construct the relative $\Quot$ scheme over any noetherian base,
  with both $M$ and the surface allowed to vary,
  since Proposition \ref{prop:noetherian_bounded} implies that the acyclic
  and globally generated twist can be chosen uniformly.
\end{rem}

\begin{rem}
  Traditionally, this is stated in terms of Hilbert polynomials, but for
  our purposes this is less natural.  However, the same argument used to
  prove Lemma \ref{lem:glob_generated_strongly_bounded} shows that $B$ can
  be chosen to depend only on the Hilbert polynomial, in which case $N^+$
  represents the moduli functor classifying quotients with the chosen
  Hilbert polynomial.
\end{rem}

Putting the above results together gives the following.

\begin{thm}
  The quasi-scheme $X_{\rho;q;C}$ is a noncommutative smooth proper surface
  in the sense of \cite{ChanD/NymanA:2013}.
\end{thm}

\begin{proof}
  All of the axioms are either proved explicitly above or follow using the
  same arguments as in \cite{ChanD/NymanA:2013} for the ruled surface case.
\end{proof}

\section{Moduli spaces}
\subsection{Poisson structures}

Now that we know $X_{\rho;q;C}$ is a well-behaved noncommutative surface,
our next objective is to understand moduli spaces of sheaves on
$X_{\rho;q;C}$.  In particular, since noncommutative surfaces are
deformations of Poisson surfaces, and moduli spaces of sheaves on Poisson
surfaces are Poisson
\cite{TyurinAN:1988,BottacinF:1995,HurtubiseJC/MarkmanE:2002b}, we would
like to prove something similar in this setting.  In addition, it was shown
in \cite{poisson} for the commutative setting that there are a number of
natural Poisson morphisms between moduli spaces arising from birational
morphisms, which again we would like to extend.

The first thing we need to do is construct the Poisson moduli space itself.
The discussions in
\cite{TyurinAN:1988,BottacinF:1995,HurtubiseJC/MarkmanE:2002b} concern
moduli spaces of {\em stable} sheaves, which appear to be difficult to
construct in general (see the next section for some special cases), as they
require boundedness results.  However, as observed in \cite{poisson}, a
discussion of Poisson structures makes sense on general algebraic spaces,
and there is a general construction \cite{AltmanAB/KleimanSL:1980} of an
algebraic space representing a moduli problem that includes all stable
sheaves.  Thus our first order of business is to extend this construction
to the noncommutative setting.

If $I_S$ is a flat family of coherent sheaves on $X_{\rho;q;C}$ (where as
usual $(\rho,q,C)$ may vary over the base scheme $S$), we follow
\cite{AltmanAB/KleimanSL:1980} in calling $I_S$ {\em simple} if for any
base change, the natural map
\[
\sO_T\to \sHom_T(I_T,I_T)
\]
is an isomorphism.  Note that Serre duality allows us to rephrase the
condition as saying that
\[
\sHom_T(\sExt^2_T(I_T,I_T),\sO_T)\to \sO_T
\]
is an isomorphism.  This makes the proof of
\cite[Prop.~5.2]{AltmanAB/KleimanSL:1980} carry through, so that
``simple'' is an open condition.  (To be precise, there is an open,
retrocompact subscheme of the base such that a morphism factors through the
subscheme iff the base change is simple.)  This allows us to rephrase the
condition as saying that every endomorphism of a geometric fiber of $I_S$
is a multiple of the identity.  In addition, this allows us to define a
functor $\Spl_{X_{\rho;q;C}/S}$ such that for any $S$-scheme $T$,
$\Spl_{X_{\rho;q;C}/S}(T)$ is the set of equivalence classes of simple
families over $T$, where two families are equivalent if there is an
isomorphism
\[
I\otimes_S L\cong J
\]
where $L$ is an invertible sheaf on $S$.  Similarly, define a {\em twisted
  family of simple sheaves} over $T$ to be a family of simple sheaves over
an \'etale cover $U\to T$ such that the two induced families on $U\times_T
U$ are equivalent; two such families are equivalent if they become
equivalent under base change to a common \'etale cover.

\begin{thm}\label{thm:spl_is_space}
  There is a quasi-separated algebraic space $\Spl_{X_{\rho;q;C}/S}$
  locally finitely presented over $S$ which represents the moduli functor
  of simple sheaves, in the sense that there is a natural bijection between
  twisted families of simple sheaves on $X_{\rho;q;C}$ and morphisms to
  $\Spl_{X_{\rho;q;C}/S}$.
\end{thm}

\begin{proof}
  We adapt the argument of \cite[Thm.~7.4]{AltmanAB/KleimanSL:1980}, with
  some minor modifications to avoid using Castelnuovo-Mumford regularity.
  Fix a universally ample divisor class $D$.  For each pair $r,n$ of
  positive integers, we may consider the open subfunctor $\Sigma_{r,n}$
  classifying (twisted families of) simple sheaves $M$ such that
  $\Ext^i(\sO_X(-rD),M)=0$ for $i>0$, $\chi(\sO_X(-rD),M)=n$, and the
  natural morphism $\sO_X(-rD)\otimes \Hom(\sO_X(-rD),M)\to M$ is
  surjective.  As $r$ and $n$ vary, these subfunctors certainly cover
  $\Spl_{X_{\rho;q;C}/S}$, so it suffices to show that these subfunctors
  are represented by algebraic spaces.  Now, consider the projective scheme
  $\Quot(\sO_X(-rD)^n)$.\footnote{Note that we only need that it is locally of
  finite type and a countable union of projective subschemes, which follows
  by the same argument as Theorem \ref{thm:quot_is_projective} without
  needing the boundedness result we have yet to prove.}  This scheme has a
  universal family, and thus there is an open subscheme parametrizing
  simple sheaves such that $\Ext^i(\sO_X(-rD),M)=0$ for $i>0$ and
  $\chi(\sO_X(-rD),M)=n$; we may furthermore impose the open condition that
  the induced map
  \[
  \Hom(\sO_X(-rD),\sO_X(-rD)^n)\to \Hom(\sO_X(-rD),M)
  \]
  is an isomorphism.  This open subscheme $Z_{r,n}$ still covers
  $\Sigma_{r,n}$, and is in fact a $\PGL_n$-bundle over $\Sigma_{r,n}$: it
  represents the moduli problem in which we include the additional data of
  a basis of $\Hom(\sO_X(-rD),M)$, with proportional bases being
  equivalent.  Moreover, given any morphism to $\Sigma_{r,n}$, the pullback
  to $Z_{r,n}$ is quasi-projective, as it is an open subscheme of
  \[
  \P(\Ext^2(M,\theta \sO_X(-rD)^n));
  \]
  the sheaf $\Ext^2(M,\theta \sO_X(-rD)^n)$ is locally free since it is the
  only nonvanishing $\Ext$ sheaf.  It follows as in
  \cite{AltmanAB/KleimanSL:1980} that $\Sigma_{r,n}$ is representable by
  the quotient $Z_{r,n}/\PGL_n$, and the claim follows.
\end{proof}

The covering by quotients of $\Quot$ schemes allows us to compute the
cotangent sheaf, as in \cite{LehnM:1998,poisson}.  Note that an
\'etale morphism $T\to \Spl_{X_{rho;q;C}/S}$ is a formally universal {\em
  twisted} family of simple sheaves, but by definition there is always an
\'etale cover $U\to T$ such that the twisted family becomes an honest
family.  In particular, any \'etale cover of $\Spl_{X_{\rho;q;C}/S}$ can be
refined to a cover by formally universal families, and thus sheaves on
$\Spl_{X_{\rho;q;C}/S}$ can be described in terms of the induced sheaves on
the base of such families.

\begin{lem}\label{lem:spl_cotangent}
   Let $M$ be a formally universal family of simple sheaves on
   $X_{\rho;q;C}/S$ with base scheme $U$.  Then there is a natural
   isomorphism
   \[
   \Omega_U\cong \sExt^1_U(M,\theta M).
   \]
\end{lem}

This allows one to define a biderivation (following \cite{TyurinAN:1988})
on $\Spl_{X_{\rho;q;C}/S}$, or equivalently on the base of every formally
universal family, by
\[
\Omega_U^{\otimes 2}
\cong
\sExt^1_U(M,\theta M)^{\otimes 2}
\to
\sExt^1_U(M,M)
\otimes
\sExt^1_U(M,\theta M)
\to
\sO_U.
\]
Note that for this to make sense, we must fix a natural transformation
$\theta\to \id$.  In particular, if we fix a nonzero holomorphic
differential $\omega$ on $C$, then this allows us to identify $\theta$ and
$\bar\theta$, at which point $T$ gives such a natural transformation.
(Such a choice also appears in the commutative setting as the difference
between a choice of anticanonical curve and a choice of nontrivial Poisson
structure.)

We then have the following result.

\begin{thm}\label{thm:spl_is_poisson}
  Let $k$ be an algebraically closed field, let $(C,\rho,q,\omega)$ be a
  quadruple defined over $k$.  Then the above biderivation equips
  $\Spl_{X_{\rho;q;C}}$ with a Poisson structure, and for any
  complex $M^{\cdot}_C$ of locally free sheaves on $C$, the subspace of
  $\Spl_{X_{\rho;q;C}}$ parametrizing sheaves $M$ with $M|^{\dL}_C\cong
  M^{\cdot}_C$ is a (smooth) symplectic Poisson subspace.
\end{thm}

\begin{rem}
  As in \cite{poisson}, the claim that the biderivation gives a Poisson
  structure actually works fine in families, but the symplectic leaf claim
  has difficulties over more general schemes.
\end{rem}

The argument in \cite{BottacinF:1995,HurtubiseJC/MarkmanE:2002b} for the
commutative analogue of this fact (modulo the general claim about
symplectic leaves) works by first showing that it holds for the subspace
parametrizing vector bundles, and then showing that there is an open
covering by subspaces isomorphic to open subspaces parametrizing vector
bundles, in such a way that the isomorphism respects the pairing on the
cotangent sheaf.  Of course, there is a major difficulty in applying this
argument in the noncommutative case: we do not have a natural notion of
vector bundles on $X_{\rho;q;C}$ (and certainly cannot expect to do
explicit calculations involving transition matrices!).  That said, the
reduction step itself works, and will prove useful.

Following \cite{HurtubiseJC/MarkmanE:2002b} (see also the discussion in
\cite{poisson}), let $U\to \Spl_{X_{\rho;q;C}}$ be an \'etale neighborhood
with corresponding family $M$ of simple sheaves, where we suppose $U$
noetherian.  There thus exists a divisor class $D$ such that
$\sExt^i_U(\sO_X(-D),M)=\sExt^i_U(\sO_X(-D),\theta M)=0$ for $i>0$ and
$\sO_X(-D)\otimes_U \sExt^i_U(\sO_X(-D),M)\to M$ is surjective.  We may
then consider the short exact sequence
\[
0\to V\to \sO_X(-D)\otimes_U \sExt^i_U(\sO_X(-D),M)\to M\to 0
\]
Then one can show (indeed, the argument as described in \cite{poisson}
carries over directly; ``locally free'' is only used to claim that one has
reduced to the vector bundle case) that $V$ is simple, and thus defines a
new morphism $U\to \Spl_{X_{\rho;q;C}}$.  Moreover, one can reconstruct $M$
from $V$, and the two induced biderivations on $U$ are the same (up to an
overall sign).  Finally, the derived restrictions $M|^{\dL}_C$ and
$V|^{\dL}_C$ determine each other (once we fix the class of $M$ in
$K_0(X_{\rho;q;C})$).  In other words, the Theorem holds on the \'etale
neighborhood corresponding to $M$ iff it holds on the \'etale neighborhood
corresponding to $V$.

We should note here that this reduction applies equally well when $M$ has a
$0$-dimensional subsheaf (at which point simplicity implies that $M$ is
$0$-dimensional); in the commutative case, this was unnecessary since the
only simple $0$-dimensional sheaves are structure sheaves of points, but in
the noncommutative setting, there are examples of any Euler characteristic
$\le |\langle q\rangle|$, see below.

This reduction, though it does not lead to a relatively straightforward
case like vector bundles, does improve the sheaves under consideration in
some respects.  First, since the middle of the short exact sequence
defining $V$ is a sum of line bundles, so is acyclic for $|^\dL_C$, $V$
must also be acyclic for $|^\dL_C$ (since derived restriction has
homological dimension 1). If $M$ were already acyclic for $|^\dL_C$, then
$V|_C$ would be a subsheaf of some power $\sO_X(-D)^N|_C$, and thus
torsion-free.  And in either case, since $V\subset \sO_X(-D)^N$, it must be
a pure $2$-dimensional sheaf.  Finally, if we twist by $D$ (which has no
effect on the biderivation), we find that $R\Gamma(V)=0$.

At this point, we make use of the following (rather surprising) fact.  Let
$\cN_{\rho;q;C}$ denote the triangulated subcategory of $D^b\coh
X_{\rho;q;C}$ consisting of objects $M$ such that $R\Gamma M=0$.

\begin{prop}
  The triangulated subcategory $\cN_{\rho;q;C}$ is independent of $q$. More
  precisely, for any $q$, there is an equivalence $\kappa_q:\cN_{\rho;q;C}\cong
  \cN_{\rho;1;C}$ such that
  \[
  (\kappa_q M)|^{\dL}_C \cong M|^{\dL}_C\otimes q.
  \]
\end{prop}

\begin{proof}
  We have already seen that $D^b\coh X_{\rho;q;C}$ is generated by the
  exceptional collection
  \[
  \sO_X(e_m)/\sO_X,
  \sO_X(e_{m-1})/\sO_X,
  \dots
  \sO_X(e_1)/\sO_X,
  \sO_X(-s-2f),
  \sO_X(-s-f),
  \sO_X(-f),
  \sO_X.
  \]
  Moreover, $D^b\coh X_{\rho;q;C}$, as a derived category, comes with an
  enhancement to a dg-category (see, e.g., \cite{OrlovD:2014}), and that
  enhancement is equivalent to the (subcategory of compact objects of the)
  dg-category of modules over the exceptional collection.  The same is
  therefore true for $\cN_{\rho;q;C}$, where we simply omit $\sO_X$ from
  the exceptional collection.  Now, in general, if $M$ and $N$ are two
  objects with $R\Hom(N,M)=0$, then $R\Hom(M,\theta N)=0$ by Serre duality,
  and thus
  \[
  R\Hom(M,N)
  \cong
  R\Hom(M,N|^{\dL}_C)
  \cong
  R\Hom_C(M|^{\dL}_C,N|^{\dL}_C).
  \]
  But this means that the dg-subcategory with objects the exceptional
  collection can be obtained from the corresponding subcategory of $D^b\coh
  C$ by simply omitting all backwards morphisms and nonidentity
  endomorphisms.

  The restrictions of the objects in the exceptional collection can be
  computed from the fact that they are sheaves with known class in
  $K_0(C)$, giving
  \[
  \sO_{x_m},\sO_{x_{m-1}},\dots,\sO_{x_1},\rho(-s-2f)\otimes
  q^{-1},\rho(-s-f)\otimes q^{-1},\rho(-f)\otimes q^{-1}.
  \]
  If we tensor this with $q$, the result is independent of $q$, and the
  claim follows immediately.
\end{proof}

\begin{rem}
  Note that the derived category $D^b\coh X_{\rho;q;C}$ can be
  reconstructed from the sequence
  \[
  \sO_{x_m},\sO_{x_{m-1}},\dots,\sO_{x_1},\rho(-s-2f),\rho(-s-f),\rho(-f),q;
  \]
  we will use this below to construct derived equivalences.
\end{rem}

In principle, $\kappa_q$ depends strongly on the particular choice of
exceptional collection used to generate $\cN_{\rho;q;C}$.  Luckily, this
dependence is relatively mild, and indeed, $\kappa_q$ commutes with the
action of admissible elements of $W(E_{m+1})$.  This follows in a
straightforward way from the following Lemma.

\begin{lem}
  For all $0\le l\le m$, and all integers $d$,
  \[
  \kappa_q(\sO_X(-s-df+e_1+\cdots+e_l))
  \cong \sO_X(-s-df+e_1+\cdots+e_l).
  \]
  In addition, $\kappa_q(\sO_X(-2s-2f+e_1+\cdots+e_l))\cong
  \sO_X(-2s-2f+e_1+\cdots+e_l)$, $\kappa_q(\sO_X(s)/\sO_X)\cong
  \sO_X(s)/\sO_X$, and $\kappa_q(\sO_X(f-e_1)/\sO_X)\cong \sO_X(f-e_1)/\sO_X$.
\end{lem}

\begin{proof}
  For the first claim for $l=0$, we note the short exact sequences
  \[
  0\to \sO_X(-s-(d-1)f)\to \sO_X(-s-df)^2\to \sO_X(-s-(d+1)f)\to 0
  \]
  for all $d$.  Since the claim holds for $d\in \{1,2\}$ and the
  morphisms in the sequence are canonical, the claim holds for all $d$.
  Similarly, we have a canonical short exact sequence
  \[
  0\to \sO_X(-s-df+e_1+\cdots+e_{l-1})\to \sO_X(-s-df+e_1+\cdots+e_l)
  \to \sO_X(e_l)/\sO_X
  \to 0,
  \]
  and thus the claim holds for all $l$ by induction.

  For $-2s-2f$, consider the short exact sequence
  \[
  0\to \sO_X(-2s-2f)\to \sO_X(-s-f)\oplus \sO_X(-s-2f)\to \sO_X(-f)\to 0.
  \]
  The map $\sO_X(-s-f)\oplus \sO_X(-s-2f)\to \sO_X(-f)$ is in the unique
  dense orbit of such maps under the action of $\Aut(\sO_X(-s-f)\oplus
  \sO_X(-s-2f))$, and thus this will remain true after applying $\kappa_q$.
  The extension to all $l$ is as above.

  Finally, we note that we have short exact sequences
  \[
  0\to \sO_X(-2s-2f)\to \sO_X(-s-2f)\to \sO_X(s)/\sO_X\to 0
  \]
  and
  \[
  0\to \sO_X(-s-2f+e_1)\to \sO_X(-s-f)\to \sO_X(f-e_1)/\sO_X\to 0,
  \]
  with the first morphism begin unique in each case; this gives the final
  claims.
\end{proof}

\begin{prop}
  If $w$ is an admissible element of $W(E_{m+1})$ on $X_{\rho;q;C}$ (so
  also on $X_{\rho;1;C}$), then $\kappa_q\circ w = w\circ \kappa_q$.
\end{prop}

\begin{proof}
  We observe that the Lemma allows us to
  compute $\kappa_q$ using the exceptional collections
  \[
  \sO_X(e_m)/\sO_X,
  \sO_X(e_{m-1})/\sO_X,
  \dots
  \sO_X(e_1)/\sO_X,
  \sO_X(s)/\sO_X,
  \sO_X(-2s-2f),
  \sO_X(-s-f)
  \]
  and
  \[
  \sO_X(e_m)/\sO_X,
  \sO_X(e_{m-1})/\sO_X,
  \dots
  \sO_X(f-e_1)/\sO_X,
  \sO_X(-s-2f+e_1),
  \sO_X(-s-f+e_1),
  \sO_X(-f).
  \]
  Each admissible simple reflection of $W(E_{m+1})$ acts by swapping two
  adjacent orthogonal elements of one of these two exceptional collections,
  and thus commutes with $\kappa_q$.
\end{proof}

Note that $R\alpha_{m*},L\alpha_m^*$ preserve the condition
that $R\Gamma=0$, and thus induce functors between the respective
categories $\cN_{\rho;q;C}$.

\begin{cor}
  One has natural isomorphisms
  \begin{align}
  \kappa_q R\alpha_{m*}\cong R\alpha_{m*}\kappa_q,\notag\\
  \kappa_q L\alpha_{m}^*\cong L\alpha_{m}^*\kappa_q,\notag\\
  \kappa_q L\alpha_{m}^!\cong L\alpha_{m}^!\kappa_q,\notag
  \end{align}
  as well as
  \[
  \kappa_q R\ad\cong R\ad \kappa_q.
  \]
\end{cor}

\begin{proof}
  Indeed, in the first two cases, both functors have the same action on the
  elements of the appropriate exceptional collection.  (For $m=0$, use the
  $\P^2$-based exceptional collection instead.)  For the third case, use
  the distinguished triangles
  \[
  R\Hom(\sO_{e_m}(-1),L\alpha_m^*\kappa_q M)\otimes^{\dL}_k \sO_{e_m}(-1)\to
  L\alpha_m^*\kappa_q M\to L\alpha_m^!\kappa_q M
  \to
  \]
  and
  \[
  R\Hom(\sO_{e_m}(-1),L\alpha_m^* M)\otimes^{\dL}_k \sO_{e_m}(-1)\to
  \kappa_q L\alpha_m^*M\to \kappa_q L\alpha_m^! M
  \to
  \]
  together with the fact that
  \[
  R\Hom(\sO_{e_m}(-1),L\alpha_m^* M)
  \cong
  R\Hom(\sO_{e_m}(-1),\kappa_q L\alpha_m^*M)
  \]
  to obtain compatible isomorphisms between the first two terms, making the
  third terms isomorphic.

  Finally, for $R\ad$, we note that the Lemma directly tells us that
  $\kappa_q$ and $R\ad$ commute on the line bundles in the exceptional
  collection.  For the sheaves $\sO_X(e_l)/\sO_X$, we have
  \[
  \sO_X(e_l)/\sO_X
  \cong
  L\alpha_m^*\cdots L\alpha_{l+1}^* \sO_{e_l}(-1),
  \]
  and thus
  \[
    R\ad(\sO_X(e_l)/\sO_X)
    \cong
    L\alpha_m^!\cdots L\alpha_{l+1}^! \sO_{e_l}(-1),
  \]
  and this is preserved by $\kappa_q$.
\end{proof}

Since $\cN_{\rho;q;C}$ is compactly generated, there is a right adjoint
$\Phi_{\rho;q;C}$ to the functor $|^{\dL}_C$, and we immediately have the
following.

\begin{lem}
  We have a natural isomorphism $\kappa_q \Phi_{\rho;q;C} M
  \cong \Phi_{\rho;1;C} (M\otimes q)$.
\end{lem}

Now, define a functor $\bar\theta_{\cN}:\cN_{\rho;q;C}\to D^b\coh
X_{\rho;q;C}$ via the distinguished triangle
\[
\bar\theta_{\cN} M\to M\to \Phi_{\rho;q;C} M|^{\dL}_C\to
\]
and let $T_{\cN}$ denote the natural transformation $\bar\theta_{\cN}\to
\text{id}$.  (Here we are using the dg enhancement to allow us to define
the cone functor.)  Let us also define $\theta_{\cN}\cong
\bar\theta_{\cN}\otimes_k H^0_C(\omega_C)$.

\begin{prop}
  We have $\bar\theta_{\cN}\kappa_q\cong \kappa_q\bar\theta_{\cN}$ and
  $T_{\cN}\kappa_q\cong \kappa_q T_{\cN}$.  Moreover, $\cN_{\rho;q;C}$
  satisfies Serre duality in the form of functorial isomorphisms
  \[
  \Ext^i(M,N) \cong \Ext^{2-i}(N,\theta_{\cN} M).
  \]
\end{prop}

\begin{proof}
  The first two claims follow immediately from the Lemma.  For the
  remaining claim, note that we have a distinguished triangle
  \[
  R\Hom(N,\bar\theta_{\cN} M)\to R\Hom(N,M)\to R\Hom(N,\Phi_{\rho;q;C} M|^{\dL}_C)\to
  \]
  which by adjunction is naturally equivalent to
  \[
  R\Hom(N,\bar\theta_{\cN} M)\to R\Hom(N,M)\to R\Hom_C(N|^{\dL}_C,M|^{\dL}_C)\to
  \]
  and thus to
  \[
  R\Hom(N,\bar\theta_{\cN} M)\to R\Hom(N,M)\to R\Hom(N,M|^{\dL}_C)\to
  \]
  We thus have a natural isomorphism
  \[
  R\Hom(N,\bar\theta_{\cN} M)\cong R\Hom(N,\bar\theta M)
  \]
  and thus the claim follows from Serre duality on $X_{\rho;q;C}$.
\end{proof}

\medskip

We may now finish the proof of Theorem \ref{thm:spl_is_poisson}.  If $M$ is
an object in $\cN_{\rho;q;C}\cap \coh X_{\rho;q;C}$ such that $\kappa_q
M\in \coh X_{\rho;1;C}$, or more precisely a formally universal flat family
of such objects, then we may compute the corresponding pairing on the
cotangent sheaf in either of $X_{\rho;q;C}$ or $X_{\rho;1;C}$.  But the two
computations are the same: we are simply replacing $\theta$ by $\theta_\cN$
in the definition of the biderivation (which has no effect on the $\Ext^1$)
and using Serre duality on $\cN_{\rho;q;C}$ to define the pairing.

Thus to prove that the biderivation gives a Poisson structure, we simply
need to show that we can arrange for our object $V\in \cN_{\rho;q;C}\cap \coh
X_{\rho;q;C}$ to correspond to a sheaf in $X_{\rho;1;C}$.

\begin{lem}\label{lem:kappa_q_is_sheaf}
  Let $M$ be a sheaf on $\coh X_{\rho;q;C}$, and define a sequence of
  objects $M_0$,\dots,$M_m$ by $M_m=M$, $M_{l-1}=R\alpha_{l*}M_l$.
  Suppose that
  \begin{itemize}
    \item[(1)] $R\Hom(\sO_X,M)$, $R\Hom(\sO_X(f),M)$, $R\Hom(\sO_X(s+f),M)$, and
      $R\Hom(\sO_X(s+2f),M)$ are supported in degree 0.
    \item[(2)] Each $M_l$ is an $\alpha^*_l$-acyclic sheaf, and the
      morphisms $L\alpha^*_lM_{l-1}\to M_l$ are surjective.
  \end{itemize}
  Finally, let $V\in \cN_{\rho;q;C}$ be the kernel in the short exact
  sequence
  \[
  0\to V\to \Gamma(M)\otimes \sO_X\to M\to 0.
  \]
  Then $\kappa_q V$ is a sheaf of homological dimension 1 on $X_{\rho;1;C}$.
\end{lem}

\begin{proof}
  Let $V_m=V$, $V_{l-1}=R\alpha_{(l-1)*}V_l$.  Repeatedly applying
  $R\alpha_*$ tells us that $V_l$ is a sheaf for each $0\le l\le m$ and
  fits into a short exact sequence
  \[
  0\to V_l\to \Gamma(M)\otimes \sO_X\to M_l\to 0
  \]
  Suppose we know that $\kappa_q V_{l-1}$ is a sheaf of homological
  dimension 1.  Then it is, in particular, $\alpha_l^*$-acyclic, so that
  \[
  \kappa_q L\alpha_l^* R\alpha_{l*} V_l
  =
  L\alpha_l^* \kappa_q R \alpha_{l*} V_l
  \]
  is a sheaf.  In particular, $L\alpha_l^*R\alpha_{l*}$ is exact on the
  above short exact sequence, and we may apply the snake lemma to the
  resulting map of short exact sequences.  Since
  $L\alpha_l^*R\alpha_{l*}\to\id$ is an isomorphism in the middle, we
  conclude that we have a short exact sequence of the form
  \[
  0\to L\alpha_l^*R\alpha_{l*} V_l\to V_l\to \sO_{e_l}(-1)^n\to 0.
  \]
  This is a short exact sequence of sheaves in $\cN_{\rho;q;C}$, so we may apply
  $\kappa_q$ to obtain a distinguished triangle of such sheaves; since the
  kernel and cokernel are sheaves, so is $\kappa_q V_l$.  Moreover, both
  kernel and cokernel have homological dimension 1, and thus the same holds
  for $\kappa_q V_l$.

  Thus, by induction, we reduce to considering the case $l=0$.  Now, $M$ is
  represented in $D^b\coh X_{\eta,x_0;q;C}$ by a bounded complex in which
  each term is a sum of sheaves $\sO_X$, $\sO_X(-f)$, $\sO_X(-s-f)$,
  $\sO_X(-s-2f)$, since that exceptional collection generates the category
  and there are no extensions between the sheaves.  We may further arrange
  for there to be no nonzero maps in the complex between line bundles of
  the same divisor class.  Taking (derived) homomorphisms from $\sO_X$,
  $\sO_X(f)$, $\sO_X(s)$, $\sO_X(s+f)$ allows us to compute the
  dimensions; since $M_0$ is acyclic in each case, we find that $M_0$ has a
  presentation of the form
  \[
  0\to \sO_X(-s-2f)^{n_1}\to \sO_X(-s-f)^{n_2}\oplus \sO_X(-f)^{n_3}\to
  \sO_X^{n_4}\to M_0\to 0.
  \]
  This in turn implies that $V_0$ has a presentation
  \[
  0\to \sO_X(-s-2f)^{n_1}\to \sO_X(-s-f)^{n_2}\oplus \sO_X(-f)^{n_3}\to V_0\to 0.
  \]
  Since $|^{\dL}_C$ has homological dimension 1, the map $\theta V_0\to
  V_0$ is injective, which implies that the presentation for $V_0$
  restricts to a presentation for $V_0|_C$.  In particular, we find that
  the map of sums of line bundles is injective on $C$.  Applying $\kappa_q$
  gives a morphism of vector bundles which is still injective on $C$, and
  is thus injective on $X_{\eta,x_0;1;C}$.  It follows that $\kappa_q V_0$
  is a sheaf of homological dimension 1 as required.
\end{proof}

At this point, it is straightforward to complete the proof of Theorem
\ref{thm:spl_is_poisson}.  The only thing we need to observe is that for
any 2-dimensional sheaf $M$, there is a divisor class $D$ such that $M(D)$
satisfies the above conditions, as well as $H^i(\theta M(D))=0$ for $i>0$.
Indeed, each requirement holds for $D$ sufficiently ample, and thus there
is a nonempty set of $D$ satisfying all of the conditions.
\qed

\bigskip

We next would like to understand how the various Poisson morphisms of
\cite{poisson} extend to the present setting.  It turns out that much of
that extension is entirely straightforward, with only a few places that
require different arguments or hypotheses.  The most significant of the
latter is the requirement that various sheaves have homological dimension
1.  This, of course, makes no sense in our setting, since we do not have a
well-behaved notion of ``locally free''.  However, from
\cite[Prop.~4.2]{poisson}, we see that a sheaf on a commutative surface has
homological dimension 1 iff it contains no $0$-dimensional subsheaf.  The
latter notion, of course, extends immediately to the noncommutative
setting; we thus call a sheaf ``pointless'' if it has no $0$-dimensional
subsheaf.  Note that there are some instances in which a sheaf on
$X_{\rho;q;C}$ is inarguably of homological dimension 1 by any reasonable
definition; luckily, those sheaves are indeed pointless.

\begin{lem}
  The cokernel of an injective morphism between two sums of
  line bundles is pointless.
\end{lem}

\begin{proof}
  This follows immediately via the long exact sequence from the fact that
  for any $0$-dimensional sheaf $M$ and divisor class $D$,
  $\Ext^i(M,\sO_X(D))=0$ for $i<2$.
\end{proof}

\begin{lem}
  A simple coherent sheaf is either pointless or $0$-dimensional.
\end{lem}

\begin{proof}
  Let $M$ be a simple coherent sheaf which is neither pointless nor
  $0$-dimensional, and let $N$ be the sum of all $0$-dimensional subsheaves
  of $M$, making $N$ the maximal $0$-dimensional subsheaf of $M$.  Then $M$
  is a nontrivial extension of $M/N$ by $N$, so that $\Ext^1(M/N,N)\ne 0$.
  Since $N$ is $0$-dimensional, we have $\chi(M/N,N)=\rank(M/N)\chi(N)\ge
  0$, and may thus conclude that
  \[
  \dim\Hom(M/N,N) + \dim\Hom(\theta^{-1}N,M/N)
  =
  \dim\Hom(M/N,N)+\dim\Ext^2(M/N,N)>0.
  \]
  A nonzero morphism $M/N\to N$ would contradict simplicity of $M$, while a
  nonzero morphism $\theta^{-1}N\to M/N$ would contradict maximality of
  $N$.
\end{proof}

One helpful fact to establish is that sheaves on $-1$-curves are rigid.
More precisely, we have the following.  For each integer $d$, define
\[
\sO_{e_m}(d) = \sO_X(-de_m)/\sO_X(-(d-1)e_m),
\]
which is pointless by the Lemma.

\begin{lem} 
  If $M$ is a pointless coherent sheaf with numerical invariants
  $(0,e_m,d)$, then
  \[
  M\cong \sO_{e_m}(d-1).
  \]
\end{lem}

\begin{proof}
  Since $e_m$ is not a sum of two nonzero effective divisors (indeed, there
  are ample divisors $D$ such that $D\cdot e_m=1$), any map between sheaves
  of Chern class $e_m$ has either $0$-dimensional image or $0$-dimensional
  cokernel.  In particular, if both sheaves are pointless, then the
  morphism must be injective, since neither the image nor the kernel can be
  0-dimensional.
  
  Since $\chi(M,\sO_{e_m}(d-1))=1$, we conclude that at least one of the
  spaces $\Hom(M,\sO_{e_m}(d-1))$ or $\Ext^2(M,\sO_{e_m}(d-1))\cong
  \Hom(\sO_{e_m}(d-1),\theta M)^*$ is nonzero.  In the first case, we have
  an injective morphism $M\to \sO_{e_m}(d-1)$, and $K_0$ considerations
  tell us that the cokernel is trivial, and thus the map is an isomorphism.
  In the second case, the morphism must still be injective, but now the
  cokernel would be a $0$-dimensional sheaf with Euler characteristic $-1$.
\end{proof}

\begin{rem}
  In particular, $\theta \sO_{e_m}(d)\cong \sO_{e_m}(d-1)$, and any
  subsheaf of $\sO_{e_m}(d)$ has the form $\sO_{e_m}(d')$ for $d'<d$.
\end{rem}

We may use this to give an example of exotic $0$-dimensional simple
sheaves.

\begin{lem}\label{lem:exotic_zero_dim}
  If $d\le |\langle q\rangle|$, then the cokernel of $T^d:\bar\theta^d
  \sO_{e_m}(k)\to \sO_{e_m}(k)$ is a simple sheaf for any $k$.
\end{lem}

\begin{proof}
  Let $M_{d,k}$ be the given cokernel.  It is clear that $M_{d,k}|_C\cong
  M_{1,k}$, and since $T^d$ is injective for all $d$, we find $M_{1,k}\cong 
  \sO_{q^{k+1}x_m}$.  Moreover, $M_{d,k}$ has a natural filtration with
  quotients $M_{1,k}$, $M_{1,k-1}$,\dots, $M_{1,k-d+1}$.

  Let $d$ be minimal such that $M_{d,k}$ is not simple (note that since
  $\theta M_{d,k}\cong M_{d,k+1}$, simplicity is independent of $k$), and
  let $\phi$ be a nonscalar endomorphism of $M_{d,k}$.  Consider the
  restriction of $\phi$ to the subsheaf $M_{d-1,k-1}$.  If
  $\phi(M_{d-1,k-1})\not\subset M_{d-1,k-1}$, then we can compose
  with the natural map $M_{d,k}\to M_{d,k}|_C$ to obtain a map
  $M_{d-1,k-1}\to M_{1,k}$.  Since $d\le |\langle q\rangle|$, none of the
  subquotients of $M_{d-1,k-1}$ map to $M_{1,k}$, giving a contradiction.

  It follows that $\phi(M_{d-1,k-1})\subset M_{d-1,k-1}$.  In that case,
  induction tells us that the restriction must be scalar, and thus
  subtracting that scalar gives a nonzero endomorphism that annihilates
  $M_{d-1,k-1}$, and thus a morphism $M_{1,k}\to M_{d,k}$.  Since $T$
  annihilates the domain, it must also annihilate the image, and thus this
  morphism actually maps $M_{1,k}\to M_{1,k-d+1}$.  But this cannot happen
  unless $d-1$ is a multiple of $|\langle q\rangle|$.
\end{proof}

\begin{rem} Note that if $d=|\langle q\rangle|$, then there are other
  examples, coming from structure sheaves of points on the center.
\end{rem}

We also need to know a bit about how these sheaves interact with
$R\alpha_{m*}$, $L\alpha_m^*$, and $L\alpha_m^!$.

\begin{lem}
  We have
  \begin{align}
    R\alpha_{m*}\sO_{e_m}&\cong \sO_{qx_m},\notag\\
    R\alpha_{m*}\sO_{e_m}(-1)&=0\notag\\
    R\alpha_{m*}\sO_{e_m}(-2)&\cong\sO_{x_m}[-1]\notag
  \end{align}
  In addition, for $y\in C$, $L\alpha^*_m \sO_y\cong \sO_y$ if $y\ne qx_m$,
  for which
  \begin{align}
    \alpha^*_m \sO_{qx_m} &\cong \sO_{e_m}\notag\\
    L_1\alpha^*_m \sO_{qx_m} &\cong \sO_{e_m}(-1).\notag
  \end{align}
  Similarly, $L\alpha^!_m\sO_y\cong \sO_y$ for $y\ne x_m$, and
  \begin{align}
    \alpha^!_m \sO_{x_m} &\cong \sO_{e_m}(-1)\notag\\
    L_1\alpha^!_m \sO_{x_m} &\cong \sO_{e_m}(-2).\notag
  \end{align}
\end{lem}

\begin{proof}
  That $R\alpha_{m*}\sO_{e_m}(-1)=0$ was established in
  \cite{VandenBerghM:1998}.  The other two claims for $R\alpha_{m*}$ then
  follow from the short exact sequences
  \begin{align}
  0\to \sO_{e_m}(-2)&\to \sO_{e_m}(-1)\to \sO_{x_m}\to 0\notag\\
  0\to \sO_{e_m}(-1)&\to \sO_{e_m}\to \sO_{qx_m}\to 0\notag
  \end{align}

  The formula for $L\alpha_m^*\sO_y$ is
  \cite[Prop.~8.3.2]{VandenBerghM:1998}, and
  the formula for $L\alpha_m^!$ is then straightforward.
\end{proof}

\smallskip

The discussion of \cite{poisson} could be applied to any birational
morphism, but since most of the arguments involved restricting to monoidal
transformations, we will mainly consider that case; that is, how moduli
spaces on $X_m$ and $X_{m-1}$ are related.

\begin{lem}
  If $M$ is a pointless coherent sheaf on $X_{m-1}$, then $L\alpha_m^*$ and
  $L\alpha_m^!$ are pointless $\alpha_{m*}$-acyclic sheaves with direct
  image naturally isomorphic to $M$.
\end{lem}

\begin{proof}
  We certainly have $M\cong R\alpha_{m*}L\alpha_m^*M$ by the natural map,
  and thus the natural spectral sequence gives:
  \[
  \alpha_{m*}L_1\alpha_m^*M=R^1\alpha_{m*}\alpha_m^*M=0
  \]
  along with a short exact sequence
  \[
  0\to R^1\alpha_{m*}L_1\alpha_m^*M\to M\to \alpha_{m*}\alpha_m^*M\to 0
  \]
  Since the image of $R^1\alpha_{m*}$ is always $0$-dimensional
  (\cite[Prop.~6.5.2]{VandenBerghM:1998}), pointlessness of $M$ implies that
  $R\alpha_{m*}L_1\alpha_m^*M=0$, so that $L_1\alpha_m^*M\cong \sO_e(-1)^n$
  for some $n$.  But we then have, as above,
  \[
  \Hom(\sO_{e_m}(-1),L_1\alpha^*M)
  \cong
  \Hom(\sO_{qx_m},M)
  =
  0
  \]
  since $M$ is pointless.  It follows that $L\alpha_m^*M$ is a sheaf with
  $\alpha_{m*}\alpha_m^*M\cong M$.  The direct image of a $0$-dimensional
  subsheaf of $\alpha_m^*M$ thus gives a $0$-dimensional subsheaf of $M$,
  implying that $L\alpha_m^*M$ is pointless.

  Then claim for $L\alpha_m^!M$ follows immediately, since $\theta$
  takes pointless sheaves to pointless sheaves.
\end{proof}

\begin{lem}
  For any sheaf $M$ on $X_{m-1}$, there is a natural isomorphism
  $(L\alpha_m^*M)|^{\dL}_C\cong M|^{\dL}_C$.
\end{lem}

\begin{proof}
  Since the derived category of $C$ is generated by the line bundles
  $\sO_{X_{m-1}}(-D)|_C$ where $D$ ranges over divisor classes on
  $X_{m-1}$, and $\sO_{X_{m-1}}(-D)|_C\cong \sO_{X_m}(-D)|_C$, it suffices
  to exhibit a natural isomorphism
  \[
  R\Hom_C(\sO_{X_m}(-D)|_C,(L\alpha_m^*M)|^{\dL}_C)
  \cong
  R\Hom_C(\sO_{X_{m-1}}(-D)|_C,M|^{\dL}_C).
  \]
  We in fact have
  \begin{align}
  R\Hom_C(\sO_X(-D)|_C,(L\alpha_m^*M)|^{\dL}_C)
  &\cong
  R\Hom_{X_m}(\sO_X(-D),(L\alpha_m^*M)|^{\dL}_C)\notag\\
  &\cong
  R\Hom_{X_m}(\alpha_m^*\sO_X(-D),(L\alpha_m^*M)|^{\dL}_C)\notag\\
  &\cong
  R\Hom_{X_{m-1}}(\sO_X(-D),R\alpha_{m*}((L\alpha_m^*M)|^{\dL}_C)),
  \end{align}
  so that the claim reduces to showing
  \[
  R\alpha_{m*}((L\alpha_m^*M)|^{\dL}_C))
  \cong
  M|^{\dL}_C.
  \]
  We have a distinguished triangle
  \[
  \theta L\alpha_m^*M\to L\alpha_m^*M\to L\alpha_m^*M|^{\dL}_C\to,
  \]
  or equivalently
  \[
  L\alpha_m^! \theta M\to L\alpha_m^*M\to (L\alpha_m^*M)|^{\dL}_C\to.
  \]
  Applying $R\alpha_{m*}$ gives a distinguished triangle
  \[
  \theta M\to M\to R\alpha_{m*}((L\alpha_m^*M)|^{\dL}_C)\to
  \]
  from which the claim follows.
\end{proof}

\begin{rem}
  Note that this can fail as stated in degenerate cases, as blowing up may
  not preserve the isomorphism class of the anticanonical curve.  There is
  no such difficulty with the statement
  $R\alpha_{m*}((L\alpha_m^*M)|^{\dL}_C)\cong M|^{\dL}_C$, however.
\end{rem}

Now, the natural isomorphism $R\alpha_{m*}L\alpha_m^*\to \id$ (the
inverse of the natural map coming from the adjunction) is transformed by
adjunction to a natural transformation $L\alpha_m^*\to L\alpha_m^!$.

\begin{defn}
  Let $M$ be a pointless coherent sheaf on $X_{m-1}$.  Then the ``minimal
  lift'' of $M$ is the image $\alpha_m^{*!}M$ of the morphism
  $\alpha_m^*M\to \alpha_m^!M$.
\end{defn}

\begin{rem}
  As in \cite{poisson}, this extends readily to a functor lifting pointless
  coherent sheaves from $X_k$ to $X_m$, either by taking
  $\alpha_m^{*!}\alpha_{m-1}^{*!}\cdots\alpha_{k+1}^{*!}M$ or by taking the
  image of the natural map
  \[
  \alpha_m^{*}\alpha_{m-1}^{*}\cdots\alpha_{k+1}^{*}M
  \to
  \alpha_m^{!}\alpha_{m-1}^{!}\cdots\alpha_{k+1}^{!}M.
  \]
\end{rem}

The name ``minimal lift'' follows from the fact (with the same proof as in
\cite{poisson}) that any coherent $\alpha_{m*}$-acyclic sheaf on $X_m$ with
direct image $M$ has a subquotient isomorphic to $\alpha_m^{*!}M$, such
that the other factors are isomorphic to powers of $\sO_{e_m}(-1)$.

Much of the discussion in \cite[\S 5]{poisson} of the category of sheaves
annihilated by $R\alpha_{m*}$ or more generally $R\alpha_{(k+1)*}\cdots
R\alpha_{(m-1)*}R\alpha_{m*}$ carries over with little difficulty to the
noncommutative setting.  One notable exception is Proposition 5.4 op. cit.,
for which the hypotheses do not make sense (they refer to the locus where a
morphism vanishes); this requires some modifications to later arguments.
In Lemma 5.9 op. cit., ``homological dimension $\le$ 1'' and ``homological
dimension 2'' need to be replaced by ``pointless'' and ``not pointless''
respectively.  There, we need to use the fact that a $0$-dimensional
subsheaf of the (higher) direct image of a pointless sheaf necessarily has
the form $\sO_p$ for $p\in C$.  We record this here for future reference,
including the analogue of Corollary 5.10 op. cit.

\begin{lem}\label{lem:poisson5.10}
  Let $M$ be a pointless coherent sheaf on $X_m$.  If $M$ is not
  $\alpha_{m*}$-acyclic, then $\Hom(M,\sO_{e_m}(-2))\ne 0$, while if
  $\alpha_{m*}M$ is not pointless, then $\Hom(\sO_{e_m},M)\ne 0$.
  Moreover, if $\Hom(\sO_{e_m}(-1),M)=\Hom(M,\sO_{e_m}(-1))=0$ then
  $R\alpha_{m*}M$ is a pointless sheaf, with $M\cong
  \alpha_m^{*!}\alpha_{m*}M$.
\end{lem}

In addition, there are natural subtleties
arising around $-2$-curves; one finds that the simple objects in the kernel
of $R\alpha_{(k+1)*}\cdots R\alpha_{(m-1)*}R\alpha_{m*}$ have the form
$\sO_X(e_i)/\sO_X(e_j)$ if $j$ is the minimal index such that $x_i=x_j$,
and $\sO_X(e_i)/\sO_X$ if there is no such index.  The key observation one
must use to rule out other subsheaves is that a sheaf with Chern class
$e_i-e_j$ has degree 0 derived restriction to $C$, and thus to be pointless
must have {\em trivial} derived restriction.  Since this is a subcategory
of $\cN_{\rho;q;C}$, we may use the functor $\kappa_q$ to move everything
to $\cN_{\rho;1;C}$; this identifies the respective sets of simple objects,
and thus identifies the actual abelian categories.

\medskip

Continuing on with minimal lifts, one has the following, essentially
combining the non-sheaf-$\Hom$ version of \cite[Lem.~6.1]{poisson} with
\cite[Prop.~6.4]{poisson}.

\begin{lem}
  Suppose $M$ and $N$ are pointless coherent sheaves on $X_{m-1}$. Then
  there is an isomorphism
  \[
  \Hom(\alpha_m^{*!}M,\alpha_m^{*!}N)\cong \Hom(M,N)
  \]
  and an exact sequence
  \begin{align}
  0\to \Ext^1(\alpha_m^{*!}M,\alpha^{*!}N)
  &\to \Ext^1(M,N)\notag\\
  &\to \Hom_k(\Hom(N,\sO_{x_m}),\Ext^1(M,\sO_{x_m}))
  \to \Ext^2(\alpha_m^{*!}M,\alpha^{*!}N)
  \to
  0.\notag
  \end{align}
\end{lem}

\begin{proof}
  We have exact sequences
  \begin{align}
  0&\to E_1\to \alpha_m^*M\to \alpha^{*!}M\to 0\notag\\
  0&\to \alpha^{*!}N\to \alpha^!N\to E_2\to 0\notag
  \end{align}
  where $E_1$ and $E_2$ are both powers of $\sO_{e_m}(-1)$.  We find
  \[
  R\Hom(E_1,\alpha^{*!}N)\cong R\Hom(E_1,E_2)[-1]\cong 
  R\Hom(\alpha^{*!}M,E_2),
  \]
  and thus it remains only to determine $E_1$, $E_2$, or equivalently the
  spaces $\Hom(\sO_{e_m}(-1),\alpha_m^*M)$ and
  $\Hom(\alpha_m^!M,\sO_{e_m}(-1))$.  For the first, we have
  \[
  \Hom(\sO_{e_m}(-1),\alpha_m^*M)
  \cong
  \Hom(\alpha_m^* M,\sO_{e_m}(-2))^*
  \cong
  \Ext^1(M,\sO_{x_m})^*
  \]
  and a similar argument gives
  \[
  \Hom(\alpha_m^!M,\sO_{e_m}(-1)) \cong \Hom(M,\sO_{x_m}).
  \]
\end{proof}

\begin{cor}\label{cor:pseudo_twist}
  Let $M$ be a pointless coherent sheaf on $X_{m-1}$.  Then the sheaves
  $\theta^{\pm 1}\alpha_m^{*!}M$ are $\alpha_{m*}$-acyclic, with pointless
  direct image.  Moreover, there are short exact sequences
  \[
  0\to M\to \theta^{-1}\alpha_{m*} \theta \alpha_m^{*!}M\to \sO_{qx_m}\otimes_k \Ext^1(\sO_{qx_m},M)\to 0
  \]
  and
  \[
  0\to \theta \alpha_{m*} \theta^{-1} \alpha_m^{*!}M\to M\to \Hom_k(\Hom(M,\sO_{x_m}),\sO_{x_m})\to 0.
  \]
\end{cor}

\begin{rem}
  Note also that one can read off the dimensions of $\Ext^1(\sO_{qx_m},M)$
  and $\Hom(M,\sO_{x_m})$ from the numerical invariants of
  $\alpha_m^{*!}M$.
\end{rem}

There is a subtlety in extending the proof of Theorem 6.6 of
\cite{poisson}, as it involves comparing sheaves which are twists by
different line bundles.  We can replace this by an inductive argument as
follows.

\begin{lem}
  Let $M$, $N$ be pointless sheaves on $X_m$ and $X_{m-1}$.
  If there is a short exact sequence
  \[
  0\to \alpha_m^{*!}N\to M\to \sO_x^l\to 0
  \]
  for some $x\in C$, then $R\alpha_{m*}M$ is a pointless sheaf.
  Similarly, if there is a short exact sequence
  \[
  0\to M\to \alpha_m^{*!}N\to \sO_x^l\to 0,
  \]
  then $R\alpha_{m*}M$ is a pointless sheaf.
\end{lem}

\begin{proof}
  For the first claim, we note that $R\alpha_{m*}M$ is an extension of
  sheaves, so a sheaf, and it remains only to show that it is pointless.
  We may now observe that since $\Hom(\sO_x,\sO_{e_m}(-1))=0$, the
  pushforward exact sequence
  \[
  0\to \alpha_m^!N\to M'\to \sO_x^l\to 0
  \]
  is non-split, and thus $M'$ is an extension of a power of $\sO_{e_m}(-1)$
  by $M$, so remains pointless and has the same direct image as $M$.  The
  claim then follows from the isomorphism
  \[
  \Ext^1(\sO_x,\alpha_m^!N)\cong \Ext^1(R\alpha_{m*}\sO_x,N)\cong
  \Ext^1(\sO_x,N),
  \]
  which implies that the direct image of such a nonsplit extension remains
  nonsplit, making $\alpha_{m*}M=\alpha_{m*}M'$ pointless.

  The other claim is analogous, where now the nontrivial task is to prove
  that $R\alpha_{m*}M$ is a sheaf.
\end{proof}

Given the Lemma, the analogue of \cite[Thm.~6.6]{poisson} is
straightforward.

\begin{thm}\cite{poisson}
  Let $M$ be a pointless coherent sheaf on $X_k$.  Then the direct images
  in $D^b\coh X_k$ of $\theta^{\mp 1}\alpha_{m*} \theta^{\pm 1}
  \alpha_m^{*!}\cdots\alpha_{k+1}^{*!} M$ are pointless coherent sheaves.
\end{thm}
  
Another important observation is that (per \cite[Lem.~6.8]{poisson}, with
the same proof) given any pure 1-dimensional sheaf on $X_m$ on which $T$ is
injective, we can always choose further blowups in such a way that the
minimal lift to some $X_{m+k}$ will have $T$ an isomorphism.  The argument
lifts directly, with some simplification due to the fact that $C$ is
smooth, and thus ``pseudo-twists'' are not required.  This is important for
our interpretations and generalizations of Painlev\'e below; we will obtain
Painlev\'e-like integrable systems via twisting morphisms between moduli
spaces of pure 1-dimensional sheaves disjoint from $C$, and use the above
results to see how the operations act in terms of difference equations.

In general, two difference equations are related by a gauge equivalence
(a.k.a. isomonodromy deformation) iff their corresponding sheaves are {\em
  comparable}, in the sense that there is a common subsheaf with the same
Chern class.  Part of the justification of reducing to the ``disjoint from
$C$'' case is that two equations are gauge equivalent iff their disjoint
from $C$ versions are equivalent via {\em canonical} gauge transformations
(i.e., those coming from twists).

\begin{prop}\label{prop:comparable_implies_pseudo-twist}
  Suppose $q$ is non-torsion and $M$, $N$ are comparable pure
  $1$-dimensional sheaves on $X_m$ which are both transverse to $C$.  Then
  $M$ and $N$ have a common pseudo-twist.
\end{prop}

\begin{proof}
  We first observe that it suffices to consider the case $M\subset N$, as
  otherwise we may instead consider the pair $(M\cap N)^2\subset M\oplus N$
  where $M\cap N$ is the maximal common subsheaf.  If $M$ and $N$ are
  disjoint from $C$, then their Euler characteristics must agree by
  Corollary \ref{cor:pseudo_bounded} below, and thus $M=N$.  Otherwise,
  choose a point $x_{m+1}$ in the support of $N|_C$, and let $X_{m+1}$
  denote the blowup of $X_m$ in that point.  The snake lemma applied to
  \[
  \begin{CD}
      0@>>> \theta \alpha_{(m+1)*} \theta^{-1} \alpha_{m+1}^{*!}M@>>> M@>>> \Hom_k(\Hom(M,\sO_{x_{m+1}}),\sO_{x_{m+1}})@>>> 0\\
     @. @VVV @VVV @VVV @.\\
      0@>>> \theta \alpha_{(m+1)*} \theta^{-1} \alpha_{m+1}^{*!}N@>>> N@>>> \Hom_k(\Hom(N,\sO_{x_{m+1}}),\sO_{x_{m+1}})@>>> 0
  \end{CD}\notag
  \]
  implies that
  \[
  \theta \alpha_{{m+1}*} \theta^{-1} \alpha_{m+1}^{*!}M
  \subset
  \theta \alpha_{{m+1}*} \theta^{-1} \alpha_{m+1}^{*!}N
  \]
  and thus we may feel free to perform any finite sequence of such
  pseudo-twists.  In particular, for any $q^\Z$-orbit in $C$, we can perform a
  sequence of downwards pseudo-twists in such a way to ensure that only one
  point in that orbit is in the support of $M|_C\oplus N|_C$.

  Let $x_{m+1}$ be such a point.  Then 
  $ \Hom(N,\sO_{q^d x_{m+1}}) = 0$
  for $d\ne 0$ and thus an easy induction using the short exact sequence
  \[
  0\to \sO_{e_{m+1}}(d-1)\to \sO_{e_{m+1}}(d)\to \sO_{q^{d+1}x_{m+1}}\to 0
  \]
  implies that
  $
  \Hom(\alpha_{m+1}^* N,\sO_{e_{m+1}}(d))=0
  $
  for all $d$.  Since $\alpha_{m+1}^{*!}N$ is a quotient of
  $\alpha_{m+1}^*N$, we then also have
  $
  \Hom(\alpha_{m+1}^* N,\sO_{e_{m+1}}(d))=0
  $
  for all $d$.  Similarly, for $d\ne 0$,
  \[
  \Ext^1(\sO_{q^{d+1}x_{m+1}},M)
  \cong
  \Ext^1(M,\sO_{q^d x_{m+1}})^*
  \cong
  \Ext^1(M|_C,\sO_{q^d x_{m+1}})^*
  \cong
  \Hom(\sO_{q^d x_{m+1}},M|_C)
  =
  0
  \]
  and thus
  $
  \Hom(\sO_{e_{m+1}}(d),\alpha_{m+1}^! M)=0
  $
  for all $d$, implying
  $
  \Hom(\sO_{e_{m+1}}(d),\alpha_{m+1}^{*!} M)=0$.

  Now, consider the morphism
  $
  \phi\,{:}\,\alpha_{m+1}^{*!}M\to \alpha_{m+1}^{*!}N.
  $
  The direct image is injective with $0$-dimensional cokernel, and thus
  comparing the two spectral sequences for the direct image implies that
  $\alpha_{(m+1)*}\ker(\phi)=R^1\alpha_{(m+1)*}\coker(\phi)=0$, while
  $R^1\alpha_{(m+1)*}\ker(\phi)=\alpha_{(m+1)*}\coker(\phi)$ are
  $0$-dimensional.  It follows immediately that $c_1(\ker(\phi))$ and
  $c_1(\coker(\phi))$ are both multiples of $e_{m+1}$.  Since
  $\alpha_{m+1}^{*!}M$ has no such subsheaf and $\alpha_{m+1}^{*!}N$ has no
  such quotient, it follows that $\phi$ is itself injective with
  $0$-dimensional cokernel.  Since $\alpha_{m+1}^{*!}N$ is transverse to
  $C$ with $\chi((\alpha_{m+1}^{*!}N)|_C)<\chi(N|_C)$, the claim follows by
  induction on $\chi(N|_C)$.
\end{proof}

\begin{rem}
  It follows from the proof that if $M\subset N$ and the support of
  $M|_C\oplus N|_C$ hits every $q^\Z$-orbit at most once, then $M=N$, as
  this property is then preserved under taking minimal lifts.
\end{rem}

\begin{rem}
  If $C$ is singular, then $x_{m+1}$ may be fixed by the analogue of
  $q^\Z$; since $e_{m+1}$ is then an anticanonical component, the
  desired condition on $\alpha_{m+1}^{*!}M$ and $\alpha_{m+1}^{*!}N$ is
  simply that they both be transverse to $C$.
\end{rem}

\begin{rem}
  When $q$ is torsion, the claim fails at the base case of the induction:
  take $M$ to be the kernel of the map from $N$ to a generic
  $0$-dimensional sheaf coming from its support on the center of $X$.
\end{rem}

\medskip

Let $\Spl_{X_m,X_{m-1}}$ denote the subspace of $\Spl_{X_m}$ classifying
simple sheaves such that the derived direct image is a simple sheaf.  This
is an open subspace (by the same proof as \cite[Lem.~7.2]{poisson}; note
that $R^1\alpha_{m*}M$ is a sheaf on $C$, so it makes sense to consider its
support), and there is a natural morphism
$\alpha_{m*}:\Spl_{X_m,X_{m-1}}\to \Spl_{X_{m-1}}$, which certainly
surjects on the subspace of pointless sheaves, as it is split by
$\alpha_m^{*!}$.  (That $\alpha_{m*}$ preserves flatness in
$\alpha_{m*}$-acyclic families follows as in \cite[Lem.~7.1]{poisson})

\begin{thm}\cite[Thm.~7.4]{poisson}
  The morphism $\alpha_{m*}:\Spl_{X_m,X_{m-1}}\to \Spl_{X_{m-1}}$ is
  Poisson.
\end{thm}

\begin{proof}
  The only thing missing to make the proof work is an analogue of
  \cite[Lem.~7.3]{poisson}.  The argument there does not carry over, but we
  can fix this by realizing that the claim made there has an analogue in
  the derived category, namely that there is a commutative diagram
  \[
  \begin{CD}
    \theta L\alpha_m^*R\alpha_{m*}@>
    \theta L\alpha_m^* T
>> \theta
    L\alpha_m^*\theta^{-1}R\alpha_{m*}\\
    @VVV @AAA\\
    \theta@>T>> \id
  \end{CD}
  \]
  of triangulated functors, where the right map is
  \[
  \id\to L\alpha^!_mR\alpha_{m*}\cong \theta
  L\alpha^*_m\theta^{-1}R\alpha_{m*}.
  \]
  It suffices to verify this on generators of the derived category, so in
  particular on the sheaves $\sO_{e_m}(-1)$ and $\sO_X(D)$ with
  $D=X_{m-1}(D)$.  The first case is trivial, while the second simply
  states that the pullback of $T$ factors through $T$.

  In particular, we do not need $L\alpha_m^* \alpha_{m*}M$ or
  $L\alpha_m^!\alpha_{m*}M$ to be sheaves when applying the Lemma, making
  the proof work even in the case that $\alpha_{m*}M$ is not pointless.
\end{proof}

\begin{cor}
  The functors $\alpha_m^*$ and $\alpha_m^!$ induce Poisson morphisms
  $\Spl_{X_{m-1}}\to \Spl_{X_m,X_{m-1}}$.
\end{cor}

Let $\Spl_{X_{m-1},x_m,k}$ denote the locally closed subspace of
$\Spl_{X_{m-1}}$ consisting of pointless sheaves such that
$\dim\Ext^1(\sO_{qx_m},M)=k$.

\begin{thm}\cite{poisson}
  The algebraic space $\Spl_{X_{m-1},x_m,k}$ is a Poisson subspace of
  $\Spl_{X_{m-1}}$, and $\alpha_m^{*!}$ induces a Poisson morphism
  $\Spl_{X_{m-1},x_m,k}\to \Spl_{X_m}$.
\end{thm}

\begin{proof}
  That this is a Poisson subspace follows as in \cite[Lem.~7.6]{poisson}
  (noting that for simple sheaves, ``pointless'' is equivalent to ``not
  $0$-dimensional'', so is both open and closed), and $\alpha_m^{*!}$
  preserves flatness for families in $\Spl_{X_{m-1},x_m,k}$, so induces a
  morphism. That morphism is Poisson since its inverse, $\alpha_{m*}$, is
  Poisson.
\end{proof}

If $M$ is a pure $1$-dimensional sheaf, then we may define $M^D:=R^1\ad M$,
which will again be a pure $1$-dimensional sheaf. We then find the
following, just as in the commutative case.

\begin{prop}
  \cite[Prop.~7.11]{poisson}
  Let $M$ be a pure $1$-dimensional sheaf on $X_{m-1}$.  Then $\alpha_m^*M$,
  $\alpha_m^!M$, $\alpha_m^{*!}M$ are pure $1$-dimensional sheaves, and
  \begin{align}
    (\alpha_m^*M)^D&\cong \alpha_m^!(M^D)\notag\\
    (\alpha_m^{*!}M)^D&\cong \alpha_m^{*!}(M^D)\notag\\
    (\alpha_m^!M)^D&\cong \alpha_m^*(M^D).\notag
  \end{align}
\end{prop}

\begin{proof}
  We first note that if $N$ is $d$-dimensional, then $R^i\ad N=0$ for
  $i+d<2$, by considering $R\Hom(\sO_X(-D),R^i\ad N)$ for sufficiently
  ample $D$ and applying Serre duality.  Moreover, we find, since $R\ad
  R\ad\cong\id$ and $R\ad M$ is supported in degrees 1 and 2 for $M$
  $1$-dimensional, that $R^1\ad M$ is at most $1$-dimensional and $R^2\ad
  M$ is at most $0$-dimensional.  Moreover, the relevant spectral sequence
  indicates that $R^2\ad R^2 \ad M$ is a subsheaf of $M$, and thus $R^2\ad
  M=0$.  It then follows that for a $1$-dimensional sheaf $M$, $R^2\ad M=0$
  iff $M$ is pure, at which point we conclude that $M^D$ is pure
  $1$-dimensional iff $M$ is pure $1$-dimensional.  That $\alpha_m^!M$ is
  pure $1$-dimensional follows from the fact that (by adjunction) the
  direct image of a $0$-dimensional subsheaf of $\alpha_m^!M$ is a
  $0$-dimensional subsheaf of $M$.  Since $\theta$ preserves dimension, we
  similarly find that $\alpha_m^*M$ is pure $1$-dimensional, and
  $\alpha_m^{*!}M$ is pure $1$-dimensional as a subsheaf of $\alpha_m^!M$.

  It remains only to understand how the functors interact with duality.
  That duality exchanges $\alpha^*$ and $\alpha^!$ follows from the
  corresponding statement for $R\ad$, for which it suffices to consider
  line bundles.  That the dual of the minimal lift is the minimal lift of
  the dual then follows from the fact that duality is exact on pure
  $1$-dimensional sheaves.
\end{proof}

The following has the same proof as the commutative case.

\begin{prop}
  Let $\Spl^1_{X_m}$ denote the subspace of $\Spl_{X_m}$ parametrizing pure
  $1$-dimensional sheaves.  Then $-^D$ is an anti-Poisson involution on
  $\Spl^1_{X_m}$.
\end{prop}

\medskip

One last thing to consider from \cite{poisson} is the discussion of
rigidity.  The following extends directly.  Define a relative Euler
characteristic $\chi_C(M,N)$ by
\[
\chi_C(M,N) = \dim(\Hom(M,N)) - \dim(\im(\Ext^1(M,\theta N)\to
\Ext^1(M,N))) + \dim(\Hom(N,M));
\]
note that if $M$ is a simple sheaf, then the dimension of the symplectic
leaf containing $M$ is $2-\chi_C(M,M)$.

\begin{prop}
  One has
  \begin{align}
  \chi_C(M,N) = \chi_C(N,M) =
  \chi(M,N)
+ \dim(\Ext^1_C(M|^\dL_C,N|^\dL_C))-\dim(\Ext^2_C(M|^\dL_C,N|^\dL_C)).
  \end{align}
  Moreover, if $M$ and $N$ are pointless sheaves on $X_{m-1}$, then
  $\chi_C(\alpha_m^{*!}M,\alpha_m^{*!}N)=\chi_C(M,N)$.
\end{prop}

If $M$ is a simple sheaf with $\chi_C(M,M)=2$, then $M$ is infinitesimally
rigid, in that the moduli space of simple sheaves with the same numerical
invariants and derived restriction to $C$ as $M$ is a collection of
isolated points (since the tangent space is $0$-dimensional everywhere).
There are difficulties in extending the classification of rigid
1-dimensional sheaves from the commutative case (the arguments depend too
strongly on the notion of ``support'').  For the explicit example, however,
we have something significantly stronger.

\begin{prop}
  Suppose $q$ is not torsion, and let $D$ be a divisor class on $X_m$ which
  is admissibly equivalent to a simple root.  Then there is at most one
  isomorphism class of pure $1$-dimensional sheaves with first Chern class
  $D$, and any such sheaf is simple.
\end{prop}

\begin{proof}
  The admissible equivalence to a simple root allows us to restrict to the
  case that $D$ is a simple root.  Since there is then a universally ample
  divisor $D'$ such that $D'\cdot D=1$, it follows that there is no
  decomposition $D=D_1+D_2$ with $D_1$, $D_2$ nonzero effective divisor
  classes.  If $M$, $N$ are two pure 1-dimensional sheaves of Chern class
  $D$, then $\chi(M,N)=-2$, and thus at least one of $\Hom(M,N)$ or
  $\Hom(N,M)$ is nonzero.  The lack of a decomposition into effective
  divisors implies that such a morphism must be injective.  Thus the claim
  will follow if we can show that $\chi(M)=\chi(N)$.

  Consider the derived restriction $M|^{\dL}_C$.  Since $M$ is
  $1$-dimensional, the derived restriction has rank 0; moreover, for any
  simple root, $D\cdot C_m=0$, and thus the derived restriction has degree
  0.  If $M|^{\dL}_C$ were not a sheaf, then $\theta M$ would have a
  nontrivial subsheaf supported on $C$, and thus the same would apply to
  $M$ itself.  But this cannot happen; a $0$-dimensional subsheaf
  contradicts purity, while a $1$-dimensional subsheaf contradicts
  indecomposability.  Thus $M|^{\dL}_C$ is a sheaf of rank and degree 0 on
  $C$, implying $M|^{\dL}_C=0$.  Since
  \[
  c_1(M|^{\dL}_C) = q^{-\chi(M)} \rho(c_1(D)),
  \]
  we conclude that there is at most one value of $\chi(M)$ making this
  work.
\end{proof}

\begin{rem}
  Note that if $q$ is torsion, we still have rigidity, in that the sheaf is
  uniquely determined by its numerical invariants; in the torsion case, the
  Euler characteristic is only fixed modulo the order of $q$.
\end{rem}

We record the following consequence of the proof, which will be useful as a
substitute for boundedness results.

\begin{cor}\label{cor:pseudo_bounded}
  Let $q$ be non-torsion, and let $D$ be a divisor class on $X_m$.  Then
  there is at most one integer $\chi$ such that there exists a pure
  1-dimensional sheaf $M$ with $c_1(M)=D$, $\chi(M)=\chi$ and
  $M|^{\dL}_C=0$.
\end{cor}

\begin{prop}
  Let $D$ be a (real) root of $W(E_{m+1})$ such that $\rho(D)=q^{-l}$, and
  suppose that there is a unique isomorphism class of simple sheaves on
  $X_{\rho;q;C}$ with numerical invariants $(0,D,l)$.  Then $D$ is
  admissibly equivalent to a simple root.
\end{prop}

\begin{proof}
  Since $D$ is certainly not a simple root, there is a simple root
  $\alpha_i$ with $\alpha_i\cdot D<0$.  If $\alpha_i$ is not effective,
  then we may perform the corresponding reflection and proceed by induction
  (on the height of $D$, say).  Since by assumption $D$ is not admissibly
  equivalent to a simple root, this process must terminate with $\alpha_i$
  effective and different from $D$.

  Let $M$ be a sheaf with the given invariants (if no such sheaf exists,
  there is nothing to prove), and let $M_0$ be a pure $1$-dimensional sheaf
  with Chern class $\alpha_i$.  Then $\chi(M_0,M)=-\alpha_i\cdot D>0$, so
  that at least one of $\Hom(M_0,M)$ or $\Hom(M,M_0)$ must be nonzero.  If
  $q$ has order $r<\infty$, we may choose a divisor class $D'$ such that
  $D'\cdot D=0$, $D'\cdot \alpha_i\ne 0$, and note that twisting by $nrD'$
  for $n\gg 0$ or $n\ll 0$ as appropriate will give a nonisomorphic sheaf
  (by preventing a map from or to $M_0$).  We may thus assume $q$
  non-torsion at this point.  Then $M_0$ has no nonzero proper subsheaves,
  and thus any composition of nonzero maps $M\to M_0\to M$ would be a
  nontrivial endomorphism of $M$, so that at most one of the two spaces can
  be nonzero.

  If $\Hom(M_0,M)\ne 0$ (the case $\Hom(M,M_0)\ne 0$ is analogous), then
  let $N$ be the complex fitting into the distinguished triangle
  \[
  R\Hom(M_0,M)\otimes M_0\to M\to N\to\quad;
  \]
  that is $N$ is the image of $M$ under the appropriate spherical twist
  functor \cite{SeidelP/ThomasR:2001} (which we call a {\em reflection
    functor} to avoid confusion with twisting by line bundles).  The map
  $\Hom(M_0,M)\otimes M_0\to M$ is injective, and thus $N$ is a sheaf, and
  $\Hom(M_0,N)=0$.  Since $M_0$ is a spherical object, the above process is
  reversible, so in particular $\Hom(N,N)\cong \Hom(M,M)$ and $N$ is
  simple.  If $\Hom(N,M_0)=0$, then repeating the construction gives a
  simple sheaf $M'$ with the same Chern class as $M$ but satisfying
  $\Hom(M_0,M')=0$, so that $M\not\cong M'$.  On the other hand, if
  $\Hom(N,M_0)\ne 0$, then $c_1(N)$ is a root of lower height than $D$, and
  thus by induction there exists a nonisomorphic simple sheaf $N'$ with the
  same numerical invariants.  If $\Hom(N',M_0)=0$, then repeating the above
  construction gives a sheaf $M'\not\cong M$; if $\Hom(M_0,N')=0$, then we
  instead apply the inverse:
  \[
  M'\to N'\to R\Hom_k(R\Hom(N',M_0),M_0)\to\quad .
  \]
\end{proof}

\begin{rem}
  It follows (for $q$ non-torsion) that given any such sheaf $M$, there is
  a sequence of reflection functors taking $M$ to a non-isomorphic sheaf
  $M'$ with the same invariants.  It also follows from the above induction
  that any simple 1-dimensional sheaf with Chern class a root is related
  via a sequence of reflection functors (preserving sheaf-hood) and
  admissible reflections to a sheaf with Chern class a simple root.
\end{rem}

\subsection{Hilbert schemes}

As we mentioned above, one of the difficulties in dealing with moduli
spaces on these noncommutative surfaces is the absence of good boundedness
results.  As a result, we can only give partial results even on the
existence of (quasi)projective moduli spaces of (semi)stable sheaves.
Luckily, the cases we {\em can} deal with include some particularly nice
examples.  In the present section, we consider the case of Hilbert schemes
of points; in the next, we will discuss moduli spaces of 1-dimensional
sheaves.

We should first note the following general statements.  By $D$-stable or
$D$-semistable for an ample divisor class $D$, we mean the obvious analogue
of the usual Gieseker notions from the commutative setting, in terms of the
Hilbert polynomial $p_M(t):=\chi(\sO_X(-tD),M)$.

\begin{prop}
  For any ample divisor class $D$, the moduli functor of $D$-stable sheaves
  on $X_{\rho;q;C}$ is separated.
\end{prop}

\begin{proof}
  The standard proof in the commutative case carries over directly: given
  two families of stable sheaves on a discrete valuation ring with
  isomorphic generic fibers, semicontinuity gives a nonzero morphism
  between the special fibers, and stability forces that to be an
  isomorphism.
\end{proof}

\begin{prop}
  For any ample divisor class $D$, the moduli functor of $D$-semistable
  sheaves on $X_{\rho;q;C}$ is proper.
\end{prop}

\begin{proof}
  Again, the commutative argument carries over.  Given a semistable sheaf
  over the open point of a discrete valuation ring, we may extend in any
  way to a family over the ring itself, and then proceed as in
  \cite{LangtonSG:1975,MehtaVB/RamanathanA:1984,HuybrechtsD/LehnM:2010} to
  modify this to have semistable special fiber.  (Here, we need the fact
  that if $M$ is a $d$-dimensional sheaf, then there is a natural injection
  $M\mapsto R^{2-d}\ad R^{2-d}\ad M$, but this follows immediately from the
  isomorphism $M\cong R\ad R\ad M$ using the appropriate spectral
  sequence.)

  It remains to show that the moduli functor is separated; in other words,
  given two semistable families $M$ and $N$ over a discrete valuation ring
  with isomorphic generic fibers, we need to show that the special fibers
  are S-equivalent.  The isomorphism of generic fibers implies the existence
  of an injective morphism $\pi^kN\to M$ for some $k$; replacing $N$ by the
  image of $\pi^k N$ reduces to the case $\pi^l M\subset N\subset M$.

  Thus suppose $M$ and $N$ are semistable sheaves with $\pi^l M\subset
  N\subset M$, and let $M':=N+\pi M$.  Since $\pi^{l-1}M'\subset N\subset
  M'$, it will suffice to show that $M'$ is semistable with special fiber
  S-equivalent to that of $M$.  If $M'=M$ or $M'=\pi M$, then the special
  fibers are isomorphic.  Otherwise, we have a four-term exact sequence
  \[
  0\to \pi M/\pi M'\to M'/\pi M'\to M/\pi M\to M/M'\to 0
  \]
  in which the middle morphism is neither 0 nor an isomorphism.  The image
  of the middle morphism is simultaneously a subsheaf of $M/\pi M$ and a
  quotient of $N/\pi N$, and thus has the same reduced Hilbert polynomial;
  it follows that the same applies to $M/M'\cong \pi M/\pi M'$.  But this
  immediately implies that $M'/\pi M'$ and $M/\pi M$ are S-equivalent, as
  required.
\end{proof}

In the commutative setting, the Hilbert scheme of points is most naturally
dealt with as a $\Quot$ scheme.  This behaves badly in the noncommutative
setting; instead, one simply (following \cite{NevinsTA/StaffordJT:2007})
considers sheaves with the same numerical invariants as an ideal sheaf of
$n$ points.

\begin{defn}
  The {\em Hilbert scheme of $n$ points on $X_{\rho;q;C}$}, denoted
  $\Hilb^n(X_{\rho;q;C})$ is the moduli space of pure $2$-dimensional
  sheaves $M$ with numerical invariants $(1,0,1-n)$.
\end{defn}

\begin{rem}
  The notation reflects the fact that when $q=1$, this is just the Hilbert
  scheme of $n$ points on $X_{\rho;1;C}$.  Indeed, any rank 1 sheaf $M$ on
  a commutative surface is contained in a line bundle with the same Chern
  class, namely the double dual $M^{**}$, and thus any sheaf on a
  commutative surface with numerical invariants $(1,0,1-n)$ is an ideal
  sheaf.
\end{rem}

\begin{prop}
  Suppose $M$ is a pure $2$-dimensional sheaf of rank 1.  Then $M$ is
  simple.
\end{prop}

\begin{proof}
  Suppose $M$ is not simple, and let $\phi:M\to M$ be a nonscalar
  endomorphism of $M$.  Since $\rank(M)>0$, we find that $M|_C\ne 0$, and
  thus $\phi$ induces an endomorphism of $M|_C$.  In particular, some
  linear combination of $\phi$ and $\id(M)$ will fail to be
  surjective on $M|_C$, and thus fail to be surjective.  It follows (by
  Euler characteristic considerations) that the kernel of $\phi$ is also
  nontrivial.  Purity of $M$ implies that $\rank(\ker(\phi))>0$ and
  $\rank(\im(\phi))>0$; since $\rank(\ker(\phi))+\rank(\im(\phi))=1$, this
  is impossible.
\end{proof}

\begin{lem}
  If $n<0$, then $\Hilb^n(X_{\rho;q;C})$ is empty, while
  $\Hilb^0(X_{\rho;q;C})$ consists of the single point $\sO_X$.
\end{lem}

\begin{proof}
  If $I$ is pure 2-dimensional with numerical invariants $(1,0,c)$ with
  $c\ge 1$, then $I$ has a global section since Chern class considerations
  forbid a morphism from $I$ to $\theta\sO_X$.  Any global section is
  injective, and the cokernel is a $0$-dimensional sheaf, $N$ say.  Since
  any $0$-dimensional sheaf satisfies $\Ext^1(N,\sO_X)\cong
  \Ext^1(\sO_X,\theta N)^*=0$, the extension is split, contradicting the
  pure 2-dimensionality of $I$ unless $N=0$.
\end{proof}

This gives the following characterization of line bundles on
$X_{\rho;q;C}$.

\begin{cor}\label{cor:line_bundle_bound}
  If $M$ is a pure $2$-dimensional sheaf on $X_{\rho;q;C}$ of rank 1, then
  $\chi(M,M)\le 1$, with equality iff $M$ is a line bundle.
\end{cor}

\begin{proof}
  This follows from the Lemma and the fact that $M(-c_1(M))$ is a pure
  $2$-dimensional sheaf with numerical invariants $(1,0,\frac{1+\chi(M,M)}{2})$.
\end{proof}

\begin{rem}
  Substituting in the formula for $\chi$ gives the bound in the form
  \[
  \chi(M) \le 1 +\frac{c_1(M)\cdot (c_1(M) + C_m)}{2}.
  \]
\end{rem}

This implies some useful boundedness results for globally generated
sheaves.

\begin{lem}\label{lem:OX_quot_bound}
   If $M$ is a $1$-dimensional sheaf such that there is a surjective
   morphism $\sO_X\to M$, then $\chi(M)\ge -\frac{c_1(M)\cdot
     (c_1(M)-C_m)}{2}$.
\end{lem}

\begin{proof}
  The kernel $I$ of such a morphism is a pure 1-dimensional sheaf with
  $c_1(I)=-c_1(M)$, $\chi(I)=1-\chi(M)$, so that the result follows
  immediately from Corollary \ref{cor:line_bundle_bound}.
\end{proof}

\begin{rem}
  By the proof, the bound is tight iff $M$ is the quotient of a morphism
  $\sO_X(-D)\to \sO_X$.  In the commutative case, the same quantity is an
  upper bound if we assume that $M$ is pure, but for $q$ non-torsion, there
  is no upper bound in general.  Indeed, on $X_0$ with $q$ non-torsion, any
  sheaf of the form $\sO_{s-f}(d)$, $d\ge 0$, is generated by any of its
  global sections, but has Euler characteristic exceeding the lower bound
  by $d$.
\end{rem}

\begin{prop}\label{prop:glob_generated_bound}
  For every divisor class $D$, there is a bound $B(D)$ such that if $M$ is
  a globally generated sheaf on any $X_{\rho;q;C}$ with numerical
  invariants $(r,D,l)$, then $l\ge r+B(D)$.  Moreover, if $D$ is not
  effective, then there are no such sheaves.
\end{prop}

\begin{proof}
  If $D=0$, take $B(D)=0$.  Otherwise, let $M$ be such a sheaf.  If $M$ has
  a $0$-dimensional subsheaf $T$, we can lower $\chi(M)$ without changing
  $c_1(M)$ by taking the quotient of $M$ by $T$.  Similarly, if
  $\rank(M)>0$, then there is an injective map $\sO_X\to M$, and we can
  take the quotient without changing $\chi(M)-\rank(M)$.  Assume,
  therefore, that $M$ is pure $1$-dimensional (implying $c_1(M)$
  effective!), and let $\phi:\sO_X\to M$ be a global section of $M$.  Then
  $D'=c_1(\im\phi)\ne 0$, so $c_1(\coker\phi)<c_1(M)$, and thus by
  induction we have
  \[
  \chi(M)\ge -\frac{D'\cdot (D'-C_m)}{2}+B(D-D').
  \]
  By the proof of Lemma \ref{lem:subsheaf_Cherns_finite}, there are only
  finitely many possible values for $D'$ even as we allow the surface to
  vary, and thus we may take the minimum
  \[
  B(D):=\min_{D'}\left(-\frac{D'\cdot (D'-C_m)}{2}+B(D-D')\right),
  \]
  where $D'$ ranges over nonzero divisors such that both $D'$ and $D-D'$
  are in the dual to the universal nef cone.
\end{proof}

\begin{rem}
  If we expand the above description to the full decomposition of $D$,
  we can simplify the resulting expression to obtain
  \[
  B(D) =-\frac{D\cdot (D-C_m)}{2}
  +\min_{D_1+\cdots+D_l=D}
    \sum_{1\le i<j\le l} D_i\cdot D_j,
  \]
  where the minimum is over $D_i$ in the dual to the universal nef cone.
  Note that if $D_i\cdot D_j\ge 0$ for some $i<j$, then we can merge those
  two components without increasing the value, and thus only those
  decompositions for which $D_i\cdot D_j<0$ for all $i<j$ need be
  considered.
\end{rem}

\begin{cor}\label{cor:glob_generated_finite_invs}
  Let ${\cal M}$ be a set of globally generated sheaves on surfaces $X_m$
  such that the corresponding set of numerical invariants is finite.  Then
  the set of numerical invariants of globally generated subsheaves of
  sheaves in ${\cal M}$ is also finite.
\end{cor}

\begin{proof}
  Let $M\in {\cal M}$ and let $N$ be a subsheaf. Since $0\le \rank(N)\le
  \rank(M)$, there are only finitely many possible values for $\rank(N)$.
  Moreover, since both $N$ and $M/N$ are globally generated, we find that
  $c_1(N)$ and $c_1(M)-c_1(N)$ are both effective.  Then Lemma
  \ref{lem:subsheaf_Cherns_finite} tells us that there are only finitely
  many possible values for $c_1(N)$.  We then have the bounds
  \[
  \rank(N)+B(c_1(N))\le \chi(N),\quad\text{and}\quad
  \rank(M)-\rank(N)+B(\rank(M)-c_1(N))\le \chi(M)-\chi(N),
  \]
  giving a finite range of Euler characteristics for each choice of
  $c_1(N)$.
\end{proof}

\begin{lem}\label{lem:sheaves_of_class_df}
  If $M$ is a pure $1$-dimensional sheaf of Chern class $df$ on $X_0$ or
  $X'_0$, then there is a filtration $M_i$ of $M$ such that each
  $M_{i+1}/M_i$ is pure $1$-dimensional of Chern class $f$ and such that
  \[
  \chi(M_1)\ge \chi(M_2/M_1)\ge\cdots \chi(M/M_{d-1}).
  \]
\end{lem}

\begin{proof}
  Since $M$ is pure $1$-dimensional, there is some twist of $M$ that has no
  global sections; the claim is invariant under twisting, so we may assume
  $h^0(M(-s))=0$.  Since $\chi(M(ns))>0$ for $n\gg 0$, there is some
  smallest $n_1$ such that $M(n_1s)$ has a global section, and again we may
  assume $n_1=0$.  Now, let $N$ be the image of the natural map
  $\sO_X\otimes_k\Hom(\sO_X,M)\to M$.  If $N$ has Chern class $d_1f$, then
  we have $\chi(N)\ge B(d_1f)$.  An easy induction shows that $B(df)=d$ for
  all $d$, and thus $\chi(N)\ge d_1$.  Since $h^0(N(-s))=0$, we also have
  $\chi(N(-s))\le 0$, and this implies $\chi(N)=d_1$.  It follows that $N$
  is a direct sum of quotients $\sO_X/\sO_X(-d'f)$, and is thus the
  pullback of a $0$-dimensional sheaf on $\P^1$.  Taking a filtration of
  that $0$-dimensional sheaf with degree $1$ subquotients gives a
  filtration of $N$ in which each subquotient is the pullback of a point.

  In particular, $N$ is acyclic, and thus $M/N$ has no global
  sections, so is pure $1$-dimensional.  By induction, $M/N$ has a
  filtration as desired, and $h^0(M/N)=0$ implies that the Euler
  characteristic of the bottom subsheaf is nonpositive, and thus
  strictly less than the Euler characteristics of the subquotients
  of the filtration of $N$.  Gluing the two filtrations gives the
  desired result.
\end{proof}

We can also control acyclicity of twists.

\begin{lem}\label{lem:glob_generated_bounded}
  If $M$ is a globally generated coherent sheaf, then for any nef divisor
  class $D$ with $D\cdot C_m>0$ and $c_1(M)-D-C_m$ ineffective, $M(D)$ is
  acyclic.
\end{lem}

\begin{proof}
  If $n=1$, $\rank(M)=1$, this is immediate from $M\cong \sO_X$.  If $n=1$,
  $\rank(M)=0$, let $I$ be the kernel of the given global section.  Since
  $D$ is nef, $\Ext^p(\sO_X(-D),\sO_X)=0$ for $p>0$, and thus we find
  $\Ext^2(\sO_X(-D),M)=0$ and
  \[
  \Ext^1(\sO_X(-D),M)\cong \Ext^2(\sO_X(-D),I)\cong \Hom(I,\theta
  \sO_X(-D))^*.
  \]
  A nonzero morphism $I\to \theta \sO_X(-D)$ would be injective (both
  sheaves have only rank 1 subsheaves) and the cokernel would be a
  $1$-dimensional sheaf of ineffective Chern class $c_1(M)-C_m-D$.

  For general $n$, choose one global section, and let $M_1$ be its image.
  Then $c_1(M_1)-D-C_m\le c_1(M)-D-C_m$ is ineffective, so the globally
  generated sheaf $M_1$ is acyclic, and similarly for $M/M_1$ by induction
  in $n$.
\end{proof}

We will need some boundedness results on the Hilbert scheme.

\begin{lem}\label{lem:Hilb_bounded_Xm}
  For $m>0$, if $I$ is a pure $2$-dimensional sheaf on $X_m$ with numerical
  invariants $(1,0,1-n)$, then $\alpha_{m*}I(-ne_m)$ is a pure
  $2$-dimensional sheaf on $X_{m-1}$ with numerical invariants
  $(1,0,1-n(n+3)/2)$, and $I(-ne_m)\cong \alpha_m^{*!}\alpha_{m*}I(-ne_m)$.
\end{lem}

\begin{proof}
  Since $I$ is pure $2$-dimensional, $\Hom(\sO_{e_m}(-1),I(-ne_m))=0$.  On
  the other hand, if there were a morphism $I(-ne_m)\to \sO_{e_m}(-1)$,
  then the cokernel would be $0$-dimensional, and thus the kernel would be
  a pure $2$-dimensional sheaf with numerical invariants
  $(1,-(n+1)e_m,d-n(n+3)/2)$ with $d>0$, violating Corollary
  \ref{cor:line_bundle_bound}.  The claim follows from Lemma
  \ref{lem:poisson5.10}.
\end{proof}

\begin{lem}\label{lem:Hilb_bounded_X0}
  If $I$ is a pure $2$-dimensional sheaf on $X_0$ or $X'_0$ with numerical
  invariants $(1,0,1-n)$, then $I(nf+D)$ is acyclic and globally
  generated for any nef divisor class $D$ such that $D\cdot f\ge n$.
\end{lem}

\begin{proof}
  Since $\chi(I(nf))=1$ and $\Ext^2(\sO_X(-nf),I)\cong \Hom(I,\theta
  \sO_X(-nf))=0$ by Chern class considerations, there is a morphism
  $\sO_X(-nf)\to I$, giving a short exact sequence
  \[
  0\to \sO_X(-nf)\to I\to M\to 0.
  \]
  Since $M(nf)$ has no global sections, it follows from Lemma
  \ref{lem:sheaves_of_class_df} that it has a filtration in which each
  subquotient has numerical invariants $(0,f,l)$ for $l\le 0$.  It follows
  that each $l\le n$, and thus $M(ds+d'f)$ is acyclic and globally
  generated for $d\ge n$.  Since $\sO_X(-nf+ds+d'f)$ is acyclic for $0\le
  d\le d'-n$, we conclude that $I(ds+d'f)$ is acyclic and globally generated
\end{proof}

This allows us to fill in the missing boundedness result for Theorem
\ref{thm:quot_is_projective}.

\begin{lem}\label{lem:glob_generated_strongly_bounded}
  For any choice of numerical invariants $(r,D,l)$ and any divisor class
  $D_a$, there is an integer $B$ such that for any $X_m$ on which $D_a$ is
  ample and any short exact sequence
  \[
  0\to I\to \sO_X^n\to M\to 0
  \]
  for which $M$ has numerical invariants $(r,D,l)$, the sheaves $I(bD_a)$
  and $M(bD_a)$ are acyclic and globally generated for $b\ge B$.
\end{lem}

\begin{proof}
  The same argument as in the proof of Lemma
  \ref{lem:glob_generated_bounded} allows us to reduce to the case $n=1$,
  $r=0$.  Then $I$ has numerical invariants $(1,-D,1-l)$, and thus $I(D)$
  has invariants $(1,0,1-l-D\cdot(D-C_m)/2)$.  For $m>-1$, we may
  inductively apply Lemmas \ref{lem:Hilb_bounded_Xm} and
  \ref{lem:Hilb_bounded_X0} to construct a (universal) divisor class
  $D_{m,n}$ such that $I(D+D_{m,n})$ is acyclic and globally generated.
  Then Lemma \ref{lem:glob_generated_bounded} extends this to a translate
  of the relevant nef cone, and thus to multiples of $D_a$.

  For $m=-1$, we simply apply Proposition \ref{prop:noetherian_bounded} to
  the Hilbert scheme, which is projective by
  \cite[Thm.~8.11]{NevinsTA/StaffordJT:2007}.
\end{proof}

\begin{rem}
  Note that since $D$ is effective, the proof of Lemma
  \ref{lem:subsheaf_Cherns_finite} tells us that there is only a finite set
  of possible numerical invariants for any given Hilbert polynomial.
\end{rem}

\bigskip

We can now prove that the Hilbert scheme is projective.

\begin{lem}
  The Hilbert scheme $\Hilb^n(X_{\rho;q;C})$ is a smooth, proper,
  Poisson algebraic space of dimension $2n$.
\end{lem}

\begin{proof}
  That it is a Poisson algebraic space is inherited from
  $\Spl_{X_{\rho;q;C}}$, and smoothness follows from the fact that
  $\Ext^2(M,M)\cong \Hom(M,\theta M)^*=0$, so deformations are
  unobstructed, while properness follows from the fact that a sheaf with
  the given numerical invariants is semistable iff it is stable iff it is
  pure $2$-dimensional.  The dimension follows from $\chi(M,M)=1-2n$ and
  the fact that $M$ is simple.
\end{proof}

\begin{thm}\label{thm:hilbn_is_nice}
  The Hilbert scheme $\Hilb^n(X_{\rho;q;C})$ is a smooth, irreducible
  projective scheme.
\end{thm}

\begin{proof}
  For $X_{-1}$, this was shown in
  \cite[Thm.~8.11]{NevinsTA/StaffordJT:2007}.  Otherwise, it follows from
  Lemmas \ref{lem:Hilb_bounded_Xm} and \ref{lem:Hilb_bounded_X0} that there
  is a divisor class $D$ such that for any $I$ (on any $X_m$) with
  numerical invariants $(1,0,1-n)$, $I(D)$ is acyclic and globally
  generated.  Since Corollary \ref{cor:glob_generated_finite_invs} implies
  that there are only finitely many possible values for the numerical
  invariants of a globally generated subsheaf of $I(D)$, it follows from
  \cite[Lems.~4.4.5,4.4.6]{HuybrechtsD/LehnM:2010} that there is a
  linearization of the action of $\GL(\Hom(\sO_X(-D),I))$ on the $\Quot$
  scheme such that every point corresponding to $I$ is GIT stable.  (Note
  that since $I$ has rank 1, any subsheaf still satisfies the relevant
  bound, even if it has more global sections than expected.) But then
  $\Hilb^n(X_{\rho;q;C})$ is a subscheme of the GIT quotient, and is
  therefore quasiprojective, so projective.

  It remains to show irreducibility.  Here, we use the fact that since the
  moduli problem is unobstructed, the moduli space remains smooth even as
  we vary $X_{\rho;q;C}$ in a family.  (Moreover, our choices of ample
  divisor class could be made in a universal fashion, so the construction
  continues to work as we allow the surface to vary.)  In particular, we
  may consider the family with base $\Pic^0(C)$ obtained by keeping $\rho$
  and $C$ fixed but allowing $q$ to vary.  The argument of
  \cite[Prop.~8.6]{NevinsTA/StaffordJT:2007} then applies to say that if
  any fiber is connected, then all fibers are connected.  Since for $q=1$
  this is just the usual Hilbert scheme, connectedness (and thus
  irreducibility since $\Hilb^n(X_{\rho;q;C})$ is smooth) follows
  immediately.
\end{proof}

\begin{rem}
  The quasiprojective subscheme on which $I|_C\cong \sO_C$ is symplectic,
  and can be thought of as an analogue of Calogero-Moser space, see the
  introduction to \cite{NevinsTA/StaffordJT:2007}.
\end{rem}

In the case $n=1$, $\Hilb^1(X_{\rho;q;C})$ is a surface, so we ought to
identify precisely which projective surface results.

\begin{prop}
  For any $\rho,q,C$, $\Hilb^1(X_{\rho;q;C})\cong X_{\rho;1;C}$.
\end{prop}

\begin{proof}
  The key idea is that if $M$ is a pure $2$-dimensional sheaf with the
  given numerical invariants, then $R\Hom(\sO_X,M)=0$.  Indeed,
  $\Hom(\sO_X,M)=0$ and $\Hom(M,\theta \sO_X)=0$ since such a morphism
  would need to be injective, and the cokernel would have ineffective Chern
  class.  Since $\chi(\sO_X,M)=\chi(M)=0$, this implies that
  $\Ext^1(\sO_X,M)=0$ as well, giving $R\Hom(\sO_X,M)=0$.

  In other words, $M$ is an object in the subcategory
  $\cN_{\rho;q;C}\subset D^b\coh X_{\rho;q;C}$.  We may thus use the
  functor $\kappa_q$ to move this to $\cN_{\rho;1;C}$.  We claim that
  $\kappa_q M$ is a pure $2$-dimensional sheaf, and conversely that for any
  point $x\in X_{\rho;1;C}$, $\kappa_q^{-1} I_x$ is a pure $2$-dimensional
  sheaf.  Moreover, if $x\in C$, then $\kappa_q^{-1}I_x$ is
  the ideal sheaf of $qx$.

  We first consider the case $m=-1$.  In this case, we similarly find that
  $\Hom(\sO_X(\pm 1),M)=\Ext^2(\sO_X(\pm 1),M)=0$.  This implies that $M$
  has a (minimal) resolution of the form
  \[
  0\to \sO_X(-2)\to \sO_X(-1)^2\to M\to 0;
  \]
  conversely, any pair of linearly independent morphisms in
  $\Hom(\sO_X(-2),\sO_X(-1))$ will have pure $2$-dimensional quotient.
  Since this condition is preserved by $\kappa_q$, the claim follows.
  (Note also that the moduli space in this case is
  $\text{Gr}(2,\Hom(\sO_X(-2),\sO_X(-1))))\cong \P^2$.)  The claim in the
  case of an ideal sheaf of a point of $C$ follows from the fact that
  $\kappa_q$ respects restriction to $C$, and the torsion subsheaf of the
  restriction to $C$ of the ideal sheaf of a point $x\in C$ is the
  structure sheaf of $x/q$.  A similar argument applies to
  $X_{\eta,\eta';q;C}$.

  For the remaining cases, let $\sO_e(-1)$ denote either $\sO_{e_m}(-1)$ or
  $\sO_s(-1)$ as appropriate, with $\alpha_*$, $\alpha^*$ the corresponding
  functors associated to the blowdown.  If $\Hom(M,\sO_e(-1))=0$, then $M$
  is $\alpha_*$-acyclic, and $\alpha_*M$ is pure $2$-dimensional, with
  $M\cong \alpha^*\alpha_*M$, and thus the claim follows by induction.

  If $\Hom(M,\sO_e(-1))\ne 0$, then the image is $\sO_e(-d)$ for some $d\ge
  1$, but then applying Corollary \ref{cor:line_bundle_bound} as in the
  proof of Lemma \ref{lem:Hilb_bounded_Xm} tells us that $d=1$ and that we
  have a short exact sequence
  \[
  0\to \sO_X(-e)\to M\to \sO_e(-1)\to 0.
  \]
  Since $\sO_X(-e)\cong \alpha^*\alpha_*\sO_X(-e)$ and $\alpha_*\sO_X(-e)$
  is the ideal sheaf of a point, this structure is preserved by $\kappa_q$
  and $\kappa_q^{-1}$ as required.
\end{proof}

\medskip

In the commutative case, the Hilbert scheme is rational, suggesting that
the same should hold for our deformations.  This is indeed the case, and in
fact we can establish an explicit birational map to corresponding Hilbert
schemes of {\em commutative} rational surfaces.

The basic idea in the construction is that if $D$ is a divisor class such
that $\chi(I(D))=0$, then we can expect an open subset on which $I(D)\in
\cN_{\rho;q;C}$, allowing us to apply $\kappa_q$ and then untwist.  Of
course, this will in general only give us an object in the derived
category, but we can at least hope that this will generically be a sheaf.
On one side, this is well-behaved.

\begin{lem}
  Let $S$ be a locally noetherian scheme, and let $\phi:M\to N$ be a
  morphism of $S$-flat families of sheaves on $X_{\rho;q;C}$.  Then
  the set of points of $S$ for which the corresponding fiber of $\phi$ is
  surjective is open.
\end{lem}

\begin{proof}
  Suppose first that $m>0$ and (as we may) that $S$ is noetherian.  By the
  proof of Proposition \ref{prop:noetherian_bounded}, there is an integer
  $D$ such that for $d\ge D$, $\theta^{-d}M$ and $\theta^{-d}N$ are both
  acyclic for $\alpha_{m*}$.  By induction, for each $d\ge D$, there is an
  open subset $U_d$ on which $\alpha_{m*}\theta^{-d}\phi$ is surjective.
  Now, if $\alpha_{m*}\theta^{-d}\phi$ is surjective, then the standard
  spectral sequence gives $R\alpha_{m*}\coker(\theta^{-d}\phi)=0$, and thus
  $\coker(\theta^{-d}\phi)\cong \sO_{e_m}(-1)^a$ for some $a$.
  But then $\coker(\theta^{-d-1}\phi)\cong \sO_{e_m}^a$ and
  $\alpha_{m*}\coker(\theta^{-d-1}\phi)\ne 0$.  In other words, $\phi$ is
  surjective at every point in $U_d\cap U_{d+1}$.  But any surjective fiber
  is in $U_d$ for sufficiently large $d$, and thus the desired open subset
  can be written as the union $\cup_{d\ge D} (U_d\cap U_{d+1})$.

  For $m\le 0$, let $D_a$ be an ample divisor as in Corollary
  \ref{cor:regularity_on_X0} or \ref{cor:regularity_on_P2} as appropriate.
  Let $U_l$ be the open subset of $S$ where $N$ is $(-l,D_a)$-regular and
  the induced map
  \[
  \Hom(\sO_X(-lD_a),M)\to \Hom(\sO_X(-lD_a),N)
  \]
  is surjective.  Then $\phi$ is surjective on the fiber $v\in S$ iff $v\in
  U_l$ for some $l$, giving the desired result.
\end{proof}

\begin{lem}
  Let $S$ be a locally noetherian scheme, and let $M^\cdot$ be an $S$-flat
  family of bounded coherent complexes on $X_{\rho;q;C}$.  Then the set of
  points of $S$ for which the corresponding fiber of $M^\cdot$ only has
  cohomology in degree 0 is an countable intersection of open subsets.
\end{lem}

\begin{proof}
  By the previous lemma, we know that the set where $M^\cdot$ has nonpositive
  cohomology is open.  On the other hand, the set where $M^\cdot$ has only
  nonnegative cohomology can be expressed as the countable intersection of
  open subsets $\Ext^i(\sO_X(D),M^\cdot)=0$ for $i<0$, $D\in \Pic(X)$.
\end{proof}

\begin{rem}
  In particular, if $S$ is integral, then the generic fiber is a sheaf as
  long as at least one fiber is a sheaf.
\end{rem}

\begin{lem}
  For any $n$, there is a birational map
  $\Hilb^n(X_{\rho;q;C}(-(n-1)f))\ratto \Hilb^n(X_{\rho;1;C})$
  given by
  \[
  I\mapsto \kappa_q(I((n-1)f))(-(n-1)f)
  \]
  on the nonempty intersection of open sets where $I((n-1)f)\in \cN_{\rho;q;C}$
  and $\kappa_q(I((n-1)f))$ is a pure $2$-dimensional sheaf.
\end{lem}

\begin{proof}
  Since all of the conditions we are imposing are open (or intersections of
  open conditions), we need only show that they hold generically, or
  equivalently that there is {\em some} point $I\in
  \Hilb^n(X_{\rho;q;C}(-(n-1)f))$ for which everything holds.  Let
  $y_1,\dots,y_n$ be sufficiently general points of $C$, and consider the
  corresponding blowup $X_{m+n}$.  The direct image on $X_m(-(n-1)f)$ of
  the sheaf $\sO_X(-e_{m+1}-\cdots-e_{m+n})$ on $X_{m+n}(-(n-1)f)$
  corresponds to a point of $\Hilb^n(X_{\rho;q;C}(-(n-1)f))$.  Since
  $R\Gamma(\sO_X((n-1)f-e_{m+1}-\cdots-e_{m+n}))=0$ and $\kappa_q$ respects
  blowups and the action of $W(E_{m+n+1})$, we reduce to showing that
  \[
  \kappa_q(\sO_X((n-1)f-e_1-\cdots-e_n))
  \cong
  \sO_X((n-1)f-e_1-\cdots-e_n)
  \]
  on the generic $n$-point blowup of $X_0$.  In fact, using reflections in
  $W(D_{m+n})$ along with an elementary transformation if $n$ is odd, we
  find that this is equivalent to showing
  \[
  \kappa_q(\sO_X(-f+e_1+\cdots+e_n))
  \cong
  \sO_X(-f+e_1+\cdots+e_n),
  \]
  which follows from the $n=0$ case and the compatibility of $\kappa_q$
  with $\alpha_m^!$.
\end{proof}

Since Hilbert schemes of commutative rational surfaces are rational, we
conclude the following.

\begin{cor}\label{cor:Hilbn_is_rational}
  The Hilbert scheme $\Hilb^n(X_{\rho;q;C})$ is rational.
\end{cor}

It will be helpful below to consider a generalization of the above
construction.

\begin{prop}\label{prop:D_patch_of_Hilb}
  Let $D$ be a nef divisor on $X_{\rho;q;C}$ with $D\cdot (D+C_m)/2=n-1$,
  $D\cdot C_m>0$.  Then the generic point $I\in \Hilb^n(X_{\rho;q;C}(-D))$
  satisfies $I(D)\in \cN_{\rho;q;C}$ and $(\kappa_q I(D))(-D)\in
  \Hilb^n(X_{\rho;1;C})$.
\end{prop}

\begin{proof}
  It will be convenient to prove instead that the generic point $I\in
  \Hilb^n(X_{\rho;1;C})$ is such that $I(D)\in \cN_{\rho;q;C}$ and
  $\kappa_q^{-1}I(D)$ is a pure $2$-dimensional sheaf; this will still give
  us a birational map, and thus the twist of the image of the generic point
  in $\Hilb^n(X_{\rho;1;C})$ will be the generic point in
  $\Hilb^n(X_{\rho;q;C}(-D))$ as required.

  We may in addition assume that $D$ is universally nef.  Moreover, we have
  already dealt with the case $D=(n-1)f$; we also find that the conclusion
  holds for the non-nef divisors $D\in \{-s-f,-2s-2f\}$ on $X_{-1}$ and
  $D=-s+lf$ on either $X_0$, since in each of those cases $n=0$ so $I\cong
  \sO_X$.  Similarly, if $D\cdot e_m=0$, then the generic $I$ is an inverse
  image of some $I$ on $X_{m-1}$, and thus we may reduce to the blown down
  surface.  The same argument allows us to blow down from $X_1$ if $D\cdot
  (f-e_1)=0$.

  Having blown down when possible, we find that $D-C_m$ is either
  universally nef or one of the above non-nef special cases, and thus by
  induction the claim holds for $D-C_m$.  It suffices to prove that for
  {\em some} $I$, $R\Gamma I(D)=0$ and $\kappa_q^{-1} I(D)$ is a pure
  2-dimensional sheaf, since the conditions are intersections of open
  conditions.  Thus let $Z'$ be a generic $D\cdot (D-C_m)/2+1$-point
  subscheme of $X_{\rho;1;C}$, and let $I_{Z'}$ be its ideal sheaf.  By
  induction, $\theta \kappa_q^{-1} I_{Z'}(D)\cong \kappa_q^{-1}
  I_{Z'}(D-C_m)$ is a pure 2-dimensional sheaf.  Let $Z''$ be a generic
  $D\cdot C_m$-point subscheme of $C$, and consider the subscheme
  $Z:=Z'\cup Z''$.  We have a non-split short exact sequence
  \[
  0\to \theta I_{Z'}\to I_Z\to {\cal L}\to 0
  \]
  where ${\cal L}$ is the ideal sheaf of $Z''$ on $C$.  In particular,
  ${\cal L}(D)$ is a degree 0 line bundle, and genericity gives ${\cal
    L}(D)\in \cN_{\rho;1;C}$, so the same holds for $I_Z(D)$.  Twisting by
  $D$ and applying $\kappa_q^{-1}$ gives the short exact sequence
  \[
  0\to \kappa_q^{-1}\theta I_{Z'}(D)\to \kappa_q^{-1}I_Z(D)\to {\cal
    L}(D)\otimes q^{-1}\to 0,
  \]
  where we used the fact that, by genericity, ${\cal L}(D)\otimes q^{-1}\in
  \cN_{\rho;1;C}$.  It follows that $\kappa_q^{-1}I_Z(D)$ is a sheaf, and
  it remains only to show that it is pure 2-dimensional.

  To see this, note that since $I_{Z'}(D)$ is pure 2-dimensional, the
  torsion subsheaf of $I_Z(D)$ must be a subsheaf of the line bundle ${\cal
    L}(D)\otimes q^{-1}$, so if $I_{Z}(D)$ is not pure 2-dimensional, its
  torsion subsheaf is a line bundle ${\cal L}'$ on $C$.  Then the long
  exact sequence for restriction to $C$ starts
  \[
  0\to \theta {\cal L}'\to \theta {\cal L}(D)\otimes q^{-1}\to \theta
  I_{Z'}(D)|_C\otimes q^{-1}
  \]
  Since $I_{Z'}(D)|_C$ is torsion-free, we conclude that ${\cal L}'\cong
  {\cal L}(D)\otimes q^{-1}$.  But this contradicts the non-split
  condition!
\end{proof}

We could also make the Euler characteristic 0 by twisting down rather than
twisting up.  Such cases can be easily understood, at least when $q$ is
non-torsion, via the following fact; compare
\cite[Thm.~8.11(3)]{NevinsTA/StaffordJT:2007}.

\begin{prop}\label{prop:ad_Hilbert_scheme}
  If $I\in \Hilb^n(X_{\rho;q;C})$ with $I|_C\cong q^{-n}$ and $n<|\langle
  q\rangle|$, then $R\ad I$ is a sheaf and $\theta^{-1}\ad I\in
  \Hilb^n(X_{\rho;1/q;C})$.
\end{prop}

\begin{proof}
  We recall that $R^2\ad I$ is $\le 0$-dimensional and $R^1\ad I$ is $\le
  1$-dimensional.  In fact, if either bound were tight, then the
  isomorphism $R\ad R\ad I\cong I$ would produce a $<2$-dimensional
  subsheaf of $I$.  We thus conclude that $R^2\ad I=0$ and $R^1\ad I$ is
  $\le 0$-dimensional.  Similarly, if $\ad I$ had a $<2$-dimensional
  subsheaf, this would prevent $R\ad R\ad I\cong I$ from being a sheaf, and
  thus $\ad I$ is pure $2$-dimensional.  Since
  \[
  (R\ad I)|^{\dL}_C\cong R\sHom_C(I|^{\dL}_C,\omega_C)\cong q^n
  \]
  is a sheaf, we conclude that $R^1\ad I$ must be disjoint from $C$, and
  thus $\chi(R^1\ad I)=l|\langle q\rangle|$ for some $l\ge 0$.  If $l=0$,
  we are done; otherwise, we deduce that $\chi(\ad I)=1-n+l|\langle
  q\rangle|>1$, contradicting Corollary \ref{cor:line_bundle_bound}.
\end{proof}

\begin{rem}
  In other words, $I\mapsto \theta^{-1}\ad I$ gives an isomorphism between
  the two Calogero-Moser spaces.
\end{rem}

Now, consider
\[
\kappa_q(\theta I(-D))\cong R\ad \kappa_q(R\ad I(D+C_m)),
\]
where $D$ satisfies the hypotheses of Proposition
\ref{prop:D_patch_of_Hilb}.  Since $\theta^{-1}R\ad I$ is generically an
ideal sheaf, Proposition \ref{prop:D_patch_of_Hilb} applies to show that
$\kappa_q(R\ad I(D))(C_m-D)$ is generically an ideal sheaf.  But the dual
of an ideal sheaf on a {\em commutative} surface is never a sheaf, and we
thus conclude that $\kappa_q(\theta I(-D))$ is never a sheaf, so that we
cannot obtain an analogue of Proposition \ref{prop:D_patch_of_Hilb} in this
way.

\bigskip

Any Poisson scheme has an associated rank decomposition; that is, it is the
disjoint union of the (locally closed) fibers of the function giving the
rank of the associated pairing on the cotangent space.  In the case of
$\Hilb^n(X_{\rho;q;C})$, these subschemes are not in general connected, and
we thus obtain a finer decomposition.

Recalling that we have an exact sequence
\[
0\to \End(I)\to \End(I|_C)\to \Ext^1(I,\theta I)\to \Ext^1(I,I),
\]
we see that the rank of the Poisson structure may be computed as
\[
\dim(\Ext^1(I,\theta I))-\dim\End(I|_C)+\dim\End(I)
=
2n+1-\dim\End(I|_C).
\]
Now, $I|_C$ is a sheaf of rank 1, and is therefore a direct sum of a line
bundle $L_I$ and a torsion sheaf $Z_I$, with $L_I$ determined by $Z_I$
since $\det(I|_C)$ is fixed; we moreover compute
\[
\dim\End(I|_C) = \dim\End(L_I)+\dim\Hom(L_I,Z_I)+\dim\End(Z_I)
               = 1 + h^0(Z_I)+\dim\End(Z_I).
\]
Any torsion sheaf on a smooth curve is determined by a nonincreasing,
eventually 0, sequence $D_{1,I}\ge D_{2,I}\ge\cdots$ of effective divisors,
with
\[
Z_I\cong \sO_{D_{1,I}}\oplus \sO_{D_{2,I}}\oplus\cdots.
\]
Let $\lambda(I)$ be the induced partition such that
$\lambda(I)_i:=\deg(D_{i,I})$.  Then $h^0(Z_I)=|\lambda(I)|$, while
\[
\dim\End(Z_I) = \sum_{1\le i,j} \lambda(I)_{\max(i,j)}
= \sum_j (2j-1)\lambda(I)_j.
\]
so that the Poisson structure has rank
\[
2n-\sum_j 2j\lambda(I)_j
\]
at $I$.

Let ${\cal H}_\lambda$ denote the locally closed subscheme of
$\Hilb^n(X_{\rho;q;C})$ on which $\lambda(I)=\lambda$.  There is an induced
morphism ${\cal H}_\lambda\to \prod_i \Sym^{\lambda_i-\lambda_{i+1}}(C)$
given by taking $I\in {\cal H}_\lambda$ to the sequence of effective
divisors $D_{1,I}-D_{2,I},D_{2,I}-D_{3,I},\dots$, and the fibers are
(large) symplectic leaves of $\Hilb^n(X_{\rho;q;C})$.  It follows that
${\cal H}_\lambda$ has dimension $2n+\lambda_1-\sum_j 2j\lambda_j$.  In
particular, since the degeneracy locus is an anticanonical Cartier divisor
(cut out by the pfaffian of the Poisson structure), it follows that ${\cal
  H}_1$ is dense in the degeneracy locus, since it is the only
$1$-dimensional piece.

\begin{prop}
  The locally closed subscheme ${\cal H}_1\subset \Hilb^n(X_{\rho;q;C})$ is
  the smooth locus of the anticanonical divisor.
\end{prop}

\begin{proof}
  Since the anticanonical divisor is cut out by a pfaffian, the
  multiplicity of a point on the anticanonical divisor is bounded below by
  half the corank of the corresponding alternating form.  In particular,
  for $|\lambda|>1$, any point of ${\cal H}_\lambda$ has multiplicity $>1$
  on the anticanonical divisor, so is not in the smooth locus.

  It remains to show that the tangent space has codimension 1 at
  every point $I\in {\cal H}_1$.  Equivalently, we need to show that there
  is an infinitesimal deformation of $I$ along which the derivative of the
  pfaffian is nonzero.  Since the special fiber has corank 2, the
  derivative of the overall pfaffian can be computed in terms of the
  corresponding form on the radical.  The radical is the quotient of
  $\End_C(I|_C)$ by $\End(I)$, and the form induced by an element
  $\gamma\in\Ext^1(I,I)$ is given by $(\alpha,\beta)\mapsto
  (\gamma|_C\alpha)(\beta)-(\alpha\gamma|_C)(\beta) =
  \gamma|_C([\alpha,\beta])$, where we use the duality between
  $\Ext^1_C(I|_C,I|_C)$ and $\End(I|_C)$.  Since $\End(I|_C)$ is
  nonabelian, there exists $\gamma$ such that this form is nonzero, and
  thus the corresponding deformation is not a tangent vector to ${\cal
    H}_1$ at $I$.
\end{proof}

\begin{rem}
  It should be possible to prove by similar means that every point of
  the anticanonical divisor has multiplicity equal to half the corank, by
  showing that the commutator form has maximal rank for generic $\gamma$.
\end{rem}

\begin{lem}
  The fibers of the morphism ${\cal H}_1\to C$ are smooth, geometrically
  connected, rational varieties.
\end{lem}

\begin{proof}
  The fiber over $x\in C$ consists of those sheaves $I$ such that $I|_C$
  has torsion subsheaf $\sO_x$.  Let $\alpha:\tilde{X}\to X_{\rho;q;C}$ be
  the blowup in $x$, with exceptional divisor $e$, and consider the sheaf
  $\alpha^{*!}I$.  This has numerical invariants $(1,-e,1-n)$, and thus
  twisting by $e$ gives a sheaf with numerical invariants $(1,0,2-n)$,
  which we claim is a point in the open symplectic leaf of
  $\Hilb^{n-1}(\tilde{X})$.  It suffices to show that $(\alpha^{*!}I)|_C$
  is torsion-free, since if $\alpha^{*!}I$ had a torsion subsheaf, it would
  necessarily have a subsheaf of the form $\sO_{e_{m+1}}(-d)$, which would
  prevent the restriction to $C$ from being torsion-free.  Since
  $\sO_e(-1)$ and $\alpha^!I$ are transverse to $C$, it follows that
  $\alpha^{*!}I$ is transverse to $C$, and thus there is a short exact
  sequence
  \[
  0\to \sO_e(-1)|_C\to \alpha^*I|_C\to \alpha^{*!}I|_C\to 0
  \]
  from which it follows that $\alpha^{*!}I|_C$ is the quotient of $I|_C$ by
  its torsion subsheaf $\sO_x$.

  It follows that the fiber over $x$ is contained in the smooth,
  geometrically connected, rational variety $\Hilb^{n-1}(\tilde{X})$.
  Since the fiber has dimension $2n-2$ (if nonempty), it is dense in
  $\Hilb^{n-1}(\tilde{X})$, so also rational.  To see that the fiber is
  nonempty, observe that if $I$ is any point in the open symplectic leaf of
  $\Hilb^{n-1}(X)$, then the kernel $I'$ of the unique morphism $I\to
  \theta^{-1}\sO_x$ is a point of $\Hilb^n(X)$ and $I'|_C$ has torsion
  subsheaf $\sO_x$.
\end{proof}

\begin{cor}
  The Albanese torsor of ${\cal H}_1$ is canonically isomorphic to $C$.
\end{cor}

\begin{proof}
  Indeed, the Albanese torsor of a fibration with geometrically rational
  fibers is canonically isomorphic to the Albanese torsor of the base, and
  a smooth curve of genus 1 is canonically isomorphic to its Albanese
  torsor.
\end{proof}

There is also an induced covering of the anticanonical divisor by a family
of subvarieties parametrized by $C$: to each $x\in C$, we may associate the
closure of the corresponding fiber of ${\cal H}_1$.

There is a useful constraint on the closure of the fibers of ${\cal
  H}_\lambda$ over the product of symmetric powers.  Such fibers are
specified by a map from $C$ to the set of partitions such that $\sum_{x\in
  C} \lambda(x)=\lambda$: indeed, we simply take
$\lambda_i(x):=\deg_x(D_i)$.  We can determine $\lambda(x)$ geometrically
by observing that
\[
\Hom(\sO_{lx},I|_C)
=
\sum_{1\le i} \max(l,\lambda_i(x))
=
\sum_{1\le i\le l} \lambda'_i(x).
\]
But then semicontinuity tells us that for any point in the closure of the
fiber, the corresponding function $\mu$ satisfies
\[
\sum_{1\le i\le l} \mu'_i(x)\ge \sum_{1\le i\le l} \lambda'_i(x),
\]
a version of dominance ordering.  (This is not quite the traditional
dominance ordering, however, as $|\mu|$ could be larger than $|\lambda|$.)
Note that
\[
\sum_j 2j \lambda_j = \sum_j \lambda'_j(\lambda'_j+1)
                    = \sum_x \sum_j \lambda'_j(x)(\lambda'_j(x)+1).
\]
Also, if $\mu'(x)$ dominates $\lambda'(x)$ and $|\mu(x)|>|\lambda(x)|$,
then $\mu'(x)$ dominates the partition obtained from $\lambda'(x)$ by
adjoining a part 1.  This decreases the rank of the Poisson structure by 2
but can be done in a 1-parameter family of ways, so accounts for a boundary
divisor in the closure.  Otherwise, by results on classical dominance
ordering (\cite[Prop. 2.3]{BrylawskiT:1973}), $\mu(x)'$ dominates some
partition obtained from $\lambda(x)'$ by adding 1 to a part and subtracting
1 from a later part.  This yields finitely many symplectic leaves, each of
which has dimension at least 2 less than the original symplectic leaf (more
if the two parts being modified are distinct).

\begin{lem}
  If $n-\sum_j j\lambda_j=0$, then ${\cal H}_\lambda$ is isomorphic to the
  associated product of symmetric powers of $C$ and consists precisely of
  the direct images of sheaves $\sO_X(-\lambda'_1e_{m+1}-\cdots-\lambda'_l
  e_{m+l})$ on iterated blowups of $X_{\rho;q;C}$.
\end{lem}

\begin{proof}
  Let $I$ be a point in ${\cal H}_\lambda$, and let $x\in C$ be a point
  such that $\dim\Hom(\sO_x,I)=\lambda'_1=:l$.  If we blow up $x$, then
  $\dim\Hom(\sO_e(-1),\alpha^*I)=l$, and thus $\alpha^{*!}I$ has numerical
  invariants $(1,-le,1-n)$.  If $\alpha^{*!}I$ is torsion-free, then
  $(\alpha^{*!}I)(le)$ is a point in the $(n-l(l+1)/2)$-point Hilbert
  scheme of the twisted blowup, and the torsion subscheme of
  $(\alpha^{*!}I)(le)|_C$ is isomorphic to the quotient of the torsion
  subscheme of $I$ by $\sO_x^l$.  We may thus proceed by induction to find
  that $(\alpha^{*!}I)(le)$ is the direct image of a suitable line bundle on
  a blowup, and thus the same is true of $I$.  Note that any sheaf
  $\sO_X(-\lambda'_1e_{m+1}-\cdots-\lambda'_le_{m+l})$ is indeed a minimal
  lift: since it is torsion-free, it has no morphisms from sheaves
  $\sO_{e_{m+i}}(-1)$ or $\sO_{e_{m+i}-e_{m+j}}(-1)$, and any morphism to
  such a sheaf would correspond to a global section of
  $\sO_{e_{m+i}}(-1-\lambda'_i)$ or
  $\sO_{e_{m+i}-e_{m+j}}(-1-\lambda'_i+\lambda'_j)$ on the appropriate twist.
  It follows in particular that the fibers of the morphism to the product
  of symmetric powers are nonempty.

  It remains only to show that $\alpha^{*!}I$ is forced to be
  torsion-free.  Suppose otherwise, and let $T$ be the torsion subsheaf of
  $\alpha^{*!}I$, with torsion-free quotient $F$.  The direct image of $T$
  would be a subsheaf of $\alpha_*\alpha^{*!}I\cong I$, and thus
  $\alpha_*T=0$, forcing $T$ to have numerical invariants $(0,ae,-b)$ for
  $a>0$, $b\ge 0$.  Moreover, if $b=0$, then $R\alpha_*T=0$, contradicting
  the fact that $\alpha^{*!}I$ is a minimal lift.  We thus find that $F$
  has numerical invariants $(1,-(l+a)e,1-n+b)$, so that $F((l+a)e)$ has
  numerical invariants $(1,0,1-n+b+(a+l)(a+l+1)/2)$, and corresponds to a
  point in the appropriate Hilbert scheme.  Moreover, we can control the
  symplectic leaf containing $F((l+a)e)$ since $F((l+a)e)|_C$ is a twist of
  $F|_C$ and we have an exact sequence
  \[
  0\to T|_C\to \alpha^{*!}|_C\to F|_C\to 0.
  \]
  This induces a corresponding exact sequence for the torsion subsheaves,
  and thus the torsion subsheaf of $F|_C$ is a quotient of the torsion
  subsheaf of $\alpha^{*!}|_C$ such that the Euler characteristic has
  decreased by $a$.  The corresponding partition $\mu$ must therefore
  satisfy $\mu'_i\le \lambda'_{i+1}$ for each $i$, with
  $|\lambda|-|\mu|=l+a$.  It follows that the symplectic leaf containing
  $F((l+a)e)$ has dimension
  \begin{align}
  2n-{}&l(l+1)-a(a+1)-2al-2b-\sum_i \mu'_i(\mu'_i+1)\notag\\
  &=
  2n-\sum_i \lambda'_i(\lambda'_i+1)
  -a(a+1)-2b
  -\sum_i (2l-\lambda'_{i+1}-\mu_i-1)(\lambda'_{i+1}-\mu'_i)\notag\\
  &\le
  2n-\sum_i \lambda'_i(\lambda'_i+1)-4
  \end{align}
  Since $\sum_i \lambda'_i(\lambda'_i+1)=\sum_i 2i \lambda_i=2n$,
  the symplectic leaf containing $F((l+a)e)$ has negative dimension, so
  must in fact be empty.
\end{proof}

In the next result, when we talk about the fiber of ${\cal H}_\lambda$ over
a tuple of points in $C$, we mean the fiber over the corresponding point
in the product of symmetric powers.

\begin{lem}
  If $\sum_i i\lambda_i=n-1$, then for any sequence of points
  $(z_1,\dots,z_l)\in C^l$ (with $l:=\ell(\lambda)$) such that the ordered
  pairs $(z_i,\lambda'_i)$ are distinct, the closure of the corresponding
  fiber of ${\cal H}_\lambda$ is isomorphic to
  $X_{\rho,q^{\lambda_1}z_1,\dots,q^{\lambda_l}z_l;1;C}$.
\end{lem}

\begin{proof}
  The same bound on the rank of the Poisson structure as in the previous
  Lemma tells us that any point of the fiber has torsion-free minimal lift,
  of the form $I'(-\sum_i \lambda'_i e_{m+i})$ with $I'\in
  \Hilb^1(X_{\rho,z_1,\dots,z_l;q,C}(\sum_i \lambda'_ie_{m+i}))$.  More
  generally, suppose $I$ corresponds to a point in the closure of the
  fiber, with associated partition function $\mu(x)$.  The constraint on
  $z_1,\dots,z_l$ implies that for each partition $\lambda'(x)$, the
  nonzero parts are distinct, and thus $|\mu(x)|=|\lambda(x)|$ for all $x$
  would imply that the symplectic leaf had negative dimension.  Thus
  $\mu(x)$ is obtained by adding a part $1$ at the end of some
  $\mu'(z_{l+1})$, so that $\mu$ itself has an additional part $1$.  It
  follows from the previous Lemma that such an $I$ still satisfies the
  minimal lift condition.

  In other words, the twisted minimal lift induces a morphism from the
  closure of the fiber to the $1$-point Hilbert scheme; since this is
  inverted by taking the direct image, it remains only to show that it is
  dominant.  But the condition to be a minimal lift translates to the
  vanishing of certain $\Hom$ sheaves, and is therefore an open condition,
  and the $0$-dimensional symplectic leaves show that the open set is
  nonempty and therefore dense.
\end{proof}

This allows us to prove that the family of Poisson schemes
$\Hilb^n(X_{\rho;q;C})$ is a truly nontrivial deformation.  In fact, we
have something slightly stronger.

\begin{thm}
  For $n>1$, if the anticanonical divisors of $\Hilb^n(X_{\rho;q;C})$ and
  $\Hilb^n(X_{\rho';q';C'})$ are isomorphic, then there is an isomorphism
  $\phi:C\cong C'$ and an element $w\in W(E_m)$ such that $\rho'=w\cdot
  \rho\circ\phi$, $q'=\phi(q)$.
\end{thm}

\begin{proof}
  The isomorphism between the anticanonical divisors induces an isomorphism
  $C\cong C'$ by comparing the Albanese torsors of their smooth loci.
  For simplicity, assume that this isomorphism is the identity, so that we
  need to show $q'=q$ and $\rho'\in W(E_m)\rho$.  For each $z\in C$, the
  associated fiber of ${\cal H}_1$ maps to the fiber labelled by $z$ in
  ${\cal H}'_1$, and thus the same applies to the closures of the fibers,
  and to the intersections thereof.  In particular, for any $(n-1)$-tuple
  $(z_1,\dots,z_{n-1})$ of distinct elements of $C$, we may take the
  corresponding intersections of fiber closures, which will agree with the
  closures of the fibers of ${\cal H}_{n-1}$ or ${\cal H}'_{n-1}$ as
  appropriate.  We thus conclude that for any such $(n-1)$-tuple, there is
  an induced isomorphism
  \[
  X_{\rho,qz_1,\dots,qz_{n-1};1,C}
  \cong
  X_{\rho',q'z_1,\dots,q'z_{n-1};1;C},
  \]
  identifying the two anticanonical curves.  It follows that there is an
  element of $W(E_{m+n-1})$ taking $\rho,qz_1,\dots,qz_{n-1}$ to
  $\rho',q'z_1,\dots,q'z_{n-1}$.  For generic $\vec{z}$, this element must
  be contained in $W(E_m)$, giving $\rho'\in W(E_m)\rho$ as required.  That
  $q'=q$ then follows from $q'z_1=qz_1$.
\end{proof}

\begin{cor}
  For $n>1$, if there is an isomorphism $\Hilb^n(X_{\rho;q;C})\cong
  \Hilb^n(X_{\rho';q';C'})$ that respects the Poisson structure up to a
  nonzero scalar, then there are $\phi:C\cong C'$ and $w\in W(E_m)$ as
  above.
\end{cor}

\begin{proof}
  Indeed, then the isomorphism respects the anticanonical divisors, so
  induces an isomorphism of anticanonical divisors.
\end{proof}

\begin{rem}
  The conclusion of course fails for $n=1$, and some condition on the
  anticanonical divisor is needed when $q=1$, $m<8$.  On the other hand, it
  follows from \cite{HitchinN:2012} that for {\em generic} $q$ and $n>1$,
  there is a unique Poisson structure on $\Hilb^n(X_{\rho;q;C})$ up to
  scalars, and thus any isomorphism will be anticanonical.  See also
  \cite{LiC:2014} for a version of the Corollary in a neighborhood of the
  commutative del Pezzo surface case.
\end{rem}

\begin{rem}
  Most of the time, an element of $W(E_m)$ induces an isomorphism of
  noncommutative surfaces, and thus automatically induces an isomorphism of
  $n$-point Hilbert schemes.  Does this continue to hold in the presence of
  effective divisors of self-intersection $-2$?
\end{rem}

\begin{rem}
  The corresponding claim for the open symplectic leaf is actually false as
  stated: Proposition \ref{prop:ad_Hilbert_scheme} shows that the
  open symplectic leaves of $\Hilb^n(X_{\rho;q;C})$ and
  $\Hilb^n(X_{\rho;1/q;C})$ are isomorphic as long as $q$ has sufficiently
  high order.  It seems likely from the results of \cite{elldaha} that the
  same symmetry applies to the (singular) model of $\Hilb^n(X_{\rho;q;C})$
  coming from the deformation of an ample divisor on
  $\Sym^n(X_{\rho;1;C})$.  Indeed, \cite{elldaha} constructs a
  $m+4$-dimensional family of (noncommutative) deformations of
  $\Sym^n(X_{\rho;q;C})$ invariant under an involution in the additional
  deformation parameter.  Presumably this family has a resolution of
  singularities extending the family $\Hilb^n(X_{\rho;q;C})$ by an
  additional noncommutative deformation parameter.
\end{rem}

\bigskip

Another version of the Hilbert scheme of points arises when $q$ is torsion.
As we saw above, in the $q$ torsion case, there are exotic instances of
simple $0$-dimensional sheaves not supported on a power of the
anticanonical curve.  This suggests that we should investigate the moduli
space of such sheaves.  The requisite boundedness is trivial in this case,
since any $0$-dimensional sheaf is automatically both acyclic and globally
generated.  We thus immediately find that the moduli space of semistable
sheaves with numerical invariants $(0,0,r)$ is a projective scheme for any
$r$.  Let $M$ be a $0$-dimensional sheaf.  If $M|_C\ne 0$, then $M$ is
S-equivalent to the direct sum of $M|_C$ and the image of $T_M$.  Since
sheaves supported on $C$ are themselves S-equivalent to sums of structure
sheaves of points, we thus find that any $0$-dimensional sheaf is
S-equivalent to a sum of structure sheaves of points on $C$ and
0-dimensional sheaves disjoint from $C$.  A $0$-dimensional sheaf disjoint
from $C$ only exists if $q$ has finite order, in which case it is supported
on a finite subscheme of the center.  We then similarly find that such a
sheaf is S-equivalent to a sum of sheaves each of which is a simple module
over a fiber of the relevant Azumaya algebra.  We thus find that the
desired moduli space is the disjoint union of locally closed subschemes
$\Sym^{r_1}(C)\times \Sym^{r_2}(X_{\phi_*\circ \rho;1;C'}\setminus C')$,
where $\phi:C\to C'$ is the isogeny corresponding to $q$ and
$r_1+r_2|\langle q\rangle|=r$.

It remains to understand how these pieces are related.  For this, it
suffices to understand the case $|\langle q\rangle|=r$, so that the moduli
space is the disjoint union of $\Sym^r(C)$ and $X_{\phi_*\circ
  \rho;1;C'}\setminus C'$.  Let $Y$ be the Zariski closure of the latter in
the moduli space.  Since the generic sheaf classified by $Y$ satisfies
$\theta M\cong M$, it follows that {\em any} sheaf corresponding to a point
of $Y$ must at least be S-equivalent to its image under $\theta$.  It
follows that $Y$ meets $\Sym^r(C)$ along the image of $C'$ under the map
taking a point to its preimage under $\phi$.  Moreover, for any point $x\in
C$, we may apply Lemma \ref{lem:exotic_zero_dim} to the blowup of $X_m$ in
$x$ to obtain a {\em simple} sheaf in the corresponding S-equivalence
class.  In particular, we find that $Y$ is covered by an open subset of the
{\em smooth} algebraic space classifying simple 0-dimensional sheaves of
Euler characteristic $r$.  It follows that $Y$ is smooth, implying that
$Y\cong X_{\phi_*\circ\rho;1;C'}$.  We have thus shown the following.

\begin{prop}\label{prop:points_q_torsion}
  Let $q$ be an $r$-torsion point of $C$, with corresponding isogeny
  $\phi:C\to C'$.  Then for any $\rho$, the moduli space of semistable
  $0$-dimensional sheaves of Euler characteristic $r$ is isomorphic to the
  union of $\Sym^r(C)$ and $X_{\phi_*\circ \rho;1;C'}$ with $C'$ identified
  with its image under $\phi^*:C'\to \Sym^r(C)$.
\end{prop}

\begin{rem}
  Note that since the Azumaya algebra is nontrivial, this is not a fine
  moduli space; the Azumaya algebra is in itself the Brauer obstruction to
  having a universal family.
\end{rem}

\begin{rem}
  Of course, this gluing applies more generally to determine how
  $C^{r_1}\times X_{\phi_*\circ
    \rho;1;C'}^{r_2}$ maps to the moduli space for larger $r$.
\end{rem}

The most striking thing about the above result is not the result itself,
but that the argument and the conclusion are both quite similar to that of
\cite[Thm.~7.1]{rat_Hitchin}, a theorem about {\em $1$-dimensional} sheaves
on {\em commutative} surfaces.  We will explain this in Section
\ref{sec:derived} below.

\subsection{Semistable $1$-dimensional sheaves}

Since difference equations correspond to $1$-dimensional sheaves, it is
natural to spend particular effort on the corresponding moduli spaces.

\begin{lem}
Let $D_a$ be an ample divisor class, $D$ an effective divisor class, and
$l$ an integer.  Then there is an integer $n$, depending only on $D_a^2$,
$D\cdot D_a$, and $l$, such that for any $D_a$-semistable sheaf $M$ with
numerical invariants $(0,D,l)$, $M(nD_a)$ is acyclic and globally
generated.
\end{lem}

\begin{proof}
  We may feel free to replace $D_a$ by any multiple of $D_a$, and
  may thus ensure that $D\cdot D_a<D_a^2$ and $D_a\cdot C_m\ge 2$.
  We may also feel free to start by twisting $M$ by any multiple
  of $D_a$, and may thus arrange to have $\chi(M)>0$.  (Both steps
  clearly depend only on $l$, $D\cdot D_a$ and $D_a^2$.)  Note that
  semistability ensures that for any nonzero quotient $M'$ of $M$,
  $\chi(M'(nD_a))>0$ for all $n\ge 0$.

  Now, for $0\le n$, let $M_n$ denote the image of the natural map
  \[
  \sO_X(-nD_a)\otimes_k\Hom(\sO_X(-nD_a),M)\to M.
  \]
  Since $M_n(nD_a)$ is globally generated, $M_n((n+1)D_a)$ is again
  globally generated, but now also acyclic by Lemma
  \ref{lem:glob_generated_bounded}.  It follows immediately that $M_{n+1}$
  contains the maximal sheaf $M_n\subset \tilde{M}_n\subset M$ such that
  $\tilde{M}_n/M_n$ is 0-dimensional.  Moreover, if $M\ne \tilde{M}_n$,
  then
  \[
  \chi((M/\tilde{M}_n)((n+1)D_a))>0
  \]
  so $(M/\tilde{M}_n)((n+1)D_a)$ has a global section, implying that
  $M_{n+1}$ is strictly larger than $\tilde{M}_n$.  Thus in general,
  $c_1(M_{n+1})\cdot D_a\ge c_1(M_n)\cdot D_a$, with equality iff
  $c_1(M_n)\cdot D_a=D\cdot D_a$.  Since $c_1(M_0)\ne 0$, we conclude that
  $c_1(M_n)\cdot D_a\ge \min(D\cdot D_a,n+1)$, and thus $c_1(M_{D\cdot
    D_a-1})=c_1(M)$.  It follows that $M/M_{D\cdot D_a-1}$ is
  $0$-dimensional, and thus $M((D\cdot D_a)D_a)$ is an extension of acyclic
  and globally generated sheaves.
\end{proof}

We will  in the following form.

\begin{cor}\label{cor:lepotier_simpson}
  Let $D_a$ be an ample divisor class such that $D_a\cdot C_m\ge 2$.  For each integer $r>0$, any
  semistable 1-dimensional sheaf $M$ with $c_1(M)\cdot D_a=r$ and
  $\chi(M)>r^2(r+1)$ is acyclic and globally generated.
\end{cor}

\begin{proof}
  If $c=\lfloor \frac{r}{D_a^2}+1\rfloor$, then $(cD_a)^2>c_1(M)\cdot
  (cD_a)$, and thus the above argument applies and tells us that if $M_0$
  is a semistable sheaf with $c_1(M_0)=c_1(M)$ and $\chi(M_0)>0$, then
  $M_0(rcD_a)$ is acyclic and globally generated.  Thus $M$ will be acyclic
  and globally generated as long as
  \[
  0<\chi(M(-rcD_a))=\chi(M)-r^2c.
  \]
  Since $c\le r+1$, the claim holds.
\end{proof}

This will allow us to bound the number of global sections of a semistable
sheaf, but we will first need to know something about the generic
behavior of cokernels of morphisms of line bundles.

\begin{prop}
  Let $D$ be a nef divisor class such that $D\cdot C_m\ge 2$, and let $Q$
  be the family of quotients $\sO_X/\sO_X(-D)$.  Then for any pure
  $1$-dimensional sheaf $M$, there is a fiber $Q_v$ such that
  $\Hom(Q_v,M)=0$.
\end{prop}

\begin{proof}
  We first note that for any filtration of $M$, if $Q_v$ maps to $M$, then
  it maps to some subquotient, and thus without loss of generality $M$ is
  irreducible.  Suppose $M$ is not disjoint from $C$.  Then, since the
  support of $Q_v|_C$ is disjoint from $C(k)$, we find that
  $\Hom(Q_v,M|_C)=0$ and thus any morphism from $Q_v$ to $M$ factors
  through the kernel of the map $M\mapsto M|_C$.  We can repeatedly take
  this kernel, giving a chain of subsheaves of $M$ with the same Chern
  class and strictly decreasing Euler characteristics.  Since the image of
  a nonzero morphism from $Q_v$ to $M$ has the same Chern class as $M$,
  this implies that $\Hom(Q_v,M)=0$ as required.
  
  We now proceed to reduce $D$, starting of course with putting it
  in the fundamental chamber.  Since every fiber still has morphisms
  to $M$, this in particular applies to those for which the map
  $\sO_X(-D)\to \sO_X$ factors through $T_{\theta^{-l}\sO_X(-lD)}$
  for any $l\ge 0$ such that the map {\em can} so factor.  But this
  implies that the generic quotient $\sO_X/\sO_X(lC_m-D)$ also maps
  to $M$, allowing us to induct as long as either $lC_m-D$ is nef
  with $(lC_m-D)\cdot C_m\ge 2$ or $D=lC_m$ (since the claim is
  trivial for $D=0$).  For $m\ge 2$, $l=D\cdot e_m$ always works,
  and thus lets us reduce to the case $D\cdot e_m=0$.  (Similarly,
  for $m=1$, we may take $l=\min(D\cdot e_1,D\cdot (f-e_1))$, and
  apply an elementary transformation as necessary.) Since disjointness
  from $C$ implies that $M$ is $\alpha_{m*}$-acyclic with pure
  $1$-dimensional direct image, we can then reduce to $X_{m-1}$,
  and thus by another induction to $X_0$, $X'_0$ or $X_{-1}$.  Since
  there are no pure $1$-dimensional sheaves disjoint from $C$ in
  the latter two cases, only $X_0$ remains, in which case $M$ must
  be $\sO_{s-f}(l)$ for some $l$.

  The same reduction reduces to $D=df$ or $D=s+df$, with the first
  case ruled out since $df-(s-f)$ is ineffective.  Then the kernel
  of a nonzero morphism $Q_v\to M$ has Chern class $(d+1)f$, and
  is thus an extension of sheaves of Chern class $f$ by Lemma
  \ref{lem:sheaves_of_class_df}.  But this implies that the points
  in the support of $Q_v|_C$ can be paired so that their products
  are in $q^\Z\eta$, and this does not happen generically.
\end{proof}

\begin{lem}
  If $M$ is a semistable $1$-dimensional sheaf with $c_1(M)\cdot D_a=r$,
  then
  \[
  h^0(M)\le \max((r+1)^3+\chi(M),0).
  \]
\end{lem}

\begin{proof}
  Using $\_^{D}$, we may instead prove that
  \[
  h^1(M)\le \max((r+1)^3-\chi(M),0).
  \]
  If $\chi(M)>r^2(r+1)$, this follows from Corollary
  \ref{cor:lepotier_simpson}.  Otherwise, let $c = \lceil
  (r+1)^2-\frac{\chi(M)}{r}\rceil$, so that $\chi(M(cD_a))\ge
  r(r+1)^2>r^2(r+1)$, and thus $M(cD_a)$ is acyclic.  It follows from the
  previous Lemma that there is a quotient $N\cong \sO_X/\sO_X(-cD_a)$ such
  that the corresponding map
  \[
  \Hom(\sO_X,M)\to \Hom(\sO_X(-cD_a),M)
  \]
  is injective.  Thus $\Hom(N,M)=0$ while $\Ext^2(N,M)\cong
  \Ext^2(\sO_X,M)=0$.  We deduce that
  \[
  h^1(M)\le \dim\Ext^1(N,M) = -\chi(N,M) = cD_a\cdot c_1(M)=
  cr<(r+1)^3-\chi(M).
  \]
\end{proof}

This may be used in place of the Le Potier-Simpson estimate (see
\cite[Thm. 3.3.1]{HuybrechtsD/LehnM:2010} for the specific version we
adapted) to show that sheaves with sufficiently small slope cannot be
destabilizing for the GIT quotient.  More precisely, the analogue of
\cite[Cor. 3.3.8]{HuybrechtsD/LehnM:2010} holds for pure $1$-dimensional
sheaves (with the same proof!), except that the constant $C$ should be
replaced by $r^2+3r+5$.  It follows as in the proof of
\cite[Thm. 4.4.1]{HuybrechtsD/LehnM:2010} that a subsheaf $M'$ such that
\[
\frac{\chi(M')}{c_1(M')\cdot D_a}
<
\frac{\chi(M)}{c_1(M)\cdot D_a}
-
\frac{(c_1(M)\cdot D_a)^3-1}{c_1(M)\cdot D_a}-2
\]
is not destabilizing on the $\Quot$ scheme representing $M$ as a quotient
of $\sO_X(-mD_a)$ for (uniformly) sufficiently large $m$.  There are only
finitely many possible values for the remaining numerical invariants;
since the family of sheaves $M$ is bounded, so is the corresponding family
of subsheaves with those numerical invariants.  We may thus find a twist
such that all of those subsheaves are acyclic, and thus every subsheaf has a
sufficiently good estimate on the dimension of its space of global
sections.  It follows that for $m$ uniformly sufficiently large, every
(semi)stable sheaf $M$ is GIT-(semi)stable as a point of the corresponding
$\Quot$ scheme, and thus the semistable moduli space is quasiprojective.
Since we already have shown that it is proper, projectivity follows.

We can say somewhat more.  Call a pair $(D,l)\in \Pic(X)\times \Z$ {\em
  primitive} if there is no integer $r>1$ such that $D\in r\Pic(X)$, $l\in
r\Z$.

\begin{thm}\label{thm:MDl_nice}
  Let $D$ be an effective divisor class, and let $D_a$ be an ample divisor
  class.  Then the moduli problem of classifying S-equivalence classes of
  $D_a$-semistable sheaves $M$ with numerical invariants $(0,D,l)$ is
  represented by a projective scheme $M_{D,l}$.  The stable locus of
  $M_{D,l}$ is a Poisson subscheme, which is smooth if $D\cdot C_m>0$, as
  is the Zariski closure of the sublocus with $M|_C=0$ if $D\cdot C_m=0$.
  If $(D,l)$ is primitive, the stable locus admits a universal family.
\end{thm}

\begin{proof}
  Projectivity follows as discussed above.  That the stable locus is
  Poisson follows from the fact that stable sheaves are simple.  If $D\cdot
  C_m>0$, then stability implies $\Hom(M,\theta M)=0$ (since the slope
  drops) and thus deformations are unobstructed.  Similarly, if $D\cdot
  C_m=0$, then $M$ and $\theta M$ are stable sheaves of the same slope, and
  thus any map is an isomorphism.  Semicontinuity then implies that
  $\dim\Ext^2(M,M)=\dim\Hom(M,\theta M)=1$ on the Zariski closure (in the
  stable locus!) of the given symplectic leaf, and thus $\dim\Ext^1(M,M)$
  is constant, so that the variety is generically reduced with constant
  tangent space dimension.

  Moreover, the stable locus admits a universal family as an algebraic
  space, which means that as a quasiprojective variety, it has a universal
  family \'etale locally.  Moreover, since $(D,l)$ is primitive, there
  exists a complex $N^\cdot$ such that $\chi R\Hom(N^\cdot,M)=1$ for any
  $M$ with numerical invariants $(0,D,l)$.  If $M$ is a universal family
  over some \'etale cover of the stable moduli space, then $M'=M\otimes
  \det R\Hom(N^\cdot,M)^{-1}$ is again universal.  Since twisting $M$ by a
  line bundle has no effect on the resulting $M'$, we find that $M'$
  actually descends to a universal family on the stable moduli space.
\end{proof}

\begin{rems}
  Note that if $(D,l)$ is primitive and $l\ne 0$, we can always
  choose the ample divisor so that every semistable sheaf is stable,
  by taking $D_a$ so that $D\cdot D_a$ is prime to $l$.  In fact,
  given any choice of ample bundle, we can modify it so that not
  only is every semistable sheaf stable, but all of the sheaves
  which were stable for the original bundle remain stable.  To do
  this, simply observe that there are only finitely many possible
  divisor classes of subsheaves of a sheaf with Chern class $D$,
  and thus a discrete set of possible slopes of subsheaves.  In
  particular, there is a lower bound on a nonzero difference between
  the slope of a subsheaf and $\mu=l/D\cdot D_a$.  Thus if we could
  replace $D_a$ by $D_a+\epsilon D_0$ for $\epsilon$ sufficiently
  small, the slopes on either side of $\mu$ would remain sufficiently
  small.  Replacing $D_a$ by $nD_a+D_0$ for $n\gg 0$ thus has a
  similar effect (essentially using $D_0$ to break ties between
  sheaves of the same slope).  This gives a smooth projective Poisson
  moduli space with a natural morphism to the original semistable
  moduli space, which if $D\cdot C_m=0$ is a symplectic resolution.
  (We can do something similar if $l=0$ if we first twist by $D_a$
  before changing the ample bundle.)
\end{rems}

\begin{rems}
  There are imprimitive examples, even in the commutative case, for which
  the obstruction in the Brauer group to the existence of a universal
  family is nontrivial, see the remark following
  \cite[Prop.~6.3]{rat_Hitchin} or Proposition
  \ref{prop:no_fine_moduli_space} below.
\end{rems}

\begin{rems}
  Note that any disjoint-from-$C$ case with a universal family gives rise
  to a Lax pair, in the following way.  Let the difference equation
  corresponding to the universal family (on the main symplectic leaf) be
  \[
  v(qz) = A(z;\mu;\rho;q;p)v(z),
  \]
  where $\mu$ is a point of the moduli space.  If we twist the sheaf by a
  line bundle, then (as we will see below) the new equation has the form
  \[
  \hat{v}(qz) = C(qz;\mu;\rho;q;p) A(z;\mu;\rho;q;p)
  C(z;\mu;\rho;q;p)^{-1} \hat{v}(z)
  \]
  where $C$ is a meromorphic map between the corresponding vector bundles.
  The resulting sheaf (assuming it is still stable, e.g., if the sheaf is
  irreducible) is of the form classified by a related moduli space, and is
  thus equivalent to the equation
  \[
  v(qz) = A(z;\phi(\mu);\tilde\rho;q;p) v(z)
  \]
  where $\tilde\rho$ are the new parameters and $\phi$ is a morphism
  between the stable loci of the two moduli spaces.  But then there is an
  isomorphism $C_0$ between the respective vector bundles identifying the
  two equations. Absorbing this into $C$ produces a pair of equations
  \begin{align}
    v(qz;\mu;\rho;q;p) &= A(z;\mu;\rho;q;p)v(z;\mu;\rho;q;p)\notag\\
    v(z;\phi(\mu);\tilde\rho;q;p) &= C(z;\mu;\rho;q;p) v(z;\mu;\rho;q;p)\notag
  \end{align}
  satisfying the consistency condition
  \[
  A(z;\phi(\mu);\tilde\rho;q;p)C(z;\mu;\rho;q;p)
  =
  C(qz;\mu;\rho;q;p)A(z;\mu;\rho;q;p),
  \]
  and thus giving the desired Lax pair.
  
  This fails, of course, when there is no universal family on the moduli
  space, though there are possible approaches to work around this issue.
  For instance, as suggested by \cite{ellGarnier}, one could choose a
  trivialization of $\rho_*M$ at some auxiliary point $u$ not in the orbit
  of any singularities.  We can then carry this trivialization through the
  twist by any divisor class in $\langle f,e_1,\dots,e_m\rangle$, and thus
  that subgroup of $\Pic(X)$ acts on the new moduli space, a
  $\PGL_n$-bundle over the true moduli space.  Moreover, this new moduli
  space now has a universal family, and thus one obtains a Lax pair as
  above.  (If one also chooses a trivialization of $\rho_*M(-s)$, one can
  make all of $\Pic(X)$ act, though now $\Pic(X)$ acts nontrivially on the
  auxiliary base points.)
\end{rems}

\bigskip

We should note here that we have not actually shown that the stable locus
is nonempty!  In general, it is unclear how to do this (especially since
there are certainly cases in which there are no stable sheaves), but there
are a few cases in which we have particularly nice constructions.
In each case, we obtain an explicit birational map from one component of
the stable locus to a corresponding commutative moduli space; in addition,
in certain cases, there are natural birational maps between the cases.

The first case is $l=0$.  Since such sheaves have $\chi=0$, it is natural
to consider the case $M\in \cN_{\rho;q;C}$ and hope that $\kappa_q M$ will
generically be a sheaf.  As in the proof of Proposition
\ref{prop:D_patch_of_Hilb}, it will be convenient to work backwards
instead.  Thus let $M$ be a generic semistable sheaf on $X_{\rho;1;C}$ with
invariants $(0,D,0)$.  We assume in addition that the linear system $|D|$
has an integral representative (and that $M$ is generic in the
corresponding component), since otherwise we have no hope of constructing
stable sheaves.  In that case, $D$ is represented by an integral curve of
genus $g$, and $M$ is an ineffective line bundle of degree $g-1$ on such a
curve.  In particular, it follows that $M$ is the minimal lift of a sheaf
on $X_0$.  Its image $M_0$ on $X_0$ is still in the kernel of $R\Gamma$,
and thus (arguing as in the proof of Lemma \ref{lem:kappa_q_is_sheaf}) has
a presentation
\[
0\to \pi^*W(-s)\to \sO_X(-f)^{c_1(M)\cdot f}\to M_0\to 0
\]
where $W$ is a vector bundle on $\P^1$ of rank $c_1(M)\cdot f$.  Moreover,
this morphism is injective on $C$, since $M_0$ is transverse to $C$.
Applying $\kappa_q^{-1}$ gives essentially the same presentation, and since
the morphism is still injective on $C$, the cokernel will still be a
pure 1-dimensional sheaf transverse to $C$.  Since $\kappa_q$ commutes
with the minimal lift, it follows that $\kappa_q^{-1}M$ is a pure
$1$-dimensional sheaf disjoint from $C$ with the desired Chern class.
Moreover, since $\kappa_q^{-1}M$ has no global sections, it follows that
any subsheaf of $\kappa_q^{-1}M$ has nonpositive Euler characteristic, and
thus $\kappa_q^{-1}M$ is semistable.

\begin{rem}
  The condition on $D$ is quite weak; per \cite[Thm.~4.8]{rat_Hitchin}
  (based on \cite[Thm.~III.1(c)]{HarbourneB:1997}), the only universally nef
  divisor classes disjoint from $C_m$ which are not generically integral
  are those of the form $rC_8$ with $\rho(C_8)$ torsion of order strictly
  dividing $r$, and of the form $rC_8+e_8-e_9$ with $\rho(e_8-e_9)=1$, and
  in neither case are there strictly stable sheaves of Euler characteristic
  0.
\end{rem}

The significance of the above presentation is that it corresponds directly
(after twisting by $f$) to a symmetric elliptic difference equation of the
form
\[
v(qz) = A(z) v(z)\quad v(q\eta/z)=v(z)
\]
with $A(z)$ a matrix (rather than a morphism of vector bundles) satisfying
$A(\eta/z)A(z)=1$, and with singularities depending on the points being
blown up and the divisor class $D$.  In particular, the conditions on $A$
are independent of $q$, which is why $\kappa_{q'}^{-1}\kappa_q$ simply
changes $q$ to $q'$ in the difference equation.

\medskip

The next case to consider is $l=-1$.  The key idea in this case is that if
such a sheaf has no global sections (which one would expect, by analogy
with the standard commutative heuristic, to be the complement of a
codimension $1-\chi(M)=2$ condition), then there is a universal extension
\[
0\to \theta \sO_X\to I\to M\to 0,
\]
where $I$ has invariants $(1,D-C_m,0)$ and $R\Gamma I=0$.  In particular,
$I(C_m-D)$ corresponds to a point of $\Hilb^{D\cdot
  (D-C_m)/2+1}(X_{\rho;q;C}(C_m-D))$.  Since $(D-C_m)\cdot C_m>0$ (unless
$D=rC_m$ for some $r$), we understand what the generic such sheaf looks
like, suggesting that we should construct $M$ by reversing this
construction.

It turns out that the resulting construction works as long as either
$\rho(D)=q$ or $D\cdot C_m>0$.  Let $Z$ be a generic $D\cdot
(D-C_m)/2+1$-point subscheme in $X_{\rho;1;C}$, and consider the sheaf
\[
I:=\kappa_q^{-1} I_Z(D-C_m)
\]
on $X_{\rho;q;C}$.  Then $R\Hom(\theta \sO_X,\theta I)$=0, and thus
\[
\Hom(\theta \sO_X,I)
\cong
R\Hom(\theta \sO_X,I|_C)
\cong
R\Gamma(C;(\theta^{-1} I)|_C)
\cong
R\Gamma(C;q^{D\cdot C_m-1}\rho(D))
\]
In particular, we obtain a $\P^{\max(0,D\cdot C_m-1)}$ of such maps up to
scalar multiples.  The usual argument says that any such nonzero morphism
is injective, and thus we have a short exact sequence
\[
0\to \theta \sO_X\to I\to M\to 0.
\]
Now, $M$ certainly has the correct invariants, so it remains only to show
that it is stable.  Since $R\Gamma(I)=0$, we find $\Gamma(M)=0$, and thus
no subsheaf has a global section.  Thus the only way a subsheaf $N$ could
be destabilizing would be if it also had no $H^1$.  But in that case, the
inclusion map $N\to M$ would lift to a map $N\to I$, contradicting pure
$2$-dimensionality of $I$.  We thus obtain a rational map from a
$\P^{\max(0,D\cdot C_m-1)}$-bundle over $\Hilb^{D\cdot
  (D-C_m)/2+1}(X_{\rho;1;C})$ to the moduli space of stable sheaves of
Chern class $D$ and Euler characteristic $-1$.  (Note that both spaces have
dimension $D^2+1$ if $D\cdot C_m>0$ or $D^2+2$ if $\rho(D)=q$.)

The case $l=1$ is then straightforward: if $M$ is a stable 1-dimensional
sheaf of Euler characteristic $-1$, then $R^1\ad M$ is a stable
1-dimensional sheaf of Euler characteristic 1, and vice versa.
If $q$ is not torsion (or the arithmetic genus of $D$ is sufficiently small
compared to the order of $q$), then we can dualize the description of $M$
in terms of $I$, obtaining a short exact sequence
\[
0\to \ad I\to \sO_X\to R^1\ad M\to 0.
\]
In other words, the sheaves with $l=1$ we obtain are those that are
generated by their unique global section.  This, too, has a natural
interpretation in terms of difference equations: these are the
``straight-line'' difference equations.  Indeed, we generically have a
unique global section
\[
\Hom(\sO_X(-D-gf),\ad I)
\]
which combines to give an expression for $R^1\ad M$ as a quotient of
some quotient $\sO_X/\sO_X(-D-gf)$.  In other words, there is an operator
${\cal D}\in \hat\cS'_{\rho;q;C}(0,D+gf)$ such that $R^1\ad M$ corresponds
to the equation ${\cal D}v=0$.  (The kernel of the map to $R^1\ad M$ has no
maps or $\Ext^1$ to $\Mer$, see below; the only effect is to introduce $g$
apparent singularities.)

Note that in the case $D\cdot C_m=0$, the construction for $l=1$ gives an
acyclic globally generated sheaf with $M\cong \theta M$.  But then the map
from $M$ to $I$ is the same (anti-)Poisson map we used to prove the Jacobi
identity.  Since $\kappa_q$ is also Poisson, we obtain a birational
symplectic map between the moduli space classifying $M$ and the appropriate
commutative Hilbert scheme.

Similar ideas give us the following.

\begin{prop}
  The generic sheaf in any irreducible component of the subscheme of
  $M_{D,l}$ classifying irreducible sheaves disjoint from $C$ satisfies
  either $H^0(M)=0$ or $H^1(M)=0$.
\end{prop}

\begin{proof}
  Using the adjoint, it suffices to prove that $H^0(M)=0$ for the generic
  such sheaf when $l\le 0$.  By tangent space considerations, the subscheme
  of irreducible sheaves has dimension $D^2+2$ everywhere, and thus it will
  suffice to prove that the locally closed subscheme $M_n$ where $h^0(M)=n$
  has strictly smaller dimension for $n>0$.  The corresponding moduli space
  of (surjective, by irreducibility) maps $\sO_X\to M$ is a
  $\P^{n-1}$-bundle over $M_n$, so for $n>0$ has dimension $n-1+\dim M_n$.
  On the other hand, the kernel of such a map has the form $I(-D)$ for
  $I\in \Hilb^{D^2/2-l}$, and one has a short exact sequence
  \[
  0\to \Hom(O_X,O_X)\to \Hom(I(-D),O_X)\to \Ext^1(M,O_X)\to 0
  \]
  with $\dim\Ext^1(M,O_X)=\dim H^1(M)=n+l$.  Thus the moduli space of maps
  $\sO_X\to M$ can also be viewed as a $\P^{n+l}$-bundle over a
  subscheme of $\Hilb^{D^2/2-l}$, so has dimension at most $D^2+n-l$.
  Comparing gives
  \[
  \dim M_n\le D^2+2-(l+1)<D^2+2
  \]
  as required.
\end{proof}

\begin{rem}
  Note that unless $D=C_m=C_8$, an irreducible sheaf with $c_1(M)=D$ must
  be disjoint from $C$, since otherwise $M|^{\dL}_C$ would be a
  1-dimensional subsheaf.
\end{rem}

This implies that the above construction of 1-dimensional sheaves from
sheaves in the Hilbert scheme is generically invertible for irreducible
sheaves disjoint from $C_0$, giving the following.

\begin{cor}
  The subscheme of $M_{D,1}$ classifying irreducible sheaves disjoint from
  $C$ is either empty or birational to $\Hilb^{1+D^2/2}(X(D))$.  The
  subscheme of $M_{D,-1}$ classifying irreducible sheaves disjoint from $C$
  is either empty or birational to $\Hilb^{1+D^2/2}(X(-D))$.
\end{cor}

\begin{rem}
One can show that, with a few exceptions, the above construction
generically produces irreducible sheaves, and thus rules out the ``empty''
case above.  The idea (we omit the details) is that the inductive
construction of Proposition \ref{prop:D_patch_of_Hilb} in most case
produces an extension of two irreducible sheaves, and thus one need simply
show that there are sheaves without a component of either of those Chern
classes (either as a consequence of transversality or by a dimension
computation).  The few possible universally nef exceptions (arising either
from failures in base cases or the fact that blowing up can introduce
components) are the ones one would expect from the commutative setting (see
\cite[\S 4.3]{rat_Hitchin}, also \cite{HarbourneB:1997}):
\begin{itemize}
  \item[(1)] $D=rf$, $r>1$
  \item[(2)] For some $1\le i\le m$, $D\cdot e_i=0$ and $\rho(D)\in x_i q^\Z$
  \item[(3)] $D=rC_8+e_8$, $\rho(C_8)\in q^\Z$
  \item[(4)] $D=rC_8$, $q\rho(C_8)^r=1$, $\rho(C_8)^{r'}=1$ for some $r'<r$.
  \item[(5)] $D=rC_8+e_8-e_9$, $q\rho(D)=1$, $x_8/x_9\in q^\Z$
\end{itemize}
The given conditions are where the simplest argument fails; it is
conceivable that there may be some instances of (2--5) where the
stable sheaves from the above construction are generically irreducible.
Note that (4) and (5) are the only cases which can be disjoint from $C_m$,
and (4) forces $q$ to be torsion.
\end{rem}

\medskip

If $c_1(M)\cdot e_m=1$, then $\chi(M(e_m))=\chi(M)+1$, and thus twisting by
$e_m$ relates sheaves with $\chi=0$ to sheaves with $\chi=1$.  It turns out
that one can give an explicit description of this action on generic
sheaves, in terms of the respective commutative data.  It is somewhat
easier to describe this in terms of sheaves on $X_{m-1}$, bearing mind that
$M$ is generically a minimal lift.

In particular, given generic $M$ on $X_{m-1}$ with invariants $(0,D,0)$,
$D\cdot C_{m-1}=1$, coming from $\kappa_q^{-1}({\cal L})$ with ${\cal L}$
an ineffective degree $g-1$ line bundle on an irreducible curve in $|D|$,
we want to understand the point in $\Hilb^n$ corresponding to
\[
M'=\theta \alpha_{m*}\theta^{-1}\alpha_m^{*!}M.
\]
We recall from Corollary \ref{cor:pseudo_twist} that this is the universal
extension
\[
0\to M\to M'\to \sO_{qx_m}\to 0;
\]
here we use the fact that $\alpha_m^{*!}M$ is disjoint from $C$ to compute
$\dim\Ext^1(\sO_{qx_m},M)=1$.  Similarly, we can reconstruct $M$ from $M'$
as the kernel of the universal map to $\sO_{qx_m}$.

To relate this to the commutative Hilbert scheme, we need to consider the
kernel $K$ of the unique global section of $M'$. An easy snake lemma
calculation gives a short exact sequence
\[
0\to K\to I_{qx_m}\to M\to 0
\]
Each sheaf in this short exact sequence is in the kernel of $R\Gamma$, so
we may apply $\kappa_q$ to obtain a distinguished triangle
\[
\kappa_q K\to I_{x_m}\to {\cal L}\to
\]
To relate this to the Hilbert scheme, we observe that both $I_{x_m}$ and
${\cal L}$ have a unique extension by $\sO_{x_m}$ (which is the point of
intersection of the support of ${\cal L}$ with $C$), and thus the complex
\[
\begin{CD}
I_{x_m}@>>> {\cal L}\hphantom{(x_m).}
\end{CD}
\]
is quasi-isomorphic to the complex
\[
\begin{CD}
\sO_X@>>> \sO(x_m).
\end{CD}
\]
The dual of this complex is, up to twisting, the ideal sheaf of the unique
effective representative of the divisor class of $\sO(x_m)$ on its
support.

We thus obtain the following description: Starting with a line bundle
${\cal L}$ on a curve $C'$ in $|D|$, we twist up by the unique point of
intersection, and view the corresponding effective divisor on $C'$ as a
subscheme of $X$.  The reverse operation is equally straightforward: given
a $g(D)$-point subscheme of $X$, there is generically a unique curve in
$|D|$ passing through that subscheme, producing a line bundle of
degree $g$ on that curve.  Twisting down by the point of intersection gives
a line bundle of degree $g-1$, which will be generically ineffective.

This procedure is most interesting when we also have $D\cdot e_{m-1}=1$, as
we can then take the commutator of the above operation with reflection in
$e_{m-1}-e_m$ to twist by $e_{m-1}-e_m$.  This should be compared with the
description of elliptic Painlev\'e in
\cite{KajiwaraK/MasudaT/NoumiM/OhtaY/YamadaY:2006} (to which it is related
by a derived equivalence, see below); it is also a direct generalization of
\cite[Thm.~2]{P2Painleve}.

Of course, we can similarly twist between $\chi=-1$ and $\chi=0$; since
minimal lifting and $\kappa_q$ both commute with $R^1\ad$, this is just the
conjugate of the above operation by duality.  It would also be of interest
to understand twisting between $\chi=-1$ and $\chi=1$ in the case $D\cdot
e_m=2$.

\subsection{Lax pairs for elliptic Painlev\'e}

When $D\cdot C_m=D^2=0$, the above moduli space is a quasiprojective
surface, which raises the natural question of whether we can identify the
components.  We may as well assume that $D$ is in the fundamental chamber,
as other cases will be related by a sequence of reflections (or reflection
functors).  If $D\cdot e_m<0$, then $D$ cannot be disjoint from $C_m$, and
thus we reduce to the case that $D$ is universally nef.  It is then
straightforward to see that a universally nef divisor satisfying $D\cdot
C_m=D^2=0$ must be a multiple of $C_8$.  In other words, the 2-dimensional
cases are essentially just the moduli spaces classifying sheaves on $X_8$
with $c_1(M)=rC_8$, $\chi(M)=d$.

For convenience in notation, we fix $\eta,x_0,\dots,x_7$ and reparametrize
by $x_8 = (\eta^2 x_0^3/x_1\cdots x_7z)$ to obtain a surface
$X_{z,q}$.  As we have noted above, the parameters of $X_{\rho;q;C}$ can be
recovered from the restrictions of various classes of $K_0(X_{\rho;q;C})$
to $C$; the parameters $z$ and $q$ in particular corresponding to
restrictions of sheaves supported on $C$:
\[
q = c_1(\sO_x|^{\dL}_C),\quad
z = c_1(\sO_C|^{\dL}_C).
\]
Note in particular that if $c_1(M)=rC_8$, $\chi(M)=d$, then
$c_1(M|^{\dL}_C)=q^dz^r$, so for such a sheaf to be disjoint from $C$
requires that $q^dz^r=1$.  If $r$ and $d$ are relatively prime, this is
equivalent to asking that $z=w^{-d}$, $q=w^r$ for some $w$.

\begin{thm}\label{thm:rd_painleve}
  Suppose $r>0$ and $d$ are relatively prime integers, and let $D_a$ be an
  ample divisor class such that any semistable $1$-dimensional sheaf on
  $X_{w^{-d},w^r}$ with $c_1(M)=rC_8$, $\chi(M)=d$ is actually stable.
  Then the corresponding semistable moduli space is isomorphic to
  $X_{w,1}$.
\end{thm}

\begin{proof}
  Since the tangent space has dimension
  \[
  \dim\Ext^1(M,M)=\dim\Hom(M,M)+\dim\Hom(M,\theta M)=2
  \]
  at the point
  corresponding to $M$, we find that every component is either nonreduced
  or smooth, and must be the latter if it contains a sheaf disjoint from
  $C$.  Moreover, the closed subscheme corresponding to sheaves supported
  on $C$ is just the moduli space of rank $r$ degree $d$ vector bundles on
  $C$, so is naturally isomorphic to $\Pic^d(C)$, and is in particular
  connected and $1$-dimensional.

  Fix a component $Y$ containing sheaves disjoint from $C$ (a smooth
  projective surface), and let $F$ be a corresponding universal family
  (which exists since $r$ and $d$ are relatively prime); that such a
  component exists will be shown in Corollary \ref{cor:rd_painleve_exists}
  below.  Just as in the commutative setting, this gives rise to a
  Fourier-Mukai functor $\Phi_F:D^b\coh Y\to D^b\coh X_{w^{-d},w^r}$.
  Indeed, we may view $F$ as a sheaf on the (flat) base change of
  $X_{w^{-d},w^r}$ to $Y$, and thus we can tensor $F$ with any complex on
  $Y$, then take the derived direct image to $X_{w^{-d},w^r}$.  We claim
  that $\Phi_F$ is actually an equivalence between the two derived
  categories.  By \cite[Thm.~2.4]{BridgelandT/KingA/ReidM:2001} (using the
  fact that point sheaves are a spanning class for $D^b\coh Y$), it
  suffices to prove the following:
  \begin{itemize}
    \item[(1)] $\Phi_F$ has both a left and a right adjoint.
    \item[(2)] For all $x\in Y$, $\Phi_F(\sO_x\otimes \omega_Y)\cong \theta
      \Phi_F(\sO_x)$.
    \item[(3)] For any two points $x,y\in Y$ and $p\in \{0,1,2\}$, the
      natural morphism
      \[
      \Ext^p(\sO_x,\sO_y)\to \Ext^p(\Phi_F(\sO_x),\Phi_F(\sO_y))
      \]
      is a bijection.
    \item[(4)] The category $D^b \coh X_{w^{-d},w^r}$ is indecomposable.
  \end{itemize}
  For (1), the right adjoint applied to $N$ is $R\sHom_Y(F,N_Y)$, just as
  for commutative Fourier-Mukai functors, and conjugating by the respective
  Serre functors gives a left adjoint.

  For (2) and (3), note that $\Phi_F(\sO_x)$ is just the fiber at $x$ of
  the universal family.  For $(2)$ we thus reduce to the fact that every
  sheaf in the family satisfies $\theta M\cong M$.  For (3) with $x\ne y$,
  we have two nonisomorphic sheaves of the given kind, and stability
  prevents there being a morphism in either direction.  But then Euler
  characteristic considerations imply that
  $R\Hom(\Phi(\sO_x),\Phi(\sO_y))=R\Hom(\sO_x,\sO_y)=0$ as required.
  Similarly, $R\Hom(\Phi(\sO_x),\Phi(\sO_x))$ has the same Betti numbers as
  $R\Hom(\sO_x,\sO_x)$, since $\Hom(\Phi_F(\sO_x),\theta\Phi_F(\sO_x))\cong
  \Hom(\Phi_F(\sO_x),\Phi_F(\sO_x))=k$.

  Now, the map $\Hom(\sO_x,\sO_x)\to \Hom(\Phi_F(\sO_x),\Phi_F(\sO_x))$
  certainly takes the identity to the identity; since both sides are
  spanned by the identity, this is an isomorphism.  Next, the map
  $\Ext^1(\sO_x,\sO_x)\to \Ext^1(\Phi_F(\sO_x),\Phi_F(\sO_x))$ is just the
  usual (Kodaira-Spencer) identification of the tangent space to the moduli
  space with the space of infinitesimal deformations.  Finally, the map
  $R\Hom(\sO_x,\sO_x)\to R\Hom(\Phi_F(\sO_x),\Phi_F(\sO_x))$ respects the
  Yoneda pairing, and Serre duality implies that the image of
  $\Ext^1(\sO_x,\sO_x)\otimes \Ext^1(\sO_x,\sO_x)$ spans
  $\Ext^2(\sO_x,\sO_x)$.

  Finally, for indecomposability, we simply note that since
  $R\Hom(\sO_X(-D),\sO_X(-D))=k$, each line bundle must actually be in a
  component, and two line bundles differing by an ample divisor class must
  be in the same component.  But this implies that there is a component
  containing all line bundles, and line bundles generate.

  We have thus shown that $\Phi_F$ is a derived equivalence.  Now, the
  inverse of $\Phi_F$ transports the natural transformation
  $T:\theta\to\id$ to a natural transformation $T:\_\otimes\omega_X\to
  \id$, so that $Y$ has an anticanonical curve $C'$.  Moreover, we find
  that $\Phi_F$ restricts to an equivalence $D^b\coh C'\cong D^b\coh C$.
  Since the derived category of a smooth curve determines the curve, we
  conclude that $C'\cong C$.  From the classification of Poisson surfaces
  \cite{poisson}, any Poisson surface with smooth, connected anticanonical
  curve is rational.  Moreover, $\Phi_F$ identifies $K_0(Y)$ and
  $K_0(X_{w^{-d},w^r})$, and thus $K_Y^2=0$.

  Thus $Y$ can be expressed as the blowup of $F_1$ in $8$ points of a
  suitably embedded copy of $C$; we fix such an expression to obtain a
  basis of $K_0(Y)$. The action of $\Phi_F$ respects restriction to $C$ (up
  to an autoequivalence of $C$), so respects both the image and the kernel
  of the composition
  \[
  R_C:K_0(X)\to K_0(C)\to K_0(X),
  \]
  and similarly for $Y$.  Since $R_C^2=0$, it is natural to consider
  $\ker(R_C)/\im(R_C)$.  The natural pairing induced on this subquotient
  identifies it with the lattice $\Lambda_{E_8}$, and the action of
  $W(E_9)$ on either side acts as $W(E_8)$ on this lattice.  In particular,
  there exists a description of $Y$ as a blowup of $F_1$ with the property
  that $\Phi_F$ identifies the corresponding bases of this subquotient.
  Now, consider how $\Phi_F$ acts on the kernel of $R_C$.  For each simple
  root $r_i$ of $E_8$, there is a corresponding class $[\sO_{r_i}(-1)]$ of
  $K_0(X)$ (i.e., the class with numerical invariants $(0,r_i,0)$), and
  $\Phi_F^{-1}$ of this class will have rank 0 and $c_1\in r_i+\langle
  C_8\rangle$.  The action of $\Lambda_{E_8}\subset W(E_9)$ lets us make
  the Chern class of each $\Phi_F^{-1}[\sO_{r_i}(-1)]$ equal to $r_i$
  (again, at the cost of changing how we represent $Y$ as a blowup), and
  the action of $\Lambda_{E_8}$ as twisting by line bundles then lets us
  make the Euler characteristics 0.  (Note that this twists $F$ by a line
  bundle (and thus changes the derived equivalence), but universal families
  are only defined up to twisting by a line bundle in any event.)  In other
  words, we may assume that our blowdown structure on $Y$ and universal
  family $F$ are such that $\Phi_F[\sO_{r_i}(-1)]=[\sO_{r_i}(-1)]$ for any
  simple root of $E_8$.  Moreover, since $\Phi_F$ must take $\im(R_{C'})$
  to $\im(R_C)$, there is a corresponding element of $\GL_2(\Z)$.  Since
  \[
  \Phi_F([\sO_x]) = r[\sO_C]+d[\sO_x],
  \]
  we find that
  \[
  \Phi_F([\sO_{C'}]) = a[\sO_C]+b[\sO_x]
  \]
  where $ad-br\in \{\pm 1\}$.

  In particular, up to the uncertainty in $a$ and $b$, we have determined
  how $\Phi_F$ acts on $\im(R_{C'})$.  Since we have a commutative diagram
  \[
  \begin{CD}
    \im(R_{C'}) @>\Phi(F)>> \im(R_C)\\ 
    @VVV @VVV\\
    K^0(C') @>\Phi(F)>> K^0(C),
  \end{CD}
  \]
  this is enough to compute how $|^{\dL}_{C'}$ acts on $\im(R_{C'})$.  In
  particular, if $Y$ has parameters $\tilde\eta,\tilde x_0,\dots,\tilde
  x_8$, then (identifying $\Pic^0(C')$ with $\Pic^0(C)$)
  \begin{align}
  \tilde x_i/\tilde x_{i+1} &= x_i/x_{i+1},\quad 0\le i\le 6,\notag\\
  \tilde\eta/\tilde x_1\tilde x_2 &= \eta/x_1x_2.\notag
  \end{align}
  But then we need merely multiply the isomorphism $C\cong C'$ by a
  translation on $C'$ to also ensure that $\tilde x_7 = x_7$, so that
  $Y\cong X_{z,1}$ for some $z$.  To compute $z$, we just need to restrict
  $[\sO_{C'}]$ to $C'$; under $\Phi_F$, this becomes $w^{-ad+br}\in w^{\pm
    1}$, and thus either $Y\cong X_{w,1}$ or $Y\cong X_{1/w,1}$.  Since
  these two surfaces are isomorphic, the claim follows.

  To see that this is the only component, observe that it follows from the
  above argument that any component containing sheaves disjoint from $C$
  also contains the curve $C'\cong \Pic^d(C)$ of sheaves supported on $C$.
  It follows that the moduli space is smooth and connected, so irreducible
  as required.
\end{proof}

\begin{rem}
  Note that since there can be multiple inequivalent choices of stability
  condition, the above derived equivalences are not canonical.
\end{rem}

In particular, we see that for any $d/r\in \Q\cup\{\infty\}$ there is an
interpretation of $X_{w,1}$ as a moduli space of sheaves.  Moreover, any
element of $\Pic(X)\rtimes W(E_9)$ induces isomorphisms between the moduli
spaces of sheaves and thus between surfaces of the form $X_{w,1}$.  In the
case $r=0$, $\Pic(X)$ acts trivially, and thus one has an action of
$W(E_9)$ alone, the translation subgroup of which is the elliptic
Painlev\'e equation of \cite{SakaiH:2001}.  In general, one can transport
the action through the derived equivalence to find that the kernel is still
isomorphic to $\Pic(X)$, but now with a slightly different embedding.  In
each case, we find that the complementary subgroup of $\Pic(X)\rtimes
\Lambda_{E_8}\cong \Z.\Lambda_{E_8}^2$ acts on the moduli spaces as the
elliptic Painlev\'e equation.  In particular, we find that twisting by line
bundles (which we will see acts on the difference equations as isomonodromy
transformations) acts on the moduli space as $r\Lambda_{E_8}\subset
W(E_9)$.  As a result, every such moduli problem gives rise to a Lax pair
for the elliptic Painlev\'e equation (decimated by a factor of $|r|$); see
the remark following Theorem \ref{thm:MDl_nice}, corresponding to a
symmetric elliptic difference equation of order $2r$ on a vector bundle of
degree $d$ such that $A(x_i)$ has rank $r$ for $1\le i\le 8$.  The Lax pair
of \cite{YamadaY:2009}, which was given in straight-line form, presumably
corresponds to the case $r=d=1$.  Similarly, the Lax pair of
\cite{isomonodromy} corresponds to the ``hypergeometric solution'' in the
case $r=1$, $d=0$.  More precisely, it corresponds to the case in which $M$
has a quotient $\sO_{s-f}(n-1)$ for $n\ge 0$, which for $n>0$ is stable,
corresponding to a point on an appropriate $-2$-curve on the moduli space.
(In fact \cite{isomonodromy} considered the more general case in which
$c_1(M)=2s+(l+2)f-e_1-\cdots-e_{2l+6}$, $\chi(M)=0$, and $M$ has a quotient
of the form $\sO_{s-f}(n-1)$ for $n\ge 0$.  For the general case with these
numerical invariants, see \cite{ellGarnier}.)

\medskip

In the case $d=1$, we can use the Hilbert scheme idea to give a relatively
explicit construction of the universal family; more precisely, the Hilbert
scheme construction gives a generically stable (in fact, irreducible)
family, and the universal family will be obtained by taking stable limits.

Let $w$ be an element of $\Pic^0(C)$, $r$ a positive integer less than
$|\langle w\rangle|$, and let $I$ be a pure 2-dimensional sheaf on
$X_{w^{-1},w^r}$ with numerical invariants $(1,0,0)$.  As we have seen above,
the moduli space of such $I$ is naturally identified with $X_{w,1}$, so
there certainly exists such an $I$, and we may further insist that $T_I$ is
injective.  Now, consider the distinguished triangle
\[
R\Hom(\theta^l I,\theta \sO_X)\to R\Hom(\theta^l I,\sO_X)
\to R\Hom_C(\theta^l I|_C,\sO_C)\to
\]
Since $\theta^l I|_C$ is the line bundle with isomorphism class
$w^{l-r}$, we find that
\[
R\Hom(\theta^l I,\sO_X)
\cong
R\Hom(\theta^l I,\theta \sO_X)
\cong
R\Hom(\theta^{l-1}I,\sO_X)
\]
for $0\le l<r$.  Since $R\Hom(\theta^{-1}I,\sO_X)=0$, we conclude by
induction that $R\Hom(\theta^{r-1}I,\sO_X)=0$.  But then the same
distinguished triangle gives
\[
R\Hom(\theta^r I,\sO_X)\cong R\Hom_C(\theta^r I|_C,\sO_C)
\cong R\Hom_C(\sO_C,\sO_C).
\]
In particular, we find that $\Hom(\theta^r I,\sO_X)\cong k$.

Fix a generator of that space.  The image is a nonzero subsheaf of $\sO_X$,
so must have rank 1, and thus the kernel has rank 0, and thus vanishes
since $I$ is pure $2$-dimensional.  Thus the map is injective and gives
rise to a quotient $M_I=\sO_X/\theta^r I$, with numerical invariants
$(0,rC_m,1)$.  We claim that this family is generically stable (and thus in
particular a stable sheaf exists!).  To see this, note that $M_I$ actually
makes sense for an arbitrary sheaf $I\in \Hilb^1(X_{w^{-1},w^r})$; in the
case of the ideal sheaf of a point $x$, the fact that
$\Hom(\sO_x,\sO_X)=\Ext^1(\sO_x,\sO_X)=0$ implies that $\Hom(\theta^r
I_x,\sO_X)\cong \Hom(\theta^r \sO_X,\sO_X)=\langle T^r\rangle$.  The
corresponding quotient $M_{I_x}$ is not, in fact, stable, but {\em does}
satisfy $R\Hom(N,M_{I_x})=0$ for every sheaf disjoint from $C$.

The assumption on $w$ means that $M_I$ (in the disjoint from $C$ case)
cannot have a proper nontrivial subsheaf of Chern class proportional to
$C_m$ (since there are no legal Euler characteristics), and thus if $M_I$
is unstable, then it is destabilized by a sheaf of the form $\sO_\alpha(f)$
for $\alpha$ in some finite set of roots.  For each such effective
$\alpha$, let $f_\alpha$ be the minimum legal degree such that a subsheaf
$\sO_\alpha(f_\alpha)$ would destabilize $M_I$.  Then any unstable $M_I$
has a nonzero homomorphism from $\bigoplus_\alpha \sO_\alpha(f_\alpha)$.
Since
\[
R\Hom(\bigoplus_\alpha \sO_\alpha(f_\alpha),M_{I_x})=0
\]
for all $x\in C$, we conclude that
\[
R\Hom(\bigoplus_\alpha \sO_\alpha(f_\alpha),M_I)=0
\]
for generic $I\in \Hilb^1(X_{w^{-1},w^r})$ as required.

Note that for $w$ non-torsion, essentially the same argument shows not only
that $M_I$ is generically stable, but in fact that $M_I$ is generically
irreducible.

\subsection{Torsion-free sheaves}

In the torsion-free case, we have not been able to prove the bounds needed
to prove that the standard GIT construction carries over.  We can, however,
show that the family of semistable sheaves of a given rank and slope
(defined for simplicity as $\mu(M):=c_1(M)\cdot D_a/\rank(M)$, rather than
the more complicated but equivalent ratio of coefficients of the Hilbert
polynomial) is bounded, which is enough to prove that the corresponding
moduli functor is not only proper, but of finite type.

\begin{lem}
  If the torsion-free sheaf $M$ is globally generated, then for any nef
  divisor class $D$ with $D^2>0$,
  \[
  \chi(M)\le \frac{(c_1(M)\cdot D)^2}{2 D^2} + \frac{c_1(M)\cdot
    C_m}{2}+\rank(M).
  \]
\end{lem}

\begin{proof}
  If $\rank(M)=1$, then Corollary \ref{cor:line_bundle_bound} says
  \[
  \chi(M)\le \frac{c_1(M)^2}{2} + \frac{c_1(M)\cdot C_m}{2}+1,
  \]
  while the Hodge index theorem implies (since $D^2>0$)
  \[
  c_1(M)^2\le \frac{(c_1(M)\cdot D)^2}{D^2},
  \]
  so the claim holds for torsion-free sheaves of rank 1 (regardless of
  whether they are globally generated).  More generally, fix a global
  section of $M$, and let $M_1$ be the corresponding saturated subsheaf
  (that is, $M_1$ is the minimal subsheaf containing the image of the
  global section such that $M/M_1$ is torsion-free).  Since $M/M_1$ is
  globally generated and $M_1$ is torsion-free of rank 1, we find by
  induction that
  \begin{align}
  \chi(M)&=\chi(M_1)+\chi(M/M_1)\notag\\
  &\le
  \frac{(c_1(M_1)\cdot D)^2}{2 D^2}+\frac{((c_1(M)-c_1(M_1))\cdot D)^2}{2 D^2}
  +
  \frac{c_1(M)\cdot C_m}{2} + 1\notag\\
  &\le
  \frac{(c_1(M)\cdot D)^2}{2 D^2} + \frac{c_1(M)\cdot C_m}{2} + 1,
  \end{align}
  as required.  Note that for the second inequality, we need to know that
  $c_1(M_1)\cdot D\ge 0$, which follows from the fact that $c_1(M_1)$ is
  effective.
\end{proof}

\begin{cor}\label{cor:glob_gen_tf_finite_invs}
  There are only finitely many possible values for the numerical invariants
  of a globally generated torsion-free sheaf of fixed rank and bounded
  $c_1(M)\cdot D_a$.
\end{cor}

\begin{proof}
  Since $c_1(M)$ is effective, the bound on $c_1(M)\cdot D_a$ implies that
  there are only finitely many possible values for $c_1(M)$.  We have just
  shown an upper bound on $\chi(M)$, and Proposition
  \ref{prop:glob_generated_bound} gives a lower bound on $\chi(M)$ for each
  choice of Chern class.
\end{proof}

\begin{lem}
  If $M$ is a slope-semistable sheaf of rank $r>0$ and slope $\mu>0$, then
  $h^2(M)=0$.
\end{lem}

\begin{proof}
  If $H^2(M)\ne 0$, then Serre duality gives a nonzero morphism
  $\phi:\theta^{-1}M\to \sO_X$.  Since slope-semistability of torsion-free
  sheaves is preserved under arbitrary twisting, $\theta^{-1}M$ is
  slope-semistable, and again has positive slope.  But $\sO_X$ is also
  slope-semistable, and has slope 0, so that there can be no nonzero
  morphism as required.
\end{proof}

\begin{thm}
  Let $D_a$ be an ample divisor class on $X_{\rho;q;C}$.  Then for any
  positive integer $r>0$, slope $\mu\in \frac{1}{r}\Z$, and integer $\chi$,
  there is an integer $N_{r,\mu,\chi}$ such that for any slope-semistable
  sheaf $M$ of rank $r$, slope $\mu$ and Euler characteristic $\chi$,
  $M(N_{r,\mu,\chi}D_a)$ is acyclic and globally generated.
\end{thm}

\begin{proof}
  We may freely replace $D_a$ by any positive integer multiple of $D_a$,
  and may thus assume that $D_a\cdot C_m>1$ (so that $\sO_X(D_a)$ is
  globally generated) and that $D_a-C_m$ is nef, so that a globally
  generated torsion-free sheaf $M'$ of rank $r'$ and slope $\mu'$ satisfies
  \[
  \chi(M)\le \frac{(r'\mu')^2}{2 D_a^2} + \frac{r'\mu'}{2}+r'.
  \]
  Now, suppose $M$ is slope-semistable of rank $r$ and slope $\mu$.
  Slope stability is preserved under twisting by any line bundle, and in
  particular under twisting by multiples of $D_a$.  We may thus assume WLOG
  that $\mu,\chi>0$, since both go to infinity as we twist by increasing
  multiples of $D_a$.  In particular, for any $l\ge 0$, any quotient of
  $M(lD_a)$ will have vanishing $H^2$.

  For each $l\ge 0$, let $M_l$ denote the image of the natural morphism
  \[
  \sO_X(-lD_a)\otimes_k \Gamma(\sO_X(-lD_a),M)\to M.
  \]
  Since for $l\ge 0$ we have $\chi(M(lD_a))>0$ and $h^2(M(lD_a))=0$, we
  find that each $M_l$ is a nonzero subsheaf of $M$, and (since we have
  arrange for $\sO_X(D_a)$ to be globally generated) this is a
  nondecreasing chain of subsheaves.  Let $r_l$, $\mu_l$ be the rank and
  slope of $M_l$, so that $r_l\le r$, $\mu_l\le \mu$.  Then for each $l$,
  we have the bound
  \[
  \chi(M_l(lD_a))\le \frac{(r_l(\mu_l+lD_a^2)^2}{2 D_a^2} + \frac{r_l(\mu_l+lD_a^2)}{2}+r_l.
                 \le \frac{(r(\mu+lD_a^2)^2}{2 D_a^2} + \frac{r(\mu+lD_a^2)}{2}+r.
  \]
  For $k\ge l$, we have
  \[
  \chi(M(kD_a)) =
  r (kD_a\cdot (kD_a+C_m))/2
  +
  k r \mu
  +
  \chi(M),
  \]
  and
  \[
  \chi(M_l(kD_a))
  =
  r_l ((k-l)D_a\cdot ((k-l)D_a+C_m))/2
  +
  (k-l) r_l \mu_l
  +
  \chi(M_l(lD_a))
  .
  \]
  Since $\chi(M)>0$ and $\chi(M_l(lD_a))$ is uniformly bounded, it follows
  that for each $l$, there is a corresponding (uniform) $k_l>l$ such that
  either $M=M_l$ or $\chi(M(k_lD_a))>\chi(M_l(k_lD_a))$.  Moreover, we may
  apply Lemma \ref{lem:glob_generated_bounded} to the globally generated
  sheaf $M_l(lD_a)$ to see that if $k\ge l$ with
  \[
  \frac{r_l \mu_l+(r_l+1) l D_a^2 -D_a\cdot C_m}{D_a^2}<k,
  \]
  or more uniformly
  \[
  \frac{r \mu+(r+1) l D_a^2-D_a\cdot C_m}{D_a^2}<k,
  \]
  then $M_l(kD_a)$ is acyclic.  We may thus incorporate this as an
  additional constraint on $k_l$, so that $M_l(k_lD_a)$ is acyclic.  We
  then have $h^0(M(k_lD_a))\ge
  \chi(M(k_lD_a))>\chi(M_l(k_lD_a))=h^0(M_l(k_lD_a))$, from which it
  follows that for every $M$, either $M_l=M$ or $M_{k_l}$ is strictly
  larger than $M_l$.

  Define a sequence $L_i$ by $L_0=0$, $L_{i+1}=k_{L_i}$, so that for any
  $M$, the sequence $M_{L_i}$ is strictly increasing until it reaches $M$.
  It remains only to show that this subsequence reaches $M$ in a uniformly
  bounded number of steps.  There are three types of step to consider,
  depending on the dimension of $M_{L_{i+1}}/M_{L_i}$.

  If the quotient dimension is 2, then the rank increases, and thus clearly
  there can be at most $r$ such steps.  If the quotient dimension is 0,
  then since $M_{L_i}(L_{i+1}D_a)$ is acyclic, $M_{L_{i+1}}/M_{L_i}$ must
  be the entire $0$-dimensional subsheaf of $M/M_{L_i}$.  It follows that
  we cannot have a $0$-dimensional quotient in two consecutive steps.
  
  Finally, if the quotient dimension is 1, then $M_{L_{i+1}}$ has strictly
  larger slope than $M_{L_i}$.  Since $M_{L_i}(L_i D_a)$ is globally
  generated, its Chern class is effective, giving the bound
  \[
  \mu_{L_i}\ge -r_{L_i} L_i D_a^2 \ge -r L_i D_a^2,
  \]
  so
  \[
  \mu-\mu_{L_i} \le r L_i D_a^2.
  \]
  Every pair of consecutive steps increases the slope by at least $1/r$
  (taking a pair since every other step could have $0$-dimensional
  quotient), and thus after at most $2 r^2 L_i D_a^2$ steps, the slope
  would be forced to be greater than $\mu$, giving a uniform bound on the
  number of steps required to force the rank to increase or the sequence to
  terminate.
\end{proof}

\begin{rem}
  In general, the above bound depends on the surface.  However, if $D_a$ is
  universally ample such that $D_a-C_m$ is universally nef (not just nef as
  required above), then the above argument depends only on the
  combinatorial data, and thus gives a bound valid for any surface.
\end{rem}

\begin{cor}
  For each triple $(r,D,\chi)$ and each ample divisor $D_a$, the family of
  $D_a$-semistable sheaves with numerical invariants $(r,D,\chi)$ is
  bounded.
\end{cor}

\begin{proof}
  Indeed, it is dominated by an open subscheme of the appropriate $\Quot$
  scheme.
\end{proof}

\begin{cor}
  For any ample divisor class $D_a$, the subspace of $\Spl_X$ classifying
  stable sheaves with numerical invariants $(r,D,\chi)$ is of finite type.
  If $\gcd(r,D\cdot D_a,\chi)=1$, then the subspace is proper and admits a
  universal family.
\end{cor}

\begin{proof}
  Since a stable sheaf is simple, the functor is certainly represented by
  a smooth algebraic space.  The $\gcd$ condition ensures that it is not
  possible for a sheaf to be strictly semistable, and thus the moduli
  functor agrees with the semistable moduli functor, so is proper.  The
  existence of a universal family follows as in the 1-dimensional case,
  since the $\gcd$ condition ensures that $(r,D,\chi)$ is primitive.
\end{proof}

We would, of course, like to know that the semistable moduli space is
actually projective.  As in the rank 0 case, since we already know
properness, it would suffice to show quasiprojectivity, say by showing that
every semistable sheaf is semistable for the GIT quotient of the $\Quot$
scheme corresponding to a uniformly sufficiently large twist.  If there
were a function $C(r)$ (depending on $D_a$) such that
\[
\frac{h^0(M)}{D_a^2 \rank(M)}
\le
1/2
\max(
\frac{D_a\cdot c_1(M)}{D_a^2\rank(M)}
+
C(\rank(M))
,
0
)^2
\]
for any semistable torsion-free sheaf $M$, then this would prevent any
subsheaf with a large difference in slopes from having enough global
sections to be destabilizing for the GIT quotient.  (That is, any such
subsheaf will have strictly fewer than $(\rank(M')/\rank(M))h^0(M)$ global
sections.)  The remaining sheaves would have bounded difference in slope,
so the same would apply to the quotient sheaves.  We could thus deal with
those cases using the following analogue of
\cite[Lem. 1.7.9]{HuybrechtsD/LehnM:2010}.  (The constraint that the quotient be
torsion-free is not significant, as once the saturated subsheaf is acyclic
and globally generated, the original subsheaf will have strictly fewer
global sections, so cannot be destabilizing for the GIT quotient.)

\begin{cor}
  For any triple $(r,D,\chi)$ of numerical invariants, and any real number
  $\mu_0$, the set of torsion-free sheaves $M$ such that $c_1(M)\cdot D_a
  \le \mu_0\rank(M)$ and $M$ is a quotient of some semistable sheaf with
  invariants $(r,D,\chi)$ is bounded, as is the corresponding set of
  subsheaves.
\end{cor}

\begin{proof}
  Choose $m$ such that $M'(mD_a)$ is acyclic and globally generated for any
  semistable sheaf $M'$ with invariants $(r,D,\chi)$.  Then for any sheaf $M$
  satisfying the given conditions, $M(mD_a)$ is globally generated, and
  satisfies
  \[
  c_1(M(mD_a))\cdot D_a
  =
  c_1(M)\cdot D_a
  +
  \rank(M) m D_a^2
  \le
  \rank(M)(\mu_0+m D_a^2)
  \]
  and thus by Corollary \ref{cor:glob_gen_tf_finite_invs} there are only
  finitely many possibilities for the numerical invariants of $M(mD_a)$.
  It follows that the desired set of sheaves is covered by a subset of the
  corresponding union of $\Quot$ schemes.
\end{proof}

Unfortunately, the proof in \cite[Thm. 3.3.1]{HuybrechtsD/LehnM:2010} of
the commutative version of the desired bound involves induction on
intersections with linear subspaces, and thus does not carry over to the
noncommutative setting.  (We could deal with a single hyperplane section as
we did in the rank 0 case, i.e., by replacing global sections of a
hyperplane section with $\Ext^1(\sO_X/\sO_X(-lD_a),M)$, but this does not
work for an intersection with more than one hyperplane.)

\medskip

When this $\gcd$ condition is satisfied, so the stable moduli space is
proper, the fact that it could conceivably fail to be projective (or a
scheme!) does not prevent us from asking a couple of natural questions, to
wit whether we can understand line bundles on the moduli space, or more
generally the corresponding Chow ring.  (Indeed, intersection theory is as
well-behaved for smooth, proper algebraic spaces of finite type as it is
for smooth, projective schemes of finite type.  This follows as a special
case of the theory for Artin stacks, see \cite{KreschA:1999}.)

Fix $(r,D,\chi)$ satisfying the $\gcd$ condition, and let ${\cal M}$ be the
corresponding moduli space, with universal family $M_{\cal M}$.  Given any
sheaf $N\in \coh(X_{\rho;q;C})$, or more generally any object in the
bounded derived category, we may consider the family of complexes
$R\Hom(N,M_{\cal M})$, and observe that this may itself be represented by a
perfect complex on ${\cal M}$.  (Moreover, we may arrange for the ranks of
the terms to be constant, not just locally constant.)  This induces a
homomorphism $K_0(X_{\rho;q;C})\to K_0({\cal M})$ which we may compose with
the Chern class morphism to obtain a homomorphism $K_0(X_{\rho;q;C})\to
A^*{\cal M})^*$.  We also note the related morphism $K_0(X_{\rho;q;C})\to
\Pic({\cal M})$ given by $\det R\Hom(N,M_{\cal M})$, compare
\cite{DrezetJ-M:1988}.

In general, we may always modify a universal family by twisting by a line
bundle on the base, and this will have a nontrivial effect on the
associated Chern morphism, since it tensors $R\Hom(N,M_{\cal M})$ by the
corresponding line bundle.  Luckily, the choice of universal family has no
effect on the validity of the following result.

\begin{thm}
  Let $D_a$ be an ample divisor class on $X_{\rho;q;C}$, and let
  $(r,D,\chi)$ be a tuple of numerical invariants with $r>0$ such that
  $\gcd(r,D\cdot D_a,\chi)=1$.  Let ${\cal M}$ be the corresponding stable
  moduli space, with universal family $M_{\cal M}$.  Then ${\cal M}$ is
  either connected or empty, and its Chow ring is generated by the Chern
  classes of $R\Hom(N,M_{\cal M})$ where $N$ ranges over a system of
  generators of $K_0(X_{\rho;q;C})$.
\end{thm}

\begin{proof}
  We follow the approach of \cite{EllingsrudG/StrommeSA:1993}, in
  particular Theorem 2.1 op cit.; see also \cite{MarkmanE:2007} for general
  commutative Poisson surfaces.  Note that although this is stated only for
  projective varieties, the argument applies equally well to any smooth
  proper algebraic space of finite type.  We are thus led to consider the
  class of the diagonal in $A^*({\cal M}\times {\cal M})$.  If $M_1$, $M_2$
  are two $D_a$-stable sheaves with numerical invariants $(r,D,\chi)$, then
  $\Ext^2(M_1,M_2)\cong \Hom(M_2,\theta M_1)^*$ and since $M_1$ is
  transverse to $C_m$, we find that $\Hom(M_2,\theta M_1)\subset
  \Hom(M_2,M_1)$.  Stability ensures that any morphism $M_2\to M_1$ is an
  isomorphism, and thus $\Ext^2(M_1,M_2)=0$.  It follows that if we let
  $M_1$, $M_2$ be two instances of $M_{\cal M}$, then $R\Hom(M_1,M_2)$ is
  not only represented by a perfect complex on ${\cal M}\times {\cal M}$,
  but we may truncate it to a perfect complex supported in degrees 0 and 1.
  Moreover, the diagonal is precisely the locus where this complex has
  cohomology in degree 0, and it is straightforward to verify that any
  point of the diagonal has multiplicity 1 in the corresponding
  determinantal subscheme.  Moreover, the determinantal subscheme has
  expected codimension $1-\chi(M_1,M_2)=\dim {\cal M}$, and thus its
  codimension is as expected.  We may thus apply Porteous' formula to
  compute the class of the diagonal:
  \[
  c_{\dim{\cal M}}(R\Hom(M_1,M_2)[1]).
  \]

  Now, let $E_1,\dots,E_{m+2}$ be a full exceptional collection in
  $X_{\rho;q;C}$, and let $E^\perp_{m+2},\dots,E^\perp_1$ be the ``dual''
  exceptional collection, so that $R\Hom(E_i,E^\perp_j)\cong
  k^{\delta_{ij}}$.  Then we have the following identity in the
  Grothendieck group of ${\cal M}\times {\cal M}$:
  \[
    [R\Hom(M_1,M_2)[1]]
    =
    -\sum_i \pi_1^*[R\Hom(M_1,E^\perp_i)]\otimes \pi_2^*[R\Hom(E_i,M_2)]
  \]
  It follows immediately that $c_{\dim{\cal M}}(R\Hom(M_1,M_2)[1])$ is a
  homogeneous polynomial in the Chern classes of
  $\pi_1^*[R\Hom(M_1,E^\perp_i)]$ and $\pi_2^*[R\Hom(E_i,M_2)]$.  (This
  follows from the formula for the Chern classes of a tensor product; note
  in particular that we need to use the fact that the two factors have
  constant rank.)

  Now, let $z$ be any class in $A^*{\cal M}$.  We have
  \[
  z = \pi_{2*} (\Delta\cap \pi_1^*z);
  \]
  since $\Delta$ is a homogeneous polynomial in the given Chern classes,
  this gives an expression for $z$ as a polynomial in the Chern classes of
  $R\Hom(E_i,M_{\cal M})$.  In particular, if we take $z$ to be the
  fundamental class of a component (assuming ${\cal M}$ is nonempty!), then
  the resulting expression gives $z$ as an integer multiple of the
  fundamental class of $M_{\cal M}$, which must therefore be connected.
\end{proof}

\begin{rem}
  If $D_a$ is universally ample, then the condition that ${\cal M}$ is
  nonempty is an open condition on the parameters $\rho;q;C$.  Indeed, each
  moduli space is covered by an open subscheme of a uniformly chosen
  $\Quot$ scheme, and thus the same holds for the relative moduli space.
  For each surface $X_{\rho;q;C}$ in that open subfamily, we have an
  expression for the Chow ring of ${\cal M}$ as a quotient of a fixed
  homogeneous polynomial ring.  The degree $\dim{\cal M}$ part of this
  morphism can be computed as an Euler characteristic, and is thus
  (locally, so globally) constant, so that the corresponding part of the
  ideal is also constant.  (Here, we use the fact that since the
  normalization of the universal family depends only on a choice of
  numerical invariant such that $\chi(N,M)=1$, the universal family is
  defined globally.)  More generally, a homogeneous polynomial of degree
  $d$ in the Chern classes is numerically trivial iff its product by any
  homogeneous polynomial of degree $\dim{\cal M}-d$ in the Chern classes is
  numerically trivial, and thus the kernel of the morphism to the {\em
    numerical} Chow ring is locally constant in every degree.  The argument
  of \cite[Thm. 2.1]{EllingsrudG/StrommeSA:1993} tells us that numerical
  equivalence is equivalent to rational equivalence on ${\cal M}$, and thus
  the true ideal is itself locally constant: on any component of the locus
  of surfaces on which the moduli space is nonempty and projective, the
  Chow rings of distinct fibers are canonically isomorphic.  In particular,
  the Chow ring of a given stable moduli space is determined entirely by
  the corresponding {\em combinatorial} data.
\end{rem}

\begin{rem}
  The difficulty in extending this to rank 0 is that there may be stable
  sheaves $M$, $M'$ with the given invariants such that $\Ext^2(M,M')\ne
  0$, and thus the complex $R\Hom(M,M')$ cannot be represented as a
  two-term complex.  Indeed, this will always happen when $c_1(M)\cdot
  C_m=0$, since then $M$ and $\theta M$ are isomorphic.  (It is
  straightforward to verify that if $m<8$, then $\Ext^2(M,M')=0$ for
  (semi)stable sheaves, but this excludes the surfaces with the most
  interesting moduli spaces.)
\end{rem}

\begin{cor}
  With notation and hypotheses as above, the morphism $K_0(X_{\rho;q;C})\to
  \Pic({\cal M})$ is surjective.
\end{cor}

\begin{proof}
  The class group is generated by the appropriate first Chern classes;
  since ${\cal M}$ is smooth, its class group agrees with its Picard group.
\end{proof}

\begin{rem}
  This leads to the question of how to express the canonical bundle in this
  form.  Since the tangent sheaf is $\Ext^1(M,M)$ and $\Hom(M,M)$ is the
  trivial bundle, we see that we may identify the canonical bundle with
  $\det R\Hom(M,M)$.  This, in turn, may be computed as the restriction to
  the diagonal of $\det R\Hom(M_1,M_2)$, which in turn simplifies using the
  seesaw principle:
  \begin{align}
  \det R\Hom(M_1,M_2)
  &\cong
  \det R\Hom(M_1,[M])
  \otimes
  \det R\Hom([M],M_2)\notag\\
  &\cong
  \det R\Hom(-\theta^{-1}[M],M_1)
  \otimes
  \det R\Hom([M],M_2).\notag
  \end{align}
  Restricting to the diagonal gives
  \[
  \det R\Hom(M,M)\cong \det R\Hom([M]-\theta^{-1}[M],M)
  \cong
  \det R\Hom([\theta^{-1}M|_C],M).
  \]
\end{rem}

\medskip

In particular, we note that the $\gcd$ hypothesis always holds in the rank
$1$ case, so that the above results in particular hold for the Hilbert
scheme.  Here, there is a particularly nice choice of normalization coming
from the fact that $\chi(\sO_x,I)=1$, so that there is a unique universal
family such that $R\Hom(\sO_X,I)$ is trivial.  In the commutative case,
this agrees with the natural universal family coming from the
interpretation as a $\Quot$ scheme.  Indeed, there we have a universal
short exact sequence
\[
0\to I\to \sO_X\to \sO_Z\to 0,
\]
and $\det R\Hom(\sO_x,\sO_Z)$ is trivial since $R\Hom(\sO_x,\sO_Z)$ is
supported in codimension 2, so that $\det R\Hom(\sO_x,I)\cong \det
R\Hom(\sO_x,\sO_X)\cong \det R\Hom(\sO_x,\sO_Z)^{-1}$ is also trivial.

We thus conclude as above that the Chow ring of $\Hilb^n(X_{\rho;q;C})$
depends only on $n$ and $m$, and that the Picard group is generated by the
image of the Grothendieck group.

\begin{prop}
  The morphism $K_0(X_{\rho;q;C})\to \Pic(\Hilb^1(X_{\rho;q;C}))\cong
  \Pic(X_{\rho;1;C})$ takes $(r,D,\chi)$ to $D$.
\end{prop}

\begin{proof}
  We first observe that $K_0(X_{\rho;q;C})$ is generated by the class of a
  point and the image of ${\cal N}_{X_{\rho;q;C}}$, and thus it suffices to
  determine how the morphism acts on ${\cal N}_{X_{\rho;q;C}}$.  Since the
  ideal sheaves are also in that subcategory, we may apply $\kappa_q$ and
  thus find that the morphism is independent of $q$, allowing us to reduce
  to the commutative case.

  Now, let $D$ be any divisor class, and compute that
  \[
  c_1(R\Hom(\sO_X(D),I))
  =
  -c_1(R\Hom(\sO_X(D),\sO_x))
  =
  -c_1(R\Gamma(\sO_x(-D)))
  =
  D.
  \]
  Since $K_0(X_{\rho;1;C})$ is generated by line bundles, the claim follows.
\end{proof}
  
\begin{cor}
  For $n>1$, the kernel of the morphism $K_0(X_{\rho;q;C})\to
  \Pic(\Hilb^n(X_{\rho;q;C}))$ is generated by the class of a point, and
  thus $\Pic(\Hilb^n(X_{\rho;q;C}))\cong \Z^{m+1}$.
\end{cor}

\begin{proof}
  Compose with the restriction to a general fiber
  $\Hilb^1(X_{\rho,qz_1,\dots,qz_{n-1};q;C})$ of ${\cal H}_{n-1}$, and
  observe that, by adjunction, we may express the result as a composition
  \[
  K_0(X_{\rho;q;C})\to K_0(X_{\rho,qz_1,\dots,qz_{n-1};C})\to
  \Pic(\Hilb^1(X_{\rho,qz_1,\dots,qz_{n-1};q;C}))
  \]
  The second map has kernel of rank 2, and only the class of a point is in
  the image of the first map.
\end{proof}

\section{Derived equivalences}
\label{sec:derived}

In the proof of Theorem \ref{thm:rd_painleve}, we constructed a number of
derived equivalences between our noncommutative surfaces and commutative
rational surfaces.  Another such equivalence is suggested by Proposition
\ref{prop:points_q_torsion} in comparison with
\cite[Thm.~7.1]{rat_Hitchin}: the relative $\Pic^r$ of a rational elliptic
surface with an $r$-tuple fiber $rC$ is given an explicit description there
which is naturally isomorphic to a moduli space of the form of Proposition
\ref{prop:points_q_torsion}.  In fact, we expect there to be a closer
relation: the relative $\Pic^r$ is not a fine moduli space, and thus gives
rise to a natural sheaf of central simple algebras on the stable locus; it
is thus natural to expect that sheaf of central simple algebras to be an
Azumaya algebra of the form $X_{\rho;q;C}$.

Together, these results suggest that we should have a more general family
of derived equivalences between surfaces of the form $X_{z,w}$.  One such
equivalence is straightforward.

\begin{lem}
  There is an equivalence $\Phi_{\sO_{e_8}(-1)}:\coh X_{z/q,q}\cong \coh
  X_{z,q}$ given by $M\mapsto \theta M(-e_8)$.
\end{lem}

It is instructive to consider how this equivalence interacts with the
semiorthogonal decomposition corresponding to the blowdown from $X_8$ to
$X_7$.  Any object $M\in D^b\coh X_8$ fits into a distinguished triangle
\[
N\otimes_k \sO_{e_8}(-1)\to M\to L\alpha_8^! M'\to
\]
where $M'\in D^b\coh X_7$ and $N\in D^b\coh \Spec(k)$.  Applying $M\mapsto
\theta M(-e_8)$ and using the relation between $\alpha_8^*$ and
$\alpha_8^!$ gives a distinguished triangle
\[
N\otimes_k \sO_{e_8}(-1)\to \theta M(-e_8)\to
L\alpha_8^!\theta M'\to
\]
In other words, the given derived equivalence acts trivially on one part of
the semiorthogonal decomposition, and acts as a shift of the induced Serre
functor on the other part.

In general, we could attempt to do something similar replacing
$\sO_{e_8}(-1)$ by any exceptional object $E$.  Indeed, for any exceptional
object $E$, \cite{BondalAI/KapranovMM:1990} tells us that there is a
corresponding semiorthogonal decomposition, and an induced Serre functor on
the kernel of ${\cal K}_E$ of $R\Hom(E,\_)$.  In our case, the induced Serre
functor has the form $\theta_{{\cal K}_E}[2]$ where there is a functorial
distinguished triangle
\[
\theta_{{\cal K}_E} M
\to
\theta M
\to
R\Hom_k(R\Hom(M,E),\theta E)
\to
\]
(As usual, we use the fact that our category, as a derived category, has a
natural dg-enhancement, so we have functorial cones.)  We want to construct
a derived functor by taking the distinguished triangle
\[
E\otimes_k R\Hom(E,M)\to M\to N\to
\]
with $N\in {\cal K}_E$ to a distinguished triangle of the form
\[
E\otimes_k R\Hom(E,M)\to M\to \theta_{{\cal K}_E} N\to
\]
Of course, the complication here is that the latter distinguished triangle
lies in a different derived category with its own semiorthogonal
decomposition.  We would like the new derived category to also be of the
form $D^b\coh X'_m$, and thus need to have a natural isomorphism
\[
R\Hom_{X'_m}(\theta_{{\cal K}_E} M,E)
\cong
R\Hom_{X_m}(M,E)
\]
We thus need to understand these two functors.  In general, if $M\in {\cal
  K}_E$, then $R\Hom(M,\theta E)=R\Hom(E,M)=0$, and thus $R\Hom(M,E)\cong
R\Hom_C(M|^{\dL}_C,E^{\dL}_C)$.  Thus to understand the functor
$R\Hom(\theta_{{\cal K}_E}\_,E)$, we just need to understand $(\theta_{{\cal
    K}_E}\_)|^{\dL}_C$.  Here, we have a distinguished triangle
\[
(\theta_{{\cal K}_E} M)|^{\dL}_C
\to
\theta M|^{\dL}_C
\to
R\Hom_k(R\Hom(\theta M|^{\dL}_C,\theta E|^{\dL}_C),\theta E|^{\dL}_C)
\to
\]
This, of course, we can compute entirely inside $D^b\coh C$.  In fact, up
to an application of $\theta$ (which on $C$ simply twists by a line
bundle), this is just a twist functor \cite{SeidelP/ThomasR:2001} in the
spherical object $\theta E|^{\dL}_C$.  As a result, we can compute
$(\theta_{{\cal K}_E} M)|^{\dL}_C$ as the image of $M|^{\dL}_C$ under a
derived autoequivalence of $D^b\coh C$.  But then to understand
$R\Hom(\theta_{{\cal K}_E}M,E)$ for general $M$, we just need to compute
the image of $E|^{\dL}_C$ under the inverse autoequivalence.  This is
easily computed as the sheaf $F$ on $C$ in the distinguished triangle
\[
R\Hom_C(E|^{\dL}_C,\theta^{-1}E|^{\dL}_C)\otimes_k E|^{\dL}_C
\to
\theta^{-1} E|^{\dL}_C
\to
F
\to
\]
using the standard description of the inverse of a spherical twist functor.

Now, $X_{\rho;q;C}$ is naturally a sort of derived blowup of ${\cal K}_E$
in the sheaf $E|^{\dL}_C$, and thus the above construction gives an
alternate description as a derived blowup of ${\cal K}_E$ in the sheaf
$F$.  Most of the time, this cannot be simplified further, but if $F$ is a
deformation of $E|^{\dL}_C$, we can hope to recognize the new description
as a different noncommutative rational surface.  This is what happens in
the case $E=\sO_{e_m}(-1)$, as the sheaf $F$ is then just the structure sheaf of
$qx_m$, giving a derived equivalence $D^b\coh X_{\rho;q;C}\cong
D^b\coh X_{\rho;q;C}(-e_m)$ (which of course in that case is an actually
abelian equivalence).

When $m=8$, it turns out that the case $E=\sO_X$ also works, giving
us the following.

\begin{lem}
  There is an equivalence $\Phi_{\sO_X}:D^b\coh X_{z,qz}\cong D^b\coh X_{z,q}$
  which takes $\sO_X$ to $\sO_X$ and on ${\cal K}_{\sO_X}$ acts
  as $\kappa_q^{-1}\kappa_{qz}\theta_{{\cal K}_{\sO_X}}$.
\end{lem}

\begin{proof}
  By \cite{OrlovD:2014}, to specify an equivalence between dg-categories
  equipped with semiorthogonal decompositions, it suffices to give
  equivalences between the two subcategories that are functorial for maps
  between the two subcategories.  In other words, we just need to show that
  we have a functorial isomorphism
  \[
  R\Hom(M,\sO_X)
  \cong
  R\Hom(\kappa_q^{-1}\kappa_{qz} \theta_{{\cal K}_{\sO_X}} M,\sO_X).
  \]
  for $M\in {\cal K}_{\sO_X}\cong \cN_{\rho;qz;C}$.  But this reduces
  to the above calculation on $C$.  Note that for a generic ineffective
  line bundle ${\cal L}$ of degree 0 on $C$, we have
  \[
  \kappa_q^{-1}\kappa_{qz} \theta_{{\cal K}_{\sO_X}} {\cal L}
  \cong {\cal L}
  \]
  and thus the inverse of the relevant derived equivalence preserves
  $\sO_C$ as required.
\end{proof}

\begin{rem}
  Note that since Serre functors are unique, we can compute the Serre
  functor on $\cN_{\rho;qz;C}$ by conjugating the induced Serre functor on
  $\cN_{\rho;1;C}$ by $\kappa_{qz}$.
\end{rem}

Combining the above operations gives the following.

\begin{thm}\label{thm:sl2Z_derived}
  For any element $\begin{pmatrix} a&b\\c&d\end{pmatrix}\in \SL_2(\Z)$,
    there is a derived equivalence $D^b\coh X_{z,q}\cong D^b\coh
    X_{z^aq^b,z^cq^d}$, which acts as an autoequivalence on $D^b\coh C$.
\end{thm}

\begin{rems}
  Here we should caution the reader that this does not correspond to an
  action of $\SL_2(\Z)$, as the relations of $\SL_2(\Z)$ correspond to
  nontrivial autoequivalences (acting on the Grothendieck group in the same
  way as suitable powers of $\theta$).  If we allow contravariant
  equivalences, we can of course include the duality $R\ad$, extending
  $\SL_2(\Z)$ to $\GL_2(\Z)$.  Note that since $R\ad$ is contravariant, it
  inverts the Serre functors, and thus each of $\Phi_{\sO_X}$
  or $\Phi_{\sO_{e_8}(-1)}$ conjugates to its inverse under $R\ad$.
  We should also note that the abelian equivalence
  $X_{z,q}\cong X_{1/z,1/q}$ does {\em not} correspond to an element of the
  above family, since the corresponding autoequivalence of $D^b\coh C$ is
  the action of a hyperelliptic involution.
\end{rems}

\begin{rems}
  It is useful to record how the above operations act on $K_0(X_{z,q})$.
  Note first that for any root $\alpha$ of $E_8\subset E_9$, the two
  generating equivalences preserve the class with numerical invariants
  $(0,\alpha,0)$, as this class is in both ${\cal K}_{\sO_{e_8}(-1)}$ and
  ${\cal K}_{\sO_X}$ and is invariant under $\theta$.  We thus need only
  determine how they act on the $4$-dimensional orthogonal complement in
  $K_0(X_{z,q})$ of that $8$-dimensional space, or in other words on
  classes of the form $(r,se_8+tC_8,u)$.  We find that the action on such
  classes of the derived equivalence $D^b\coh X_{z,q}\to D^b\coh X_{qz,q}$
  is
  \begin{align}
    (1,0,0)&\mapsto (1,-e_8-C_8,0)\notag\\
    (0,e_8,0)&\mapsto (0,e_8,0)\notag\\
    (0,C_8,0)&\mapsto (0,C_8,-1)\notag\\
    (0,0,1)&\mapsto (0,0,1)\notag
  \end{align}
  while the action of $D^b\coh X_{z,q}\to D^b\coh X_{z,q/z}$ is
  \begin{align}
    (1,0,0)&\mapsto (1,-C_8,0)\notag\\
    (0,e_8,0)&\mapsto (1,e_8-C_8,0)\notag\\
    (0,C_8,0)&\mapsto (0,C_8,0)\notag\\
    (0,0,1)&\mapsto (0,C_8,1).\notag
  \end{align}
  The images of $(0,C_8,0)$ and $(0,0,1)$ in each case follow immediately
  from the action on parameters, and the remaining degree of freedom (given
  that we must respect $\chi(\_,\_)$) is determined by the known fixed
  element $(0,e_8,0)$, resp. $(1,0,1)$.  For the general equivalence
  $D^b\coh X_{z,q}\to D^b\coh X_{z^a q^b,z^c q^d}$, we cannot be quite as
  precise (due to the aforementioned central extension); all we can say is
  that
  \begin{align}
    (1,0,0)&\mapsto (d,-be_8+\frac{dh-b+c}{2} C_8,\frac{-bh-b-a+d}{2})\\
    (0,e_8,0)&\mapsto (-c,ae_8+\frac{-ch+c+a-d}{2} C_8,\frac{ah+b-c}{2})\\
    (0,C_8,0)&\mapsto (0,dC_8,-b)\\
    (0,0,1)&\mapsto (0,-cC_8,a)
  \end{align}
  for some $h\in ab+bc+cd+2\Z$ (corresponding to the application of a power
  of $\theta$).  Note that this gives a non-split extension of $\SL_2(\Z)$
  by $\Z$, split (by taking $h=0$) over the index 2 (congruence) subgroup
  on which $ab+bc+cd\in 2\Z$.
\end{rems}

One application of these equivalences is to finish up the proof of
Theorem \ref{thm:rd_painleve}.

\begin{cor}\label{cor:rd_painleve_exists}
  For any pair $r>0$, $d$ with $\gcd(r,d)=1$, there exists a stable sheaf
  on $X_{w^{-d},w^r}$ with numerical invariants $(0,rC_m,d)$ and disjoint
  from $C$.
\end{cor}

\begin{proof}
  Let $\Phi:D^b\coh X_{w,1}\cong D^b\coh X_{w^{-d},w^r}$ be a derived
  equivalence of the form guaranteed by the Theorem.  We claim that if $x$
  is a generic point of $X_{w,1}$, then (up to some constant shift) $\Phi
  \sO_x$ is a stable sheaf with the desired numerical invariants and
  disjoint from $C$.  Each of these conditions is an intersection of open
  conditions, so it will suffice to find points $x_0$, $x_1$ such that
  $\Phi \sO_{x_0}$ is a stable sheaf with the correct invariants and $\Phi
  \sO_{x_1}$ is disjoint from $C$.  The latter is easy, since any derived
  equivalence between categories of our form respects restriction to $C$
  (up to an autoequivalence of $D^b\coh C$), so $\Phi \sO_{x_1}|^{\dL}_C=0$
  whenever $x_1\notin C$.  For the former, we take $x_0\in C$.  Since
  $\Phi$ acts as a derived autoequivalence on $C$, it follows that $\Phi
  \sO_{x_0}$ is an even shift of a simple sheaf on $C$ with rank $r$ and
  degree $d$.  Thus (fixing the shift as necessary) we conclude that $\Phi
  \sO_{x_0}$ is a stable vector bundle on $C$, and thus stable as a sheaf
  on $X_{w^{-d},w^r}$.
\end{proof}

\begin{rems}
  Of course, the intersection of open sets where the above construction
  works will depend on the choice of stability condition.  In particular,
  unlike the derived equivalence constructed in the proof of Theorem
  \ref{thm:rd_painleve}, we cannot expect that {\em every} point will map
  to a stable sheaf, especially in situations with multiple inequivalent
  stability conditions.
\end{rems}

\begin{rems}
  The construction given after the proof of Theorem \ref{thm:rd_painleve}
  is actually a special case of the above construction; it is
  straightforward to apply the equivalence $D^b\coh X_{z,q}\to D^b\coh
  X_{z,z^{-r} q}$ to a generic point sheaf in the case $q=1$, as long as
  $r<|\langle z\rangle|$.  What happens when $r\ge |\langle z\rangle|$ is
  that the internal Serre functor begins to diverge from the usual Serre
  functor; if one replaces $\theta^l I$ by $\theta_{{\cal K}_{\sO_X}}^l I$,
  the construction works in general.
\end{rems}

We single out one particular composition of the above functors:
\[
\Psi:=\theta^{-1}\Phi_{\sO_{e_8}(-1)}\Phi_{\sO_X}\Phi_{\sO_{e_8}(-1)}
       \cong (\Phi_{\sO_X}\theta \_(-e_8))(-e_8)
\]
taking $D^b\coh X_{z,q}\to D^b\coh X_{q,1/z}$.

\begin{lem}
  We have $\Psi(\sO_{e_8})\cong\sO_X$ and $\Psi(\sO_X)\cong\sO_{e_8}(-1)[-1]$.
\end{lem}

\begin{proof}
  Using the fact that $\theta$ commutes with any derived equivalence, we
  may rewrite the second claim using
  \[
  \Psi=\Phi_{\sO_{e_8}(-1)}\Phi_{\sO_X}\theta^{-1}\Phi_{\sO_{e_8}(-1)},
  \]
  and thus reduce to showing
  \begin{align}
    \Phi_{\sO_X}(\sO_{e_8}) &\cong \sO_X(e_8),\\
    \Phi_{\sO_X}(\sO_X(-e_8)) &\cong \sO_{e_8}(-1)[-1].
  \end{align}
  These are equivalent via the distinguished triangle
  \[
  \Phi_{\sO_X}(\sO_X(-e_8))\to \Phi_{\sO_X}(\sO_X)\to
  \Phi_{\sO_X}(\sO_{e_8})\to
  \]
  with $\Phi_{\sO_X}(\sO_X)\cong \sO_X$, so we need merely compute
  $\Phi_{\sO_X}(\sO_X(-e_8))$.  For that, we note the distinguished
  triangle
  \[
  \theta_{{\cal K}_{\sO_X}} \sO_X(-e_8)
  \to
  \theta \sO_X(-e_8)
  \to
  \theta \sO_X
  \to
  \]
  so that ${\cal K}_{\sO_X} \sO_X(-e_8)\cong \sO_{e_8}(-1)[-1]$; the action of
  $\kappa_q^{-1}\kappa_{qz}$ is then trivial to compute.
\end{proof}

This lets us prove a relation between our derived equivalences.

\begin{cor}
  One has a natural isomorphism
  \[
  \Phi_{\sO_{e_8}(-1)} \Phi_{\sO_X} \Phi_{\sO_{e_8}(-1)}
  \cong
  \Phi_{\sO_X} \Phi_{\sO_{e_8}(-1)} \Phi_{\sO_X}.
  \]
\end{cor}

\begin{proof}
  Consider the two compositions
  \[
  \Psi\circ \Phi_{\sO_X}
  \qquad\text{and}\qquad
  \Phi_{\sO_{e_8}(-1)}\circ\Psi.
  \]
  Since $\Psi(\sO_X)=\sO_{e_8}(-1)[-1]$, we find that $\Psi$ takes $\ker
  R\Hom(\sO_X,\_)$ to $\ker R\Hom(\sO_{e_8}(-1),\_)$.  Since each of
  $\Phi_{\sO_X}$ and $\Phi_{\sO_{e_8}(-1)}$ acts as the internal Serre
  functor on the appropriate subcategory, and Serre functors are canonical,
  we conclude that the two compositions are naturally equivalent.  Applying
  $\theta \Phi_{\sO_{e_8}(-1)}^{-1}$ to both sides and simplifying gives
  the desired result.
\end{proof}

\begin{rem}
  The action on $K_0(X_{z,w})$ suggests that $\theta^{-2}\Psi^4[2]$ should
  be the identity.  If $2C_8+e_8$ is ample, then this is true:
  $\theta^{-2}\Psi^4[2]$ preserves $\sO_X$ and commutes with
  $\theta^{-1}\Phi_{\sO_{e_8}(-1)}$ and thus preserves $\sO_X(r(2C_8+e_8))$
  for all $r$.  Since it preserves a set of generators, it must be trivial.
  On the other hand, it is easy to see that if $M$ is any nonzero sheaf
  with $M\cong \theta M$, $R\Gamma(M)=R\Hom(\sO_{e_8}(-1),M)=0$, then
  $\Psi(M)=M$, so $\theta^{-2}\Psi^4 M[2]\cong M[2]\not\cong M$.  Such a
  sheaf exists whenever there is a root $\alpha$ of $E_8$ with
  $\rho(\alpha)=1$.  Presumably one finds in such cases that
  $\theta^{-2}\Psi^4[2]$ is a composition of reflection functors, compare
  Theorem \ref{thm:all_derived_equivalences}.
\end{rem}

Another interesting consequence is that we can find derived categories of
noncommutative surfaces with $K^2=1$ inside derived categories of {\em
  commutative} surfaces.

\begin{cor}
  There is an isomorphism
  \[
  D^b\coh X_{\eta,x_0,\dots,x_7;q;C}\cong \cN_{\eta,x_0,\dots,x_7,\eta^2 x_0^3/x_1\cdots x_7q;1;C}.
  \]
\end{cor}

\begin{proof}
  We can think of $D^b\coh X_{\eta,x_0,\dots,x_7;q;C}$ via $\alpha_8^*$ as
  the subcategory of $D^b\coh X_{z,q}$ in which $R\Hom(\sO_{e_8},M)=0$.
  Thus to establish the above isomorphism, we simply need to apply a derived
  equivalence of the form
  \[
  D^b\coh X_{z,q}\to D^b\coh X_{z',q'}
  \]
  taking $\sO_{e_8}$ to $\sO_X$.  As we have seen, $\Psi$ fills the bill.
\end{proof}

\begin{rems}
  In fact, any $D^b\coh X_m$ is a triangulated subcategory of some $D^b\coh
  Y$ for a commutative variety $Y$, as follows from results of
  \cite{OrlovD:2014}, but the resulting varieties $Y$ have higher dimension
  than coming from the above construction, which induces an embedding in
  the derived category of a commutative surface whenever $m\le 7$.
\end{rems}

%
%
%
%

Another case of interest starts with $X_{z,q}$ with $q$ torsion of order
$l$.  In that case, we can ask how the derived equivalences act on simple
$0$-dimensional sheaves of degree $l$.  The argument of Corollary
\ref{cor:rd_painleve_exists} tells us that the result will generically be a
semistable sheaf disjoint from $C$, and gives us the following consequence.

\begin{prop}\label{prop:no_fine_moduli_space}
  Suppose $ad-rb=1$, with $r>0$.  Then for $q$ of order $l$ and arbitrary
  $z$, the moduli space of stable sheaves on $X_{z^d q^{-b},z^{-r} q^a}$
  with invariants $(0,lrC_8,ld)$ is nonempty.  Moreover, the obstruction to
  the existence of a universal family over the generic point of the
  resulting component of the moduli space is given by the generic fiber of
  the Azumaya algebra $X_{z,q}$.
\end{prop}

\begin{proof}
  The one thing we need to argue is that the sheaves are actually
  generically stable, not just semistable.  The only way a sheaf disjoint
  from $C$ with these numerical invariants can fail to be irreducible is if
  it has a subsheaf with Chern class a root of $W(E_9)$, and only finitely
  many such roots can appear.  Since the sheaves supported on $C$ do not
  have morphisms to or from such root sheaves, the generic sheaf also has
  no such morphisms.
\end{proof}

This gives us a generalization of \cite[Thm.~7.1]{rat_Hitchin} (or, more
precisely, the special case of that result with the additional hypothesis
that the anticanonical curve is smooth; we will remove that hypothesis in
\cite{noncomm2}).  The idea here is that if $z$ is $l$-torsion, then the
linear system $|lC_8|$ on $X_{z,1}$ is an elliptic pencil with a fiber of
the form $lC$.  Then for any degree $d$, we may compute the relative
$\Pic^d$ of this fibration, and the result will have a canonical minimal
proper regular model, which we can hope to identify.

\begin{cor}
  Let $z$ be an $l$-torsion point of $\Pic^0(C)$, let $d$ be an integer,
  and let $\phi:C\to C'$ be the $\gcd(l,d)$-isogeny with kernel $\langle
  z^{l/\gcd(l,d)}\rangle$.  Then the minimal proper regular model of the
  relative $\Pic^d$ of the elliptic surface
  \[
  X_{\eta,x_0,\dots,x_7,\eta^2 x_0^3/x_1\cdots x_7 z;1;C}
  \]
  is isomorphic to
  \[
  X_{\phi(\eta),\phi(x_0),\dots,\phi(x_7),\phi(\eta^2 x_0^3/x_1\cdots x_7 z^a);1;C'},
  \]
  where $a$ is such that $ad\cong\gcd(l,d)\bmod l$.
\end{cor}

\begin{proof}
  Set $X:=X_{z,1}$, $Y:=X_{z^a,z^{l/\gcd(l,d)}}$, and let $Z$ denote the
  putative minimal proper regular model, a.k.a. the center of $Y$.  The
  constraint on $a$ implies the existence of a unique matrix
  \[
  \begin{pmatrix} a&b\\l/\gcd(l,d)&d/\gcd(l,d)\end{pmatrix}\in \SL_2(\Z)
  \]
  and thus a corresponding derived equivalence $D^b\coh Y\cong D^b\coh X$.

  Now let $M_X$ be the moduli space of simple sheaves on $X$ with
  invariants $(0,lC_8,d)$, and let $M_Y$ be the moduli space of simple
  sheaves on $Y$ with invariants $(0,0,\gcd(l,d))$.  The derived
  equivalence has the correct effect on invariants to map sheaves in $M_Y$
  to sheaves in $M_X$, and the usual argument tells us that it generically
  takes sheaves to sheaves (up to an overall shift which we can correct if
  necessary) and respects disjointness from $C$.  In fact something
  stronger is true: since Serre functors of derived categories are unique,
  the derived equivalence commutes with $\theta$ and thus with $\theta^l$,
  so in particular identifies the spaces of natural transformations
  $\theta^l\to \id$ on either side.  On the commutative surface $X$, such
  natural transformations are in natural bijection with the global sections
  of $\omega_X^{-l}$, and thus, mod scalars, correspond to fibers of the
  elliptic fibration on $X$.  For $Y$, we note that the global sections of
  the anticanonical bundle on $Z$ induce such natural transformations, and
  thus by dimensionality exhaust the space of such global sections.  In
  particular, we have a natural identification between the fibers of $X$
  and the fibers of $Z$.  We thus find that the image of a sheaf
  supported on a given fiber of $Z$ maps via the derived equivalence to an
  object supported on the corresponding fiber of $X$.  (That is, the
  corresponding morphism $\theta^l M\to M$ is 0.)

  In particular, given any point of $Z$ not on $C$, we may take a simple
  module over the corresponding fiber of the Azumaya algebra and apply the
  derived equivalence, and the result will generically be a line bundle of
  degree $d$ on the corresponding fiber of $X$.  As a result, the derived
  equivalence identifies the generic point of the generic fiber of $Z$ with
  the generic point of $\Pic^d$ of the generic fiber of $X$.  Since a
  smooth curve is determined by its function field, this is enough to
  determine the entirety of $\Pic^d$ of the generic fiber of $X$.  Since
  $Z$ is a minimal proper regular elliptic fibration with the same generic
  fiber as the relative $\Pic^d$ of $X$, it must in fact be the minimal
  proper regular model as required.
\end{proof}

\begin{rem}
  We could, of course, also consider the relative moduli space of stable
  vector bundles of rank $r$ and degree $d$, but this is generically the
  same as the relative $\Pic^d$, so has the same minimal proper regular
  model.  We should also note here that $\phi(z^a)$ is independent of the
  choice of $a$, since by assumption
  $\phi(z^a)^{d/\gcd(l,d)}=\phi(z)$ and $d/\gcd(l,d)$ is relatively prime
  to $|\ker\phi|$.
\end{rem}

\medskip

Since we have a reasonable understanding of the Hilbert scheme, it is
natural to investigate how the Hilbert scheme interacts with the derived
equivalences, in particular how the equivalences act on ``ideal sheaves''.
Most of the time, this will give sheaves of high rank, but there are a
couple of cases where things are tractable.  The first of these is when the
image again has rank 1, giving maps between Hilbert schemes.  If $I\in
\Hilb^n(X_{z,q})$, then the derived equivalence $\Phi:D^b\coh X_{z,q}\to
D^b\coh X_{z,q/z}$ produces an object $\Phi(I)$ with numerical invariants
$(1,-(n-1)C_8,1-n)$, and thus $\theta^{1-n}\Phi(I)$ has the same numerical
invariants as an object in $\Hilb^n(X_{z,q/z})$.

\begin{prop}
  If $n<\max(|\langle q/z\rangle|,|\langle q\rangle|)$, then the above
  construction induces a birational map $\Hilb^n(X_{z,q})\ratto
  \Hilb^n(X_{z,q/z})$.
\end{prop}

\begin{proof}
  Suppose $I$ is given as the kernel of a surjection $I'\to \sO_y$ with
  $y\in C$ and $I'$ a generic point of $\Hilb^{n-1}(X_{z,q})$.  Since
  $\Phi(I')$ is a point in $\Hilb^{n-1}(X_{z,q/z})$ by induction, and
  $\Phi(\sO_y)$ is a line bundle of degree 1, we find that $\Phi(I)$ is
  represented by a complex
  \[
  \Phi(I')\to {\cal L}.
  \]
  Taking the adjoint represents $R\ad \Phi(I)$ as an extension
  \[
  0\to \ad\Phi(I')\to R\ad \Phi(I)\to R^1\ad {\cal L}\to 0,
  \]
  where $R\ad\Phi(I')$ is a sheaf as long as $n\le |\langle q/z\rangle|$, per
  Proposition \ref{prop:ad_Hilbert_scheme}.  In particular, $R\ad\Phi(I)$
  is a sheaf, which is pure $2$-dimensional since the map $\Phi(I')\to
  {\cal L}$ is nonzero, so has $0$-dimensional cokernel.

  It follows that for generic $I$, $R\ad \Phi(I)$ is a pure $2$-dimensional
  sheaf, and thus, again using Proposition \ref{prop:ad_Hilbert_scheme},
  the same applies to $\Phi(I)$.

  The analogous calculation with $\Phi R\ad$ instead of $R\ad \Phi$
  (applied to a sheaf in $\Hilb^n(X_{\rho;1/q;C})$) gives a proof
  valid when $n<|\langle q\rangle|$.
\end{proof}

\begin{rem}
  This presumably corresponds to the usual $c\mapsto c+1$ symmetry of
  Calogero-Moser spaces.
\end{rem}

The other case of interest is when the image has rank $0$.  We find that if
$I\in \Hilb^n(X_{z,q})$, then $\Psi I$ has numerical invariants
$(0,-nC_8-e_8,0)$, and thus we expect generically to obtain a shift of a
semistable sheaf with invariants $(0,nC_8+e_8,0)$.  To guarantee the sheaf
is semistable, it is simplest to ensure that it is in the kernel of
$R\Gamma$, or equivalently that $I$ is in the image of $\alpha_8^*$.

To compute $\Psi I$ in this case, we note that $\theta^n I(-e_8)$ has
numerical invariants $(1,-nC_8-e_8,0)$.  Assuming this is in ${\cal
  K}_{\sO_X}$ (which we know happens generically), applying $\Phi_{\sO_X}$
reduces to the internal Serre functor, which in turn generically gives the
cone of the morphism
\[
\theta^{n+1} I(-e_8)\to \theta \sO_X.
\]
We thus obtain a short exact sequence
\[
0\to \theta^2 I(-2e_8)\to \theta^{2-n} \sO_X(-e_8)\to \Psi I[1]\to 0.
\]
which we thus see is closely related to the relation of the previous section
between the Hilbert scheme and $1$-dimensional sheaves.

This is a noncommutative elliptic analogue of the construction of
\cite[Defn.~5.2.3]{EtingofP/GanWL/OblomkovA:2006}, which constructed a
generalized Calogero-Moser space from certain Deligne-Simpson problems
which by the results of \cite{rat_Hitchin} can be translated into moduli
spaces of sheaves on rational surfaces; up to admissible reflections, these
are degenerations of the above moduli problem with $z=1$.  In the case
$q=z=1$, the moduli space can be explicitly computed
\cite{perverseHilbert}, and is in general (when there are effective roots)
a perverse Hilbert scheme rather than a Hilbert scheme; it is unclear
therefore whether the above birational map is regular for generic $q$, $z$.
Of course, to understand this would require better control over the action
of the derived equivalences on the t-structure.

\bigskip

The use of derived equivalences in proving Theorem \ref{thm:rd_painleve}
above relied crucially on the fact that we could recover the surface from
the restriction of $K_0$ to $C$.  This turns out to be an even stronger
constraint for $m\ne 8$, and we have the following.

\begin{lem}
  Suppose $\Phi:D^b\coh X_{\rho;q;C}\cong D^b\coh X_{\rho';q';C'}$ is a
  derived equivalence.  Then $m=m'$, $C\cong C'$, and the derived
  equivalence restricts to a derived equivalence $D^b\coh C\cong D^b\coh
  C'$.  If $m\ne 8$, then this derived equivalence preserves rank and the
  isomorphism $C\cong C'$ can be chosen to identify $q$ and $q'$.
\end{lem}

\begin{proof}
  That we obtain a derived equivalence $D^b\coh C\cong D^b\coh C'$, and
  thus $C\cong C'$, follows as in the proof of Theorem
  \ref{thm:rd_painleve}; that $m=m'$ follows from the fact that both
  surfaces have isomorphic $K_0$ along with the fact that
  $\rank(K_0(X))=m+4$.  We furthermore have $\Phi(M)|^{\dL}_{C'}\cong
  \Phi(M|^{\dL}_C)$ for any object $M$.  In particular, we observe that
  both $[\sO_C]|_C$ and $[\sO_{C'}]|_{C'}$ have rank 0 and degree $m-8$, so
  that when $m\ne 8$, $\Phi$ must preserve the class of a point, and thus
  respects $q=[pt]|_C\in K_0(C)$.  Since it also respects the $\chi$
  pairing, it preserves rank.
\end{proof}

Note that for $m=8$, if a derived equivalence $\Phi$ does not preserve the
class of a point, then for any $x$, $\Phi \sO_x$ will be a shift of a
vector bundle of rank $r$ and degree $d$ with $\gcd(r,d)=1$.  But then
there is a derived equivalence coming from Theorem \ref{thm:sl2Z_derived}
taking $\Phi \sO_x$ back to a shift of a point.  Thus for any $m$ the
problem of understanding derived equivalences reduces to that of
understanding derived equivalences taking points to points, and thus
(applying an element of $\Aut(C)$ as appropriate) such that $\Phi\sO_x\cong
\sO_x$ for all $x\in C$.  (Then $\Phi$ acts on $D^b\coh C$ as the twist by
a line bundle.)

Since our categories have exceptional collections, it will suffice to
understand how the exceptional objects in those collections behave under
such derived equivalences.  The assumption on $\Phi$ ensures that if
$\rank(E)=1$, then $\Phi E|^{\dL}_C$ is a line bundle, while if
$\rank(E)=0$, then $\Phi E|^{\dL}_C$ is the structure sheaf of a point.
Since $|^{\dL}_C$ has homological dimension 1, we conclude that
$h^0 (h^p \Phi E)|^{\dL}_C=0$ unless $p=0$ and
$h^{-1} (h^p \Phi E)|^{\dL}_C=0$ unless $p=1$.  This implies that $h^0
\Phi E$ is transverse to $C$ and $h^p \Phi E$ is disjoint from $C$ for
$p\ne 0$.

To understand the structure further, we have the following result.

\begin{lem}
  Let $E$ be an exceptional object of $D^b\coh X_{\rho;q;C}$.
  Then $\Ext^1(h^p E,h^q E)=0$ for all $p,q$.
\end{lem}

\begin{proof}
  In \cite[III(4.6.10.1)]{VerdierJ-L:1996}, Verdier constructed a spectral
  sequence
  \[
  E_2^{p,q}=\prod_{d\in \Z} \Ext^p(h^{d-q}M,h^d N) \Rightarrow \Ext^{p+q}(M,N)
  \]
  in the derived category of an abelian category with enough injectives,
  where the cohomology of $M$ is bounded above and $N$ can be represented
  by a complex which is bounded below.  In our case, we can represent both
  $M$ and $N$ by bounded complexes, and we may compute inside $D\qcoh
  X_{\rho;q;C}$, so there is no difficulty applying this for $M,N\in
  D^b\coh X_{\rho;q;C}$.  In particular, for $M=N=E$, we immediately
  conclude (since the sequence collapses on the $E_3$ page, and there are
  no nonzero maps hitting $E_2^{1,q}$) that
  \[
  \prod_{d\in \Z} \Ext^1(h^{d-q}E,h^d E)=0
  \]
  for $q\ne -1$.  For $q=-1$, we simply note that $k=\End(E)$ contributes
  to the $E_\infty^{0,0}$ term of the limit, and thus
  $E_2^{1,-1}=E_\infty^{1,-1}$ must still vanish.
\end{proof}

\begin{rem}
  This was shown in \cite{IshiiA/UeharaH:2005} for spherical objects, but
  the proof is essentially the same.
\end{rem}

\begin{lem}\label{lem:inheriting_rigidity}
  Suppose we have an exact sequence $0\to T\to M\to F\to 0$ such that
  $\Hom(T,F)=\Ext^2(F,T)=\Ext^1(M,M)=0$.  Then $\Ext^1(T,T)=\Ext^1(F,F)=0$.
\end{lem}

\begin{proof}
  The composition
  \[
  \Ext^1(M,T)\to \Ext^1(T,T)\to \Ext^1(T,M)
  \]
  factors through $\Ext^1(M,M)=0$ and thus vanishes.  On the other hand,
  the first map has cokernel $\Ext^2(F,T)=0$, so is surjective and the
  second map has kernel $\Hom(T,F)=0$, so is injective; it follows that
  $\Ext^1(T,T)=0$.

  The dual argument gives $\Ext^1(F,F)=0$.
\end{proof}

Let $S_{\rho;q;C}$ be the set of $-2$-curves of $X_{\rho;q;C}$; i.e.,
$\alpha\in \Pic(X)$ such that there is an irreducible 1-dimensional sheaf
disjoint from $C$ with Chern class $\alpha$.  As in the commutative case,
this is the set of simple roots of a root system $\Phi(S_{\rho;q;C})$
defined as the set of real roots $\alpha$ of $E_{m+1}$ such that
$\rho(\alpha)\in q^\Z$.  Call a sheaf of the form $\sO_\alpha(d)$ with
$\alpha$ a $-2$-curve a ``$-2$-sheaf'', and similarly for ``$-1$-sheaf''.

\begin{lem}\label{lem:built_from_roots}
  Suppose $M\in \coh X_{\rho;q;C}$ satisfies $M|_C=0$, $\Ext^1(M,M)=0$.
  Then $M$ admits a filtration in which every subquotient is a $-2$-sheaf.
\end{lem}

\begin{proof}
  Note that $M$ cannot be $0$-dimensional, since $0$-dimensional sheaves
  always admit deformations.  More generally, if $T$ is the maximal
  $0$-dimensional subsheaf of $M$, then Lemma \ref{lem:inheriting_rigidity}
  implies $\Ext^1(T,T)=0$ and thus $T=0$, so that $M$ is pure
  $1$-dimensional.  In addition, $c_1(M)\cdot e_m\ge 0$, since otherwise it
  would have a sub- or quotient sheaf of the form $\sO_{e_m}(d)$,
  contradicting disjointness from $C$.

  Now, if $M$ is any pure $1$-dimensional sheaf with $c_1(M)=l\alpha$ for
  some $\alpha\in S_{\rho;q;C}$ (ignoring the rigidity condition
  $\Ext^1(M,M)=0$), then $\alpha\cdot c_1(M)<0$, and thus there is a
  morphism between $M$ and $\sO_{\alpha}(d)$ for any $d$ such that the
  latter exists.  A map $M\to \sO_\alpha(d)$ has image $\sO_{\alpha}(d')$
  and kernel which by induction has a filtration as required.  If $q$ is
  torsion, then we can always increase $d$ until there is no map from
  $\sO_\alpha(d)$ and thus end up in this case.  Otherwise $q$ is
  non-torsion, and then disjointness from $C$ implies that the cokernel of a map
  from $\sO_\alpha(d)$ is pure $1$-dimensional, so that we may again
  induct.

  More generally, the condition $\Ext^1(M,M)=0$ implies $\chi(M,M)>0$ and
  thus $c_1(M)^2<0$.  Since $c_1(M)$ is effective, it cannot be nef, and
  thus (since it satisfies $c_1(M)\cdot C_m=0$, $c_1(M)\cdot e_m\ge 0$)
  there is some simple root $\sigma$ having negative intersection with
  $c_1(M)$.  If that root is ineffective, we may apply the corresponding
  reflection.  Since the process cannot otherwise terminate, we will reach
  a point where the simple root is effective, and the corresponding divisor
  class on the original surface will be in $S_{\rho;q;C}$.  In other words,
  there is some $\alpha\in S_{\rho;q;C}$ such that $\alpha\cdot c_1(M)<0$.
  Now, let $T$ be the maximal subsheaf of $M$ which has Chern class a
  multiple of $\alpha$.  Then $T$ is disjoint from $C$ with
  $\Hom(T,M/T)=0$, so that Lemma \ref{lem:inheriting_rigidity} gives
  $\Ext^1(M/T,M/T)=0$, allowing us to induct.
\end{proof}

\begin{lem}\label{lem:rigid_rank_1_nice}
  Let $M$ be a coherent sheaf on $X_{\rho;q;C}$ such that $\Ext^1(M,M)=0$
  and $M|^{\dL}_C$ is a line bundle.  Then the torsion subsheaf of $M$ is
  an extension of $-2$-sheaves and the torsion-free quotient is a line
  bundle.
\end{lem}

\begin{proof}
  If $M|_C$ is a line bundle, so $\rank(M)=1$, let $T$ be the torsion
  subsheaf of $M$.  Since $M/T$ is pure $2$-dimensional, it is transverse
  to $C$, and thus we have a short exact sequence
  \[
  0\to T|_C\to E|_C\to (M/T)|_C\to 0.
  \]
  Since $M|_C$ is a line bundle and $\rank(T|_C)=0$, we conclude that $T$
  is disjoint from $C$.  But then Lemma \ref{lem:inheriting_rigidity}
  implies $\Ext^1(T,T)=\Ext^1(M/T,M/T)=0$.  The claim about the torsion
  subsheaf then follows from Lemma \ref{lem:built_from_roots}, while the
  claim for $M/T$ follows from Corollary \ref{cor:line_bundle_bound}, since
  $\rank(M/T)=\rank(M)=\rank(M|^{\dL}_C)=1$.
\end{proof}

\begin{lem}\label{lem:rigid_rank_0_nice}
  Let $M$ be a coherent sheaf on $X_{\rho;q;C}$ such that $\Ext^1(M,M)=0$
  and $M|^{\dL}_C$ is the structure sheaf of a point.  Then
  $M$ has a filtration in which all but one subquotient is a $-2$-sheaf and
  the remaining subquotient is a $-1$-sheaf.
\end{lem}

\begin{proof}
  Lemma \ref{lem:inheriting_rigidity} allows us to peel off the maximal
  subsheaf which is an extension of $-2$-sheaves, and then the maximal
  quotient which is such an extension, reducing to the case that $M$ has no
  maps to or from $-2$-sheaves.  Now, $\rank(M)=\rank(M|^{\dL}_C)=0$, so
  $c_1(M)<0$, and thus $c_1(M)$ is not nef.  Since $c_1(M)\cdot \alpha<0$
  for all $-2$-curves $\alpha$ and $c_1(M)\cdot C_m>0$, the only way the
  algorithm for testing nefness of $c_1(M)$ can terminate is with $w
  c_1(M)\cdot e_m<0$.  This implies that there is a map between $M$ and any
  sheaf of the form $\sO_e(d)$, where $e$ is the corresponding $-1$-curve.
  In particular, $M$ has an irreducible constituent of this form, and the
  residual sheaf is rigid and disjoint from $C$, so (since $M$ has no maps
  to or from $-2$-sheaves) must vanish.
\end{proof}

If $q$ is not torsion, then each $\alpha\in S_{\rho;q;C}$ corresponds to a
unique sheaf $\sO_{\alpha}(d_\alpha)$, and we may then define a {\em reflection
functor} $R_\alpha$ on $D^b\coh X_{\rho;q;C}$ as the cone
\[
\sO_\alpha(d_\alpha)\otimes_k R\Hom(\sO_\alpha(d_\alpha),M)\to M\to R_\alpha
M\to,
\]
that is, the twist functor in the sense of \cite{SeidelP/ThomasR:2001} in
the spherical object $\sO_\alpha(d_\alpha)$.  Note
\[
R_\alpha \sO_\alpha(d_\alpha) = \sO_\alpha(d_\alpha)[-1],
\]
and we have an inverse functor fitting into the distinguished triangle
\[
R_\alpha^{-1}M\to M\to
\Hom_k(R\Hom(M,\sO_\alpha(d_\alpha)),\sO_\alpha(d_\alpha)) \to.
\]
The two constructions are dual, and indeed $R_\alpha R\ad = R\ad
R_\alpha^{-1}$.  In addition, if $\alpha$ and $\beta$ are two distinct
roots of $S_{\rho;q;C}$ then we have the braid relations
\cite{SeidelP/ThomasR:2001}
\begin{align}
  R_\alpha R_\beta &\cong R_\beta R_\alpha & \alpha\cdot\beta=0\notag\\
  R_\alpha R_\beta R_\alpha &\cong R_\beta R_\alpha R_\beta &
  \alpha\cdot\beta=1,
\end{align}
so that the reflection functors and their inverses induce a representation
of the associated braid group.  In particular, for any element $w$ of the
corresponding Coxeter group, we can take the composition of reflection
functors corresponding to any reduced word for $w$, and the result $R(w)$
will be independent of the choice of reduced word.
  
\begin{lem}
  Suppose $q$ is not torsion.  Then for all $w\in W(S_{\rho;q;C})$,
  $R(w)$ preserves the nonnegative part of the $t$-structure, and
  $R(w)^{-1}$ preserves the nonpositive part of the $t$-structure.
  Moreover, if $M=\tau_{\ge 0}M$ is such that for any
  $-2$-sheaf $N$, $\Hom(N,M)=0$, then the same holds for $R(w)M$.
  Dually, if $M=\tau_{\le 0}M$ is such that for any $-2$-sheaf $N$,
  $\Hom(M,N)=0$, then the same holds for $R(w)^{-1}M$.
\end{lem}

\begin{proof}
  This reduces immediately to the case that $w$ is a simple reflection, and
  via the adjoint symmetry to the case of $R_\alpha$.  If $q$ is not
  torsion, then any subsheaf of $\sO_\alpha(d_\alpha)^l$ has the same form,
  and thus the natural map
  \[
  \sO_\alpha(d_\alpha)\otimes_k \Hom(\sO_\alpha(d_\alpha),M)\to h^0(M)
  \]
  is necessarily injective, so that $h^pR_\alpha M=0$ for $p<0$.  Moreover,
  if $\Hom(\sO_\alpha(d_\alpha),M)=0$, then $h^0(R_\alpha M)$ is an
  extension of $\sO_\alpha(d_\alpha)$ by $M$, so still has no map from a
  $-2$-sheaf other than $\sO_\alpha(d_\alpha)$, for which we have
  \[
  \Hom(\sO_\alpha(d_\alpha),R_\alpha M)
  \cong
  \Hom(R_\alpha^{-1}\sO_\alpha(d_\alpha),M)
  \cong
  \Hom(\sO_\alpha(d_\alpha)[1],M)
  =0.
  \]
\end{proof}

\begin{rem}
  Note that $R_\alpha$ is actually the derived functor of $h^0 R_\alpha$,
  but this fails for $\ell(w)>1$; if $\alpha\cdot\beta\ne 0$, then $h^0
  R_\alpha$ does not take injective objects to $R_\beta$-acyclic objects in
  general, so the composition of the derived functors is not the derived
  functor of the composition.
\end{rem}

\begin{lem}\label{lem:reflections_kill_sheaves}
  Suppose $q$ is non-torsion and the root system $\Phi(S_{\rho;q;C})$ is
  finite.  Let $M$ be an extension of $-2$-sheaves, and define a sequence
  $M_0=M$,$M_1$,\dots,$M_n$ of objects such that for each $i$ we have
  $M_i=R_\alpha M_{i-1}$ for some $\alpha$ such that
  $\Hom(\sO_{\alpha}(d_\alpha),M_{i-1})\ne 0$.  Then $n\le
  |\Phi(S_{\rho;q;C})|$ and $h^0(M_n)=0$.
\end{lem}

\begin{proof}
  Let $w_i$ be the corresponding sequence of Weyl group elements such that
  $M_i = R(w_i)M$, with $w_i = s_{\alpha_i} w_{i-1}$.  It will suffice to
  show that $\ell(w_i)=\ell(w_{i-1})+1$, since then the process will be
  forced to terminate after at most $|\Phi(S_{\rho;q;C})|$ steps and the
  only way it can terminate is when $h^0(M_n)=0$.

  To see that the length must increase, simply note that for any expression
  $w_{i-1}=s_{\alpha} w'$, we have
  \[
  \Hom(\sO_{\alpha}(d_\alpha),R(w_{i-1})M)
  \cong
  \Hom(R_\alpha^{-1} \sO_{\alpha}(d_\alpha),R(w')M)
  \cong
  \Hom(\sO_\alpha(d_\alpha)[1],R(w')M)
  =
  0,
  \]
  and thus $\alpha_i\ne \alpha$.
\end{proof}

\begin{rem}
  Note that there are instances with $|\Phi(S_{\rho;q;C})|=\infty$ where
  this fails even for sheaves with $\Ext^1(M,M)=0$.  Indeed, suppose
  $\alpha$ and $\beta$ are two $-2$-curves with $l=\alpha\cdot \beta>1$,
  and consider the sheaf $M=R_\beta \sO_\alpha(d_\alpha)$.  This has no
  maps from $\sO_\beta(d_\beta)$, so our only choice is to apply
  $R_\alpha$, and (since the length must increase) alternate between that
  and $R_\beta$ from then on.  One can show that $M_n$ is always a sheaf,
  and thus invertibility prevents the process from terminating.  (The
  argument reduces to showing that the generic quadratic algebra in $l$
  variables with a single relation is Koszul, with Hilbert series
  $1/(1-lt+t^2)$.)
\end{rem}

\begin{lem}\label{lem:exceptional_objects_nice}
  Suppose $S_{\rho;q;C}$ is empty or $q$ is non-torsion and the root system
  $\Phi(S_{\rho;q;C})$ is finite.  If $E\in D^b\coh X_{\rho;q;C}$ is
  an exceptional object such that $E|^{\dL}_C$ is a the
  structure sheaf of a point, then there is an element $\Psi$ of the group
  generated by reflection functors such that $\Psi E$ is a $-1$-sheaf.
\end{lem}

\begin{proof}
  It suffices to construct $\Psi$ such that $\Psi E$ is a sheaf with no
  maps to or from $-2$-sheaves, since then Lemma
  \ref{lem:rigid_rank_0_nice} implies the desired result.

  In general, we know that $h^p E$ for $p\ne 0$ is disjoint from $C$ and
  has no self-$\Ext^1$, so is an extension of $-2$-sheaves.  If
  $S_{\rho;q;C}$ is empty, then $E$ is already a sheaf and we are done.

  Otherwise, $q$ is non-torsion and $S_{\rho;q;C}$ generates a finite root
  system.  An easy induction using the adjoint of Lemma
  \ref{lem:reflections_kill_sheaves} tells us that for any object of
  $D^b\coh X_{\rho;q;C}$ such that every cohomology sheaf is an extension
  of $-2$-sheaves, there is a finite composition of inverse reflection
  functors taking the object to one supported only in negative degrees.

  Now, $E$ fits into a distinguished triangle
  \[
  T\to E\to F\to
  \]
  such that $\Ext^p(T,N)=0$ for any $-2$-sheaf $N$ and any $p\le 0$.
  If $\Psi_1$ is a product of inverse reflection functors such that
  $\tau_{\ge 0} \Psi_1 F=0$, then we have a distinguished triangle
  \[
  \Psi_1 T\to \Psi_1 E\to \Psi_1 F\to
  \]
  from which we may conclude that $\tau_{\ge 1}\Psi_1 E=0$ and $\Hom(\Psi_1
  E,N)=0$ for all $-2$-sheaves $N$.  We may thus reduce to the case that
  $E=\tau_{\le 0}E$ and $h^0(E)$ has no maps to $-2$-sheaves.

  Now, let $p$ be minimal such that $h^p E\ne 0$.  If there is no
  $-2$-sheaf $N$ with $\Hom(N,h^p E)\ne 0$, then $p=0$ and we are done.
  Otherwise, if we can show that $\Ext^1(N,h^0 E)=0$ for any such sheaf,
  then applying the corresponding reflection functor will not affect the
  $\Ext^p(E,N)=0$ condition and it will then follow easily from Lemma
  \ref{lem:reflections_kill_sheaves} that this process will terminate as
  required in an exceptional sheaf with no maps to or from $-2$-curves.
  
  Consider the short exact sequence
  \[
  0\to N^l\to h^p E\to M\to 0,
  \]
  with $l$ maximal.  Since $\Ext^1(h^p E,h^0 E)=0$, the long exact sequence
  reads in part
  \[
  0\to \Ext^1(N^l,h^0 E)\to \Ext^2(M,h^0 E),
  \]
  so that it suffices to prove vanishing of $\Ext^2(M,h^0 E)$.  If $p<0$,
  then $\theta M=M$ and thus $\Ext^2(M,h^0 E)\cong \Hom(h^0 E,M)^*$, which
  vanishes since $M$ is an extension of $-2$-sheaves, to which $h^0 E$ has
  no maps.

  Now, for $p=0$, Lemma \ref{lem:rigid_rank_0_nice} tells us that $E$ is an
  extension of a $-1$-sheaf $E'$ by an extension of $-2$-sheaves.  In
  particular, this tells us that any subsheaf of $E$ which is disjoint from
  $C$ is an extension of $-2$-sheaves, and the same holds for $M$.  In
  particular, we have a short exact sequence
  \[
  0\to M_1\to M\to E'\to 0
  \]
  where $M_1$ is an extension of $-2$-sheaves.

  Since $\Ext^2(M,N)\cong \Hom(N,M)^*=0$, we have $\Ext^2(M,E)\cong
  \Ext^2(M,M)$.  But this vanishes, since we have the exact sequences
  \[
  \Ext^2(E',M)\to \Ext^2(M,M)\to \Ext^2(M_1,M)=0
  \]
  and
  \[
  0=\Ext^2(E',M_1)\to \Ext^2(E',M)\to \Ext^2(E',E')=0,
  \]
  using $\Ext^2(M_1,M)\cong \Hom(M,M_1)^*=0$ and $\Ext^2(E',M_1)\cong
  \Hom(M_1,E')^*=0$.
\end{proof}

\begin{thm}\label{thm:all_derived_equivalences}
  Suppose either $S_{\rho;q;C}$ is empty or $|\langle q\rangle|=\infty$ and
  $\Phi(S_{\rho;q;C})$ is finite.  Then any derived equivalence
  \[
  \Psi:X_{\rho;q;C}\to X_{\rho';q;C}
  \]
  such that $\Psi\sO_x=\sO_x$ for all $x\in C$ is in the group generated by
  twists by line bundles, admissible reflections, and reflection functors.
\end{thm}

\begin{proof}
  Suppose $m\ge 1$, and consider the exceptional object
  $\Psi\sO_{e_m}(-1)$.  Lemma \ref{lem:exceptional_objects_nice} tells us
  that there is an element $\Psi_1$ of the group generated by reflection
  functors such that
  \[
  \Psi_1\Psi \sO_{e_m}(-1)
  \]
  is a $-1$-sheaf.  But then there is an admissible element $w\in
  W(E_{m_1})$ such that
  \[
  w\Psi_1\Psi \sO_{e_m}(-1) \cong \sO_{e_m}(d)
  \]
  for some $d$, at which point twisting by $(d+1)e_m$ gives
  $\sO_{e_m}(-1)$.  We thus reduce to proving the Theorem in the case
  $\Phi\sO_{e_m}(-1)=\sO_{e_m}(-1)$, which then immediately reduces to the
  corresponding claim for $X_{m-1}$.

  This similarly works to reduce from $X'_0$ to $X_{-1}$, replacing $e_m$
  by $s$, and thus leaves us only the $X_{-1}$ and $X_0$ cases to consider.

  For $X_{-1}$, there are no $-2$-curves, so that $\Psi\sO_X$ is already a
  line bundle by Lemma \ref{lem:rigid_rank_1_nice}, and we may as well
  assume $\Psi\sO_X\cong \sO_X$.  But then $\Psi\theta^{-l} \sO_X\cong
  \theta^{-l} \Psi\sO_X\cong \theta^{-l} \sO_X$ for all $l$, and thus
  $\Psi$ actually preserves the t-structure.  Now Euler characteristic
  considerations force $\Psi \sO_X(-d)\cong \sO_X(-d)$ for $d\in
  \{0,1,2\}$, and since restriction to $C$ is faithful on that three-object
  subcategory (which generates $D^b\coh X$), this forces $\Psi$ to be
  naturally isomorphic to the identity.

  For $X_0$, let $x_1$ be any point of $C$, and observe that any derived
  equivalence between $X_{\eta,\eta';q;C}$ and $X_{\bar\eta,\bar\eta';q;C}$
  extends to a derived equivalence on the blowup that fixes $\sO_{e_1}(-1)$
  (relative to the ``even'' basis of $\Pic(X)$).  Since the stabilizer of
  $\sO_{e_1}(-1)$ contains the reflection (whether admissible or a
  reflection functor), we easily verify that it is the correct group.
\end{proof}

\begin{rem}
  Note that the finiteness hypothesis is automatically satisfied for $m<8$
  as long as $q$ is not torsion.  It is likely that finiteness is not
  necessary; for instance, the argument of \cite{IshiiA/UeharaH:2005} can
  very likely be modified to deal with the case that $S_{\rho;q;C}$ has
  components of the form $\tilde{A}_n$.  (For small $n$, the question for
  affine root systems can be translated via our derived equivalence to one
  for finite root systems but on a commutative surface.)  If we denote the
  nonisomorphic constituents of $\Phi(\sO_{e_m}(-1))$ by $S_0$,\dots,$S_n$,
  with $S_0\cong \sO_e(d)$, then we can work entirely inside the (Serre)
  subcategory of extensions of the $S_i$.  The algebra $\Ext^*(\oplus_i
  S_i)$ depends only on the intersection numbers $c_1(S_i)\cdots c_1(S_j)$,
  and can be recognized as the Koszul dual of the ``partial preprojective
  algebra'' $\Lambda$ \cite{EtingofP/EuC-H:2007} of the corresponding
  quiver.  It thus seems very likely that the category generated by the
  $S_i$ is isomorphic to the category of finite modules over $\Lambda$, and
  thus depends only on the combinatorics of the situation.
\end{rem}

\begin{rem}
  If we denote the pure braid group of the root system ${\cal
    S}_{\rho;q;C}$ (i.e., the kernel of the map from the braid group to the
  Weyl group) by $P_{S_{\rho;q;C}}$, then we would expect to have an action
  of the extension $P_{S_{\rho;q;C}}.W(E_{m+1})$ via derived equivalences,
  acting on $\Pic(X)$ in the natural way.  (This may not be faithful in
  general, although it follows from \cite{SeidelP/ThomasR:2001} that it is
  faithful whenever every component of ${\cal S}_{\rho;q;C}$ is of type
  $A$.)  There is a mild technicality here, in that reflection functors do
  not act in quite the right way on $K_0$.  Luckily, if the corresponding
  $-2$-sheaf is $\sO_\alpha(d_\alpha-1)$, then we need simply twist by
  $d_\alpha\alpha$ before or after reflecting to fix everything.  The
  case $\alpha=e_{m-1}-e_m$ is particularly interesting.  If
  $d=d_{e_{m-1}-e_m}<0$, then we find
  \begin{align}
  R_{e_{m-1}-e_m}(\alpha_m^* \sO_{e_{m-1}}(-1))(de_{m-1}-de_m) &\cong
  \sO_{e_m}(-1)\notag\\
  R_{e_{m-1}-e_m}(\sO_{e_m}(-1))(de_{m-1}-de_m) &\cong
  \alpha_m^! \sO_{e_{m-1}}(-1),
  \end{align}
  while every other object in the exceptional collection is taken to the
  corresponding object.  We may thus think of the combined operation as
  being a mutation of the exceptional collection, with something similar
  applying to the inverse operation when $d\ge 0$, or to the other types of
  simple reflections.  This observation will be key to understanding
  birational geometry when $C$ degenerates, see \cite{noncomm2}.
\end{rem}

\begin{cor}
  Suppose $m\ne 8$ and the hypotheses of Theorem
  \ref{thm:all_derived_equivalences} hold for $X_{\rho;q;C}$.  Then any
  surface $X_{\rho';q';C'}$ derived equivalent to $X_{\rho;q;C}$ is Morita
  equivalent to $X_{\rho;q;C}$.
\end{cor}

\begin{proof}
  Indeed, reflection functors do not change the surface, and the other
  generators are abelian equivalences.
\end{proof}

\begin{rem}
  For $m<8$, one can omit the hypothesis that $q$ is non-torsion here, as
  it is straightforward to show in that case that the known derived
  equivalences surject onto the subgroup of $\Aut(K_0)$ that preserves
  $\chi(\_,\_)$, and thus account for all possible destination surfaces.
  Similarly, for $m=8$, it suffices to also throw in generators of
  $\SL_2(\Z)$.  For $m>8$, however, the stabilizer of the pairing is almost
  certainly larger than $\Pic(X)\rtimes W(E_{m+1})$.
\end{rem}

\medskip

We are also interested in abelian equivalences.  Here, things are much
better behaved.

\begin{prop}\label{prop:abelian_isomorphisms}
  Suppose $\Phi:\coh X_{\rho;q;C}\cong \coh X_{\rho';q';C'}$ is an abelian
  equivalence.  Then there is an automorphism $\phi:C\cong C'$ taking $q$
  to $q'$ and $\rho$ to an element of the same $\Pic(X)\rtimes W(E_{m+1})$
  orbit as $\rho'$.  Moreover, if $\phi=\id$, then the equivalence is (up
  to natural isomorphism) the composition of a twist with an admissible
  element $w\in W(E_{m+1})$.
\end{prop}

\begin{proof}
  Here $\Phi$ must induce an equivalence $\coh C\cong \coh C'$, which must
  therefore preserve the class of a point.  It follows that the equivalence
  preserves rank, and thus $\Phi$ in particular takes $-2$-sheaves to
  $-2$-sheaves.  But then we can argue as in Theorem
  \ref{thm:all_derived_equivalences}; since $\Phi$ preserves the property
  of having no maps to or from $-2$-sheaves, we can skip all steps involving
  reflection functors, so never leave the realm of abelian equivalences.
\end{proof}

\begin{cor}
  The automorphism group of $\coh X_{\rho;q;C}$ can be identified with the
  subgroup of $\Aut(C)\times \Pic(X)\rtimes W(E_{m+1})$ that preserves the
  parameters and the set of effective roots.  The induced action on
  $\Spl(X_{\rho;q;C})$ preserves the Poisson structure up to a root of
  unity which is trivial on the intersection of $\Aut(\coh X_{\rho;q;C})$ with
  $\Pic^0(C)\times \Pic(X)\rtimes W(E_{m+1})$.
\end{cor}

\begin{proof}
  For the statement about the Poisson structure, recall that
  to determine the Poisson structure, we needed to specify a nonzero
  holomorphic differential on $C$.  The group $\Pic(X)\rtimes W(E_{m+1})$
  acts trivially on $\omega_C$, as does $\Pic^0(C)\subset \Aut(C)$.
\end{proof}

\begin{rem}
  For {\em derived} automorphisms, we instead (subject to our usual
  finiteness assumption) obtain a group of the form
  \[
  P_{S_{\rho;q;C}}.G
  \]
  where $G$ is the subgroup of $\Aut(C).\Pic(X)\rtimes W(E_{m+1})$ that
  preserves $\rho$ and $q$.  (The only thing to be shown is that any
  element of $W(E_{m+1})$ has a unique representation as the product of an
  element of $W(S_{\rho;q;C})$ and an admissible reflection.)  There is,
  of course, the possibility that the pure braid group does not act
  faithfully.
\end{rem}

Note that if $q$ is preserved by a non-translation automorphism of $C$,
then either $q$ is $2$-torsion or $q$ is $3$-torsion and $j(C)=0$, and
thus this does not arise in truly noncommutative cases.  It is, however,
useful to include {\em contravariant} equivalences. Since $\ad$ inverts
$q$, we can combine it with any hyperelliptic involution on $C$ to obtain
a contravariant equivalence between surfaces with the same value of $q$,
and this will again be a Poisson map.

We thus obtain a Poisson action of the appropriate subgroup of
$(\Pic^0(C)\rtimes \Z/2\Z)\times (\Pic(X)\rtimes W(E_{m+1})$ (preserving
$\rho$ and the space of effective roots) on the moduli space of simple
sheaves.  Since the subspace fixed by a Poisson automorphism is again
Poisson, this gives a way to construct even more Poisson moduli spaces.

Of particular interest are cases in which the new moduli spaces are
themselves examples of the original construction.  As an example, consider
a surface with parameters
\[
\eta,y_0,y_1,\sqrt{q}\eta/y_1,y_2,\sqrt{q}y_2,y_3,\sqrt{q}y_3,\dots
\]
Translating by $\sqrt{q}$ turns this into
\[
q\eta,\sqrt{q}y_0,\sqrt{q}y_1,q\eta/y_1,\sqrt{q}y_2,qy_2,\sqrt{q}y_3,q/y_3,\dots
\]
Reflecting in the commuting reflections $f-e_1-e_2$, $e_3-e_4$, $e_5-e_6$,
etc., all of which are generically effective, then gives
\[
q\eta,y_0,y_1,\sqrt{q}\eta/y_1,qy_2,\sqrt{q}y_2,qy_3,\sqrt{q}y_3,\dots
\]
which becomes the original surface when we twist by $s+f-e_3-e_5-\cdots$
Consider a 1-dimensional sheaf which is isomorphic to its image under this
isomorphism.  Using the prescription of \cite{rat_Hitchin} for how twisting and
reflections act on difference equations (as justified below), we find that
the equation can be factored in the form
\[
v(qz) = C(\sqrt{q}z)C(z) v(z)
\]
and thus we can express any solution in terms of solutions of the
$\sqrt{q}$-difference equations $v(\sqrt{q}z)=\pm C(z)v(z)$.  (Note here
that $C$ can have additional singularities at the ramification points of
the relevant involutions.)

The elliptic Riemann-Hilbert correspondence considered below suggests
considering the same configuration with $\sqrt{q}$ replaced by $\sqrt{p}$;
this is similarly related to a moduli space of $\sqrt{p}$-elliptic
difference equations.  There are presumably also analogues corresponding to
the map from $q^{1/r}$-difference equations to $q$-difference equations or
from $p^{1/r}$-elliptic difference equations to $p$-elliptic difference
equations.

Another case meriting further investigation is the composition
\[
X_{\rho;q;C}
\to
X_{\rho;1/q;C}
\to
X_{\iota_\eta\circ \rho;q;C}
\to
X_{\rho;q;C}
\]
where $\iota_\eta$ is the hyperelliptic involution corresponding to $\eta$
and the last isomorphism comes from the central element of $D_m$ (assuming
$m$ even).  The difference equations preserved by this contravariant
equivalence would have the form $v(qz)=A(z)v(z)$ with $A(z)C A(z)^t\propto
C$ for some constant matrix $C$, so that in particular we would have
components of the moduli spaces in which $A$ is an orthogonal or symplectic
similitude.  It follows from the above considerations that the
corresponding moduli spaces are Poisson, but of course one would like to
know their dimensions (and especially if there are rigid or $2$-dimensional
cases).  More generally, we could hope to consider moduli spaces of
equations in which $A$ takes values in some fixed reductive group.

\section{Sheaves as difference equations}
\label{sec:diffeq2}

\subsection{Spaces of meromorphic solutions}

As the reader may recall, our original motivation for studying these
noncommutative surfaces was to understand moduli spaces of (symmetric)
elliptic difference equations and maps between them.  Though we used an
association between sheaves on $X_m$ and difference equations to motivate
our constructions, there remain several loose ends.  In particular, there
are a number of analytic questions to be addressed regarding this
association.  As a result, we restrict our attention in this section to
the case $C\cong \C^*/\langle p\rangle$.

Since our interpretation of sheaves as difference equations relies on the
functor $\Hom(\_,\Mer)$, we first need to understand this functor better.
In particular, it is natural to ask whether this functor is exact, i.e.,
whether the sheaf $\Mer$ is injective.  We first ask a somewhat simpler
question: If $\phi:\sO_X(D_1)\to \sO_X(D_2)$ is a nonzero morphism of line
bundles, is $\Hom(\phi,\Mer)$ surjective?  Since this acts as a (symmetric
elliptic) difference operator on the space of symmetric meromorphic
functions in the form $\Hom(\sO_X(D_2),\Mer)$, we see that we must
understand the solution spaces of inhomogeneous symmetric difference
equations.

If we do not impose any symmetry condition on the operator, the homogeneous
case reduces easily to results of \cite{PraagmanC:1986}.

\begin{prop}
  Let $0<|q|<1$ and let $\oD$ be a $q$-difference operator
  $
  \sum_{m_1\le k\le m_2} c_k(z) T^k
  $
  such that each $c_k(z)$ is meromorphic on $\C^*$ and
  $c_{m_1}(z),c_{m_2}(z)\ne 0$.  Then $\ker(\oD)$ is a vector space of
  dimension $m_2-m_1$ over the space of $q$-elliptic functions.
\end{prop}

\begin{proof}
  Since $\ker(\oD)=\ker(T^{-{m_1}}\oD)$, we may as well assume that $m_1=0$,
$m_2=m$.  We can then rewrite the equation $\oD f=0$ as a system of first-order
equations in variables $v_k(z) = T^{k-1} f(z)$:
\begin{align}
T v_k(z) &= v_{k+1}(z),\quad 1\le k<m\notag\\
T v_m(z) &= \sum_{1\le k<m} \frac{-c_{k-1}(z)}{c_m(z)} v_k(z)\notag
\end{align}
Since the corresponding shift matrix is invertible, we may apply
\cite[Thm.~3]{PraagmanC:1986} to conclude that there is an $m$-dimensional space of
solutions over the field of $T$-invariant functions (i.e., the space of
$q$-elliptic functions).
\end{proof}

\begin{cor}
Let $0<|q|<1$ and let $\oD=\sum_{m_1\le k\le m_2} c_k(z) T^k$ be a
meromorphic $q$-difference operator such that
$c_{m_1+m_2-k}(z)=c_k(\eta/q^{m_1+m_2}z)$ and $c_{m_1}(z)\ne 0$.  Then the
space of functions $f$ such that $\oD f=0$ and $f(\eta/z)=f(z)$ is a vector
space of dimension $m_2-m_1$ over the field of $\eta$-symmetric
$q$-elliptic functions.
\end{cor}

\begin{proof}
Again, we may reduce to the case $m_1=0$, $m_2=m$, and have already
observed that $\ker(\oD)$ is an $m$-dimensional vector space over the field
of $q$-elliptic functions.  Since the field of $q$-elliptic functions is
Galois over the field of $\eta$-symmetric $q$-elliptic functions, and the
Galois group preserves $\ker(\oD)$ (as follows from the symmetry conditions),
the claim follows from Galois theory.
\end{proof}

\begin{rem}
  In principle, this should also follow from \cite[Thm.~3]{PraagmanC:1986}, in
  the case that the group acting is the infinite dihedral group rather than
  the infinite cyclic group; the translation from straight-line equation to
  cocycle is somewhat unpleasant, however.
\end{rem}

To understand inhomogeneous operators, it remains only to prove existence
of solutions.  For unconstrained operators, this again follows immediately
from \cite{PraagmanC:1986}.

\begin{lem}
  Let $\oD$ be a nonzero meromorphic $q$-difference operator.  Then for any
  meromorphic function $g$ on $\C^*$, there is a meromorphic function $f$
  on $\C^*$ such that $\oD f=g$.
\end{lem}

\begin{proof} Reduce to a matrix equation as above and apply
  \cite[Thm.~4]{PraagmanC:1986}.
\end{proof}

\begin{cor}
  Let $\oD=\sum_{m_1\le k\le m_2} c_k(z) T^k$ be a nonzero meromorphic
  $q$-difference operator such that
  $c_{m_1+m_2-k}(z)=c_k(\eta/q^{m_1+m_2}z)$.  Then for any meromorphic
  function $g$ on $\C^*$ such that $g(\eta/q^{m_1+m_2}z)=g(z)$, there is a
  meromorphic function $f$ on $\C^*$ such that $\oD f=g$ and
  $f(\eta/z)=f(z)$.
\end{cor}

\begin{proof}
  Apply the lemma to find a function $h(z)$ such that $\oD h=g$, and observe
  that $\hat{h}:=h(\eta/z)$ also satisfies $\oD\hat{h}=g$, so that
  $f:=(h+\hat{h})/2$ is the desired solution.
\end{proof}

In terms of our surfaces, this has the following rephrasing as the desired
special case of injectivity of $\Mer$.

\begin{cor}
  Let $D_1$, $D_2$ be divisor classes and $\phi:\sO_X(D_1)\to
  \sO_X(D_2)$ a nonzero morphism.  Then $\Hom(\phi,\Mer)$ is surjective.
\end{cor}

\smallskip

We next deal with blowups.

\begin{lem}
  Let $\Mer_m$ denote the sheaf $\Mer$ on $X_m$.  Then $\Mer_m\cong
\alpha_m^*\Mer_{m-1}$ for $m>0$.
\end{lem}

\begin{proof}
We first note that $\alpha_{m*}\Mer_m\cong \Mer_{m-1}$ (by inspection since
$\Mer_m$ is saturated and its restriction to $\alpha_m^*\Pic(X_{m-1})$ is
precisely $\Mer_{m-1}$).  Moreover, since the generators of $[e_m]$ and
$[f-e_m]$ are invertible as difference operators, we find that the natural
map
\[
\alpha_m^*\alpha_{m*}\Mer_m\to \Mer_m
\]
is surjective; indeed, the image is equal to $\Mer_m$ as a module, not just
as a sheaf.  Moreover, invertibility of $[e_m]$ also implies that
\[
\Hom(\sO_{e_m}(-1),\Mer_m)=0,
\]
and thus the natural map
\[
\Mer_m\to \alpha_m^!\alpha_{m*}\Mer_m\cong \alpha_m^!\Mer_{m-1}
\]
is injective.  Since $T_{\Mer}$ is an isomorphism, we may conjugate by
$\theta$ to obtain an injective map $\Mer_m\to\alpha_m^*\Mer_{m-1}$ with
direct image the identification $\alpha_{m*}\Mer_m\cong \Mer_{m-1}$.
Composing with the surjection $\alpha_m^*\Mer_{m-1}\to \Mer_m$ gives a map
$\Mer_m\to \Mer_m$ agreeing with the identity on $\alpha_m^*\Pic(X_{m-1})$,
and thus (since elements of those degrees generate the module) agreeing
with the identity everywhere.  It follows that $\Mer_m$ is a direct summand
of $\alpha_m^*\Mer_{m-1}$.  The other direct summand is necessarily a power
of $\sO_{e_m}(-1)$, and since there are no maps $\alpha_m^*\Mer_{m-1}\to
\sO_{e_m}(-1)$, we find that the map is an isomorphism as required.
\end{proof}

\begin{lem}
  If $T_M=0$, then $R\Hom(M,\Mer)=0$.
\end{lem}

\begin{proof}
  This follows immediately from the fact that $T_\Mer$ is an isomorphism.
\end{proof}

In particular, this implies that $\Mer_{m-1}$ is $\alpha_m^*$-acyclic, as
it has no $0$-dimensional subsheaf.

\begin{prop}
  For any divisor class $D$, $\Ext^p(\sO_X(D),\Mer)=0$ for $p>0$.
\end{prop}

\begin{proof}
  Since $\Mer$ is translation-invariant, it suffices to consider the case
  $D=0$.  For $m>0$, we find
  \begin{align}
  R\Hom(\sO_{X_m},\Mer_{X_m})
  &\cong
  R\Hom(\sO_{X_m},L\alpha_m^!\Mer_{X_{m-1}})\notag\\
  &\cong
  R\Hom(R\alpha_{m*} \sO_{X_m},\Mer_{X_{m-1}})\notag\\
  &\cong
  R\Hom(\sO_{X_{m-1}},\Mer_{X_{m-1}}),
  \end{align}
  and thus we reduce by induction to the case $m=0$.  Now, as above, we
  find that if $(\rho_*,\rho^*)$ are the natural functors between $\qcoh(X_0)$
  and $\qcoh(\P^1)$, then
  \[
  \rho^*\rho_*\Mer\to \Mer
  \]
  is surjective (though only as a map of sheaves), since the submodule of
  $\Mer$ generated by elements of degree in $\Z f$ contains all elements of
  degree $D$ with $D\cdot f\ge 0$.  It follows that $\Mer$ is acyclic for
  $\rho_*$, and thus
  \[
  R\Hom(\sO_X,\Mer)
  \cong
  R\Hom(\rho^*\sO_{\P^1},\Mer)
  \cong
  R\Hom(\sO_{\P^1},R\rho_*\Mer)
  \]
  Since $R\rho_*\Mer$ is the sheaf on $\P^1$ induced by a vector space over
  $\C(\P^1)$, it is injective, and thus the claim follows.
\end{proof}

Call a coherent sheaf ``vertical'' if its Chern class is orthogonal to $f$.

\begin{lem}
  Suppose $M$ is a vertical sheaf.  Then $R\Hom(M,\Mer)=0$.
\end{lem}

\begin{proof}
  Since the claim holds for $0$-dimensional sheaves, we may reduce to
  considering irreducible $1$-dimensional sheaves.  Since $R^1\alpha_{m*}M$
  is $0$-dimensional, we may similarly reduce to the case $m=0$, so that
  $c_1(M)\propto f$.  Then Lemma \ref{lem:sheaves_of_class_df} implies
  $c_1(M)=f$.  Twisting by a suitable multiple of $s$ then lets us ensure
  $\chi(M)=1$, and by the proof of Lemma \ref{lem:sheaves_of_class_df},
  this implies $M\cong \sO_X/\sO_X(-f)$.  Any nonzero map $\sO_X(-f)\to
  \sO_X$ becomes an isomorphism under $\Hom(\_,\Mer)$, and thus the
  quotient satisfies $R\Hom(M,\Mer)=0$ as required.
\end{proof}

\begin{lem}
  If $M$ is a quotient of $\sO_X(D)$ for some $D$, then $\Ext^p(M,\Mer)=0$
  for $p>0$.
\end{lem}

\begin{proof}
  If $M=\sO_X(D)$, then the result is already known, and thus we may assume
  $\rank(M)=0$.  We also note that if $M=\coker\phi$ with
  $\phi:\sO_X(D_1)\to \sO_X(D_2)$, then the claim follows from surjectivity
  of $\Hom(\phi,\Mer)$ together with acyclicity of $\sO_X(D)$ for
  $\Hom(\_,\Mer)$.

  More generally, we may use the translation invariance of $\Mer$ to
  arrange that $D=c_1(M)$.  Now, consider the short exact sequence
  \[
  0\to N\to \sO_X(c_1(M))\to M\to 0.
  \]
  The sheaf $N$ has numerical invariants $(1,0,\chi(N))$ where
  \[
  \chi(N) = \frac{c_1(M)\cdot (c_1(M)+C_m)}{2} + 1-\chi(M).
  \]
  In other words, $N$ corresponds to a point in $\Hilb^n(X)$, where
  $n=1-\chi(N)$.  By the proof of Lemma \ref{lem:Hilb_bounded_X0}, there is
  an injective morphism $\sO_X(-nf)\to N$.  Consider the quotient short exact
  sequence
  \[
  0\to N/\sO_X(-nf)\to \sO_X(c_1(M))/\sO_X(-nf)\to M\to 0.
  \]
  Since $c_1(N/\sO_X(-nf))=nf$, the first term is annihilated by
  $R\Hom(\_,\Mer)$, and thus
  \[
  \Ext^p(M,\Mer)\cong \Ext^p(\sO_X(c_1(M))/\sO_X(-nf),\Mer)=0
  \]
  for $p>0$.
\end{proof}

\begin{rem}
  The key idea above was the association of a straight-line equation to a
  quotient of a line bundle in such a way as to not change the space of
  solutions.  If $\dim\Ext^1(\sO_X(-nf),N)=d>0$, then the association in
  the proof is not unique, but this can be fixed.  Indeed, we find that
  $\dim\Hom(\sO_X(-(n-d)f),N)=1$ and any map $\sO_X(-nf)\to N$ factors
  through the unique (mod scalars) map $\sO_X(-(n-d)f)\to N$.  We thus find
  that the straight-line equation arising in the proof is unique once we
  remove common factors of the coefficients.
\end{rem}

\begin{cor}
  Let $M$ be a pure $1$-dimensional coherent sheaf, and let $\phi:\sO_X\to M$ be
  a morphism.  Then for any morphism $\psi:M\to \Mer$ the symmetric
  meromorphic function corresponding to $\psi\circ \phi$ satisfies a
  symmetric difference equation of order $\le c_1(M)\cdot f$.
\end{cor}

Given the above, it is straightforward to obtain the following.

\begin{prop}
  The sheaf $\Mer$ is injective.
\end{prop}

\begin{proof}
  Let $M$ be any coherent sheaf, and let $\phi:\sO_X(D)\to M$ be any
  nonzero morphism.  Then $\im\phi$ is a quotient of a line bundle, so that
  $\Ext^p(\im\phi,\Mer)=0$ for $p>0$, while $\Ext^p(\coker\phi,\Mer)=0$ by
  induction on the rank and first Chern class.  This implies that
  $\Ext^p(M,\Mer)=0$ for $p>0$ as required.  Since $\Hom(\_,\Mer)$ respects
  short exact sequences of coherent sheaves, it is exact on $\qcoh X_m$ as
  required.
\end{proof}

We can moreover use similar ideas to understand the space of solutions more
precisely.  Here there is a technical issue to be considered: although
$\Hom(\sO_X(-D),\Mer)$ is naturally a vector space over the field of
$q^{D\cdot f+1}\eta$-symmetric $q$-elliptic functions, there is no such
structure for $\Hom(M,\Mer)$ in general.  However, if we fix a surjection
$\sO_X(-D)^n\to \Mer$, then this allows us to view $\Hom(M,\Mer)$ as a
subspace $\Hom(\sO_X(-D),\Mer)^n$, which {\em is} a vector space over the
field of symmetric elliptic functions.

\begin{thm}
  Let $M$ be a $1$-dimensional sheaf transverse to the anticanonical curve.
  Then for all sufficiently ample $D$, the natural surjection
  \[
  \sO_X(-D)\otimes_{\C} \Hom(\sO_X(-D),M)\to M
  \]
  identifies $\Hom(M,\Mer)$ with a vector space of dimension $c_1(M)\cdot
  f$ over the field of $q^{D\cdot f+1}\eta$-symmetric $q$-elliptic
  functions.
\end{thm}

\begin{proof}
  This property is preserved under extensions, so we may reduce to the case
  that $M$ is irreducible.  Then without loss of generality, $M$ is
  generated by a global section $\sO_X\to M$.  Fix such a global section,
  with kernel $I$.  Then for sufficiently ample $D$, both $I$ and $M$ are
  $\sO_X(-D)$-globally generated and acyclic, so in particular there is a
  surjection
  \[
  \Hom(\sO_X(-D),\sO_X)\to \Hom(\sO_X(-D),M).
  \]
  In particular, it follows that given any morphism $\phi:M\to \Mer$, we
  can compute the composition $\phi_D:\sO_X(-D)\otimes_{\C} \Hom(\sO_X(-D),M)\to
  \Mer$ from the composition $g:\sO_X\to \Mer$ as the images of suitable
  difference operators.  Conversely, given $\phi_D$, we can determine the
  image of $g$ under {\em all} operators of degree $-D$, which by
  saturation of $\Mer$ allows us to determine $g$.  We thus reduce to
  showing that the image of $\Gamma(M)$ in $\Gamma(\Mer)$ is a space of the
  requisite dimension.

  Twisting by $c_1(M)$, we reduce as above to a short exact sequence
  \[
  0\to N\to \sO_X(c_1(M))\to M\to 0,
  \]
  and may further choose an injection $\sO_X(-nf)\to N$.  Then up to a
  $0$-dimensional subsheaf, $N/\sO_X(-nf)$ is transverse to $C$, and thus
  so is $\sO_X(c_1(M))/\sO_X(-nf)$.  It follows that the difference
  operator corresponding to the map $\sO_X(-nf)\to \sO_X(c_1(M))$ has
  nonzero constant term, and thus that the corresponding solution space has
  dimension equal to the degree of the operator, namely
  \[
  (c_1(M)+nf)\cdot f = c_1(M)\cdot f.
  \]
  The claim then follows via the identification
  \[
  \Hom(M,\Mer)\cong \Hom(\sO_X(c_1(M))/\sO_X(-nf),\Mer).
  \]
\end{proof}

\begin{rem}
  Note that if $\rank(M)>0$, then we have an injective map $\sO_X(D)\to M$
  (there is a surjection from a sum of line bundles to $M$, and at least
  one of those bundles must have image of positive rank), so that
  $\Hom(M,\Mer)$ has the infinite-dimensional space $\Hom(\sO_X(D),\Mer)$
  as a quotient.  Similarly, if $M$ is 1-dimensional but not transverse to
  $C$, then we may replace it with the maximal subsheaf transverse to $C$
  without changing the corresponding space of solutions.
\end{rem}

\smallskip

Though the connection to straight-line forms was convenient for
understanding inhomogeneous equations, we should also make the connection
to equations in matrix form precise.  The construction is essentially that
of \cite{rat_Hitchin}.  Clearly, we need only consider the case $m=0$.

\begin{prop}
  Let $M$ be a sheaf on $X_0$ such that
  \[
  R^1\rho_*M=R^1\rho_*M(-s)=0,
  \]
  and such that $\rho_*M$ $\rho_*M(-s)$ are vector bundles on $\P^1$.  Then
  there is a short exact sequence
  \[
  0\to \rho^*V(-s)\to \rho^*W\to M\to 0
  \]
  with vector bundles $V$, $W$.
\end{prop}

\begin{proof}
  Since $D^b\coh(X_0)$ is generated by $\sO_X$, $\sO_X(-f)$, $\sO_X(-s-f)$,
  $\sO_X(-s-2f)$, there is a complex representing $M$ in which the terms
  are direct sums of these line bundles.  In particular, we find that there
  is a distinguished triangle
  \[
  \rho^*V(-s)\to \rho^*W\to M\to
  \]
  with $V,W\in D^b\coh(\P^1)$; more precisely, $\rho^*W$ is the complex in
  which we omit the terms $\sO_X(-s-f)$ and $\sO_X(-s-2f)$, and similarly
  for $\rho^*V$.  Applying $R\rho_*$ and $R\rho_*(\_(-s))$ gives the
  desired result.
\end{proof}

\begin{rem}
  Note that (again following \cite{rat_Hitchin}) we have isomorphisms
  \begin{align}
  \Hom(R^1\rho_*M,\sO_x)&\cong \Hom(M,\theta\rho^*\sO_x)\notag\\
  \Hom(R^1\rho_*M(-s),\sO_x)&\cong \Hom(M,\theta\rho^*\sO_x(s))\notag\\
  \Hom(\sO_x,\rho_*M)&\cong \Hom(\rho^*\sO_x,M)\notag\\
  \Hom(\sO_x,\rho_*M(-s))&\cong \Hom(\rho^*\sO_x(s),M).\notag
  \end{align}
  In particular, $R^1\rho_*M$ and $R^1\rho_*M(-s)$ vanish unless $M$ has a
  map to some sheaf of the form $\sO_f(-1)$, while $\rho_*M$ and
  $\rho_*M(-s)$ are vector bundles unless $M$ has a map from a sheaf of the
  form $\sO_f$.  In particular if $M$ is pure $1$-dimensional, then when the
  hypotheses are violated, we can replace $M$ by a sheaf with smaller
  $c_1(M)\cdot (s+f)$ and with the same space of solutions.
\end{rem}

Given such a presentation, we can identify $\Hom(M,\Mer)$ with a subspace
of $\Hom(\rho^*W,\Mer)$, which in turn is equal to a space of vectors (of
length $\rank(W)$) of $q\eta$-symmetric meromorphic functions.  (Here we
use the fact that $\Hom(\sO_X(ef),\Mer)$ is independent of $e$; when we
consider ``holomorphic'' solutions below, this will no longer be the case.)

If $M$ is pure $1$-dimensional, then Chern class considerations tell us
that $\rank(V)=\rank(W)$, and we thus find that $\Hom(M,\Mer)$ can be identified
with the space of solutions of
\begin{align}
B(\eta/z)^tw(\eta/z) &= B(z)^t w(z)\notag\\
w(q\eta/z)&=w(z)\notag
\end{align}
where $B$ corresponds to the map $\rho^*V(-s)\to \rho^*W$.  More precisely,
if we write
\begin{align}
V&\cong \bigoplus_{1\le i\le n} \sO_{\P^1}(d_i),\notag\\
W&\cong \bigoplus_{1\le i\le n} \sO_{\P^1}(e_i),\notag
\end{align}
then $\hat{B}\in \Hom(\rho^*V(-s),\rho^*W)$ corresponds to a matrix of
difference operators, such that
\[
\hat{B}_{ij} \in \cS'(-e_if,s-d_jf).
\]
Taking the analytic form of the category, we see that the corresponding
operator has the form $\frac{z}{\theta_p(z^2/\eta)}(B(z)-B(\eta/z)T)$,
where $B$ is a holomorphic matrix such that
\[
B_{ij}(pz)
= 
(-qx_0/z) (\eta/pz^2)^{-d_j} (q\eta/pz^2)^{e_i}
B_{ij}(z)
\]
and thus the solution $w$ must satisfy
\[
\sum_i B_{ij}(z) w_i(z) - B_{ij}(\eta/z) w_i(qz) = 0.
\]
Since $w_i(q\eta/z)=w_i(z)$ and thus $w_i(qz)=w_i(\eta/z)$, this is easy to
put in the above form.  Note that $B$ will be invertible precisely when
$T_M$ is injective.

It also follows from the above proof that the association of $V$ and $W$ to
$M$ is functorial.  As a result, if $\phi:M\to N$ is a morphism between two
sheaves that correspond in this way to matrix difference equations, then
there are corresponding morphisms between vector bundles on $\P^1$, and
thus there are matrices $C$, $D$ with $C(\eta/z)=C(z)$, $D(q\eta/z)=D(z)$
such that
\[
B_N(z)C(z) = D(z)B_M(z).
\]
We find that $C$ and $D$ are invertible iff the kernel and cokernel of
$\phi$ have Chern class a multiple of $f$ (or,
on some blowup, orthogonal to $f$).  Indeed, consider the commutative
diagram of short exact sequences
\[
\begin{CD}
  0@>>> \rho^*V_M(-s)@>B_M >> \rho^*W_M @>>> M@>>> 0\\
  @. @V{\rho^*C(-s)}VV @V{\rho^*D}VV @V{\phi}VV @.\\
  0@>>> \rho^*V_N(-s)@>B_N >> \rho^*W_N @>>> N@>>> 0.
\end{CD}
\]
Now, $C$ and $D$ are invertible iff they are injective with $0$-dimensional
cokernel.  Thus if $C$ and $D$ are invertible, then $\ker\phi$ and
$\coker\phi$ are the kernel and cokernel of a map between vertical sheaves.
In the other direction, $\rho_*\ker\phi$ would be a $0$-dimensional sheaf
contained in $\rho_*M$, and thus $\rho_*\ker\phi=0$.  But this implies that
$\rho_*M\subset \rho_*N$, so that $D$ and thus $C$ are injective.
Moreover, the cokernel of $D$ has a filtration with subquotients
$R^1\rho_*\ker\phi$ and $\rho_*\coker\phi$, both of which are
$0$-dimensional, so that $\coker\rho^*D$ is vertical, making $D$
invertible.  Since we have a four-term exact sequence in which all but one
term is vertical, the remaining term $\coker\rho^*C$ must also be vertical,
so that $C$ is also invertible.

Note that the corresponding identification between the solution spaces is
also an immediate consequence of the fact that $c_1(N/M)\cdot f=0$, so that
$R\Hom(N/M,\Mer)=0$.

Of particular interest is the case that $M$ and $N$ are pseudo-twists
(i.e., that $N$ is the direct image of a twist of the minimal lift of $M$
by some $e_i$, or vice versa).  In that case, one finds that the relation
between the two equations is precisely as predicted in \cite{rat_Hitchin}
(apart from the effects of the scalar gauge transformation if we work with
the elliptic form of the category), with precisely the same argument as in
the commutative case.  (The same applies to the other atomic operations
considered there.)

To deal with more general gauge/isomonodromy relations, we define an
equivalence relation $\sim_h$ on pure $1$-dimensional sheaves on $X_0$ or
$X'_0$ to be the weakest equivalence relation such that $M\sim_h N$ if
there is a morphism $\phi:M\to N$ with vertical kernel and cokernel.

\begin{prop}
  Suppose $M$ and $N$ are sheaves with no vertical subsheaf and no pure
  vertical quotient sheaf.  If $M\sim_h N$, then $M$ and $N$ are comparable
  in the sense of Proposition \ref{prop:comparable_implies_pseudo-twist}.
\end{prop}

\begin{proof}
  For any coherent sheaf $M$, let $GM$ denote the quotient of $M$ by its
  maximal vertical subsheaf.  This operation is functorial, and $M\sim_h
  GM$ via the natural morphism $M\to GM$.  Moreover, given a morphism
  $\phi:M\to N$ with vertical kernel and cokernel, the snake lemma applied
  to
  \[
  \begin{CD}
    0@>>> M' @>>> M @>>> GM @>>>0\\
    @. @VVV @V{\phi}VV @V G\phi VV @.\\
    0@>>> N' @>>> N @>>> GN @>>> 0
  \end{CD}
  \]
  gives an exact sequence in which $\ker G\phi$ is sandwiched between two
  vertical sheaves, so is vertical.  Since $GM$ has no map from a vertical
  sheaf, $G\phi$ is injective.  Applying $R^1\ad G R^1\ad$ (i.e., taking
  the largest horizontal subsheaf with {\em pure} vertical kernel) gives a
  morphism which is still injective but now has $0$-dimensional cokernel.

  In other words, if we take any sequence of pure 1-dimensional sheaves in
  which consecutive sheaves are related by such morphisms in either
  direction, then applying $R^1\ad G R^1\ad G$ to the sequence gives
  a sequence of {\em comparable} sheaves.  Since the original sheaves $M$,
  $N$ have $R^1\ad G R^1\ad GM\cong M$ and $R^1\ad G R^1\ad G N\cong N$,
  the claim follows.
\end{proof}

If $M$ is a sheaf corresponding to a difference equation, then we can use
the operations of \cite{rat_Hitchin} to produce a sheaf with no maps to or
from vertical sheaves; indeed, if $M$ has a map to or from some sheaf of
the form $\sO_f(d)$, then performing suitable pseudo-twists in one of the
points of intersection turns the sheaf into a subsheaf $\sO_f(-1)$ or a
quotient $\sO_f$, both of which are automatically removed in the
translation between sheaves and equations.  It follows that if $M\sim_h N$,
then we can apply those operations to obtain new sheaves with no vertical
sub- or quotient sheaves, and thus apply the Proposition to conclude that
the new sheaves are comparable.  But then we may apply Proposition
\ref{prop:comparable_implies_pseudo-twist} to conclude that the resulting
sheaves are themselves related by a sequence of pseudo-twists.  In other
words, if the equations corresponding to $M$ and $N$ are related by a
gauge/isomonodromy equivalence, then we can make the two equations agree
via some sequence of these canonical operations.

\subsection{Sheaves of holomorphic solutions}
\label{sec:diffeq2_sheaves}

In \cite{isomonodromy}, the term ``isomonodromy'' was justified via a
notion of ``weak monodromy'': given a symmetric difference equation
\[
v(qz) = A(z) v(z),
\]
with $A(pz)=A(z)$, any symmetric fundamental matrix $M(z)$ necessarily
\cite{EtingofPI:1995} satisfies an equation of the form
\[
M(pz) = M(z) \hat{A}(z)
\]
where $\hat{A}(z)$ is $q$-elliptic.  Of course, $M(z)$ is not unique, and
thus $\hat{A}(z)$ is only determined up to a suitable equivalence relation
(corresponding to taking the quotient of the category of coherent sheaves
on $X'_0$ by the subcategory of $0$-dimensional sheaves).  This suggests
that there should be a relation between the corresponding noncommutative
surfaces.

In \cite{KricheverIM:2004}, an analogous construction was given in the case
of a discrete connection on a vector bundle on $\C^*/\langle p\rangle$
(satisfying certain genericity conditions that unfortunately exclude
pullbacks from $\P^1$).  In contrast to the trivial bundle case,
Krichever's construction gives a unique choice of $\hat{A}(z)$.  A careful
consideration of that construction reveals the source of this
rigidification: one is looking for a fundamental solution with significant
constraints on the singularities of the solutions.  The argument of
\cite{PraagmanC:1986} for the existence of meromorphic solutions reduces to
the study of a particular sheaf of solutions (consisting of all local
meromorphic solutions), and it turns out that we can adapt Praagman's
construction to our purposes.  In particular, we will find that the
symmetric solutions satisfying suitable constraints on their singularities
can be naturally identified with the space of global sections of a certain
vector bundle on the quotient of $\C^*/\langle q\rangle$ by the involution,
and we can use this identification to define a suitable analogue of a
fundamental matrix.

As in the case of meromorphic solutions, it will be convenient to
understand this sheaf via descent, and thus first consider the case of a
non-symmetric $q$-difference equation.  Consider an equation $v(qz) =
A(z)v(z)$ where $A(z)\in \GL_n(\mer(\C^*))$.  If $A$ and $A^{-1}$ were {\em
  holomorphic}, then we could interpret this equation as imposing a
$\langle q\rangle$-equivariant structure on the trivial bundle
$\sO_{\C^*}^n$.  Since $\langle q\rangle$ acts freely, any $\langle
q\rangle$-equivariant bundle descends through the quotient $\pi_q\to
\C^*\to \C^*/\langle q\rangle$, and this gives us the desired vector bundle
$V_A$ on $\C^*/\langle q\rangle$.  Moreover, this bundle is naturally
interpreted as a sheaf of holomorphic solutions of the equation: for any
open subset $U\subset \C^*/\langle q\rangle$, $\Gamma(U;V_A)$ is naturally
identified with the space $\{v(z):v(z)\in
\Gamma(\pi_q^{-1}U;\sO_{\C^*}^n)=\A(\pi_q^{-1}U)|v(qz)=A(z)v(z)\}$, where
$\A(U)$ denotes the space of analytic functions on $U$.  (Note that $U$ is
open in the analytic topology, not the Zariski topology.)

Of course, if $A$ is not invertibly holomorphic, then the above
construction fails.  Indeed, if $A$ has a pole at $x$, then to avoid a pole
at $qx$, we must constrain the value of $v$ at $x$, and vice versa if
$A^{-1}$ has a pole at $x$.  We could resolve this issue by choosing one
point from each orbit and insisting only that the solution be holomorphic
there, but it will be cleaner to allow more general local conditions on the
solution.  At the same time, since we will need to apply the construction
to discrete connections on vector bundles, we will also replace
$\sO_{\C*}^n$ by a more general vector bundle.  (In fact, since $\C^*$ is
noncompact, any vector bundle is trivial, but it is still worth making the
distinction since the bundle is not {\em canonically} trivial.)

In general, suppose $V$ is a vector bundle on the Riemann surface $X$.
(For our purposes, $X$ will be one of $\C^*$, an elliptic curve, or the
Riemann sphere.)  A {\em local condition} on $V$ at the point $x\in X$ is a
space $\sigma_x$ of germs of meromorphic sections of $V$ at $x$, such that
$\sigma_x$ is a free module over the ring $\sO_x$ of analytic germs at $x$,
with $\rank(\sigma_x)=\rank(V)$.  We then say that a meromorphic section
$v$ of $V$ {\em satisfies the local condition $\sigma_x$} if the germ of
$v$ at $x$ is contained in $\sigma_x$.  In addition to the trivial local
condition (where $\sigma_x$ is the space $V_x$ of germs of {\em
  holomorphic} sections of $V$), this allows us to impose conditions such
that $v(x)$ lies in a particular subspace of the fiber, or that $v$ has at
most a simple pole at $x$, with residue in a particular subspace of the
fiber.

Note that since $\sigma_x$ is a free $\sO_x$-module, we may also represent
it via a basis, and thus there is an invertible meromorphic map
$M:\sO_x^n\ratto V_x$ such that $\sigma_x = M \sO_x^n$.  Composing $M$ with
an isomorphism $V_x\cong \sO_x^n$ makes it an element of $\GL(V\otimes
\sK_x)$, where $\sK_x$ is the field of meromorphic germs at $x$.  Right
multiplying $M$ by an automorphism of $V_x$ has no effect on the image $M
V_x$, and we thus find that the set of local conditions can be naturally
identified with the quotient $\GL(V\otimes \sK_x)/\GL(V_x)$, and thus
(non-canonically) with the {\em affine Grassmannian}
$\GL_n(\sK_x)/\GL_n(\sO_x)$.  (See \cite{ZhuX:2016} for a survey.)

More generally, a {\em system of local conditions} $\sigma$ on $V$ is an
assignment of a local condition $\sigma_x$ at every point $x\in X$, subject
to the constraint that the set of points assigned nontrivial local
conditions must be discrete.

\begin{prop}
  Let $V$ be a vector bundle on the Riemann surface $X$, and let $\sigma$
  be a system of local conditions on $V$.  Then there exists a vector
  bundle $W$ and a meromorphic map $M:W\ratto V$ such that $\sigma_x=M W_x$
  for all $x\in X$.  Moreover, this pair $(W,M)$ is uniquely determined up
  to unique isomorphism.
\end{prop}

\begin{proof}
  In general, we may define a sheaf $V_\sigma$ by letting
  $\Gamma(U;V_\sigma)$ be the space of meromorphic sections of $V$ on $U$
  that satisfy the local condition $\sigma_x$ for all $x\in U$.  If this
  sheaf is a vector bundle of the same rank at $V$, then we have a natural
  isomorphism $M_\sigma:V_\sigma\otimes \sK_X\cong V\otimes \sK_X$, giving
  a meromorphic map $M$ as required.  Moreover, if $(W,M)$ is any pair with
  $\sigma_x=M W_x$, then $M_\sigma (V_\sigma)_x=M W_x$ and thus
  $M^{-1}M_\sigma$ is an isomorphism.  It thus remains only to show that
  $V_\sigma$ is in fact a vector bundle.  This in turn reduces to showing
  that this holds on a suitable covering of $X$ by open subsets.

  Since the set of $x$ with $\sigma_x$ nontrivial is discrete, its
  complement $U_{\reg}$ is open, and we immediately find
  $V_\sigma|_{U_{\reg}}=V|_{U_{\reg}}$.  If $\sigma_x$ is nontrivial, then
  it is represented by a matrix $M_x$ of meromorphic germs, and there is
  thus a neighborhood $U_x$ of $x$ such that $M_x$ is invertibly
  holomorphic on the punctured neighborhood $U_x\setminus x$.  We then find
  $V_\sigma|_{U_x}\cong M_x V|_{U_x}$, so that again $V_\sigma|_{U_x}$ is a
  vector bundle.  Since $X = U_{\reg} \cup \bigcup_x U_x$, this establishes
  that $V_\sigma$ is indeed a vector bundle as required.
\end{proof}

\begin{rem}
  When $X$ is a compact Riemann surface (so an algebraic curve) and $V\cong
  \sO_X^n$, this reduces to a standard construction in algebraic number
  theory (Weil uniformization), usually expressed in terms of the ring of
  ad\`eles.  In the noncompact case, every vector bundle is trivial, and
  thus this can be restated to say that every system of local conditions on
  $\sO_X^n$ comes from an element of $\GL_n(\mer(X))$; when $X$ is a region in
  $\C$, this is a result of Birkhoff \cite{BirkhoffGD:1916}, in which an
  explicit product expression is given for the meromorphic matrix.
\end{rem}

Define a category structure on the pairs $(V,\sigma)$ by taking the
morphisms $(V,\sigma)\to (W,\tau)$ to consist of those holomorphic maps
$A:V\to W$ such that $A\sigma_x\subset \tau_x$ for all $x\in X$.
Similarly, the triples $(V,W,M:W\ratto V)$ (i.e., with $M$ an invertible
meromorphic map) have a natural category structure: holomorphic maps $V\to
V'$, $W\to W'$ making the obvious diagram commute.

\begin{prop}
  There is an equivalence of categories from the category of pairs
  $(V,\sigma)$ and the category of triples $(V,W,M:W\ratto V)$ given on
  objects by $(V,\sigma)\mapsto (V,V_\sigma,M_\sigma)$; $(V,W,M)\mapsto
  (V,M W_x)$.
\end{prop}

In fact, this equivalence is {\em monoidal}: there is a natural notion of
tensor product on both categories preserved by the equivalence.  It is easy
to see that if $\sigma$ is a system of local conditions on $V$ and $\tau$
is a system of local conditions on $W$, then $(\sigma\otimes
\tau)_x:=\sigma_x\otimes_{\sO_x} \tau_x$ defines a system of local
conditions on $V\otimes W$, with $(V\otimes W)_{\sigma\otimes \tau}\cong
V_\sigma\otimes W_\tau$, identifying $M_{\sigma\otimes\tau}$ with
$M_\sigma\otimes M_\tau$.  The identity for the tensor product is
$(\sO_X,1)$ (the trivial system of local conditions), mapping to
$(\sO_X,\sO_X,1)$.  It follows that the equivalence respects arbitrary
Schur functors, in particular both symmetric and exterior powers; the case
$\wedge^n=\det$ will be particularly useful below.  Similarly, there is a
natural notion of duality on local conditions induced by the isomorphism
$g\mapsto g^{-t}:\GL(V)\to \GL(V^*)$, and the equivalence takes
$(V^*,\sigma^*)$ to $M_\sigma^{-t}$.

In addition, this is a {\em holomorphic} equivalence of categories, in the
following sense.  Suppose $\pi:X\to Y$ is an analytic map of complex
manifolds such that the fibers are Riemann surfaces.  Then a vector bundle
on $X$ may be viewed as a (flat) family of vector bundles on the fibers.
We may thus define a (flat, holomorphic) family of triples $(V,W,M)$ to be
a meromorphic isomorphism $M:W\ratto V$ of vector bundles on $X$ such that
every fiber of $\pi$ contains a point where $M$ is invertibly holomorphic.
In particular, $M$ determines a meromorphic isomorphism
$W|_{\pi^{-1}(y)}\ratto V|_{\pi^{-1}(y)}$ on every fiber.  We may then
apply the above equivalence to obtain a system of local conditions on each
fiber of $V$.  Moreover, the constraint that a given family of vector
bundles with local conditions arises in this way is a local condition: all
we need is that at each point $x\in X$, there is an invertible matrix of
meromorphic germs at $x$ that determines the appropriate local condition at
every point in some neighborhood of $x$.  We may thus use the latter local
condition to define a holomorphic family of pairs $(V,\sigma)$, and find
that both the equivalence and its inverse take holomorphic families to
holomorphic families.

There is one crucial bit of additional structure: the category of triples
$(V,W,M)$ has a natural involution $(V,W,M)\mapsto (W,V,M^{-1})$.  It thus
follows immediately that the category of vector bundles with local
conditions {\em also} has an involution, which we may express as
$(V,\sigma)\mapsto (V_\sigma,\sigma^{-1})$.  This involution is, of course,
covariant, but there is also a contravariant analogue: the monoidal
property in particular implies that the functors respect duality, and thus
we could instead take $(V,W,M)\mapsto (W^*,V^*,M^t)$.  This is less natural
from a solution sheaf perspective, so our general discussion will use the
covariant functor.  Note, however, that the contravariant functor will turn
out to be somewhat more natural from a {\em geometric} perspective.

Relatedly, there is a natural notion of the composition of two triples
$(V_1,V_2,M_1)$, $(V_2,V_3,M_2)$, namely $(V_1,V_3,M_1M_2)$, and the
involution is of course just the inverse with respect to this composition.

\begin{prop}
  Suppose $(V,\sigma)$, $(W,\tau)$ are two vector bundles with systems of
  local conditions.  Then $\Hom(V_\sigma,W_\tau)$ can be naturally
  identified with the space of meromorphic maps $A:V\ratto W$ such
  that $A\sigma_x \subset \tau_x$ for all $x\in X$.
\end{prop}

\begin{proof}
  We have natural isomorphisms
  \[
  \Hom(V_\sigma,W_\tau)\cong \Gamma(X;V_\sigma^*\otimes W_\tau)
  \cong \Gamma(X;(V^*\otimes W)_{\sigma^*\otimes \tau})
  \]
  Meromorphic sections of $V^*\otimes W$ are naturally identified with
  meromorphic maps $V\ratto W$, and we readily see that such a map $A$
  satisfies the local condition $(\sigma^*\otimes \tau)_x$ iff
  $A\sigma_x\subset \tau_x$.
\end{proof}
\medskip

To apply this to $q$-difference equations, we take $X$ to be $\C^*$ and
let $q\in \C^*$ be such that $|q|<1$.  The analogue of a $q$-difference
equation on a vector bundle is the following.

\begin{defn}
  A {\em meromorphic $q$-connection} on a vector bundle $V$ is a
  meromorphic map $A:V\ratto q^*V$.
\end{defn}

\begin{rem} Note that $A$ makes the sheaf $V\otimes {\cal K}$ equivariant
  relative to the action of $\langle q\rangle$, and conversely any such
  equivariant structure on $V\otimes {\cal K}$ gives rise to a meromorphic
  $q$-connection.
\end{rem}

\begin{defn}
  A {\em meromorphic solution} of a meromorphic $q$-connection $A$ on $V$
  is a meromorphic section $v$ of $V$ such that $q^*v = A v$.
\end{defn}

\begin{rem}
  Of course, if $V\cong \sO_X^n$, then there is a natural choice of
  isomorphism $V\cong q^*V$, which allows us to express $A$ as a matrix.
  The condition to be a solution then becomes $v(qz)=A(z)v(z)$ as one would
  expect.
\end{rem}

Now, suppose we are given a system of local conditions $\sigma$ on $V$.  We
could then define a $\sigma$-holomorphic solution of the meromorphic
$q$-connection to be a meromorphic solution which satisfies the local
condition $\sigma_x$ for all $x$.  We then find that a solution is
$\sigma$-holomorphic iff it is $A^{-1} q^*\sigma$-holomorphic, and thus
$\sigma\cup A^{-1}q^*\sigma$-holomorphic.  (Here $A^{-1}q^*\sigma$ is the
system of local conditions given by $(A^{-1}q^*\sigma)_x
=A^{-1}\sigma_{qx}$.)  Thus to obtain a well-behaved notion of a
holomorphic solution, we should insist that $\sigma$ satisfy
$\sigma=A^{-1}q^*\sigma$.  (Otherwise, we could easily specify conditions
on a discrete set of points that imply nontrivial conditions on a dense set
of points.)

\begin{defn}
  A {\em system of local conditions} on the meromorphic $q$-connection
  $(V,A)$ is a system of local conditions $\sigma$ on $V$ such that
  $A\sigma = q^*\sigma$.
\end{defn}

Given a system of local conditions $\sigma$ on the meromorphic
$q$-connection $(V,A)$, we may then define a sheaf $\Sol(V,A,\sigma)$ on
the quotient $\C^*/\langle q\rangle$ by taking $\Gamma(U;\Sol(V,A,\sigma))$
to be the space of meromorphic solutions of $(V,A)$ on $\pi_q^{-1}U$
satisfying the local condition $\sigma_x$ at every point $x\in
\pi_q^{-1}U$.

\begin{prop}
  The sheaf $\Sol(V,A,\sigma)$ is a vector bundle of the same rank as $V$.
\end{prop}

\begin{proof}
  We first observe that there is an induced meromorphic $q$-connection
  $A_\sigma$ on the bundle $V_\sigma$ induced by gauging by
  $M_\sigma:V_\sigma\ratto V$: $A_\sigma = q^*M_\sigma^{-1} A_\sigma
  M_\sigma$.  In fact, since $A\sigma=q^*\sigma$ for all $g$, this
  $q$-connection is holomorphic: $A_\sigma$ is actually an isomorphism
  $V_\sigma\cong q^*V_\sigma$.  In other words, $A_\sigma$ makes $V_\sigma$
  a $q$-equivariant vector bundle, allowing us to descend to a vector
  bundle $W$ on the quotient.  By definition, $\Gamma(U;W) =
  \Gamma(\pi_q^{-1}U;V_\sigma)^{\langle q\rangle} =
  \Gamma(U;\Sol(V,A,\sigma))$, so that $\Sol(V,A,\sigma)$ is indeed a
  vector bundle as required.
\end{proof}

\begin{defn}
  The {\em holomorphic fundamental matrix} of a meromorphic $q$-connection
  with a chosen system of local conditions is the composition
  $M_{V,A,\sigma}:\pi_q^*\Sol(V,A,\sigma)\cong V_\sigma\ratto V$.
\end{defn}

We then immediately find the following.

\begin{thm}
  The map $(V,A,\sigma)\mapsto (V,\Sol(V,A,\sigma),M_{V,A,\sigma})$ induces
  a natural (monoidal) equivalence from the category of triples
  $(V,A,\sigma)$ (meromorphic $q$-connections with a system of local
  conditions) to the category of triples $(V,W,M:\pi_q^*W\ratto V)$, where
  $V$ is a vector bundle on $\C^*$, $W$ is a vector bundle on $\C^*/\langle
  q\rangle$, and $M$ is an invertible meromorphic map.
\end{thm}

\begin{rem}
  The inverse equivalence is given by taking $M$ to the meromorphic
  $q$-connection $q^*M M^{-1}$ on $V$, with associated system of local
  conditions $(\sigma_M)_x:=M(\pi_q^*W)_x$.
\end{rem}

\begin{cor}
  There is also a natural (monoidal) equivalence between the categories of
  triples $(V,A,\sigma)$ and the category of pairs $(W,\tau)$ with $W$ a
  vector bundle on $\C^*/\langle q\rangle$ and $\tau$ a system of local
  conditions on $\pi_q^*W$.
\end{cor}

\begin{proof} Apply the involution $M\mapsto M^{-1}$, then map back to the
  category of systems of local conditions.
\end{proof}

\begin{cor}
  There is a natural identification between
  $\Hom(\Sol(V,A,\sigma),\Sol(W,B,\tau))$ and the space of meromorphic maps
  $M:V\ratto W$ such that $M\sigma\subset\tau$ and $q^*M A = B M$.
\end{cor}

In particular, the solution sheaves are isomorphic (ignoring the local
conditions) iff the $q$-connections are gauge equivalent.

The symmetry of the original functor also tells us that there is a natural
meromorphic map $\Sol(V,A,\sigma)\ratto \Sol(V,A,\tau)$ associated to any
change in the system of local conditions.  In fact, this is defined on the
quotient $\C^*/\langle q\rangle$: the point is that if $\sigma$ and $\tau$
are two systems of local conditions on $(V,A)$, then so is
$\sigma\cap\tau$, and functoriality gives us actual morphisms
$\Sol(V,A,\sigma\cap\tau)\to \Sol(V,A,\sigma)$ and
$\Sol(V,A,\sigma\cap\tau)\to\Sol(V,A,\tau)$ of sheaves on $\C^*/\langle
q\rangle$.  If $\sigma$ and $\tau$ agree outside a single orbit
$\pi_q^{-1}(x)$, then $\Sol(V,A,\sigma)\ratto \Sol(V,A,\tau)$ will be an
isomorphism on the complement of $x$.

The monoidal structure tells us that $\det\Sol(V,A,\sigma)\cong
\Sol(\det(V),\det(A),\det(\sigma))$, so that we would in particular like to
understand the line bundle associated to a first-order equation (a
$q$-connection on a line bundle).  Taking $\det(V)\cong \sO_{\C^*}$ gives
an equation $v(qz)=a(z)v(z)$, and the system of local conditions simply
specifies the divisor of the desired solution.  In particular, Weierstrass
tells us that there is a function $f(z)$ with precisely that divisor, and
gauging by this function gives a new equation
$\hat{v}(qz)=\hat{a}(z)\hat{v}(z)$ with $\hat{a}(z)=f(qz)^{-1}a(z)f(z)$ and
{\em trivial} local conditions.  Since the new equation is compatible with
the trivial system of local conditions, it must be invertibly holomorphic
on $\C^*$, so that there is an expression of the form $\hat{a}(z) = C z^k
\exp(\sum_{l\ne 0} c_l z^l)$.  (Here $k$ comes from the winding number of
$\hat{a}(z)$ around $0$, and dividing by $z^k$ gives a function with a
global logarithm on $\C^*$.)  Gauging by $\exp(\sum_{l\ne 0} c_l
z^l/(q^l-1))$ gives the equation $w(qz)=C z^k w(z)$, and this identifies
the line bundle.

Another approach one can take is to first solve the equation
meromorphically and gauge to obtain a trivial equation with nontrivial
system of local conditions.  The associated divisor is $q$-periodic, so
induces a divisor on $\C^*/\langle q\rangle$ and thus a line bundle (of
$q$-periodic functions with at least the given divisor).  This approach is
trickier to push through in complete generality, but is more
straightforward when the equation has rational or elliptic coefficients.

\medskip

In addition to the relation to discrete isomonodromy transformations (i.e.,
gauge transformations \cite{BorodinA:2004}), another reason we consider
this an analogue of monodromy comes from the case of $q$-difference
equations with rational coefficients.  In that case, $A$ has only finitely
many singularities, and there are two particularly natural choices for the
system of local conditions: $\sigma^0$, defined by the condition that
$\sigma^0_x=1$ for all sufficiently small $x$, and $\sigma^\infty$, defined
by the condition that $\sigma^\infty_x=1$ for all sufficiently large $x$.
(When $A$ has ``stable'' singularities (see below), we will give a more
general construction of natural systems of local conditions which will
include these two as special cases.)  We thus obtain solution bundles $W^0$
and $W^\infty$, where local sections of $W^0$ consist of solutions which
are holomorphic on a punctured neighborhood of $0$ and similarly for
$W^\infty$.  Since these came from the same $q$-connection, there is also
an induced meromorphic map between them.  If $A$ is holomorphic at $0$ and
$\infty$ with $A(0)=A(\infty)=1$, then both of these bundles are trivial,
and the meromorphic map becomes a matrix of $q$-elliptic functions.  This
recovers Birkhoff's notion of monodromy in this case
\cite{BirkhoffGD:1913}.  Note that the elliptic monodromy will similarly be
constructed as a map between two solution bundles with different systems of
local conditions.

\medskip

We have thus associated to every meromorphic $q$-difference equation with
associated local conditions a vector bundle $\Sol(V,A,\sigma)$ on the
elliptic curve $\C^*/\langle q\rangle$.  It follows from Atiyah's
classification of vector bundles on elliptic curves that every such bundle
arises in this way, and in fact one can choose the $q$-connection (on the
trivial bundle and with trivial system of local conditions) to lie in
$\GL_n(\C[z,1/z])$.  This gives a more concrete, though less natural,
notion of fundamental matrix: if $A\in \GL_n(\mer(\C^*))$ defines a
$q$-difference equation and $\mu\in \GL_n(\C[z,1/z])$ is a multiplier
representing $\Sol(\sO_{C^*}^n,A,\sigma)$, then there is an invertible
matrix $M$ such that (a) every column of $M$ satisfies the local conditions
$\sigma$, and (b) $M$ satisfies the $q$-difference equation $M(qz) =
A(z)M(z)\mu(z)^{-1}$. Indeed, this is just the matrix associated to the
isomorphism $\Sol(\sO_{C^*}^n,\mu,1)\cong \Sol(\sO_{\C^*}^n,A,\sigma)$,
where $1$ denotes the trivial local condition.

Note that although Atiyah's construction gives rise to a fairly natural
choice of multiplier $\mu$ (up to the indeterminacy already familiar from
the line bundle case, where $v(qz)=C z^k v(z)$ and $v(qz)=q C z^k v(z)$
represent the same bundle), the mere constraint that $\mu$ have no finite
singularities leaves far more freedom.

Consider, for instance, the equation $v(qz)=A(z)v(z)$ with
\[
A(z) = \begin{pmatrix} 1&-z\\-z&1+z^2\end{pmatrix},
\]
essentially Ismail's $q$-Airy equation \cite{IsmailMEH:2005}.  This has no
finite singularities, so we may take the trivial system of local
conditions.  Since $A(0)=1$, this has a fundamental matrix which is
holomorphic at $0$ with value $1$, which is then easy to compute
inductively:
\[
M(z)
=
\prod_{k\ge 0} A(q^k z)^{-1}
=
\begin{pmatrix}
  \sum_{0\le k} \frac{q^{k(k-1)}z^{2k}}{(q;q)_{2k}} &
  \sum_{0\le k} \frac{q^{k^2}z^{2k+1}}{(q;q)_{2k+1}} \\
  \sum_{0\le k} \frac{q^{k(k+1)}z^{2k+1}}{(q;q)_{2k+1}} &
  \sum_{0\le k} \frac{q^{k^2}z^{2k}}{(q;q)_{2k}}
\end{pmatrix}.
\]
In particular, it follows that $\Sol(\sO_{\C^*}^2,A,1)$ is trivial.  But of
course $A$ is far from the standard multiplier one would normally use to
represent the trivial bundle! In fact, it is easy to see that $M$ is
unique, and thus any fundamental matrix which is holomorphic at 0 has an
essential singularity at $\infty$.  It follows in particular that the two
equations $v(qz)=A(z)v(z)$ and $v(qz)=v(z)$, though gauge equivalent over
the group of invertibly holomorphic matrices on $\C$, are not gauge
equivalent over the field of rational functions.

This is also an example showing that Birkhoff's notion of monodromy loses
information when the equation is sufficiently irregular at $0$ or $\infty$.
Indeed, since the ``regular near 0'' and ``regular near $\infty$'' systems
of local conditions agree for the $q$-Airy equation and give rise to the
trivial bundle, the analogue of the Birkhoff monodromy is the identity map
on the trivial bundle!

\medskip

To obtain elliptic equations (possibly twisted by a vector bundle), we need
to impose a condition of the form $A(pz)=C z^k A(z)$, to make sense of
which requires that $V$ itself have a $\langle p\rangle$-equivariant
structure.  If $|p|,|q|<1$, let $\EllDiff_{p,q,C z^k}$ denote the category
of triples $(V,A,\sigma)$ where $V$ is a vector bundle on $\C^*/\langle
p\rangle$, $A$ is a meromorphic $\langle q\rangle$-connection on $\pi_p^*V$
satisfying $p^* A = C z^k A$, and $\sigma$ is a system of local conditions
on $(\pi_p^*V,A)$.

\begin{prop}
  There is a natural equivalence between the categories $\EllDiff_{p,q,C
    z^k}$ and the category of triples $(V,W,M)$ where $V$ is a vector
  bundle on $\C^*/\langle p\rangle$, $W$ is a vector bundle on
  $\C^*/\langle q\rangle$, and $M:\pi_q^*W\ratto \pi_p^*V$ is an invertible
  meromorphic map satisfying $(pq)^*M = C z^k q^*M M^{-1} p^*M.$
\end{prop}

\begin{proof}
  We can express the $\langle p\rangle$-equivariant structure on $V$
  together with the associated condition on $A$ in categorical terms, and
  thus transport it through the usual $\Sol$ functor.
\end{proof}

\begin{rem}
  This is again essentially monoidal, with the caveat that the tensor
  product multiplies the twisting factors $C z^k$.  Thus, for instance, the
  determinant of a triple $(V,A,\sigma)$ with twisting factor $C z^k$
  has twisting factor $(C z^k)^{\rank(V)}$, and the dual has twisting
  factor $C^{-1}z^{-k}$.
\end{rem}

Of course, we now find that the category of triples $(V,W,M)$ treats $V$
and $W$ in essentially symmetrical ways, so that we can apply the
involution to obtain $(W,V,M^{-1})$, replacing $C z^k$ by $C^{-1} z^{-k}$.
This gives us the following, a version of the ``elliptic Riemann-Hilbert
correspondence''.

\begin{thm}
  There is a natural equivalence $\Sol_{p,q}$ from the category
  $\EllDiff_{p,q,C z^k}$ to the category $\EllDiff_{q,p,C^{-1}z^{-k}}$,
  such that $\Sol_{q,p}\circ \Sol_{p,q}\cong \text{id}$.
\end{thm}

\begin{rem}
  One curious feature of this construction is that we did not need to
  assume that $p$ and $q$ were in any way independent; this functor is
  perfectly well-defined (and nontrivial) even when $p=q$.  The actual
  assumption we need on $p$ and $q$ is simply that neither one is on the
  unit circle, but taking them inside the unit circle does not lose any
  significant generality and makes later constructions more
  straightforward.
\end{rem}

As in the general $q$-difference case, if $\tau$ is another system of local
conditions on $(\pi_q^*V,A)$, then there is a meromorphic map
$\Sol_{p,q}(V,A,\sigma)\ratto \Sol_{p,q}(V,A,\tau)$.  Since the
$p$-connection on $\Sol_{p,q}(V,A,\sigma)$ arises
as the composition
\[
\Sol_{p,q}(V,A,\sigma)\ratto
\Sol_{p,q}(V,A,p^*\sigma)\cong p^*\Sol_{p,q}(V,A,\sigma),
\]
the induced $p$-connections on $\Sol_{p,q}(V,A,\sigma)$ and
$\Sol_{p,q}(V,A,\tau)$ are in fact gauge equivalent via this meromorphic
map.  Moreover, any meromorphic gauge equivalence arises from a change in
local conditions.  By symmetry, it follows that if we replace
$(V,A,\sigma)$ by a gauge-equivalent triple $(V',A',\sigma')$, then
$\Sol_{p,q}(V',A',\sigma')\cong \Sol_{p,q}(V,A,\sigma)$ as sheaves, and
with the same $p$-connection; only the local conditions change.

This is our main justification for viewing $\Sol_{p,q}$ as a monodromy
functor: in degenerate cases, the isomonodromy transformations are
precisely the gauge equivalences, and these are precisely the
transformations that preserve the sheaf $\Sol_{p,q}$ and its associated
meromorphic $p$-connection.

Note that when determining the determinant of the solution bundle, or
equivalently when determining the solution bundle of a first-order
equation, there is an additional complication, since the line bundle $V$
need not be trivial.  We may, of course, represent $V$ itself as associated
to a $p$-difference equation, so that $V$ is represented by an equation
$v(pz)=C_p z^{k_p} v(z)$ and then $A$ corresponds to an equation
$v(qz)=a(z)v(z)$ such that $a(pz)=C z^k q^{k_p} a(z)$.  Since $a$ is then a
meromorphic theta function, we can factor it into functions $\theta_p$ and
thus obtain an equation we may solve via elliptic Gamma functions.  Of
course, this will typically not have the correct divisor to generate the
system of local conditions, but this is again easy to fix: simply multiply
the elliptic Gamma function by an appropriate meromorphic theta function.
Gauging by the resulting solution turns the $q$-difference equation into
one associated to a line bundle, and turns the $p$-difference equation into
a $p$-connection on that bundle.

An interesting special case is when the system of local conditions is
trivial; in that case, $a$ must be holomorphic, so has the form $a(z)=C_q
z^{k_q}$, and we can take the fundamental solution to be $1$: the functor
$\Sol_{p,q}$ in that case simply swaps the roles of $p$ and $q$ in the pair
of equations.

\medskip

Of course, the above considerations do not really solve the primary issue
that the monodromy (i.e., $\Sol(V,A,\sigma)$ with its $p$-connection) is
nonunique: it simply replaces the nonuniqueness of the fundamental
meromorphic matrix with the nonuniqueness of the system of local
conditions.  Luckily, there is enough structure in the latter to allow us
to single out a relatively small set of natural choices.  Since the
solution functor is self-inverse on elliptic equations, this is as much as
we can hope for: we need to have some ability to change the system of local
conditions on the ``monodromy'' in order to capture the full set of natural
isomonodromy transformations (associated to the elliptic Painlev\'e
equation, say).

Suppose for the moment that we want to construct a system of local
conditions on a general meromorphic $q$-difference equation
$v(qz)=A(z)v(z)$.  If $A$ is invertibly holomorphic, then we have already
observed that the trivial system of local conditions on $V$ is valid for
the $q$-connection.  Otherwise, the connection has singularities, the
simplest form of which is either a simple pole of $A$ or a simple pole of
$A^{-1}$.  Let us suppose for the moment that (a) every singularity of the
connection is either a simple pole of $A$ with $A^{-1}$ holomorphic or a
simple pole of $A^{-1}$ with $A$ holomorphic, and (b) every orbit of
$\langle q\rangle$ contains at most one such singularity.  In that case,
for each orbit containing a singularity, there are two natural choices of
how the system of local conditions restricts to the orbit: it can be
holomorphic near 0 on the orbit, or holomorphic near $\infty$.  If we
associate each choice to the given singular point, then we see that we must
choose ``near 0'' for every sufficiently large singularity and ``near
$\infty$'' for every sufficiently small singularity, in order for the
singularities of the family of local conditions to be discrete.  We may
thus refer to the ``near 0'' singularities as ``large'' and the ``near
$\infty$'' singularities as ``small''.

If $A$ is $p$-elliptic (or twisted elliptic), then we have similar choices,
except that now it makes sense to impose the additional condition that if
the singular point $x$ is chosen to be small, then so should $px$.  It is
then not difficult to see that with this convention, the corresponding
equation on the solution sheaf has equally simple singularities and its
associated system of local conditions is of the same form.  (In any event,
this will follow from our more general results below.)

Although equations with such singularities suffice (via some kludges) for
the construction of the full (symmetric) elliptic Riemann-Hilbert
correspondence below, it turns out that we can relax these conditions
considerably.  The basic idea is that the family of all systems of local
conditions is closely related to the Beilinson-Drinfeld Grassmannian (see
\cite{ZhuX:2016}), which is proper in a suitable sense.  As a result, given
an equation with arbitrarily bad singularities, we could attempt to
construct a natural system of local conditions by embedding it in a
1-parameter family of equations which on a punctured neighborhood have only
the simple singularities considered above.  As long as we make continuous
choices of ``small'' and ``large'' over the punctured neighborhood, the
result will be a family of systems of local conditions on the punctured
neighborhood, and under reasonable conditions, that family will have a
limit on the puncture.  Since the systems of local conditions are
compatible with the $q$-connection away from the puncture, the limit must
also be compatible.

Of course, this would be a rather cumbersome construction to carry out in
practice, and the limit could a priori depend quite subtly on the specific
way in which we deformed the equation.  It turns out, however, that if we
impose relatively mild conditions on the equation, then there is a unique
(up to similar choices of ``small'' and ``large'') possibility for the
limit for any deformation that does not make the singularities ``worse'',
and this limit can be shown to exist without reference to any particular
deformation (or even the mere existence of a deformation).  (In particular,
this means we do not need to worry about making precise the conditions on
the deformation.)

Recall that a local condition at $x$ is labelled by a coset $\GL(V\otimes
\sK_x)/\GL(V_x)$, so (noncanonically) by a point of the affine Grassmannian
$\GL_n(\sK_x)/\GL_n(\sO_x)$.  There is a natural notion of ``distance'' on
the affine Grassmannian coming from a classification of double cosets
$\GL_n(\sO_x)\setminus \GL_n(\sK_x)/\GL_n(\sO_x)$.  Indeed, the double
cosets are labelled by dominant coweights $\lambda$ of $\GL_n$, or in other
words by nonincreasing sequences $\lambda_1\ge \cdots\ge \lambda_n$ of
integers, with associated double coset representative given by the diagonal
matrix $f^\lambda$ with entries $f^{\lambda_i}$, where $f$ is any germ with
a simple zero at $x$.  Any two local conditions give rise to a coweight in
this way, and thus in particular any given local condition has a coweight
attached (comparing it to the regular local condition).  Note that this
``distance'' is not quite symmetric:
$\lambda(\sigma_x,\tau_x)=\lambda(\tau_x,\sigma_x)^{-1}$, where for a
coweight $\lambda$, $\lambda^{-1}$ is the sequence
$(-\lambda_n,\dots,-\lambda_1)$.

In particular, a system of local conditions on a rank $n$ vector bundle on
$X$ gives rise to a map from $X$ to the set of coweights of $\GL_n$ which
is zero on the complement of a discrete set.  Moreover, any holomorphic
family of such systems satisfies a semicontinuity result: each point $x\in
X$ is a limit of finitely many points where the nearby systems of local
conditions are singular, and the coweight at $x$ of the limit is bounded in
dominance order by the sum of the corresponding coweights.  (This includes
the case in which only one singular point approaches $x$, as the coweight
can indeed change in such a family; indeed, it is even possible for a
singular point to become regular in the limit.)  Again, we have glossed
over some technical details needed to make this statement precise, as we
will only be using this as motivation for the actual construction.

What we will actually use is the following fact about the affine
Grassmannian; although it is well-known, an explicit proof is difficult to
find in the literature, and thus we give an (elementary) argument.

Define the coweight $\lambda(A)$ of an element of $\GL_n(\sK_x)$ to be the
coweight of the corresponding double coset.  We will, in fact, work over
the formal analogue $\GL_n(\C((z)))$ for simplicity; the classification of
double cosets is the same (indeed, one can replace one of the two copies of
$\GL_n(\sO_x)$ by $\GL_n(\C[z-x])$), and thus there is no difficulty
transporting formal statements to the convergent case.  Recall that the
dominance (partial) order on coweights is given by taking $\lambda\ge \mu$
iff $\sum_{1\le i\le k} \lambda_i\ge \sum_{1\le i\le k}\mu_i$ for all $k$,
with equality for $k=n$, or equivalently iff $\sum_{k<i\le n} \lambda_i\le
\sum_{k<i\le n} \mu_i$ with equality for $k=0$.

\begin{prop}
  For any $A_1,\dots,A_m\in \GL_n(\C((z)))$, one has the inequality
  \[
  \lambda(A_1\cdots A_m)\le \lambda(A_1)+\cdots+\lambda(A_m)
  \]
  in dominance ordering.  Conversely, for any matrix $A\in \GL_n(\C((z)))$
  and decomposition $\lambda(A)=\sum_{1\le i\le m} \mu^{(i)}$ of the coweight
  of $A$ as a sum of dominant coweights, there is a factorization
  $A=A_1\cdots A_m$ such that $\lambda(A_i)=\mu^{(i)}$, and any other such
  factorization has the form \[
  A=(A_1 C_1^{-1})(C_1 A_2
  C_2^{-1})\cdots(C_{m-1} A_m)\]
  with $C_1,\dots,C_{m-1}\in \GL_n(\C[[z]])$.
\end{prop}

\begin{proof}
  We first observe that the fractional ideal generated by the entries of
  $A$ is generated by $z^{\lambda_n(A)}$; this is true for the standard
  double coset representative $A=z^\lambda$, and the fractional ideal is
  clearly constant over the double coset.  The inequality
  \[
  \lambda_n(A_1\cdots A_m)\ge \lambda_n(A_1)+\cdots+\lambda_n(A_m)
  \]
  follows immediately.  More generally, the sum
  $
  \sum_{n-k<i\le n} \lambda_i(A)
  $
  has a similar interpretation in terms of the fractional ideal generated
  by $k\times k$ minors of $A$, or equivalently generated by the entries of
  $\wedge^k A$, and thus all of the inequalities comprising the dominance
  ordering follow.  Moreover, when $k=n$, the matrices are scalars and thus
  we have an equality as required.
  
  For the claim about factorizations saturating the bound, we may WLOG
  suppose that each $\mu^{(i)}$ is a partition with at most $n-1$ parts, as
  otherwise we may simply move overall powers of $z$ between the various
  factors and $A$.  Moreover, we may as well take $A$ to be the standard
  double coset representative $z^\lambda$.  Existence of the factorization
  is then trivial: $z^\lambda = z^{\mu^{(1)}}\cdots z^{\mu^{(m)}}$.

  For uniqueness, we note that if we replace a given $A_i$ by a further
  such factorization, then uniqueness for the finer factorization implies
  uniqueness for the original factorization.  We thus see that it suffices
  to prove uniqueness when each $\mu^{(i)}$ is a partition of the form
  $1^r$; that is, when each part of each $\mu^{(i)}$ is either 1 or 0.
  Uniqueness in that case then follows by an easy induction if we can prove
  uniqueness in the case $m=2$, $\mu^{(1)}=1^r$, $\mu^{(2)}$ general.

  Inverting the matrices and using the interpretation of local conditions
  as double cosets turns this into the following: We need to show that if
  $z \C[[z]]^n\subset \sigma\subset \C[[z]]^n$ is such that
  $z^{-\lambda}\sigma$ has coweight $(\lambda-1^r)^{-1}$, then $\sigma$ is
  the condition that the first $r$ coordinates vanish.  But this is easy:
  if $v_1(0)\ne 0$, then $(z^{-\lambda}v)_1$ has a pole of order
  $\lambda_1$, which is more than allowed, and thus $\sigma$ must be
  contained in the local condition on which the first coordinate vanishes
  at $0$, and thus has a basis containing $z e_1$.  If $r>1$, then applying
  $z^{-\lambda}$ to a pair $(z e_1,v)$ with $v_2(0)\ne 0$ and taking the
  determinant of the first two columns gives a pole of order
  $\lambda_1+\lambda_2-1$, which is still too large.  Continuing this
  procedure gives an inductive proof that the first $r$ coordinates of any
  element of $\sigma$ must indeed vanish at 0.  Comparing determinants
  forbids there from being any further vanishing conditions, and thus we
  have identified $\sigma$ as required.
\end{proof}

\begin{rem}
  Equality holds generically, in the sense that for any matrices $A$, $B$,
  the locus of elements $g\in \GL_n(\C[[z]])$ such that $\lambda(A g
  B)<\lambda(A)+\lambda(B)$ is cut out by a system of polynomial equations
  in the Taylor series coefficients of the entries of $g$.
\end{rem}

We will also need the following variant of the unique factorization property.
Let $\sim$ denote the equivalence relation on coweights induced by the Weyl
group; equivalently, $\alpha\sim \beta$ iff $z^\alpha$ and $z^\beta$
generate the same double coset.

\begin{cor}\label{cor:comparing_factorizations_i}
  Suppose the element $A\in \GL_n(\C((z)))$ is equipped with two
  factorizations $A=BC=DE$ such that
  $\lambda(A)=\lambda(B)+\lambda(C)=\lambda(D)+\lambda(E)$.  Then
  $\lambda(B^{-1}D)\sim\lambda(D)-\lambda(B)$ and
  $\lambda(E C^{-1})\sim\lambda(E)-\lambda(C)$.
\end{cor}

\begin{proof}
  Left- or right-multiplying by an element of $\GL_n(\C[[z]])$ has no
  effect on the various coweights, and thus we may as well assume that
  $A=z^{\lambda(A)}$ is a standard coset representative.  The unique
  factorization result then implies that there are matrices $F$, $G\in
  \GL_n(\C[[z]])$ such that
  \[
  B = z^{\lambda(B)} F,\quad
  C = F^{-1} z^{\lambda(C)},\quad
  D = z^{\lambda(D)} G,\quad
  E = G^{-1} z^{\lambda(E)}
  \]
  It follows that $B^{-1}D = F^{-1} z^{\lambda(D)-\lambda(B)} G$.  The
  claim for $EC^{-1}$ is analogous.
\end{proof}

\begin{rem}
  This includes unique factorization as a special case, since
  $\lambda(B^{-1}D)=0$ iff $B^{-1}D$ is invertibly holomorphic.
\end{rem}

Applying the triangle inequality and unique factorization results to local
conditions gives the following.  In each case, we simply interpret one or
more of the matrices in the previous results as the inverse of a coset
representative of a local condition and use the fact that
$\lambda(A^{-1})=\lambda(A)^{-1}$.

\begin{cor}\label{cor:affine_triangle}
  Let $A\in \GL_n(\sK_x)$, $\sigma\in \GL_n(\sK_x)/\GL_n(\sO_x)$.  Then
  $\lambda(A)\le \lambda(A\sigma)+\lambda(\sigma)^{-1}$.
  Conversely, for any $A\in \GL_n(\sK_x)$ and dominant coweights $\mu$, $\nu$
  such that $\lambda(A)=\mu+\nu^{-1}$, there is a unique local condition
  $\sigma$ such that $\lambda(\sigma)\le \nu$, $\lambda(A\sigma)\le \mu$,
  and equality holds for this local condition.
\end{cor}

For the next version, we use the fact that $-^{-1}$ is linear, so makes
sense for arbitrary coweights, and $\alpha^{-1}\sim -\alpha$.

\begin{cor}\label{cor:comparing_factorizations_ii}
  Let $B,C\in \GL_n(\sK_x)$, $\sigma\in \GL_n(\sK_x)/\GL_n(\sO_x)$ be such
  that
  $\lambda(BC)=\lambda(B)+\lambda(C)=\lambda(BC\sigma)+\lambda(\sigma)^{-1}$.
  Then $\lambda(C\sigma)\sim \lambda(\sigma)-\lambda(C)^{-1}$
\end{cor}

There is also a version of this statement with two local conditions.

\begin{cor}\label{cor:comparing_factorizations_iii}
  Suppose $\sigma,\tau\in \GL_n(\sK_x)/\GL_n(\sO_x)$ and $A\in
  \GL_n(\sK_x)$ are such that
  $\lambda(A)=\lambda(A\sigma)+\lambda(\sigma)^{-1}=\lambda(A\tau)+\lambda(\tau)^{-1}$.
  Then $\lambda(\tau^{-1}\sigma)\sim \lambda(\sigma)-\lambda(\tau)$.
\end{cor}

\medskip

More generally, given $M:W\ratto V$, let $\lambda(x;M)$ denote the coweight
of the corresponding local condition at $x$; in other words, if
$\psi_V:V_x\cong \sO_x^n$, $\psi_W:W_x\cong \sO_x^n$ are isomorphisms,
then $\lambda_x(M):=\lambda(\psi_V M \psi_W^{-1})$, where we note that the
double coset in $\GL_n(\sK_x)$ containing $\psi_V M\psi_W^{-1}$ is
independent of the choice of $\psi_V$, $\psi_W$.

\begin{defn}
  Let $V$ be a vector bundle on $\C^*$, and let $A:V\ratto q^*V$ be a
  meromorphic $q$-connection on $V$.  Then $A$ has {\em stable
    singularities} if for any nonnegative integer $m$ and any point $x\in
  \C^*$, one has
  \[
  \lambda(x;A(q^{m-1}z)A(q^{m-2}z)\cdots A(z))
  =
  \lambda(x;A(q^{m-1}z))+\lambda(x;A(q^{m-2}z)+\cdots+\lambda(x;A(z)).
  \]
\end{defn}

\begin{rem}
Note that if $A$ has at most one singular point in each orbit, then it
automatically has stable singularities, as only one factor will lie in a
nontrivial double coset.  On the other hand, if
\[
  \lambda(x;A(q^{m-1}z)A(q^{m-2}z)\cdots A(z))
  <
  \lambda(x;A(q^{m-1}z))+\lambda(x;A(q^{m-2}z)+\cdots+\lambda(x;A(z)),
\]
for some $x$, then one can show that there is a gauge transform of $A$ by a
rational matrix that preserves $\lambda(x;A(q^{m-1}z)\cdots A(z))$ but
restores the equality.  Thus the $q$-connections with unstable
singularities are, roughly speaking, the ones for which there are simple
gauge transformations making the singularities simpler.  (More precisely,
one should only consider the {\em projective} complexity of the
singularities, as tensoring with a rank 1 equation has no effect on the
stability of the singularities.)  Note that in the case of an elliptic
equation with $\langle p,q\rangle$ a group of rank 2, one can always gauge
by a map of vector bundles on $\C^*/\langle p\rangle$ so that there is at
most one singularity in each orbit of $\langle q\rangle$, and thus under
this very mild condition, every elliptic $q$-connection is gauge equivalent
to one with stable singularities.
\end{rem}

\begin{rem}
  If $A$ and $B$ are two $q$-connections with stable singularities, then so
  is $A\otimes B$, since the coweight of a tensor product of matrices is
  linear in the two coweights.  It also follows using the triangle
  inequality that a direct summand of a $q$-connection with stable
  singularities has stable singularities, and thus in particular Schur
  functors preserve stable singularities, as does duality.
\end{rem}

If $A$ is a $q$-connection with stable singularities, then we can construct
systems of local conditions on $A$ in the following way.  Let
$\lambda^\infty$, $\lambda^0$ be maps from $\C^*$ to the set of dominant
coweights of $\GL(V)$ such that $\lambda^\infty$ is trivial outside a
discrete subset of $\hat{\C}\setminus 0$, $\lambda^0$ is trivial outside a
discrete subset of $\C$, and $\lambda(x;A) =
\lambda^\infty(x)+\lambda^0(x)$.  Note that if $A$ has simple
singularities, this essentially consists of assigning each singularity to
one of $\lambda^\infty$ or $\lambda^0$ (with an additional freedom of
moving a multiple of $1^n$ between the two; this will simply multiply the
fundamental matrix by an overall scalar function), corresponding to our
earlier distinction between ``holomorphic near $\infty$'' and ``holomorphic
near $0$''.  (In other words, $\lambda^\infty$ corresponds to the ``small''
singularities, while $\lambda^0$ corresponds to the ``large''
singularities.)

\begin{prop}
  Let $V$ be a vector bundle on $\C^*$, let $A$ be a $q$-connection on $V$
  with stable singularities, and $\lambda^\infty$, $\lambda^0$ be functions
  as described.  Then there is a unique system $\sigma$ of local conditions
  on $(V,A)$ such that for all $x\in \C^*$,
  \[
  \lambda(\sigma_x)
  \le
  \sum_{k<0} \lambda^\infty(q^k x)
  +
  \sum_{0\le k} \lambda^0(q^k x)^{-1}.
  \]
  Moreover, this system of local conditions satisfies
  \[
  \lambda(\sigma_x)
  \sim
  \sum_{k<0} \lambda^\infty(q^k x)
  -
  \sum_{0\le k} \lambda^0(q^k x).
  \]
  for all $x$.
\end{prop}

\begin{proof}
  Since $A$ has stable singularities, then for $j\gg 0$, we have
  \[
  \lambda(q^{-j}x;A(q^{2j-1}z)\cdots A(qz)A(z))
  =
  \sum_{-j\le k<j} \lambda^\infty(q^k x) + \lambda^0(q^k x)
  =
  \sum_{k<j} \lambda^\infty(q^k x)
  +
  \sum_{-j\le k} \lambda^0(q^k x),
  \]
  where we use the fact that $\lambda^\infty(q^{-m} x)=\lambda^0(q^m x)=0$
  for $m\gg 0$.  It then follows immediately from Corollary
  \ref{cor:affine_triangle} that there is a unique local condition
  $\sigma_{q^{-j}x}$ at $q^{-j}x$ such that
  \begin{align}
  \lambda(\sigma_{q^{-j}x})&\le \sum_{-j\le k} \lambda^0(q^k x)^{-1},\notag\\
  \lambda(A(q^{2j-1}x)\cdots A(qz)A(z)\sigma_{q^{-j}x})
  &\le \sum_{k<j} \lambda^\infty(q^k x),
  \end{align}
  and both inequalities are tight for this local condition.

  This local condition at $q^{-j}x$ determines a full system of local
  conditions on the $q$-orbit of $x$, and it remains only to show that the
  given formula for its coweight holds, and that this forces the
  singularities of the full system of local conditions to be discrete.  For
  $-j\le l\le j$, we may use Corollary
  \ref{cor:comparing_factorizations_ii} to compute
  \begin{align}
  \lambda(\sigma_{q^lx})
  &=
  \lambda(A(q^{l+j-1}z)\cdots A(z)\sigma_{q^{-j}x})\notag\\
  &\sim
  \lambda(q^{-j}x;A(q^{l+j-1}z)\cdots A(z))-\lambda(\sigma_{q^{-j}x})^{-1}\notag\\
  &=
  \sum_{-j\le k<l} \lambda^\infty(q^k x)
  +
  \sum_{-j\le k<l} \lambda^0(q^k x)
  -
  \sum_{-j\le k} \lambda^0(q^k x)
  \notag\\
  &=
  \sum_{k<l} \lambda^\infty(q^k x)
  -
  \sum_{l\le k} \lambda^0(q^k x),
  \end{align}
  agreeing with the desired formula.  To extend beyond the interval, we
  need merely note that had we taken a larger value of $j$ in the
  beginning, the resulting system of local conditions on the orbit
  would satisfy the same inequalities on the original interval, and must
  therefore agree by uniqueness.

  The coweight
  \[
  \lambda(\sigma_x)
  =
  \sum_{k<0} \lambda^\infty(q^k x)
  -
  \sum_{0\le k} \lambda^0(q^k x)
  \]
  is 0 outside a discrete set, and thus the resulting system $\sigma$ is
  regular outside a discrete set.  Indeed, on any compact subset of $\C^*$,
  the two infinite sums are {\em uniformly} finite, and thus their sum has
  finite support.
\end{proof}

\begin{rem}
  One approach to constructing meromorphic solutions of a first-order
  $q$-difference equation $v(qz)=a(z)v(z)$ involves factoring
  $a(z)=a_\infty(z) a_0(z)$ with both factors meromorphic on $\C^*$ and
  such that $a_\infty$ is holomorphic near $\infty$ and $a_0$ is
  holomorphic near $0$.  One then has the product solution
  \[
  v(z)
  =
  \prod_{j<0} a_\infty(q^j z)
  \prod_{0\le j} a_0(q^j z)^{-1}.
  \]
  The corresponding system of local conditions is precisely of the above
  form, with $\lambda^0=\lambda(a_0)$, $\lambda^\infty=\lambda(a_\infty)$.
\end{rem}

We can also use Corollary \ref{cor:comparing_factorizations_iii} to control
what happens if we change $\lambda^0$ and $\lambda^\infty$.

\begin{prop}
  Let $V$ be a vector bundle on $\C^*$, let $A$ be a $q$-connection on $V$
  with stable singularities, and let $\lambda^\infty$, $\lambda^0$;
  $\hat\lambda^\infty$, $\hat\lambda^0$ be two pairs of functions with
  associated systems of local conditions $\sigma$ and $\hat\sigma$.  Then
  the corresponding morphism $\phi$ of solution bundles on $\C^*/\langle
  q\rangle$ has coweight function
  \[
  \lambda(x;\phi) \sim \sum_{\pi_q(y)=x}
  (\hat\lambda^\infty(y)-\lambda^\infty(y)).
  \]
\end{prop}

\begin{proof}
  For any fixed representative $y$ of
  $\pi_q^{-1}(x)$, Corollary \ref{cor:comparing_factorizations_iii} gives
  \[
  \lambda(x;\phi)
  =
  \sum_{k<0} (\hat\lambda^\infty(q^k y)-\lambda^\infty(q^k y))
  +
  \sum_{0\le k} (\lambda^0(q^k y)-\hat\lambda^0(q^k y)),
  \]
  and the claim follows by observing that
  $\hat\lambda^\infty-\lambda^\infty=\lambda^0-\hat\lambda^0$.
\end{proof}

\begin{rem}
  Note in particular that if
  \[
  \sum_{\pi_q(y)=x} (\hat\lambda^\infty(y)-\lambda^\infty(y))=0
  \]
  for all $x$, then the two systems of local conditions will in fact be
  equal.  In particular, the system of local conditions does not actually
  determine $\lambda^0$ and $\lambda^\infty$ in general.
\end{rem}

For the elliptic case, the only thing we need to do is impose a condition
analogous to the condition that $px$ is small whenever $x$ is small.  This
translates immediately to the condition that
\[
\mu(x):=\lambda^\infty(px)-\lambda^\infty(x)
       =\lambda^0(x)-\lambda^0(px)
\]
should be a dominant coweight for every $x$.  The expression in terms of
$\lambda^\infty$ implies that $\mu(x)=0$ on a punctured neighborhood of
$\infty$, while the expression in terms of $\lambda^0$ implies that
$\mu(x)=0$ on a punctured neighborhood of $0$.  Since both descriptions
imply discreteness in $\C^*$, we conclude that $\mu(x)=0$ on the complement
of a {\em finite} subset of $\C^*$.  Moreover, one has the expressions
\begin{align}
\lambda^0(x) &= \sum_{k\ge 0} \mu(p^k x)\notag\\
\lambda^\infty(x) &= \sum_{k<0} \mu(p^k x),\notag\\
\lambda(x;A) &= \sum_{k\in \Z} \mu(p^k x),
\end{align}
so that the function $\mu$ fully encodes the choices being made.  More
precisely, any valid $\mu$ can be obtained from $\lambda(A)$ as follows.
For each $x\in \C^*/\langle p\rangle$, we have a natural decomposition
\[
\lambda(x;A) = \lambda(x;A)_n^n
+
\sum_{1\le i\le \lambda(x;A)_1-\lambda(x;A)_n} 1^{r_i}
\]
for a nondecreasing sequence $1\le r_1\le r_2\le\cdots<n$.  There is a
similar decomposition of $\mu(y)$ for each $y$ lying over $x$, and we find
that the overall set of $r_i$ is the same.  We thus see that we can
reconstruct $\mu$ from the function $\mu(y)_n$ and an assignment of orbit
representative to each $r_i$.  Note that the choice of function $\mu(y)_n$
has little effect on the overall picture: changing it will simply multiply
the fundamental matrix by a $q$-theta function.

\begin{thm}
  Let $p$, $q\in \C^*$ be such that $|p|,|q|<1$, let $V$ be a vector bundle
  on $\C^*/\langle p\rangle$, let $A$ be an elliptic (or twisted elliptic)
  $q$-connection on $\pi_p^*V$ with stable singularities, and let $\mu$ be
  a finitely supported map from $\C^*$ to the set of dominant coweights such
  that $\lambda(x;A)=\sum_{k\in \Z} \mu(p^k x)$ for all $x$.
  Then there is a unique system of local conditions such that
  \[
  \lambda(\sigma_x)
  \le
  \sum_{k,l<0} \mu(p^k q^l x)
  +
  \sum_{0\le k,l} \mu(p^k q^l x)^{-1},
  \]
  and this system of local conditions satisfies
  \[
  \lambda(\sigma_x)
  \sim
  \sum_{k,l<0} \mu(p^k q^l x)
  -
  \sum_{0\le k,l} \mu(p^k q^l x).
  \]
  Moreover, the $p$-connection corresponding to $(V,A,\sigma)$ under
  $\Sol_{p,q}$ also has stable singularities, and the associated system of
  local conditions on the $p$-connection agrees with that corresponding to
  the function $\mu(x)^{-1}$.
\end{thm}

\begin{proof}
  Indeed, the $p$-connection matrix and its iterates can all be viewed as
  morphisms between solution bundles in which the system of local
  conditions has been shifted by the appropriate powers of $p$.  The
  constraint on $\lambda^0$, $\lambda^\infty$ allows one to express the
  corresponding coweights as sums of shifts of $\mu$, and the result follows
  immediately.
\end{proof}

\begin{rem}
  The structure of the bound on the system of local conditions is of course
  quite reminiscent of the structure of the singularities of the elliptic
  Gamma function.  This is naturally not a coincidence, since the equation
  of the elliptic Gamma function has stable singularities.  In particular,
  the function $\Gampq(z/x)$ is the solution of the equation $v(qz) =
  \theta_p(z/x) v(z)$ corresponding to taking $\mu(x)=1$ with all other
  values $0$.  More generally, tensoring with the corresponding object of
  $\EllDiff$ preserves the condition of having stable singularities, and
  simply adds $1^n$ to $\mu(x)$.  Similarly, the fact that $\Gampq(p^a q^b
  x)\Gampq(x)/\Gampq(p^a x)\Gampq(q^b x)$ is invertibly holomorphic on
  $\C^*$ gives rise to modifications we can perform on $\mu$ without
  changing the system of local conditions.  If $\mu(y)-1^r$ and $\mu(p^a
  q^b y)-1^r$ are both dominant, then we can subtract $1^r$ from
  $\mu(y)$ and $\mu(p^a q^b y)$ and add it to $\mu(p^a y)$ and $\mu(q^b
  y)$, and the result will produce the same system of local conditions.
\end{rem}

Let $y\in \C^*/\langle p\rangle$, and suppose
$\lambda(y;A)_r>\lambda(y;A)_{r+1}$.  (Here we slightly abuse notation,
since $\lambda(x;A)$ depends only on the $p$-orbit of $x$.)  Then unique
factorization tells us that there is a unique local condition $\tau$ at $y$
of coweight $0^{n-r},-1^r$ such that $\lambda(A\tau)=\lambda(y;A)-1^r$, and
this produces a new vector bundle $V'$ (containing $V$) to which we may
transport $A$.  (That is, we extend $\tau$ to a system of local conditions
by taking the trivial local condition away from $y$.)  The new
$p$-connection $A'$ on $V'$ continues to have stable singularities, and one
can similarly use unique factorization to recover $V$ from $V'$ and $A'$.
We call this operation and its inverse a ``canonical isomonodromy
transformation''.

To justify the word ``isomonodromy'', choose $\mu$ as above, and let
$\sigma$ be the resulting system of local conditions on $V$.  We can
compose the resulting map $W\ratto \pi_p^*V$ with the inclusion map $V\to
V'$ to obtain a map $W\ratto \pi_p^*V'$ and thus an induced system of local
conditions $\sigma'$ on $\pi_p^*V'$.  Since $V'$ was obtained from $V$ and
$A$ via unique factorization, it is easy to compute the coweights of
$\sigma'$, and we find
\[
  \lambda(\sigma'_x)
  \sim
  \delta_{\pi_p(x),y} (1^r)
  +
  \sum_{k,l<0} \mu(p^k q^l x)
  -
  \sum_{0\le k,l} \mu(p^k q^l x).
\]
(Indeed, for any $x$, $\sigma'_{q^{-j}x}=\sigma_{q^{-j}x}$ for $j\gg 0$, so
we can just replace the iterate of $A$ by the corresponding iterate of $A'$
to compute $\lambda(\sigma'_x)$.)  This coweight function is of the correct
form to come from a function $\mu'$.  Indeed, the fact that
$\lambda(y;A)-1^r$ is dominant implies that $\mu(y')-1^r$ is dominant for
some $y'$ in the $p$-orbit of $y$.  Subtracting $1^r$ from $\mu(y')$ and
adding it to $\mu(y'/q)$ gives a function $\mu'$ inducing the desired
system of local conditions.  (Of course, $y'$ need not be unique, but this
just corresponds to the already discussed indeterminacy in $\mu$.)

We can rephrase the above discussion as follows.  Given a triple
$(V,A,\mu)$, a point $y\in \C^*$, and an integer $1\le r\le n$ such that
$\mu(y)-1^r$ is dominant, then subtracting $1^r$ from $\mu(y)$, adding it
to $\mu(qy)$ and applying the associated canonical transformation to
$(V,A)$ gives a new triple $(V',A',\mu')$, and the two triples give rise to
isomorphic $q$-elliptic $p$-connections; i.e., they have the same
``monodromy''.  Note that by symmetry, changing $\mu$ with $(V,A)$ fixed
has the effect of applying a sequence of canonical isomonodromy
transformations to the monodromy of $(V,A,\mu)$.

\medskip

For our purposes, of course, we are primarily interested in the case of
{\em symmetric} equations.  Of course, we can certainly feel free to impose
the condition $A(1/qz)=A(z)^{-1}$ (relative to a choice of $(z\mapsto
1/z)$-equivariant structure on $V$), but this will in general only ensure
that the solution sheaf is {\em meromorphically} equivariant under the
action of $z\mapsto 1/z$.  Naturally, imposing the additional condition
that $\sigma=(z\mapsto 1/z)^*\sigma$ makes this equivariance holomorphic,
but this is still not {\em quite} the right thing.  The issue is that we
are interested in {\em symmetric} solutions, and though an invariant system
of local conditions ensures that the space of all holomorphic solutions (on
an invariant open set) is preserved by the symmetry, that space need not be
spanned by {\em symmetric} holomorphic solutions.  Luckily, this is an
entirely local question, so it is not too difficult to determine when one
needs to replace the local condition at a point $x=\pm q^{k/2}$ by a
slightly smaller module.  As long as the system of local conditions behaves
correctly at such points, the equivariant bundle of solutions will indeed
be the pullback of a vector bundle on the quotient $\P^1_q:=(\C^*/\langle
q\rangle)/\langle z\mapsto 1/z\rangle$.

Now, suppose $A$ has stable singularities, and consider the system of local
conditions associated to a decomposition
$\lambda(x;A)=\lambda^0(x)+\lambda^\infty(x)$.  Since this system of local
conditions is uniquely determined by the bounds on the coweights, it will
suffice to take $\lambda^0(x)=\lambda^\infty(1/qx)^{-1}$ to ensure that
those bounds (and thus the system of local conditions) are invariant under
$z\mapsto 1/z$.  (The precise condition to have invariance is unclear,
since the coweight function of the system of local conditions does not
determine $\lambda^0$, $\lambda^\infty$.)  The descent condition is much
harder to arrange; in fact, it turns out that the only way to ensure that
the natural system of local conditions satisfies the descent conditions is
to assume that $A$ is actually regular at the ramification points: not only
should it be holomorphic and invertible at every point of the form $\pm
q^{k/2}$, but in fact one must take $A(\pm q^{-1/2})=1$.  Of course, what
this really indicates is that just as one needs to change the definition of
a singular point to take into account the symmetry, one will also need to
change the definition of {\em stable} singularities, presumably based on a
symmetric analogue of the triangle inequality and unique factorization.

Similarly, if $V$ is a vector bundle on $\P^1_p$ and $A$ is a symmetric
{\em elliptic} $q$-connection (i.e., with $C z^k=1$) with stable
singularities (in the nonsymmetric sense) on the pullback of $V$, then the
function $\mu$ should be chosen such that $\mu(x)=\mu(1/pqx)^{-1}$, and
should moreover vanish on points of the form $\pm p^{i/2}q^{j/2}$.  One
should moreover have $A(\pm p^{j/2}q^{-1/2})=1$ for any $j\in \Z$ to
guarantee that there are no local obstructions to having symmetric
solutions at those points.  The image under $\Sol_{p,q}$ will of course
satisfy the same symmetry, and because $V$ was a pullback will also satisfy
the condition of regularity (in the symmetric sense) on the ramification
points.

To allow more general twisting factors, the simplest thing to do is to
tensor the equation with the equation satisfied by a symmetric product of
elliptic Gamma functions, $\prod_{1\le i\le k} \Gampq(b_iz^{\pm 1})$.  As
long as none of the $a_i$ are of the form $\pm p^{j/2}q^{k/2}$, this has no
effect on the regularity, and changes $\mu$ in a predictable way
(preserving $\mu(x)=\mu(1/pqx)^{-1}$).  This operation takes
$\EllDiff_{p,q,C}$ to $\EllDiff_{p,q}$ with $C=\prod_i b_i^2/(pq)^k$,
preserving symmetry, and the conditions $A(\pm p^{-1/2}q^{-1/2})=1$ pull
back to the the condition that the values be $(\mp
p^{1/2}q^{1/2})^k/\prod_i b_i$.  (Note that the values are the same or
different depending on the parity of $k$; this of course corresponds to the
familiar distinction between even and odd Hirzebruch surfaces.)  It is easy
to see that this operation simply multiplies the fundamental matrix by the
same product of elliptic Gamma functions, and thus induces equivalences
between each of the four subcategories of symmetric equations (regular on
the ramification locus) in $\EllDiff_{p,q,C}$ and a corresponding
subcategory of $\EllDiff_{q,p,1/C}$.

To preserve the symmetry of $\mu$, we cannot simply apply a single
canonical isomonodromy transformation.  If $\mu(y)-1^r$ is dominant, then
there is a combination of such transformations producing
\[
\mu'(x)
=
\mu(x)
-\delta_{x,y}(1^r)
+\delta_{x,y/q}(1^r)
-\delta_{x,1/pqy}(0^{n-r},-1^r)
+\delta_{x,1/py}(0^{n-r},-1^r),
\]
a composition of a canonical isomonodromy transformation (of rank $r$), an
inverse canonical isomonodromy transformation (of rank $n-r$), and
tensoring with a rank 1 equation.  The relevant gauge transformations again
come from unique factorization, with a factor on both the right and the
left, and it is straightforward to verify that the factors are swapped
under the $A(z)\mapsto A(1/qz)^{-1}$ symmetry.  As a result, we find that
the transformation involves gauging by a local condition at the image of
$y$ in $\P^1_p$, and thus preserves symmetry as required.

There is an additional operation one can perform in the symmetric case.  We
consider the untwisted version; for more general twists, we can again
simply tensor with appropriate Gamma factors.  We took the symmetry to be
$z\mapsto 1/z$ above, but of course there is no reason (other than
notational convenience) we could not have taken $z\mapsto pq\eta/z$ (to
match up with our normal convention).  But then $A$ (since it is untwisted)
also satisfies the symmetry (and regularity) condition for $p\eta$,
so we could apply $\Sol_{p,q}$ to obtain a $p^2q\eta$-symmetric
$q$-elliptic $p$-difference equation.  The corresponding operation on the
monodromy is then an isomonodromy-type transformation, which we would like
to understand.

Of course, for this to be well-defined, we need to specify how to change
$\mu$.  The original $\mu$ satisfies the symmetry condition
$\mu(x)=\mu(\eta/x)^{-1}$, and we need to modify it so that
$\mu'(x)=\mu'(p\eta/x)^{-1}$.  There are, of course, many ways we might do
so, but there is one particularly natural choice.  Define a function
$\nu(x)$ by $\nu(x)_j=\max(\mu(x)_j,0)$, and note that
$\mu(x)=\nu(x)+\nu(\eta/x)^{-1}$.  Then $\mu'(x) =
\nu(px)+\nu(\eta/x)^{-1}$ is again a valid coweight function for $A$, but
now satisfies the shifted symmetry condition.  One finds that the new
system of local conditions satisfies
\[
\lambda(\sigma_x^{-1}\sigma'_x) = \sum_l \nu(q^l x).
\]
This is not only dominant, but nonnegative, so corresponds to an actual
morphism of solution bundles.  Since the bundles have different symmetries,
this corresponds to a Theorem 90-style factorization of the symmetric
$p$-connection.  Moreover, in the case of simple singularities, it is
easily seen to give a minimal such factorization.

\medskip

The above construction can be generalized in a number of ways.  One quite
natural generalization, given the role of the affine Grassmannian, is to
replace $\GL_n$ everywhere with a more general (split) reductive group $G$.
Indeed, there is still a functorial correspondence between meromorphic maps
from a given principal $G$-bundle and a corresponding family of points of
the affine Grassmannian of type $G$.  One can define a meromorphic
$q$-connection in the same way, and the analogous compatibility condition
on the system of local conditions ensures that the new principal $G$-bundle
$W$ is {\em holomorphically} $q$-equivariant, and thus descends to a
principal $G$ bundle on $\C^*/\langle q\rangle$ as desired.  (The actual
interpretation as a sheaf of solutions may require more work!)  The
elliptic case is analogous as well; one can even include twisting factors
living in the center of $G$.  The monoidal structure survives in the form
of an outer product and functoriality in $G$.

Even better, the notion of ``stable singularities'' carries over to
$q$-connections on principal $G$-bundles.  Indeed, the only thing we needed
to make that work was the triangle inequality on the affine Grassmannian
and its associated unique factorization results, and both properties hold
for general $G$; the only change is to replace the coweights of $\GL_n$ by
coweights of $G$.  These results are again well-known in the affine
Grassmannian community, and can be deduced by applying the $\GL_n$ results
to the image under some faithful family of irreps of $G$.  As a result, for
any finite function $\mu$ from $\C^*$ to the set of dominant coweights of
$G$, there is a corresponding category of $p$-elliptic $q$-connections with
stable singularities measured by $\mu$ on principal $G$-bundles, and there
is an analytic functor that swaps $p$ and $q$ and inverts $\mu$.  The
situation with isomonodromy transformations is somewhat more complicated,
but such transformations can always be expressed as a composition of
$\GL_n$-type canonical transformations applied to the image under some
faithful representation.  This also carries over to the symmetric case,
though now one must be particularly careful about singularities at the
ramification points, as the absence of a general analogue of Hilbert's
Theorem 90 makes it particularly difficult to classify such singularities.
(It may be that the theory requires one to assume the existence of a
factorization, or equivalently that the corresponding class in nonabelian
cohomology is trivial.)

For the main functors, one can also replace the action of $\langle
q\rangle$ on $\C^*$ by any proper action of a discrete group on a Riemann
surface, though just as for symmetric equations, there are additional
conditions one must impose on the local conditions at points with
nontrivial stabilizer.  The main potential application of this version
would be to ``hyperbolic'' equations.  The hope there would be to have an
analogous monodromy functor relating $q$-difference equations with rational
coefficients to $\tilde{q}$-difference equations with rational
coefficients, where $\tilde{q}$ is the modular transform of $q$.  This
would be related in turn to a pair of actions of $\Z$ on $\C$.  There
would, however, need to be some additional constraints imposed on the
solutions (e.g., growth conditions), as a direct application of the
analogous results would only give a $\tilde{q}$-difference equation with
meromorphic coefficients.

A further direction in which it would be nice to have a generalization has
to do with the fact that the categories arising above are not actually
abelian.  One consequence is that although we have constructed the elliptic
Riemann-Hilbert correspondence for symmetric equations with fairly general
singularities, the categories on which it is defined are generally not the
most natural ones arising from geometry.  As a result, we can only directly
apply the above results when the geometric categories embed in the analytic
categories, which in turn loses nearly all of the generality we allowed in
the singularities.  The existence of the geometrically defined categories
suggests that it should be possible to enlarge $\EllDiff_{p,q}$ or its
symmetrical version to an abelian category on which the correspondence
remains defined, and in this way avoid the somewhat kludgy extension to
general parameters below.

A partial step in this direction would be to show that $\EllDiff_{p,q}$ (or
the subcategory corresponding to stable singularities) can be given the
structure of an {\em exact} category (see the survey \cite{BuhlerT:2010}),
by taking the admissible exact sequences to consist of those three term
complexes of triples $(V,W,M)$ which are locally split in a suitable sense.
Note that although this would, in fact, immediately produce an embedding in
an abelian category, the standard construction does not preserve the
existence of an exact duality, which the geometric categories have whenever
$\langle p,q\rangle$ has rank 2.  Presumably, though, this would give the
correct {\em derived} category, and one could then hope to be able to give
a suitable self-dual $t$-structure.

Note that the exact structure would be enough to allow one to define $\Ext$
functors.  The self-$\Ext^1$ of a triple $(V,W,M)$ would correspond to
infinitesimal deformations that do not change the coweight function of $M$,
i.e., the tangent space to the equisingular moduli stack.  One could then
hope to define a symplectic structure via the Yoneda pairing and a suitable
trace function on the self-$\Ext^2$, which would agree with the pairing
from noncommutative geometry if one had a suitable comparison result.  Such
symplectic structures should, of course, exist in the $G$-bundle setting as
well.

\subsection{The elliptic Riemann-Hilbert correspondence}

Now, let ${\cal Q}_!(X)\subset \coh X$ denote the subcategory of sheaves
disjoint from the anticanonical sheaf (necessarily pure $1$-dimensional
sheaves when $q$ is non-torsion).  Geometrically, this is a noncommutative
analogue of the category of compactly supported sheaves on the
quasiprojective surface $X\setminus C$, while analytically it is a category
of difference equations with singularities at specified locations.

Let us first restrict our attention to a noncommutative surface
$X_{q,p}:=X_{\eta,x_0,\dots,x_m;q;p}$, where we suppose that no root of
$D_m$ is effective, and for technical reasons that $x_i^2\notin q^\Z$ for
$1\le i\le m$.  More precisely, let $\eta$,$x_0$,\dots,$x_m$,$q\in \C^*$ be
such that $x_i/x_j\notin p^\Z q^\Z$ for $1\le i<j\le m$ and
$\eta/x_ix_j\notin p^\Z q^\Z$ for $1\le i,j\le m$.  (We will similarly
replace $\rho$ by a lift $\rho:\Pic(X)\to \C^*$.)

Now, any sheaf in ${\cal Q}_!(X_{q,p})$ is the minimal lift of its direct
image in $X_0$, and that direct image cannot have any subquotients with
$c_1\in (1+\N)f$, since this would give points of intersection with the
anticanonical curve violating our conditions on $X_{q,p}$.  We thus find
that ${\cal Q}_!(X_{q,p})$ can be identified with a subcategory of the
category of elliptic $q$-difference equations of the form (undoing the
gauge transformation making things elliptic)
\begin{align}
B(p\eta/z)^tv(p\eta/z) &= B(z)^t v(z)\notag\\
v(pq\eta/z)&=v(z)\notag
\end{align}
where $B$ is a holomorphic matrix such that
\[
B_{ij}(pz) = (-pqx_0/z) (q\eta/z^2)^{e_i} B_{ij}(z) (\eta/z^2)^{-d_j}.
\]
This, of course, corresponds to a map $V\to W\otimes {\cal L}$ where $V$,
$W$ are pulled back from vector bundles on the relevant quotient $\P^1$s,
and ${\cal L}$ is the line bundle with multiplier $-pqx_0/z$.

Moreover, the subcategory consists precisely of the equations corresponding
to $B$ such that $\det(B)$ vanishes only at $x_1/pq,\dots,x_m/pq\in
\C^*/\langle p\rangle$ and such that at each point, $\ord_{x_i/pq}\det(B) =
n-\rank(B(x_i/pq))$.  It follows that the corresponding $pq\eta$-symmetric
$q$-connection is regular on the ramification points, and has stable (in
fact simple) singularities.

Applying the elliptic Riemann-Hilbert correspondence gives us a new
$pq\eta$-symmetric equation.  Since this inverts all the singularities, we
compose this with duality (which, by the monoidal property, commutes with
the Riemann-Hilbert correspondence) to fix this, at the cost of making it a
contravariant functor.  This gives us the following result.

\begin{thm}\label{thm:RH}
  For any parameters $\rho$, $p$, $q$ with $|p|$, $|q|<1$, $\rank(\langle
  p,q\rangle)=2$, there is a contravariant equivalence
  $\Sol_{\rho;q;p}:{\cal Q}_!(X_{\rho;q;p})^{\text{\rm op}} \cong {\cal
    Q}_!(X_{\rho;p;q})$ such that $\Sol_{\rho;p;q}\Sol_{\rho;q;p} \cong
  \id$, and acting on numerical invariants so that
  $c_1(\Sol_{\rho;q;p}(M))=c_1(M)$ and $ p^{\chi(\Sol_{\rho;q;p}(M))}
  q^{\chi(M)} \rho(c_1(M))=1$.  In addition, if $D$ is any divisor class,
  then there is a natural isomorphism
  \[
  \Sol_{p^{-(D\cdot \_)}\rho;q;p}(M)
  \cong
  \Sol_{\rho;q;p}(M)(D).
  \]
  Moreover, $\Sol_{\rho;q;p}$ is locally analytic, in the sense that if $M$
  is a flat holomorphic family of objects, then $\Sol_{\rho;q;p}M$ is a
  locally holomorphic family.
\end{thm}

\begin{proof}
  First suppose that $X_{\rho;q;p}$ is a surface for which no root of $D_m$
  is effective and $\rho(f-2e_i)\notin p^\Z q^\Z$ for $1\le i\le m$.  In
  that case one may take $\Sol_{\rho;p;q}$ to be the elliptic
  Riemann-Hilbert correspondence.  Note that a priori $\Sol_{\rho;q;p}(M)$
  is only defined as a $p$-connection on an {\em analytic} vector bundle on
  $\P^1$, but per GAGA this can be identified with a unique algebraic
  vector bundle structure.  Moreover, if $M$ ranges over a holomorphic
  family, then we can cover the base by the open subsets on which the
  solution sheaf is $d$-regular, allowing us to algebraize
  $\Sol_{\rho;q;p}(M)$ on such an open subset in a holomorphic way.
  
  It is straightforward to see that $\tilde{A}$ has only simple zeros at
  the points $x_1/pq$,\dots,$x_m/pq$, and thus indeed corresponds to a
  (unique!) element of ${\cal Q}_!(X_{\rho;p,q})$, giving a functor as
  required.  Moreover, the order of the equation is the same, so that
  $c_1(\Sol(M))\cdot f = c_1(M)\cdot f$, and the multiplicities of the
  singularities are the same, so that $c_1(\Sol(M))\cdot e_i = c_1(M)\cdot
  e_i$.  Since $c_1(\Sol(M))\cdot K = c_1(M)\cdot K=0$, the remaining
  degree of freedom of the Chern class must agree as well.

  The relation between the Euler characteristics then follows from the fact
  that $M$ and $\Sol(M)$ are disjoint from their respective anticanonical
  curves.  (It could also be recovered by applying the elliptic
  Riemann-Hilbert correspondence to $\det(M)$.)

  For twisting, we need only verify the claim on a basis of $\Pic(X)$, say
  $e_1,\dots,e_m,s-e_1-\cdots-e_m,f$.  The first $m$ are just the standard
  symmetric versions of the canonical isomonodromy transformations, and the
  action of $s-e_1-\cdots-e_m$ on parameters gives rise to the analogous
  transformation that shifts the symmetry parameter as well.  Per the
  computations in \cite{rat_Hitchin} (and extended to the noncommutative
  case above) the result in each case agrees with the twist.  It thus
  remains only to show that
  \[
  \Sol_{\eta,x_0/p,x_1,\dots,x_m;q;p}(M)
  \cong
  \Sol_{\eta,x_0,x_1,\dots,x_m;q;p}(M)(f).
  \]
  Since shifting $x_0$ has the effect of multiplying the twisted matrix
  $A(z)$ by $p\eta/z^2$, the matrix $\tilde{A}(z)$ is actually unchanged,
  and simply has a different value of $\mu_q(z)$ (and a corresponding
  change to $C$.)  In other words, the only effect is to twist the
  underlying vector bundle by $\sO_{\P^1}(1)$, again agreeing with the
  action of twisting by $f$.

  It remains to prove the result when the above parameter constraints are
  violated.  We observe that for such a functor to exist for
  $X_{\rho;q;p}$, it suffices for there to be {\em some} rational pencil
  $R$ such that (a) for any root $r$ with $R\cdot r=0$, $\rho(r)\notin p^\Z
  q^\Z$, and (b) for any $-1$-class $e$ with $R\cdot e=0$,
  $\rho(R-2e)\notin p^\Z q^\Z$.  Indeed, we can then apply suitable
  elements of $W(E_{m+1})$ (preserving the set of effective roots) to take
  $R$ to $f$, and observe that the resulting surface admits a well-defined
  Riemann-Hilbert correspondence.

  Next, we note that if one blow up an additional point $x_{m+1}$, then
  most of the time one has an algebraic isomorphism $\alpha_{m+1}^*:{\cal
    Q}_!(X_m)\cong {\cal Q}_!(X_{m+1})$.  Indeed, suppose $x_{m+1}$ is
  chosen to be in the complement of the countably many hypersurfaces
  $\rho(v)\in p^\Z q^\Z$ where $v$ ranges over elements of
  $\Lambda_{E_{m+1}}$ with $v\cdot e_{m+1}\ne 0$.  Then any object $M\in {\cal
    Q}_!(X_{m+1})$ has $c_1(M)\in \Lambda_{E_{m+1}}$ and $\rho(c_1(M))\in
  p^\Z q^\Z$, so that $c_1(M)\cdot e_{m+1}=0$.  But this implies $M\cong
  \alpha_{m+1}^*\alpha_{(m+1)*}M$, giving the desired isomorphism.
  
  Now, starting with our original surface $X_m$, perform $2m+1$ such
  blowups, and consider the divisor class
  \[
  R_m:=2^{m-1}s + 2^m f
  -\sum_{1\le k\le m-1} 2^{m-1-k} e_k
  -\sum_{1\le k\le m} 2^{m-k} (e_{m+k}+e_{2m+k})
  -e_{3m+1}.
  \]
  on $X_{3m+1}$.  If $X_m$ were itself sufficiently general, then $R_m$ would
  be the class of a rational pencil on $X_{3m+1}$.  Indeed, this is
  certainly true for $R_0=s+2f-e_2-e_3-e_4$, while in general if we reflect
  $R_m$ in $f-e_{m+1}-e_{2m+1}$ then $s-e_1$, the result is in the same
  $S_{3m+1}$-orbit as $R_{m-1}$.

  For general $X_m$, the algorithm for testing whether $R_m$ is a rational
  pencil involves repeated reflections in roots having negative
  intersection with $R_m$, but we easily see that $R_m$ has positive
  intersection with every root of $E_{m+1}$.  It follows that if $r$ is any
  root with $r\cdot R_m\ne 0$, then $r\notin \Lambda_{E_{m+1}}$, and thus
  (by the assumption on our blowups) $\rho(r)\notin p^\Z q^\Z$.  In other
  words, the procedure never involves reflecting in an effective root, and
  thus $R_m$ is indeed a pencil on our surface.

  It remains to show that $R_m$ satisfies the requirements to induce a
  Riemann-Hilbert correspondence.  That $\rho(r)\notin p^\Z q^\Z$ for roots
  orthogonal to $R_m$ follows by the same argument, since again this
  implies $r\notin \Lambda_{E_m}$.  Similarly, if $e$ is a $-1$-class
  orthogonal to $R_m$, then $(R_m-2e)\cdot e_{3m+1} = 1 + 2(e\cdot
  e_{3m+1})\ne 0$, so that $\rho(R_m-2e)$ is a function of $x_{3m+1}$.
  Thus at the cost of excluding a further countable set of values of
  $x_{3m+1}$, we may ensure that $\rho(R_m-2e)\notin p^\Z q^\Z$ for any
  such $e$.  It follows that there is indeed a well-defined Riemann-Hilbert
  correspondence on ${\cal Q}_!(X_{3m+1})$ corresponding to $R_m$.
\end{proof}

\begin{rem}
  We will see below that the Riemann-Hilbert functor $\Sol$ commutes with
  the actions of simple reflections whenever both surfaces admit
  Riemann-Hilbert functors.  This will then imply that the construction via
  auxiliary blowups agrees with the straightforward Riemann-Hilbert functor
  when the latter is defined.
\end{rem}

\begin{rem}
  The parameter constraints actually ensure that any connected holomorphic
  family $M$ is bounded, and thus we could omit the word ``locally'' above.
  However, we will not be using this fact below, and there are variants for
  which we cannot guarantee global analyticity.
\end{rem}

\begin{rem}
  The elliptic Riemann-Hilbert functor ``integrates'' twisting by
  line bundles, in that we can express such a twist as a composition of two
  elliptic Riemann-Hilbert functors:
  \[
  \_(D) \cong \Sol_{q^{-(D\cdot \_)}\rho;p;q}\Sol_{\rho;q;p}.
  \]
  This is, of course, directly analogous to the role of the usual
  Riemann-Hilbert correspondence in solving the Painlev\'e VI equation.
\end{rem}  

It turns out that the condition $\rank(\langle p,q\rangle)=2$ is not
actually necessary for the conclusion to hold.  (Note that although we saw
something similar for general elliptic difference equations, the conclusion
does not carry over, as ${\cal Q}_!(X_{\rho;q,p})$ also contains
$0$-dimensional sheaves!)  Indeed, when $\langle p,q\rangle$ has rank 1,
the quasiprojective surfaces $X_{\rho;q;p}[1/T]$ and $X_{\rho;p;q}[1/T]$
are Azumaya algebras over the same commutative quasiprojective surface
(with anticanonical curve $\C^*/\langle p,q\rangle$).  Thus the existence
of a locally analytic equivalence ${\cal Q}_!(X_{\rho;q;p})\cong {\cal
  Q}_!(X_{\rho;p;q})$ would follow (apart from some mild technical issues)
if we could prove that the two Azumaya algebras were Morita equivalent.  In
fact, we have the following.

\begin{prop}\label{prop:analytic_brauer_trivial}
  Let $X$ be a commutative rational surface over $\C$ and let $C$ be an
  anticanonical curve on $X$.  Any holomorphic Azumaya algebra on
  $X\setminus C$ is trivial.
\end{prop}

\begin{proof}
  Since the map $\Br(X\setminus C)\to H^2((X\setminus
  C)^{\text{an}};\sO_X^*)$ is injective, it suffices to show that the
  latter is trivial.  Since $\hat{X}:=X\setminus C$ is noncompact, it
  follows from \cite[Prop.~2.1]{SchroerS:2005} that
  $H^2(\hat{X}^{\text{an}};\sO_{\hat{X}}^*)$ is isomorphic to the purely
  topological group $H^3(\hat{X};\Z)$.  We will show that
  $H^p(\hat{X};\Z)=0$ for $p\ge 3$.

  Suppose $\pi:X\to Y$ is a birational morphism with exceptional curve $e$,
  so that $\pi(C)$ is again anticanonical on $Y$.  Then $\hat{Y}$ is
  naturally identified with an open subset of $\hat{X}$ with complement
  $e\setminus (e\cap C)\cong \A^1(\C)$.  Let $U$ be a tubular neighborhood
  of the complement.  Then $U$ retracts to $\A^1(\C)$, so is contractible,
  while $U\cap \hat{Y}\cong \A^1(\C)\times \C^*$ is homotopy equivalent to
  $S^1$.  Applying Mayer-Vietoris to $\hat{X}=\hat{Y}\cup U$ thus implies
  $H^p(\hat{X};\Z)\cong H^p(\hat{Y};\Z)$ for $p\ge 3$.

  We may thus reduce to the case that $X$ is an odd Hirzebruch surface.  In
  that case, we find that if $X\cong F_d$, then $dC-(d-1)F_d$ is an ample
  divisor with reduced support $C$.  It follows that $\hat{X}$ is affine,
  and thus a Stein manifold.  The claim then follows from
  \cite[Thm.~1]{AndreottiA/FrankelT:1959} and the universal coefficient
  theorem.
\end{proof}

\begin{rem}
  This holds more generally for any Poisson surface which is not
  symplectic.  Again we may reduce to ruled surfaces, and find that
  $\hat{X}$ is affine, a line bundle over a Riemann surface, or a
  $\C^*$-bundle over an elliptic curve.  In the first two cases, we again
  find $H^3(\hat{X};\Z)=0$, while in the third case $H^3(\hat{X};\Z)\cong \Z$.
  Either way, the torsion subgroup of $H^3(\hat{X};\Z)$ is trivial, and the
  claim follows.
\end{rem}

\medskip

If we view $\Sol(M)$ as the elliptic analogue of monodromy, then we see
that we should expect to obtain an isomonodromy transformation from any
isomorphism between surfaces $X_{\rho;p,q}$.  We have already considered
twisting by line bundles, and will consider the action of $W(E_{m+1})$ more
carefully below.  There are some other isomorphisms, however: it is also
true that the surface only depends on the image of $p$ in $\C^*/\langle
q\rangle$, and more subtly that the construction of the surface respects
isomorphisms between curves and thus respects {\em modular}
transformations.

To understand these latter transformations, replace $\rho$ with a map
$\rho:\Pic(X)\to \C$ and $p$, $q$ by a triple $\omega_1,\omega_2,\omega_3$
of periods.  If $\omega_2/\omega_1$ is in the upper half-plane, then we may
define a surface
\[
X_{\rho;\omega_1,\omega_2,\omega_3}
:=
X_{\exp(2\pi\sqrt{-1}\rho/\omega_1);
   \exp(2\pi\sqrt{-1}\omega_3/\omega_1);
   \exp(2\pi\sqrt{-1}\omega_2/\omega_1)}.
\]
To extend this to the case that $\omega_2/\omega_1$ is in the lower
half-plane, we take the convention that
\[
X_{\rho;q;p}:=X_{1/\rho;1/q;1/p}^{\text{op}}
\]
if $|p|>1$.

\begin{cor}\label{cor:gl3Z}
  Let $\omega:\Z^3\to \C$ be a linear map.  For any $g\in \GL_3(\Z)$ such
  that both $\langle \omega_1,\omega_2\rangle$ and $\langle
  (g\omega)_1,(g\omega)_2\rangle$ are lattices, there is a locally analytic
  equivalence
  \[
  {\cal Q}_!(X_{\rho;\omega})
  \cong
  {\cal Q}_!(X_{\det(g)\rho;g\omega}).
  \]      
\end{cor}

\begin{proof}
  Note first that
  $
  X_{\rho;q;p}^{\text{op}}\cong X_{\rho;1/q;p}\cong X_{1/\rho;q;p}$,
  from which one concludes that
  $
  X_{-\rho;\omega}\cong X_{\rho;\omega}^{\text{op}}$.

  If $\langle \omega_1,\omega_2,\omega_3\rangle$ is a lattice, then
  $X_{\rho;\omega}[1/T]$ is an Azumaya algebra, so is analytically trivial
  by Proposition \ref{prop:analytic_brauer_trivial} above, and is thus
  Morita equivalent to its center $Z$.  In other words,
  $X_{\rho;\omega}[1/T]$ is the endomorphism ring of an analytic vector
  bundle $V$ on $Z$.  Now, let $Y$ be any (possibly nonreduced) projective
  subvariety of $Z\setminus C$.  Since $Y$ is projective, $V|_Y$ is
  isomorphic to a unique {\em algebraic} vector bundle on $Y$, and thus
  induces a Morita equivalence between coherent sheaves on
  $X_{\rho;\omega}$ supported on $Y$ and coherent sheaves on $Z$ supported
  on $Y$.  Taking the union over all $Y$ gives the desired equivalence
  ${\cal Q}_!(X_{\rho;\omega})\cong {\cal Q}_!(Z)$.  Now, $Z$ is the
  commutative rational surface with $C=\C/\langle
  \omega_1,\omega_2,\omega_3\rangle$ and induced parameter map $\rho$.
  Since this is $\GL_3$-invariant (and ${\cal Q}_!(X_{\rho;1;p})$ is
  self-dual), the existence of a family of algebraic equivalences follows
  immediately.  That the family is locally analytic follows once we observe
  that $V|_Y$ is locally algebraizable in flat families.

  We now consider the case that $\langle \omega_1,\omega_2,\omega_3\rangle$
  has rank 3.  Since $X$ is functorial in the curve, the result clearly
  holds for the subgroup $\ASL_2(\Z)$ preserving $\langle
  \omega_1,\omega_2\rangle$.  Moreover, this extends to $\AGL_2(\Z)$ since
  our conventions ensure
  $
  X_{\rho;\omega_1,\omega_2,\omega_3}
  \cong
  X_{-\rho;-\omega_1,\omega_2,\omega_3}$.
  Moreover, the elliptic Riemann-Hilbert correspondence gives
  $
  {\cal Q}_!(X_{\rho;\omega_1,\omega_2,\omega_3})
  \cong
  {\cal Q}_!(X_{-\rho;\omega_1,\omega_3,\omega_2})
  $
  whenever both $\langle \omega_1,\omega_2\rangle$ and $\langle
  \omega_1,\omega_3\rangle$ are lattices.  (Note that there are four cases
  to consider, depending on which half-space contains $\omega_2/\omega_1$
  and $\omega_3/\omega_1$.)
  The above elements generate $\GL_3(\Z)$, so the result follows as long as
  the image of every rank 2 subgroup is a lattice, or equivalently when
  $\omega$ has dense image.

  In the remaining case, the topological closure of the image of $\omega$
  is either a line (in which case {\em no} rank 2 subgroup has image a
  lattice) or a union of cosets of a line.  Using the $\C^*$ symmetry, we
  may assume that the closure consists of those elements of $\C$ with
  integer imaginary part.  In particular, we have a vector
  $\alpha:=\Im(\omega)$, with the same action of $\GL_3(\Z)$.  The
  condition that $\omega_1$, $\omega_2$ generate a lattice is then that
  $(\alpha_1,\alpha_2)\ne 0$.  Using the $\ASL_2(\Z)$ symmetry, we have a
  locally analytic equivalence with a surface for which $\alpha_2=0$, in which
  case we may apply the Riemann-Hilbert functor to obtain a surface with
  $\alpha_3=0$.  A further application of $\ASL_2(\Z)$ then gives a surface
  with $\alpha_2=\alpha_3=0$, and thus (up to $-1\in \SL_2(\Z)$)
  $\alpha_1=-1$.

  It thus remains only to show that the claim holds when both surfaces have
  $\alpha=(-1,0,0)$.  The stabilizer of this vector is a different copy of
  $\AGL_2(\Z)$, so we need only generate this group.  The original
  $\ASL_2(\Z)$ includes operations $\omega_1\mapsto \omega_1+\omega_2$ and
  $\omega_3\mapsto \omega_3+\omega_2$, and it is easy to see that these
  operations together with the Riemann-Hilbert correspondence generate
  $\AGL_2(\Z)$ as required.
\end{proof}

\begin{rem}
  It is unclear whether the relations of $\GL_3$ (relative to the above
  implicit presentation) act trivially, or as nontrivial (holomorphic!)
  autoequivalences of ${\cal Q}_!$.
\end{rem}

\begin{rem}
  There is, of course, a nonsymmetric version of this action coming from
  the nonsymmetric elliptic Riemann-Hilbert correspondence.  Now, if $A$
  and $\tilde{A}$ are related by the (contravariant) elliptic
  Riemann-Hilbert correspondence, then the fundamental matrix (expressed in
  terms of multipliers for the given vector bundles) exhibits a cohomology
  between the (nonabelian) $1$-cocycles
  \[
  v(qz) = A(z)v(z)\quad
  v(pz) = \mu_p(z)v(z)
  \]
  and
  \[
  v(qz) = \mu_q(z)^{-t}v(z)\quad
  v(pz) = \tilde{A}(z)^{-t}v(z).
  \]
  Replacing $z$ by $\exp(2\pi\sqrt{-1}z/\omega_1)$ gives a pair of
  cohomologous $1$-cocycles for $\langle
  \omega_1,\omega_2,\omega_3\rangle$.  We thus find more generally that any
  two equations related by the general action of (the appropriate extension
  of) $\SL_3(\Z)$ correspond to cohomologous $1$-cocycles (and similarly in
  the symmetric case, now on the group $\langle
  \omega_1,\omega_2,\omega_3\rangle \rtimes \Z/2\Z$.)  Given any relation,
  the corresponding product of fundamental matrices induces an automorphism
  of the appropriate equation, and will thus in many cases be forced to be
  scalar.  If this could be made precise, this would give rise to a
  nonabelian analogue of the results of \cite{FelderG/VarchenkoA:2000} on
  the elliptic Gamma function, replacing the latter by the appropriate
  fundamental matrices.
\end{rem}

\begin{rem}
  Of course, this extends to an action of $\Pic(X)^3\rtimes \GL_3(\Z)$ (modulo
  autoequivalences): $\Pic(X)^2$ acts by changing $\rho$ to a different
  representative modulo $\langle \omega_1,\omega_2\rangle$, while the third
  $\Pic(X)$ acts by twisting.
\end{rem}

It is tempting to combine the above result with our derived equivalences
for $m=8$ to obtain an action of $\SL_4(\Z)$.  There are some technical
issues, however.  The first is that in order for ${\cal
  Q}_!(X_{\rho;\omega})$ to be nonempty, there needs to be a nef divisor
class with $\rho(D)\in \im\omega$; in the $m=8$ case, this imposes an
integer linear dependence between the four parameters, and thus some points
in the $\SL_4(\Z)$-orbit do not correspond to surfaces.  More seriously,
there is a potential issue in showing that $D^b{\cal Q}_!(X_{\rho;\omega})$
is preserved by the derived equivalences.  The difficulty here is that the
natural subcategory of $D^b\coh X$ is the kernel of $|^\dL_C$, or
equivalently the category of bounded complexes with cohomology in ${\cal
  Q}_!$.  But not every such complex is represented by a complex in ${\cal
  Q}_!$.  Indeed, if we consider a surface $X$ such that the only
irreducible object in ${\cal Q}_!(X)$ is $\sO_r$ for a fixed $-2$-curve
$r$, then $\Ext^1_X(\sO_r,\sO_r)=0$ implies that $D^b{\cal Q}_!(X)\cong
D^b\Vect$.  It follows that $\Ext^2_{{\cal Q}_!(X)}(\sO_r,\sO_r)=0$, while
Poincar\'e duality tells us that $\Ext^2_X(\sO_r,\sO_r)\cong \C$.

It turns out, however, that if we impose additional conditions on the
surface, we can arrange for the derived equivalence to restrict to an {\em
  abelian} equivalence on ${\cal Q}_!(X)$, giving the following.  For
notation, given $\rho:\Pic(X_7)\to \C$ and $\omega:\Z^4\to \C$, define a
surface
$Y_{\rho;\omega}:=X_{\rho,\rho(2s+3f-e_1-\cdots-e_7)+\omega_4;\omega_1,\omega_2,\omega_3}$.

\begin{thm}
  Let $\omega:\Z^4\to \C$ be a linear map with rank 3 image and
  $\rho:\Pic(X_8)\to \C$ be such that for any root $r$ of $E_8$,
  $\rho(r)\notin \im\omega$.  Fix $g\in \SL_4(\Z)$ and suppose that both
  $\langle\omega_1,\omega_2\rangle$ and $\langle
  (g\omega)_1,(g\omega)_2\rangle$ are lattices.  Then there is a locally
  analytic equivalence ${\cal Q}_!(Y_{\rho;\omega})\cong {\cal
    Q}_!(Y_{\rho;g\omega})$.
\end{thm}

\begin{proof}
  Let $\beta$ be a primitive element of $\ker\omega$.  Since $\langle
  \omega_1,\omega_2\rangle$ is a lattice, it is in particular rank 2,
  implying $(\beta_3,\beta_4)\ne 0$.  We can then apply the $\GL_3(\Z)$
  action to ensure $\beta_1=\beta_2=0$ and thus $\gcd(\beta_3,\beta_4)=1$.

  Now, let $g\in \SL_4$ be an element of the $\SL_2$ acting on the last two
  coordinates such that $(g^t\beta)_4=0$, $(g^t\beta)_3=1$.  The proof of
  Theorem \ref{thm:rd_painleve} then gives a derived equivalence $D^b\coh
  Y_{\rho;g\omega}\to D^b\coh Y_{\rho;\omega}$.  Now, every object in
  ${\cal Q}_!(Y_{\rho;g\omega})$ is an extension of structure sheaves of
  points, and thus this derived equivalence takes ${\cal
    Q}_!(Y_{\rho;g\omega})\to {\cal Q}_!(Y_{\rho;\omega})$.  Since every
  irreducible object in $Y_{\rho;\omega}$ is in the image of this functor,
  it follows that the inverse derived equivalence also preserves the
  $t$-structure, and thus we obtain an equivalence of abelian categories as
  desired.

  We may thus reduce to the case $\beta=(0,0,0,1)$, or by using
  \[
    {\cal Q}_!(Y_{\rho;\omega_1,\omega_2,\omega_3,0})
    \cong
    {\cal Q}_!(Y_{\rho;\omega_1,\omega_2,0,\omega_3})
  \]
  to the case $\beta=(0,0,0,1)$.  The stabilizer of this covector is
  $\AGL_3(\Z)$, but only the $\GL_3(\Z)$ quotient acts nontrivially on the
  parameters, and thus a further application of Corollary \ref{cor:gl3Z}
  gives the desired result.
\end{proof}

\begin{rem}
   Note that essentially the entire content of this Theorem is in the word
   ``analytic''.  Indeed, without the holomorphic structure, the
   category ${\cal Q}_!(Y_{\rho;\omega_1,\omega_2,0,\omega_3})$ is simply
   the direct sum of uncountably many copies of the category of sheaves on
   a surface supported on a single smooth point, and thus any two such
   categories are {\em algebraically} equivalent\dots
\end{rem}

\begin{rem}
  More generally, we have an (almost) action of the group
  $\C.\Lambda_{E_8}^4\rtimes (W(E_8)\times \SL_4(\Z))$ on these categories
  (or, more precisely, an extension of this action by a possibly trivial
  group of holomorphic autoequivalences), where $\C$ acts on $\rho$ by
  adding a multiple of $(2,1,1,1,1,1,1,1,1)$.
\end{rem}

\smallskip

One interesting consequence of the elliptic Riemann-Hilbert correspondence
is the following.

\begin{cor}
  There is a biholomorphic map from the moduli space of simple sheaves on
  $X_{\rho;q;p}$ disjoint from the anticanonical curve to the corresponding
  moduli space on $X_{\rho;p;q}$, which can be chosen to be respect the
  symplectic structures up to a scalar multiple.
\end{cor}

\begin{proof}
  The nontrivial statement is that the map is symplectic (up to a scalar).
  The difficulty is that the pairing is defined in terms of the $\Ext^2$
  groups in $\coh X$, but the functor is defined on ${\cal Q}_!$, and we
  have seen this need not have the same $\Ext^2$ groups.  We will thus need
  to embed our categories in slightly larger categories so that we can
  represent all of the relevant self-2-extensions inside the category.

  The proof of Theorem \ref{thm:RH} lets us assume without loss of
  generality that every object $M\in {\cal Q}_!(X_{\rho;q;p})$ satisfies
  $c_1(M)\cdot f>0$.  Now, blow up two more points in such a way that
  $\rho(f-e_{m+1}-e_{m+2})=1$, and let $\Sol$ be the functor on ${\cal
    Q}_!X_{m+2}$ constructed in Theorem \ref{thm:RH}.  This restricts to an
  equivalence
  \[
  \alpha_{m+2}^*\alpha_{m+1}^*{\cal Q}_!(X_{\rho;q;p})
  \cong
  \alpha_{m+2}^*\alpha_{m+1}^*{\cal Q}_!(X_{\rho;p;q}),
  \]
  so still induces a biholomorphic map on moduli spaces.  But we now have
  an additional sheaf $R=\sO_{f-e_{m+1}-e_{m+2}}(-1)$ with the property
  that
  $
  \Ext^1(\alpha_{m+2}^*\alpha_{m+1}^*M,R)\ne 0
  $
  for all $M\in {\cal Q}_!(X_{\rho;q;p})$.  For simple $M$, let $\alpha$ be
  an extension of $R$ by $\alpha_{m+2}^*\alpha_{m+1}^*M$, and $\beta$ an
  extension of $\alpha_{m+2}^*\alpha_{m+1}^*M$ by $R$ such that the Serre
  duality pairing between $\alpha$ and $\beta$ is nonzero.  In particular
  we find that $\Ext^2(R,R)\cong \C$ is generated by the Yoneda product
  $\alpha\beta$.  We can thus represent any nontrivial $2$-extension of $R$
  by itself as a complex inside ${\cal Q}_!(X_{m+2})$, and can then apply
  $\Sol$ to obtain a map
  $
  \Ext^2(R,R)\to \Ext^2(\Sol(R),\Sol(R))$,
  necessarily an isomorphism.  Similarly, we can generate $\Ext^2(M,M)\cong
  \C$ by $\beta\alpha$, obtaining an isomorphism
  $
  \Ext^2(M,M)\to \Ext^2(\Sol(M),\Sol(M))$.
  Moreover, we have
  \[
  \tr(\Sol(\beta\alpha))
  =
  -\tr(\Sol(\alpha\beta))
  =
  -c\tr(\alpha\beta)
  =
  c\tr(\beta\alpha)
  \]
  for some nonzero constant $c$.  It follows that the pullback of the
  symplectic structure on $X_{\rho;p;q}$ is a multiple of the symplectic
  structure on $X_{\rho;q;p}$ as required.
\end{proof}

\begin{rem}
  In fact, since we have not specified a differential on $\C^*/\langle
  p\rangle$, the two symplectic structures are in principle only defined up
  to a scalar.  Of course, there {\em is} a natural choice of differential,
  namely $dz/2\pi\sqrt{-1}z$, and presumably the Riemann-Hilbert
  correspondence is either symplectic or antisymplectic relative to this
  choice of differential.
\end{rem}

Under some additional assumptions, we can arrange to preserve stability.
Given an effective divisor such that $D\cdot C_m=0$, let $\chi(D)$ denote
the Euler characteristic of a sheaf disjoint from $C$ with that Chern
class; note that this extends to a linear functional on $\rho^{-1}\langle
p,q\rangle$.

\begin{prop}
  Suppose $|p|,|q|<1$ with $\rank\langle p,q\rangle=2$, and fix
  $\rho:\Pic(X_m)\to \C^*/\langle p\rangle$, $D_0\in\rho^{-1}\langle
  p,q\rangle$ a nontrivial effective divisor with $D_0\cdot C_m=0$.
  Suppose $D_a$ is an ample divisor such that for any $D\in
  \rho^{-1}\langle p,q\rangle\cap C_m^\perp$ with $D$, $D_0-D$ nontrivial
  effective, either
  \[
  \frac{D}{D\cdot D_a}=\frac{D_0}{D_0\cdot D_a}
  \quad\text{or}\quad
  \frac{\chi(D)}{D\cdot D_a}\ne \frac{\chi(D_0)}{D_0\cdot D_a}.
  \]
  Then there is a lift of $\rho$ to $\Hom(\Pic(X_m),\C^*)$ such that for
  any $M$ with $c_1(M)=D_0$, $M$ is $D_a$-stable iff $\Sol_{\rho;q;p}(M)$
  is $D_a$-stable, and thus $\Sol_{\rho;q;p}$ induces a biholomorphic map
  between the corresponding stable moduli spaces.
\end{prop}

\begin{proof}
  Define a map $\bar\chi:\rho^{-1}\langle p,q\rangle\cap C_m^\perp\to \Q$
  by
  \[
  \bar\chi(D) = \chi(D)-(D\cdot D_a)\frac{\chi(D_0)}{D_0\cdot D_a}.
  \]
  This is clearly linear with $\bar\chi(D_0)=0$, and is thus determined by
  its restriction to the negative definite lattice
  \[
  \rho^{-1}\langle p,q\rangle\cap \langle D_a,C_m\rangle^\perp.
  \]
  It follows that there is an element
  \[
  D'\in \langle D_a,C_m\rangle^\perp\otimes \Q
  \]
  such that $\bar\chi(D)=D'\cdot D$.

  Now, as $M$ ranges over all sheaves with $c_1(M)=D_0$ and $M$ disjoint
  from the anticanonical curve, it follows from Lemma
  \ref{lem:subsheaf_Cherns_finite} that there are only finitely many
  possible Chern classes of proper nontrivial subsheaves of $M$.  Such a
  sheaf $M$ is stable iff for every proper nontrivial subsheaf $N$ of $M$,
  $D'\cdot c_1(N)=\bar\chi(c_1(N))<0$.

  Now choose any lift $\rho$ to $\C^*$.  Since $\Sol$ is contravariant and
  preserves Chern classes, we see that $M$ is stable iff every proper
  nontrivial {\em quotient} $N$ of $\Sol(M)$ has $D'\cdot c_1(N)<0$.  But
  then for $r\gg 0$,
  \[
  \chi(N(-rD'))
  >
  \frac{c_1(N)\cdot D_a}{c_1(M)\cdot D_a}\chi(M).
  \]
  Since there are only finitely many possible numerical invariants of such
  $N$, we may choose $r$ independently of $N$, and furthermore insist that
  $rD'\in \Pic(X)$.  For such an $r$, we find that for quotients of
  $\Sol(M)$, $D'\cdot c_1(N)<0$ iff $\chi(N(-rD'))>\frac{c_1(N)\cdot
    D_a}{c_1(M)\cdot D_a}\chi(M)$, and thus conclude that $M$ is stable iff
  $\Sol(M)(-rD')$ is stable.  The claim follows once we absorb the twist by
  $-rD'$ into the choice of lift of $\rho$.
\end{proof}

In the $K^2=0$ case, this gives the following.

\begin{cor}
  For arbitrary parameters such that $|p|$, $|q|<1$ and $\rank\langle
  p,q\rangle=2$, the (commutative) quasiprojective surfaces
  $X_{\eta,x_0,\dots,x_7,(\eta^2 x_0^3/x_1\cdots x_7)q;1;p}[T^{-1}]$ and
  $X_{\eta,x_0,\dots,x_7,(\eta^2 x_0^3/x_1\cdots x_7)p;1;q}[T^{-1}]$ are
  biholomorphic.
\end{cor}

\begin{rem}
  Of course, this, too, extends to an analytic-isomorphism-class-preserving
  action of $\GL_3(\Z)$, or more precisely $\C.\Lambda_{E_8}^3\rtimes
  (W(E_8)\times \GL_3(\Z))$.  If we remove all of the $-2$-curves from the
  surfaces, then we obtain an honest action of this group (extended by
  analytic automorphisms as usual).  Note that the elliptic Painlev\'e
  equation appears inside this action as the action of a subgroup
  $\Z\subset \Lambda_{E_8}$.
\end{rem}

This raises the following question: If $X$, $Y$ are two quasiprojective
surfaces obtained by removing the anticanonical curve from a rational
surface with $K^2=0$, when are they biholomorphic?  (And what is the
holomorphic automorphism group of $X$?)

\subsection{The Fourier transform}
\label{sec:diffeq2_fourier}

To understand how the action of $W(E_m)$ interacts with the Riemann-Hilbert
functor, it will be helpful to give an alternate construction of the
functor.  In particular, on the (Morita/Mukai) principle that a functor
between categories of sheaves should come from a sheaf on the product, we
want to look for a sheaf $M_{\Sol}$ on $X_{\rho;q;p}\times X_{\rho;p;q}$
such that $\Sol(M)=\Hom(M,M_{\Sol})$ for $M\in {\cal Q}_!(X)$.  Since
$\Hom(\sO_{X_{\rho;p;q}},\Sol(M))$ is the space of suitably holomorphic
solutions of the difference equation corresponding to $M$, this suggests
that we try to construct $M_{\Sol}$ so that
$\Hom(\sO_{X_{\rho;p;q}},M_{\Sol})$ is a subsheaf of $\Mer$.

We in particular want to define a space of meromorphic functions that
contains every suitably holomorphic solution of our difference equations
but has a more analytic definition.  Since we are assuming that the
equation has only simple singularities, part of the condition is
straightforward: at every point, we have a bound on the order of the pole
of the solution; since this includes the requirement that the solution be
holomorphic where it is required to be holomorphic, this will certainly
contain the functions we require.  Unfortunately, the resulting space (even
with the assumption of symmetry) is too big; the problem is that our
operators have a natural tendency to
introduce additional singularities at points with $z^2\in p^\Z q^\Z \eta$,
and thus tend to end up in a larger space.  We thus need to impose
additional conditions at those points.

A similar issue arises with the Fourier transformation.  Recall that we had
to define the generalized Fourier transformation as a purely formal
transformation, despite the fact that we can associate to it a perfectly
well-behaved kernel function.  The difficulty there, of course, is that we
need to control the singularities of the input to the transformation in
order to specify the contour of integration.  Again, though, we find that
constraining the singularities alone is not enough to give a well-behaved
transformation, as the integral itself can introduce additional
singularities.  However, in that case, we may use the results of \cite[\S
  10]{xforms} to guide us, in particular Lemma 10.4 op. cit.  In addition
to constraints on singularities, that result has an additional (much
stranger) hypothesis, which suggests the following.

\begin{lem}\label{lem:strange_for_holonomic}
  Let $A(z)\in \GL_n(\C^*/\langle p\rangle)$, and suppose that
  $A(p\eta/z)=A(z)^{-1}$, $A(\pm\sqrt{\eta})=A(\pm\sqrt{p\eta})=1$, and $A$
  is holomorphic at every point with $x^2\in p^\Z q^\Z \eta$.  If $v$ is
  any $pq\eta$-symmetric solution of $v(qz)=A(z)v(z)$ which is also
  holomorphic at such points, then
  \[
  v(p^{(1+i)/2}q^{(1+j)/2}\sqrt{\eta})
  =
  v(p^{(1+i)/2}q^{(1-j)/2}\sqrt{\eta})
  =
  v(p^{(1-i)/2}q^{(1-j)/2}\sqrt{\eta})
  =
  v(p^{(1-i)/2}q^{(1+j)/2}\sqrt{\eta}),
  \notag
  \]
  for all $i,j\in \Z$ and either square root of $\eta$.
\end{lem}

\begin{proof}
  The symmetry of $v$ means that it suffices to prove
  \[
  v(p^{(1+i)/2}q^{(1+j)/2}\sqrt{\eta})
  =
  v(p^{(1+i)/2}q^{(1-j)/2}\sqrt{\eta}),
  \]
  where we may assume $j\ge 0$.  We can now write
  \[
  v(p^{(1+i)/2}q^{(1+j)/2}\sqrt{\eta})
  =
  A_j
  v(p^{(1+i)/2}q^{(1-j)/2}\sqrt{\eta}),
  \]
  where
  \[
  A_j = 
  A(q^{-1}p^{(1+i)/2}q^{(1+j)/2}\sqrt{\eta})
  A(q^{-2}p^{(1+i)/2}q^{(1+j)/2}\sqrt{\eta})
  \cdots
  A(q^{-j}p^{(1+i)/2}q^{(1+j)/2}\sqrt{\eta}),
  \]
  and the hypotheses ensure that the matrices are all defined.  We then
  find
  \begin{align}
  A_0 &= 1\notag\\
  A_1 &= A(p^{(1+i)/2}\sqrt{\eta})\notag\\
  A_{j+2} &= A(p^{(1+i)/2}q^{(j+1)/2}\sqrt{\eta}) A_j
  A(p^{(1+i)/2}q^{-(j+1)/2}\sqrt{\eta}).\notag
  \end{align}
  Since
  \begin{align}
    A(p^{(1+i)/2}\sqrt{\eta})&=1,\notag\\
    A(p^{(1+i)/2}q^{(j+1)/2}\sqrt{\eta})
    &=
    A(p^{(1+i)/2}q^{-(j+1)/2}\sqrt{\eta})^{-1}\notag,
  \end{align}
  it follows by induction that $A_j=1$ for all $j$, and the result follows.
\end{proof}

This leads us to the following definition.
  
\begin{defn}
  Let $\eta,x_0,\dots,x_m$,$p$,$q$ be parameters in $\C^*$ such that
  $|p|,|q|<1$, $\langle p,q\rangle$ has rank 2 and $x_i^2\notin p^\Z
  q^\Z\eta$ for $0\le i\le m$.  Then define ${\cal
    M}_{\eta,x_0,\dots,x_m;p,q}$ to be the space of functions $g(z)\in
  \mer(\C^*)$ such that
  \begin{itemize}
  \item[(a)] $g(pq\eta/z)=g(z)$.
  \item[(b)] $\Gampq(z/pqx_0,\eta/zx_0)\prod_{1\le i\le m}
    (z/x_i,pq\eta/x_iz;p,q) g(z)$ is holomorphic on $\C^*$
  \item[(c)] For all
    $i,j\in \Z$ and either square root of $\eta$,
    $
    g(p^{(1+i)/2}q^{(1+j)/2}\sqrt{\eta})
    =
    g(p^{(1+i)/2}q^{(1-j)/2}\sqrt{\eta})$.
  \end{itemize}
\end{defn}

\begin{rem}
Note that this ``strange'' condition is indeed {\em very} strange from an
analytic perspective: the points $p^{(1+i)/2}q^{(1+j)/2}\sqrt\eta$ and
$p^{(1+k)/2}q^{(1+l)/2}\sqrt\eta$ can be arbitrarily close together (and
indeed, the set has an at least $1$-dimensional set of limit points)
without implying any particular constraint on the distance between
$p^{(1+i)/2}q^{(1-j)/2}\sqrt\eta$ and $p^{(1+k)/2}q^{(1-l)/2}\sqrt\eta$.
As a result, it appears likely to be very difficult to answer even very
basic questions about the spaces ${\cal M}$.  Indeed, apart from cases when
we can apply Lemma \ref{lem:strange_for_holonomic}, it looks hard to even
determine if ${\cal M}$ is empty!
\end{rem}

To deal with the strange condition, the following result will be useful.

\begin{lem}\label{lem:weird_condition_gamma}
  Suppose $u^2\notin p^\Z q^\Z\eta$, and let
  $h(z):=\Gampq(z/pqu,\eta/zu)$.  Then $h(pq\eta/z)=h(z)$ and for all
  $i,j\in \Z$ and either square root of $\eta$,
  \[
  h(p^{(1+i)/2}q^{(1+j)/2}\sqrt\eta)
=
(-pqu/\sqrt\eta)^{ij}
h(p^{(1+i)/2}q^{(1-j)/2}\sqrt{\eta}).
\notag
\]
\end{lem}

\begin{proof}
  We may write
  \[
  \frac{h(p^{(1+i)/2}q^{(1+j)/2}\sqrt\eta)}
     {h(p^{(1+i)/2}q^{(1-j)/2}\sqrt{\eta})}
  =
  \frac{\Gampq(p^{(-1+i)/2}q^{(-1+j)/2}\sqrt\eta/u,p^{(-1-i)/2}q^{(-1-j)/2}\sqrt\eta/u)}
       {\Gampq(p^{(-1+i)/2}q^{(-1-j)/2}\sqrt\eta/u,p^{(-1-i)/2}q^{(-1+j)/2}\sqrt\eta/u)}.
  \]
  Denoting the right-hand side by $F_{ij}(u)$, we find
  $F_{i0}(u)=F_{0j}=1$, and can easily simplify
  \[
   F_{i1}(u)=\frac{\theta_p(p^{(-1+i)/2}\sqrt\eta/qu)}
                    {\theta_p(p^{(-1-i)/2}\sqrt\eta/qu)},
  \]
  which evaluates to $(-pqu/\sqrt\eta)^i$ as required, using the
  quasiperiodicity relation for $\theta_p$.  We then find
  \[
  \frac{F_{i(j+2)}(u)}{F_{ij}(u)}
  =
  F_{i1}(q^{(-1-j)/2}u)F_{i1}(q^{(1+j)/2}u)
  =
  (-pqu/\sqrt\eta)^{2i},
  \]
  from which the claim follows by induction in $|j-1/2|$.
\end{proof}

\begin{rem}
  This makes it straightforward to extend the definition of ${\cal M}$ to
  the case $x_0\in p^\Z q^\Z$, by imposing the appropriate condition on
  $\Gampq(z/pqx_0,\eta/zx_0) g(z)$.
\end{rem}

\begin{lem}
  Let $v(qz)=A(z)v(z)$ be the symmetric $q$-theta equation corresponding to
  a sheaf $M\in {\cal Q}_!(X_{\eta,x_0,\dots,x_m;q;p})$, where we suppose
  $x_i^2\notin p^\Z q^\Z\eta$ and no root of $D_m$ is effective.  Let $g$
  be the function arising from a global section of $M$ and a symmetric
  global section of the dual of the solution bundle.  Then
  \[
  \Gampq(z/pqx_0,\eta/zx_0)^{-1} g(z)\in {\cal M}_{\eta,x_0,\dots,x_m;p,q}.
  \]
\end{lem}

\begin{proof}
  Note that the fundamental matrix used to define the (contravariant)
  elliptic Riemann-Hilbert correspondence induces a pairing between the
  respective vector bundles, and thus a pairing on their global sections,
  allowing us to define $g$.  That $g$ is symmetric and has the required
  singularities is clear.  For the strange condition, we first observe that
  if $x_0^2\notin p^\Z q^\Z \eta$, then changing to the algebraic gauge
  gives an equation to which Lemma \ref{lem:strange_for_holonomic} applies.
  The function $g$ then satisfies
  \[
  \Gampq(z/pqx_0,\eta/zx_0)^{-1}g(z) = \sum_i h_i(z) v_i(z)
  \]
  where $v(z)$ is a solution of the elliptic equation and $h_i(z)$ are
  symmetric $p$-elliptic functions with poles controlled by $x_0$.  It then
  follows easily that $\Gampq(z/pqx_0,\eta/zx_0)^{-1}g$ satisfies the
  strange condition.
\end{proof}

The following is trivial.

\begin{prop} For $1\le i<m$,
\[
  {\cal M}_{\eta,x_0,\dots,x_i,x_{i+1},\dots,x_m;p,q}
  =
  {\cal M}_{\eta,x_0,\dots,x_{i+i},x_{i},\dots,x_m;p,q}
\]
\end{prop}

Denote this isomorphism by $s_{e_i-e_{i+1}}$.  There is also an isomorphism
corresponding to reflection in $f-e_1-e_2$, though this is somewhat more
complicated.

\begin{prop}
  If $g(z)\in {\cal M}_{\eta,x_0,\dots,x_m;p,q}$, then
\begin{align}
  (s_{f-e_1-e_2}g)(z):=
  \frac{\Gampq(x_1z/\eta,pqx_1/z,z/pqx_0,\eta/zx_0)}
       {\Gampq(z/x_2,pq\eta/x_2z,zx_1x_2/pqx_0\eta,x_1x_2/zx_0)}
       &g(z)
       \notag\\
       {}\in{} &{\cal M}_{\eta,\eta x_0/x_1x_2,\eta/x_2,\eta/x_1,x_3,\dots,x_m;p,q}.
       \notag
\end{align}
\end{prop}

\begin{proof}
  That $s_{f-e_1-e_2}g$ satisfies the symmetric condition is trivial, and
  the holomorphy condition says that
\begin{align}
  &\Gampq(x_1x_2 z/pq \eta x_0,x_1x_2/zx_0)
  (x_1z/\eta,pqx_1/z,x_2z/\eta,pqx_2/z;p,q)\notag\\
  &\times\prod_{3\le i\le m} (z/x_i,pq\eta/x_iz;p,q)
  (s_{f-e_1-e_2}g)(z)
\end{align}
  is holomorphic, but this is equal to
  \[
  \Gampq(z/pqx_0,\eta/zx_0)
  \prod_{1\le i\le m} (z/x_i,pq\eta/x_iz;p,q)
  g(z),
\]
  so that the holomorphy condition is automatic as well.  The strange
  condition is an immediate consequence of Lemma
  \ref{lem:weird_condition_gamma}.
\end{proof}

As the remaining reflection corresponds to the generalized Fourier
transformation, we would hope to be able to express it as a contour
integral, and indeed we can do this.  Define a kernel function (defined as
long as $x_0/pqx_1\notin p^\N q^\N$)
\begin{align}
K(z,w&;\eta,x_0,x_1;p,q)
:=
\notag\\
&
\frac{(p;p)(q;q)}{2\Gampq(x_0/x_1)}
\frac{\Gampq(z/w,zw/pq\eta,pq\eta x_0/x_1zw,wx_0/zx_1,x_1w/\eta,pqx_1/w,w/pqx_0,\eta/wx_0)}
     {\Gampq(x_1z/\eta,pqx_0/z,z/pqx_1,\eta x_0/zx_1^2,w^2/pq\eta,pq\eta/w^2)}.
\end{align}
This comes from the expression for the kernel of the generalized Fourier
transformation given above by applying the appropriate gauge transformation
and making suitable choices for lifts of the parameters to $\C^*$.

If the parameters satisfy the additional conditions
\begin{align}
x_1/x_0&\notin p^\N q^\N\notag\\
x_k/pq x_1&\notin p^\N q^\N, 2\le k\le m\notag\\
x_kx_l/pq\eta&\notin p^\N q^\N, 2\le k<l\le m,\notag
\end{align}
then for generic z, there exists a contour C which contains the
(generically distinct!) points
$p^\N q^\N z,
p^\N q^\N pq \eta x_0/x_1z,
p^\N q^\N pqx_1$ and $p^\N q^\N pq\eta/x_k, 2\le k\le m$
and excludes their images under $w\mapsto pq\eta/w$.  Then for any $g\in {\cal
M}_{\eta,x_0,x_1,\dots,x_m;p,q}$, we may define
\[
(s_{s-e_1}g)(z)
=
\int_C
K(z,w;\eta,x_0,x_1;p,q)
g(w)
\frac{dw}{2\pi\sqrt{-1}w}.
\]
Per \cite[\S 10]{xforms}, this extends to a meromorphic function in $z$.

\begin{prop}
  If the parameters satisfy the above conditions as well as
\begin{align}
  \eta/x_0x_1&\notin p^\Z q^\Z\notag\\
  \eta/x_1x_k&\notin p^\Z q^\Z, 2\le k\le m\notag\\
  x_k/x_l&\notin p^\Z q^\Z, 2\le k<l\le m,\notag
\end{align}
then $s_{s-e_1}$ induces a linear transformation
\[
  {\cal M}_{\eta,x_0,\dots,x_m;p,q}
  \to
     {\cal M}_{\eta x_0/x_1,x_1,x_0,x_2,\dots,x_m;p,q}.
\]
\end{prop}  

\begin{proof}
  We first note (to let us simplify notation) that if we rescale $z$, $w$
  by $u$, then this has the effect of rescaling $\eta$ by $u^2$ and
  $x_0$,\dots,$x_m$ by $u$.  It thus suffices to consider the case $\eta =
  1/pq$, which we do to make $g\in {\cal M}(\eta,x_0,\dots,x_m;p,q)$
  satisfy $g(1/z)=g(z)$.

Now, since both the kernel and the constraints on the contour are symmetric
under $z\mapsto x_0/x_1z$, it follows immediately that
$
(s_{s-e_1}g)(x_0/x_1z) = (s_{s-e_1}g)(z)$.
Similarly, once we have shown that $s_{s-e_1}g$ is holomorphic at points of
the form $p^{i/2}q^{j/2}\sqrt{x_0/x_1}$, the strange condition
follows from the fact that
\[
K(p^{i/2}q^{j/2}\sqrt{x_0/x_1},w)=K(p^{i/2}q^{-j/2}\sqrt{x_0/x_1},w)
\]
(appropriately modified if $x_0/x_1^3\in p^\Z q^\Z$).

It thus remains only to show the holomorphy condition, or equivalently
that
\begin{align}
&(pq x_1z,pq x_0/z;p,q)
\prod_{2\le i\le m} (z/x_i,x_0/x_1x_iz;p,q)\notag\\
&\qquad\times\int_C
\frac{\Gampq(z w^{\pm 1},(x_0/x_1z)w^{\pm 1},pq x_1 w^{\pm 1},w^{\pm 1}/pq x_0)}
     {\Gampq(w^{\pm 2})}
g(w)
\frac{dw}{2\pi\sqrt{-1}w}
\end{align}
analytically continues to the points where the contour is not defined.  But
this follows immediately from \cite[Lem.~10.4]{xforms}.
\end{proof}

\begin{rem}
  In particular, the strange condition on $g$ corresponds precisely to the
  condition that
  $\Delta(\pm p^{i/2}q^{j/2})=-\Delta(\pm p^{i/2}q^{-j/2})$,
  where $\Delta(w) = K(z,w) g(w)$.
\end{rem}

\begin{rem}
  The conditions
  \begin{align}
    \eta/x_1x_k&\notin p^\Z q^\Z, 2\le k\le m\notag\\
    x_k/x_l&\notin p^\Z q^\Z, 2\le k<l\le m,\notag
  \end{align}
  are needed to ensure that the integrand has no multiple poles.
  Presumably the claim still holds without those conditions, as
  \cite[Lem.~10.4]{xforms} continues to hold in families violating the
  multiple pole condition, as long as the {\em generic} integral in the
  family has only simple poles.
\end{rem}

\begin{rem}
  In the case $m=8$, a different analytic representation of $W(E_9)$
  (acting in the same way on difference equations!) was considered in
  \cite{RuijsenaarsSNM:2015}, which in particular constructed solutions by
  taking eigenfunctions of the operator corresponding to a translation.
\end{rem}

We thus see that for every simple reflection in $W(E_{m+1})$, there is a
corresponding linear transformation between spaces ${\cal M}$, and thus for
any word in $W(E_{m+1})$ we obtain such a linear transformation as long as
the intermediate spaces are defined.

\begin{thm}
  Suppose $s_1\cdots s_k=s'_1\cdots s'_l$ are two words for the same
  element of $W(E_{m+1})$.  Then the corresponding linear transformations
  between spaces ${\cal M}$ agree whenever both are defined.
\end{thm}

\begin{proof}
  Clearly, it suffices to prove this for the Coxeter relations.  Most of
  these relations at most multiply the function by a product of elliptic gamma
  functions, and the corresponding identity reduces to the reflection
  principle $\Gampq(x,pq/x)=1$.  All but two of the relations involving $s-e_1$
  are similarly trivial (as both sides are integral operators with the same
  kernel), so that we need only prove
  \begin{align}
  s_{s-e_1}^2 &= 1,\notag\\
  s_{s-e_1}s_{e_1-e_2}s_{s-e_1} &= s_{e_1-e_2}s_{s-e_1}s_{e_1-e_2}.\notag
  \end{align}

  For the braid relation, we need to prove
  \begin{align}
  \int_{C_1}&\int_{C_2}
  K(z,y;\eta x_0/x_1,x_1,x_2)
  K(y,w;\eta,x_0,\dots,x_m;p,q)
  g(w)
  \frac{dw}{2\pi\sqrt{-1}w}
  \frac{dy}{2\pi\sqrt{-1}y}\notag\\
  &=
  \int_{C_3}
  K(z,w;\eta,x_0,x_2)
  g(w)
  \frac{dw}{2\pi\sqrt{-1}w},
  \end{align}
  as an identity of meromorphic functions in $z$. For $z$ sufficiently
  small, we can interchange the integrals on the left, and thus reduce to
  showing
  \[
  \int_C
  K(z,y;\eta x_0/x_1,x_1,x_2)
  K(y,w;\eta,x_0,\dots,x_m;p,q)
  \frac{dy}{2\pi\sqrt{-1}y}
  =
  K(z,w;\eta,x_0,x_2).
  \]
  Without loss of generality, we may rescale the parameters so that $\eta =
  x_1/pqx_0$, at which point the desired identity is simply the elliptic beta
  integral \cite{SpiridonovVP:2001} with parameters
  \[
  (1/x_0,pq x_2,z,x_1/x_2z,x_0w/x_1,1/w).\notag
  \]

  It remains to show that $s_{s-e_1}$ is an involution.  We may view
  ${\cal M}_{\eta,x_0,\dots,x_m}$ as a subspace of ${\cal
    M}_{\eta,x_0,x_1,x_0/u,x_2,\dots,x_m}$ for any $u$, and then apply the
  above braid relation to conclude that
  \[
  \int\int
  K(z,y;\eta x_0/x_1,x_1,x_0/u)
  K(y,w;\eta,x_0,x_1)
  g(w)
  =
  \int
  K(z,w;\eta,x_0,x_0/u)
  g(w).
  \]
  Since $s_{s-e_1}^2g$ is the limit of the left-hand side as $u\to 1$,
  it remains to show that
  \[
  \lim_{u\to 1}
  \int
  K(z,w;\eta,x_0,x_0/u)
  g(w)
  =
  g(z).
  \]
  The contour condition cannot be satisfied at $u=1$, and thus we must
  deform the contour in order to take the limit.  The factor
  $\Gampq(u)^{-1}$ in the kernel vanishes at $u=1$, and thus ensures that
  only the residue terms contribute to the limit.  In particular, we find
  (using symmetry)
  \[
  \lim_{u\to 1}
  \int
  K(z,w;\eta,x_0,x_0/u)
  g(w)
  =
  2
  \lim_{u\to 1}
  \Res_{w=z}
  K(z,w;\eta,x_0,x_0/u)
  g(w)
  =
  g(z)
  \]
  as required.
\end{proof}

\begin{rem}
  Note that the involution property is a variant of
  \cite[Thm.~4.1]{SpiridonovVP/WarnaarSO:2006}, with different conditions on the
  input functions.  Presumably the proof there could be adapted, with some
  additional complications arising from the fact that the contour is no
  longer the unit circle.
\end{rem}

Since we were able to use a suitable limit as sum of residues to prove the
involution property, this suggests looking at other such limits.  The next
simplest is the following, where up to symmetry we have residues at two
points.

\begin{prop}\label{prop:fourier_becomes_difference}
  Let $g\in {\cal M}_{\eta,x_0,x_1,\dots,x_m}\subset {\cal
    M}_{\eta,x_0,x_0/u,x_1,\dots,x_m}$.  Then
  \[
    \lim_{u\to 1/q}
    \int
    K(z,w;\eta,x_0,x_0/u)
    g(w)
    =
    (-1/pq^2x_0)
    \frac{z\theta_p(z/qx_0,z/q^2x_0,\eta/qzx_0,\eta/q^2zx_0)}
         {\theta_p(z^2/\eta)}
    (g(z)-g(qz)).
  \notag
  \]
\end{prop}

Here, we should recognize the right-hand side (apart from the factor
$-1/pq^2x_0$) as the action of the standard operator of degree $s$.
In general, taking the limit $u\to q^{-l}$ gives the operator of degree
$ls$; we omit the details.

Now, suppose the parameters satisfy $\eta/x_ix_j\notin p^\Z q^\Z$ for $1\le
i,j\le m$ and $x_i/x_j\notin p^\Z q^\Z$, $1\le i<j\le m$.

\begin{thm}
  With parameters as above, there is a subsheaf
  $\widetilde\Mer\subset \Mer$ such that
  \[
  \Hom(\sO_X(-ds-d'f+r_1e_1+\cdots+r_m e_m),\widetilde\Mer)
  =
  {\cal M}_{q^{-d}\eta,q^{d-d'}x_0,q^{-r_1}x_1,\dots,q^{-r_m}x_m;p,q}
  \]
\end{thm}

\begin{proof}
  We first show that the given family of spaces forms a module over
  $\cS'_{\eta,x_0,\dots,x_m;p,q}$.  Equivalently, we need to show that if
  $D=ds+d'f-r_1e_1-\cdots-r_me_m$, then
  \[
    [D]{\cal M}_{\eta,x_0,x_1,\dots,x_m;p,q}
    \subset
        {\cal M}_{q^{-d}\eta,q^{d-d'}x_0,q^{-r_1}x_1,\dots,q^{-r_m}x_m;p,q}.
  \]
  The case $D=e_1$ follows from        
  \[
  {\cal M}_{\eta,x_0, x_1,x_2,\dots,x_m;p,q}
  \subset
  {\cal M}_{\eta,x_0,qx_1,x_2,\dots,x_m;p,q},
  \]
  and the case $D=f-e_1$ is nearly as straightforward.  The case $D=s$ is
  more complicated.  To deal with this case, we must consider the function
  \[
  h(z) =
  \frac{z\theta_p(z/qx_0,z/q^2x_0,\eta/qzx_0,\eta/q^2zx_0)}
       {\theta_p(z^2/\eta)}
       (g(z)-g(qz)).
  \]
  This certainly satisfies the modified symmetry condition
  $h(p\eta/z) = h(z)$, and
  \[
  \Gampq(z/pq^2x_0,\eta/q^2 zx_0)\prod_{1\le i\le m}
  (z/x_i,p\eta/x_iz;p,q) h(z)
  \]
  is clearly holomorphic outside points with $z^2\in p^\Z\eta$.  That it is
  holomorphic at $p^{i/2}\sqrt\eta$ then follows from the strange condition in
  the special case
  $
  g(p^{1/2}q\sqrt\eta) = g(p^{1/2}\sqrt\eta)$.
  It thus remains to show that $h$ satisfies the strange condition, in the
  form
  $
    h(p^{(1+i)/2}q^{(j)/2}\sqrt{\eta}) = h(p^{(1+i)/2}q^{(-j)/2}\sqrt{\eta})$.
  By the symmetry of $h$, this is equivalent to
  $
    h(p^{(1+i)/2}q^{(j)/2}\sqrt{\eta}) = h(p^{(1-i)/2}q^{j/2}\sqrt{\eta})$,
  which reduces easily to the corresponding conditions for the factors
  \[
    \frac{z\theta_p(z/qx_0,z/q^2x_0,\eta/qzx_0,\eta/q^2zx_0)}
         {\theta_p(z^2/\eta)}
  \]
  and $g(z)-g(qz)$.

  Now, the conditions on the parameters allow us to apply any reflection in
  $A_{m-1}$, implying that operators of degree $[e_i]$ and $[f-e_i]$ also act
  correctly.  This also ensures that $\cS'_{\eta,x_0,x_1;q;p}$ is generated by
  $[s]$, $[e_1]$ and $[f-e_1]$, and thus that
  \[
    [ds+d'f-r_1e_1]{\cal M}_{\eta,x_0,\dots,x_m;p,q}
    \subset
    {\cal M}_{q^{-d}\eta,q^{d-d'}x_0,q^{-r_1}x_1,x_2,\dots,x_m;p,q},
  \]
  and similarly for $[ds+d'f-r_ie_i]$.  More generally,
  \begin{align}
    [ds+d'f-r_1e_1-\cdots-r_me_m]{\cal M}_{\eta,x_0,\dots,x_m;p,q}
    &\subset
    [ds+d'f-r_ie_i]{\cal M}_{\eta,x_0,\dots,x_m;p,q}\notag\\
    &\subset
    {\cal M}_{q^{-d}\eta,q^{d-d'}x_0,\dots,x_{i-1},q^{-r_i}x_i,x_{i+1},\dots,x_m;p,q}.
  \end{align}
  Since
  \[
  \bigcap_{1\le i\le m}
         {\cal M}_{q^{-d}\eta,q^{d-d'}x_0,\dots,x_{i-1},q^{-r_i}x_i,x_{i+1},\dots,x_m;p,q}
         =
  {\cal M}_{q^{-d}\eta,q^{d-d'}x_0,q^{-r_1}x_1,\dots,q^{-r_m}x_m;p,q},
  \]
  it follows that this indeed gives a module as required.

  It remains only to show that this module is saturated (it is a submodule
  of $\Mer$, so does not contain torsion).  The conditions on the
  parameters ensure that $C_m+mf$ is ample, and thus it suffices to show
  saturation relative to elements of degree $C_m+mf$.  Since $\Mer$ is
  saturated, this reduces to showing that if $g$ is a $pq\eta$-symmetric
  meromorphic function such that
  \[
    [C_m+mf]g\subset
    {\cal M}_{\eta/q^2,x_0/q^{m+1},x_1/q,x_2/q,\dots,x_m/q;p,q},
  \]
  then $g\in {\cal M}_{\eta,x_0,x_1,x_2,\dots,x_m;p,q}$.
  Since $T[mf]\subset [C_m+mf]$, we conclude that
  \[
    [mf]g\subset
    T^{-1}{\cal M}_{\eta/q^2,x_0/q^{m+1},x_1/q,x_2/q,\dots,x_m/q;p,q}
    =
    {\cal M}_{\eta,x_0/q^m,x_1,x_2,\dots,x_m;p,q},
  \]
  and thus reduce to showing that if
  $
    [f]g\subset {\cal M}_{\eta,x_0/q,x_1,x_2,\dots,x_m;p,q}$,
  then $g\in {\cal M}_{\eta,x_0,\dots,x_m;p,q}$.  This is straightforward:
  the symmetry condition and the strange condition are unchanged under
  multiplication by symmetric elliptic theta functions, so only holomorphy
  could be at issue.  Dividing by a general section of $[f]$ at worst
  introduces a single $q^\Z$-orbit of poles, but these cannot actually
  arise, since no other section of $[f]$ would cancel them.
\end{proof}

This sheaf is essentially $W(E_{m+1})$-invariant.  More precisely, for
every simple reflection, the action of that simple reflection on ${\cal M}$
(where defined) induces a semilinear action on the corresponding sheaves,
compatible with the corresponding isomorphism of surfaces.  Again, it
suffices to consider the action on $[e_1]$, $[f-e_1]$ and $[s]$, which is
mostly straightforward.  Indeed, in each case (including the cases
involving the Fourier transform), we can obtain the requisite identity as a
limit of a braid relation in a suitable family of actions of $W(E_{m+2})$,
using Proposition \ref{prop:fourier_becomes_difference} and analogous limits for other reflections.

Now, the definition of the degree 0 subspace of $\widetilde\Mer$ is invariant
under swapping $p$ and $q$.  It follows that as $\eta$, $x_0,\dots,x_m$
vary over their $p^\Z q^\Z$ orbits, the corresponding spaces not only give
us modules over $\cS'_{\eta,x_0,\dots,x_m;p,q}$ but also over
$\cS'_{\eta,x_0,\dots,x_m;q,p}$.  As a result, this construction gives rise
to a quasicoherent sheaf on $X_{\eta,x_0,\dots,x_m;p,q}\times
X_{\eta,x_0,\dots,x_m;q,p}$, which is precisely the sheaf $M_{\Sol}$ we
wished to construct above.

Since this sheaf $M_{\Sol}$ respects the action of $W(E_{m+1})$ where things
are defined, we obtain the following.

\begin{prop}
  The elliptic Riemann-Hilbert correspondence commutes with the action of
  $W(D_m)$.
\end{prop}

\begin{prop}
  Suppose that $\rho(s-e_1)\notin \langle p,q\rangle$ and both
  $X_{\rho;q;p}$ and $X_{s_{s-e_1}\rho;q,p}$ admit Riemann-Hilbert
  functors.  Then $s_{s-e_1}\Sol_{\rho;q;p}\cong \Sol_{s_{s-e_1}\rho;q;p}
  s_{s-e_1}$.
\end{prop}

\subsection{Generalized hypergeometric functions}

If $M$ is a sheaf such that $\chi(M)=\chi(\Sol(M))=1$, then we would
typically expect that both $M$ and $\Sol(M)$ have unique global sections,
which would combine to give a specific solution of the corresponding
straight-line equation (which would also be a solution of the straight-line
$p$-difference equation).  Of course, in general, we could only expect
$\chi(M)=1$ to imply $\dim\Gamma(M)=1$ on an open set, but there is an
important special case.  Indeed, in the rigid case, if $M\cong \sO_r$ where
$r$ is a root of $W(E_{m+1})$, then we automatically have $H^1(M)=0$, and
thus $\Gamma(M)\cong \C$.  More precisely, if $r$ is a root of $W(E_{m+1})$
such that $r\cdot f>0$, then for sufficiently general $\rho$ such that
$pq\rho(r)=1$, there is a sheaf $\sO_r$ on $X_{\rho;q;p}$, and
$\Sol(\sO_r)\cong \sO_r$.  In particular, any such root gives rise to a
natural family of functions (modulo scalars), and it is natural to wonder
whether we can give an explicit description of this family.

In the simplest case $r=s-e_1$, this is quite simple: indeed, the
corresponding straight-line equation is just $[s-e_1]g=0$, and the function
$1$ clearly satisfies both the $q$-difference and the $p$-difference
equation (and, more importantly, has the correct singularities).

More generally, if $G_r(z;\rho;p,q)$ is the solution corresponding to the
root $r$ and $s$ is any admissible simple reflection, then it follows
immediately from our general theory that
\[
G_{sr}(z;\rho;p,q)\propto s G_r(z;s^{-1}\rho;p,q).
\]
Since those simple reflections which do not act trivially either multiply
by a product of Gamma functions or integrate against the standard kernel,
we conclude the following.  

\begin{prop}
  For any root $r\in W(E_{m+1})$, there is a nonzero meromorphic function
  (a {\em generalized elliptic hypergeometric function})
  $G_r(z;\rho;p,q)=G_r(z;\rho;q,p)$ on the subspace of parameters with
  $pq\rho(r)=1$, expressible as a (possibly multiple) contour integral of a
  product of elliptic $\Gamma$ functions, and such that
  $[r]G_r(z;\rho;p,q)=0$.
\end{prop}

\begin{proof}
Note that although the Fourier transform was only defined as an integral
when certain parameter conditions were satisfied, it follows from
\cite[Thm.~10.2]{xforms} that the iterated integrals we obtain are
meromorphic on the full parameter space.
\end{proof}

\begin{rem}
  In other words, there is an integral representation for a solution to the
  equation $[r]g=0$.  Note, however, that we do not claim that this
  function is in any way unique; indeed, we will see that changing the path
  from $s-e_1$ to $r$ can multiply $G_r$ by a product of elliptic gamma
  functions involving the parameters (but not $z$).
\end{rem}

\begin{rem}
  For any element $w\in W(E_{m+1})$, $w\notin W(D_m)$, the corresponding
  action on meromorphic functions is given by an iterated integral
  transformation, which by Fubini can be expressed as a univariate integral
  transformation in which the kernel is itself expressed as an iterated
  integral.  In fact, the resulting kernel is (proportional to) a special
  case of an appropriate generalized elliptic hypergeometric function!
\end{rem}

The simplest cases of this are the first-order equations.  Applying
suitable reflections in $W(D_m)$ to $s-e_1$ gives us the following
expressions (where we set $\eta=1/pq$ for simplicity):
\begin{align}
G_{s+f-e_1-e_2-e_3}(z;1/pq,x_0,x_1,x_2,x_3;p,q)
&=
\frac{\Gampq(z^{\pm 1}/x_1,z^{\pm 1}/x_2,z^{\pm 1}/x_3)}
     {\Gampq(z^{\pm 1}/pqx_0)},
& x_0&=x_1x_2x_3\notag
\\
G_{s+2f-e_1-\cdots-e_5}(z;1/pq;x_0,x_1,\dots,x_5;p,q)
&=
\frac{\Gampq(z^{\pm 1}/x_1,\dots,z^{\pm 1}/x_5)}
     {\Gampq(z^{\pm 1}/pqx_0)},
& x_0&=pq x_1\cdots x_5\notag
\\
G_{s+3f-e_1-\cdots-e_7}(z;1/pq;x_0,x_1,\dots,x_7;p,q)
&=
\frac{\Gampq(z^{\pm 1}/x_1,\dots,z^{\pm 1}/x_7)}
     {\Gampq(z^{\pm 1}/pqx_0)},
& x_0&= p^2q^2 x_1\cdots x_7\notag
\end{align}
and so forth.  Permuting so that $e_1\cdot r=0$ and then reflecting in
$s-e_1$ gives a sequence of univariate integrals:
\begin{align}
  &G_{s+f-e_2-e_3-e_4}(z;1/pq,x_0,x_1,x_2,x_3,x_4;p,q)\notag\\
  &{}=
  \frac{\Gampq(z^{\pm 1}/x_1)}
       {\Gampq(x_1/x_0,z^{\pm 1}/pqx_0)}
  I^{(0)}(pq\sqrt{x_0x_1},\sqrt{x_0/x_1}/x_2,\sqrt{x_0/x_1}/x_3,\sqrt{x_0/x_1}/x_4,\sqrt{x_1/x_0}z^{\pm 1};p,q)
\notag
\end{align}
  with $x_2x_3x_4=x_0$,
\begin{align}
  &G_{2s+2f-e_1-e_2-e_3-e_4-e_5-e_6}(z;1/pq,x_0,x_1,x_2,x_3,x_4,x_5,x_6;p,q)\notag\\
  &=
  \frac{\Gampq(z^{\pm 1}/x_1)}
       {\Gampq(x_1/x_0,z^{\pm 1}/pqx_0)}
       I^{(1)}(pq\sqrt{x_0x_1},\sqrt{x_0/x_1}/x_2,\dots,\sqrt{x_0/x_1}/x_6,\sqrt{x_1/x_0}z^{\pm 1};p,q)
\notag
\end{align}
       with $pq x_1x_2\cdots x_6 = x_0^2$,
\begin{align}       
  &G_{3s+3f-2e_1-e_2-e_3-e_4-e_5-e_6-e_7-e_8}(z;1/pq,x_0,x_1,x_2,x_3,x_4,x_5,x_6,x_7,x_8;p,q)\notag\\
  &=
  \frac{\Gampq(z^{\pm 1}/x_1)}
       {\Gampq(x_1/x_0,z^{\pm 1}/pqx_0)}
  I^{(2)}(pq\sqrt{x_0x_1},\sqrt{x_0/x_1}/x_2,\dots,\sqrt{x_0/x_1}/x_8,\sqrt{x_1/x_0}z^{\pm 1};p,q)\notag
\end{align}
  with $p^2q^2 x_1x_2\cdots x_8 = x_0^2$, and so forth, where $I^{(m)}$
  denotes the ``order $m$ elliptic beta integral'', see in particular
  \cite{dets}, which derived the matrix form of the $q$-difference
  equations satisfied by $I^{(m)}$.  In particular, the limits of $I^{(m)}$
  include the ordinary hypergeometric function ${}_mF_{m-1}$, justifying
  the name ``generalized elliptic hypergeometric function'' above.

Note that for $m=0$, we did not actually change the root when reflecting in
$s-e_1$, and thus (for $x_2x_3x_4=x_0$)
\begin{align}
  &\frac{\Gampq(z^{\pm 1}/x_1)}
       {\Gampq(z^{\pm 1}/pqx_0)\Gampq(x_1/x_0)}
  I^{(0)}(pq\sqrt{x_0x_1},\sqrt{x_0/x_1}/x_2,\sqrt{x_0/x_1}/x_3,\sqrt{x_0/x_1}/x_4,\sqrt{x_1/x_0}z^{\pm 1};p,q)\notag\\
  &\propto
  \frac{\Gampq(z^{\pm 1}/x_2,z^{\pm 1}/x_3,z^{\pm 1}/x_4)}
       {\Gampq(z^{\pm 1}/pqx_0)}.
\end{align}
Using the explicit evaluation of the elliptic beta integral, we find that
the constant is
\[
\frac{\Gampq(x_2/x_1,x_3/x_1,x_4/x_1)}
     {\Gampq(x_2/x_0,x_3/x_0,x_4/x_0)}.
\]
In particular, we see as mentioned above that the value of $G_r$ can change
if we change the sequence of simple reflections taking $s-e_1$ to $r$.
This path dependence can be measured by a $1$-cocycle over the stabilizer
in $W(E_{m+1})$ of $s-e_1$, which is generated by $s+f-e_1-e_2-e_3$,
$f-e_1-e_2$, and $e_i-e_{i+1}$ for $2\le i<m$, only the first of which has
any effect on $G_{s-e_1}=1$.  We could remove this ambiguity if we could
extend this cocycle to $W(E_{m+1})$ (in particular if we could trivialize the
cocycle), as we could then rescale the action of the simple reflections in
such a way that the stabilizer of $s-e_1$ fixed $1$.

Note that the nontrivial transformation of $I^{(1)}$ corresponds to
reflection in the simple root $e_1-e_2$, though the above calculations only
determine the dependence on $z$ of the appropriate constant.  A cleaner
proof of the transformation using the above ideas is to compare expressions
for the action an appropriate element of $W(E_{m+1})$ coming from different
reduced words.  We omit the details.

\begin{appendices}
\section{An integrable system on relations}
\label{sec:integrable}

The original approach to noncommutative geometry was via generators and
relations, an approach we have mostly eschewed above.  We did, however,
make brief use of a presentation for $F_0$ (valid whenever $s-f$ is
ineffective) in proving flatness in that case.  The key property of this
presentation was that it gave us a quadratic Gr\"obner basis for the
$\Z^2$-algebra, suggesting the general question of when such a presentation
satisfies the Gr\"obner basis property.

In particular, suppose we have a presentation of a $\Z^2$-algebra in which
starting at each object we have four generators (two each of degrees
$(0,1)$ and $(1,0)$) and six relations (one each of degree $(2,0)$ and
$(0,2)$, and four of degree $(1,1)$).  If we view the horizontal relation
as an element of the tensor product of the two $\Hom$ spaces, then it must
have rank 2, lest the $\Z^2$-algebra have zero divisors.  But then we can
identify the two $\Hom$ spaces in such a way that the relation becomes
commutation, and similarly for the vertical relations.

Thus the general form of the presentation involves generators
\begin{align}
  x^{(d')}_1,x^{(d')}_2\in \Hom((d,d'),(d+1,d'))&=\Hom((0,d'),(1,d'))\notag\\
  y^{(d)}_1,y^{(d)}_2\in \Hom((d,d'),(d,d'+1))&=\Hom((d,0),(d,1)),\notag
\end{align}
which satisfy relations
\begin{align}
  x^{(d')}_1 x^{(d')}_2 &= x^{(d')}_2 x^{(d')}_1\notag\\
  y^{(d)}_1 y^{(d)}_2 &= y^{(d)}_2 y^{(d)}_1\notag\\
    x^{(d'+1)}_i y^{(d)}_j &=
    \sum_{k,l} A^{(d,d')kl}_{ij} y^{(d+1)}_k x^{(d')}_l,\notag
\end{align}
where $A^{(d,d')}$ is an invertible matrix for every $d,d'$.  As in the
elliptic case, the $\Hom$ spaces in quotient $\Z^2$-algebra are clearly
spanned by products of the form $y_2^{a_1} y_1^{a_2} x_2^{a_3} x_1^{a_4}$,
so have dimension at most $\max(d+1,0)\max(d'+1,0)$.  The question is thus
when these spanning sets are bases.

Note that we can rescale $A^{(d,d')}$ by rescaling one of the bases, and we
have enough such freedom to allow us to rescale all of the $A^{(d,d')}$
independently.  We will thus view each $A^{(d,d')}$ as a point in the
appropriate projective space.

By the theory of noncommutative Gr\"obner bases, the ordered monomials will
form a basis iff they form a basis in degree $(2,1)$ and $(1,2)$.  To
express this relation, it will be convenient to modify the parametrization
of the relations slightly.  The key idea is that a 2-dimensional vector
space admits an essentially unique symplectic form $\epsilon$ allowing us
to identify the space with its dual.  This allows us to replace
$A^{(d,d')}$ by an element of the tensor product of four 2-dimensional
vector spaces.  The relations then take the form
\begin{align}
\epsilon^{ij} x^{(d')}_i x^{(d')}_j &= 0\notag\\
\epsilon^{ij} y^{(d)}_i y^{(d)}_j &= 0\notag\\
x^{(d'+1)}_i y^{(d)}_j
&=
\epsilon^{km}\epsilon^{ln} A^{(d,d')}_{ijkl} y^{(d+1)}_m x^{(d')}_n,
\notag
\end{align}
where $\epsilon$ is the symplectic form (normalized so that
$\epsilon^{12}=1$, say) and we take the convention from physics that any
index that appears as both an upper and a lower index should be summed
over.  The Gr\"obner basis property in degree $(2,1)$ then says that
modulo the $A$ relations, we have
\[
\langle
(\epsilon^{ij} x^{(d'+1)}_i x^{(d'+1)}_j) y^{(d)}_1,
(\epsilon^{ij} x^{(d'+1)}_i x^{(d'+1)}_j) y^{(d)}_2
\rangle
=
\langle
y^{(d+2)}_1 (\epsilon^{ij} x^{(d')}_i x^{(d')}_j),
y^{(d+2)}_2 (\epsilon^{ij} x^{(d')}_i x^{(d')}_j) 
\rangle.
\]
This induces an isomorphism between $\Hom((d,d'),(d+1,d'))$ and
$\Hom((d,d'+2),(d+1,d'+2))$, so let us for the moment change basis in the
latter so that the condition becomes
\[
(\epsilon^{ij} x^{(d'+1)}_i x^{(d'+1)}_j) y^{(d)}_k
=
y^{(d+2)}_k (\epsilon^{ij} x^{(d')}_i x^{(d')}_j).
\]
We have
\begin{align}
(\epsilon^{ij} x^{(d'+1)}_i x^{(d'+1)}_j) y^{(d)}_k
&=
\epsilon^{ij}\epsilon^{ln}\epsilon^{mo}
A^{(d,d')}_{jklm}
x^{(d'+1)}_i y^{(d+1)}_n x^{(d')}_o\notag\\
&=
\epsilon^{ij} \epsilon^{ln}\epsilon^{mo} \epsilon^{eg}\epsilon^{fh}
A^{(d,d')}_{jklm}
A^{(d+1,d')}_{inef}
y^{(d+2)}_g x^{(d')}_h x^{(d')}_o,
\end{align}
and thus the Gr\"obner basis condition is that
\[
\epsilon^{ij}\epsilon^{ln}\epsilon^{mo} \epsilon^{eg}\epsilon^{fh}
A^{(d,d')}_{jklm}
A^{(d+1,d')}_{inef}
=
\delta^g_k \epsilon^{ho}
\]
(up to a change of basis in $g$), which can be restated as
\[
\epsilon^{ie}\epsilon^{kf}
A^{(d,d')}_{ijkl}
A^{(d+1,d')}_{efgh}
=
-\epsilon_{jg}\epsilon_{lh},
\]
where $\epsilon_{ij}$ is the inverse tensor,
$\epsilon^{ij}\epsilon_{jk}=\delta^i_k$.
Similarly, the other Gr\"obner basis property becomes
\[
\epsilon^{ih}\epsilon^{jf}
A^{(d,d')}_{ijkl}
A^{(d,d'+1)}_{efgh}
=
-
\epsilon_{kg}\epsilon_{le}.
\]
Each of these can be easily rewritten as stating that two matrices are
inverse to each other, and thus we see that $A^{(d,d')}$ uniquely
determines $A^{(d\pm 1,d')}$ and $A^{(d,d'\pm 1)}$, up to the
aforementioned choice of basis.  We in fact find that either way of
computing $A^{(d+1,d'+1)}$ gives the same answer, namely
\[
A^{(d+1,d'+1)}_{ijkl}
=
A^{(d,d')}_{lkji}.
\label{eq:shift11}
\]
(This reduces easily to the fact that $AB=1$ iff $BA=1$.)  It follows that
the choices of basis made above are consistent as we move around the
category.  In particular, we obtain a unique $\Z^2$-algebra associated to
any sufficiently general choice of $A^{(0,0)}$, in which all of the other
relations are determined as rational functions of the original relation.

It turns out that the resulting $\Z^2$-algebra is generically just our
elliptic $F_0$.  Indeed (following \cite{BhargavaM/HoW:2016}), if we choose
a genus 1 curve $C$, two degree 2 line bundles ${\cal L}_1$, ${\cal L}_2$,
a degree 0 line bundle $q$, and bases of $\Gamma(C;{\cal L}_1)$,
$\Gamma(C;{\cal L}_2)$, $\Gamma(C;{\cal L}_1\otimes q)$, $\Gamma(C;{\cal
  L}_2\otimes q)$, then the isomorphism
\[
  \Gamma(C;{\cal L}_1)\otimes \Gamma(C;{\cal L}_2\otimes q)
  \cong
  \Gamma(C;{\cal L}_1\otimes ({\cal L}_2\otimes q))
  \cong
  \Gamma(C;({\cal L}_1\otimes q)\otimes {\cal L}_2)
  \cong
  \Gamma(C;{\cal L}_1\otimes q)\otimes \Gamma(C;{\cal L}_2)
\]
induces a matrix of the form $A^{(0,0)}$, and one can moreover reconstruct
the data from $A^{(0,0)}$.  Indeed, $C$ is reconstructed as the
decomposable tensors (mod scalars) that map to decomposable tensors, and in
particular inherits four maps to $\P^1$.  The tuple $(C,{\cal L}_1,{\cal
  L}_2,q)$ varies over a 3-dimensional moduli space, and each choice of
basis modulo scalars introduces 3 more parameters, for a total of 15.  That
is, the moduli space of tuples of data as above is 15-dimensional, and the
map to $A^{(0,0)}$ embeds it in $\P^{15}$ as an open subset.  The operation
taking $A^{(0,0)}$ to $A^{(0,1)}$ is then just the operation of multiplying
${\cal L}_1$ by $q$, and the fact that the new space has an induced basis
arises from the natural isomorphism
\[
\Gamma(C;{\cal L}_1\otimes q^2)\cong \Gamma(C;{\cal L}_1)
\]
coming from translation by $q$.

Note that this gives a possible approach to understanding degenerations of
the noncommutative $\P^1\times \P^1$: simply understand the natural
stratification of the complement of the elliptic locus in $\P^{15}$.  This
is actually quite doable, but since it only gives cases admitting Fourier
transformations, this is not the approach we take in \cite{noncomm2}.
We do, however, mention that the case in which $A^{(0,0)}$ corresponds to
the relations
\begin{align}
  x_1y_1&=y_1x_1\notag\\
  x_1y_2&=y_2x_1\notag\\
  x_2y_1&=y_1x_2\notag\\
  x_2y_2&=y_2x_2+y_1x_1
\end{align}
is a fixed point for the evolution, and thus the resulting $\Z^2$-algebra
comes from a bigraded algebra.  This is essentially a compactification of
the Weyl algebra $Dt-tD=1$, which we can see by taking $(x_1,x_2)=(1,D)$,
$(y_1,y_2)=(1,t)$.  The analogue of the generalized Fourier transform gives
relations
\begin{align}
  x_1y_1&=y_1x_1\notag\\
  x_1y_2&=y_2x_1\notag\\
  x_2y_1&=y_1x_2\notag\\
  x_2y_2&=y_2x_2-y_1x_1,
\end{align}
but of course these are equivalent to the original relations by taking
$y_2\mapsto -y_2$.  In other words, the generalized Fourier transformation
acts on this algebra as $(D,t)\mapsto (-t,D)$, justifying the name.

\medskip

We now wish to consider the matrix-valued rational functions $A^{(d,d')}\in
\GL_4(k(A^{(0,0)}))$ in more detail.  In light of equation
\eqref{eq:shift11}, this essentially reduces to the case $d=0$.  Here, we
may note that the expression for $A^{(0,d')}$ in terms of $A^{(0,0)}$ is
the $d'$-th iterate of the function giving $A^{(0,1)}$ in terms of
$A^{(0,0)}$.  The relation to arithmetic on a family of abelian surfaces
$\Pic^0(C)^2$ suggests that this should be integrable (in the algebraic
entropy sense, see below), in that the degrees of these rational functions
should grow unexpectedly slowly.  Not only is this true, but we can in fact
be quite precise about the degrees.

\begin{prop}\label{prop:rels_are_integrable}
  Let $T$ be the birational map $\P^{15}\to \P^{15}$ such that
  $A^{(0,1)}\propto T A^{(0,0)}$.  Then $T^d$ has degree $2d^2+1$.
\end{prop}

To prove this, we note that $T$ fits into a larger family of birational
maps coming from the above structure.  Indeed, $T$ can be expressed in
terms of the obvious action of $S_4$ together with the birational
involution $R_{12}$ defined by
\[
\epsilon^{ie}\epsilon^{jf}
A_{ijkl} (R_{12}A)_{efgh}
=
-\epsilon_{kg}\epsilon_{lh}.
\]
The following is again straightforward.

\begin{lem}
  The map $R_{12}$ commutes with the subgroup $\langle
  (12),(34),(14)(23)\rangle$ of $S_4$.
\end{lem}

We thus see that $R_{12}$ has three distinct conjugates under $S_4$, which
we may denote by $R_1=R_{23}$, $R_2=R_{13}$ and $R_3=R_{12}$.  We thus
obtain a representation of $W_3\rtimes S_4$ as birational
automorphisms of $\P^{15}$, where $W_3$ denotes the free Coxeter
group on three generators (acting as $R_1$, $R_2$, $R_3$.  Since
any element of $S_4$ simply permutes the coordinates, to understand degrees
of elements in this group, it suffices to understand degrees of elements in
$W_3$.  Note that as a Coxeter group, $W_3$ comes with a
reflection representation $\rho$, in which the generators act as
\[
s_1=\begin{pmatrix}-1&2&2\\\hphantom{-}0&1&0\\\hphantom{-}0&0&1\end{pmatrix}\!,\quad
s_2=\begin{pmatrix}1&0&0\\2&-1&2\\0&0&1\end{pmatrix}\!,\quad
s_3=\begin{pmatrix}1&0&0\\0&1&0\\2&2&-1\end{pmatrix}\!,
\]
i.e., as reflections in the coordinate vectors relative to the inner
product
\[
\begin{pmatrix} \hphantom{-}2 & -2 & -2\\ -2 & \hphantom{-}2 & -2\\ -2 & -2 & \hphantom{-}2\end{pmatrix}.
\]

We then have the following.

\begin{thm}
  For any element $g\in W_3$, let $T_g$ denote the corresponding
  birational automorphism of $\P^{15}$.  Then the degree of $T_g$ is 2 less
  than the sum of the coefficients of $\rho(g)$.
\end{thm}

\begin{proof}
  For each $i$, $R_i$ is essentially just taking the inverse transpose of
  an appropriate matrix (relative to the symmetric form $\epsilon\otimes
  \epsilon$).  Let $\delta_i(A)$ denote the determinant of that matrix, and
  let $\Sigma_i$ denote the (smooth!) subscheme $\P^3\times \P^3\subset
  \P^{15}$ where that matrix has rank 1.

Now, for $d,r_1,r_2,r_3\in \N$, let $V(d,r_1,r_2,r_3)$ denote the space of
homogeneous polynomials of degree $d$ that vanish to order at least $r_i$
along $\Sigma_i$.  We first claim that if $\tilde{R}_1$ denotes the
primitive polynomial representation of $R_1$, one has
\[
\tilde{R}_1 V(d,r_1,r_2,r_3)
\subset
\delta_1(A)^{r_1}
V(3d-4r_1,2d-3r_1,r_2,r_3).
\]
Indeed, inverse transpose respects tensor products, so that $\tilde{R}_1$
is defined on and preserves $\Sigma_2$ and $\Sigma_3$, and therefore maps
the corresponding ideals into themselves.  Similarly, since the
coefficients of $\tilde{R}_1$ are (up to sign and ordering) just the
$3\times 3$ minors of the relevant matrix, they have degree 3 and vanish to
order 2 along $\Sigma_1$.  It follows that
\[
\tilde{R}_1 V(d,r_1,r_2,r_3)
\subset
\tilde{R}_1 V(d,0,r_2,r_3)
\subset
V(3d,2d,r_2,r_3).
\]
Now, the ideal of $\Sigma_1$ is generated by $2\times 2$ minors of the
appropriate matrix, and the image of such a minor under $\tilde{R}_i$ is
the product of the complementary minor by the determinant.  We thus find
that any element of $\tilde{R}_1 V(d,r_1,r_2,r_3)$ is a multiple of
$\delta_1(A)^{r_1}$.  Since this has degree 4 and vanishes to order $3$ along
$\Sigma_1$, we conclude
\[
\tilde{R}_1 V(d,r_1,r_2,r_3)
\subset
\delta_1(A)^{r_1}
V(3d-4r_1,2d-3r_1,r_2,r_3)
\]
as required.

Since $\tilde{R}_1^2(f) = \delta_1^{2\deg(f)} f$, $\tilde{R}_1 \delta_1(A) =
\delta_1(A)^3$, we further have
\begin{align}
\delta_1(A)^{2d} V(d,r_1,r_2,r_3)
&=
\tilde{R}_1^2 V(d,r_1,r_2,r_3)\notag\\
&\subset
\tilde{R}_1
\delta_1(A)^{r_1}
V(3d-4r_1,2d-3r_1,r_2,r_3)\notag\\
&\subset
\delta_1(A)^{3r_1}
\delta_1(A)^{2d-3r_1}
V(d,r_1,r_2,r_3)\notag\\
&=
\delta_1(A)^{2d} V(d,r_1,r_2,r_3)
\end{align}
Both inclusions are therefore equalities, and thus since the map on
parameters is an involution, we conclude that
\[
\tilde{R}_1 V(d,r_1,r_2,r_3)
=
\delta_1(A)^{r_1}
V(3d-4r_1,2d-3r_1,r_2,r_3),
\]
and similarly for $\tilde{R}_2$, $\tilde{R}_3$.  Note also that if
the elements of $V(d,r_1,r_2,r_3)$ are relatively prime, then the same will
be true for $V(3d-4r_1,2d-3r_1,r_2,r_3)$, as otherwise we could apply
$\tilde{R}_1$ to transport the $\gcd$ back.

Note that if we define
$W(r_1,r_2,r_3):=V(r_1+r_2+r_3-2,r_1-1,r_2-1,r_3-1)$
 for $r_1,r_2,r_3\ge 1$,
then the above action becomes
\[
\tilde{R}_1 W(r_1,r_2,r_3)
=
\delta_1(A)^{r_1-1}
W(-r_1+2r_2+2r_3,r_2,r_3).
\]
Here, of course, the action on the parameters is precisely the reflection
representation.

Now, if $s_{i_1}s_{i_2}\cdots s_{i_m}$ is any word in $W_3$, the
corresponding birational map is $\tilde{R}_{i_1}\cdots\tilde{R}_{i_m}$,
and the desired degree is the residual degree after removing the common
factors from
\[
\tilde{R}_{i_1}\cdots\tilde{R}_{i_m}V(1,0,0,0)
=
\tilde{R}_{i_1}\cdots\tilde{R}_{i_m}W(1,1,1)
\]
After removing the $\gcd$, we are left with a space of the form
$W(r_1,r_2,r_3)$ where $(r_1,r_2,r_3)$ is the image of $(1,1,1)$ under
$\rho(s_{i_1}\cdots s_{i_m})$, and the result thus follows from the fact
that elements of $W(r_1,r_2,r_3)$ have degree $r_1+r_2+r_3-2$.
\end{proof}

\begin{rem}
  Relative to the above identification of an open subset of $\P^{15}$ with
  a moduli space of smooth genus 1 curves with additional data, we see that
  $W_3$ acts as automorphisms on the corresponding abelian surface
  $\Pic^0(C)^2$.  It follows that the generic $W_3$ orbit is dense in a
  Kummer surface, namely $\Pic^0(C)^2/\pm 1$.  One can verify that each
  such Kummer surface is a quartic hypersurface in some $\P^3\subset
  \P^{15}$ with $12$ singular points.  The above argument arose from
  considerations of how $W_3$ acts on monodromy-invariant divisors on such
  Kummers; in particular, the union over all such Kummers of the singular
  points turns out to have closure precisely $\Sigma_1\cup \Sigma_2\cup
  \Sigma_3$.
\end{rem}

To prove Proposition \ref{prop:rels_are_integrable}, we need simply compute
$\rho((s_1s_2)^k)$ and $\rho((s_1s_2)^ks_1)$, as the corresponding $S_4$
double cosets contain $T^{2k}$ and $T^{2k+1}$ respectively.  The claim thus
reduces to showing
\[
\rho((s_1s_2)^k) =
\begin{pmatrix}
  2k+1 & -2k & 4k^2+2k\\
  2k & 1-2k & 4k^2-2k\\
  0 & 0 & 1
\end{pmatrix},
\]
easily verified by induction.

\medskip

Recall that the {\em algebraic entropy} of a birational automorphism $T$ of
projective space is given by \cite{BellonMP/VialletC-M:1999}
\[
\lim_{n\to\infty} \frac{1}{n}\log\deg(T^n),
\]
with an automorphism considered integrable if its algebraic entropy
vanishes.  The above result makes it straightforward to compute the
algebraic entropy of any birational automorphism in the above group.  This
is most simply expressed in terms of the homomorphism $\phi:W_3\rtimes
S_4\to \PGL_2(\Z)$ induced by the classical isomorphism $W_3\rtimes
S_3\cong \PGL_2(\Z)$.  Although elements of $\PGL_2(\Z)$ do not have
eigenvalues per se, their preimages in $\GL_2(\Z)$ have eigenvalues, and
the choice of preimage only affects the sign; as a result, we may define
{\em absolute values} of eigenvalues of elements in $\PGL_2(\Z)$.

\begin{cor}
  For any $g\in W_3\rtimes S_4$, let $\lambda$ denote the largest absolute
  value of an eigenvalue of $\phi(g)$.  Then $T_g$ has algebraic entropy
  $2\log \lambda$.
\end{cor}

\begin{proof}
  We first observe that replacing $g$ by $g^{12}$ multiplies both
  the algebraic entropy and $2\log \lambda$ by 12; since $g^{12}\in
  W_3$ for any $g\in W_3\rtimes S_4$, we thus reduce to the case
  $g\in W_3$.  Moreover, the reflection representation of $W_3$ is
  just (up to rational change of basis) the restriction of the
  adjoint representation of $\PGL_2(\Z)$ (i.e., the representation
  corresponding to the symmetric square of the 2-dimensional
  representation of $\GL_2(\Z)$).  If some preimage of $\phi(g)$
  has eigenvalues $z,1/z$, then the symmetric square has
  eigenvalues $z^2,1,1/z^2$.  It follows that the largest
  eigenvalue of $\rho(g)$ has absolute value $\lambda^2$.  It further
  follows that the matrix coefficients of $\rho(g^n)$ are $O(\lambda^{2n}
  n^2)$, and thus the same holds for $\deg(T_g^n)$.  The claim for
  the algebraic entropy follows immediately.
\end{proof}

\medskip

If we attempt to give a similar presentation for a deformation of
$(\P^1)^3$, then in addition to each plane having to satisfy the above
Gr\"obner basis property, there is also an additional consistency condition
in degree $(1,1,1)$, comparing the two ways of reducing $z_iy_jx_k$.  A
calculation along the above lines gives a Yang-Baxter-type identity, of the
form
\[
(A\otimes 1_2)(1_2\otimes B)(C\otimes 1_2)
=
(1_2\otimes D)(E\otimes 1_2)(1_2\otimes F)
\]
where each of the six $4\times 4$ matrices comes from the relations on a
given face of the cube.  If one moves to an adjacent cube, one of the
relations is directly carried over, while four others are determined by the
planar Gr\"obner basis property.  This overdetermines the remaining matrix,
which in fact depends only on the relation on the opposite face.  In
particular, the $2\times 2$ elliptic solution of the dynamic Yang-Baxter
equation \cite{FelderG:1994} gives rise to a deformation of $(\P^1)^3$
generalizing the elliptic $\P^1\times \P^1$, which further extends to
$(\P^1)^n$.  We omit the details, as we will give an equivalent
construction below.  We do mention, however, that the set of solutions of
the above YBE is invariant under the transformation $(A,B,C,D,E,F)\mapsto
(A^{-1},D,E,B,C,F^{-1})$ as well as under linear transformations arising by
rewriting the relation as
\[
\epsilon^{gj}\epsilon^{hk}\epsilon^{il}
A_{abgl}
B_{cdhj} 
C_{efik}
=
\epsilon^{mp}\epsilon^{nq}\epsilon^{or}
D_{bcmr}
E_{fanp}
F_{deoq}
\]
and observing that this admits a permutation action of the dihedral group
of order 12.  In particular, there is an integrable action on $(\P^{15})^6$
(coming from a subgroup of the wreath product $W_3^6\rtimes S_6$) that
preserves the locus of solutions of the generalized YBE.

\section{Blowing up $\P^n$}
\label{sec:genK6}

In addition to $\P^1\times \P^1$ and $F_1$, there is another case in which
we have (at least generically) a particularly nice presentation of our
category with a well-behaved Gr\"obner basis.  This is suggested by the
following observation in the commutative setting.

\begin{prop}\label{prop:X2_is_free}
  Let $X_2$ be the blowup of (commutative) $\P^2$ in the three coordinate
  vectors.  Then the $\Pic(X_2)$-graded coordinate ring of $X_2$ is the
  free polynomial ring in six generators of degree $e_1$, $e_2$, $e_3$,
  $h-e_2-e_3$, $h-e_1-e_3$, and $h-e_1-e_2$.
\end{prop}

\begin{proof}
  Consider any divisor class $dh-r_1e_1-r_2e_2-r_3e_3$.  Elements of that
  degree are homogeneous polynomials of degree $d$ that vanish to order
  at least $r_1$, $r_2$, $r_3$ at the three coordinate vectors.  This space
  is clearly spanned by monomials, and $x^{a_1}y^{a_2}z^{a_3}$ is in the
  space iff $a_2+a_3\ge r_1$, $a_1+a_3\ge r_2$, $a_1+a_2\ge r_3$,
  corresponding to the decomposition
  \begin{align}
  dh-r_1e_1-r_2e_2-r_3e_3
  ={}&
  (a_2+a_3-r_1)e_1
  +(a_1+a_3-r_2)e_2
  +(a_1+a_2-r_2)e_3\notag\\
  &{}+a_1(h-e_2-e_3)
  +a_2(h-e_1-e_3)
  +a_3(h-e_1-e_2).
  \end{align}
  In particular, the dimension of the given homogeneous subspace of the
  multigraded coordinate ring is precisely equal to the number of monomials
  in generators of the given degrees.
\end{proof}

\begin{rem}
  The reduction of dimension computations to combinatorial calculations
  reflects the fact that this particular rational surface is a toric
  variety.
\end{rem}

\begin{rem}
  Similarly, one finds that if one blows up two distinct points of $\P^2$,
  the resulting $\Pic(X_1)$-graded coordinate ring is the free polynomial
  ring in five generators, of degrees $e_1$, $e_2$, $h-e_1-e_2$, $h-e_1$,
  and $h-e_2$.  Note, however, that the generators of degree $h-e_1$ and
  $h-e_2$ are not unique (even when considered mod scalars), making the
  resulting deformed presentation more complicated.
\end{rem}

In the noncommutative setting, we have the following.

\begin{lem}
  If the surface $X_{\eta,x_0,x_1,x_2;q;C}$ has no effective $-2$-curves,
  then $\cS'_{\eta,x_0,x_1,x_2;q;C}$ is generated by morphisms of degree
  $s$, $e_1$, $e_2$, $s+f-e_1-e_2$, $f-e_2$, $f-e_1$.  This gives a
  presentation for $\cS'_{\eta,x_0,x_1,x_2;q;C}$ with $6$ generators from
  each object, with $15$ relations forming a quadratic Gr\"obner basis.
\end{lem}

\begin{proof}
  We first need to show that for any divisor class $D$, we can span $[D]$
  by products of the given degrees.  If $D\cdot e_2<0$, then
  $[D]=[e_2][D-e_2]$, so we can proceed by induction in $D\cdot C_2$.  By
  symmetry, the same applies to any of the six $-1$-classes.  We thus
  reduce to the case that $D$ is nef, and using the $W(E_3)$-symmetry, that
  $D$ is universally nef.  By our surjectivity results, this reduces to the
  cases $D\in \{f,s+f,s+2f-e_1,2s+3f-e_1-e_2\}$.  For the first three
  cases, it suffices to consider leading coefficients, and we have:
  \begin{align}
    [f]&=[e_1][f-e_1]+[e_2][f-e_2]\notag\\
    [s+f]&=[s][f]+[s+f-e_1-e_2][e_1][e_2]\notag\\
    [s+2f-e_1]&=[f-e_1][s+f]+[e_2][s+f-e_1-e_2][f]\notag
  \end{align}
  Similarly, we certainly obtain all leading terms for $[2s+3f-e_1-e_2]$,
  and $T\in [s+2f-e_1][s+f-e_2]$ follows by taking suitable choices of $v$,
  $v'$, $w$ in the proof of Lemma \ref{lem:central_elt}.

  Each of the spaces of degree a $-1$-class is $1$-dimensional, and
  thus we have the 6 generators as required.  The degree of each generator
  has intersection 1 with $C_2$, giving us a positive grading on the
  category.  We find by flatness that the $\Hom$ spaces from any given
  object have Hilbert series $1/(1-t)^6$ relative to this grading, and thus
  there must be $15$ quadratic relations.

  In fact, the degree 2 space is spanned by ordered products of the form
  \[
    [s]^{a_1}[e_1]^{a_2}[e_2]^{a_3}[s+f-e_1-e_2]^{b_1}[f-e_2]^{b_2}[f-e_1]^{b_3}
  \]
  with $a_1+a_2+a_3+b_1+b_2+b_3=2$.  Indeed, if $e$, $e'$ are orthogonal
  $-1$-classes, then
  \[
    [e+e'] = [e][e'] = [e'][e],
  \]
  letting us put those classes into preferred order.  There are three more
  possible sums of distinct $-1$-classes, for which we have
  \begin{align}
    [s+f-e_1] &= [e_2][s+f-e_1-e_2]+[s][f-e_1]\notag\\
    [s+f-e_2] &= [e_1][s+f-e_1-e_2]+[s][f-e_2]\notag\\
    [f] &= [e_1][f-e_1]+[e_2][f-e_2],\notag
  \end{align}
  from which the claim follows.

  By our Hilbert series calculation, these ordered products must in fact
  form a basis, and thus not only can there be no additional relations, but
  the quadratic relations must form a Gr\"obner basis.
\end{proof}
  
\begin{rem}
Note that the relations have a very similar form to the $\P^1\times \P^1$
case, in that the relations are precisely those quadratic relations
obtained from multiplication on $C$.
\end{rem}

\medskip

It turns out that we can use this description to define deformations of the
blowup of $\P^n$ in the coordinate vectors for any $n\ge 2$.  The argument
of Proposition \ref{prop:X2_is_free} works {\em mutatis mutandis} to show that
the resulting multigraded algebra is freely generated by commuting elements
of degree $e_1$,\dots,$e_{n+1}$, $f_1$,\dots,$f_{n+1}$, where we denote
\[
f_i := e_i + (h-e_1-\cdots-e_{n+1}).
\]
We thus need to construct deformed relations that still form a Gr\"obner
basis.

Let $C$ be a smooth genus 1 curve, and let $x_1$,\dots,$x_{n+1}$, $H$, $q$ be
divisors on $C$ of degree $1,\dots,1$, $n+1$, and $0$ respectively; assume
furthermore that no divisor of the form $q^l H/x_1\cdots x_{n+1}$ or
$q^l x_i/x_j$ is principal.  Also define a pairing on $\Lambda:=\langle
h,e_1,\cdots,e_{n+1}\rangle$ by
\[
h^2 = n-1,
\quad
h\cdot e_i = 0,
\quad
e_i\cdot e_j = -\delta_{ij}.
\]
Finally, define a linear functional $\kappa$ on $\Lambda$ by
$\kappa(h)=n+1$ and $\kappa(e_i)=1$, 
and note that $\kappa(f_i) = 1$ and
$\kappa(v)\in v^2+2\Z$ for all $v\in \Lambda$.

Now, given $v=dh-r_1e_1-\cdots-r_{n+1}e_{n+1}\in \Lambda$, let $D_v$ denote
the divisor
\[
q^{\frac{-v^2-\kappa(v)}{2}} H^d \prod_{1\le i\le n+1} x_i^{-r_i}
\]
of degree $\kappa(v)$.  Then we may define a $\Lambda$-algebra ${\cal
  B}_{H,x_1,\dots,x_{n+1};q;C}$ such that
\[
  {\cal B}(v,w) = \sO(D_w/D_v),
\]
with composition given by multiplication of functions on $C$.  Note that
for any function $f$ on $xC$, we can add $\div(f)$ to any of the defining
divisors and obtain an isomorphic category.  In particular, the isomorphism
class of the category depends only on the images of the defining divisors
in the Picard group.

Note that in addition to the
obvious $S_{n+1}$ symmetry, we also have an isomorphism
\[
  {\cal B}_{H,x_1,\dots,x_{n+1};q;C}
  \cong
      {\cal B}_{(n-1)H (x_1\cdots x_{n+1})^{2-n}
           ,H/x_2\cdots x_{n+1},\dots,H/x_1\cdots x_n;q;C}
\]
acting on $\Lambda$ via the reflection
\begin{align}
e_i&\mapsto h-e_1-\cdots-e_{n+1}+e_i\notag\\
h &\mapsto nh-(n-1)e_1-\cdots-(n-1)e_{n+1}.\notag
\end{align}
In addition, the sub-$\Z^{n+1}$-algebra with objects in $\langle
h-e_{n+1},e_1,\dots,e_n\rangle$ is isomorphic to
\[
  {\cal B}_{H/x_{n+1},x_1,\dots,x_n;q;C}.
\]

Now, let ${\cal B}'_{H,x_1,\dots,x_{n+1};q;C}$ be the subcategory of ${\cal
  B}_{H,x_1,\dots,x_{n+1};q;C}$ generated by elements of degree
$e_1$,\dots,$e_{n+1}$, $f_1$,\dots,$f_{n+1}$, all of $\kappa$-degree 1.
(Here we define the {\em $\kappa$-degree} of a morphism of degree $v$ to be
$\kappa(v)$.)  This collection of elements is invariant under the $S_{n+1}$
symmetry as well as the additional reflection, and thus ${\cal B}'$
inherits these symmetries, as well as the subalgebra property.

Finally, let ${\cal A}_{H,x_1,\dots,x_{n+1};q;C}$ be the category obtained
from ${\cal B}'_{H,x_1,\dots,x_{n+1};q;C}$ by removing all relations of
$\kappa$-degree $\ne 2$; again, this inherits an $A_1\times A_n$ symmetry
and reduction to subalgebras.

\begin{thm}
  For any $v,w\in \Lambda$, we have the direct sum decomposition
  \[
  {\cal A}(v,w) = \bigoplus_{\substack{\alpha,\beta\in \N^{n+1}\\ \sum_i
      (\alpha_i e_i + \beta_i f_i)=w-v}} [e_1]^{\alpha_1}\cdots
  [e_{n+1}]^{\alpha_{n+1}} [f_1]^{\beta_1}\cdots [f_{n+1}]^{\beta_{n+1}}.
  \]
  In particular, ${\cal A}_{H,x_1,\dots,x_{n+1};q;C}$ is a flat family of
  $\Lambda$-algebras.
\end{thm}

\begin{proof}
  Note first that one has an isomorphism
  \[
    {\cal S}'_{\eta,x_0,x_1,x_2;q;C}
    \cong
    {\cal A}_{\eta x_0,x_0,x_1,x_2;q;C};
  \]
  indeed, one finds that $\fD_{\eta x_0,x_0,x_1,x_2;q;C}(v)-D_v$ is
  principal for all $v$, giving an identification on leading coefficients;
  since we have shown that ${\cal S}'$ has precisely the given quadratic
  relations, the claim follows.  In particular, it follows that the Theorem
  holds in the case $n=2$.
  
  Now, the category ${\cal A}$ has $2n+2$ generators from each object.  If
  we map the generators to $\{1,\dots,n+1\}$ by taking both $e_i$ and $f_i$
  to $i$, then any $d$ of the generators determine a subset of
  $\{1,\dots,n+1\}$ of size at most $d$.  We may then use the $A_1\times
  A_n$ symmetry to make that subset contained in $\{1,\dots,d\}$, and then
  use the subalgebra property to let us work in any $n\ge d-1$.

  In particular, to understand the relations, we may work in any $n\ge 1$;
  since we know the result holds for $n=2$, it follows that the given
  elements span.

  It remains to show the Gr\"obner basis property, but for this it suffices
  to consider products of three generators.  Again, it suffices to verify
  the Gr\"obner basis property for $n=2$, where we already know it holds.
\end{proof}

\begin{rem}
  The form of divisors $D_v$ was of course chosen to match up with ${\cal
    S}'$, but one can show in general that for $n\ge 2$, $D_v$ must have
  the given form in order for the category to have enough relations.  We
  omit the details.
\end{rem}

\begin{rem}
  It is, in fact, not difficult to directly verify the Gr\"obner basis
  property for $n=2$ (and thus avoid using any results from the main part
  of the paper), as this simply involves verifying in each case that the
  resulting elements remain linearly independent in ${\cal B}$.
\end{rem}

Let $Z_n$ denote the $\Lambda$-graded algebra which in degree
$dh-r_1e_1-\cdots-r_{n+1}e_{n+1}$ consists of homogeneous polynomials of
degree $d$ on $\P^n$ vanishing at the $i$-th coordinate vector with
multiplicity at least $r_i$.  This becomes a $\Lambda$-algebra ${\cal Z}_n$
in the usual way.

\begin{prop}
  There is an isomorphism ${\cal A}_{H,x_1,\dots,x_{n+1};1;C}\cong {\cal
    Z}_n$ acting trivially on objects.
\end{prop}

\begin{proof}
  We may as well assume that $H$, $x_1$,\dots,$x_{n+1}$ are all effective.
  We can then use $H$ to embed $C$ in $\P^n$ in such a way that the images
  of $x_1$,\dots,$x_{n+1}$ are the coordinate vectors.  We in particular
  find that there is a natural identification between the $\kappa$-degree 1
  elements of the two categories; since the relations in ${\cal A}$ simply
  state that multiplication commutes, this identification extends to a
  functor.  But we can argue as in Proposition \ref{prop:X2_is_free} to see
  that $Z_n$ has no relations other than commutativity, and thus has the
  same relations as ${\cal A}$.
\end{proof}

The following is an immediate consequence.

\begin{cor}
  The subcategory of ${\cal A}_{H,x_1,\dots,x_{n+1};q;C}$ with objects $\Z
  h$ is a flat deformation of the $\Z$-algebra form of the homogeneous
  coordinate ring of $\P^n$.
\end{cor}

It turns out that this is not a new deformation of $\P^n$, as it can be
seen to be a $\Z$-algebra version of the algebra $Q_n({\cal E},\eta)$
constructed by Feigin and Odesskii, see \cite{OdesskiiAV:2002}.  In fact,
one finds that the above basis of ${\cal A}(0,dh)$ identifies that module
with a module over $Q_n({\cal E},\eta)$ of the form considered in
\cite[Prop.~8]{OdesskiiAV:2002}.  An interesting consequence is that the
resulting $\Z$-algebra is independent of $x_1$,\dots,$x_{n+1}$, suggesting
that it may be possible to extend the definition of ${\cal A}$ to more
general point configurations by taking a flat limit as in Section
\ref{sec:blowups}.

Though the deformation is not new, the extra scaffolding allows us to prove
some embedding results.  Note that the first two facts generalize the fact
that $F_1$ is a blowup of $\P^2$ and the fact that $\P^1\times \P^1$ is a
quotient of the Sklyanin algebra.

\begin{prop}
  For $d\ge 1$, the subalgebra of $\cS'_{\eta,x_0;q;C}$ with objects in
  $\Z(s+df)$ is a quotient of the above noncommutative $\P^{2d}$.
\end{prop}

\begin{prop}
  For $d\ge 1$, the subalgebra of $\cS_{\eta,\eta';q;C}$ with objects in
  $\Z(s+df)$ is a quotient of the above noncommutative $\P^{2d+1}$.
\end{prop}

\begin{prop}
  The second Veronese of the elliptic noncommutative $\P^2$ is a quotient
  of the above noncommutative $\P^5$.
\end{prop}

\begin{proof}
  In each case, let $D$ be the given divisor, and consider the surface
  obtained by blowing up $n+1=-D\cdot K$ points in general position; in
  particular, we insist that no two of the points are in the same
  $q^\Z$-orbit, and that $D-e_1-\dots-e_{n+1}$ not be a $-2$-curve.  Then
  there is a map from $\langle h,e_1,\dots,e_{n+1}\rangle$ to
  $\Pic(X_{n+1})$ given by
  $h\mapsto D$ and $e_i\mapsto e_i$,
  which moreover preserves the pairing and takes $\kappa$ to intersection
  with $-K$.  In particular, we find that the leading coefficient map is
  injective on $\Hom$ spaces in $\cS$ or $\cS'$ of degree corresponding to
  sums of one or two of the $e_i$ or $f_i$.  Since the above map on objects
  extends to a functor on the saturated leading coefficient categories and
  the relations come from leading coefficients, it follows that restricting
  to $\Z h$ gives a functor as required.

  To see that this functor is surjective, we note that the target category
  is in each case generated in degree 1, and thus it suffices to prove
  surjectivity in degree $h$.  But in each case the leading coefficient map
  is an isomorphism in the appropriate degree, and thus the result follows.
\end{proof}

\begin{rem}
  The above argument works in any case such that the leading coefficient
  map identifies $\Gamma(\sO_X(D))$ with the global sections of a very
  ample line bundle on $C$.  However, for such a $D$, $D+K$ is ineffective,
  while $D$ is nef, and consideration of the universally nef case shows
  that $D$ must be equivalent under $W(E_{m+1})$ to one of the above three
  forms.  In particular, we find that these are the only cases in which $D$
  is ample.
\end{rem}

\begin{rem}
  Note that although the noncommutative $\P^n$ is the $\Z$-algebra
  associated to a graded algebra, this is not true for the quotients in the
  Hirzebruch surface cases; the ideal is not invariant under the
  relevant automorphism of the $\Z$-algebra.
\end{rem}

A more novel family of deformations comes from the following.  Define
$g_{ij} := f_i+e_j=e_i+f_j$.

\begin{prop}
  The Veronese subalgebra of $Z_n$ corresponding to the sublattice spanned
  by $\langle g_{1,n+1},\dots,g_{n,n+1}\rangle$ is isomorphic to the
  multigraded coordinate ring of $(\P^1)^n$.
\end{prop}

\begin{proof}
  Consider an element of degree
  \[
  (r_1+\cdots+r_n)h-(r_2+\cdots+r_n)e_1-\cdots-(r_1+\cdots+r_{n-1})e_n
  \]
  in $Z_n$.  If $r_1<0$, then such an element vanishes at the first
  coordinate vector to higher order than its degree, and must therefore be
  0.  It follows that any nonzero element has $r_1,\cdots,r_n\ge 0$.

  Let the generators of the coordinate ring of $\P^n$ be
  $y_1,\dots,y_{n+1}$.  Then we find more generally that an element of the
  above degree has degree at most $r_i$ in $y_i$, and thus if we set
  $y_{n+1}$ we obtain the inhomogeneous form of a multihomogeneous
  polynomial of multidegree $(r_1,\dots,r_n)$.  On the other hand, any
  monomial of degree $\le r_i$ in $y_i$ gives rise to an element of the
  desired space, and the result follows.
\end{proof}

\begin{cor}
  The subcategory of ${\cal A}$ with objects $\langle
  g_{1,n+1},\dots,g_{n,n+1}\rangle$ is a flat deformation of the natural
  $\Z^n$-algebra of $(\P^1)^n$.
\end{cor}

In fact, we can say more.

\begin{prop}
  The given flat deformation of $(\P^1)^n$ is generated in degrees
  $g_{1,n+1}$,\dots,$g_{n,n+1}$, and the relations form a quadratic
  Gr\"obner basis.
\end{prop}

\begin{proof}
  We first note that the Gr\"obner basis property of ${\cal A}$ also holds
  relative to the ordering
  \[
    [e_1]^{a_1} [f_1]^{b_1}
    \cdots
    [e_{n+1}]^{a_{n+1}} [f_{n+1}]^{b_{n+1}};
  \]
  here we use the expansion $ [g_{i,j}] = [e_i][f_j]+[f_i][e_j]$ instead of
  $[g_{i,j}]=[e_i][f_j]+[e_j][f_i]$.  Next, we observe that
  \begin{align}
    [e_{n+1}][g_{i,n+1}] &= [g_{i,n+1}][e_{n+1}]\\
    [f_{n+1}][g_{i,n+1}] &= [g_{i,n+1}][f_{n+1}].
  \end{align}
  Indeed, using our standard basis for $[g_{i,n+1}]$, we need merely note
  that
  \[
  [e_{n+1}]([e_i][f_{n+1}]) = [e_i][e_{n+1}][f_{n+1}] =
  ([e_i][f_{n+1}])[e_{n+1}]
  \]
  and
  \[
    [e_{n+1}]([f_i][e_{n+1}]) = ([e_{n+1}][f_i])[e_{n+1}]
     \subset [g_{i,n+1}][e_{n+1}].
  \]
  It is then easy to show by induction (using $[e_i][f_i]=[f_i][e_i]$) that
  \[
    [g_{i,n+1}]^c
    =
    \bigoplus_{a+b=c}
    [e_i]^a [f_i]^b [e_{n+1}]^b [f_{n+1}]^a.
  \]  
  and then that
  \begin{align}
    [g_{1,n+1}]^{c_1}\cdots [g_{n,n+1}]^{c_n}
    &=
    \bigoplus_{a_1+b_1=c_1}
             [e_1]^{a_1} [f_1]^{b_1}
             [e_{n+1}]^{b_1} [f_{n+1}]^{a_1}
             [g_{2,n+1}]^{c_2}\cdots [g_{n,n+1}]^{c_n}\notag\\
    &=
    \bigoplus_{a_1+b_1=c_1}
             [e_1]^{a_1} [f_1]^{b_1}
             [g_{2,n+1}]^{c_2}\cdots [g_{n,n+1}]^{c_n}
             [e_{n+1}]^{b_1} [f_{n+1}]^{a_1}
  \end{align}
  But this implies by induction that
  \[
    [c_1g_{1,n+1}+\cdots+c_n g_{n,n+1}]
    =
    [g_{1,n+1}]^{c_1}\cdots [g_{n,n+1}]^{c_n},
  \]    
  so that the algebra is generated as required.
  
  We obtain $n$ quadratic relations (starting from any given object) on
  these generators from the fact that $\dim [2g_{i,n+1}]=3$ and $2n(n-1)$
  quadratic relations from the fact that $\dim [g_{i,n+1}+g_{j,n+1}]=4$,
  and these relations are clearly enough to reduce to the basis of ordered
  monomials.  It follows that there are no more relations, and moreover
  that the given relations form a Gr\"obner basis.
\end{proof}

\begin{rem}
  As we mentioned at the end of the previous section, the Gr\"obner basis
  property for these relations is a Yang-Baxter-type equation; the $n=3$
  case of the above can thus be viewed as giving a new proof of the
  $2\times 2$ elliptic solution of the dynamical YBE \cite{FelderG:1994}.
\end{rem}

\begin{rem}
The reader should be cautioned that the above constructions do not directly
give noncommutative schemes without making some sort of choice of ample
divisor.  Although in the 2-dimensional case, the choice turned out to be
mostly irrelevant, this is not the case for $n>2$.  The issue here is that
the above $\Z^{n+2}$-algebra could be viewed both as an $n+1$-point blowup
of $\P^n$ and as a $2$-point blowup of $(\P^1)^n$, and these rational
$n$-folds are no longer isomorphic for $n>2$.  For instance, the very ample
line bundles $\sO_X(d,d,d)$ on $(\P^1)^3$ have Euler characteristic
$(d+1)^3$, while the corresponding line bundles
$\sO_X(3dh-2de_1-2de_2-2de_3)$ on the blowup of $\P^3$ have the same global
sections, but have Euler characteristic $\frac{d^3+6d^2+7d+2}{2}$.
\end{rem}

We also obtain models for noncommutative rational surfaces as quotients of
this $(\P^1)^n$.  The nicest is probably the following.  Again we work
relative to an even blowdown structure for convenience.

We first consider deformations of (smooth) quintic del Pezzo surfaces.

\begin{prop}
  Let $X_3$ be a noncommutative rational surface such that no $-2$-curve is
  effective.  Let ${\cal F}$ be the category obtained from $\cS$ by
  insisting that all objects be integer linear combinations of $s$, $f$,
  $s+f-e_2-e_3$, $s+f-e_1-e_3$, $s+f-e_1-e_2$, and that all morphisms be
  {\em nonnegative} integer linear combinations of those elements.  Then
  ${\cal F}$ is generated in those degrees, and is a quotient of a suitable
  $(\P^1)^5$ as constructed above.
\end{prop}

\begin{proof}
  The claim about generation is essentially the only nontrivial claim, as
  it is easy to see that the morphisms satisfy the requisite relations.  We
  need to show that any $\Hom$ space of degree in the relevant cone is
  generated as claimed.  Since our generators are $W(E_4)\cong
  S_5$-invariant, we may restrict our attention to the intersection with
  the universal nef cone, which we may verify is the free monoid generated
  by the elements
  \[
  f,s+f,2s+2f-e_1-e_2,3s+3f-2e_1-e_2-e_3,4s+4f-2e_1-2e_2-2e_3.
  \]
  (Here we should note that $C_3$ is a {\em nonnegative} combination of the
  generating degrees, but not an {\em integer} combination.)  By our
  standard surjectivity results, it will suffice to show that $\Hom$ spaces
  of those five degrees are generated as required.

  The first space {\em is} a generator, while $[s+f]=[s][f]$ is standard.
  That $[2s+2f-e_1-e_2]=[s+f][s+f-e_1-e_2]$ was discussed in the proof of
  Proposition \ref{prop:X2_is_free}.

  For the fourth space, we observe that
  \[
    [3s+3f-2e_1-e_2-e_3]
    =
    [2s+2f-e_1-e_2][s+f-e_1-e_3]
    +
    [2s+2f-e_1-e_3][s+f-e_1-e_2],
  \]
  since the right-hand side certainly surjects on leading coefficients,
  reducing to
  \[
    [s+f-e_1] = [e_3][s+f-e_1-e_3]+[e_2][s+f-e_1-e_2].
  \]
  Since the four spaces are all in the given cone, the claim for this case
  follows from previous cases.

  Finally, we claim that
  \[
    [s+f-e_2-e_3][3s+3f-2e_1-e_2-e_3] = [4s+4f-2e_1-2e_2-2e_3].
  \]
  Again, this is surjective on leading coefficients, while
  \[
    [s+f-e_2-e_3][s+f-e_1] = [2s+2f-e_1-e_2-e_3]
  \]
  follows from the proof of Lemma \ref{lem:central_elt}, since none of
  $f-e_2-e_3$, $s-e_2-e_3$, $e_1-e_2$, $e_1-e_3$ are effective.
\end{proof}
  
\begin{rem}
  Note that this monoid contains ample elements, so that we indeed obtain
  a model for the noncommutative quintic del Pezzo surface as a
  ``closed subscheme'' of a noncommutative $(\P^1)^5$.
\end{rem}

There is also a construction for more general surfaces.
We begin by constructing surjections from the full category ${\cal A}$.
The functors will be straightforward to construct, but showing that the
appropriate subcategory has the correct generators will require a technical
lemma.

\begin{lem}
Let $C\subset \P^1\times \P^1$ be a smooth biquadratic curve, let
$d,d'\in 1+\N$, and suppose $p_1,\dots p_{d+d'}$ are distinct points of $C$.
For $i\in \{1,\dots,d+d'\}$, let $f_i$ be the nonzero bidegree $(1,0)$
polynomial vanishing at $p_i$, and let $g_i$ be the nonzero bidegree
$(0,1)$ polynomial vanishing at $p_i$.  Then the polynomials $\prod_{i\in
  S} f_i \prod_{i\notin S} g_i$ with $S$ ranging over $d$-element subsets
of $\{1,\dots,d+d'\}$ span the space of bidegree $(d,d')$ polynomials
vanishing at $p_1,\dots,p_{d+d'}$.
\end{lem}

\begin{proof}
Suppose first that $d,d'\ge 2$.  By induction, the claim holds for the
point set $p_1,\dots,p_{d+d'-1}$ in bidegree $(d-1,d')$ and $(d,d'-1)$.
Let $V_{d,d'}$ be the space of bidegree $(d,d')$ polynomials vanishing at
the base points, and let $V_{d-1,d'}$ and $V_{d,d'-1}$ be the corresponding
spaces vanishing at the first $d+d'-1$ base points.  We have $\dim
V_{d,d'}=d d' + 1$ as there are $(d-1)(d'-1)$ polynomials vanishing on $C$
and a line bundle of degree $2(d+d')$ has $d+d'$ sections vanishing on the
base points.  We similarly have $\dim V_{d-1,d'}=(d-1)d'+1$ and $\dim
V_{d,d'-1}=d(d'-1)+1$.

Now, we want to show that the map
\[
\begin{CD}
V_{d-1,d'}\oplus V_{d,d'-1}@> (f_{d+d'},g_{d+d'})>> V_{d,d'}
\end{CD}
\]
is surjective, and given what we have shown about the dimensions, it would
suffice to show that the kernel has dimension $(d-1)(d'-1)$.  Now, the
kernel consists of bidegree $(d-1,d'-1)$ polynomials $h$ such that
$f_{d+d'} h$ and $g_{d+d'} h$ both vanish at the base points.  Since the
only common zero of $f_{d+d'}$ and $g_{d+d'}$ is $p_{d+d'}$, we find that
$W$ is the space of polynomials of bidegree $(d-1,d'-1)$ vanishing at the
first $d+d'-1$ base points, and we may calculate as before that $\dim W =
(d-1)(d'-1)$ as required.

It remains to consider the case $d=1$.  The case $d'=1$ is straightforward
(there are only three distinct $\PGL_2^2$ orbits to consider), so we may
assume $d'>1$.  Then we still have $\dim V_{1,d'}=d'+1$, $\dim
V_{1,d'-1}=d'$, and multiplication by $g_{d'+1}$ injects $V_{1,d'-1}$ in
$V_{1,d'}$.  It will thus suffice to show that
$h=f_{d'+1}\prod_{1\le i\le d'} g_i$ is not in the image of $V_{1,d'-1}$.
If $g_i\ne g_{d'+1}$ for $1\le i\le d$, then $h$ is not even a multiple of
$g_{d'+1}$, so the claim is immediate.  Otherwise, suppose
$g_{d'}=g_{d'+1}$.  But then we see that none of the factors of
\[
f_{d'+1}\prod_{1\le i<d'} g_i
\]
vanish at $p_{d'}$, so again this polynomial is not in $V_{1,d'-1}$.
\end{proof}

\begin{prop}\label{prop:map_to_toric_blowup}
  Let $X_m=X_{\rho;q;C}$ be a noncommutative surface equipped with an even
  blowdown structure such that no root of the forms $s-f$, $f-e_i-e_j$,
  $s-e_i-e_j$, $e_i-e_j$ is effective.  Let ${\cal F}$ be the subcategory
  associated to the monoid generated by $f-e_1$,\dots,$f-e_m$ and
  $s-e_1$,\dots,$s-e_m$.  Then there is a functor from a special case of
  ${\cal A}$ to ${\cal F}$ acting on objects as the linear transformation
  \begin{align}
  h&\mapsto s+(m-1)f-e_1-\cdots-e_m\notag\\
  e_i&\mapsto f-e_i,\notag
  \end{align}
  and this functor is surjective on $\Hom$ spaces.
\end{prop}

\begin{proof}
  As above, the existence of a functor is straightforward, and the
  nontrivial part of the proof is showing that ${\cal F}$ is generated by
  elements of the given degrees.  The relevant monoid consists of elements
  \[
  ds+d'f-\sum_i r_i e_i
  \]
  where $d,d',r_1,\dots,r_m\ge 0$ and $d+d'=\sum_i r_i$, so we need to
  prove that every $\Hom$ space of that degree is spanned as required.

  Suppose first that $\max(r_i)=1$.  Then without loss of generality we
  have an element of the form
  \[
  ds+d'f-e_1-\cdots-e_m
  \]
  with $d+d'=m$ (as otherwise we can act by $S_m$ and then blow down) and
  $d\le d'$ (as otherwise we can reflect in $s-f$).  Since
  $mf-e_1-\cdots-e_m$ has negative intersection with each $-1$-curve of the
  form $f-e_i$, we have $[mf-e_1-\cdots-e_m]=[f-e_1]\cdots [f-e_m]$ as
  required, giving the case $d=0$.

  Otherwise, for any $d$-element subset $S\subset \{1,\dots,m\}$, we have
  a $1$-dimensional space
  \[
  \prod_{i\notin S} [f-e_i]
  \prod_{i\in S} [s-e_i]
  \subset
  [ds+d'f-e_1-\cdots-e_m]
  \]
  If we left-multiply by $\prod_i [e_i]$, we obtain the space
  \[
  \prod_{i\notin S} ([e_i][f-e_i])
  \prod_{i\in S} ([e_i][s-e_i])
  \subset
  [df][d's]
  \]  
  It follows from the Lemma that as $S$ varies, these span
  $ds+d'f-e_1-\cdots-e_m$.

  It remains to consider the case $\max(r_i)>1$, say $r_1>1$.
  In that case, we wish to show
  \begin{align}
    [ds+d'f-r_1e_1-\cdots-r_me_m]
    ={}&
    \hphantom{{}+{}}[f-e_1][ds+(d'-1)f-(r_1-1)e_1-\sum_{2\le i\le m} r_i e_i]\notag\\
    &+
    [ds+(d'-1)f-(r_1-1)e_1-\sum_{2\le i\le m} r_i e_i][s-e_1]
  \end{align}
  Since $s-f$ is ineffective, this holds for leading coefficients, so
  reduces to showing
  \begin{align}
    &[(d-2)s+(d'-3)f-(r_1-1)e_1-\cdots-(r_m-1)e_m]\notag\\
    &\qquad=
    \hphantom{{}+{}}[f-e_1][(d-2)s+(d'-3)f-(r_1-2)e_1-\sum_{2\le i\le m} (r_i-1) e_i]\notag\\
    &\qquad\hphantom{{}={}}+
    [(d-2)s+(d'-3)f-(r_1-2)e_1-\sum_{2\le i\le m} (r_i-1) e_i][s-e_1],
  \end{align}
  which holds by induction.
\end{proof}

\begin{cor}
  Let $X_m$ be a blowup of noncommutative $\P^2$ in $m+1$ points in
  distinct $q^\Z$-orbits, and let ${\cal F}$ be the subcategory associated
  to the monoid generated by $h-e_1$,\dots,$h-e_{m+1}$.
  Then ${\cal F}$ is generated in those degrees, and is a quotient of the
  above noncommutative $(\P^1)^{m+1}$.
\end{cor}

\begin{proof}
  If we blow up two more points in sufficiently general position, then
  modify the blowdown structure appropriately, we may rephrase this in
  terms of the monoid generated by $s+f-e_1-e_2$,\dots,$s+f-e_1-e_{m+2}$ on
  $X_{m+2}$.  The resulting surface satisfies the hypotheses of Proposition
  \ref{prop:map_to_toric_blowup}, and restricting the functor to the
  corresponding $(\P^1)^{m+1}$ (which is independent of the auxiliary
  blowups) gives the desired result.
\end{proof}

\begin{rem}
  Up to the action of $W(E_{m+1})$, this case and the quintic del Pezzo
  case are the only configurations of rational pencils such that any two of
  the pencils have intersection $1$.  Indeed, there is certainly an even
  blowdown structure in which the first pencil is $f$, at which point the
  remaining pencils have the form $s+d'f-\sum_i r_i e_i$.  The action of
  $D_m$ takes the next pencil to $s$, at which point the remaining pencils
  must all have the form $s+f-e_i-e_j$ for $i<j$.  We thus obtain a graph
  on the $m$ points such that any two edges share a vertex.  The only two
  cases are the complete graph $K_3$, corresponding to the quintic del
  Pezzo surface case, and star graphs, which relative to a $\P^2$ basis
  have the above form.
\end{rem}

\end{appendices}

\bibliographystyle{plain}

\begin{thebibliography}{10}

\bibitem{AltmanAB/KleimanSL:1980}
A.~B. Altman and S.~L. Kleiman.
\newblock Compactifying the {P}icard scheme.
\newblock {\em Adv. in Math.}, 35(1):50--112, 1980.

\bibitem{AndreottiA/FrankelT:1959}
A.~Andreotti and T.~Frankel.
\newblock The {L}efschetz theorem on hyperplane sections.
\newblock {\em Ann. of Math. (2)}, 69:713--717, 1959.

\bibitem{ArtinM/TateJ/VandenBerghM:1990}
M.~Artin, J.~Tate, and M.~Van~den Bergh.
\newblock Some algebras associated to automorphisms of elliptic curves.
\newblock In {\em The {G}rothendieck {F}estschrift, {V}ol.\ {I}}, volume~86 of
  {\em Progr. Math.}, pages 33--85. Birkh\"auser Boston, Boston, MA, 1990.

\bibitem{ArtinM/TateJ/VandenBerghM:1991}
M.~Artin, J.~Tate, and M.~Van~den Bergh.
\newblock Modules over regular algebras of dimension {$3$}.
\newblock {\em Invent. Math.}, 106(2):335--388, 1991.

\bibitem{ArtinM/ZhangJJ:1994}
M.~Artin and J.~J. Zhang.
\newblock Noncommutative projective schemes.
\newblock {\em Adv. Math.}, 109(2):228--287, 1994.

\bibitem{ArtinM/ZhangJJ:2001}
M.~Artin and J.~J. Zhang.
\newblock Abstract {H}ilbert schemes.
\newblock {\em Algebr. Represent. Theory}, 4(4):305--394, 2001.

\bibitem{BellonMP/VialletC-M:1999}
M.~P. Bellon and C.-M. Viallet.
\newblock Algebraic entropy.
\newblock {\em Comm. Math. Phys.}, 204(2):425--437, 1999.

\bibitem{BhargavaM/HoW:2016}
M.~Bhargava and W.~Ho.
\newblock Coregular spaces and genus one curves.
\newblock {\em Camb. J. Math.}, 4(1):1--119, 2016.

\bibitem{BirkhoffGD:1913}
G.~D. Birkhoff.
\newblock The generalized {R}iemann problem for linear differential equations
  and the allied problems for linear difference and {$q$}-difference equations.
\newblock {\em Proc. Am. Acad. Arts Sci.}, 49(9):521--568, 1913.

\bibitem{BirkhoffGD:1916}
G.~D. Birkhoff.
\newblock Infinite products of analytic matrices.
\newblock {\em Trans. Amer. Math. Soc.}, 17(3):386--404, 1916.

\bibitem{BondalAI/KapranovMM:1990}
A.~I. Bondal and M.~M. Kapranov.
\newblock Representable functors, {S}erre functors, and reconstructions.
\newblock {\em Math. USSR-Izv.}, 35(3):519--541, 1990.

\bibitem{BondalAI/PolishchukAE:1993}
A.~I. Bondal and A.~E. Polishchuk.
\newblock Homological properties of associative algebras: the method of
  helices.
\newblock {\em Izv. Ross. Akad. Nauk Ser. Mat.}, 57(2):3--50, 1993.

\bibitem{BorodinA:2004}
A.~Borodin.
\newblock Isomonodromy transformations of linear systems of difference
  equations.
\newblock {\em Ann. of Math. (2)}, 160(3):1141--1182, 2004.

\bibitem{BottacinF:1995}
F.~Bottacin.
\newblock Poisson structures on moduli spaces of sheaves over {P}oisson
  surfaces.
\newblock {\em Invent. Math.}, 121(2):421--436, 1995.

\bibitem{BridgelandT/KingA/ReidM:2001}
T.~Bridgeland, A.~King, and M.~Reid.
\newblock The {M}c{K}ay correspondence as an equivalence of derived categories.
\newblock {\em J. Amer. Math. Soc.}, 14(3):535--554 (electronic), 2001.

\bibitem{BrylawskiT:1973}
T.~Brylawski.
\newblock The lattice of integer partitions.
\newblock {\em Discrete Math.}, 6:201--219, 1973.

\bibitem{BuhlerT:2010}
T.~B\"uhler.
\newblock Exact categories.
\newblock {\em Expo. Math.}, 28(1):1--69, 2010.

\bibitem{ChanD/NymanA:2013}
D.~Chan and A.~Nyman.
\newblock Non-commutative {M}ori contractions and {$\mathbb{P}^1$}-bundles.
\newblock {\em Adv. Math.}, 245:327--381, 2013.

\bibitem{CooperS:2002}
S.~Cooper.
\newblock The {A}skey-{W}ilson operator and the {${}_6\Phi_5$} summation
  formula.
\newblock {\em South East Asian J. Math. Math. Sci.}, 1(1):71--82, 2002.

\bibitem{vanDiejenJF:1994}
J.~F. {\vphantom{Diejen}}van~Diejen.
\newblock Integrability of difference {C}alogero-{M}oser systems.
\newblock {\em J. Math. Phys.}, 35(6):2983--3004, 1994.

\bibitem{DrezetJ-M:1988}
J.-M. Drezet.
\newblock Groupe de {P}icard des vari\'et\'es de modules de faisceaux
  semi-stables sur {${\bf P}_2({\bf C})$}.
\newblock {\em Ann. Inst. Fourier (Grenoble)}, 38(3):105--168, 1988.

\bibitem{EllingsrudG/StrommeSA:1993}
G.~Ellingsrud and S.~A. Str{\o}mme.
\newblock Towards the {C}how ring of the {H}ilbert scheme of {${\bf P}^2$}.
\newblock {\em J. Reine Angew. Math.}, 441:33--44, 1993.

\bibitem{EtingofP/EuC-H:2007}
P.~Etingof and C.-H. Eu.
\newblock Koszulity and the {H}ilbert series of preprojective algebras.
\newblock {\em Math. Res. Lett.}, 14(4):589--596, 2007.

\bibitem{EtingofP/GanWL/OblomkovA:2006}
P.~Etingof, W.~L. Gan, and A.~Oblomkov.
\newblock Generalized double affine {H}ecke algebras of higher rank.
\newblock {\em J. Reine Angew. Math.}, 600:177--201, 2006.

\bibitem{EtingofPI:1995}
P.~I. Etingof.
\newblock Galois groups and connection matrices of {$q$}-difference equations.
\newblock {\em Electron. Res. Announc. Amer. Math. Soc.}, 1(1):1--9
  (electronic), 1995.

\bibitem{FelderG:1994}
G.~Felder.
\newblock Elliptic quantum groups.
\newblock In {\em X{I}th {I}nternational {C}ongress of {M}athematical {P}hysics
  ({P}aris, 1994)}, pages 211--218. Int. Press, Cambridge, MA, 1995.

\bibitem{FelderG/VarchenkoA:2000}
G.~Felder and A.~Varchenko.
\newblock The elliptic gamma function and {${\rm SL}(3,{\bf Z})\ltimes{\bf
  Z}^3$}.
\newblock {\em Adv. Math.}, 156(1):44--76, 2000.

\bibitem{perverseHilbert}
T.~Graber and E.~M. Rains.
\newblock Perverse {H}ilbert schemes of elliptic surfaces.
\newblock in preparation.

\bibitem{HarbourneB:1997}
B.~Harbourne.
\newblock Anticanonical rational surfaces.
\newblock {\em Trans. Amer. Math. Soc.}, 349(3):1191--1208, 1997.

\bibitem{HitchinN:2012}
N.~Hitchin.
\newblock Deformations of holomorphic {P}oisson manifolds.
\newblock {\em Mosc. Math. J.}, 12(3):567--591, 669, 2012.

\bibitem{HurtubiseJC/MarkmanE:2002b}
J.~C. Hurtubise and E.~Markman.
\newblock Elliptic {S}klyanin integrable systems for arbitrary reductive
  groups.
\newblock {\em Adv. Theor. Math. Phys.}, 6(5):873--978 (2003), 2002.

\bibitem{HuybrechtsD/LehnM:2010}
D.~Huybrechts and M.~Lehn.
\newblock {\em The geometry of moduli spaces of sheaves}.
\newblock Cambridge Mathematical Library. Cambridge University Press,
  Cambridge, second edition, 2010.

\bibitem{IshiiA/UeharaH:2005}
A.~Ishii and H.~Uehara.
\newblock Autoequivalences of derived categories on the minimal resolutions of
  {$A_n$}-singularities on surfaces.
\newblock {\em J. Differential Geom.}, 71(3):385--435, 2005.

\bibitem{IsmailMEH:2005}
M.~E.~H. Ismail.
\newblock Asymptotics of {$q$}-orthogonal polynomials and a {$q$}-{A}iry
  function.
\newblock {\em Int. Math. Res. Not.}, (18):1063--1088, 2005.

\bibitem{IRS}
M.~E.~H. Ismail, E.~M. Rains, and D.~Stanton.
\newblock Orthogonality of very well-poised series.
\newblock In preparation.

\bibitem{KajiwaraK/MasudaT/NoumiM/OhtaY/YamadaY:2006}
K.~Kajiwara, T.~Masuda, M.~Noumi, Y.~Ohta, and Y.~Yamada.
\newblock Point configurations, {C}remona transformations and the elliptic
  difference {P}ainlev\'e equation.
\newblock In {\em Th\'eories asymptotiques et \'equations de Painlev\'e},
  volume~14 of {\em S\'emin. Congr.}, pages 169--198. Soc. Math. France, Paris,
  2006.

\bibitem{KreschA:1999}
A.~Kresch.
\newblock Cycle groups for {A}rtin stacks.
\newblock {\em Invent. Math.}, 138(3):495--536, 1999.

\bibitem{KricheverIM:2004}
I.~M. Krichever.
\newblock Analytic theory of difference equations with rational and elliptic
  coefficients and the {R}iemann-{H}ilbert problem.
\newblock {\em Uspekhi Mat. Nauk}, 59(6(360)):111--150, 2004.

\bibitem{LangtonSG:1975}
S.~G. Langton.
\newblock Valuative criteria for families of vector bundles on algebraic
  varieties.
\newblock {\em Ann. of Math. (2)}, 101:88--110, 1975.

\bibitem{LehnM:1998}
M.~Lehn.
\newblock On the cotangent sheaf of {Q}uot-schemes.
\newblock {\em Internat. J. Math.}, 9(4):513--522, 1998.

\bibitem{LiC:2014}
C.~Li.
\newblock {\em Deformations of the {H}ilbert scheme of points on a del {P}ezzo
  surface}.
\newblock PhD thesis, Univ. Illinois Urbana-Champaign, 2014.

\bibitem{LowenW/VandenBerghM:2006}
W.~Lowen and M.~Van~den Bergh.
\newblock Deformation theory of abelian categories.
\newblock {\em Trans. Amer. Math. Soc.}, 358(12):5441--5483 (electronic), 2006.

\bibitem{MarkmanE:2007}
E.~Markman.
\newblock Integral generators for the cohomology ring of moduli spaces of
  sheaves over {P}oisson surfaces.
\newblock {\em Adv. Math.}, 208(2):622--646, 2007.

\bibitem{MehtaVB/RamanathanA:1984}
V.~B. Mehta and A.~Ramanathan.
\newblock An analogue of {L}angton's theorem on valuative criteria for vector
  bundles.
\newblock {\em Proc. Roy. Soc. Edinburgh Sect. A}, 96(1-2):39--45, 1984.

\bibitem{MumfordD:1966}
D.~Mumford.
\newblock {\em Lectures on curves on an algebraic surface}.
\newblock Annals of Mathematics Studies, No. 59. Princeton University Press,
  Princeton, N.J., 1966.

\bibitem{NevinsTA/StaffordJT:2007}
T.~A. Nevins and J.~T. Stafford.
\newblock Sklyanin algebras and {H}ilbert schemes of points.
\newblock {\em Adv. Math.}, 210(2):405--478, 2007.

\bibitem{NymanA:2005}
A.~Nyman.
\newblock Serre duality for non-commutative {${\mathbb P}^1$}-bundles.
\newblock {\em Trans. Amer. Math. Soc.}, 357(4):1349--1416 (electronic), 2005.

\bibitem{OdesskiiAV:2002}
A.~V. Odesski{\u\i}.
\newblock Elliptic algebras.
\newblock {\em Russian Math. Surveys}, 57(6):1127--1162, 2002.

\bibitem{P2Painleve}
A.~Okounkov and E.~Rains.
\newblock Noncommutative geometry and {P}ainlev\'e equations.
\newblock {\em Algebra Number Theory}, 9(6):1363--1400, 2015.

\bibitem{OrlovD:2014}
D.~Orlov.
\newblock Smooth and proper noncommutative schemes and gluing of dg categories.
\newblock arXiv:1402.7364v2.

\bibitem{ellGarnier}
C.~Ormerod and E.~M. Rains.
\newblock An elliptic {G}arnier system.
\newblock arXiv:1607.07831.

\bibitem{PraagmanC:1986}
C.~Praagman.
\newblock Fundamental solutions for meromorphic linear difference equations in
  the complex plane, and related problems.
\newblock {\em J. Reine Angew. Math.}, 369:101--109, 1986.

\bibitem{sklyanin_anal}
E.~Rains and S.~Ruijsenaars.
\newblock Difference operators of {S}klyanin and van {D}iejen type.
\newblock {\em Comm. Math. Phys.}, 320(3):851--889, 2013.

\bibitem{noncomm2}
E.~M. Rains.
\newblock The birational geometry of noncommutative surfaces.
\newblock arXiv:1907.11301.

\bibitem{poisson}
E.~M. Rains.
\newblock Birational morphisms and {P}oisson moduli spaces.
\newblock arXiv:1307.4032.

\bibitem{elldaha}
E.~M. Rains.
\newblock Elliptic double affine {H}ecke algebras.
\newblock arXiv:1709.02989.

\bibitem{rat_Hitchin}
E.~M. Rains.
\newblock Generalized {H}itchin systems on rational surfaces.
\newblock arXiv:1307.4033.

\bibitem{kernel}
E.~M. Rains.
\newblock Multivariate quadratic transformations and the interpolation kernel.
\newblock arXiv:1408.0305.

\bibitem{xforms}
E.~M. Rains.
\newblock Transformations of elliptic hypergeometric integrals.
\newblock {\em Ann. of Math. (2)}, 171(1):169--243, 2010.

\bibitem{isomonodromy}
E.~M. Rains.
\newblock An isomonodromy interpretation of the hypergeometric solution of the
  elliptic {P}ainlev\'e equation (and generalizations).
\newblock {\em Symmetry, Integrability, and Geometry: Methods and
  Applications}, 7(088):24 pages, 2011.
\newblock arXiv:0807.0258.

\bibitem{dets}
E.~M. Rains and V.~P. Spiridonov.
\newblock Determinants of elliptic hypergeometric integrals.
\newblock {\em Funktsional. Anal. i Prilozhen.}, 43(4):67--86, 2009.

\bibitem{RuijsenaarsSNM:1997}
S.~N.~M. Ruijsenaars.
\newblock First order analytic difference equations and integrable quantum
  systems.
\newblock {\em J. Math. Phys.}, 38:1069--1146, 1997.

\bibitem{RuijsenaarsSNM:2015}
S.~N.~M. Ruijsenaars.
\newblock Hilbert-{S}chmidt operators vs. integrable systems of elliptic
  {C}alogero-{M}oser type {IV}. {T}he relativistic {H}eun (van {D}iejen) case.
\newblock {\em SIGMA Symmetry Integrability Geom. Methods Appl.}, 11:Paper 004,
  78, 2015.

\bibitem{SakaiH:2001}
H.~Sakai.
\newblock Rational surfaces associated with affine root systems and geometry of
  the {P}ainlev\'e equations.
\newblock {\em Comm. Math. Phys.}, 220(1):165--229, 2001.

\bibitem{SchlosserMJ/YooM:2016}
M.~J. Schlosser and M.~Yoo.
\newblock Elliptic hypergeometric summations by {T}aylor series expansion and
  interpolation.
\newblock {\em SIGMA Symmetry Integrability Geom. Methods Appl.}, 12:Paper No.
  039, 21, 2016.

\bibitem{SchroerS:2005}
S.~Schr{\"o}er.
\newblock Topological methods for complex-analytic {B}rauer groups.
\newblock {\em Topology}, 44(5):875--894, 2005.

\bibitem{SeidelP/ThomasR:2001}
P.~Seidel and R.~Thomas.
\newblock Braid group actions on derived categories of coherent sheaves.
\newblock {\em Duke Math. J.}, 108(1):37--108, 2001.

\bibitem{SklyaninEK:1982}
E.~K. Sklyanin.
\newblock Some algebraic structures connected with the {Y}ang-{B}axter
  equation.
\newblock {\em Funktsional. Anal. i Prilozhen.}, 16(4):27--34, 96, 1982.

\bibitem{SklyaninEK:1983}
E.~K. Sklyanin.
\newblock Some algebraic structures connected with the {Y}ang-{B}axter
  equation. {R}epresentations of a quantum algebra.
\newblock {\em Funktsional. Anal. i Prilozhen.}, 17(4):34--48, 1983.

\bibitem{SpiridonovVP:2001}
V.~P. Spiridonov.
\newblock On the elliptic beta function.
\newblock {\em Russian Math. Surveys}, 56(1):185--186, 2001.

\bibitem{SpiridonovVP:2009}
V.~P. Spiridonov.
\newblock The continuous biorthogonality of an elliptic hypergeometric
  function.
\newblock {\em St. Petersburg Math. J.}, 20(5):791--812, 2009.

\bibitem{SpiridonovVP/WarnaarSO:2006}
V.~P. Spiridonov and S.~O. Warnaar.
\newblock Inversions of integral operators and elliptic beta integrals on root
  systems.
\newblock {\em Adv. Math.}, 207(1):91--132, 2006.

\bibitem{TyurinAN:1988}
A.~N. Tyurin.
\newblock Symplectic structures on the moduli spaces of vector bundles on
  algebraic surfaces with {$p_g>0$}.
\newblock {\em Izv. Akad. Nauk SSSR Ser. Mat.}, 52(4):813--852, 896, 1988.

\bibitem{VandenBerghM:1998}
M.~Van~den Bergh.
\newblock Blowing up of non-commutative smooth surfaces.
\newblock math.QA/9809116v1.

\bibitem{VandenBerghM:2012}
M.~Van~den Bergh.
\newblock Non-commutative {$\mathbb{P}^1$}-bundles over commutative schemes.
\newblock {\em Trans. Amer. Math. Soc.}, 364(12):6279--6313, 2012.

\bibitem{VerdierJ-L:1996}
J.-L. Verdier.
\newblock Des cat\'egories d\'eriv\'ees des cat\'egories ab\'eliennes.
\newblock {\em Ast\'erisque}, 239:xii+253 pp., 1996.
\newblock (A slightly edited reprint of Verdier's 1967 thesis).

\bibitem{YamadaY:2009}
Y.~Yamada.
\newblock A {L}ax formalism for the elliptic difference {P}ainlev\'e equation.
\newblock {\em SIGMA Symmetry Integrability Geom. Methods Appl.}, 5:Paper 042,
  15, 2009.

\bibitem{ZhuX:2016}
X.~Zhu.
\newblock An introduction to affine {G}rassmannians and the geometric {S}atake
  equivalence.
\newblock arXiv:1603.05593v2.

\end{thebibliography}

\end{document}